\documentclass[10pt,a4paper,reqno]{amsart}
\usepackage{amsmath,amsthm,amsfonts,amssymb,euscript,bbm,color,slashed}

\usepackage{environ}

\makeatletter
\def\thebibliography#1{%
  \@xp\section\@xp*\@xp{\refname}%
  \normalfont\fontsize{10pt}{12pt} 
  \@xp\labelsep\@xp=\the\labelsep\relax
  \list{\@biblabel{\@arabic\c@enumiv}}{\settowidth\labelwidth{\@biblabel{#1}}%
    \leftmargin\labelwidth \advance\leftmargin\labelsep 
    \addtolength{\rightmargin}{1cm} 
      \addtolength{\leftmargin}{-.4cm} 
    \usecounter{enumiv}%
    \let\p@enumiv\@empty
    \renewcommand\theenumiv{\@arabic\c@enumiv}}%
  \sloppy \clubpenalty\@M \widowpenalty\clubpenalty
  \sfcode`\.\@m
}
\makeatother

\usepackage{marvosym,accents}
\usepackage{mathrsfs}
\usepackage{bbold}

\usepackage{esint}

\usepackage[margin=1in,dvips]{geometry}
\usepackage{bbm}
\usepackage{graphicx}

\usepackage{hyperref}

\newtheorem{theorem}{Theorem}[section]
\newtheorem{proposition}[theorem]{Proposition}
\newtheorem{lemma}[theorem]{Lemma}
\newtheorem{remark}[theorem]{Remark}
\newtheorem{definition}[theorem]{Definition}

\newtheorem*{remark*}{Remark}

\numberwithin{equation}{section}
\numberwithin{theorem}{section}

\def\a {\alpha}

\def\b {\beta}

\def\ga {\gamma}

\def\f {\frac}

\def\de {\delta}

\def\hf{\hat{f}}

\def\perpi{\Pi^{\perp}}
\def\pernabla{\nabla^{\perp}}

\def\perR{R^{\perp}}
\def\bA{\bar{A}}
\def\bB{\bar{B}}
\def\bC{\bar{C}}

\def\bS{\overline{\Sigma}}
\def\bY{\overline{Y}}
\def\bn{\bar{n}}

\def\JapD{\langle D \rangle}

\newcommand{\be}{\begin{equation}}
\newcommand{\ee}{\end{equation}}
\newcommand{\bea}{\begin{eqnarray}}
\newcommand{\eea}{\end{eqnarray}}
\def\beaa{\begin{eqnarray*}}
\def\eeaa{\end{eqnarray*}}

\newcommand{\tr}{\mathrm{tr}}

\newcommand{\pa}{\partial}
\newcommand{\cal}{\mathcal}
\newcommand{\les}{\lesssim}

\allowdisplaybreaks

\title{On well-posedness for the timelike minimal surface equation}

\author{Georgios Moschidis}
\address{EPFL\\Institute of Mathematics,
MA B1 595\\
CH-1015 Lausanne\\
Switzerland}
\email{georgios.moschidis@epfl.ch}

\author{Igor Rodnianski}
\address{Princeton University\\
Department of Mathematics\\
Fine Hall, Washington Road\\
Princeton, NJ 08544\\
United States of America}
\email{irod@math.princeton.edu}

\begin{document}

\begin{abstract}
We obtain improved local well-posedness results for the Lorentzian timelike minimal surface equation. In dimension $d=3$, for a surface of an arbitrary co-dimension, we show a gain of $1/3$ derivative regularity compared to a generic equation of this type. The first result of this kind, going substantially beyond a general Strichartz threshold for quasilinear hyperbolic equations,  was shown for minimal surfaces of 
co-dimension one with a gain of $1/4$ regularity in \cite{AIT}. We use a geometric formulation of the problem, relying on its parametric 
representation. The natural dynamic variables in this formulation are the parametrizing map, the induced metric and the second 
fundamental form of the immersion. The main geometric observation used in this paper is the Gauss (and Ricci) equation,
dictating that the Riemann curvature of the induced metric (and of the normal bundle) 
can be expressed as the wedge product of the second fundamental form with itself. The second fundamental form, in turn, satisfies a wave 
equation with respect to the induced metric. Exploiting the problem's
diffeomorphism freedom, stemming from a non-uniqueness of a parametrization, we introduce a new gauge -- a choice of a coordinate 
system on the parametrizing manifold -- in which the metric recovers full regularity of its Riemann curvature, including the 
crucial $L^1L^\infty$ estimate for the first derivatives of the metric. Analysis of minimal surfaces with co-dimension bigger than one
requires that we impose and take advantage of an additional special gauge on the normal bundle of the surface. The proof also uses both the additional structure contained in the wedge product and the Strichartz estimates with losses developed earlier in the context of a well-posedness theory for quasilinear hyperbolic equations.
  \end{abstract}
\maketitle

\tableofcontents

\section{Introduction}

We consider the local well-posedness problem for the $3+1$ Lorentzian timelike minimal surface equation which describes an immersion
of a 4-dimensional timelike surface (with the topology of ${\cal M}=I\times \mathbb{T}^3$, where $I\subset \Bbb R$ is an interval) in ${\cal N}=\mathbb{R}\times \mathbb{T}^3 \times \mathbb{R}^{n-3}$ with vanishing mean curvature. The case $n-3=1$ is the so called co-dimension $1$ case in which the surface can be 
represented, at least locally, as a graph $x^4=f(t,x^1,..,x^3)$ and the minimal surface equation takes the form of a scalar quasilinear
hyperbolic equation 
\be\label{eq:scalar}
m^{\a\b} \pa_\a\left(\frac {\pa_\b f}{\sqrt{1+m^{\mu\nu}\pa_\mu f\pa_\nu f}}\right)=0 
\ee
where $m$ is the standard Minkowski metric on $\mathbb{R}\times \mathbb{T}^3$. The condition that the surface is timelike is just 
the statement that the vector $(m^{\a\b} \pa_\b f,1)$ -- the normal to the surface -- is spacelike with respect to the Minkowski metric 
on $\mathbb{R}\times \mathbb{T}^3 \times \mathbb{R}$, i.e., $1+m^{\mu\nu}\pa_\mu f\pa_\nu f>0$.

The minimal surface equation also admits another, {\it parametric}, representation: 
\be\label{eq:param}
g^{\a\b} k_{\a\b}=0
\ee
in which $g$ and $k$ are the first and second fundamental forms of the immersion $Y:{\cal M}=I\times \mathbb{T}^3\to {\cal N}=\mathbb{R}\times \mathbb{T}^3 \times \mathbb{R}^{n-3}$. The metric $g$ is induced by the ambient Minkowski metric $m$ on ${\cal N}$ and is explicitly given by
$$
g_{\a\b}=m(\pa_\a Y,\pa_\b Y),
$$
while
$$
k_{\a\b}=\Pi^\perp\pa_\a\pa_b Y,
$$
with $\Pi^\perp$ -- the orthogonal projection on the normal bundle of $Y({\cal M})$ in ${\cal N}$. The minimal surface equation is then 
just the statement of vanishing of the mean curvature vector of the immersion.

The timelike minimal surface equation is a Lorentzian counterpart 
of the Riemannian minimal surface problem which also arises naturally, in particular, in the study of the cosmological structure formation -- {\it domain walls} or {\it membranes}
and {\it strings}, see \cite{VS} and the references therein, and is generated as Euler-Lagrange equations from the Lagrangian
$$
{\cal L} = \int_{\cal M} \sqrt{-|g|}
$$

Examining both \eqref{eq:scalar} and \eqref{eq:param}, 
we see that they are both quasilinear hyperbolic equations of the type 
\be\label{eq:type}
g^{\a\b}(\pa\phi)\pa_\a\pa_\b\phi=0
\ee
in which $g$ is a Lorentzian metric with components dependent on the {\it derivatives} of the unknown solution $\phi$. We should note
already that there is a substantial difference between \eqref{eq:scalar} and \eqref{eq:param}: the former is a scalar equation for 
an unknown function while the latter is a system of $(n-3)$-equations for $n+1$-unknown components of $Y$. The underdeterminancy 
of the parametric representation plays an important role in this paper. For now, however, we continue the discussion of equations of the 
type \eqref{eq:type} in the case when the number of equations is equal to the number of the unknowns. In addition to the case of the minimal 
surface equation, they also arise in the study of compressible fluid dynamics. In fact, \eqref{eq:scalar} also describes the propagation 
of the velocity potential for the irrotational compressible Euler equations with the linear equation of state \cite{C}.  
Equations of the type \eqref{eq:type} can be contrasted with the ones of the form 
\be\label{eq:type1}
g^{\a\b}(\phi)\pa_\a\pa_\b\phi=0
\ee
which appear for example in general relativity, \cite{Ch-Br}.

The two general classical questions arising in the study of equations of both of these types are $1)$ the long term dynamics and stability of
their stationary solutions and $2)$ the local dynamics of general data. The former question is very sensitive to the structure of the specific equation. For the 
co-dimension one minimal surface equation \eqref{eq:scalar}, the stability of the trivial solution  $u=0$ was studied in \cite{Br,Lind}. 
More recently, the stability of planar waves and of a stationary catenoid was addressed in \cite{AW,KL,DKSW,Lu}. The question of formation of singularities 
was studied in \cite{NT,BM}.

For the latter problem, that of local well-posedness, one asks for the optimal (lowest) Sobolev exponent~$s$ such that the problem 
with initial data $(\phi_0, \phi_1=\pa_t\phi|_{t=0})\in H^s\times H^{s-1}$ admits a unique solution in the space $C([0,T), H^s)$ with a length of the time interval $T$ dependent only on the norm of the data. The local well-posedness statement is also typically accompanied by 
continuous dependence on the initial data and propagation of higher regularity.

The classical result \cite{HKM}, which combines the (higher $H^s$ version of the) standard energy estimate 
$$
\|(\phi,\pa_t\phi)(t)\|_{H^1\times L^2} \les \exp{\left (\int_0^t\|\pa g\|_{L^\infty_x}\right)} \|(\phi_0,\phi_1)\|_{H^1\times L^2} 
$$
with the Sobolev embedding $L_x^\infty \subset H^{\frac 32+}$ in dimension $3$, shows that the equations \eqref{eq:type1} and \eqref{eq:type} are well-posed in $H^s$ with $s>\frac 32+1$ and $\frac 32+1+1$ respectively. The difference of one degree of 
differentiability comes from the fact that for \eqref{eq:type1}, $\pa g\sim \pa\phi$ while for \eqref{eq:type}, $\pa g\sim \pa^2\phi$. 
In dimension $d$, the same result holds with $\frac 32$ replaced by $\frac d2$.
The exponents $\frac d2$ and $\frac d2+1$ are critical and correspond to the scaling of the equations \eqref{eq:type1} and \eqref{eq:type}. 

For semilinear equations of the form
\begin{equation}\label{General semiliner equation}
\Box\phi=Q(\pa\phi,\pa\phi)
\end{equation}
where $\Box$ is the standard D'Alembert operator on Minkowski space, the critical exponents are sometimes reachable (see e.g. \cite{Tao1,RTao}) 
when 
$Q$ is a quadratic form satisfying the null condition \cite{K2}
\be\label{eq:null}
Q(\xi,\xi)=0, \quad \forall \xi: \, m(\xi,\xi)=0.
\ee
Even in the absence of the null condition, however, for equations of the form \eqref{General semiliner equation} the classical well-posedness result can be improved by replacing 
the Sobolev embedding with the {\it Strichartz} estimate which, for instance, in dimension $d=3$ takes the form 
\be\label{eq:str}
\|\pa\phi\|_{L^2_tL^\infty_x} \les \|(\phi_0,\phi_1)\|_{H^{2+}\times H^{1+}} + \|\Box\phi\|_{L^1_t H^{1+}}
\ee
Here, the $L^p_t$-norms are taken on a finite interval $[0,T]$, $2+$ and $1+$ denote exponents which are strictly larger but can be arbitrarily close 
to the indicated number, 
and the constant in the above inequality depends on $T$.
This estimate exploits the dispersive properties of the $\Box$ and improves the well-posedness result by $\frac 12$ derivative.

In the case of quasilinear equations, going beyond the classical result requires grappling with the problem of proving the analog of
\eqref{eq:str}  with $\Box$ replaced by $\Box_g=g^{\a\b}\pa_\a\pa_\b$. If $g$ is smooth or even $C^{1,1}$ the estimate 
indeed holds \cite{Ka},\cite{Sm}. However, for the problems of the form \eqref{eq:type1} or \eqref{eq:type}, $g$ depends on the solution (or its derivative) and thus, already 
in the classical case (e.\,g.~when $\phi\in H^s$ with $s>\frac 52$ for \eqref{eq:type1}), $g$ is merely $C^1$.

The revolution in the study of Strichartz estimates for linear wave equations with rough coefficients and their applications 
was initiated in \cite{B-C1} (\cite{Ta1},\cite{B-C2},\cite{Ta2}). These developments culminated in the following Strichartz estimate, proven 
in \cite{Ta},
for solutions of {\it linear} wave equations 
$$
\pa_\a\left(g^{\a\b} \pa_\b\phi\right)=f_1+f_2
$$
on $\Bbb R^{3+1}$ with a metric $g$ satisfying local uniform Lorentzian bounds in a given coordinate system as well as the 
assumption 
$$
\|\pa g\|_{L^1_t L^\infty_x} \le \mu 
$$
for some constant $\mu$. Then
\be\label{eq:loss}
\||D|^{-\frac 16-} \phi\|_{L^2_tL^\infty_x}\les \|(\phi_0,\phi_1)\|_{H^1\times L^2} + \|f_1\|_{L^1_tL^2_x}+\||D| f_2\|_{L^2_t L^1_x}
\ee
Here, $D$ denotes the spatial derivatives and the constants are allowed to depend on $\mu$. In fact, that dependence 
can be made precise and, additionally, the norms of the inhomogeneous terms can be strengthened. But this is the form of the 
estimate that we will use in this paper.

Let us note that, in applications, one typically uses this result for a solution $\phi$ localized to a dyadic spatial frequency $\lambda$ 
and a metric $g$ supported on frequencies $\le \lambda$. Under these conditions, the estimate \eqref{eq:loss} can be transformed
into a more useful form, directly applicable to nonlinear problems:
\be\label{eq:loss'}
\|\phi\|_{L^2_tW^{\sigma,\infty}_x}\les \|(\phi_0,\phi_1)\|_{H^{\sigma+1+\frac 16+}\times H^{\sigma+\frac 16+}} + \|f_1\|_{L^1_tH^{\sigma+\frac 16+}}+\|f_2\|_{L^2_t W^{1+\sigma+\frac 16+,1}}.
\ee
The same estimate also holds for the equation 
$$
g^{\a\b} \pa_\a\pa_\b\phi=f.
$$
The bound \eqref{eq:loss'} is a Strichartz estimate with $\frac 16$-derivative loss, relative to the one that holds for a wave equation with smooth 
coefficients. The assumption $\pa g\in L^1_t L^\infty_x$ is precisely what one would need to close the energy estimates for the 
wave equation and, simultaneously, what one would hope to recover for a solution of a quasilinear problem. This estimate almost immediately implies local well-posedness in $H^s$ with $s>\frac 32+\frac 12+\frac 16$ for equation \eqref{eq:type1} and 
$s>\frac 52+\frac 12+\frac 16$ for \eqref{eq:type}. Its sharpness for general metrics satisfying the assumptions of the theorem 
was shown in \cite{S-T1}.

For nonlinear problems it turned out, however, that this result is not sharp, \cite{KR}. This eventually led to the optimal (at least from 
the point of view of Strichartz technology) well-posedness results without losses, first for the Einstein equations \cite{KR} and then 
for general equations \cite{S-T2} (later also with a different proof in \cite{Wa}.) All of these results exploit the observation that the 
metric itself satisfies an equation similar to that for $\phi$:
$$
g^{\a\b}\pa_\a\pa_\b g_{\mu\nu}={\text{better}}.
$$
This has deep implications for the geometry of null hypersurfaces associated to $g$, which, in turn, play an important role 
in Strichartz estimates. 

As is well-known \cite{Lind}, going beyond the Strichartz--compatible well-posedness threshold requires the presence 
of special structures in the equation. For semilinear problems, the appropriate structure is the null condition \eqref{eq:null}. To exploit it, one has to pass from 
the Strichartz estimates to {\it bilinear} estimates for the wave equation, see e.g. \cite{KM}. 

For quasilinear problems of the form \eqref{eq:type} or 
\eqref{eq:type1}, the situation is much more complicated and largely completely open. 
The first result in this direction was shown for the Einstein equations. In a certain gauge, these equations take the form  \eqref{eq:type1},
but the classical null structure is absent in the system. Nonetheless, the work \cite{KRS} established what is known as the solution of the $L^2$ curvature 
conjecture, corresponding to the statement of $H^2$ well-posedness. This amounts to an $\epsilon$-improvement of the results for generic equations 
of the form  \eqref{eq:type1}. There are some indications that this result might be sharp, e.g. \cite{EL}.

The first, and until now the only, outstanding progress going beyond an $\epsilon$-improvement of a Strichartz compatible 
well-posedness result for a quasilinear problem was achieved very recently in \cite{AIT}. That work focused specifically 
on the 
minimal surface equation \eqref{eq:scalar} and showed that, remarkably, it is well-posed 
in $H^s$ with $s>\frac 32+1+\frac 14$ which constitutes a $\frac 14$ improvement over a generic Strichartz result. The same 
gain takes place in higher dimensions, while in dimension two the result gives the gain of $\frac 38$-derivative.

Model quasilinear elliptic-hyperbolic cubic systems in higher dimensions, using bilinear or improved Strichartz estimates, had been also previously 
considered in \cite{B-C3}.
\vskip 1pc

In this paper we re-examine the Lorentzian minimal surface problem. The purpose here is threefold.
\begin{itemize}
\item We show that the minimal surface equation in dimension $d=3$ is well posed in $H^s$ with $s>\frac d2+1+\frac 16$. In fact, the result holds for 
immersed surfaces of arbitrary co-dimension. Similar improved results also hold in higher dimensions.
\vskip .5pc

\item Instead of \eqref{eq:scalar}, we adopt the parametric point of view \eqref{eq:param} on the minimal 
surface equation which allows us
to unlock its geometric properties and structures, providing a clear reason behind the result. Rather than using the equation 
\eqref{eq:param} for the map $Y$, we derive the wave equation for the second fundamental form of the immersion $k$
\be\label{eq:wavek}
g^{\a\b}\nabla_\a\nabla_\b k= {\mathrm{Riem(g)}}\cdot k 
\ee
where ${\text{Riem}}(g)$ is the Riemann curvature tensor of $g$. The Gauss-Codazzi equations for the immersion imply that 
$$
{\text{Riem}}(g)=k\wedge k
$$
This is the first main geometric structure exploited in the paper. Its significance is not that ${\text{Riem}}$ appears on the 
right hand side of the equation \eqref{eq:wavek} but rather that it gives a measure of regularity of the metric $g$. Based on the expression for $g$:
\[
g_{\a\b}=m(\pa_\a Y,\pa_\b Y)
\]
one would expect that the Riemann curvature of $g$ depends on the 3rd derivatives of the map $Y$. Instead, since $k=\Pi^\perp\pa^2 Y$,
it is quadratic in the second derivatives of $Y$.
In the scale of the Sobolev spaces $L^\infty_t H^\sigma$, the 
Riemann curvature has exactly the same regularity one expects from the fact that the metric $g$ is expressed through the first derivatives 
of the map $Y$. However, the regularity of ${\text{Riem}}(g)$ improves significantly in the $L^1_t W^{\sigma,\infty}_x$ norm, provided one has the Strichartz estimate for 
the second fundamental form $k$. Such an improvement for the Riemann tensor should be then translated into the corresponding 
estimate for the metric $g$ which then in turn, in view of \eqref{eq:loss'}, would give us the Strichartz estimate for $k$, closing 
the loop. 

We should note that unlike the previous works on quasilinear problems, quoted above, where improvements in regularity (which took place well above the level considered here) came from the fact that the metric $g$ itself satisfies a wave equation, the gain 
here is more profound and happens through the observation that the Riemann tensor enjoys better space-time integrability properties. 

Obtaining estimates on the metric $g$ from the Riemann curvature is directly related to the previously discussed 
underdeterminancy of the parametric formulation. To find the metric from its curvature we need to fix the gauge.
\vskip .5pc

\item Finding the gauge that allows us to recover full regularity of the metric from its curvature is an aspect of this paper which may 
be of independent interest. The gauge itself and the procedure of reconstructing the metric from curvature are general and do not 
depend on the minimal surface equation. The same holds for a good portion of quantitative estimates which only require 
certain structure (better behavior) of one component of the Ricci tensor. More refined estimates, including the important 
$L^1_t W^{\sigma,\infty}_x$ estimates, employ the full structure of the minimal surface equations, reminiscent of the bilinear 
estimates for semilinear problems. These estimates address the so called ``high-high" interactions, in which the product of 
two functions with high frequencies is of a low frequency itself. The gauge is evolutionary and some of its aspects could be perhaps compared  with the {\it caloric} gauge for the wave map and Yang-Mills problems \cite{Tao,SJO} which generates a parabolic 
flow but with the help of an additional artificial time variable.  

There is actually another gauge, which can be thought of as almost independent of the parametrization gauge discussed above, that appears in this work. This gauge is used to address the problem of higher ($>1$) co-dimension minimal surfaces and is imposed on 
the connection on the normal bundle of the surface $Y({\mathcal M})$. In the co-dimension $1$ case the bundle is trivial, but in the
higher co-dimension case we need to choose a frame on the normal bundle with the property that the connection 
coefficients inherit the full regularity of the curvature tensor. Similar to the choice of the gauge on ${\cal M}$, 
it leads to an elliptic-parabolic system for the connection coefficients. The gauge is designed to handle ``high-high" interactions and 
could be again compared to the caloric gauge in some of its aspects.

\end{itemize}
We now state an informal version of the main result, for the precise version see Theorems \ref{thm:Existence} and \ref{thm:Uniqueness}, 
and then review the remaining two points described above.
\begin{theorem}\label{thm:Existence'}
Let $s>\f52+\f16$. There exist constants $0 < \epsilon <1$ and $C>1$ depending only on $s$ such that, for any initial data pair $ \bY: \mathbb T^3 \rightarrow \mathcal N = \mathbb R \times \mathbb T^3 \times \mathbb R^{n-3}$, $\bn: \mathbb T^3 \rightarrow \mathbb R^{n+1}$ satisfying
\begin{equation}\label{D}
\mathcal{D} \doteq \|\bY-\bY_0\|_{H^s}+\|\bn-\bn_0\|_{H^{s-1}} < \epsilon
\end{equation}
(where $(\bY_0,\bn_0)$ is the flat initial data set), there exists a unique (modulo reparametrizations of $[0,1]\times \mathbb{T}^3$) solution
$Y:[0,1]\times \mathbb{T}^3 \rightarrow \mathcal{N}$  such that
\begin{equation}\label{Theorem embedding'}
\sum_{l=0}^2 \| \partial^l (Y - Y_0) \|_{L^{\infty}_tH^{s-l}} \le C \mathcal{D},
\end{equation}
\begin{equation}\label{Theorem metric}
\| \partial g \|_{L^1_t W^{\delta,\infty}} +\sum_{l=1}^2\Big( \| \partial^l g \|_{L^2_t H^{2-l+\f16}} + \|  \partial^l g \|_{L^{\infty}_t H^{s-1-l}}\Big) \le C \mathcal{D}.
\end{equation}
\end{theorem}
\begin{remark}
The same results with $s>\frac d2+1+\frac 16$ hold in dimensions $d\ge 4$.
The proof in this paper applies almost mutatis mutandis and requires a version of the Strichartz estimate \eqref{eq:loss'} 
for the appropriate dimension, also available in \cite{Ta}.
\end{remark}
\subsection{Minimal surface equation in parametric form}
We view a surface immersed in ${\cal N}=\mathbb{R}\times \mathbb{T}^3 \times \mathbb{R}^{n-3}$ as the image of a 
map $Y:{\cal M}=I\times \mathbb{T}^3\to {\cal N}$. 
The surface $Y({\cal M})$ is minimal if the trace, with respect to the metric $g_{\a\b}=m(\pa_\a Y,\pa_\b Y)$ induced by the 
Minkowski metric on ${\cal N}$, of the second fundamental form of the embedding $k_{\a\b}=\Pi^\perp\pa_\a\pa_\b Y$ 
vanishes:
\be\label{eq:param'}
g^{\a\b} k_{\a\b}=0
\ee
We note that such parametrization of $Y(M)$ is intrinsically not unique as the same surface can also be parametrized 
by $Y\circ\Phi(M)$ for any diffeomorphism $\Phi:{\cal M}\to {\cal M}$. 
 
The Gauss--Codazzi equations associated to an immersion $Y(M)$ into a {\it flat} target manifold provide relations between 
its intrinsic geometry, in the form of the metric $g$ and its Riemann curvature ${\mathrm Riem}=R$, and the second fundamental form $k$:
\bea\label{eq:Gauss}
&R_{\a\b\ga\de}=k_{\a\ga} k_{\b\de} - k_{\a\de} k_{\b\ga}=\left(k\wedge k\right)_{\a\b\ga\de},\\
&\nabla_\ga k_{\a\b}=\nabla_\a k_{\ga\b}
\eea
(in the higher codimension case, one also considers additionally the Ricci equation relating the curvature of the normal bundle to the second fundamental form; more on that later). Rather than considering the equation implied by \eqref{eq:param'} for the map $Y$, we derive equations for the geometric quantities $g$ and $k$ themselves.

The equation 
\be\label{eq:k}
g^{\a\b} \nabla_\a\nabla_b k= {\mathrm Riem}\cdot k
\ee
is the wave equation for $k$. Its derived from the Gauss-Codazzi equations in two steps: first, taking the trace and using 
\eqref{eq:param'} and, then, contracting with another derivative. Since the derivation uses \eqref{eq:param'}, this equation encodes the minimality of the immersion. The equation \eqref{eq:k} needs to be supplemented by one for $g$. The Gauss equation gives us its 
Riemann curvature in terms of $k$. To recover the metric from ${\mathrm Riem}$ we need to fix the gauge. We will address this issue 
in the next section. For the time being we will simply write schematically
\be\label{eq:Gau}
g=\pa^{-2} {\mathrm Riem}=\pa^{-2} \left(k\wedge k\right)
\ee
Let us now examine the expected regularity of $g$ and $k$. If the map $Y\in L^\infty_t H^s$ with $s>\frac 52+\frac 16$, as stated 
in Theorem \ref{thm:Existence'}, then $g\in L^\infty_t H^{s-1}$ and $k\in L^\infty H^{s-2}$. To close the ($H^{s-2}$) energy estimate 
for the wave equation for $k$ we need to be able to control $\partial g\in L^1_tL^\infty$ as well as ${\mathrm Riem}\cdot k \in L^1_tH^{s-3}$.

If ${\mathrm Riem}\in L^\infty_t H^{s-3}$ as could be expected from the fact that $g\in L^\infty_t H^{s-1}$, then $k$ would have to 
lie in the space $L^1_tL^\infty$ which, in view of the $H^{s-2}$  regularity of the initial data for $k$ (recalling that $s-2>\frac 12+\frac 16$),
is far below what one would need for such a Strichartz estimate for $k$, even if the metric $g$ was smooth.

The key point is that the Riemann curvature tensor is better than what might have been expected simply from the regularity of the metric 
$g$. Recall that, schematically, $g\sim (\pa Y)^2$ which implies that 
$$
{\mathrm Riem} \sim (\pa^3 Y) (\pa Y) + (\pa^2 Y)^2
$$
The Gauss equation, on the other hand, tells us that the first term on the right hand side above vanishes:
$$
{\mathrm Riem}=k\wedge k\sim (\pa^2 Y)^2.
$$
While from the point of view of the Sobolev space~$L^\infty_t H^{s-3}$ this structure provides only a marginal gain, it is 
the time integrated (and Strichartz related) $L^2_t H^\sigma$ and $L^1_t W^{\sigma,\infty}$ norms that exhibit a substantial advantage.
In particular, even though pointwise in time the curvature belongs to a Sobolev space with a negative ($s-3=-\frac12+\frac 16+$) index, 
it also lies in the space $L^2_t H^{\frac 16+}$. This can be compared with the geometric considerations that played a very important 
role in the solution of the $L^2$ curvature conjecture in \cite{KRS}. There, the control of the $L^\infty_t L^2$ norm
of curvature was absolutely critical. Somewhat ironically, that control was needed to prove a Strichartz estimate {\it without 
a loss}, yet, in this problem, while the curvature is only in the $L^\infty_t H^{-\frac 12+\frac 16+}$ space and the Strichartz estimate
loses $\frac 16$-derivative, it leads to a better regularity result. The key, of course, is the better space-time integrability properties 
of ${\mathrm{Riem}}$.

Let us now put it all together. We assume 
$$
k\in L^\infty_t H^{s-2}\cap L^2_t W^{-\frac 12,\infty},\qquad \pa g\in L^1_t L^\infty
$$
and check the consistency of these norms with the estimates that we can extract from the system of equations for $g$ and $k$. 
As usual, the proper set up should be a bootstrap argument with smallness coming from the size of initial data and the nonlinear terms.
Below, we will ignore various $\epsilon$-losses and work with the exact numerology, e.g. $s$ will simply be replaced by $\frac 52+\frac 16$.

The assumption that $\pa g\in L^1_t L^\infty$ allows us to bound both the energy and the Strichartz norm (with $\frac 16$-derivative loss) 
of $k$:
$$
\|k\|_{L^\infty_t H^{\frac 12+\frac 16}} + \|k\|_{L^2_t W^{-\frac 12,\infty}}\les \|k_0\|_{H^{\frac 12+\frac 16}}+ \|{\text {Riem}}\cdot k\|_{L^1_t H^{-\frac 12+\frac 16}}
$$
We should note that, in order to apply the Strichartz estimate \eqref{eq:loss'}, we first need to convert the wave equation \eqref{eq:k} to a 
scalar one. This introduces additional inhomogeneous terms of the type $g\cdot\pa g\cdot \pa k$ and $g\cdot \pa^2 g\cdot k$. 
Their treatment is not dissimilar to that one of ${\mathrm {Riem}}\cdot k$. However, some of the ``high-high" interactions require 
both the use of the second $L^2_t W^{\frac 12+\frac 16,1}$ inhomogeneous norm in \eqref{eq:loss'} and a direct analysis of the 
energy estimate. For the sake of the exposition, in the introduction we will only consider the ``high-low" interaction in the term 
${\mathrm {Riem}}\cdot k$. 
 For that, we can estimate
 $$
 \|{\mathrm {Riem}}\cdot k\|_{L^1_t H^{-\frac 12+\frac 16}}\les \|{\mathrm {Riem}}\|_{L^2_t H^{\frac 16}} \|k\|_{L^2_t W^{-\frac 12,\infty}}
 $$ 
 Here is where we use the relation ${\mathrm Riem}=k\wedge k$. Interpolating between $L^\infty_t H^{\frac 12+\frac 16}$ and 
 $L^2_tW^{-\frac 12,\infty}$ norms of $k$ shows that  $\|{\mathrm {Riem}}\|_{L^2_t H^{\frac 16}}$ is bounded. This closes estimates for $k$.
 
 It remains to recover the $L^1_tL^\infty$ estimate for $\pa g$. According to our scheme,
 $$
 \pa g \sim \pa^{-1} {\mathrm Riem} =\pa^{-1} \left(k\wedge k\right)
 $$
 where we will remain vague about the meaning of $\pa^{-1}$, simply assuming that it will gain any desired derivative. This is of course 
 far from a trivial issue but it will be part of the gauge discussion in the next section.
 
 For the ``high-low" part of the $k\wedge k$ product we can estimate 
 $$
\| \pa^{-1} \left(k_{\text high}\wedge k_{low}\right)\|_{L^1_tL^\infty}\les \|k\|^2_{L^2_tW^{-\frac12,\infty}}
 $$
 For the ``high-high" interactions the argument is more subtle and uses an additional structure present in the product $k\wedge k$,
 reminiscent of one used in bilinear estimates, e.g \cite{KM}.
 Recall that 
 $$
 \nabla_\gamma k_{\a\b}=\nabla_\a k_{\b\gamma}
 $$
 Below, we will ignore lower order terms arising from commutators. We will also let $i$ denote spatial indices and $\Delta=\nabla^i\nabla_i$  
 be a Laplace operator associated to spatial $3$-dimensional slices. We can then write
 $$
 k_{\a\b}=\frac{\nabla^i}{\Delta}  \nabla_i k_{\a\b}=\nabla_\a \frac{\nabla^i}{\Delta} k_{i\b}
 $$
 Using this repeatedly, we obtain 
 $$
 k_{\a\gamma} k_{\beta\delta} = \nabla_\a\left(\frac{\nabla^i}{\Delta} k_{i\gamma} k_{\b\delta}\right)-\nabla_\b\left(\frac{\nabla^i}{\Delta} k_{i\gamma} k_{\a\delta}\right)+k_{\b\gamma} k_{\a\delta}
 $$
 which implies that 
 $$
 R_{\a\b\gamma\delta}= \nabla_\a\left(\frac{\nabla^i}{\Delta} k_{i\gamma} k_{\b\delta}\right)-\nabla_\b\left(\frac{\nabla^i}{\Delta} k_{i\gamma} k_{\a\delta}\right).
 $$
 It is now easy to see that for the ``high-high" portion 
 $$
\| \pa^{-1} {\mathrm Riem}\|_{L^1_tL^\infty}\les \|k\|^2_{L^2_tW^{-\frac12,\infty}}.
 $$
 \subsection{Balanced gauge(s)}
 We now return to the question of recovery of the metric $g$ from its curvature tensor. They are both defined on the base manifold 
 ${\cal M}=I\times {\Bbb T}^3$ and are subject to the diffeomorphism freedom of the problem, in which the image of an 
 immersion $Y(M)$ remains invariant under maps $\Phi:{\cal M}\to {\cal M}$. 
 
 In the analysis above we schematically expressed 
 $$
 g\sim \pa^{-2} {\mathrm{Riem}}
 $$
 which allowed us to say that $\pa g\sim \pa^{-1} {\mathrm{Riem}}$ and to exploit the structure of ${\mathrm{Riem}}=k\wedge k$, effectively 
 resulting in 
 $$
 \pa g\sim \pa^{-\frac 12} k\cdot \pa^{-\frac 12} k.
 $$
 All of these need to be justified through the choice of a gauge. As should be clear already, in the chosen gauge, the metric should 
 have the full regularity of the curvature tensor with respect to {\it all} derivatives. The gauge will also be evolutionary, that is it will 
 be constructed from the Riemannian metric $\bar g$ induced on the slices $\bar\Sigma_\tau=\{x^0=\tau\}$, the lapse function $N$ and 
 the shift vector $\beta$, starting from the initial hypersurface $\bar\Sigma_0$. The metric $g$ takes the form
 $$
 g = -N^2 (dx^0)^2 +\bar{g}_{ij}(dx^i + \b^i dx^0) (dx^j + \b^j dx^0)
 $$
 and the equations associated with the geometry of the $\bar\Sigma_\tau$ foliation are as follows:
 \begin{align}
\partial_\tau \bar{g}_{ij} \doteq \mathcal L_{\partial_{x^0}} \bar{g}_{ij} &= -2N h_{ij} + \bar{\nabla}_i \b_j + \bar{\nabla}_j \b_i, \label{Variation metric'}\\
\partial_\tau h_{ij} \doteq \mathcal{L}_{\partial_{x^0}} h_{ij} &= -\bar{\nabla}_i \bar{\nabla}_j N -N h_{ik}h^k_j +N R_{i\alpha j\beta}\hat n^\a \hat n^\b +h_{kj}\bar{\nabla}_i \b^k + h_{ki}\bar{\nabla}_j \b^k + \b^k \bar{\nabla}_k h_{ij}, \label{Variation second fundamental form'}
\end{align}
 with $\hat n=N^{-1}\pa_\tau-\beta$ -- the unit normal to $\bar\Sigma_\tau$. These are the general equations relating the components 
 of the metric $g$ to the Riemann tensor but they do not yet constitute a solvable system.
 
 There are several canonical gauges that can be associated to the $3+1$ decomposition: maximal $+$ transported coordinates 
 \cite{CK} and constant mean curvature (CMC) $+$ spatially harmonic \cite{AM}, for instance. 
 None of them however satisfy our main requirement 
 that $g$ recovers the full regularity of the Riemann tensor.
 
 For instance, the maximal plus transported coordinate gauge is characterized by the conditions 
 $$
 \tr h =0,\qquad \beta=0,
 $$
where $\tr (\cdot)$ denotes the trace of a symmetric $2$-tensor with respect to $\bar g$. The lapse $N$ then satisfies an elliptic equation
 $$
 -\Delta_{\bar g} N+ N\left(R_{i0j0} \bar g^{ij} + |h|^2\right)=0.
 $$
 Ignoring the (unknown) term $|h|^2$, this looks like a very strong relation between the component of the metric $g$ -- lapse $N$ -- 
 and the Riemann tensor:
 $$
 N\sim \Delta_{\bar g}^{-1} {\mathrm{Riem}}
 $$
 The problem however is that while this relation recovers the full {\it spatial} regularity of $N$, it does not do the same for its 
 time derivatives:
 $$
 \pa_\tau N\sim \Delta_{\bar g}^{-1} \pa_\tau{\mathrm{Riem}}
 $$
 and even worse,
 $$
 \pa^2_\tau N\sim \Delta_{\bar g}^{-1} \pa^2_\tau{\mathrm{Riem}}.
 $$
 One may be tempted to remember that the Riemann tensor satisfies the Bianchi identities which, perhaps, could be used 
 to trade $\pa_\tau$ for $\bar\nabla$ derivatives, but the specific component $g^{ij} R_{i0j0}$ arising here does not enjoy 
 this property. Even in the case when this particular component vanishes, e.g. the Einstein equations, the same problem 
 would still reappear somewhere else.
 
 We now introduce the new gauge. We first impose the following conditions 
 \begin{align}\label{eq:gauge}
 &\tr h=-\Delta_{\bar g} |D|^{-1} \left(N-1\right),\\
& \Delta_{\bar g} \left(x^i\right)=\Delta_{\bar g} |D|^{-1} \left(\beta^i\right)
 \end{align}
Here $|D|$ is a an operator with the symbol $|\xi|$, defined relative to the $x^i$ coordinates on $\bar\Sigma_\tau$. Note that, setting the 
right hand side of the equations above to $0$ gives rise to the maximal plus spatially harmonic gauge. 
The conditions \eqref{eq:gauge} however are {\it not} perturbations of that gauge since the right sides above are top 
order, both in terms of regularity and size.  

Taking the trace of the $h$-equation now gives a parabolic equation for the lapse:
$$
\pa_\tau (N-1) + |D|(N-1)\sim |D|^{-1}\left(R_{i0j0} \bar g^{ij}\right).
$$
The same happens for the shift vector $\beta$:
$$
\pa_\tau \beta + |D|\beta \sim |D|^{-1}\left(R_{imjk}\right).
$$
The metric $\bar g$ still satisfies an echo of the spatially harmonic gauge:
\be\label{eq:ell}
\Delta_{\bar g} \bar g\sim R_{imjk},\qquad \Delta_{\bar g} h \sim R_{0mjk}.
\ee
We also still have the relations
\be\label{eq:tran}
\partial_\tau \bar{g} \sim h + D\b,\qquad \pa_\tau h\sim D^2 N + R_{i0j0}.
\ee
The point is that the time and spatial derivatives are much better balanced in this gauge. The parabolic equation for $\beta$ together 
with the Bianchi identities recovers the full regularity (space-time) of $\beta$. Spatial regularity of $\bar g$ follows from the 
first equation in \eqref{eq:ell}. Its mixed derivatives can be obtained from the first equation in \eqref{eq:tran} and the spatial 
regularity of $h$ implied by the second equation in \eqref{eq:ell}. Its time regularity follows from the last equation in \eqref{eq:tran}. 

The lapse $N$ however is still somewhat problematic. We have its full spatial regularity and can also control both $\pa_\tau$ and 
$\pa_\tau |D|$ derivatives. But its $\pa_\tau^2$-derivative is still of the form 
$$
|D|^{-1}\pa_\tau \left(R_{i0j0} \bar g^{ij}\right).
$$ 
If the term $R_{i0j0} \bar g^{ij}={\text{Ric}}_{00}$ vanished, as is the case for the Einstein vacuum equations, for instance, we would 
not have to do anything else. As it stands, we need to renormalize our gauge. First,
$$
R_{i0j0}=k_{ij} k_{00}-k_{i0} k_{j0}.
$$
As before, using the equations $\nabla_\a k_{\b\gamma}=\nabla_\gamma k_{\a\b}$ we can rewrite 
$$
R_{i0j0}\sim \pa_\tau \left(\frac{\nabla^m}{\Delta} k_{mj} k_{0i}\right) - \pa_i\left(\frac{\nabla^m}{\Delta} k_{mj} k_{00}\right).
$$
The second term has a spatial derivative in front which will offset the $|D|^{-1}$ appearing the lapse equation. To eliminate 
the first term we modify the gauge by setting instead:
$$
\tr h = -\Delta_{\bar g} |D|^{-1} \left(N-1\right) -{\cal F}, \qquad {\cal F}:= \bar g^{ij}\frac{\nabla^m}{\Delta} k_{mj} k_{0i}.
$$
The resulting structure of the equations for the metric $g$ (decomposed into $(N,\beta,\bar g)$) is such that we can effectively express 
$$
\pa g \sim |D|^{-1} {\text{Riem}}.
$$
Moreover, the components of ${\mathrm{Riem}}$ appearing above are such that when we express 
$$
{\mathrm{Riem}} = k\wedge k
$$
and use the equations for $k$ to pull the derivatives out, these will only be the spatial derivatives. That is to say, in this gauge we can indeed 
justify the relation 
$$
\pa g\sim |D|^{-\frac 12} k\cdot  |D|^{-\frac 12} k
$$
for the ``high-high" interactions. 

It is clear that the imposed gauge is of an elliptic-parabolic type and, importantly, enjoys a {\it lower 
triangular} structure:
\begin{equation}\label{Model system of equations'}
\begin{cases}
\partial_\tau N + |D| N +  \JapD^{-1} \big( (g-m)\cdot D\partial g\big) = \mathcal G_1,\\
\Delta_{\bar g} h +D(g\cdot D^2 N) = \mathcal G_2,\\
\partial_\tau \beta +|D| \beta +\JapD^{-1}\big(g\cdot Dh\big) + \JapD^{-1} \big( (g-m)\cdot D\partial g\big) = \mathcal G_3, \\
\Delta_{\bar g}  \bar g  + g\cdot D^2\beta = \mathcal G_4,
\end{cases}
\end{equation}

As we mentioned above, we also use a second gauge in this paper, imposed on the connection on the normal bundle 
to the surface $Y({\mathcal M})$. This issue is entirely absent in the co-dimension $1$ case but becomes important in higher 
co-dimensions. We give only a brief description here. The body of the paper treats the general case in detail.
The curvature of the normal bundle $R^\perp$ has the same bilinear wedge structure generated by the second fundamental 
form $k$ as the Riemann curvature $R$ on ${\mathcal M}$. A gauge on that bundle is fixed by a choice of the normal 
frame $e_{\bar A}$. Relative to that frame, the connection is expressed through its connection coefficients 
$$
(D_\alpha e_{\bar A},e_{\bar B}) = \omega_{\alpha \bar A\bar B}.
$$
 We fix the gauge on the normal bundle by requiring that the connection coefficients satisfy the following condition 
 $$
 \bar g^{ij} \bar \nabla_i \omega_{j\bar B}^{\bar A}=-\Delta_{\bar g} |D|^{-1} \omega_{0\bar b}^{\bar A}  + {\mathcal F}_\perp.
 $$ 
Setting the right hand side to zero above leads back to the standard Coulomb gauge. The addition of the first term on the 
right hand side, as was the case in our chosen gauge on ${\mathcal M}$, balances the spatial and time derivatives. 
The ${\mathcal F}_\perp$ term, again as in the case of the gauge on ${\mathcal M}$, renormalizes a problematic ``high-high"
interaction term related to curvature. In this gauge, $\omega$ satisfies an eilliptic-parabolic system of the form 
\begin{align*}
&\pa_\tau \omega_{0\bar B}^{\bar A} + |D|\omega_{0\bar B}^{\bar A}=...\\
&\Delta_{\bar g}\omega_{i\bar B}^{\bar A}=-\pa_i\left(\Delta_{\bar g} |D|^{-1} \omega_{0\bar B}^{\bar A}\right)+...
\end{align*}
\vskip .5pc
Together, the gauges on ${\cal M}$ and on the normal bundle $NY({\cal M})$ constitute what we call in the paper the 
{\it balanced gauge} for the minimal surface equation..
\vskip .5pc

The final word about the gauge and our notations is that {\it all} the space-times norms, e.g. $L^p_t W^{s,q}$, used in the paper are taken with respect to the coordinate system defined by the gauge. Since in all the norms time integration always goes {\it second}, below,
we will drop sub-index $t$ and will simply use the notation $L^p W^{s,q}$.

\vskip .5pc
Let us remark that the process described so far establishes a-priori estimates for solutions to the minimal surface equation that, via a standard bootstrap argument, lead to the existence part of Theorem \ref{thm:Existence'}. The geometric uniqueness of solutions is obtained as a corollary of a uniqueness statement for solutions with the balanced gauge condition imposed. Given two solutions (expressed in the balanced gauge) arising from the same initial data, the differences $(\dot k, \dot g, \dot \omega) = (k^{(1)}-k^{(2)}, g^{(1)}-g^{(2)}, \omega^{(1)}-\omega^{(2)})$ are estimated through a process analogous to that followed for the existence part, but with the norms controlling $(\dot k, \dot g, \dot\omega)$ lying at one level of differentiability below those controlling $(k, g,\omega)$ (a loss of regularity which is usual for quasilinear problems). Dealing with these rougher norms requires using the structure of the high-high interactions in the equations in a more delicate way. 

\vskip .5pc
We finish the introduction with a short discussion of the possible optimality of the results proved in this paper. The analysis here couples the 
Strichartz estimates with a $1/6$ regularity loss for a wave equation on a rough background described by the metric $g$,
with the bilinear estimates for $g$ in terms of the second fundamental form $k$ satisfying the above wave equation. It might be tempting 
to enquire whether the results can be improved by paring the loss coming from the Strichartz estimates. As we discussed earlier, 
however, the losses in Strichartz estimates are sharp for generic non-smooth metrics and the geometric mechanism responsible for 
eliminating these losses is tied to a certain null component of the Ricci tensor of $g$. It turns out that for the minimal surface equation
that relevant component is essentially 
$$
\Box_g \left(|D|^{-1} k\right)^2
$$
which lacks the requisite structure that would allow to control the geometry of null hypersurfaces of $g$ below the regularity considered 
in this paper. It is therefore possible that the regularity exponents provided in this work may be near a threshold for the well-posedness 
theory for the timelike minimal surface equation. At the moment, however, this remains an open question undoubtedly deserving further 
investigation.

\vskip 1pc
{\bf Acknowledgements.} Part of this work was completed while GM was supported by a Clay Research Fellowship. IR acknowledges support through NSF grants DMS-2405167 and a Simons Investigator Award.

\section{Notations and useful relations}

\subsection{The geometry of immersed hypersurfaces}\label{sec:Geometry}
Let $(\mathcal{N}^{n+1},m)$, $n\ge 4$, be a smooth Lorentzian manifold and let $Y:\mathcal{M}^{3+1} \rightarrow \mathcal{N}$ be an immersion. In this paper, we will only be interested in the case when the target manifold $\mathcal N$ is $\mathbb{R}\times \mathbb{T}^3 \times \mathbb{R}^{n-3}$ and the metric $m$ is the standard Minkowski metric.  We will also restrict ourselves to the case when the manifold $\mathcal{M}$ is of the form $I\times \bS^3$, where $I\subseteq \mathbb{R}$ is an interval, and the hypersurface $\Sigma = Y(\mathcal M) \subset \mathcal N$ is timelike, i.\,e.~its induced metric $g$ has Lorentzian signature.

Let $\{ y^A \}_{A=0}^{n}$ be a fixed \emph{Cartesian} coordinate chart on $\mathcal N$; in such a chart, the Minkowski metric $m$ takes the form
\[
m = m_{AB} dy^A dy^B = -(dy^0)^2 + \sum_{I=1}^n (dy^I)^2.
\]
Let also $\{x^{\a}\}_{\a=0}^3$ be a local coordinate chart on $\mathcal{M}$ and $\{ e_{\bar J}\}_{\bar j = 4}^n$ be a frame (not necessarily orthonormal) for the normal bundle $N\Sigma$ of $\Sigma = Y(\mathcal M)$. 

We will denote with  $\Pi : T_{\Sigma} \mathcal{N} \rightarrow T\Sigma$ and $\perpi: T_{\Sigma}\mathcal{N} \rightarrow N\Sigma$ the orthogonal projections on the tangent and normal bundle of $\Sigma$, respectively. Note that the components $\Pi^{\b}_{A}$ of $\Pi$ with respect to the coordinate chart $\{x^{\a}\}_{\a=0}^3$ on $\mathcal{M}$ and the fixed Cartesian coordinate system $\{ y^A \} $ on $\mathcal{N}$ can be expressed as explicit smooth functions of $\{\partial_{\alpha} Y^A\}_{\a, J}$; we will sometimes use the notation $\Pi(\partial Y)$ for $\Pi$ to highlight this fact.
 Similarly, given a frame $\{ e_{\bar{J}} \}_{\bar{J}=4}^{n}$ for the normal bundle $N\Sigma$, the components $(\perpi)^{\bar{J}}_A$ of $\perpi$ can be viewed as smooth functions of $e_{\bar{J}}$.

\begin{remark*}
Throughout this paper, we will always use capital letters to denote indices with respect to the fixed Cartesian coordinate system in the target manifold  $\mathcal{N}$. Capital letters with an overbar will be used to denote indices related to a chosen frame in the normal bundle $N\Sigma$ of $\Sigma$ (which we will sometimes also denote with $NY$). Lower-case letters will be used for indices related to a coordinate chart $(x^0,x^1,x^2,x^3)$ on $\mathcal{M}$; we will adopt the convention that Greek indices take values in $\{0,1,2,3\}$, while Latin indices range over $\{1,2,3\}$. We will also frequently identify coordinate charts on $\mathcal{M}$ with their corresponding push-forward charts on $\Sigma=Y(\mathcal{M})$. 
\end{remark*}

\noindent \textbf{The induced metric $g$ on $\mathcal M$.}
The metric $g$ induced on $\Sigma=Y(\mathcal{M})$ (and, via $Y_*$, on $\mathcal M$) can be expressed in the $x^{\a}$ coordinate chart as follows:
\begin{equation}
g_{\a\b} = m_{AB} \partial_{\a} Y^A \partial_{\b} Y^B \label{metric}
\end{equation}
We will adopt the standard definition for the Christoffel symbols of $g$ on $\Sigma$:
\begin{equation}\label{Christoffel symbols}
\Gamma_{\lambda\mu\nu} \doteq \f12\Big( \partial_{\mu} g_{\lambda\nu} + \partial_{\nu} g_{\lambda\mu} - \partial_{\lambda} g_{\mu\nu}\Big), \quad \Gamma_{\a\b}^\ga \doteq g^{\ga\delta}\Gamma_{\delta\a\b}.
\end{equation}
In terms of the immersion map $Y$, the Christoffel symbols can be expressed as
\begin{equation}\label{Projection Christoffel symbols}
\Gamma^{\a}_{\b\ga} = \Pi^{\a}_B \, \partial_{\b} \partial_{\ga} Y^B,
\end{equation}
while the covariant derivative of a tangential vector field $X\in\Gamma(T\mathcal{M})$ can be expressed as
\begin{equation}\label{Covariant derivative}
\nabla_\a X^\b = \Pi^{\a}_B \, \partial_\a \big( X^\delta \partial_\delta Y^B \big).
\end{equation}

We will adopt the following definition for the \emph{Riemann curvature} tensor $R_{\a\b\ga\delta}$:
\begin{equation*}
R_{\a\b\ga\delta} \doteq g\big( \nabla_{\a} \nabla_{\b} (\partial_{x^{\delta}}) - \nabla_{\b} \nabla_{\a} (\partial_{x^{\delta}}),  \partial_{x^{\ga}} \big).
\end{equation*}
The tensor $R_{\a\b\ga\delta}$ can be computed explicitly in terms of the Christoffel symbols as follows:
\begin{equation}\label{Riemann tensor coordinates}
R_{\a\b\ga\delta} = \partial_{\a} \Gamma_{\ga\b\delta} - \partial_{\b} \Gamma_{\ga\a\delta} -\Gamma_{\lambda \a\ga} \Gamma^{\lambda}_{\b\delta} + \Gamma_{\lambda\b\ga} \Gamma^{\lambda}_{\a\delta}. 
\end{equation}

\smallskip
\noindent \textbf{The connection $\pernabla$ on the normal bundle.}
The Minkowski metric $m$ on $\mathcal N $ induces a natural connection $\pernabla$ on the normal bundle $N\Sigma$; the derivative of a normal vector field $Z \in \Gamma(N\Sigma)$ with respect to this connection satisfies the relation:
\[
\pernabla_{\alpha} Z ^{\bA} = (\perpi)^{\bA}_B \partial_{\a} (Z^{\bB} e_{\bB}^B).
\]
The connection coefficients $\omega$ of $\pernabla$ with respect to the frame $\{ e_{\bar{J}}\}$ are defined by
\begin{equation}\label{Definition connection coefficients}
\omega^{\bar{A}}_{\a \bar{B}} \doteq (\perpi)^{\bA}_B \partial_{\a} e_{\bB}^B.
\end{equation}
The \emph{normal curvature} tensor $R^{\perp }_{\a\b\bA\bB}$ associated to $\pernabla$ is defined by:
\begin{align*}
R^{\perp}_{\a\b\bA\bB} \doteq m \big( \nabla^{\perp}_{\a}\nabla^{\perp}_{\b} e_{\bB} - \nabla^{\perp}_{\b}\nabla^{\perp}_{\a} e_{\bB}, e_{\bA} \big)
\end{align*}
and can be explicitly computed in terms of the connection coefficients via the formula
\begin{equation} \label{Normal curvature coordinates}
(R^{\perp})_{\a\b\hphantom{\bA}\bB}^{\hphantom{\a\b}\bA} = \partial_{\a} \omega_{\b\bB}^{\bA} - \partial_{\b} \omega_{\a\bB}^{\bA} +\omega^{\bA}_{\a \bar{C}} \omega^{\bar{C}}_{\b\bB} - \omega^{\bA}_{\b \bar{C}} \omega^{\bar{C}}_{\a\bB}
\end{equation}
(in the above, we used $m_{\bA\bB} = m(e_{\bA},e_{\bB})$ to raise the third index of $R^\perp$).

\smallskip
\noindent \textbf{The second fundamental form $k$.}
The \emph{second fundamental form} $k$ of $\Sigma$ is a symmetric, vector valued $2-$form on $\Sigma$ taking values in $N\Sigma$; its components with respect to the coordinate chart $\{ x^{\a}\}_{\a=0}^3$ and the frame $\{e_{\bar J}\}_{\bar J=4}^n$ satisfy
\begin{equation}
k_{\a\b}^{\bA} = (\perpi )^{\bA}_B \partial_{\a}\partial_{\b}Y^B. \label{Second fundamental form}
\end{equation}

It will be useful to  express the derivative of the components $e_{\bA}^A$ of the frame vectors $e_{\bA}$ with respect to $\{\partial_{y^A}\}_{A=0}^n$ in terms of the connection form $\omega$ and the second fundamental form $k$ as follows:
\begin{equation}\label{Relation e Omega k}
\partial_\b e_{\bA}^A = \omega^{\bB}_{\b \bA} e_{\bB}^A -  g^{\ga \delta} m(e_{\bA}, e_{\bB} )k^{\bB}_{\b\ga}\partial_{\delta} Y^A.
\end{equation}

The Gauss, Ricci and Codazzi equations take the following form:
\begin{align}
& R_{\a \b \ga \delta}  = m_{\bA \bB} \cdot \big( k^{\bA}_{\a\ga} k^{\bB}_{\b\delta} - k^{\bA}_{\a\delta} k^{\bB}_{\b\ga}\big), \label{Gauss}\\
& R_{\a \b}^{\perp \hphantom{i} \bA \bB} =  g^{\ga \delta} \cdot \big(  k^{\bB}_{\a\ga} k^{\bA}_{\b\delta} - k^{\bA}_{\a\ga} k^{\bB}_{\b\delta} \big), \label{Ricci}\\
& \nabla_{\a} k^{\bB}_{\b\ga} = \nabla_{\b} k^{\bB}_{\a\ga} ,\label{Codazzi}
\end{align}
where  
\[
m_{\bA \bB} \doteq m(e_{\bA}, a_{\bB}),
\]
\[
R_{\a \b}^{\perp \hphantom{i} \bA \bB} \doteq m^{\bA\bar{C}}m^{\bB\bar{D}}R^\perp_{\a\b\bar{C}\bar{D}}
\]
and, through some minor abuse of notation, $\nabla k$ denotes the derivative of $k$ with respect to the natural connection on $N\Sigma \otimes T^*\Sigma \otimes T^*\Sigma$, defined by the relation
\begin{align*}
m\big(\nabla_Z k(X,Y),N \big) = \hphantom{h} &  \nabla_Z \big(m( k(X,Y),n) \big) - m\big( k(\nabla_ZX,Y),N \big)\\
 &  - m\big( k(X,\nabla_ZY),N \big)-m\big( k(X,Y),\pernabla_Z N \big).  
\end{align*}
for all $X,Y,Z\in \Gamma(T\Sigma), N\in \Gamma(N\Sigma)$.

\subsection{Gauge conditions and the $3+1$ decomposition}\label{sec:Gauge}
For the rest of this section, we will assume that the base manifold $\mathcal M$ is of the form $I \times \mathbb T^3$, where $I\subseteq \mathbb{R}$ is an interval. Let $Y: \mathcal M \rightarrow \mathcal{N}$ be an immersion and let $g$ and $k$ be, respectively, the Lorentzian metric and second fundamental form induced by the Minkowski metric $m$ on $\Sigma = Y(\mathcal M)$, as in Section \ref{sec:Geometry}. 

\begin{definition}
A \textbf{gauge} $\mathcal{G}$ for the immersion $Y$ will correspond to a choice of a diffeomorphism $x: \mathcal{M} \rightarrow I\times \mathbb{T}^3$ and a choice of a frame $\{e_{\bar{A}}\}_{\bar{A}=4}^n$ for the normal bundle $NY$ of the image of the immersion.
\end{definition}

\begin{remark*}
Through a slight abuse of the standard terminology, we will frequently refer to the components $(x^0, x^1, x^2, x^3)$ of a diffeomorphism $x: \mathcal M \rightarrow I \times \mathbb T^3$ with respect to the product coordinates on $I \times \mathbb T^3$ as \textbf{coordinate functions}. Note that, in such a case, $x^0$ takes values in $\mathbb R$, while $x^1, x^2$ and $x^3$ take values in $\mathbb S^1 = \mathbb R / \mathbb Z$.
\end{remark*}

For the rest of this paper, we will only consider diffeomorphisms $x: \mathcal M \rightarrow I \times \mathbb T^3$ for which the slices $\overline \Sigma_\tau \doteq \{ x^0 = \tau\}$ are strictly spacelike hypersurfaces of $(\mathcal M, g)$. Let $x=(x^0,x^1,x^2,x^3)$ be such a diffeomorphism. 
\medskip

\noindent \textbf{The $3+1$ decomposition.}
We will denote with $\bar{g}$ the induced \emph{Riemannian} metric on $\overline \Sigma_\tau$. Fixing a time orientation on $\mathcal{M} \simeq I \times \mathbb T^3$ such that $\partial_{x^0}$ is future directed, we will denote with $\hat n$ the future directed, \emph{unit} timelike vector field on $\mathcal M$ which is normal to $\overline \Sigma_\tau$ for each $\tau$. We will also define the second fundamental form $h$ of the slices $\overline \Sigma_\tau$ by the relation
\begin{equation}\label{Definition h}
h(V,W) = g(\hat n,\nabla_V W) \quad \text{for all }V,W \in \Gamma(T \overline\Sigma_\tau).
\end{equation}

We will define the \emph{lapse} function $N>0$ and the \emph{shift} vector $\b \in \Gamma(T \overline \Sigma_\tau)$ by the relation
\begin{equation}\label{Relation lapse normal}
\partial_{x^0} = N \hat n + \beta.
\end{equation}
Note that, along each slice $\overline \Sigma_\tau$, the spacetime metric $g$ can be decomposed as
\begin{equation}
g = -N^2 (dx^0)^2 +\bar{g}_{ij}(dx^i + \b^i dx^0) (dx^j + \b^j dx^0). \label{Metric decomposition}
\end{equation}
In particular, the components $g_{\a\b}$ and $g^{\a\b}$ of $g$ and its inverse in the $(x^0, x^1, x^2, x^3)$ chart take the form:
\begin{align*}
g_{00} & =-N^2+|\b|_{\bar g}^2, \quad g_{0i} = \bar g_{il}\b^l, \quad  \hphantom{a\,\,} g_{ij} = \bar g_{ij},\\
g^{00} &=  -N^{-2}, \quad \hphantom{+|\b|\,\,} g^{0i} = N^{-2} \b^i, \quad g^{ij} =  \bar g^{ij} - N^{-2} \b^i \b^j,
\end{align*}
where $\bar g^{ij}$ are the components of the inverse matrix with elements $\bar g_{ij}$. Moreover, the Christoffel symbols of the metric $g$ (see \eqref{Christoffel symbols}) in the $(x^0,x^1,x^2,x^3)$ coordinate chart can be computed explicitly in terms of $(N, \b,\bar g)$ and $h$:
\begin{align}\label{Christoffel symbols 3+1}
\Gamma_{000} & = \f12 \partial_0 (-N^2 + |\b|^2_{\bar g}), \quad \quad\quad \quad\quad \quad  \Gamma_{00i}  = \f12 \partial_i (-N^2 + |\b|^2_{\bar g}),\\
 \Gamma_{i00} &= \partial_0 (\bar g_{il}\b^l) - \f12 \partial_i (-N^2 + |\b|^2_{\bar g}), \,\quad \Gamma_{0ij} = \f12(\partial_i \b_j + \partial_j \b_i - \partial_0 \bar g_{ij}), \nonumber  \\
 \Gamma_{i0j} &=  \f12(\partial_j \b_i - \partial_i \b_j + \partial_0 \bar g_{ij}), \quad \quad \quad \,\,\,\,  \Gamma_{ijk}  = \bar \Gamma_{ijk}.    \nonumber 
\end{align}
In what follows, we will raise and lower indices of tensors on $\overline \Sigma_\tau$ using $\bar g$.

 The Gauss--Codazzi equations yield the following relations between the spacetime curvature tensor $R$ and the pair $(\bar g, h)|_{\bar\Sigma_{\tau}}$:
\begin{align}
& \bar{R}_{ijkl}  - \big( h_{il} h_{jk} -  h_{ik} h_{jl} \big) =  R_{ijkl}, \label{Gauss 3+1}\\
& \bar{\nabla}_{i} h_{jk} - \bar{\nabla}_{j} h_{ik} = R_{ij\a k} \hat{n}^{\a},\label{Codazzi 3+1}
\end{align}
where $\bar{R}$ denotes the spatial Riemann curvature tensor associated to the metric $\bar g$.

\medskip

\noindent{\textbf{The first variation formulas for $(\bar g,h)$.}}
We can identify $\overline \Sigma_\tau$ with $\mathbb T^3$ via the coordinate chart  $\bar{x} = (x^1, x^2, x^3)$. Through this identification, we can view $(\bar{g}, h)|_{\overline\Sigma_\tau}$ as an one parameter family of tensors on $\mathbb T^3$; the variation of $(\bar{g},h)$ with respect to $\tau$ can be computed via the following structure equations:
\begin{align}
\partial_\tau \bar{g}_{ij} \doteq \mathcal L_{\partial_{x^0}} \bar{g}_{ij} &= -2N h_{ij} + \bar{\nabla}_i \b_j + \bar{\nabla}_j \b_i, \label{Variation metric}\\
\partial_\tau h_{ij} \doteq \mathcal{L}_{\partial_{x^0}} h_{ij} &= -\bar{\nabla}_i \bar{\nabla}_j N -N h_{ik}h^k_j -N R_{i\alpha j\beta}\hat n^\a \hat n^\b +h_{kj}\bar{\nabla}_i \b^k + h_{ki}\bar{\nabla}_j \b^k + \b^k \bar{\nabla}_k h_{ij}, \label{Variation second fundamental form}
\end{align}
where $\bar{\nabla}$ denotes the covariant derivative associated to $\bar{g}$. As a result, the variation in $\tau$ of the Christoffel symbols $\bar\Gamma^k_{ij}$ of $\bar g|_{\bar\Sigma_{\tau}}$ can be expressed as follows:
\begin{equation}\label{Variation Christoffel symbols}
\partial_{\tau} \bar\Gamma^k_{ij} = -\big(\bar\nabla_i   (N h_{j}^k) + \bar\nabla_j (N h_i^k)-\bar\nabla^k (Nh_{ij})\big) + \bar\nabla_i \bar{\nabla}_j \b^k + \bar g^{kl}\bar R_{imjl} \b^m.
\end{equation}

\begin{remark*}
Throughout this paper, we will frequently use the schematic notation $\partial g$ to denote either $\partial_\kappa g_{\lambda\mu}$, $\kappa,\lambda,\mu\in \{0,1,2,3\}$ or any of the terms $\partial_{\kappa} N$, $\partial_{\kappa} \b^i$, $\partial_{\kappa} \bar g_{ij}$, $\kappa \in \{0,1,2,3\}$, $i,j \in \{1,2,3\}$, possibly multiplied with an algebraic function of $g_{\kappa\lambda}$. 
\end{remark*}

\subsection{The initial value problem for the minimal surface equation}\label{subsec:IVP}
The hypersurface $\Sigma$ is called \emph{minimal} if the second fundamental form is traceless, i.e.~it satisfies 
\begin{equation}
g^{\a\b} k^{\bA}_{\a\b} = 0.  \label{Minimal surface equation}
\end{equation}
The above equation is of hyperbolic nature: Differentiating the Codazzi equations \eqref{Codazzi} and taking the trace with respect to $g$, we obtain the following wave-type equation for the second fundamental form of a minimal hypersurface:
\begin{equation} \label{Covariant wave equation k}
g^{\mu\nu} \nabla_{\mu}\nabla_{\nu} k ^{\bA}_{\a\b} =  R_{\mu \a \hphantom{\ga}}^{\hphantom{\mu \a }\ga \mu} k^{\bA}_{\ga \b}+R_{\hphantom{\mu} \a \hphantom{\ga} \b}^{\mu \hphantom{ \a }\ga} k^{\bA}_{\mu \ga} +(\perR)_{\hphantom{\mu} \a \hphantom{\bA} \bB}^{\mu \hphantom{\a} \bA} k^{\bB}_{\mu \b}. 
\end{equation} 
Equation \eqref{Covariant wave equation k} yields the following (schematic) equation for each component $k ^{\bA}_{\a\b}$ of $k$:
\[
\square_g k ^{\bA}_{\a\b} = g\cdot \big( \Gamma + \omega  \big) \partial k + g\cdot \big( \partial \Gamma + \Gamma \cdot \Gamma + \partial \omega + \omega \cdot \omega \big) k.
\]

We will be interested in the study of the initial value problem for equation \eqref{Minimal surface equation}. To this end, we will adopt the following definition of an initial data set for \eqref{Minimal surface equation}:

\begin{definition}\label{def: Initial data}
An initial data set for the minimal surface equation \eqref{Minimal surface equation} consists of an immersion $\bY: \mathbb T^3 \rightarrow \mathcal{N}$ and a map $\bar{n}: \mathbb T^3 \rightarrow \mathbb R^{n+1}$ such that
\begin{itemize}
\item $\bY (\mathbb T^3)$ is spacelike with respect to the Minkowski metric $m$ on $\mathcal{N}$,
\item For any $x\in \mathbb T^3$, $\bar{n}(x)$ is a unit timelike vector which is normal to $\bY(\mathbb T^3)$ at $\bY(x)$. 
\end{itemize}
\end{definition}
Note that the above conditions on the pair $(\bY, \bn)$ can be equivalently expressed in a local coordinate chart $(x^1, x^2, x^3)$ on $\mathbb T^3$ as
\[
m_{AB}(\bY(x)) \bn^A(x) \bn^B(x) = -1 
\]
and
\[
m_{AB}(\bY(x)) \partial_i \bY^A(x) \bn^B(x) = 0
\]
for all $i=1,2,3 $ and $x\in \mathbb T^3$.

\begin{definition}\label{def: Development}
Let $(\bY, \bn)$ be an initial data set for \eqref{Minimal surface equation} and let $T>0$. A map $Y:[0,T] \times \mathbb T^3  \rightarrow \mathcal{N}$ will be called a \textbf{development} of $(\bY, \bn)$ if
\begin{itemize}
\item $Y$ is a timelike immersion satisfying equation \eqref{Minimal surface equation},
\item $Y(0,x)=\bY(x)$ for any $x\in \mathbb T^3$,
\item $\partial_{x^0} Y(0,x) \in \mathrm{span}\big\{\bn(x), T_x \bY(\mathbb T^3)\big\}$ and $m \big(\partial_{x^0} Y(0,x), \bn(x)\big) <0$ for any $x\in \mathbb T^3$.\footnote{Here, $x^0$ corresponds to the variable ranging over the interval $[0,T]$.}
\item $\{0\} \times \mathbb T^3$ is a Cauchy hypersurface of $[0,T] \times \mathbb T^3$ with respect to the pullback metric $Y_* g$.
\end{itemize}
\end{definition}

Let us notice that the above definition of a development of an initial data set is highly non-geometric: For any non-trivial diffeomorphism $\tilde x$ of $[0,T] \times \mathbb T^3$ fixing $\{0\}\times \mathbb T^3$, the map $\tilde{Y}=Y\circ \tilde x$ will parametrize the same hypersurface in $\mathcal{N}$ as $Y$; however $\tilde{Y}$ and $Y$ are different developments of $(\bY, \bn)$ according to Definition \ref{def: Development}. We will, therefore, say that two developments $Y, \tilde{Y}$ of the same initial data set coincide \textbf{geometrically} if they are related by a diffeomorphism as above.

As a trivial example, consider the case of the flat initial data set $(\bY_0, \bn_0)$ with $\mathcal{N}=\mathbb{R}\times \mathbb{T}^3 \times \mathbb{R}^{n-3}$,
\begin{align}
\bY_0(x^1, x^2, x^3) &\doteq (0; x^1, x^2, x^3; 0, \ldots, 0),\label{Flat data}\\
\bn_0(x^1, x^2, x^3) &\doteq (1; 0,0,0; 0, \ldots, 0). \nonumber
\end{align}
The (unique, up to diffeomorphisms) development of $(\bY_0, \bn_0)$ is the map $Y_0: \mathbb{R}\times \mathbb{T}^3 \rightarrow \mathbb R \times \mathbb T \times \mathbb{R}^{n-3}$ given by
\[
Y_0(x^0, x^1, x^2, x^3) \doteq (x^0; x^1, x^2, x^3; 0, \ldots, 0).
\]

\subsection{Mixed Sobolev norms} \label{sec:Mixed Norms} 
Given a function $f: [0,T] \times \mathbb{T}^3 \rightarrow \mathbb{R}^N$, we will define its mixed $L^p W^{s,q}$ norm in the usual way: For any $p,q\in [1, +\infty]$ and $s\in \mathbb{R}$,
\[
\| f \|_{L^p W^{s,q}} \doteq \Big( \int_0^T \| f(x^0,\cdot) \|^p_{W^{s,q}(\mathbb{T}^3)}\,dx^0 \Big)^{\f1p}.
\]
For $q=2$, we will denote
\[
\| \cdot \|_{L^p H^s} \doteq \| \cdot \|_{L^p W^{s,2}}.
\]
We will also use the following notation regarding mixed norms of higher order derivatives of $f$:
\[
\| \partial^k f \|_{L^p W^{s,q}} \doteq \sum_{|\lambda|=k} \Big( \int_I \| \partial_{x^0}^{\lambda_0}\cdots \partial_{x^3}^{\lambda_3} f(x^0,\cdot) \|^p_{W^{s,q}(\mathbb{T}^3)}\,dx^0 \Big)^{\f1p}.
\]

Let $Y: \mathcal M \rightarrow \mathcal{N}$ be an immersion as in Section \ref{sec:Geometry} with $\mathcal M \simeq [0,T] \times \mathbb T^3$ and $\mathcal{N}=\mathbb{R}\times \mathbb{T}^3 \times \mathbb{R}^{n-3}$. Let also $\mathcal G = (x^{\a}, \{e_{\bA}\}_{\bA=4}^n )$ be a gauge for $Y$ (see Section \ref{sec:Gauge} for the relevant definitions). For any $NY$-valued tensor field $Q$, we will define the $L^p W^{s,q}$ norm of $Q$ in the gauge $\mathcal G$ componentwise, i.e.
\[
	\| Q \|_{L^p W^{s,q}[\mathcal G]} \doteq \sum_{\substack{\a_1, \ldots, \a_l,\\ \b_1, \ldots, \b_m }} \| Q^{\a_1 \ldots \a_l; \bar{A}}_{\b_1 \ldots \b_m}\|_{L^p W^{s,q}},
\]
where $Q^{\a_1 \ldots \a_l; \bar{A}}_{\b_1 \ldots \b_m}$ are the components of $Q$ with respect to the gauge $\mathcal{G}$.  Similarly, for higher order derivatives of $Q$:
\[
\| \partial^k Q \|_{L^p W^{s,q}[\mathcal G]} \doteq \sum_{\substack{\a_1, \ldots, \a_l,\\ \b_1, \ldots, \b_m }} \| \partial^k Q^{\a_1 \ldots \a_l; \bar{A}}_{\b_1 \ldots \b_m}\|_{L^p W^{s,q}}.
\]

In a similar way, we will define the mixed $L^p B^s_{q,r}$ Besov norm of functions and tensor fields as above using the definition \eqref{Besov norm} of $B^s_{q,r}$ on $\mathbb T^3$; see Section \ref{sec:Littlewood Paley} for more details.

\medskip
\noindent \textbf{Remark 1.} 
We will also adopt the same definition of $\| Q \|_{L^p W^{s,q}}$ when $Q$ is a standard tensor field (i.e.~taking values in $\mathbb{R}$ rather than $NY$) or simply a scalar function on $\mathcal M$. In this case, the choice of the normal frame $e_{\bA}$ for the normal bundle $NY$ is irrelevant for the definition of the norm.

\medskip
\noindent \textbf{Remark 2.} 
\emph{ A change of gauge $\mathcal G \rightarrow \mathcal G'$ has the effect of changing the expressions for the components $Q^{\a_1 \ldots \a_l; \bar{A}}_{\b_1 \ldots \b_m}$ of a tensor $Q$ as well as the foliation $\{x^0 = \text{const}\}$ with respect to which the mixed norm $L^p W^{s,q} [\mathcal G ]$ is defined. Due to the latter effect, we will in general not be able to compare $\| Q \|_{L^p W^{s,q} [\mathcal G ]}$ and $\| Q \|_{L^p W^{s,q} [\mathcal G']}$ when $p\neq q$ without some loss of derivatives or without exploiting some additional structure for $Q$ (even when $Q$ is a scalar function).  }

\medskip
\noindent \textbf{Remark 3.}
\emph{ In cases when the choice of a gauge $\mathcal G$ is clear from the context, we will frequently drop $\mathcal G$ from the notation for mixed norms, simply writing $\| Q\|_{L^p W^{s,q}}$ in place of $\| Q\|_{L^p W^{s,q} [\mathcal G ]}$. In particular, this will be the case in Sections \ref{sec:Balanced gauge}--\ref{sec:Persistence of regularity}, where we will only be interested in estimates with respect to a fixed \emph{balanced gauge} (for the definition of this gauge, see Section \ref{sec:Balanced gauge}). We will reinstate $\mathcal G$ in our norm notation in section \ref{sec:Uniqueness}, where it will be necessary for us to compare estimates for geometric quantities expressed in different gauges.
}


\subsection{Frequency localization, Besov norms and product decompositions}\label{sec:Littlewood Paley} 
In this section, we will set up the notation that will be used throughout this paper regarding the frequency decomposition with respect to the space variables of functions defined on $[0,T] \times \mathbb T^3$. We will then proceed to introduce two operators defined on the frequency domain: The operators $\mathcal{T}^{(i)}_{j}$, $j=1,2,3$, will act as an inverse to the divergence operator at frequencies $|\xi| \sim 2^j$, while the multilinear operator $\mathbb P^{\natural}[\cdot, \cdot, \cdot]$ will serve to isolate certain combinations of frequencies appearing in the products in the right hand sides of the relations \eqref{Gauss} and \eqref{Ricci} for $R$ and $R^{\perp}$, respectively. We will also introduce the extension operator $\mathbb E$.

\medskip
\noindent \textbf{The Littlewood--Paley decomposition.} For a function $f:[0,T] \times \mathbb{T}^3 \rightarrow \mathbb{C}$, we will denote with $P_j f$ the Littlewood--Paley projection at frequency $\xi \sim 2^j$ with respect to the \emph{spatial} variables. Thus, denoting with $\hat{\cdot}$ the Fourier transform on $\mathbb{T}^3$, we have
\[
\widehat{P_j f}(x^0, \xi) = \chi_j (\xi) \hat{f}(x^0, \xi), \hphantom{and} \xi \in \mathbb{Z}^3 
\]
where 
\begin{equation}\label{Littlewood Paley function}
\chi_j(\xi) = \begin{cases} \phi(2^{-j}|\xi|)-\phi(2^{-j+1} |\xi|), \hphantom{and} j>0,\\ \phi(|\xi|), \hphantom{and} j=0 \end{cases}
\end{equation}
 for some fixed function $\phi \in C_0^{\infty}[0,+\infty)$ with $\phi(z)=1$ for $z\le \f32$ and $\phi(z)=0$ for $z\ge \f74$. Note that $\chi_j(\xi) = 1$ for $|\xi| \in [\f78 2^j, \f32 2^j]$ and $\chi_j (\xi) = 0$ for $|\xi|<  \f32 2^{j-1}$ and $|\xi|> \f78 2^{j+1}$. We will also denote
\[
P_{\le j } f \doteq \sum_{j' \le j} P_{j'} f.
\]

Let $Y: \mathcal M \rightarrow \mathcal N$ be an immersion as in the previous sections, where $\mathcal M \simeq [0,T] \times \mathbb T^3$. Given a gauge $\mathcal G = (x, \{ e_{\bA}\}_{\bA} )$ for $Y$, we will define the Littlewood--Paley projection $P_j Q$ of any smooth $NY$-valued tensor $Q$ on $[0,T]\times \mathbb T^3$ componentwise, i.e.
\[
(P_j Q)_{\a_1 \ldots \a_m}^{\b_1\ldots \b_k ;\bA_1 \ldots \bA_l} \doteq P_j \big( Q _{\a_1 \ldots \a_m}^{\b_1\ldots \b_k ;\bA_1 \ldots \bA_l} \big),
\]
where the components of $Q$ are considered with respect to the coordinates $x^\a$ and the frame $\{ e_{\bA}\}_{\bA=4}^n$ associated to $\mathcal{G}$.

\medskip
\noindent \textbf{Besov norms.} We will adopt the standard definition of a Besov norm for functions $f:\mathbb T^3 \rightarrow \mathbb C$: For any $s\in \mathbb R$, $p,r\in[1,+\infty]$:
\begin{equation}\label{Besov norm}
\| f \|_{B^s_{p,r}} \doteq \Big\|\Big\{j \rightarrow 2^{sj}\|P_j f\|_{L^p}\Big\}\Big\|_{\ell^r}.
\end{equation}
We will also define mixed Besov norms of the form $\|\cdot \|_{L^p B^s_{q,r}}$ for functions and tensors on $[0,T]\times \mathbb T^3$ in a similar way as we did in Section \ref{sec:Mixed Norms} for the corresponding Sobolev-type norms.

\medskip
\noindent \textbf{Frequency decomposition of products.} We will perform the high-low frequency decomposition of a product of functions $f,h: [0,T] \times \mathbb{T}^3 \rightarrow \mathbb{C}$ in the usual way:
\begin{equation}\label{HL decomposition}
P_j (f \cdot h) = HL_j (f,h) + LH_j (f,h) + HH_j (f,h),
\end{equation}
where
\begin{align*}
HL_j (f,h) & = P_j \Big( \sum_{j'=j-2}^{j+2}, \sum_{j''\le j'+2} P_{j'} f \cdot P_{j''} h\Big),\\
LH_j (f,h) & = P_j \Big( \sum_{j'< j-2} \, \sum_{j''=j-2}^{j+2} P_{j'} f \cdot P_{j''} h\Big),\\
HH_j (f,h) & = P_j \Big( \sum_{j'>j+2} \, \sum_{l=-2}^{2} P_{j'} f \cdot P_{j'+l} h\Big).
\end{align*}
We will frequently use the following schematic notation for the decomposition \eqref{HL decomposition}:
\[
P_j (f \cdot h) \approx P_j f \cdot P_{\le j} h + P_{\le j} f \cdot P_j h + \sum_{j'>j} P_j \big( P_{j'} f \cdot P_{j'} h\big). 
\]

\medskip
\noindent \textbf{The $\mathcal T^{(i)}_{j}$ operators.} We will define the Riesz-type operators $\mathcal{T}^{(i)}_{j}$ and $\mathcal{T}^{(i)}$ as follows:

\begin{definition}\label{def:T multiplier}
For $i=1,2,3$ and any $j\ge 0$, we will define the multiplier operator $\mathcal{T}^{(i)}_{j}$ acting on functions on $f: [0,T]\times \mathbb{T}^3 \rightarrow \mathbb C$ as follows:
\begin{equation}\label{T multiplier}
\widehat{\mathcal{T}^{(i)}_{j} f}(t,\xi) = T^{(i)}_j (\xi) \cdot \hat{f}(t,\xi) \doteq  \frac{\xi^i}{|\xi|^2} \tilde{\chi}(\f{\xi}{2^{j}}) \cdot \hat{f}(t,\xi)
\end{equation}
where $\hat{\cdot}$ denotes the Fourier transform with respect to the spatial variables and
\[
\tilde\chi (s) \doteq \tilde\phi (\f s4) - \tilde\phi(4s),
\]
where $\tilde{\phi}:[0,+\infty) \rightarrow [0,1]$ is a smooth function supported in $[0,2]$ such that $\tilde{\phi} \equiv 1$ on $[0, 1]$. 

Furthermore, we will define for any $i=1,2,3$
\[
\mathcal{T}^{(i)} \doteq \sum_{j\ge 0} \mathcal{T}^{(i)}_j.
\]
\end{definition}

Note that the operators $\mathcal{T}^{(i)}_{j}$ are (renormalized) Mikhlin multipliers, since
\begin{equation}\label{Mikhlin property}
\big\| \partial^l_\xi T^{(i)}_{j} (\cdot) \big\|_{L^{\infty}_\xi} \lesssim_{l} 2^{-j-l} \hphantom{a} \text{ uniformly in } j \text{ for any } l \in \mathbb N.
\end{equation}
In particular, the operators $\mathcal{T}^{(i)}_{j}$ are of differential order $-1$. Moreover, the operators $\mathcal{T}^{(i)}_{j}$ satisfy the following divergence identity for any smooth function $f:[0,T]\times \mathbb{T}^3 \rightarrow \mathbb{R}$ and any frequency index $j\ge 1$:
\begin{equation}\label{Identity derivative 2}
\partial_i \Big( \mathcal{T}^{(i)}_{j} P_{j} f \Big) = P_{j} f
\end{equation}
(in the above, summation over $i=1,2,3$ is implicitly assumed).

\medskip
\noindent \textbf{The LHH projection operator $\mathbb P^{\natural}$.} 
The following operator will be used to ``isolate'' the ``low-high-high'' interactions in trilinear expressions:

\begin{definition}\label{def:LHH}
Given smooth functions $h, f_1, f_2: [0,T] \times \mathbb T^3 \rightarrow \mathbb C$, we will define the function $\mathbb P^{\natural} [h,f_1,f_2] : [0,T] \times \mathbb T^3 \rightarrow \mathbb R$ as follows:
\[
\mathbb P^{\natural} [h,f_1,f_2] \doteq \sum_{j}\sum_{\substack{j'>j-2,\\|j''-j'|\le 2}} (\check\chi_{j} * h) (\check\chi_{j'} * f_1)(\check\chi_{j''} * f_2),
\]
where $\check\cdot$ and $*$ denote, respectively, the inverse Fourier transform and convolution operators on $\mathbb T^3$ and $\chi_j$ denotes the Littlewood--Paley multiplier function \eqref{Littlewood Paley function}.
\end{definition}

\noindent Note that, for any $j$, we can express the $P_j$-projection of the above sum schematically as follows:
\[
P_j \mathbb P^{\natural}[h,f_1,f_2] \approx \sum_{\bar j} \sum_{j'>\max\{j,\bar j\}} P_j \big( P_{\bar j} h \cdot P_{j'} f_1 \cdot P_{j'} f_2 \big)
\]

\medskip
\noindent \textbf{The extension operator in the time variable.}
Let $\Psi : \mathbb R \rightarrow \mathbb R$ be a smooth function supported in $[-1,1]$ such that $\Psi \equiv 1$ on $[0,\f12]$. We will also set
\[
\Psi_j (t) \doteq \Psi(2^j t).
\]

\begin{definition}\label{def:Extension operator}
For any smooth function $h: \mathbb T^3 \rightarrow \mathbb R$, we will define the function $\mathbb E h : [0,T] \times \mathbb T^3 \rightarrow \mathbb R$ by the relation:
\begin{equation}\label{Extension operator}
\mathbb E h(t,x) \doteq \sum_j \Psi_j(t) \cdot P_j h(x).
\end{equation}
\end{definition}

\begin{remark*}
Note that 
\[
\mathbb E h(0,x) = h(x) \quad \text{and}\quad 
\partial_0^k \mathbb E h(0,x) =0 \quad \text{for all } k\ge 1.
\]
\end{remark*}

The extension operator  $\mathbb E$ satisfies the following estimates:

\begin{lemma}\label{lem:Estimates time operators}
For any smooth function $h:\mathbb T^3 \rightarrow \mathbb R$ and any $j,k\in \mathbb N$ and $p,q\in[1,+\infty]$, we can estimate
\begin{equation}\label{Extension bound}
\| \partial_0 ^k P_j (\mathbb E h) \|_{L^p L^q} \le C_k 2^{(k-\f1p)j} \|h\|_{L^q}
\end{equation}
for some constant $C_k>0$ depending only on $k$.
\end{lemma}

\begin{proof}
The bound \eqref{Extension bound} follows directly from the expression
\[
P_j (\mathbb E h) (t,x) = \sum_{j'=j-1}^{j+1} \Psi_{j'} (t) \cdot P_j P_{j'} h(x).
\]
\end{proof}

\medskip
\noindent \textbf{First order pseudo-differential operators in the spatial variables.} We will use $|D|$ and $\JapD$ to denote the Fourier multiplier operators defined by:
\begin{equation}\label{Japanese bracket}
\widehat{|D|f}(\xi) \doteq |\xi| \hat f(\xi) \quad \text{and} \quad \widehat{\JapD f}(\xi) \doteq (1+|\xi|^2)^{\f12} \hat f(\xi) \quad \text{for any } \, f\in C^\infty(\mathbb T^3).
\end{equation}
\medskip
\noindent \textbf{Remark.} Throughout this work, we will frequently encounter relations employing a schematic notation, in which the precise algebraic structure (e.g.~tensorial indices, direction of differentiation etc.) of the various terms will not be systematically kept track of. In such instances,  we will often use the schematic notation $D f$ for a given sufficiently regular $f:\mathbb T^3\rightarrow \mathbb C$ to denote any expression obtained from $f$ by applying a pseudo-differential operator of order $1$, namely an operator $f(x)\rightarrow \sum_{\xi\in \mathbb Z^3} e^{i2\pi x\xi} a(x,\xi) \hat f(\xi) $ with symbol $a: \mathbb T^3 \times \mathbb R^3 \rightarrow \mathbb C$  satisfying
\[
\big| \nabla_x^{\ell_1} \nabla_y^{\ell_2} a(x,y)\big| \lesssim_{\ell_1, \ell_2} (1+|y|)^{1-\ell_2} \quad \text{for all }\, \ell_1, \ell_2 \in \mathbb N. 
\]

\subsection{Notations on inequalities and constants}
Throughout this paper, we will frequently make use of the $\lesssim$ notation: For any two functions $h_1, h_2: V \rightarrow \mathbb R$, the expression $h_1 \lesssim h_2$ will be used to denote an inequality of the form $h_1 \le C \cdot h_2$ for some constant $C>0$ depending only on the parameter $s$ in Theorems \ref{thm:Existence} and \ref{thm:Uniqueness}. In expressions extending over multiple lines where the symbol $\lesssim$ appears repeatedly, the constant $C$ is allowed to change from line to line. If $p$ is any parameter of the problems considered, we will use $\lesssim_{p}$ for  inequalities where the implicit constant is allowed to also depend on the parameter $p$. Finally, $h_1 \sim h_2$ will be used to denote that $h_1 \lesssim h_2$ and $h_2 \lesssim h_1$.

\section{The main results}
In this section, we will state the main results obtained in this paper regarding the existence and uniqueness of solutions to the minimal surface equation \eqref{Minimal surface equation} for rough initial data.  

Before stating our results, we will need to introduce a few small parameters associated to the Sobolev exponent $s>\f52+\f16$:
\begin{definition}\label{def:Small constants}
For any $s>\f52+\f16$, we will fix a constant $\delta_0=\delta_0(s)>0$ so that 
\[
0< \delta_0 < \min\big\{ \f1{10}(s - \f52 -\f16), \f1{100}\big\}.
\]
and we will set
\[
s_0 \doteq s-\f52-\f16 -\delta_0 >0, \hphantom{and} s_1 \doteq s-\f52-\f16 - 2\delta_0>0.
\]
Note that
\[
8\delta_0 < s_1 <s_0.
\]
\end{definition}

Our first result establishes the existence of developments for initial data $(\bY, \bn)$ on $\mathbb T^3$ which are small perturbations of the flat initial data in the space $H^{s}\times H^{s-1}$ for any $s>\f52+\f16$ (see Section \ref{subsec:IVP} for the relevant definitions):

\begin{theorem}[Existence]\label{thm:Existence}
Let $s>\f52+\f16$. There exist constants $0 < \epsilon <1$ and $C>1$ depending only on $s$ such that, for any initial data pair $ \bY: \mathbb T^3 \rightarrow \mathcal N = \mathbb R \times \mathbb T^3 \times \mathbb R^{n-3}$, $\bn: \mathbb T^3 \rightarrow \mathbb R^{n+1}$ satisfying
\begin{equation}\label{D}
\mathcal{D} \doteq \|\bY-\bY_0\|_{H^s}+\|\bn-\bn_0\|_{H^{s-1}} < \epsilon
\end{equation}
(where $(\bY_0,\bn_0)$ is the flat initial data set \eqref{Flat data}), there exists a development $Y:[0,1]\times \mathbb{T}^3 \rightarrow \mathcal{N}$ and a frame $\{ e_{\bA}\}_{\bA=4}^n$ for the normal bundle $NY$ such that, in the gauge $\mathcal G \doteq (\text{Id}, e)$, we have 
\begin{equation}\label{Theorem embedding}
\sum_{l=0}^2 \| \partial^l (Y - Y_0) \|_{L^{\infty}H^{s-l}[\mathcal G]} + \| \partial^2 Y \|_{L^4 W^{\f1{12},4}[\mathcal G]} +\|\partial^2 Y\|_{L^{\f72} L^{\f{14}3}[\mathcal G]}\le C \mathcal{D},
\end{equation}
\begin{equation}\label{Theorem k}
\sum_{l=0}^2\Big( \|\partial^l k\|_{L^\infty H^{s-2-l}[\mathcal G]} + \|\partial^l k\|_{L^2 W^{-\f12-l+s_0,\infty}[\mathcal G]} \Big)\le C \mathcal D,
\end{equation}
\begin{equation}\label{Theorem metric}
\| \partial g \|_{L^1 W^{\f14\delta_0,\infty}[\mathcal G]} +\sum_{l=1}^2\Big( \| \partial^l g \|_{L^2 H^{2-l+\f16+s_1}[\mathcal G]} +\|\partial^l g\|_{L^{\f74} W^{2-l+s_1,\f73}[\mathcal G]}+ \|  \partial^l g \|_{L^{\infty} H^{s-1-l}[\mathcal G]}\Big) \le C \mathcal{D}
\end{equation}
and
\begin{align}\label{Theorem connection coefficients}
\| \omega \|_{L^1 W^{\f14\delta_0,\infty}[\mathcal G]} + & \sum_{l=0}^1\Big(\| \partial^l \omega \|_{L^2 H^{1-l+\f16+s_1}[\mathcal G]} +\|\partial^l \omega\|_{L^{\f74} W^{1-l+s_1,\f73}[\mathcal G]} + \|  \partial^l \omega \|_{L^{\infty} H^{s-2-l}[\mathcal G]}\Big)\\
&+\sum_{l=0}^1\Big( \| \partial^{1+l} e \|_{L^2 H^{1-l+\f16+s_1}[\mathcal G]} +\|\partial^{1+l} e\|_{L^{\f74} W^{1-l+s_1,\f73}[\mathcal G]}+ \|  \partial^{1+l} e \|_{L^{\infty} H^{s-2-l}[\mathcal G]}\Big) \le C \mathcal{D}     \nonumber
\end{align}
In the above, $g$ and $\omega$ are, respectively, the induced metric on $[0,1]\times \mathbb T^3$ and connection 1-form for $NY$, while $\delta_0, s_0$ and $s_1$ are defined in terms of $s$ according to Definition \ref{def:Small constants}. 
\end{theorem}

\noindent The development $Y: [0,1] \times \mathbb T^3 \rightarrow \mathcal N$ in Theorem \ref{thm:Existence} will be constructed in a way so that the gauge $\mathcal G$ will satisfy the balanced gauge conditions of Section \ref{sec:Balanced gauge}. The proof of Theorem \ref{thm:Existence} will occupy Sections \ref{sec:Continuity}--\ref{sec:Persistence of regularity}.

Our second result will establish the geometric uniqueness of the developments presented in Theorem \ref{thm:Existence}:

\begin{theorem}[Geometric uniqueness]\label{thm:Uniqueness}
Let  $s>\f52+\f16$ and let $\bY, \bn : \mathbb{T}^{3} \rightarrow \mathcal{N}$ be an initial data pair for equation \eqref{Minimal surface equation} satisfying 
\begin{equation}
 \|\bY-\bY_0\|_{H^s}+\|\bn-\bn_0\|_{H^{s-1}} < +\infty. \label{Finite initial data uniqueness}
\end{equation}
Let also $Y: [0,T] \times \mathbb{T}^3 \rightarrow \mathcal{N}$ and $Y^\prime : [0,T'] \times \mathbb{T}^3 \rightarrow \mathcal{N}$ be two developments of $(\bY, \bn)$ and $\{ e_{\bA}\}_{\bA=4}^n$, $\{ e'_{\bA}\}_{\bA=4}^n$ be frames for $NY$, $NY'$, respectively, such that the following conditions hold:
\begin{enumerate}
\item The frames $\{ e_{\bA}\}_{\bA=4}^n$ and $\{ e'_{\bA}\}_{\bA=4}^n$ agree along the initial slice, i.e.
\[
e_{\bA} (0, x) = e'_{\bA}(0,x) \quad \text{for all } x\in \mathbb T^3, \, \bA \in \{ 4, \ldots, n\}.
\]

\item With respect to the gauge $\mathcal G \doteq (\text{Id}, e)$, the development $Y$ satifies the bounds
\begin{equation}\label{Bound Y uniqueness}
\sum_{l=0}^2 \| \partial^l (Y - Y_0) \|_{L^{\infty}H^{s-l}[\mathcal G]} + \| \partial^2 Y \|_{L^4 W^{\f1{12},4}[\mathcal G]}+\|\partial^2 Y\|_{L^{\f72} L^{\f{14}3}[\mathcal G]}< \epsilon,
\end{equation}
\begin{equation}\label{Bound k uniqueness}
\sum_{l=0}^2\Big( \|\partial^l k\|_{L^\infty H^{s-2-l}[\mathcal G]} + \|\partial^l k\|_{L^2 W^{-l-\f12+s_0,\infty}[\mathcal G]} \Big) <\epsilon,
\end{equation}
\begin{equation}
\| \partial g \|_{L^1 W^{\f14\delta_0,\infty}[\mathcal G]} +\sum_{l=1}^2\Big( \| \partial^l g \|_{L^2 H^{2-l+\f16+s_1}[\mathcal G]} +\|\partial^l g\|_{L^{\f74} W^{2-l+s_1,\f73}[\mathcal G]}+ \|  \partial^l g \|_{L^{\infty} H^{s-1-l}[\mathcal G]}\Big) < \epsilon
\end{equation}
and
\begin{align}\label{Bound omega uniqueness}
\| \omega \|_{L^1 W^{\f14\delta_0,\infty}[\mathcal G]} + & \sum_{l=0}^1\Big(\| \partial^l \omega \|_{L^2 H^{1-l+\f16+s_1}[\mathcal G]} +\|\partial^l \omega\|_{L^{\f74} W^{1-l+s_1,\f73}[\mathcal G]}+ \|  \partial^l \omega \|_{L^{\infty} H^{s-2-l}[\mathcal G]}\Big)\\
&+\sum_{l=0}^1\Big( \| \partial^{1+l} e \|_{L^2 H^{1-l+\f16+s_1}[\mathcal G]} +\|\partial^{1+l} e\|_{L^{\f74} W^{1-l+s_1,\f73}[\mathcal G]}+ \|  \partial^{1+l} e \|_{L^{\infty} H^{s-2-l}[\mathcal G]}\Big) <\epsilon    \nonumber
\end{align}
for some $\epsilon>0$ small enough in terms of $s$.

\item Similarly, with respect to the gauge $\mathcal G' \doteq (\text{Id}', e')$, the development $Y'$ satisfies the bounds \eqref{Bound Y uniqueness}--\eqref{Bound omega uniqueness}.
\end{enumerate}
Then, there there exists a $T'' <T$ and a $C^1$-diffeomorphism $x' : [0,T''] \times \mathbb{T}^3 \rightarrow \mathcal{U}^\prime$, with $\mathcal{U}^\prime$ being an open neighborhood of $\{ x^0 =0\}$ in $[0,T'] \times \mathbb{T}^3$, such that
\begin{equation}\label{Geometric uniqueness}
Y = Y^\prime \circ x^\prime \text{ on } [0,T''] \times \mathbb{T}^3.
\end{equation}
\end{theorem}

\smallskip
For the proof of Theorem \ref{thm:Uniqueness}, see Section \ref{sec:Uniqueness}.

\smallskip
\noindent \textbf{Remark 1.} \emph{
Utilizing the finite speed of propagation property for the minimal surface equation~\eqref{Minimal surface equation}, one can readily infer from Theorems \ref{thm:Existence} and \ref{thm:Uniqueness} a more general existence and uniqueness result for more general target manifolds and for $H^s$ initial data $(\bY, \bn)$ which are possibly defined on a 3-manifold $\bS$ carrying little topological structure. In particular, one can immediately deduce a similar well-posedness result for asymptotically flat initial data $(\bY, \bn): \mathbb{R}^3 \rightarrow \mathbb{R}^{n+1}$ which are small in $H^{\f52+\f16+}$. Moreover, the smallness assumption for the right hand sides of \eqref{Bound Y uniqueness}--\eqref{Bound omega uniqueness} in Theorem \ref{thm:Uniqueness} can be removed.  However, in order to keep the main ideas of our proof as clear as possible, we will avoid writing down these results in full generality.}

\smallskip
\noindent \textbf{Remark 2.} \emph{The $C^1$-diffeomorphism $x'$ appearing in the statement of Theorem \ref{thm:Uniqueness} satisfies the additional bound
\[ 
\|\partial^3 x'  \|_{L^\infty H^{s-3}}  + \|\partial^3 x'\|_{L^2 H^{\f16+s_1}} +\|\partial^3 x'\|_{L^{\f74} W^{1-l+s_1,\f73}[\mathcal G]}+ \|\partial^3 x'\|_{L^1 W^{-1+\f18 \delta_0,\infty}} + \|\partial x'\|_{L^\infty L^\infty} \lesssim \epsilon.
\]
See Section \ref{sec:Uniqueness} for more details.
}

\section{The balanced gauge condition}\label{sec:Balanced gauge}
In this section, we will introduce a set of gauge conditions for immersions $Y: [0,T]\times \mathbb T^3 \rightarrow \mathcal N $ that will allow us to recover optimal regularity estimates for the components of the induced metric $g$ and the normal connection 1-form $\omega$ in terms of regularity bounds for the second fundamental form $k$. Working with a gauge that satisfies this set of conditions will be crucial for the proof of Theorem \ref{thm:Existence}.

\subsection{Auxiliary constructions}
 Let $Y:[0,T] \times \mathbb T^3 \rightarrow \mathcal N$ be an immersion and let $g$, $k$, $R$ and $R^{\perp}$ be the associated induced metric, second fundamental form, Riemann curvature tensor and normal curvature tensor, respectively. Let us also assume that a gauge $\mathcal{G} = (x,\{ e_{\bA}\}_{\bA=4}^n)$ has been fixed for the immersion $Y$ (see Sections \ref{sec:Geometry}--\ref{sec:Gauge} for the respective definitions). 
 
 The following quantities will appear as inhomogeneous terms in the relations defining the balanced gauge condition:

\begin{definition}\label{def:First antiderivative R}
We will define for any $\a,\b \in \{0,1,2,3\}$, $l\in \{1,2,3\}$:
\begin{align}
\mathcal F^{\natural}_{\a\b} & \doteq \sum_{l=1}^3 \sum_{\bA, \bB \in \{4,\ldots,n\}} \mathbb P^{\natural}[m(e)_{\bA\bB}, k^{\bB}_{\a\b}, \mathcal{T}^{(l)} k^{\bA}_{0 l}],  \label{F natural seed}
\end{align}
and
\begin{equation}
\mathcal F_{\perp}^{\bA\bB}  = - \sum_{i,j=1}^3 \sum_{\ga,\delta=0}^3 \partial_i\Big( \bar g^{ij}   \mathbb P^{\natural} [g^{\ga\delta}, k^{\bB}_{j\ga}, \mathcal T^{(b)} k^{\bA}_{b\delta}]\Big),   \label{F perp seed}
\end{equation}
where $\mathcal{T}^{(l)}$, $l=1,2,3$, is the multiplier operator introduced in Definition \ref{def:T multiplier}. We will also set, for any $(x^0,\bar x)\in [0,T]\times \mathbb T^3$:
\begin{equation}\label{F natural}
\tilde{\mathcal F}^{\natural}_{\a\b}(x^0,\bar x) \doteq \mathcal F^{\natural}_{\a\b}(x^0,\bar x) - \mathbb E \big[\mathcal F^{\natural}_{\a\b}|_{x^0=0}\big](x^0,\bar x)
\end{equation}
and
\begin{equation}\label{F perp}
\tilde{\mathcal F} _{\perp}^{\bA\bB}(x^0,\bar x) \doteq \mathcal F_{\perp}^{\bA\bB}(x^0,\bar x) - \mathbb E \big[ \mathcal F_{\perp}^{\bA\bB}|_{x^0=0}\big](x^0,\bar x).
\end{equation}
In the above, $\mathbb E$ is the extension operator introduced by Definition \ref{def:Extension operator}.
\end{definition}

\begin{remark*}
Note that, along the initial slice $\{x^0=0\}$, we have
\begin{align}\label{Vanishing F initial}
&\tilde{\mathcal F}^{\natural}_{\a\b}|_{x^0=0} = 0, \\
&\tilde{\mathcal F} _{\perp}^{\bA\bB}|_{x^0=0} = 0   \nonumber
\end{align}
 and, for $m\ge 1$:
\begin{align}\label{Vanishing F initial 2}
&\partial_0^m \tilde{\mathcal F}^{\natural}_{\a\b}|_{x^0=0} = \partial_0^m \mathcal F^{\natural}_{\a\b}|_{x^0=0}, \\
&\partial_0^m \tilde{\mathcal F} _{\perp}^{\bA\bB}|_{x^0=0} = \partial_0^m \mathcal F _{\perp}^{\bA\bB}|_{x^0=0}. \nonumber
\end{align}
\end{remark*}

\subsection{The gauge condition}\label{subsec:The balanced gauge condition}
We are now in a position to introduce the set of relations for the coordinate expressions of the induced metric $g$ and the connection 1-form $\omega$ that we will collectively refer to as the ``balanced gauge condition''. To this end, we will make use of the notation introduced in section \ref{sec:Gauge} regarding the $3+1$ decomposition of the metric components with respect to the slices $\bar \Sigma_{\tau} = \{x^0 = \tau\}$.

\begin{definition}[The balanced gauge condition]\label{def:Balanced gauge}
 Let $Y:\mathcal M \rightarrow \mathcal N$, $\mathcal M \simeq [0,T] \times \mathbb T^3$, be an immersion and let $x: \mathcal M \rightarrow [0,T] \times \mathbb T^3$ be a diffeomorphism and $\{ e_{\bA}\}_{\bA=4}^n$ be a frame for the normal bundle $NY$. We will say that the gauge $\mathcal G = (x,\{ e_{\bar A}\}_{\bar A})$  satisfies the balanced condition if the following statements hold along any slice $\overline \Sigma_\tau = \{ x^0 = \tau\}$, $\tau \in [0,T]$:

\begin{itemize}
\item \textbf{Parabolic condition for the lapse $N$. }The mean curvature $\tr_{\bar g} h$ of $\overline \Sigma_\tau$ and the lapse function $N$ satisfy
\begin{equation}\label{Mean curvature condition} 
\begin{cases}
\tr_{\bar g} h   = \Delta_{\bar{g}} |D|^{-1} (N-1) - \f1N \bar{g}^{ij} \tilde{\mathcal{F}}^{\natural}_{ij} +\displaystyle\fint_{(\bar{\Sigma}_{\tau},\bar g)}\Big( \tr_{\bar g} h + \f1N \bar{g}^{ij}  \tilde{\mathcal{F}}^{\natural}_{ij}\Big),\\[6pt] 
\displaystyle\fint_{(\bar{\Sigma}_{\tau}, g_{\mathbb T^3})} N = 1,
\end{cases}
\end{equation}
where $ \tilde{\mathcal{F}}^{\natural}_{ij}$ was defined by \eqref{F natural}, $\fint_{(\bar \Sigma_\tau,\bar g)}$ denotes the average value over $\overline \Sigma_\tau$ with respect to the volume form $\mathrm{dVol}_{\bar g}$ of the Riemannian metric $\bar g$, $(g_{\mathbb T^3})_{ij}=\delta_{ij}$ is the flat metric on $\bar\Sigma_{\tau} $ (identified with $\mathbb T^3$ via $(x^1, x^2, x^3)$), the operator $\Delta_{\bar{g}} |D|^{-1}$ acting on scalar functions is defined by
\[
\Delta_{\bar{g}} |D|^{-1} \doteq \f1{\sqrt{-\det \bar g}} \partial_i \Big( \bar{g}^{ij} \sqrt{-\det \bar g} \f{\partial_j}{|D|}\Big)
\]
and the zeroth order non-local operators $\f{\partial_j}{|D|}$ are defined in terms of the Fourier transform $\hat \cdot$ on $\mathbb T^3$ as follows:
\[
\widehat{\Big(\f{\partial_j }{|D|} f \Big)}(\xi) \doteq \begin{cases}\f{\xi_j}{|\xi|} \hat f (\xi) , \quad \xi \in \mathbb Z^3 \setminus \{0\},\\ 0, \quad \xi = 0 \end{cases}. 
\]

\item \textbf{Gauge condition for  $\bar{g}$, $\beta$.} The map $(x^1, x^2, x^3): \overline \Sigma_\tau \rightarrow \mathbb T^3$ satisfies the following modification of the classical harmonic gauge condition:
\begin{equation}\label{Harmonic condition}
\begin{cases}
\Delta_{\bar{g}} x^i =\Delta_{\bar g}|D|^{-1} \beta^i, \\[6pt]
\displaystyle\fint_{(\bar{\Sigma}_\tau, g_{\mathbb T^3})} \beta^i = 0,
\end{cases}
\quad i=1,2,3,
\end{equation}
In the above, the operator $\Delta_{\bar g}|D|^{-1}$ acting on vector fields on $\bar{\Sigma}_{\tau}$ on the right hand side is defined in terms of the tensorial Laplacian, i.e.
\[
\Delta_{\bar g} |D|^{-1} = \bar{g}^{ij} \bar\nabla_i \bar\nabla_j |D|^{-1}.
\]
Note that, when expressed in terms of the Christoffel symbols $\bar\Gamma^k_{ij}$ of  $(\overline\Sigma_\tau, \bar g)$, \eqref{Harmonic condition} reads:
\[
\begin{cases}
\bar g^{ij} \bar\Gamma^k_{ij} = -\Delta_{\bar g}|D|^{-1} \beta^k,\\[6pt]
\displaystyle\fint_{(\bar{\Sigma}_\tau, g_{\mathbb T^3})} \beta^k = 0.
\end{cases}
\]

\item \textbf{Modified parabolic condition for the normal frame.} The components $\omega_{\alpha \bar B}^{\bar A}$ of the connection coefficients associated to the frame $\{ e_{\bar{A}} \}_{\bar A = 4}^n$ satisfy
\begin{equation}\label{Divergence condition frame}
\begin{cases}
\delta_{\bar{g}} \omega_{\bar B}^{\bar A} \doteq \bar{g}^{ij} \bar{\nabla}_i \omega_{j \bar B}^{\bar A}= 
- \Delta_{\bar g}|D|^{-1}  \omega^{\bA}_{0 \bB} + 
  \Bigg( m(e)_{\bB\bC} \tilde{\mathcal{F}}_{\perp}^{\bA\bC} - 
\displaystyle\fint_{(\bar{\Sigma}_\tau, \bar g)}\Big(m(e)_{\bB\bC} \tilde{\mathcal{F}}_{\perp}^{\bA\bC}\Big)\Bigg), \\[6pt]
\displaystyle\fint_{(\bar{\Sigma}_\tau, g_{\mathbb T^3})} \omega^{\bA}_{0 \bB} = 0,
\end{cases}
\end{equation}
where $ \tilde{\mathcal{F}}_{\perp}^{\bA\bB}$ was defined by \eqref{F perp}.

\end{itemize}
\end{definition}

\medskip
\noindent \textbf{Remark 1.} \emph{Given a smooth immersion $Y: \mathcal M \rightarrow \mathcal N$, there always exists a gauge $\mathcal G$ for $Y$, at least locally in time, satisfying the balanced gauge condition; this is established in Section \ref{sec:Existence gauge}.}

\medskip
\noindent \textbf{Remark 2.} \emph{One of the main technical challenges in the proof of Theorem \ref{thm:Existence} is the need to obtain control of the $L^1 L^\infty$ size of the components of $\partial g$ and $\omega$ even though the optimal estimates obtained for $k$ do not suffice to control the $L^1 W^{-1,\infty}$ norm of all components of the curvature tensors $R$ and $R^\perp$. The role of the source functions $\tilde{\mathcal F}^\natural$ and $\tilde{\mathcal F}_\perp$ appearing in the right hand sides of \eqref{Mean curvature condition} and \eqref{Divergence condition frame}, respectively, is to ``cancel out'' the contributions of the few ``uncontrollable'' components of $R$ and $R^\perp$ to the estimates for $\partial g$ and $\omega$; see also Section \ref{subsec:Cancellations} below. }

\subsection{Relations derived from the balanced gauge condition}\label{subsec:Equations balanced gauge}
The relations \eqref{Mean curvature condition}--\eqref{Divergence condition frame} imply the following set of parabolic and elliptic equations for the components of $g$ and $\omega$ in the $3+1$ formalism:

\begin{lemma}\label{lem:Parabolic elliptic system}
Let $Y: \mathcal M \rightarrow \mathcal N $ be a smooth immersion as in Definition \ref{def:Balanced gauge} and let $\mathcal G$ be a gauge for $Y$ satisfying the balanced gauge condition. Let us also adopt the following shorthand notations for the terms that will appear in the \textbf{right hand sides} of the equations below:
\begin{itemize}
\item $m_0$ will denote the coordinate Minkowski metric (with corresponding components $N_0=1$, $\beta_0=0$, $(\bar g_0)_{ij} = \delta_{ij}$), 
\item $R_*$ will denote any component of the Riemann curvature tensor with at most one  index equal to $0$, i.e.~a component of the form $R_{\a ijk}$, $\a \in \{0,1,2,3\}$, $i,j,k \in \{1,2,3\}$. Similarly, $R_{**}$ will denote any purely spatial component $R_{ijkl}$, $i,j,k,l\in \{1,2,3,\}$.
\item $R^{\perp}_*$ will denote any component of $R^{\perp}$ of the form $R^{\perp \bA \bB}_{ij}$ with $i,j \in \{1,2,3,\}$.
\item $D$ will denote any pseudo-differential operator of order at most $1$ in the $\mathbb T^3$ variables (see the remark below \eqref{Japanese bracket}
\item $\partial$ will denote a coordinate derivative (in any direction).
\item If $\mathcal O$ is any zeroth order pseudo-differential operator in the spatial variables and $f:\mathbb T^3 \rightarrow \mathbb R$ is a smooth function, we will denote $\mathcal O f$ simply as $f$.
\item We will use $g$ to denote any of the components of any \textbf{rational} matrix valued function of $[g]$ (e.g.~$g^{-1}$). 
\item In view of Lemma \ref{lem:Elliptic estimates model} in the appendix, we will denote $\Delta_{\bar g}^{-1}$ simply as $|D|^{-2}$ when it is applied on functions for which the only Sobolev and Besov space estimates that we will consider are  compatible with \eqref{Estimate elliptic} (so that, in those spaces, $\Delta_{\bar g}^{-1}$ indeed satisfies the same estimates as $|D|^{-2}$).
\end{itemize}
Then, the components of $g$, $\omega$ and their derivatives satisfy the following relations:
\begin{enumerate}
\item \textbf{Parabolic equations for $N$, $\beta$.} 
The lapse function $N$ satisfies the parabolic equation
\begin{equation}\label{Equation lapse}
\partial_0 (N-1) + |D| (N-1) =  \mathfrak F_{N} + \mathcal E_{N}, 
\end{equation}
while the components of the shift vector field $\beta^k$ satisfies the parabolic equations:
\begin{equation}\label{Equation shift}
\partial_0 \b^k + |D|\b^k = \mathcal E_{\beta}.
\end{equation}
In the above, the expressions $\mathfrak F_{-}$, $\mathcal E_{-}$ take the following form:
\begin{itemize}
\item In equation \eqref{Equation lapse},
\[
\mathfrak F_N \doteq |D|  \Delta_{\bar g} ^{-1} \Bigg[ \f1N \bar{g}^{ij}\big( \partial_0  \tilde{\mathcal{F}}^{\natural}_{ij} - R_{i0j0}\big)    -   \fint_{(\bar\Sigma_{\tau},\bar g)} \Big\{ \f1N \bar{g}^{ij}\big( \partial_0  \tilde{\mathcal{F}}^{\natural}_{ij} - R_{i0j0}\big) \Big\} \Bigg] 
\]
and $\mathcal E_N$ can be expressed schematically as
\begin{align*}
\mathcal E_N =|D|\Delta_{\bar g}^{-1} \Big( g \cdot \partial g \cdot \tilde{\mathcal{F}}^{\natural} + (g-m_0) \cdot D\partial g +g \cdot R_{*} + g \cdot D g \cdot \partial g +g \cdot h \cdot \partial g \Big).
\end{align*}

\item In equation \eqref{Equation shift},
\[
\mathcal E_{\beta} =|D|\Delta_{\bar g}^{-1} \Big(g \cdot Dh +  (g-m_0) \cdot D\partial g + g \cdot R_{**}  +  g \cdot D g \cdot \partial g +g \cdot h \cdot \partial g \Big).
\]
\end{itemize}

\medskip
\item \textbf{Elliptic equations for $\bar g $, $h$.} 
The components of the induced Riemannian metric $\bar g _{ij}$ on $\overline \Sigma_\tau$ satisfy the  relations
\begin{equation}\label{Equation g bar}
 \bar g^{kl}\partial_k \partial_l (\bar{g}_{ij}) =  -\Delta_{\bar g} |D|^{-1} \big( \nabla_i \beta_j + \nabla_j \beta_i \big)  +  \mathcal E_{\bar g}  
\end{equation}
and
\begin{equation}\label{Equation dt g bar}
\partial_0 \bar g_{ij} = -2Nh_{ij} + \bar\nabla_i \beta_j + \bar\nabla_j \beta_i,
\end{equation}
while the  second fundamental form $h$ of the slice $\overline \Sigma_\tau$ satisfies the elliptic equation
\begin{equation}\label{Equation h}
\Delta_{\bar g} h_{ij} =  \bar\nabla^2_{ij}\big( \Delta_{\bar g} |D|^{-1}(N-1) \big) + \mathcal E_h.
\end{equation}
In the above:
\begin{itemize}
\item In equation \eqref{Equation g bar},
\[
\mathcal E_{\bar g} = g \cdot R_{**} +g\cdot \partial g \cdot \partial g.
\]

\smallskip
\item In equation \eqref{Equation h}
\[
\mathcal E_h = D^2 (g\cdot \tilde{\mathcal{F}}^{\natural} ) + D(g \cdot R_*) + g \cdot \partial g \cdot R_*
\]
\end{itemize}

\medskip
\item \textbf{Parabolic-elliptic system for $\omega$.} The temporal components $\omega_{0\bB}^{\bA}$ of the connection $1$-form $\omega$ satisfy the parabolic equation
\begin{equation}\label{Equation omega 0}
\partial_0   \omega_{0\bB}^{\bA} +|D| \omega_{0\bB}^{\bA} = \big( \mathfrak F_{\omega_0}\big)^{\bA}_{\bB}  +  \mathcal E_{\omega_0},
\end{equation}
while the spatial components $\omega_{i\bB}^{\bA}$ ($i=1,2,3$) satisfy the elliptic relations
\begin{equation}\label{Equation omega i}
  \Delta_{\bar g} \omega_{i\bB}^{\bA} = - \partial_i \big( \Delta_{\bar g}|D|^{-1} \omega_{0\bB}^{\bA} \big)+  \mathcal E_{\omega_{\bar\Sigma}}.
\end{equation}
In the above:
\begin{itemize}
\item In equation \eqref{Equation omega 0}:
\[
\big( \mathfrak F_{\omega_0}\big)^{\bA}_{\bB} \doteq 
 |D|\Delta_{\bar g}^{-1} \Bigg[  \Big(  m_{\bB\bC} \partial_0 \tilde{\mathcal{F}}_{\perp}^{\bA\bC} -\fint_{(\bar\Sigma_{\tau},\bar g)} \big( m_{\bB\bC} \partial_0 \tilde{\mathcal{F}}_{\perp}^{\bA\bC} \Big) 
 +\bar g^{ij} \bar{\nabla}_i (R^{\perp})_{j0\hphantom{\bA}\bB}^{\hphantom{\a\b}\bA} \Bigg]
\]
and
\begin{align*}
 \mathcal E_{\omega_0} = & 
 \JapD^{-1} \big( \partial g \cdot D \omega\big) + \JapD^{-1} \big( D\partial g \cdot g\cdot \omega \big) +  \JapD^{-1} \big( \partial g \cdot \omega \cdot \omega \big)\\
&  + g\cdot \omega \cdot \omega + \JapD^{-1} \big(\partial e \cdot \tilde{\mathcal F}_{\perp} + \partial g \cdot \tilde{\mathcal F}_{\perp} \big).
\end{align*}

\item In equation \eqref{Equation omega i}, $\Delta_{\bar g} \omega_{i\bB}^{\bA}$ denotes the tensorial Laplacian (associated to $\bar g$) acting on $ \omega_{\cdot \bB}^{\bA}$ viewed as an $1$-form on $\bar\Sigma_{\tau}$ and
\begin{align*}
\mathcal E_{\omega_{\bar\Sigma}} = & D \big( g \cdot R^{\perp}_* \big) 
+ g\cdot   \partial g \cdot R^{\perp}_* 
 +D\big( m(e) \cdot \tilde{\mathcal F}_\perp \big) \\
& + D( \omega \cdot \omega) +  \partial g \cdot \omega \cdot \omega + g \cdot \omega \cdot R_*.
\end{align*}
\end{itemize}

\medskip
\item \textbf{Higher order time derivatives of $N$, $\beta$, $\bar g$.} The second order time derivatives of the components of $g$ in the $3+1$-decomposition satisfy the relations:
\begin{equation}\label{Equation dt2 lapse}
\partial_0^2 N = -|D|\partial_0 N +\mathfrak F_{\partial_0^2 N} +  \mathcal E_{\partial_0^2 N},
\end{equation}
\begin{equation}\label{Equation dt2 shift}
\partial_0^2 \b^k = -|D| (\partial_0 \beta^k)  + \mathcal E_{\partial_0^2 \beta}
\end{equation}
and 
\begin{equation}\label{Equation dt2 bar g}
\partial_0^2 \bar g_{ij} = \mathcal E_{\partial_0^2 \bar g},
\end{equation}
where:
\begin{itemize}
\item In equation \eqref{Equation dt2 lapse}:
\[
\mathfrak F_{\partial_0^2 N}  \doteq |D|  \Delta_{\bar g} ^{-1} \Bigg[ \f1N \bar{g}^{ij}\partial_0 \big( \partial_0  \tilde{\mathcal{F}}^{\natural}_{ij} - R_{i0j0}\big)    -   \fint_{(\bar\Sigma_{\tau},\bar g)} \Big\{ \f1N \bar{g}^{ij}\partial_0 \big( \partial_0  \tilde{\mathcal{F}}^{\natural}_{ij} - R_{i0j0}\big) \Big\} \Bigg] 
\]
and
\begin{align*}
\mathcal E_{\partial_0^2 N} =\JapD^{-1} \Big( &\partial g \cdot \partial g \cdot \tilde{\mathcal{F}}^{\natural} + g \cdot \partial^2 g \cdot \tilde{\mathcal{F}}^{\natural} + g \cdot \partial g \cdot \partial \tilde{\mathcal{F}}^{\natural} \\
& + (g-m_0) \cdot D\partial^2 g + g\cdot \partial R_*+\partial g \cdot \partial^2 g  \\
&  + \partial g \cdot \partial g \cdot \partial g+ D\partial g \cdot  \JapD^{-1} \big( g \cdot (\partial \tilde{\mathcal F}^{\natural}-R)\big) \Big).
\end{align*}

\item In equation \eqref{Equation dt2 shift}:
\begin{align*}
\mathcal E_{\partial^2_t \beta} 
= \JapD^{-1} \Big(&  (g-m_0) \cdot D\partial^2 g + g \cdot D^3 N +g\cdot D^2 h\\
& +g\cdot DR_*  + g\cdot  \partial g \cdot \partial^2 g \Big) 
\end{align*}

\item In equation \eqref{Equation dt2 bar g}:
\[
\mathcal E_{\partial_0^2 \bar g} \doteq g \cdot D\partial g + g \cdot \partial g \cdot \partial g + g \cdot R.
\]
\end{itemize}
\end{enumerate}

\end{lemma}

\noindent \textbf{Remark.} 
 The system satisfied by the metric components $(N, \beta, \bar g)$ and the tensor $h$ is a \textbf{quasilinear parabolic-elliptic} one, the linearization of which has a lower triangular structure. More precisely, the above equations take following schematic form:
\begin{equation}\label{Model system of equations}
\begin{cases}
\partial_0 N + |D| N +  \JapD^{-1} \big( (g-m_0)\cdot D\partial g\big) = \mathcal G_1,\\
\Delta_{\bar g} h +D(g\cdot D^2 N) = \mathcal G_2,\\
\partial_0 \beta +|D| \beta +\JapD^{-1}\big(g\cdot Dh\big) + \JapD^{-1} \big( (g-m_0)\cdot D\partial g\big) = \mathcal G_3, \\
\tr_{\bar g}\big(\bar{\partial}^2  \bar g \big) + g\cdot D^2\beta = \mathcal G_4,
\end{cases}
\end{equation}
where the $\mathcal G_i$'s depend only on the auxiliary functions $\tilde{\mathcal F}$, the curvature tensor $R$ and lower order terms in $g,h$ which are at least quadratic in $\partial g, h$. The linearization of the above system around the trivial solution $(g,h)=(m_0,0)$ takes the lower triangular form form
\begin{equation}\label{Model linearized system}
\begin{cases}
\partial_0 \dot N + |D| \dot N = \dot{\mathcal G}_1,\\
\Delta_{\mathbb T^3} \dot h +m_0 \cdot D^3 \dot N = \dot{\mathcal G}_2,\\
\partial_0 \dot \beta +|D| \dot \beta +\JapD\big( m_0 \cdot D\dot h\big) = \dot{\mathcal G}_3\\
\Delta_{\mathbb T^3}\dot{\bar g} + m_0 \cdot D^2 \dot \beta = \dot{\mathcal G}_4.
\end{cases}
\end{equation}
This structure allows one to close estimates for the difference $g-m_0$ (in the context of a bootstrap argument) in terms of estimates for $\tilde{\mathcal F}$ and $R$ by closing the estimates successively for $N$, $h$, $\beta$ and $\bar g$ while using the bootstrap estimates to control the purely non-linear terms (i.e. those appearing in \eqref{Model system of equations} but not in \eqref{Model linearized system}).
\medskip

\bigskip
\begin{proof}
The derivation of equations \eqref{Equation lapse}--\eqref{Equation dt2 bar g} will be based on a series of lengthy (but straightforward computations) using the variation formulas \eqref{Variation metric}--\eqref{Variation second fundamental form} for $\partial_0 \bar g$ and $\partial_0 h$, in connection with equations \eqref{Mean curvature condition}--\eqref{Divergence condition frame} coming from the definition of the balanced gauge condition. We will also use the formulas \eqref{Christoffel symbols 3+1} to express schematically:
\[
\partial N, \partial \beta, \partial \bar g, \bar\Gamma, h \simeq g \cdot \partial g.
\]

\begin{itemize}
\item Applying a $\partial_{x^0}$ derivative to the mean curvature condition \eqref{Mean curvature condition},  a direct application of the formulas \eqref{Variation metric}--\eqref{Variation second fundamental form} for $\mathcal{L}_{\partial_{x^0}} g$ and $\mathcal{L}_{\partial_{x^0}} h$ yields
\begin{align*}
-\Delta_{\bar{g}} & (N-1) +  \b^i \bar{\nabla}_i \tr_{\bar g} h + N |h|_{\bar{g}}^2 -\f1N \Big(\bar g^{ij} R_{i0j0} - 2\bar{g}^{ij} R_{i0jk} \b^k +\bar g^{ij} R_{ikjl} \b^k \b^l \Big) \\
&= \Delta_{\bar{g}} \big( |D|^{-1} \partial_0 (N-1) \big) - \partial_0 \big( \f1N \bar{g}^{ij}  \tilde{\mathcal{F}}^{\natural}_{ij} \big)  + \partial_{\tau} \Big( \fint_{\bar \Sigma_{\tau}} \big(\tr_{\bar g} h + \f1N \bar{g}^{ij}  \tilde{\mathcal{F}}^{\natural}_{ij} \Big) \nonumber \\
& \hphantom{====} + 2\big( N h^{ij} - \bar \nabla^i \beta^j \big) \bar \nabla^2_{ij} \big( |D|^{-1}(N-1) \big)\\
& \hphantom{====} - \Big( -2\bar \nabla ^i (N h_{ik}) + \bar \nabla_k (N \tr_{\bar g} h) + \Delta_{\bar g} \beta_k +\bar g^{ij} (R_{ikjl}+h_{il} h_{kj}-h_{ij}h_{kl}) \b^l \Big)   \nonumber \\
& \hphantom{====\sum\sum\sum\sum} \times \bar \nabla^k \big( |D|^{-1} (N-1) \big).
\end{align*}
Rearranging the terms above and using the fact that $\Delta_{\bar{g}}$ is invertible when acting on the space of scalars $f:\mathbb T^3 \rightarrow \mathbb R$ satisfying $\fint_{\bar \Sigma_{\tau}} f =0 $ (recall also that $\fint_{\bar \Sigma_{\tau}} N = 1$ from \eqref{Mean curvature condition}), we deduce that
\begin{equation}\label{Parabolic lapse}
\partial_0 (N-1) + |D| (N-1) = |D| \Delta_{\bar g} ^{-1} \Big(\mathcal{G} - \fint_{\bar\Sigma_{\tau}} \mathcal{G}\Big),
\end{equation}
where
\begin{align*}
\mathcal{G} \doteq &  \f1N \bar{g}^{ij}\Big( \partial_0  \tilde{\mathcal{F}}^{\natural}_{ij}- R_{i0j0}\Big) +\partial_0 \big( \f1N \bar{g}^{ij}\big)  \tilde{\mathcal{F}}^{\natural}_{ij}  \\
 & +  \b^i \bar{\nabla}_i \tr_{\bar g} h + N |h|_{\bar{g}}^2 +\f1N \Big( 2\bar{g}^{ij} R_{i0jk} \b^k -\bar g^{ij} R_{ikjl} \b^k \b^l \Big)\\
&- 2\big( N h^{ij} - \bar \nabla^i \beta^j \big) \bar \nabla^2_{ij} \big( |D|^{-1}(N-1) \big)\\
& + \Big( -2\bar \nabla ^i (N h_{ik}) + \bar \nabla_k (N \tr_{\bar g} h) + \Delta_{\bar g} \beta_k \\
& \hphantom{-\Big( - ++}+\bar g^{ij} (R_{ikjl}+h_{il} h_{kj}-h_{ij}h_{kl}) \b^l \Big) \bar \nabla^k \big( |D|^{-1} (N-1) \big).
\end{align*}
Equation \eqref{Equation lapse} now follows directly from \eqref{Parabolic lapse}.

\item The gauge condition \eqref{Harmonic condition} for $\bar g$ can be rexpressed as 
\begin{equation}\label{Harmonic condition 2}
\bar{g}^{ij} \bar\Gamma_{ij}^k = -\Delta_{\bar g} |D|^{-1} \beta^k.
\end{equation} 
The variation of \eqref{Harmonic condition 2} with respect to the parameter $\tau$ yields
(in view of the relations \eqref{Variation metric} and \eqref{Variation Christoffel symbols} for the variation of $\bar g$ and $\bar\Gamma$):
\begin{align}\label{Equation shift once more}
\Delta_{\bar{g}}\b^k & - 2 \bar{\Gamma}_{ij}^k \bar{\nabla}^i \b^j  +  \bar{g}^{ij}\bar{g}^{lk} \bar{R}_{milj}\b^m = \\
& -\partial_0 \big( \Delta_{\bar g} |D|^{-1} \beta^k\big) +  \bar{g}^{ij}\bar{g}^{lk} \big( \bar{\nabla}_i (Nh)_{lj}  + \bar{\nabla}_j (Nh)_{il}  - \bar{\nabla}_l (Nh)_{ij} \big) - 2 \bar{\Gamma}_{ij}^k  Nh^{ij }   \nonumber 
\end{align}
where $\bar{\Gamma}$ and $\bar{R}$ are the Christoffel symbols and the Riemann curvature tensor, respectively, of $\bar{g}$, while  
\[
\Delta_{\bar{g}} \b^k \doteq \bar{g}^{ij} \bar{\nabla}_i \bar{\nabla}_j \b^k
\]
 is the tensorial Laplacian associated to $\bar g$. Rearranging the terms in \eqref{Equation shift once more} once more, we obtain:
\begin{align}\label{Equation shift once more 2}
\partial_0 \b^k + |D| \beta^k & \\
= |D|\Delta_{\bar g}^{-1} \Bigg[ &
 2 \bar{\Gamma}_{ij}^k \bar{\nabla}^i \b^j  - \bar{g}^{ij}\bar{g}^{lk} \bar{R}_{milj}\b^m - [\partial_0, \Delta_{\bar g}]|D|^{-1} \beta^k   \nonumber \\
& +  \bar{g}^{ij}\bar{g}^{lk} \big( \bar{\nabla}_i (Nh)_{lj}  + \bar{\nabla}_j (Nh)_{il}  - \bar{\nabla}_l (Nh)_{ij} \big) - 2 \bar{\Gamma}_{ij}^k  Nh^{ij }  \Bigg] \nonumber 
\end{align}
\end{itemize}

\smallskip
\begin{itemize}
\item Applying a spatial derivative to condition \eqref{Harmonic condition} for $\bar\Gamma$, we obtain the following equation for the components of the spatial metric $\bar g _{ij}$:
\begin{align}\label{Elliptic equation g}
\bar g^{kl}\partial_k \partial_l (\bar{g}_{ij}) =&- \partial_i ( \bar g_{jk}\Delta_{\bar g}|D|^{-1} \beta^k ) - \partial_j ( \bar g_{ik} \Delta_{\bar g}|D|^{-1} \beta^k ) + 2\bar \Gamma^l_{ij} \bar g_{lk}( \Delta_{\bar g}|D|^{-1} \beta^l )   \\
& -2  \bar g^{kl}  \bar{R}_{ikjl} + 2\bar g^{kl} \bar g^{cd} \Big( \bar\Gamma_{cki}\bar\Gamma_{dlj} + \bar\Gamma_{cki}\bar\Gamma_{jld} + \bar\Gamma_{ckj}\bar\Gamma_{ild}  \Big).  \nonumber
\end{align}
Equation \eqref{Equation g bar} now follows from above (using also Gauss's equation to express $\bar R = R_{**} + h \cdot h$). Equation \eqref{Equation dt g bar} is just \eqref{Variation metric}. 
\end{itemize}

\smallskip
\begin{itemize}
\item By differentiating the Codazzi equations \eqref{Codazzi 3+1} in the spatial directions, we obtain the following elliptic equation for $h$:
\begin{align}\label{Elliptic equation h}
\Delta_{\bar g} h_{ij} = \nabla^2_{ij} & \tr_{\bar g} h + \bar{g}^{kl} \bar\nabla_k \big( R_{li\a j} \hat n^{\a} \big) 
+  \bar g^{kl} \bar \nabla_i\big(  R_{lj\a k} \hat n^{\a}\big)\\
& +\bar g^{kl} \bar{R}_{kilm} h^m_j - \bar{R}_{kilj} h^{kl},   \nonumber
\end{align}
where $\Delta_{\bar{g}} h_{ij} \doteq \bar{g}^{kl} \bar{\nabla}_k \bar{\nabla}_l h_{ij}$. Equation \eqref{Equation h} now follows directly from \eqref{Elliptic equation h} by using the gauge condition \eqref{Mean curvature condition} to substitute $\tr_{\bar g} h$.
\end{itemize}

\smallskip
\begin{itemize}
\item Calculating $\bar{\nabla}^i (R^{\perp})_{i0 \hphantom{\bA}\bB}^{\hphantom{i0}\bA}$ by considering the spatial divergence of the expression \eqref{Normal curvature coordinates}, we obtain
\[
\bar\nabla^i (R^{\perp})_{i0\hphantom{\bA}\bB}^{\hphantom{\a\b}\bA}
=
\Delta_{\bar g} \omega_{0\bB}^{\bA} - \bar\nabla^i \partial_{0} \omega_{i\bB}^{\bA} +\bar\nabla^i\big(\omega^{\bA}_{i \bar{C}} \omega^{\bar{C}}_{0 \bB}\big) - \bar\nabla^i\big(\omega^{\bA}_{0 \bar{C}} \omega^{\bar{C}}_{i\bB}\big)
\]
Using the gauge condition \eqref{Divergence condition frame} to substitute for $ \partial^i \omega_{i\bB}^{\bA}$, we therefore obtain
\begin{align*}
\bar\nabla^i (R^{\perp})_{i0\hphantom{\bA}\bB}^{\hphantom{\a\b}\bA}
= & 
\Delta_{\bar g} \omega_{0\bB}^{\bA} + \partial_{0} \Delta_{\bar g}|D|^{-1}  \omega^{\bA}_{0 \bB}\\
& -\partial_0
  \Bigg( m(e)_{\bB\bC} \tilde{\mathcal{F}}_{\perp}^{\bA\bC} - 
\fint_{\bar\Sigma_{\tau}}\Big(m(e)_{\bB\bC} \tilde{\mathcal{F}}_{\perp}^{\bA\bC}\Big)\Bigg) \\
& +\bar\nabla^i\big(\omega^{\bA}_{i \bar{C}} \omega^{\bar{C}}_{0 \bB}\big) - \bar\nabla^i\big(\omega^{\bA}_{0 \bar{C}} \omega^{\bar{C}}_{i\bB}\big) + [\partial_0, \bar\nabla^i] \omega_{i\bB}^{\bA}, 
\end{align*}
from which equation \eqref{Equation omega 0} for $\omega_{0\bB}^{\bA}$ readily follows after applying the operator $|D|\Delta_{g}^{-1}$ to both sides. 
\end{itemize}

\smallskip
\begin{itemize}
\item Similarly, calculating the term $\bar g^{il} \bar{\nabla}_i  (R^{\perp})_{lj}^{\bA\bC}$ (for $j=1,2,3$) using the expression \eqref{Normal curvature coordinates}  for $R^\perp$ and using the gauge condition \eqref{Divergence condition frame} to substitute for $ \partial^i \omega_{i\bB}^{\bA}$, we  obtain:
\begin{align}\label{Elliptic equation omega i}
 \Delta_{\bar g} \omega_{j\bB}^{\bA} =
- \partial_j (\Delta_{\bar g} |D|^{-1} \omega_{0\bB}^{\bA} & ) +   \Big( \bar g^{il} \bar{\nabla}_i  (R^{\perp})_{lj\hphantom{\bA}\bB}^{\hphantom{lj}\bA}
 +\partial_j  (m_{\bB\bC} \tilde{\mathcal{F}}_{\perp}^{\bA\bC}) \Big) \\
 & + \bar g^{kl} \bar g^{im} \bar{R}_{ijmk} \omega_{l\bB}^{\bA} - \bar{\nabla}^i \big( \omega^{\bA}_{i \bar{C}} \omega^{\bar{C}}_{j\bB}\big) + \bar\nabla^i \big( \omega^{\bA}_{j \bar{C}} \omega^{\bar{C}}_{i\bB} \big), \nonumber
\end{align}
where $ \Delta_{\bar g} \omega_{j\bB}^{\bA} = \bar g^{il} \bar\nabla_i \bar\nabla_j \omega_{j\bB}^{\bA}$ is the Laplace--Beltrami operator of $\bar{g}$ acting on the 1-form $\omega_{\cdot \bB}^{\bA}$. Equation \eqref{Equation omega i} follows directly from \eqref{Elliptic equation omega i}.
\end{itemize}

\begin{itemize}
\item Equation \eqref{Equation dt2 lapse} for $\partial_0^2 N$ follows by commuting \eqref{Equation lapse} with $\partial_0$.
\end{itemize}

\smallskip
\begin{itemize}
\item Equation \eqref{Equation dt2 shift} for $\partial_0^2 \beta$ follows similarly by differentiating  \eqref{Equation shift} with respect to $x^0$ and using the Bianchi identity to express $\partial_0 R_{**} = \bar \partial R_*+\Gamma \cdot R_*$.
\end{itemize}

\smallskip
\begin{itemize}
\item Combining the variation formulas \eqref{Variation metric} and \eqref{Variation second fundamental form} for $\partial_0 \bar g$ and $\partial_0 h$, respectively, we obtain
\begin{align*}
\partial_0^2 \bar g_{ij} 
=& 2N \bar{\nabla}_i \bar{\nabla}_j N +2N^2 h_{ik}h^k_j +2N^2 R_{i\alpha j\beta}\hat n^\a \hat n^\b -2N h_{kj}\bar{\nabla}_i \b^k -2N h_{ki}\bar{\nabla}_j \b^k \\
&-2N \b^k \bar{\nabla}_k h_{ij} +\partial_0 N \cdot h_{ij} +\bar\nabla_i \partial_0 \b_j + \bar\nabla_j \partial_0 \b_i - 2 \partial_0 \bar\Gamma_{ij}^k \cdot \b_k
\end{align*}
from which equation \eqref{Equation dt2 bar g} follows readily.
\end{itemize}

\end{proof}

\subsection{Cancellations in the gauge source terms involving  $\mathcal F^{\natural}$ and $\mathcal F_{\perp}$} \label{subsec:Cancellations}
In this section, we will exhibit a number of cancellations  that appear in certain high-high interactions in the expressions defining $\mathfrak F_{N}$ and $\mathfrak F_{\omega_0}$ in equations \eqref{Equation lapse} and \eqref{Equation omega 0}, respectively. These cancellations will be play a crucial role in establishing the estimates of Section \ref{sec:Bootstrap}.

To this end, we will first need to introduce the following auxiliary quantities:

\begin{definition}
We will define for $\a,\b,\ga,\delta \in \{0,1,2,3\}$ and $\bA, \bB \in \{ 4, \ldots, n\}$:
\begin{align}
R^{\natural}_{\a\b\ga\delta} & \doteq \sum_{\bA,\bB\in \{4,\ldots,n\}} \Big(\mathbb P^{\natural} \big[ m( e)_{\bA\bB}, k^{\bB}_{\a\ga}, k^{\bA}_{\b\delta}\big] - \mathbb P^{\natural} \big[ m( e)_{\bA\bB}, k^{\bA}_{\a\delta}, k^{\bB}_{\b\ga}\big]\Big),  \label{R LHH}\\
(R^{\perp\natural})^{\bA\bB}_{\a\b} & \doteq \sum_{\ga,\delta\in\{0,\ldots,4\}} \Big(\mathbb P^{\natural} \big[ g^{\ga\delta}, k^{\bB}_{\a\ga}, k^{\bA}_{\b\delta}\big] - \mathbb P^{\natural} \big[ g^{\ga\delta}, k^{\bA}_{\a\ga}, k^{\bB}_{\b\delta}\big]\Big), \label{R perp LHH}
\end{align}
where  $\mathbb P^{\natural}$ is the ``low-high-high'' projection operator introduced in Definition \ref{def:LHH}.
\end{definition}

\begin{remark*}
 The quantities $R^{\natural}$ and $R^{\perp\natural}$ correspond to the low-high-high interaction part in the paraproduct decomposition of the right hand sides of the relations \eqref{Gauss} and \eqref{Ricci} for the tensors $R$ and $R^{\perp}$, respectively; they consist of the parts of $R$ and $R^{\perp}$ for which we will be unable to prove optimal $L^1 L^{\infty}$-type regularity estimates in terms of the bounds for the second fundamental form $k$ (see the statement of Lemma \ref{lem:Riemann estimates}). 
 \end{remark*}

We will establish the following result:

\begin{lemma}\label{lem:Cancellations}
The Littlewood--Paley projections of the quantities $\mathcal F^\natural$ and $\mathcal F_\perp$ satisfy the following schematic relations for any $j>2$:
\begin{align}\label{Calculation F natural difference}
  P_j \Big(   \partial_0 \mathcal F^{\natural}_{cd}  - R^{\natural}_{c0d0} \Big)  
 = &  
     \sum_{\substack{j_2 > j_1-2,\\|j_2-j_3|\le 2}} P_j \Big(|D| \big( P_{j_1} (m(e)) \cdot  P_{j_2} k \cdot P_{j_3} (\JapD^{-1} k)\big) \Big)  \\
&  +   \sum_{\substack{j_2 > j_1-2,\\|j_2-j_3|\le 2}} P_j \big( P_{j_1} (\partial e)  \cdot P_{j_2} k \cdot P_{j_3} (\JapD^{-1} k) \big)  \nonumber  \\
& +   \sum_{\substack{j_2 > j_1-2,\\|j_2-j_3|\le 2}} P_j \big( P_{j_1} (m(e)) \cdot (P_{j_2} ( g\cdot\partial g \cdot k + \omega \cdot k) )\cdot P_{j_3} (\JapD^{-1} k) \big),  \nonumber
\end{align}
 and
 \begin{align}\label{Computation F perp}
P_j \Big( \partial_{0} & (\mathcal F_{\perp})^{\bA\bB} +   \partial_c  (\bar g^{cd} R^{\perp\natural})_{d0}^{\bA\bB} \Big)\\
& = \sum_{\substack{j_1,j_2,j_3,j_4:\\|j_3-j_4|\le 2}}  P_j \Bigg[ 
D^2 \Big(P_{j_1} g \cdot P_{j_2} g \cdot P_{j_3} k \cdot P_{j_4}(\JapD^{-1} k) \Big)   \nonumber \\
& \hphantom{ \sum_{\substack{j_1,j_2,j_3,j_4:\\|j_3-j_4|\le 2}}  P_j \Bigg[  D}
 +D \Big(P_{j_1} ( \partial g + \partial e) \cdot P_{j_2} g \cdot P_{j_3} k \cdot P_{j_4}(\JapD^{-1} k) \Big) \Bigg].   \nonumber
\end{align}
In the case of $ \partial_0 \mathcal F^{\natural}_{cd}  - R^{\natural}_{c0d0}$, we can also ``pull out'' of the product of $k$'s one more spatial derivative compared to \eqref{Calculation F natural difference}: 
\begin{align}\label{Calculation F natural difference 2 derivatives}
  P_j \Big(   \partial_0 \mathcal F^{\natural}_{cd}  - R^{\natural}_{c0d0} \Big)  
 = &  
     \sum_{\substack{j_2 > j_1-2,\\|j_2-j_3|\le 2}} P_j \Big(|D|^2 \big[ P_{j_1} (m(e)) \cdot  P_{j_2} (\JapD^{-1}k) \cdot P_{j_3} (\JapD^{-1} k)\big] \Big)  \\
&  +   \sum_{\substack{j_2 > j_1-2,\\|j_2-j_3|\le 2}} P_j \big( P_{j_1} (\partial e)  \cdot P_{j_2} k \cdot P_{j_3} (\JapD^{-1} k) \big)  \nonumber  \\
& +   \sum_{\substack{j_2 > j_1-2,\\|j_2-j_3|\le 2}} P_j \big( P_{j_1} (m(e)) \cdot P_{j_2} ( g\cdot\partial g \cdot k + \omega \cdot k) \cdot P_{j_3} (\JapD^{-1} k) \big)  \nonumber\\
& +   \sum_{\substack{j_2 > j_1-2,\\|j_2-j_3|\le 2}} P_j \big( P_{j_1} (\partial e) \cdot P_{j_2} (\JapD^{-1}( g\cdot\partial g \cdot k + \omega \cdot k)) \cdot P_{j_3} (\JapD^{-1} k) \big)  \nonumber\\
& +   \sum_{\substack{j_2 > j_1-2,\\|j_2-j_3|\le 2}} P_j \big( P_{j_1} (D\partial e) \cdot P_{j_2} (\JapD^{-1}k) \cdot P_{j_3} (\JapD^{-1} k) \big). \nonumber 
\end{align}

In the right hand sides  of the equations above, we used the convention that $\partial$ denotes any linear combination of derivatives in the $(x^0,x^1,x^2,x^3)$ variables, $D$ denotes any first order pseudo-differential operator in the $(x^1,x^2,x^3)$ variables, while the operators  $|D|, \JapD$ are given by \eqref{Japanese bracket}. The zeroth order terms of the form $g$ denote any component of a rational matrix valued function of $g$ (such as $g^{-1}$). 
\end{lemma}

\begin{proof}

 Using the frequency-projected version of the Codazzi equation \eqref{Codazzi} for $k$, i.e.
\begin{equation}\label{Codazzi once more Paley}
\partial_\kappa P_{j'} (k^{\bar{A}}_{\lambda\mu}) - \partial_\lambda P_{j'} (k^{\bar{A}}_{\kappa \mu})=  P_{j'} \big( \partial g \cdot k + \omega\cdot k\big),
\end{equation}
as well as the divergence identity \eqref{Identity derivative 2} for $\mathcal T^{(i)}_j $, we can readily perform the following calculation for any $j>2$ when the indices $c,d$ are $\neq 0$ (see Definition \ref{def:First antiderivative R} for the definition of $\mathcal F^{\natural}_{cd}$ and \eqref{R LHH} for the definition of $R^{\natural}_{c0d0}$): Using  $\mathcal E^{\natural}[e,g,k]$ to denote terms which can be schematically represented as
\begin{align*}
\mathcal E^{\natural}[e,g,k] = 
&    \sum_{\substack{j_2 > j_1-2,\\|j_2-j_3|\le 2}} P_j \big( P_{j_1} (\partial e)  \cdot P_{j_2} k \cdot P_{j_3} (\JapD^{-1} k) \big) \\
& +    \sum_{\substack{j_2 > j_1-2,\\|j_2-j_3|\le 2}} P_j \big( P_{j_1} (m(e)) \cdot (P_{j_2} ( g\cdot\partial g \cdot k + \omega \cdot k) )\cdot P_{j_3} (\JapD^{-1} k) \big),
\end{align*}
we have:
\begin{align*}
P_j \Big( R^{\natural}_{c0d0} - \partial_0 \mathcal F^{\natural}_{cd} \Big) 
 &  \stackrel{\hphantom{\eqref{Codazzi once more Paley}}}{=} 
    \sum_{\bar j} \sum_{\substack{j'> \bar j-2 \\ |j''-j'|\le 2}}
     P_j \Big[ P_{\bar j} m(e)_{\bA\bB} \Big( P_{j'} k^{\bB}_{cd} \cdot P_{j''} k^{\bA}_{00} - P_{j'} k^{\bA}_{c0} \cdot P_{j''} k^{\bB}_{d0}\\
 & \hphantom{ =   \sum_{\bar j} \sum_{\substack{j'> \bar j-2 \\ |j''-j'|\le 2}}
  P_j \Big[ P_{\bar j} m(e)_{\bA\bB} \Big(} -\partial_0 P_{j'} k^{\bB}_{cd} \cdot P_{j''} \mathcal T^{(l)}_{j''} k^{\bA}_{0l} - P_{j'} k^{\bB}_{cd}  \cdot \partial_0 P_{j''} \mathcal T^{(l)}_{j''} k^{\bA}_{0l}\Big) \Big] \\
&\hphantom{=\sum} + \mathcal E^{\natural}[e,g,k] \\
& \stackrel{\eqref{Codazzi once more Paley}}{=}
    \sum_{\bar j} \sum_{\substack{j'> \bar j-2 \\ |j''-j'|\le 2}}
     P_j \Big[ P_{\bar j} m(e)_{\bA\bB} \Big( P_{j'} k^{\bB}_{cd} \cdot P_{j''} k^{\bA}_{00} - P_{j'} k^{\bA}_{c0} \cdot P_{j''} k^{\bB}_{d0}\\
 & \hphantom{ =   \sum_{\bar j} \sum_{\substack{j'> \bar j-2 \\ |j''-j'|\le 2}}
  P_j \Big[ P_{\bar j} m(e)_{\bA\bB} \Big(} -\partial_c P_{j'} k^{\bB}_{0d} \cdot P_{j''} \mathcal T^{(l)}_{j''} k^{\bA}_{0l} - P_{j'} k^{\bB}_{cd}  \cdot \partial_l P_{j''} \mathcal T^{(l)}_{j''} k^{\bA}_{00}\Big) \Big] \\
&\hphantom{=\sum} + \mathcal E^{\natural}[e,g,k]\\
& \stackrel{\hphantom{\eqref{Codazzi once more Paley}}}{=}   \sum_{\bar j} \sum_{\substack{j'> \bar j-2 \\ |j''-j'|\le 2}}
 P_j \Big[ P_{\bar j} m(e)_{\bA\bB} \Big( P_{j'} k^{\bB}_{cd} \cdot P_{j''} k^{\bA}_{00} - P_{j'} k^{\bA}_{c0} \cdot P_{j''} k^{\bB}_{d0}\\
 & \hphantom{ =   \sum_{\bar j} \sum_{\substack{j'> \bar j-2 \\ |j''-j'|\le 2}}
  P_j \Big[ P_{\bar j} m(e)_{\bA\bB} \Big(} + P_{j'} k^{\bB}_{0d} \cdot \partial_c P_{j''} \mathcal T^{(l)}_{j''} k^{\bA}_{0l} - P_{j'} k^{\bB}_{cd}  \cdot  \partial_l P_{j''} \mathcal T^{(l)}_{j''} k^{\bA}_{00}\\
  & \hphantom{\sum_{\bar j} \sum_{\substack{j'> \bar j-2 \\ |j''-j'|\le 2}}
   P_j \Big[ P_{\bar j} m(e)_{\bA\bB} \Big( }    - \partial_c \big( P_{j'} k^{\bB}_{0d} \cdot P_{j''} \mathcal T^{(l)}_{j''} k^{\bA}_{0l} \big)  \Big) \Big]  \\
&\hphantom{=\sum} + \mathcal E^{\natural}[e,g,k]\\
& \stackrel{\eqref{Codazzi once more Paley}}{=} 
   \sum_{\bar j} \sum_{\substack{j'> \bar j-2 \\ |j''-j'|\le 2}}
    P_j \Big[ P_{\bar j} m(e)_{\bA\bB} \Big( P_{j'} k^{\bB}_{cd} \cdot P_{j''} k^{\bA}_{00} - P_{j'} k^{\bA}_{c0} \cdot P_{j''} k^{\bB}_{d0}\\
 & \hphantom{ =  \sum_{\bar j} \sum_{\substack{j'> \bar j-2 \\ |j''-j'|\le 2}}
  P_j \Big[ P_{\bar j} m(e)_{\bA\bB} \Big(} 
 + P_{j'} k^{\bB}_{0d} \cdot \partial_l P_{j''} \mathcal T^{(l)}_{j''} k^{\bA}_{0c} - P_{j'} k^{\bB}_{cd}  \cdot  \partial_l P_{j''} \mathcal T^{(l)}_{j''} k^{\bA}_{00}\\
  & \hphantom{\sum_{\bar j} \sum_{\substack{j'> \bar j-2 \\ |j''-j'|\le 2}}
   P_j \Big[ P_{\bar j} m(e)_{\bA\bB} \Big( }    - \partial_c \big( P_{j'} k^{\bB}_{0d} \cdot P_{j''} \mathcal T^{(l)}_{j''} k^{\bA}_{0l} \big)  \Big) \Big]  \\
&\hphantom{=\sum} + \mathcal E^{\natural}[e,g,k]\\
 & \stackrel{\eqref{Identity derivative 2}}{=} 
    \sum_{\bar j} \sum_{\substack{j'> \bar j-2 \\ |j''-j'|\le 2}}
     P_j \Big[ P_{\bar j} m(e)_{\bA\bB} \Big( P_{j'} k^{\bB}_{cd} \cdot P_{j''} k^{\bA}_{00} - P_{j'} k^{\bA}_{c0} \cdot P_{j''} k^{\bB}_{d0}\\
 & \hphantom{ =    \sum_{\bar j} \sum_{\substack{j'> \bar j-2 \\ |j''-j'|\le 2}}
  P_j \Big[ P_{\bar j} m(e)_{\bA\bB} \Big(} 
 + P_{j'} k^{\bB}_{0d} \cdot  P_{j''} k^{\bA}_{0c} - P_{j'} k^{\bB}_{cd}  \cdot P_{j''} k^{\bA}_{00}\Big) \Big]\\
&\hphantom{=\sum}
- \sum_{\bar j} \sum_{\substack{j'> \bar j-2 \\ |j''-j'|\le 2}}  \partial_c \Bigg(P_j \Big[ P_{\bar j} m(e)_{\bA\bB}  \cdot  P_{j'} k^{\bB}_{0d} \cdot P_{j''} \mathcal T^{(l)}_{j''} k^{\bA}_{0l} \Big]\Bigg)\\
&\hphantom{=\sum} + \mathcal E^{\natural}[e,g,k].
\end{align*}
In view of the fact that the sum $ \sum_{\bar j} \sum_{\substack{j'> \bar j-2 \\ |j''-j'|\le 2}}$ is symmetric with respect to the indices $j',j''$ and, therefore
\[
 \sum_{\bar j} \sum_{\substack{j'> \bar j-2 \\ |j''-j'|\le 2}}
  P_j \Big[ P_{\bar j} m(e)_{\bA\bB} \Big( - P_{j'} k^{\bA}_{c0} \cdot P_{j''} k^{\bB}_{d0}    + P_{j'} k^{\bB}_{0d} \cdot  P_{j''} k^{\bA}_{0c} \Big)  \Big] = 0,
\]
the above calculation yields
\begin{equation}\label{Almost there simplifying calc F natural}
P_j \Big( R^{\natural}_{c0d0} - \partial_0 \mathcal F^{\natural}_{cd} \Big) = - \sum_{\bar j} \sum_{\substack{j'> \bar j-2 \\ |j''-j'|\le 2}}  \partial_c \Bigg(P_j \Big[ P_{\bar j} m(e)_{\bA\bB}  \cdot  P_{j'} k^{\bB}_{0d} \cdot P_{j''} \mathcal T^{(l)}_{j''} k^{\bA}_{0l} \Big]\Bigg) + \mathcal E^{\natural}[e,g,k],
\end{equation}
which implies \eqref{Calculation F natural difference}.

Applying the Codazzi equation \eqref{Codazzi once more Paley} together with the divergence identity \eqref{Identity derivative 2} repeatedly, and using $\mathcal E_{(\bar j, j', j'')}[e,g,k]$ to denote terms of the form
\begin{align*}
\mathcal E_{(\bar j, j', j'')}[e,g,k] = & 
 P_{\bar j} (m(e)) \cdot (P_{j'} \JapD^{-1}( g\cdot\partial g \cdot k + \omega \cdot k) )\cdot P_{j''} (\JapD^{-1} k)
\\
& +  P_{\bar j} (m(e)) \cdot (P_{j''} \JapD^{-1}( g\cdot\partial g \cdot k + \omega \cdot k) )\cdot P_{j'} (\JapD^{-1} k)\\
& + P_{\bar j}(\partial e) \cdot P_{j'} (\JapD^{-1} k) \cdot P_{j''} (\JapD^{-1} k),
\end{align*}
 we have:
\begin{align*}
P_{\bar j} m(e)_{\bA\bB}  \cdot  P_{j'} k^{\bB}_{0d} \cdot P_{j''} \mathcal T^{(l)}_{j''} k^{\bA}_{0l}
& = 
P_{\bar j} m(e)_{\bA\bB}  \cdot  P_{j'} \mathcal T^{(i)}_{j'} \partial_i k^{\bB}_{0d} \cdot P_{j''} \mathcal T^{(l)}_{j''} k^{\bA}_{0l}\\
& = 
P_{\bar j} m(e)_{\bA\bB}  \cdot  P_{j'} \mathcal T^{(i)}_{j'} \partial_d k^{\bB}_{0i} \cdot P_{j''} \mathcal T^{(l)}_{j''} k^{\bA}_{0l} + \mathcal E_{(\bar j, j', j'')}[e,g,k]
\\
& =  
\partial_d\Big(P_{\bar j} m(e)_{\bA\bB}  \cdot  P_{j'} \mathcal T^{(i)}_{j'}  k^{\bB}_{0i} \cdot P_{j''} \mathcal T^{(l)}_{j''} k^{\bA}_{0l} \Big)
\\
&\hphantom{==}
- P_{\bar j} m(e)_{\bA\bB}  \cdot  P_{j'} \mathcal T^{(i)}_{j'}  k^{\bB}_{0i} \cdot P_{j''} \mathcal T^{(l)}_{j''} \partial_d k^{\bA}_{0l}
+ \mathcal E_{(\bar j, j', j'')}[e,g,k]
\\
& = 
\partial_d\Big(P_{\bar j} m(e)_{\bA\bB}  \cdot  P_{j'} \mathcal T^{(i)}_{j'}  k^{\bB}_{id} \cdot P_{j''} \mathcal T^{(l)}_{j''} k^{\bA}_{0l} \Big)
\\
&\hphantom{==}
- P_{\bar j} m(e)_{\bA\bB}  \cdot  P_{j'} \mathcal T^{(i)}_{j'}  k^{\bB}_{0i} \cdot P_{j''} \mathcal T^{(l)}_{j''} \partial_l k^{\bA}_{0d}
+ \mathcal E_{(\bar j, j', j'')}[e,g,k]
\\
& = 
\partial_d\Big(P_{\bar j} m(e)_{\bA\bB}  \cdot  P_{j'} \mathcal T^{(i)}_{j'}  k^{\bB}_{id} \cdot P_{j''} \mathcal T^{(l)}_{j''} k^{\bA}_{0l} \Big)
- \partial_l \Big(P_{\bar j} m(e)_{\bA\bB}  \cdot  P_{j'} \mathcal T^{(i)}_{j'}  k^{\bB}_{0i} \cdot P_{j''} \mathcal T^{(l)}_{j''} k^{\bA}_{0d}\Big)
\\
& \hphantom{==}
+P_{\bar j} m(e)_{\bA\bB}  \cdot  P_{j'} \mathcal T^{(i)}_{j'}  \partial_l k^{\bB}_{0i} \cdot P_{j''} \mathcal T^{(l)}_{j''} k^{\bA}_{0d}
+ \mathcal E_{(\bar j, j', j'')}[e,g,k]
\\
& = 
\partial_d\Big(P_{\bar j} m(e)_{\bA\bB}  \cdot  P_{j'} \mathcal T^{(i)}_{j'}  k^{\bB}_{id} \cdot P_{j''} \mathcal T^{(l)}_{j''} k^{\bA}_{0l} \Big)
- \partial_l \Big(P_{\bar j} m(e)_{\bA\bB}  \cdot  P_{j'} \mathcal T^{(i)}_{j'}  k^{\bB}_{0i} \cdot P_{j''} \mathcal T^{(l)}_{j''} k^{\bA}_{0d}\Big)
\\
& \hphantom{==}
+P_{\bar j} m(e)_{\bA\bB}  \cdot  P_{j'} \mathcal T^{(i)}_{j'}  \partial_i k^{\bB}_{0l} \cdot P_{j''} \mathcal T^{(l)}_{j''} k^{\bA}_{0d}
+ \mathcal E_{(\bar j, j', j'')}[e,g,k]
\\
& = 
\partial_d\Big(P_{\bar j} m(e)_{\bA\bB}  \cdot  P_{j'} \mathcal T^{(i)}_{j'}  k^{\bB}_{id} \cdot P_{j''} \mathcal T^{(l)}_{j''} k^{\bA}_{0l} \Big)
- \partial_l \Big(P_{\bar j} m(e)_{\bA\bB}  \cdot  P_{j'} \mathcal T^{(i)}_{j'}  k^{\bB}_{0i} \cdot P_{j''} \mathcal T^{(l)}_{j''} k^{\bA}_{0d}\Big)
\\
& \hphantom{==}
+P_{\bar j} m(e)_{\bA\bB}  \cdot  P_{j'}  k^{\bB}_{0l} \cdot P_{j''} \mathcal T^{(l)}_{j''} k^{\bA}_{0d}
+ \mathcal E_{(\bar j, j', j'')}[e,g,k].
\end{align*}
Bringing the first term in the last line above to the left hand side and using the fact that $m(e)_{\bA\bB}$ is symmetric in $\bA,\bB$, we obtain after summing over $\bar j\in \mathbb N, j'>\bar j-1, |j''-j'|\le 2$ (note also that this range is symmetric over $j', j''$; thus, after summing, the first term in the last line above is equal to the left hand side):
\begin{align*}
 \sum_{\bar j} \sum_{\substack{j'> \bar j-2 \\ |j''-j'|\le 2}} & P_j \Big( P_{\bar j} m(e)_{\bA\bB}  \cdot  P_{j'} k^{\bB}_{0d} \cdot P_{j''} \mathcal T^{(l)}_{j''} k^{\bA}_{0l}\Big) \\
& = \sum_{\bar j} \sum_{\substack{j'> \bar j-2 \\ |j''-j'|\le 2}} \Bigg(\bar\partial P_j \Big( P_{\bar j}m(e) \cdot P_{j'}(\JapD^{-1} k) \cdot P_{j'}(\JapD^{-1} k) \Big)+  \mathcal E_{(\bar j, j', j'')}[e,g,k] \Bigg),
\end{align*}
where $\bar\partial \in \{\partial_1, \partial_2, \partial_3\}$. Note that, after applying a $\partial_c$ derivative to the expression on the left hand side above, one obtains the first term in the right hand side of \eqref{Almost there simplifying calc F natural}. Thus, substituting there, we finally infer \eqref{Calculation F natural difference 2 derivatives}.

 Similarly to the proof of \eqref{Calculation F natural difference}, we can readily calculate:
\begin{align*}
P_j \Big( \partial_{0} & (\mathcal F_{\perp})^{\bA\bB} +   \partial_c  (\bar g^{cd} R^{\perp\natural})_{d0}^{\bA\bB} \Big)\\
& \stackrel{\hphantom{\eqref{Codazzi once more Paley}}}{=}
   \sum_{\bar j} \sum_{\substack{j'> \bar j-2 \\ |j''-j'|\le 2}} 
      P_j \Bigg[  -\partial_0 \partial_c \Big(\bar g^{cd} \cdot P_{\bar j} g^{\ga\delta} \cdot P_{j'} k^{\bB}_{d\ga} \cdot P_{j''}  \mathcal T^{(b)}_{j''} k^{\bA}_{b\delta} \Big) \nonumber \\
&\hphantom{=  \sum_{\bar j} \sum_{\substack{j'> \bar j-2 \\ |j''-j'|\le 2}}
      P_j \Bigg[ \sum}
      + \partial_c  \Big(\bar g^{cd} \cdot P_{\bar j} g^{\ga\delta}  \cdot P_{j'} k^{\bB}_{d\ga} \cdot P_{j''} k^{\bA}_{0\delta} - \bar g^{cd} \cdot  P_{\bar j} g^{\ga\delta} \cdot P_{j'} k^{\bA}_{d\ga} \cdot P_{j''} k^{\bB}_{0\delta} \Big)\\
& \stackrel{\hphantom{\eqref{Codazzi once more Paley}}}{=}
   \sum_{\bar j} \sum_{\substack{j'> \bar j-2 \\ |j''-j'|\le 2}}
      \partial_c P_j \Bigg[  - \bar g^{cd} \cdot P_{\bar j} g^{\ga\delta} \cdot P_{j'} \partial_0 k^{\bB}_{d\ga} \cdot P_{j''}  \mathcal T^{(b)}_{j''} k^{\bA}_{b\delta} \\
      &\hphantom{=  \sum_{\bar j} \sum_{\substack{j'> \bar j-2 \\ |j''-j'|\le 2}}
      \partial_c P_j \Bigg[ \sum}
       - \bar g^{cd} \cdot P_{\bar j} g^{\ga\delta} \cdot P_{j'} k^{\bB}_{d\ga} \cdot P_{j''}  \mathcal T^{(b)}_{j''} \partial_0 k^{\bA}_{b\delta} \\
&\hphantom{=  \sum_{\bar j} \sum_{\substack{j'> \bar j-2 \\ |j''-j'|\le 2}}
      \partial_c P_j \Bigg[ \sum}
      + \bar g^{cd}  \cdot P_{\bar j} g^{\ga\delta}  \cdot P_{j'} k^{\bB}_{d\ga} \cdot P_{j''} k^{\bA}_{0\delta} \\
 &\hphantom{=  \sum_{\bar j} \sum_{\substack{j'> \bar j-2 \\ |j''-j'|\le 2}} 
      \partial_c P_j \Bigg[ \sum}     
      -  \bar g^{cd} \cdot  P_{\bar j} g^{\ga\delta} \cdot P_{j'} k^{\bA}_{d\ga} \cdot P_{j''} k^{\bB}_{0\delta} \Bigg] \\
& \hphantom{ =\sum }
+ \mathcal E_{\perp} [g,e,k]\\
& \stackrel{\eqref{Codazzi once more Paley}}{=}
  \sum_{\bar j} \sum_{\substack{j'> \bar j-2 \\ |j''-j'|\le 2}}
      \partial_c P_j \Bigg[  - \bar g^{cd} \cdot P_{\bar j} g^{\ga\delta} \cdot P_{j'} \partial_d k^{\bB}_{0\ga} \cdot P_{j''}  \mathcal T^{(b)}_{j''} k^{\bA}_{b\delta} \\
      &\hphantom{=  \sum_{\bar j} \sum_{\substack{j'> \bar j-2 \\ |j''-j'|\le 2}}
      \partial_c P_j \Bigg[ \sum}
       - \bar g^{cd} \cdot P_{\bar j} g^{\ga\delta} \cdot P_{j'} k^{\bB}_{d\ga} \cdot P_{j''}  \mathcal T^{(b)}_{j''} \partial_b k^{\bA}_{0\delta} \\
&\hphantom{=   \sum_{\bar j} \sum_{\substack{j'> \bar j-2 \\ |j''-j'|\le 2}}
      \partial_c P_j \Bigg[ \sum}
      + \bar g^{cd}  \cdot P_{\bar j} g^{\ga\delta}  \cdot P_{j'} k^{\bB}_{d\ga} \cdot P_{j''} k^{\bA}_{0\delta} \\
 &\hphantom{=   \sum_{\bar j} \sum_{\substack{j'> \bar j-2 \\ |j''-j'|\le 2}}
      \partial_c P_j \Bigg[ \sum}     
      -  \bar g^{cd} \cdot  P_{\bar j} g^{\ga\delta} \cdot P_{j'} k^{\bA}_{d\ga} \cdot P_{j''} k^{\bB}_{0\delta} \Bigg] \\
& \hphantom{ =\sum }
+ \mathcal E_{\perp} [g,e,k]\\
& \stackrel{\eqref{Identity derivative 2}}{=}
  \sum_{\bar j} \sum_{\substack{j'> \bar j-2 \\ |j''-j'|\le 2}}
      \partial_c P_j \Bigg[  - \partial_d \Big( \bar g^{cd} \cdot P_{\bar j} g^{\ga\delta} \cdot P_{j'} \partial_d k^{\bB}_{0\ga} \cdot P_{j''}  \mathcal T^{(b)}_{j''} k^{\bA}_{b\delta}\Big) \\
 &\hphantom{=   \sum_{\bar j} \sum_{\substack{j'> \bar j-2 \\ |j''-j'|\le 2}} 
      \partial_c P_j \Bigg[ \sum}     
       + \bar g^{cd} \cdot P_{\bar j} g^{\ga\delta} \cdot P_{j'} k^{\bB}_{0\ga} \cdot P_{j''}  \mathcal T^{(b)}_{j''} \partial_d  k^{\bA}_{b\delta} \\
      &\hphantom{=   \sum_{\bar j} \sum_{\substack{j'> \bar j-2 \\ |j''-j'|\le 2}}
      \partial_c P_j \Bigg[ \sum}
       - \bar g^{cd} \cdot P_{\bar j} g^{\ga\delta} \cdot P_{j'} k^{\bB}_{d\ga} \cdot P_{j''}  k^{\bA}_{0\delta} \\
&\hphantom{=   \sum_{\bar j} \sum_{\substack{j'> \bar j-2 \\ |j''-j'|\le 2}} 
      \partial_c P_j \Bigg[ \sum}
      + \bar g^{cd}  \cdot P_{\bar j} g^{\ga\delta}  \cdot P_{j'} k^{\bB}_{d\ga} \cdot P_{j''} k^{\bA}_{0\delta} \\
 &\hphantom{=   \sum_{\bar j} \sum_{\substack{j'> \bar j-2 \\ |j''-j'|\le 2}}
      \partial_c P_j \Bigg[ \sum}     
      -  \bar g^{cd} \cdot  P_{\bar j} g^{\ga\delta} \cdot P_{j'} k^{\bA}_{d\ga} \cdot P_{j''} k^{\bB}_{0\delta} \Bigg] \\
& \hphantom{ =\sum }
+ \mathcal E_{\perp} [g,e,k]\\
& \stackrel{\substack{\eqref{Codazzi once more Paley},\\ \eqref{Identity derivative 2}}}{=}
  \sum_{\bar j} \sum_{\substack{j'> \bar j-2 \\ |j''-j'|\le 2}}
      \partial_c P_j \Bigg[  - \partial_d \Big( \bar g^{cd} \cdot P_{\bar j} g^{\ga\delta} \cdot P_{j'} \partial_d k^{\bB}_{0\ga} \cdot P_{j''}  \mathcal T^{(b)}_{j''} k^{\bA}_{b\delta}\Big) \\
 &\hphantom{=   \sum_{\bar j} \sum_{\substack{j'> \bar j-2 \\ |j''-j'|\le 2}} 
      \partial_c P_j \Bigg[ \sum}     
       + \bar g^{cd} \cdot P_{\bar j} g^{\ga\delta} \cdot P_{j'} k^{\bB}_{0\ga} \cdot P_{j''}   k^{\bA}_{d \delta} \\
 &\hphantom{=  \sum_{\bar j} \sum_{\substack{j'> \bar j-2 \\ |j''-j'|\le 2}}
      \partial_c P_j \Bigg[ \sum}     
      -  \bar g^{cd} \cdot  P_{\bar j} g^{\ga\delta} \cdot P_{j'} k^{\bA}_{d\ga} \cdot P_{j''} k^{\bB}_{0\delta} \Bigg] \\
& \hphantom{ =\sum }
+ \mathcal E_{\perp} [g,e,k]
      \end{align*}
where $\mathcal E_{\perp} [g,e,k]$ denotes a collection of terms of the form
\begin{align*}
\mathcal E_{\perp} [g,e,k] = &
\sum_{\substack{j_1,j_2,j_3,j_4:\\|j_3-j_4|\le 2}}  P_j \Bigg[ D \Big(P_{j_1} ( \partial g + \partial e) \cdot P_{j_2} g \cdot P_{j_3} k \cdot P_{j_4}(\JapD^{-1} k) \Big) \Bigg].
\end{align*}
Note that the second and third terms in the sum above cancel out in view of the fact that $g^{\ga\delta}$ is symmetric in $\ga,\delta$ and the summation takes place over a set of indices which is symmetric with respect to $j',j''$. Therefore, we obtain \eqref{Computation F perp}.

\end{proof}

\section{Existence of solutions: The continuity argument}\label{sec:Continuity}

The proof of Theorem \ref{thm:Existence} will proceed via a standard continuity argument. The aim of this section is to lay out an outline for this argument, which will be then implemented in Sections \ref{sec: Bounds S0} and \ref{sec:Bootstrap}.

We will only prove Theorem \ref{thm:Existence} in the case of smooth initial data sets $(\bY,\bn)$ satisfying the assumption \eqref{D}. The statement regarding general initial data in the space $H^s \times H^{s-1}$ will then follow via a standard argument by approximating any general initial data set $(\bY,\bn)$ by a sequence of smooth ones $(\bY_n, \bn_n)$ and then using the fact that the estimates \eqref{Theorem embedding}--\eqref{Theorem connection coefficients} depend only on $\|(\bY,\bn)\|_{H^s \times H^{s-1}}$ to pass to a weakly convergent subsequence of developments. 

Before setting up our continuity argument, we will need to introduce some shorthand notation regarding various norms for the geometric quantities involved:

\begin{definition}\label{def:Norms 1}
Let $0<\delta_0<s_1<s_0$ be defined in terms of $s$ as in Definition \ref{def:Small constants}.\footnote{Recall that, when $s-\f52-\f16$ is small, all three of these parameters are also small.} 
 For any immersion $Y:[0,T]\times \mathbb T^3 \rightarrow \mathcal N$ and any choice of frame $\{e_{\bA}\}_{\bA=4}^n$ for the normal bundle $NY$, we will set
\begin{align}
 \mathcal{Q}_k \doteq   \sum_{l=0}^2 \Big( \|\partial^l k \|_{L^\infty H^{s-2-l}} +\|\partial^l k \|_{L^2 W^{ - \f12-l+s_0, \infty}} + \|\partial^l k \|_{L^4 W^{\f{1}{12}-l+s_0,4}}+\|\partial^l k \|_{L^{\f72} W^{-l+s_0,\f{14}3}}\Big),
\end{align}
and
\begin{align}
\mathcal{Q}_g \doteq \|g-m_0  & \|_{L^\infty L^\infty}+\| \partial g \|_{L^1 W^{\f14\delta_0,\infty}}+ \| \partial g \|_{L^{\infty} H^{s-2}}+ \| \partial g \|_{L^{2} H^{\f76+s_1}} +\|\partial g\|_{L^{\f74} W^{1+s_1,\f73}}\\
& +  \| \partial^2 g \|_{L^{\infty} H^{s-3}}+ \| \partial^2 g \|_{L^{2} H^{\f16+s_1}}+\|\partial^2 g\|_{L^{\f74} W^{s_1,\f73}}, \nonumber 
\end{align}
where $m_0$ denotes the standard Minkowski metric on $[0,T]\times \mathbb{T}^3$, which in the product coordinate system takes the form
\[
m_0 = -(dx^0)^2+(dx^1)^2+(dx^2)^2+(dx^3)^2.
\]
In the above, it is implicitly assumed that mixed Sobolev norms are defined with respect to the gauge $\mathcal G = (\text{Id}, e)$ for $Y$ (see Section \ref{sec:Mixed Norms}). We will also define:
\begin{align}
\mathcal{Q}_\perp \doteq & \| \omega \|_{L^1 W^{\f14\delta_0,\infty}}+\| \omega \|_{L^2 H^{ \f76+s_1}} +\| \omega\|_{L^{\f74} W^{1+s_1,\f73}}+ \| \omega \|_{L^{\infty} H^{s-2}}\\
& +\| \partial \omega \|_{L^2 H^{ \f16+s_1}}+\|\partial \omega\|_{L^{\f74} W^{s_1,\f73}} + \| \partial \omega \|_{L^{\infty} H^{s-3}} \nonumber \\
&+\sum_{\bA=4}^{n} \Big(\| e_{\bA} - \delta_{\bA}\|_{L^\infty L^\infty} + \sum_{l=1}^2 \big( \|\partial^l e_{\bA} \|_{L^\infty H^{s-1-l}}+\| \partial^l e_{\bA} \|_{L^2 H^{\f16+2-l+s_1}}+\|\partial^l e_{\bA}\|_{L^{\f74} W^{2-l+s_1,\f73}}\big)\Big), \nonumber
\end{align}
where $\omega$ 
is the connection form associated to the frame $\{ e_{\bar{A}}\}_{\bar{A}=4}^n$ and $\delta_{\bA}$ is the vector field in the normal bundle $NY$ with Cartesian components 
\[
\delta_{\bA}^A = 
\begin{cases}
1, \quad A = \bA,\\
0, \quad A \neq \bA.
\end{cases}
\]
\end{definition}

\medskip
\noindent \textbf{Outline of the proof of Theorem \ref{thm:Existence}. }
In order to prove Theorem \ref{thm:Existence}, it suffices to establish the following intermediate results:
\begin{itemize}
\item[\textbf{Step 1.}] \textbf{Local existence of a smooth development:} For any smooth initial data pair $\bY:\mathbb T^3 \rightarrow \mathcal N$, $\bn:\mathbb T^3 \rightarrow \mathbb R^{n+1}$, there exists a smooth development $Y:[0,T]\times \mathbb T^3 \rightarrow \mathcal N$ for some $T>0$.

\smallskip
\item[\textbf{Step 2.}] \textbf{Local existence of a balanced gauge:} For any smooth development $Y:[0,T]\times \mathbb T^3 \rightarrow \mathcal N$ of a smooth initial data pair $(\bY,\bn)$ as in the statement of Theorem \ref{thm:Existence}, there exists an open neighborhood $\mathcal U \subset [0,T]\times \mathbb T^3 $ of the slice $\{0\}\times \mathbb T^3$, a diffeomorphism $\Psi:\mathcal U \rightarrow [0,T') \times \mathbb T^3$ mapping $\{0\}\times \mathbb T^3$ to itself and a frame $\{e_{\bA}\}_{\bA=4}^n$ for $N|_{\mathcal U} Y$ such that the following conditions hold:
\begin{itemize}
\item Along the initial slice $\{0\}\times \mathbb T^3$, the diffeomorphism $\Psi$ satisfies
\begin{equation}\label{Diffeomorphism initial}
\Psi|_{\{0\}\times \mathbb T^3}= \mathrm{\text{Id}}.
\end{equation}

\smallskip
\item The gauge $\mathcal G = (\Psi,e)$ for the immersion $Y:\mathcal U \rightarrow \mathcal N$ satisfies the \textbf{balanced gauge condition} (see Definition \ref{def:Balanced gauge}).

\smallskip
\item Along the initial slice $\{0\}\times \mathbb T^3$, the frame $\{e_{\bA}\}_{\bA=4}^n$ for $N|_{\mathcal U} Y$ and the differential of the diffeomorphism $\Psi$ satisfy the bound:
\begin{align}\label{Bound slice initial}
\sum_{A=0}^n \sum_{\bA=4}^{n} \Big( \big\| e^A_{\bA}|_{\{0\}\times \mathbb T^3} -\delta^A_{\bA}\big\|_{H^{s-1}(\mathbb T^3)} 
& + \big\| \nabla^\perp_0 e^{A}_{\bA}|_{\{0\}\times \mathbb T^3} \big\|_{H^{s-2}(\mathbb T^3)}\Big) \\
&  + \sum_{\a,\b=0}^3 \big\| \partial_{\a} \Psi^{\b}|_{\{0\}\times \mathbb T^3} - \delta_{\a}^{\b} \big\|_{H^{s-1}(\mathbb T^3)} \le  C \mathcal D  \nonumber 
\end{align}
for some constant $C>0$ depending only on $s$ (recall the definition \eqref{D} of $\mathcal D$ in terms of $(\bar Y, \bar n)$).
\end{itemize}

\begin{remark*}
Note that \eqref{Diffeomorphism initial} implies that the initial data pairs $(\bar Y, \bar n):\mathbb T^3 \rightarrow \mathcal N$ associated to the immersions $Y$ and $Y\circ \Psi$ are the same (i.e.~not simply related by a diffeomorphism). We will actually show that a gauge $\mathcal G = (\Psi,e)$ as above exists for \emph{any} smooth immersion $Y$ extending an initial data pair $(\bY,\bn)$ satisfying $\eqref{D}$, i.e.~not necessarily satisfying the minimal surface equation \eqref{Minimal surface equation}.
\end{remark*}

\smallskip
\item[\textbf{Step 3.}] \textbf{Estimates along the initial hypersurface:} Let $(\bY,\bn)$ be a smooth initial data pair as in the statement of Theorem \ref{thm:Existence} and let $Y:[0,T]\times \mathbb T^3 \rightarrow \mathcal N$ be a smooth development of $(\bY,\bn)$. Assume that $\{e_{\bA}\}_{\bA=4}^n$ is a frame for $NY$ such that the following hold:
\begin{itemize}
\item The gauge $\mathcal G = (\text{Id}, e)$ satisfies the balanced gauge condition.
\item Along the initial slice $\{0\}\times \mathbb T^3$ we have 
\begin{equation}\label{Initial bound preliminary}
\| \beta^k|_{x^0=0}\|_{H^{s-1}} + \| \omega_{0\bB}^{\bA}|_{x^0=0}\|_{H^{s-1}} +  \big\| e^A_{\bA}|_{x^0=0} -\delta^A_{\bA}\big\|_{H^{s-1}} \le  C \mathcal D
\end{equation}
for some constant $C>0$ depending only on $s$.
\end{itemize}
 Then, there exists a $T'=T'\big( (\bY,\bn) \big)>0$ such that, when restricted on $[0,T']\times \mathbb T^3$, the immersion $Y$ satisfies
\begin{equation}\label{Bootstrap intro bound}
\mathcal Q_g + \mathcal Q_k +\mathcal Q_\perp + \| Y-Y_0\|_{L^\infty L^\infty} + \|\partial^2 Y\|_{L^4 W^{\f1{12},4}}+\|\partial^2 Y\|_{L^{\f72} L^{\f{14}3}}\lesssim \mathcal D
\end{equation}
(see Proposition \ref{prop: Bounds S0} and the remark below it).

\begin{remark*}
Note that the immersion $Y\circ x:[0,T')\rightarrow \mathcal N$ constructed in Step 2 satisfies \eqref{Initial bound preliminary} as a consequence of \eqref{Bound slice initial}.
\end{remark*}

\smallskip
\item[\textbf{Step 4.}] \textbf{The bootstrap argument:} For any $T\le 1$, let $Y:[0,T)\times \mathbb T^3 \rightarrow \mathcal N$ be a smooth development of $(\bY, \bn)$  as in the statement of Theorem \ref{thm:Existence} and $\{e_{\bA}\}_{\bA=4}^n$ is a frame for $NY$ such that :
\begin{itemize} 
\item[a.] The gauge $\mathcal G = (\text{Id}, e)$ satisfies the balanced gauge condition.
\item[b.] Along the initial slice $\{x^0 = 0\}$, $\mathcal G$ satisfies  \eqref{Initial bound preliminary}.
\item[c.] The immersion $Y$ satisfies the bootstrap assumption 
\[
\mathcal Q_g + \mathcal Q_k +\mathcal Q_\perp + \| Y-Y_0\|_{L^\infty L^\infty} +\|\partial^2 Y\|_{L^4 W^{\f1{12},4}}+\|\partial^2 Y\|_{L^{\f72} L^{\f{14}3}}\le C_0 \mathcal D
\] 
for some sufficiently large constant $C_0>0$ depending only on $s$.
\end{itemize}
Then, the following improved estimate holds:
\[
\mathcal Q_g + \mathcal Q_k +\mathcal Q_\perp + \| Y-Y_0\|_{L^\infty L^\infty} +\|\partial^2 Y\|_{L^4 W^{\f1{12},4}}+\|\partial^2 Y\|_{L^{\f72} L^{\f{14}3}}\le \f12 C_0 \mathcal D
\]
(see Proposition \ref{prop:Bootstrap} for a more detailed presentation of this step). 

\smallskip
\item[\textbf{Step 5.}] \textbf{Persistence of regularity:} Any development $Y:[0,T)\times \mathbb T^3 \rightarrow \mathcal N$ as in Step 4 remains uniformly smooth up to $\{x^0=T\}$.

\smallskip
\item[\textbf{Step 6.}] \textbf{Local extendibility of the balanced gauge:} Let $Y:[0,T')\times \mathbb T^3 \rightarrow \mathcal N$ be a smooth development of an initial data pair $(\bY,\bn)$ and $\{ e_{\bA}\}_{\bA=4}^n$ be a frame for the normal bundle $NY$. Assume that, for some $T<T'$, the gauge $\mathcal G = (\text{Id}, e)$ restricted to $[0,T)\times \mathbb T^3$ satisfies the balanced gauge condition. Then, there exists an open neighborhood $\mathcal U$ of  $[0,T]\times \mathbb T^3$ in  $[0,T')\times \mathbb T^3$, a diffeomorphism $\tilde x : \mathcal U \rightarrow [0,T'')\times \mathbb T^3$ (for some $T''>T$) and a frame $\{\tilde{e}_{\bA}\}_{\bA=4}^n$ for the normal bundle $N|_{\mathcal U}Y$ such that:
\begin{itemize}
\item[$\circ$] $\tilde x=\text{Id}$ and $\tilde e_{\bA} = e_{\bA}$ on $[0,T]\times \mathbb T^3$
\item[$\circ$] $\tilde{\mathcal G} = (\tilde x, \tilde e)$ satisfies the balanced condition on $\mathcal U$. 
\end{itemize}
\end{itemize} 

\medskip
\noindent Among the steps outlined above, Step 1 is trivial (as it corresponds to a classical well-posedness statement for the minimal surface equation \eqref{Minimal surface equation} with sufficiently regular data). Step 3 corresponds to estimating the quantities involved in the definition of \eqref{Bootstrap intro bound} along $\{x^0=0\}$ in terms of the geometric initial data $(\bar Y, \bar n)$ and the assumptions on the properties of the gauge $\mathcal G$ infinitesimally around $\{x^0=0\}$; this will be carried out in Section \ref{sec: Bounds S0}. The proof of Step 4 will occupy Section \ref{sec:Bootstrap}; this step constitutes by far the most challenging part of the proof of Theorem \ref{thm:Existence}. In Section \ref{sec:Existence gauge}, we will establish Steps 2 and 6 (Step 6 being essentially identical to Step 2); Step 5 will be established in Section \ref{sec:Persistence of regularity}, thus completing the proof of Theorem \ref{thm:Existence}.

\section{Step 3: Estimates along the initial hypersurface}\label{sec: Bounds S0}
Let $Y:[0,T]\times \mathbb T^3 \rightarrow \mathcal N$ be a smooth development of an initial data pair $(\bY,\bn)$. If $\{e_{\bA}\}_{\bA=4}^n$ is any frame for $NY$ such that the gauge $\mathcal G = (\text{Id}, e)$ satisfies the balanced condition, the minimal surface equation combined with the equations defining the balanced condition imply a number of estimates for the infinitesimal geometry of the immersion $Y$ along the initial slice $\{x^0=0\}$ in terms of appropriate bounds on the geometric initial data $(\bY,\bn)$ and on the gauge quantities $\beta|_{x^0=0}, \omega|_{x^0=0}$. 

In this section, we will show that the assumption \eqref{D} on the $H^s \times H^{s-1}$ size of the  initial data pair, together with the smallness assumption  \eqref{Initial bound preliminary}, imply that the immersion $Y$ satisfies ``infinitesimally'' the bound \eqref{Bootstrap intro bound}; this will correspond to establishing Step 3 of the proof of Theorem \ref{thm:Existence}, as outlined in the previous section.

\begin{proposition}\label{prop: Bounds S0}
Let $Y:[0,T)\times \mathbb T^3 \rightarrow \mathcal N$  be a smooth development of a smooth initial data pair $(\bY, \bn)$ as in the statement of Theorem \ref{thm:Existence} and $\{e_{\bA}\}_{\bA=4}^n$ be a smooth frame for the normal bundle $NY$, such that the gauge $\mathcal G = (\text{Id}, e)$ satisfies the balanced gauge condition and, along $\{ x^0=0\}$, we have
\begin{equation}\label{Boundary assumption G}
\| \beta^k|_{x^0=0}\|_{H^{s-1}} + \| \omega_{0\bB}^{\bA}|_{x^0=0}\|_{H^{s-1}} +  \big\| e^A_{\bA}|_{x^0=0} -\delta^A_{\bA}\big\|_{H^{s-1}} \le  C_1 \mathcal D
\end{equation}
where $e^A_{\bA}$ are the components of the frame vector field $e_{\bA}$ with respect to the Cartesian coordinates on $T \mathcal N$ and $C_1>0$ is a constant depending only on $s$.
Then, provided $\epsilon$ is small enough in terms of $s$, the following bounds hold along $\{ x^0 = 0\}$: 
\begin{equation}\label{Bound Y initial}
\big\| (Y - Y_0) |_{x^0=0} \big\|_{H^{s-1}}+ \big\| (\partial Y - \partial Y_0) |_{x^0=0} \big\|_{H^{s-2}} + \big\| \partial^2 Y |_{x^0=0} \big\|_{H^{s-3}} \le C\mathcal{D},
\end{equation}

\begin{equation}\label{Bound k and dk initial}
\big\| k|_{x^0=0} \big\|_{H^{s-2}} + \big\| \partial k|_{x^0=0} \big\|_{H^{s-3}} + \big\| \partial^2 k|_{x^0=0} \big\|_{H^{s-4}} \le C \mathcal D,
\end{equation}

\begin{equation}\label{Bound metric initial}
 \big\|g|_{x^0=0} - (m_0)\big\|_{H^{s-1}} 
+\big\|\partial g|_{x^0=0} \big\|_{H^{s-2}} +  \big\| \partial^2 g|_{x^0=0} \|_{H^{s-3}}+ \big\|\partial R|_{x^0=0} \big\|_{H^{s-4}} \le C \mathcal D,
\end{equation}

\begin{equation}\label{Bound connection coefficients initial}
\big\|\omega|_{x^0=0} \big\|_{H^{s-2}}+\big\|\partial \omega|_{x^0=0} \big\|_{H^{s-3}} + \big\|\partial R^{\perp}|_{x^0=0} \big\|_{H^{s-4}} \le C \mathcal D
\end{equation}
and
\begin{align}\label{Bound frame}
\sum_{\bA=4}^{n} \big\|\partial e_{\bA}|_{x^0=0} \big\|_{H^{s-2}} + \sum_{\bA=4}^{n} \big\|\partial^2 e_{\bA}|_{x^0=0} \big\|_{H^{s-3}} \le C \mathcal{D},
\end{align}
where $C>0$ is a constant depending only $s$.
\end{proposition}

\begin{remark*}
It follows immediately from the bounds \eqref{Bound Y initial}--\eqref{Bound frame} that, for any sufficiently regular development $Y:[0,T)\times \mathbb T^3 \rightarrow \mathcal N$ as in the statement of Proposition \ref{prop: Bounds S0}, there exists a (possibly small) $T'=T'(Y)>0$ such that, when restricted on $[0,T']\times \mathbb T^3$, the immersion $Y$ satisfies
\begin{align*}
 \mathcal{Q}_g  &+ \mathcal{Q}_k  + \mathcal{Q}_\perp + \| Y - Y_0 \|_{L^\infty L^\infty}\\
 &  + \| \partial Y - \partial Y_0 \|_{L^\infty H^{s-1}} + \|\partial^2 Y\|_{L^4 W^{\f1{12},4}}+\|\partial^2 Y\|_{L^{\f72} L^{\f{14}3}}\le C \mathcal{D},  
\end{align*}
i.e. $Y|_{[0,T']\times \mathbb T^3}$ satisfies the bootstrap assumption of Proposition \ref{prop:Bootstrap}.
\end{remark*}

\begin{proof}
The proof of Proposition \ref{prop: Bounds S0} will consist of a long series of successive estimates along $\{x^0=0\}$. We will adopt the conventions described in the statement of Lemma \ref{lem:Parabolic elliptic system} regarding schematic expressions. We will also use $\bar\partial$ to denote any coordinate derivative in the spatial directions.

\medskip
\noindent \textbf{Bounds for $g|_{x^0=0}$.} The definition \eqref{D} of the quantity $\mathcal{D}$ implies that the spatial derivatives of the immersion $Y$ satisfy
\begin{equation}\label{Bound spatial Y test}
\sum_{i,j=1}^3 \| \partial_i \partial_j Y|_{x^0=0} \|_{H^{s-2}} = \sum_{i,j=1}^3 \| \partial_i \partial_j \bY|_{x^0=0} \|_{H^{s-2}} \lesssim \mathcal{D}.
\end{equation}
In particular, since
\[
\bar g_{ij}|_{x^0=0} = m\big( \partial_i Y, \partial_j Y \big)|_{x^0 = 0},
\]
and $s-1>\f32$, we readily deduce that
\begin{equation}\label{Bound bar g initial}
\sum_{i,j=1}^3 \| (\bar g_{ij} - \delta_{ij}) |_{x^0=0}\|_{H^{s-1}} \lesssim \mathcal D.
\end{equation}

In view of the fact that the vector field $\bn$ belongs to the tangent space of the immersion $Y\big( [0,T)\times \mathbb T^3 \big)$ along $\{0\} \times \mathbb T^3$ and is normal to $\{0\} \times \mathbb T^3$, it has to coincide with the pushforward via $Y$ of the normal vector field $\hat n$ of the foliation $\bar\Sigma_{\tau}$ at $\tau=0$ (using the notation of Section \ref{sec:Gauge}), i.e.:
\begin{equation}\label{Relation n hats}
\bn^{A} = (Y^* \hat n|_{x^0=0})^A = \partial_{\a} Y^A \cdot \hat n^{\a}|_{x^0=0}.
\end{equation}
Thus, in view of the definition \eqref{Definition h} of the second fundamental form $h_{ij}|_{x^0=0}$ of $\bar\Sigma_0$, we can readily calculate
\begin{align}\label{Bound h initial}
h_{ij}|_{x^0=0} &= g_{\a\b} \, \hat n^{\a}|_{x^0=0} \big( \nabla_i \partial_j)^{\b}|_{x^0=0} \\
& = m_{AB} \, \bar n^{A} \, \partial^2_{ij} Y^{B}|_{x^0=0}      \nonumber
\end{align}
and, therefore, in view of  bound \eqref{D} for $\bn$ and the estimate \eqref{Bound spatial Y test} for $\partial^2_{ij} Y$ (and the functional inequalities provided by Lemma \ref{lem:Functional inequalities} in the Appendix):
\begin{equation}
\sum_{i,j=1}^3 \| h_{ij} |_{x^0=0}\|_{H^{s-2}} \lesssim \mathcal D.
\end{equation}

In view of the relation \eqref{Relation lapse normal} between $\hat n$ and $\partial_{x^0}$, we can rewrite \eqref{Relation n hats} as
\[
\partial_0 Y^A|_{x^0=0} = N|_{x^0=0} \bar n^A + \b^k|_{x^0=0} \partial_k \bY^A.
\]
Using the estimate \eqref{D} for $\bY-\bY_0$ and $\bn-\bn_0$,  the above expression immediately yields:
\begin{equation}\label{First bound dt Y}
\|( \partial_0 Y- \partial_0 Y_0) |_{x^0=0} \|_{H^{s-1}} \lesssim \mathcal D \big( \|N|_{x^0=0}\|_{L^\infty} + \| \beta |_{x^0=0} \|_{L^\infty}  + 1\big) +\|(N-1)|_{x^0=0}\|_{H^{s-1}}+ \| \b|_{x^0=0}\|_{H^{s-1}}.
\end{equation}

The gauge condition \eqref{Mean curvature condition} restricted at $\{x^0=0\}$ determines the lapse $N$ in terms of the extrinsic geometry of the slice $\{x^0=0\}$; recall, in view of the definition \eqref{F natural} of $\tilde{\mathcal F}^{\natural}$, that
\[
\tilde{\mathcal F}^{\natural}_{\a\b}|_{x^0=0} = 0
\]
(see also \eqref{Vanishing F initial}). In particular, \eqref{Mean curvature condition}, \eqref{Bound bar g initial}and \eqref{Bound h initial} imply that
\begin{equation}\label{First bound N initial}
\| (N-1)|_{x^0=0}\|_{H^{s-1}}  \lesssim \| h|_{x^0=0}\|_{H^{s-2}} + \| (\bar g_{ij} - \delta_{ij}) |_{x^0=0}\|_{H^{s-1}} 
  \lesssim \mathcal D . 
\end{equation}

Similarly, the gauge condition \eqref{Harmonic condition} restricted at $\{x^0=0\}$ allows us to determine the shift vector field $\beta^k|_{x^0=0}$ in terms of $\bar g|_{x^0=0}$ and its first order spatial derivatives (i.e.~$\bar\Gamma^k_{ij}|_{x^0=0}$): In view of the bound  \eqref{Bound bar g initial} for $\bar g$, \eqref{Harmonic condition} implies that:
\begin{equation}\label{Bound b initial}
\| \b|_{x^0=0} \|_{H^{s-1}} \lesssim \mathcal D.
\end{equation}
Combining \eqref{Bound bar g initial}, \eqref{First bound N initial} and \eqref{Bound b initial}, we obtain:
\begin{equation}\label{Bound g initial}
\| (g-m_0)|_{x^0=0} \|_{H^{s-1}} \lesssim \mathcal D.
\end{equation}

Returning to the bound \eqref{First bound dt Y} and using \eqref{First bound N initial} and \eqref{Bound b initial}, we infer
\begin{equation}\label{Final bound dt Y}
\| (\partial_0 Y - \partial_0 Y_0) |_{x^0=0} \|_{H^{s-1}} \lesssim \mathcal D .
\end{equation}
Combining \eqref{Bound spatial Y test} and \eqref{Final bound dt Y}, we obtain 
\begin{equation}\label{Final bound d Y}
\| ( Y -  Y_0) |_{x^0=0} \|_{H^{s-1}}+\| (\partial Y - \partial Y_0) |_{x^0=0} \|_{H^{s-2}} \lesssim \mathcal D .
\end{equation}

\medskip
\noindent \textbf{Bounds for $k|_{x^0=0}$.} 
The expression \eqref{Second fundamental form} for $k$ takes the form\footnote{In \eqref{Second fundamental form again}, the notation $(\perpi )^{\bA}_B(e)$ is used to remind the reader that the components $(\perpi )^{\bA}_B$ of $\perpi$, viewed as functions on $[0,T]\times \mathbb{T}^3$, can be expressed as explicit smooth functions of the vectors $\{ e_{\bar{A}}\}_{\bA =4}^{n}$, which themselves can be viewed as smooth functions $e_{\bA}: [0,T]\times \mathbb{T}^3 \rightarrow \mathbb{R}^{n+1}$. }
\begin{equation}
k_{\a\b}^{\bA} = (\perpi )^{\bA}_B(e) \cdot \partial_{\a}\partial_{\b}Y^B. \label{Second fundamental form again}
\end{equation}
 Combining \eqref{Second fundamental form again} for $\a=i,\b=j$ with the estimate \eqref{Bound spatial Y test} for $Y$ and the bound \eqref{Boundary assumption G} for $e_{\bA}$, we obtain
\begin{equation}\label{Bound k spatial initial}
\sum_{i,j=1}^3 \| k^{\bA}_{ij} |_{x^0=0} \|_{H^{s-2}} \lesssim \mathcal D \big( 1+ \|e|_{x^0=0}\|_{L^\infty} + \| \bar\partial e|_{x^0=0} \|_{H^{s-2}}\big) \lesssim \mathcal D.
\end{equation}
Similarly, using the estimate \eqref{Final bound d Y}, we obtain from \eqref{Second fundamental form again} for $\a=0,\b=i$:
\begin{equation}\label{Bound k 0spatial initial}
\sum_{i=1}^3\| k^{\bA}_{0i} |_{x^0=0}\|_{H^{s-2}} \lesssim 
 \mathcal D \big(  1+\|e|_{x^0=0}\|_{L^\infty} + \|\bar\partial e|_{x^0=0}\|_{H^{s-2}} \big) \lesssim \mathcal D. 
\end{equation}

The minimal surface equation\eqref{Minimal surface equation} can be reexpressed as:
\begin{equation}\label{Reexpressed minimal surface equation}
g^{00} k^{\bA}_{00} = -2g^{0i} k^{\bA}_{0i} - g^{ij} k^{\bA}_{ij}.
\end{equation}
In view of the bounds \eqref{Bound g initial} for $g|_{x^0=0}$ and \eqref{Bound k spatial initial}--\eqref{Bound k 0spatial initial} for $k^{\bA}_{\a i}$, we infer from \eqref{Reexpressed minimal surface equation} (using also the functional inequalities of Lemma \ref{lem:Functional inequalities}) that:
\begin{equation}\label{Bound k 00 initial}
\| k^{\bA}_{00} |_{x^0=0}\|_{H^{s-2}} \lesssim  \mathcal D.
\end{equation}
Combining \eqref{Bound k spatial initial}, \eqref{Bound k 0spatial initial} and \eqref{Bound k 00 initial}, we obtain
\begin{equation}\label{Bound k initial}
\| k |_{x^0=0}\|_{H^{s-2}} \lesssim \mathcal D.
\end{equation}

\medskip
\noindent \textbf{Bounds for $R|_{x^0=0}$ and $R^{\perp}|_{x^0=0}$.}
In view of the Ricci equation \eqref{Ricci}, we can express schematically:
\[
 R^\perp = g \cdot  k \cdot  k.
\]
 As a result, we can estimate using the bounds from Lemma \ref{lem:Functional inequalities}:
\begin{align*}
\|  R^\perp  |_{x^0=0}\|_{H^{s-3}} & \lesssim \| g \cdot k \cdot k |_{x^0=0}\|_{H^{s-3}} \\
& \lesssim \| g |_{x^0=0}\|_{H^{s-1}} \|k \cdot k |_{x^0=0} \|_{H^{s-3}} \\
 & \lesssim \| g |_{x^0=0}\|_{H^{s-1}} \|k |_{x^0=0} \|^2_{H^{s-2}}. \nonumber
\end{align*}
Using of the bounds \eqref{Bound g initial} for $g|_{x^0=0}$ and \eqref{Bound k initial} for $k$, we therefore obtain:
\begin{equation}\label{Bound R perp initial}
\|  R^\perp  |_{x^0=0}\|_{H^{s-3}} \lesssim \mathcal D^2.
\end{equation}
Arguing in a similar way using the Gauss equation \eqref{Gauss} for $R|_{x^0=0}$, expressed schematically as
\[
R = m(e) \cdot  k \cdot  k
\]
 we can also estimate
\begin{equation}\label{Bound R initial}
\| R|_{x^0=0} \|_{H^{s-3}} \lesssim \mathcal D^2.
\end{equation}

\medskip
\noindent \textbf{Bounds for $\partial e|_{x^0=0}$.} 
Let us adopt the shorthand notation $\omega_0$ for the temporal components $\omega^{\bA}_{0\bB}$ of $\omega$ and $\bar\omega$ for the spatial components $\omega^{\bA}_{i\bB}$, $i=1,2,3$. 

Our assumption \eqref{Boundary assumption G} implies that
\[
\big\| \bar\partial e|_{x^0=0} \big\|_{H^{s-2}} \lesssim \mathcal D,
\]
where $\bar\partial$ denotes any spatial derivative. Therefore, our definition \eqref{Definition connection coefficients} of the connection coefficients $\omega_{\a\bB}^{\bA}$ implies that the spatial components $\bar\omega$ satisfy
\begin{equation}\label{Bound omega i initial}
\| \bar\omega|_{x^0=0}\|_{H^{s-2}} \lesssim \mathcal D.
\end{equation}
The gauge condition \eqref{Divergence condition frame} allows us to determine $\omega_0|_{x^0=0}$ in terms of the spatial components $\bar\omega|_{x^0=0}$ and $\tilde{\mathcal F}_{\perp}|_{x^0=0}$. In view of \eqref{Vanishing F initial}, we have
\[
\tilde{\mathcal F}_{\perp}|_{x^0=0}=0.
\]
Therefore,  \eqref{Divergence condition frame} yields
\begin{equation}\label{Bound omega 0 initial}
\| \omega_0|_{x^0=0}\|_{H^{s-2}}  \stackrel{\hphantom{\eqref{Bound omega i initial}}}{\lesssim} \| \bar\omega|_{x^0=0}\|_{H^{s-2}}  \stackrel{\eqref{Bound omega i initial}}{\lesssim} \mathcal D.
\end{equation}
Combining \eqref{Bound omega i initial} and \eqref{Bound omega 0 initial}, we get
\begin{equation}\label{Bound omega initial}
\| \omega|_{x^0=0}\|_{H^{s-2}} \lesssim \mathcal D.
\end{equation}

In view of the relation \eqref{Relation e Omega k}, we can express schematically
\[
\partial e = \omega \cdot e + g\cdot m(e)\cdot  k \cdot \partial Y.
\]
 Thus, combining the bounds \eqref{Bound omega initial} for $\omega$, \eqref{Boundary assumption G} for $e$, \eqref{Bound g initial} for $g$ \eqref{Bound k initial} for $k$ and \eqref{Final bound d Y} for $\partial Y$,
we obtain using the functional inequalities from Lemma \ref{lem:Functional inequalities}:
\begin{equation}\label{Bound partial e initial}
\| \partial e |_{x^0=0} \|_{H^{s-2}}  \lesssim \mathcal D.  
\end{equation}

\medskip
\noindent \textbf{Bounds for $\partial g|_{x^0=0}$.}
We will now proceed to obtain an estimate for $\|\partial g|_{x^0=0}\|_{H^{s-2}}$. Since \eqref{Bound g initial} implies that 
\begin{equation}\label{Bound bar d g initial}
\|\bar \partial g|_{x^0=0}\|_{H^{s-2}} \lesssim \mathcal D,
\end{equation}
it remains to estimate  $\|\partial_0 g|_{x^0=0}\|_{H^{s-2}}$. To this end, we will use the expressions derived in Lemma \ref{lem:Parabolic elliptic system} for the time derivatives of $N$, $\beta$, $\bar g$.

We can express equation \eqref{Equation dt g bar} schematically as
\[
\partial_0 \bar g = g \cdot h + \bar\partial g + \bar{\Gamma} \cdot g,
\]
where $\bar\Gamma$ are the Christoffel symbols of $\bar g$ (i.e.~terms of the form $\bar g \cdot \bar\partial \bar g$). Using the bound \eqref{Bound g initial} to estimate  $\bar\Gamma|_{x^0=0}$ and $g|_{x^0=0}$, as well as \eqref{Bound h initial} for $h$, we infer
\begin{equation}\label{Bound dt g bar initial}
\| \partial_0 \bar g|_{x^0=0}\|_{H^{s-2}} \lesssim \mathcal D.
\end{equation} 

In view of \eqref{Equation shift}, we can express schematically
\[
\partial_0 \beta = D g + \JapD^{-1} \Big(g \cdot Dh +  (g-m_0) \cdot D\partial g + g \cdot R_*  + g \cdot D g \cdot \partial g +g \cdot h \cdot \partial g\Big).
\]
 Therefore, using the bounds \eqref{Bound g initial} for $g$, \eqref{Bound h initial} for $h$ and \eqref{Bound R initial} for $R$, we obtain
\begin{equation}\label{First bound dt shift initial}
\|\partial_0 \beta|_{x^0=0}\|_{H^{s-2}} \lesssim \mathcal D +\mathcal D \|\partial g|_{x^0=0}\|_{H^{s-2}}. 
\end{equation}

The relation \eqref{Equation lapse} can be reexpressed as
\begin{align*}
\partial_0 N = D g & +\JapD^{-1}\Big( g\cdot (\partial_0 \tilde{\mathcal F}^\natural - R)\Big) \\
& + \JapD^{-1} \Big( g \cdot \partial g \cdot \tilde{\mathcal F}^\natural + (g-m_0) \cdot D\partial g + g \cdot R_* + g \cdot D g \cdot \partial g +g \cdot h \cdot \partial g\Big).
\end{align*}
Using the fact that $\tilde{\mathcal F}^\natural|_{x^0=0}=0$ and $\partial_0 \tilde{\mathcal F}^\natural|_{x^0=0}=\partial_0 \mathcal F^\natural |_{x^0=0}$  (see \eqref{Vanishing F initial}--\eqref{Vanishing F initial 2}), the above relation yields (in view of the bound \eqref{Bound g initial} for $g$, \eqref{Bound h initial} for $h$ and \eqref{Bound R initial} for $R$):
\begin{equation}\label{First bound lapse initial}
\| \partial_0 N|_{x^0=0}\|_{H^{s-2}} \lesssim \|\partial_0 \mathcal F^\natural |_{x^0=0}\|_{H^{s-3} }+ \mathcal D +\mathcal D \|\partial g|_{x^0=0}\|_{H^{s-2}}.
\end{equation}
Combining the bounds \eqref{Bound dt g bar initial}--\eqref{First bound lapse initial} for $\partial_0 \bar g$, $\partial_0 \beta$ and $\partial_0 N$, as well as the bound \eqref{Bound bar d g initial} for $\bar\partial g$, we obtain:
\[
\|\partial g |_{x^0=0} \|_{H^{s-2}} \lesssim \|\partial_0 \mathcal F^\natural |_{x^0=0}\|_{H^{s-3} }+ \mathcal D +\mathcal D \|\partial g|_{x^0=0}\|_{H^{s-2}}.
\]
Provided $\epsilon$ in \eqref{D} is small enough in terms of $s$, the last term in the right hand side above can be absorbed into the left hand side, yielding
\begin{equation}\label{First bound dt g initial}
\|\partial g |_{x^0=0} \|_{H^{s-2}} \lesssim \|\partial_0 \mathcal F^\natural |_{x^0=0}\|_{H^{s-3} }+ \mathcal D.
\end{equation}

In order to estimate the term $\partial_0 \mathcal F^\natural |_{x^0=0}$ in the right hand side above, we will make use of the definition \eqref{F natural seed} of $\mathcal F^\natural$ to express schematically:
\[
\mathcal F^\natural = \sum_j \sum_{j'>j} P_j(m(e)) \cdot P_j \big( P_{j'} k \cdot |D|^{-1}P_{j'} k\big).
\]
We can, therefore, estimate, using the bounds \eqref{Boundary assumption G}  and \eqref{Bound partial e initial} for the frame $e$ and the bound \eqref{Bound k initial} for $k$ (as well as the functional inequalities of Lemma \ref{lem:Functional inequalities}):
\begin{align}
\|\partial_0 \mathcal F^\natural |_{x^0=0}\|_{H^{s-3} } 
& \lesssim \| \partial e |_{x^0=0}\|_{H^{s-2}} \|k|_{x^0=0}\|^2_{H^{s-2}}
+ \|m(e)|_{x^0=0}\|_{H^{s-1}} \|\partial_0 k|_{x^0=0}\|_{H^{s-3}} \|k|_{x^0=0}\|_{H^{s-2}}   \nonumber  \\
& \lesssim \mathcal D^3 + \mathcal D \|\partial_0 k|_{x^0=0}\|_{H^{s-3}}. \label{Bound dt F natural initial}
\end{align}
Returning to \eqref{First bound dt g initial}, we obtain
\begin{equation}\label{Second bound dt g initial}
\|\partial g |_{x^0=0} \|_{H^{s-2}} \lesssim \mathcal D + \mathcal D \|\partial_0 k|_{x^0=0}\|_{H^{s-3}}.
\end{equation}

\medskip
\noindent \textbf{Bound for $\partial k|_{x^0=0}$.}
The Codazzi equation  \eqref{Codazzi} and the expression \eqref{Reexpressed minimal surface equation} for the minimal surface equation  yield the following relations for $\partial_0 k$:
\begin{align*}
\partial_0 k^{\bar{A}}_{ij} &= \partial_i k^{\bar{A}}_{0j} + g\cdot \partial g \cdot k + \omega\cdot k,\\
\partial_0 k^{\bar{A}}_{0i} &= \partial_i k^{\bar{A}}_{00} + g\cdot \partial g \cdot k + \omega\cdot k,\\
\partial_0 k^{\bar{A}}_{00} &= -\partial_0 \Big( (g^{00})^{-1}\big( 2g^{0i} k^{\bar{A}}_{0i}+g^{ij} k^{\bar{A}}_{ij}\big)  \Big)\\
 &= - (g^{00})^{-1}\big( 2g^{0i} \partial_i k^{\bar{A}}_{00}+g^{ij} \partial_i k^{\bar{A}}_{0j}\big)+ g\cdot \partial g \cdot k + \omega\cdot k\\
&= g\cdot \bar{\partial} k  + g\cdot \partial g \cdot k + \omega\cdot k.
\end{align*}
 Thus, we can estimate using the functional inequalities of Lemma \ref{lem:Functional inequalities}:
\begin{align*}
\| \partial_0 k|_{x^0=0}\|_{H^{s-3}}  & \lesssim \| g \cdot \bar{\partial} k|_{x^0=0} \|_{H^{s-3}} +  \| g\cdot \partial g \cdot k|_{x^0=0} \|_{H^{s-3}} +\| \omega\cdot k|_{x^0=0} \|_{H^{s-3}}\\
 &  \lesssim \| g|_{x^0=0}\|_{H^{s-1}}  \cdot \| \bar{\partial} k|_{x^0=0} \|_{H^{s-3}} +   \| g |_{x^0=0}\|_{H^{s-1}}   \| \partial g |_{x^0=0}\|_{H^{s-2}}  \|k|_{x^0=0} \|_{H^{s-2}} \nonumber \\
 & \hphantom{\lesssim }\,
 +\| \omega|_{x^0=0}\|_{H^{s-2}} \| k|_{x^0=0} \|_{H^{s-2}}  \nonumber \\
 & \lesssim \mathcal D + \mathcal D \| \partial_0 k|_{x^0=0}\|_{H^{s-2}}, \nonumber 
\end{align*}
where, in the last line above, we made use of the bounds \eqref{Bound g initial} for $g$, \eqref{Second bound dt g initial} for $\partial g$, \eqref{Bound omega initial} for $\omega$ and \eqref{Bound k initial} for $k$. Provided $\mathcal D$ is sufficiently small in terms of $s$, we therefore obtain:
\begin{equation}\label{Bound transversal derivative k}
\| \partial_0 k|_{x^0=0}\|_{H^{s-3}} \lesssim \mathcal D.
\end{equation}
Thus, combining  \eqref{Bound transversal derivative k} and \eqref{Bound k initial}, we obtain:
\begin{equation}\label{Bound dt k initial}
\| \partial k |_{x^0=0}\|_{H^{s-3}} \lesssim \mathcal D.
\end{equation}
Returning to \eqref{Second bound dt g initial}  and using \eqref{Bound transversal derivative k} to estimate the last term in the right hand side, we also obtain:
\begin{equation}\label{Bound dt g initial}
\|\partial g |_{x^0=0} \|_{H^{s-2}} \lesssim \mathcal D.
\end{equation}

\medskip
\noindent \textbf{Bounds for $\partial R|_{x^0=0}$ and $\partial R^\perp|_{x^0=0}$.}
Differentiating \eqref{Gauss}, we obtain the schematic relation:
\[
\partial R = \partial e \cdot k \cdot k + m(e) \cdot k \cdot \partial k.
\]
Thus, we can estimate (using, again, Lemma \ref{lem:Functional inequalities}):
\begin{align*}
\| \partial R|_{x^0=0} \|_{H^{s-4}}  & \lesssim \|  \partial e \cdot k \cdot k|_{x^0=0} \|_{H^{s-4}} + \| m(e) \cdot k \cdot \partial k |_{x^0=0} \|_{H^{s-4}} \\
& \lesssim  \|  \partial e \cdot k|_{x^0=0} \|_{H^{s-3}} \| k|_{x^0=0} \|_{H^{s-2}} + \| m(e) \cdot k|_{x^0=0} \|_{H^{s-2}} \| \partial k |_{x^0=0} \|_{H^{s-3}} \\
& \lesssim  \|  \partial e |_{x^0=0} \|_{H^{s-2}}  \| k|_{x^0=0} \|^2_{H^{s-2}} + \| m(e)|_{x^0=0}\|_{H^{s-1}} \| k|_{x^0=0} \|_{H^{s-2}} \| \partial k |_{x^0=0} \|_{H^{s-3}} \\
& \lesssim \mathcal D^2,
\end{align*}
where, in the last line above, we made use of the bounds \eqref{Bound partial e initial} for $\partial e$, \eqref{Bound dt g initial} for $\partial g$, \eqref{Bound k initial} and \eqref{Bound dt k initial} for $k$.
Similarly, differentiating the relation \eqref{Ricci} for $R^\perp$ and arguing in exactly the same way, we obtain
\[
\| \partial R^\perp|_{x^0=0} \|_{H^{s-4}} \lesssim \mathcal D^2.
\]
Combining the above relations, we get
\begin{equation}\label{Bound dR initial}
\| \partial R|_{x^0=0} \|_{H^{s-4}} + \| \partial R^\perp|_{x^0=0} \|_{H^{s-4}} \lesssim \mathcal D^2.
\end{equation}

\medskip
\noindent \textbf{Bound for $\partial^2 k|_{x^0=0}$.}
Using the wave equation \eqref{Covariant wave equation k} for $k$, we can express schematically:
\begin{equation}\label{Equation for d_0^2}
\partial_0^2 k = g \cdot \bar{\partial} \partial k + g\cdot  \partial g \cdot \partial k +  (g\cdot \partial^2 g+ g\cdot (\partial g)^2) \cdot k +g\cdot \omega \cdot \partial k + g\cdot (\partial \omega+\omega\cdot\omega) \cdot k.
\end{equation}
Using the functional inequalities from Lemma \ref{lem:Functional inequalities}, we can therefore estimate
\begin{align}
\| \partial_0^2 k |_{x^0=0} \|_{H^{s-4}} \lesssim \nonumber
& \|g|_{x^0=0} \|_{H^{s-1}} \| \bar \partial \partial k|_{x^0=0} \|_{H^{s-4}}
+\| g |_{x^0=0} \|_{H^{s-1}} \| \partial g |_{x^0=0} \|_{H^{s-2}} \|\partial k|_{x^0=0} \|_{H^{s-3}} \nonumber \\
& + \| g |_{x^0=0} \|_{H^{s-1}} \|\partial^2 g|_{x^0=0} \|_{H^{s-3}} \|k|_{x^0=0} \|_{H^{s-2}}
+ \| g |_{x^0=0} \|_{H^{s-1}}  \|\partial g|_{x^0=0} \|_{H^{s-2}}^2  \|k|_{x^0=0} \|_{H^{s-2}} 
   \nonumber \\
&+  \|g|_{x^0=0} \|_{H^{s-1}}\|\omega|_{x^0=0} \|_{H^{s-2}} \|\partial k|_{x^0=0} \|_{H^{s-3}}  + \|g|_{x^0=0} \|_{H^{s-1}} \| \partial \omega |_{x^0=0} \|_{H^{s-3}} \| k|_{x^0=0} \|_{H^{s-2}} 
\nonumber \\
& +  \|g|_{x^0=0} \|_{H^{s-1}}\|\omega |_{x^0=0} \|_{H^{s-2}}^2 \| k |_{x^0=0} \|_{H^{s-2}}   \nonumber \\
& \lesssim \mathcal D + \mathcal D \big( \|\partial^2 g|_{x^0=0} \|_{H^{s-3}} +  \|\partial \omega|_{x^0=0} \|_{H^{s-3}} \big),   \label{Bound two time derivatives k initial}
\end{align}
where, in the last line above, we used the bounds \eqref{Bound g initial} for $g$, \eqref{Bound dt g initial} for $\partial g$, \eqref{Bound omega initial} for $\omega$, \eqref{Bound k initial} for $k$ and \eqref{Bound dt k initial} for $\partial k$. Combining the above with the estimate for $\bar\partial \partial k|_{x^0=0}$ obtained from \eqref{Bound dt k initial}, we therefore infer:
\begin{equation}\label{First bound dt2 k initial}
\| \partial^2 k |_{x^0=0} \|_{H^{s-4}} \lesssim \mathcal D + \mathcal D \big( \|\partial^2 g|_{x^0=0} \|_{H^{s-3}} +  \|\partial \omega|_{x^0=0} \|_{H^{s-3}} \big).
\end{equation}

\medskip
\noindent \textbf{Bounds for $\partial^2 g|_{x^0=0}$, $\partial \omega|_{x^0=0}$ and $\partial^2 e|_{x^0=0}$.}
The relation \eqref{Equation dt2 bar g} for $\partial_0^2 \bar g$ takes the form
\[
\partial_0^2 \bar g = g \cdot D \partial  g + g\cdot R + g \cdot \partial g \cdot \partial g.
\]
Thus, in view of the bounds \eqref{Bound dt g initial} and \eqref{Bound R initial} for $\| \partial g\|_{H^{s-2}}$ and $\|R\|_{H^{s-3}}$, respectively, we can readily estimate
\begin{equation}\label{Bound dt2 bar g initial}
\| \partial_0^2 \bar g |_{x^0=0}\|_{H^{s-3}} \lesssim \mathcal D.
\end{equation}

The expression \eqref{Normal curvature coordinates} for $(R^{\perp})_{0i\hphantom{\bA}\bB}^{\hphantom{\a\b}\bA}$  yields the following relation for $\partial_0 \bar \omega$:
\[
\partial_0 \bar \omega =D \omega_0 + R^\perp + \omega \cdot \omega.
\]
Therefore, using the bounds \eqref{Bound omega initial} and \eqref{Bound R perp initial} for $\| \omega\|_{H^{s-2}}$ and $\|R^\perp\|_{H^{s-3}}$, we obtain
\begin{equation}\label{Bound d bar omega initial}
\| \partial_0 \bar \omega |_{x^0=0} \|_{H^{s-3}} \lesssim \mathcal D. 
\end{equation}

Differentiating \eqref{Relation e Omega k} with respect to $\partial_0$ and using the relation
\[
\partial_\a \partial_\b Y^A = k_{\a\b}^{\bA} e_{\bA}^A + \Gamma_{\a\b}^{\ga} \partial_\ga Y^A,
\]
 we obtain the schematic expression for $\partial_0^2 e$:
 \begin{align*}
 \partial_0^2 e = & \partial_0 \omega_0 \cdot e + \omega \cdot \partial e + g \cdot \partial g \cdot e \cdot k \cdot \partial Y \\
 & + g \cdot \partial e \cdot k \cdot \partial Y + g \cdot e \cdot \partial k \cdot \partial Y + g \cdot e \cdot k \cdot k \cdot \partial Y.
 \end{align*}
 Thus, using the functional inequalities from Lemma \ref{lem:Functional inequalities} together with the bounds \eqref{Boundary assumption G},  \eqref{Bound partial e initial},  \eqref{Bound omega initial}, \eqref{Final bound d Y}, \eqref{Bound g initial}, \eqref{Bound dt g initial}, \eqref{Bound k initial} and \eqref{Bound dt k initial}, we obtain
 \begin{equation}\label{First bound dt2 e initial}
 \| \partial^2 e|_{x^0=0} \|_{H^{s-3}} \lesssim \mathcal D + \| \partial_0 \omega_0 |_{x^0=0}\|_{H^{s-3}}.
 \end{equation}

The relation \eqref{Equation dt2 shift} for $\partial_0^2 \beta$ takes the schematic form
\[
\partial_0^2 \beta =  \JapD^{-1} \Big( (g-m_0) \cdot D\partial^2 g + g \cdot D^3 N +g\cdot D^2 h +g\cdot DR  + g\cdot  \partial g \cdot \partial^2 g \Big).
\]
Using as before the functional inequalities from Lemma \ref{lem:Functional inequalities} and the bounds \eqref{Bound g initial}, \eqref{Bound dt g initial}, \eqref{Bound h initial}  and \eqref{Bound dR initial}, we obtain
\begin{align}
\|\partial_0^2 \beta |_{x^0=0}\|_{H^{s-3}} \lesssim
& \| (g-m_0)|_{x^0=0}\|_{H^{s-1}} \cdot \|\partial^2 g|_{x^0=0}\|_{H^{s-3}} + \|g|_{x^0=0}\|_{H^{s-1}} \cdot \|\partial g|_{x^0=0}\|_{H^{s-2}} \nonumber \\
&+ \| g|_{x^0=0}\|_{H^{s-1}} \cdot \| h|_{x^0=0}\|_{H^{s-2}} 
+ \|g|_{x^0=0}\|_{H^{s-1}}\cdot \|R|_{x^0=0}\|_{H^{s-3}} \nonumber \\
&+ \|g|_{x^0=0}\|_{H^{s-1}}\cdot \|\partial g|_{x^0=0}\|_{H^{s-2}} \cdot \|\partial^2 g|_{x^0=0}\|_{H^{s-3}} \nonumber \\
\lesssim & \mathcal D + \mathcal D \|\partial^2 g|_{x^0=0}\|_{H^{s-3}}. \label{First bound dt2 shift initial}
\end{align}

The relation \eqref{Equation dt2 lapse} for $\partial_0^2 N$ takes the form
\begin{align*}
\partial_0^2 N = D\partial g + \JapD^{-1}\Big( &
 g\cdot  \partial_0^2 \tilde{\mathcal F}^\natural + g \cdot  \partial R\\
 &+ \partial g \cdot \partial g \cdot \tilde{\mathcal{F}}^{\natural} + g \cdot \partial^2 g \cdot \tilde{\mathcal{F}}^{\natural} + g \cdot \partial g \cdot \partial \tilde{\mathcal{F}}^{\natural} \\
& + (g-m_0) \cdot D\partial^2 g + g\cdot \partial R_*+\partial g \cdot \partial^2 g  \\
&  + \partial g \cdot \partial g \cdot \partial g+ D\partial g \cdot  \JapD^{-1} \big( g \cdot (\partial \tilde{\mathcal F}^{\natural}-R)\big) \Big)
\end{align*}
In view of the fact that $\tilde{\mathcal F}^\natural|_{x^0=0}=0$ and $\partial_0 ^m \tilde{\mathcal F}^\natural|_{x^0=0}=\partial_0 ^m \mathcal F^\natural|_{x^0=0}$ for $m\ge 1$ (see \eqref{Vanishing F initial} and \eqref{Vanishing F initial 2}), we obtain from the expression for $\partial_0^2 N$ above using, as before, the functional inequalities from Lemma \ref{lem:Functional inequalities} and the bounds \eqref{Bound g initial}, \eqref{Bound dt g initial}, \eqref{Bound h initial} and \eqref{Bound dR initial}, as well as the bound \eqref{Bound dt F natural initial} for $\partial_0 \mathcal F^\natural|_{x^0=0}$:
\begin{align}
\|\partial_0^2 N  |_{x^0=0}\|_{H^{s-3}} \lesssim &
\|\partial g  |_{x^0=0}\|_{H^{s-2}} + \|g  |_{x^0=0}\|_{H^{s-1}} \cdot \|\partial_0^2 \mathcal F^\natural   |_{x^0=0}\|_{H^{s-4}} + \|g  |_{x^0=0}\|_{H^{s-1}}\cdot \|\partial R  |_{x^0=0}\|_{H^{s-4}} \nonumber  \\
& + \|(g-m_0  )|_{x^0=0}\|_{H^{s-1}}\cdot \|\partial^2 g  |_{x^0=0}\|_{H^{s-3}}
+ \| g   |_{x^0=0}\|_{H^{s-1}}\cdot \|\partial R  |_{x^0=0}\|_{H^{s-4}} \nonumber \\
&+\|\partial g  |_{x^0=0}\|_{H^{s-2}} \|\partial^2 g   |_{x^0=0}\|_{H^{s-3}}  + \| \partial g  |_{x^0=0}\|_{H^{s-2}}^3 \nonumber  \\
&+ \|\partial g  |_{x^0=0}\|_{H^{s-2}} \cdot \|g  |_{x^0=0}\|_{H^{s-1}} \cdot \|\partial_0 \mathcal F^\natural   |_{x^0=0}\|_{H^{s-3}}  \nonumber \\
\lesssim & \mathcal D + \mathcal D \| \partial^2 g_{x^0=0}\|_{H^{s-3}}+ \|\partial_0^2 \mathcal F^\natural   |_{x^0=0}\|_{H^{s-4}}. \label{Almost there dt2 lapse}
\end{align}
Using the fact that, in view of \eqref{F natural seed}, $\mathcal F^\natural$ takes the form
\[
\mathcal F^\natural = \sum_j \sum_{j'>j} P_j(m(e)) \cdot P_j \big( P_{j'} k \cdot \JapD^{-1}P_{j'} k\big),
\]
we can readily estimate:
\begin{align}
\|\partial_0^2 \mathcal F^\natural |_{x^0=0}\|_{H^{s-4} } 
\lesssim  & \| \partial^2 e |_{x^0=0}\|_{H^{s-3}} \|k|_{x^0=0}\|^2_{H^{s-2}}
+ \| \partial e |_{x^0=0}\|_{H^{s-2}} \|\partial k|_{x^0=0}\|^2_{H^{s-3}} \| k|_{x^0=0}\|^2_{H^{s-2}}\nonumber \\
& + \|m(e)|_{x^0=0}\|_{H^{s-1}} \|\partial^2 k|_{x^0=0}\|_{H^{s-4}} \|k|_{x^0=0}\|_{H^{s-2}}   
+ \|m(e)|_{x^0=0}\|_{H^{s-1}} \|\partial k|_{x^0=0}\|_{H^{s-3}}^2  \nonumber  \\
\lesssim &  \mathcal D^2 + \mathcal D \big( \| \partial^2 g|_{x^0=0}\|_{H^{s-3}} + \|\partial \omega |_{x^0=0}\|_{H^{s-3}} \big), \label{Bound dt2 F natural initial}
\end{align}
where, in the last line above, we used the bounds \eqref{Boundary assumption G} for $e$, \eqref{Bound partial e initial}  for $\partial e$, \eqref{First bound dt2 e initial} for $\partial^2 e$, \eqref{Bound k initial} for $k$, \eqref{Bound dt k initial} for $\partial k$ and \eqref{First bound dt2 k initial} for $\partial^2 k$. Therefore, returning to \eqref{Almost there dt2 lapse}, we obtain
\begin{equation}\label{First bound dt2 lapse}
\|\partial_0^2 N  |_{x^0=0}\|_{H^{s-3}} \lesssim \mathcal D + \mathcal D \big( \| \partial^2 g|_{x^0=0}\|_{H^{s-3}} + \|\partial \omega |_{x^0=0}\|_{H^{s-3}} \big).
\end{equation}

The relation \eqref{Equation omega 0} for $\partial_0 \omega_0$ takes the form:
\begin{align*}
\partial_0 \omega_0 = 
D\omega + \JapD^{-1} \Big( &m(e) \cdot \partial_0 \tilde{\mathcal{F}}_\perp + g \cdot \partial R^\perp \\
& +  \partial g \cdot D \omega +D\partial g \cdot g\cdot \omega + \partial g \cdot \omega \cdot \omega\\
& +D(g\cdot \omega \cdot \omega) + \partial e \cdot \tilde{\mathcal F}_{\perp} + \partial g \cdot \tilde{\mathcal F}_{\perp}\Big).
\end{align*}
Thus, using the fact that $\tilde{\mathcal F}_\perp|_{x^0=0}=0$ and $\partial_0 ^m \tilde{\mathcal F}_\perp|_{x^0=0}=\partial_0 ^m \mathcal F_\perp |_{x^0=0}$ for $m\ge 1$ (see \eqref{Vanishing F initial} and \eqref{Vanishing F initial 2}), we obtain by arguing similarly as before:
\begin{align}
\|\partial_0 \omega_0|_{x^0=0}\|_{H^{s-3}} \lesssim 
& \|\omega|_{x^0=0}\|_{H^{s-2}} + \|m(e)|_{x^0=0}\|_{H^{s-1}} \cdot \|\partial_0 \mathcal{F}_\perp|_{x^0=0}\|_{H^{s-4}} + \|g|_{x^0=0}\|_{H^{s-1}} \cdot \|\partial R^\perp|_{x^0=0}\|_{H^{s-4}} \nonumber \\
& +  \|\partial g|_{x^0=0}\|_{H^{s-2}} \cdot \|\omega|_{x^0=0}\|_{H^{s-2}} +\|\partial g|_{x^0=0}\|_{H^{s-2}} \cdot \|g|_{x^0=0}\|_{H^{s-1}}\cdot \|\omega|_{x^0=0}\|_{H^{s-2}} \nonumber \\
& + \|\partial g|_{x^0=0}\|_{H^{s-2}} \cdot \|\omega|_{x^0=0}\|_{H^{s-2}}^2 +\|g|_{x^0=0}\|_{H^{s-1}}\cdot \|\omega|_{x^0=0}\|_{H^{s-2}}^2\Big)  \nonumber \\
\lesssim & \mathcal D + \|\partial_0 \mathcal{F}_\perp|_{x^0=0}\|_{H^{s-4}} \label{Almost there dt omega 0}
\end{align}
(having used the bounds \eqref{Boundary assumption G}, \eqref{Bound g initial}, \eqref{Bound dt g initial}, \eqref{Bound omega initial} and \eqref{Bound dR initial}). Using the fact that
\[
\mathcal F_{\perp} = \sum_j \sum_{j'>j} D\Big( g \cdot P_j g \cdot P_j \big( P_{j'} k \cdot \JapD^{-1} P_{j'} k \big)\Big)
\]
and arguing similarly as for the proof of \eqref{Bound dt F natural initial}, we can estimate:
\[
 \|\partial_0 \mathcal{F}_\perp|_{x^0=0}\|_{H^{s-4}} \lesssim \mathcal D^2.
\]
Therefore, returning to \eqref{Almost there dt omega 0}, we infer:
\begin{equation}\label{First bound dt omega 0}
\|\partial_0 \omega_0|_{x^0=0}\|_{H^{s-3}} \lesssim \mathcal D.
\end{equation}

Combining the bounds \eqref{Bound dt2 bar g initial} for $\partial_0^2 \bar g$, \eqref{Bound d bar omega initial} for $\partial_0 \bar\omega$, \eqref{First bound dt2 shift initial} for $\partial_0^2 \beta$, \eqref{First bound dt2 lapse} for $\partial_0^2 N$, \eqref{First bound dt omega 0} for $\partial_0 \omega_0$ and \eqref{First bound dt2 e initial} for $\partial_0^2 e$, as well as the bounds \eqref{Bound dt g initial} for $\bar\partial \partial g$, \eqref{Bound omega initial} for $\bar\partial \omega$ and \eqref{Bound partial e initial} for $\bar\partial \partial e$, we obtain:
\[
\|\partial^2 g |_{x^0=0}\|_{H^{s-3}} + \|\partial \omega |_{x^0=0}\|_{H^{s-3}} + \|\partial^2 e |_{x^0=0}\|_{H^{s-3}} \lesssim \mathcal D + \mathcal D \Big( \|\partial^2 g |_{x^0=0}\|_{H^{s-3}} + \|\partial \omega |_{x^0=0}\|_{H^{s-3}}\Big).
\]
Provided $\epsilon$ (and thus $\mathcal D$) in \eqref{D} is smaller than a constant depending only on $s$, we thus infer that:
\begin{equation}\label{Bound dt2 g omega e initial}
\|\partial^2 g |_{x^0=0}\|_{H^{s-3}} + \|\partial \omega |_{x^0=0}\|_{H^{s-3}} + \|\partial^2 e |_{x^0=0}\|_{H^{s-3}} \lesssim \mathcal D.
\end{equation}
Returning to \eqref{First bound dt2 k initial}, we also obtain:
\begin{equation}\label{Bound dt2 k initial}
\|\partial^2 k |_{x^0=0}\|_{H^{s-4}} \lesssim \mathcal D.
\end{equation}

\medskip
\noindent \textbf{Finishing the proof of Proposition \ref{prop: Bounds S0}.}
The estimates of Proposition \ref{prop: Bounds S0} will now follow by combining the bounds we have already established:
\begin{itemize}
\item The bound \eqref{Bound Y initial} follows from the expression
\begin{align*}
\partial_{\a} \partial_{\b} Y^A 
& = \Pi_{B}^{\ga}\partial_{\a} \partial_{\b} Y^B \partial_{\ga} Y^A + (\Pi^{\perp})_{B}^{\bA}\partial_{\a} \partial_{\b} Y^B e_{\bA}^A\\
&= \Gamma^{\ga}_{\a\b} \partial_{\ga} Y^A + k^{\bA}_{\a\b} e_{\bA}^A
\end{align*}
for $\partial^2 Y$, combined with the bounds \eqref{Final bound d Y}, \eqref{Bound dt g initial}, \eqref{Bound k initial}, \eqref{Bound partial e initial} and \eqref{Boundary assumption G}. 
\item The bound \eqref{Bound k and dk initial} follows by combining \eqref{Bound k initial}, \eqref{Bound dt k initial} and \eqref{Bound dt2 k initial}.
\item The bound \eqref{Bound metric initial} follows by combining \eqref{Bound g initial}, \eqref{Bound dt g initial}, \eqref{Bound dR initial} and \eqref{Bound dt2 g omega e initial}.
\item The bound \eqref{Bound connection coefficients initial} follows by combining \eqref{Bound omega initial},  \eqref{Bound dR initial} and \eqref{Bound dt2 g omega e initial}. 
\item The bound \eqref{Bound frame} follows by combining \eqref{Bound partial e initial} and \eqref{Bound dt2 g omega e initial}.
\end{itemize}

\end{proof}

\section{Step 4: The bootstrap estimates} \label{sec:Bootstrap}

In this section, we will implement the bootstrap argument which lies at the heart of the proof of Theorem \ref{thm:Existence}. In particular, the following proposition will amount to Step 4 of our proof of Theorem \ref{thm:Existence} as outlined in Section \ref{sec:Continuity}:

\begin{proposition}\label{prop:Bootstrap}

For $T\in (0,1]$, let $Y: [0,T)\times \mathbb{T}^3 \rightarrow \mathcal{N}$ be a smooth development of an initial data pair $(\bY, \bn)$ as in the statement of Theorem \ref{thm:Existence} and $\{e_{\bA}\}_{\bA=4}^n$ be a smooth frame for the normal bundle $NY$. Let $C_0 > 1$ be a constant depending only on $s$ which is large compared to the constants $C_1$ and $C$ of Proposition \ref{prop: Bounds S0}.\footnote{The constant $C_0$ is independent, in particular, of $\mathcal{D}$ and $\epsilon$, where $\mathcal{D}$ is defined by \eqref{D}.} 

Assume that the following conditions hold:
\begin{itemize}
\item \textbf{Gauge condition:} The gauge $\mathcal G = (\text{Id}, e)$ for the immersion $Y$ satisfies the balanced gauge condition introduced in Definition \ref{def:Balanced gauge}.
\item \textbf{Initial bounds for the gauge:} Along the initial slice $\{ x^0=0\}$, the bound \eqref{Boundary assumption G} holds.
\item \textbf{Bootstrap assumption:}
 The immersion $Y$ satisfies
 \begin{align}\label{Bootstrap bound}
 \mathcal{Q}_g  &+ \mathcal{Q}_k  + \mathcal{Q}_\perp + \| Y - Y_0 \|_{L^\infty L^\infty}\\
 &  + \| \partial Y - \partial Y_0 \|_{L^\infty H^{s-1}} + \|\partial^2 Y\|_{L^4 W^{\f1{12},4}}+\|\partial^2 Y\|_{L^{\f72} L^{\f{14}3}}\le C_0 \mathcal{D} \nonumber 
\end{align}
(for the definition of the quantities $\mathcal Q_g$, $\mathcal Q_k$, $\mathcal Q_\perp$, see Definition \ref{def:Norms 1}).
\end{itemize}
Then, provided $\epsilon$ is sufficiently small with respect to $C_0$, the following improvement of \eqref{Bootstrap bound} holds:
\begin{align}\label{Bootstrap conclusion}
\mathcal{Q}_g  &+ \mathcal{Q}_k  + \mathcal{Q}_\perp + \| Y - Y_0 \|_{L^\infty L^\infty} \\
 &  + \| \partial Y - \partial Y_0 \|_{L^\infty H^{s-1}} + \|\partial^2 Y\|_{L^4 W^{\f1{12},4}}+\|\partial^2 Y\|_{L^{\f72} L^{\f{14}3}}\le \f12C_0 \mathcal{D}. \nonumber 
\end{align}
\end{proposition}

\noindent The proof of Proposition \ref{prop:Bootstrap} will occupy the rest of this section.

\begin{remark*}
Throughout the proof of Proposition  \ref{prop:Bootstrap}, we will assume that $\epsilon$ (appearing in \eqref{D}) is sufficiently small compared to the constant $C_0$. We will also use the notation $\lesssim$ to denote an inequality where the implicit constant  is \textbf{independent} of $C_0$ and $\mathcal D$ (but is allowed to depend on $s$).
\end{remark*}

\subsection{Energy and Strichartz estimates for $k$.} \label{subsec:Strichartz}
Let $Y:[0,T)\times \mathbb T^3 \rightarrow\mathcal N$ and $\{ e_{\bA}\}_{\bA=4}^n $ be as in the statement of Proposition \ref{prop:Bootstrap}. The covariant wave equation \eqref{Covariant wave equation k} for the second fundamental form implies that the components $k^{\bar{A}}_{\a\b}$ of $k$ satisfy a system of wave equations that takes the following schematic form: 
\begin{equation} \label{Wave equation k}
\square_g k =  g\cdot \partial g \cdot \partial k + g\cdot \omega \cdot \partial k + \big( g\cdot \partial^2 g + g\cdot \partial g \cdot \partial g\big) \cdot k + g\cdot (\partial \omega+ \omega \cdot \omega) \cdot k.  
\end{equation}

\begin{remark*}
Throughout this section, we will make use of the shorthand notation $f_j$ and $f_{\le j}$ for the Littlewood--Paley projections $P_j f$ and $\sum_{j' \le j} P_{j'} f$, respectively, of a function $f:[0,T) \times \mathbb T^3 \rightarrow \mathbb R$.
\end{remark*}

The following lemma provides an energy estimate for equation \eqref{Wave equation k}, using the bounds for the initial data of $k$ established in Proposition \ref{prop: Bounds S0}:

\begin{lemma}\label{lem:Energy estimates}
The second fundamental form $k$ of $Y$ satisfies the following estimates: 
\begin{equation}\label{Energy estimate k}
\| k \|_{L^\infty H^{s-2}}+ \| \partial k \|_{L^\infty H^{s-3}}+\| \partial^2 k \|_{L^\infty H^{s-4}} \le C \mathcal{D},
\end{equation}
where $C>0$ is a constant independent of $C_0$.
\end{lemma}

\begin{proof}
Applying a spatial Littlewood--Paley projection $P_j$ on equation \eqref{Wave equation k} and using the shorthand notation $f_j$ and $f_{\le j}$ for $P_j f$ and $\sum_{j' \le j} P_{j'} f$, respectively, we obtain, schematically:

\begin{align}\label{Wave equation projection}
\square_{g_{\le j}} k_j =& (g-m_0)_j \partial^2 k_{\le j} +\sum_{j' > j}  P_j \big((g-m_0)_{j'} \partial^2 k_{j'}\big)+(g\cdot \partial g)_j \partial k_{\le j} +\sum_{j' > j}  P_j \big((g\cdot \partial g)_{j'} \partial k_{j'}\big)\\
&+ P_j \big( g \cdot \partial g \cdot \partial k \big) + P_j \big( g\cdot \omega \cdot \partial k \big) + P_j \big( \big( g\cdot \partial^2 g + g\cdot \partial g \cdot \partial g\big) \cdot k) + P_j \big( g\cdot(\partial\omega +\omega \cdot \omega) \cdot k\big), \nonumber
\end{align} 
where $m_0=-(dx^0)^2 + \sum_{i=1}^3 (dx^i)^2$.

We will first establish that
\begin{equation}\label{Energy estimate k lower}
\sum_j \| k_j \|^2_{L^\infty H^{s-2}}+\sum_j \| \partial k_j \|^2_{L^\infty H^{s-3}} \lesssim \mathcal{D}^2.
\end{equation}
In order to obtain \eqref{Energy estimate k lower}, we multiply \eqref{Wave equation projection} with $2^{2(s-3)j} \partial_{0} k_j$ and integrate by parts over $[0,\tau]\times \mathbb{T}^3$, $0<\tau<T$. In view of the fact that $\| g - m_0 \|_{L^\infty L^\infty} \ll 1$ as a consequence of the bootstrap assumption \eqref{Bootstrap bound}  (and thus the energy fluxes of $\square_{g_\le j}$ and $\square_{m_0}$ are comparable as quadratic forms with a constant of proportionality close to 1), we obtain via this procedure after summing over $j$:
\begin{align}\label{Error energy estimates localized}
\sum_j \| 2^{(s-3)j} \partial k_j \|^2_{L^\infty L^2} \lesssim & \sum_j \| 2^{(s-3)j} \partial k_j |_{x^0=0}\|^2_{L^2}\\
 &+\sum_j\Big| 2^{2(s-3)j} \int g_j \cdot  \partial^2  k_{\le j}\cdot  \partial k_j \, dx \Big| 
  +\sum_j \sum_{j'>j} \Big| 2^{2(s-3)j} \int P_j \big( g_{j'} \partial^2  k_{j'}\big)\cdot  \partial k_j \, dx \Big| \nonumber\\
& +\sum_j\Big| 2^{2(s-3)j} \int (g\cdot \partial g)_j \cdot  \partial k_{\le j}\cdot  \partial k_j \, dx \Big| 
  + \sum_j\sum_{j'>j} \Big| 2^{2(s-3)j} \int P_j \big((g\cdot \partial g)_{j'} \partial k_{j'}\big)\cdot  \partial k_j \, dx \Big| \nonumber\\
& +\sum_j\Big| 2^{2(s-3)j} \int P_j \big( g\cdot \partial g \cdot  \partial k\big) \cdot  \partial k_j \, dx \Big| 
+  \sum_j\Big| 2^{2(s-3)j} \int P_j \big(\omega \cdot  \partial k\big) \cdot  \partial k_j \, dx \Big|  \nonumber\\
& +\sum_j\Big| 2^{2(s-3)j} \int P_j \big( (g\cdot \partial^2 g+ g\cdot \partial g\cdot \partial g)  \cdot  k\big) \cdot  \partial k_j \, dx \Big|  \nonumber\\
&+  \sum_j\Big| 2^{2(s-3)j} \int P_j \big((\partial \omega + \omega \cdot \omega \cdot  k\big) \cdot  \partial k_j \, dx \Big|.  \nonumber
\end{align}

The estimate \eqref{Energy estimate k lower} will follow once we bound the right hand side of \eqref{Error energy estimates localized} by $C \mathcal{D}^2$.

\begin{itemize}
\item In view of the initial bound \eqref{Bound k and dk initial} for $\partial k|_{x^0=0}$, we can bound the first term of the right hand side of \eqref{Error energy estimates localized} as follows:
\[
\sum_j \| 2^{(s-3)j} \partial k_j |_{x^0=0}\|^2_{L^2} \lesssim \| \partial k|_{x^0=0} \|^2_{H^{s-3}} \lesssim \mathcal{D}^2.
\]

\item For the high-low term $ g_j \cdot \partial^2 k_{\le j}$ in \eqref{Error energy estimates localized}, we have:
\begin{align}\label{First high low}
\sum_j\Big| & 2^{2(s-3)j} \int  (g-m_0)_j \cdot \partial^2 k_{\le j}\cdot  \partial k_j \, dx \Big|  \\
& \le \sum_j\Big|  \int 2^{(s-\f12-2\delta_0)j} (g-m_0)_j \cdot  2^{(3\delta_0-\f52) j} \partial^2 k_{\le j}\cdot  2^{(s-3-\delta_0)j}\partial k_j \, dx \Big| \nonumber \\
& \lesssim \sum_j \|  (g-m_0)_j \|_{L^2 H^{s-\f12-2\delta_0}} \| \partial^2 k_{\le j} \|_{L^2 W^{3\delta_0-\f52, \infty}} 2^{-\delta_0 j}\| \partial k_j \|_{L^\infty H^{s-3\phantom{\f12}}} \nonumber \\
& \lesssim \| \partial^2 k \|_{L^2 W^{3\delta_0-\f52, \infty}}  \Big( \sum_j 2^{-\delta_0 j} \| (g-m_0)_j \|^2_{L^2 H^{s-\f12-2\delta_0}}\Big)^{\f12} \cdot \Big( \sum_j 2^{-\delta_0 j}\| \partial k_j \|^2_{L^\infty H^{s-3\phantom{\f12}}}\Big)^{\f12} \nonumber \\
& \lesssim  \| \partial^2 k \|_{L^2 W^{s_0-\f52, \infty}}  \| g-m_0\|_{L^2 H^{s-\f12-2\delta_0}}\| \partial k \|_{L^\infty H^{s-3 \phantom{\f12}}} \nonumber \\
& \lesssim C_0^3 \mathcal{D}^3.\nonumber
\end{align}
In the above, we made use of the bootstrap bound \eqref{Bootstrap bound} for $ \| g-m_0\|_{L^2 H^{2+\f16+s_1}}$, $\| k \|_{L^2 W^{-\f12, \infty}}$ and $\| \partial k \|_{L^\infty H^{s-3}}$, as well as the fact that $\delta_0$ was chosen to satisfy
\[
0 < \delta_0 < \f1{10} \min \big\{ (s - \f52 -\f16), 1\big\}
\]
in Definition \ref{def:Small constants}.

\item For the high-high term $P_j \big( (g-m_0)_{j'} \partial^2 k_{j'}\big)$ in \eqref{Error energy estimates localized}, we can estimate:
\begin{align}\label{First high high}
\sum_j \sum_{j'>j} \Big| & 2^{2(s-3)j} \int P_j \big( (g-m_0)_{j'} \partial^2 k_{j'}\big)\cdot  \partial k_j \, dx \Big| \\
& \le \sum_j \sum_{j'>j}   \Big| 2^{(s-\f52 + \f1{12}+\delta_0)j}\int P_j \big(  (g-m_0)_{j'} \cdot  \partial^2 k_{j'}\big)\cdot  2^{(s-\f72-\f1{12}-\delta_0) j}\partial k_j \, dx \Big| \nonumber \\
& \lesssim \sum_j \sum_{j'>j}  \|2^{(s-\f52 + \f1{12}+\delta_0)j} P_j \big(   (g-m_0)_{j'} \cdot  \partial^2 k_{j'}\big)\|_{L^{\f43}L^{\f43}} \| 2^{(s-\f72-\f1{12}-\delta_0) j}\partial k_j \|_{L^4 L^4} \nonumber \\
& \lesssim \sum_j \sum_{j'>j}  2^{(s-\f52 + \f1{12}+\delta_0)(j-j')}\| 2^{(s-\f12-2\delta_0)j'}(g-m_0)_{j'}\|_{L^2 L^2} \Big( 2^{-\delta_0 j'} \|  2^{(-2+\f1{12}+4\delta_0 )j'}\partial^2 k_{j'}\|_{L^4 L^4} \nonumber \\
&\hphantom{ \lesssim \sum_j \sum_{j'>j}  2^{(s-\f52 + \f1{12}+\delta_0)(j-j')}\| 2^{(s-\f12-2\delta_0)j'}(g}
\times \| 2^{(s-\f72-\f1{12}-\delta_0) j}\partial k_j \|_{L^4 L^4}\Big) \nonumber \\
& \lesssim \|  g-m_0 \|_{L^2 H^{s-\f12-2\delta_0}} \big( \| \partial^2 k\|^2_{L^4 W^{-2+4\delta_0+\f1{12},4}} + \| \partial k\|^2_{L^4 W^{s-\f72-\f1{12}-\delta_0,4}}\big) \nonumber \\
& \lesssim C_0^3 \mathcal{D}^3\nonumber 
\end{align}
where, in the last step, we made use of \eqref{Bootstrap bound}.

\item For the high-low term $(g\cdot \partial g)_j \cdot \partial k_{\le j}$ in \eqref{Error energy estimates localized}, we have, similarly as in \eqref{First high low}:
\begin{align}\label{Second high low}
\sum_j\Big| 2^{2(s-3)j} \int (g \cdot \partial g)_j \cdot  \partial k_{\le j}\cdot  \partial k_j \, dx \Big| 
&\le \sum_j\Big|  \int 2^{(s-\f32-2\delta_0)j}(g \cdot \partial g)_j \cdot  2^{(2\delta_0-\f32) j} \partial k_{\le j}\cdot  2^{(s-3)j}\partial k_j \, dx \Big| \\
& \lesssim \sum_j 2^{-\delta_0 j} \| (g \cdot \partial g)_j \|_{L^2 H^{s-\f32-2\delta_0}} \| \partial k_{\le j} \|_{L^2 W^{3\delta_0-\f32, \infty}} \| \partial k_j \|_{L^\infty H^{s-3\phantom{\f12}}} \nonumber \\
& \lesssim  \| \partial k \|_{L^2 W^{3\delta_0-\f32, \infty}} \big( \| g \cdot  \partial g\|^2_{L^2 H^{s-\f32-2\delta_0}} + \| \partial k \|^2_{L^\infty H^{s-3 \phantom{\f12}}}\big)  \nonumber \\
& \lesssim C_0^3 \mathcal{D}^3.\nonumber
\end{align}

\item For the high-high term $P_j \big((g \cdot \partial g)_{j'} \partial k_{j'}\big)$ in \eqref{Error energy estimates localized}, we have, similarly as in \eqref{First high high}:
\begin{align}\label{Second high high}
\sum_j \sum_{j'>j} \Big| & 2^{2(s-3)j} \int P_j \big((g\cdot \partial g)_{j'} \partial k_{j'}\big)\cdot  \partial k_j \, dx \Big| \\
& \le \sum_j \sum_{j'>j}  2^{-\delta_0 j'+(s-\f52+\f1{12}+\delta_0)(j-j')} \Big| \int P_j \big( 2^{(s-\f32-2\delta_0) j'} (g \cdot \partial g)_{j'} \cdot 2^{(4\delta_0-1+\f1{12})j'} \partial k_{j'}\big)\cdot  2^{(s-\f72-\f1{12}-\delta_0) j}\partial k_j \, dx \Big| \nonumber \\
& \lesssim \sum_j \sum_{j'>j}   2^{-\delta_0 j'+(s-\f52+\f1{12}+\delta_0)(j-j')}  \| (g\cdot \partial g)_{j'} \|_{L^2 H^{s-\f32-2\delta_0}} \| \partial k_{j'}\|_{L^4 W^{4\delta_0-1+\f1{12},4}} \| \partial k_j \|_{L^4 W^{s-\f72-\f1{12}-\delta_0}} \nonumber \\
& \lesssim \| g\cdot \partial g \|_{L^2 H^{s-\f32-2\delta_0}} \| \partial k\|_{L^4 W^{4\delta_0-1+\f1{12},4}} \| \partial k\|_{L^4 W^{s-\f72-\f1{12}-\delta_0}} \nonumber \\
& \lesssim C_0^3 \mathcal{D}^3.\nonumber
\end{align}

\item The term involving $P_j \big( g\cdot  \partial g \cdot \partial k \big)$ can be expanded in a high-low decomposition as follows:
\begin{align}\label{First term decomposition}
\sum_j\Big|&  2^{2(s-3)j} \int P_j \big( g\cdot  \partial g \cdot  \partial k\big) \cdot  \partial k_j \, dx \Big|  \le
\sum_j\Big| 2^{2(s-3)j} (g\cdot \partial g)_{\le j} \cdot  \partial k_j \cdot  \partial k_j \, dx \Big|\\
& +\sum_j\Big| 2^{2(s-3)j} (g\cdot \partial g)_j \cdot  \partial k_{\le j} \cdot  \partial k_j \, dx \Big|+\sum_j\sum_{j'>j} \Big| 2^{2(s-3)j} \int P_j \big( (g\cdot \partial g)_{j'} \cdot  \partial k_{j'}\big) \cdot  \partial k_j \, dx \Big|.\nonumber
\end{align}
The last two terms in the right hand side can be estimated exactly as in \eqref{Second high low} and \eqref{Second high high}, respectively, while for the first term we have
\begin{align}\label{Third low high}
\sum_j\Big| 2^{2(s-3)j} (g\cdot \partial g)_{\le j} \cdot  \partial k_j \cdot  \partial k_j \, dx \Big| 
&\le \sum_j\Big|  \int (g\cdot \partial g)_{\le j} \cdot  2^{(s-3) j} \partial k_{j}\cdot  2^{(s-3)j}\partial k_j \, dx \Big| \\
& \lesssim \sum_j \| (g\cdot \partial g)_{\le j} \|_{L^1 L^{\infty}} \| \partial k_j \|_{L^\infty H^{s-3}} \| \partial k_j \|_{L^\infty H^{s-3}} \nonumber \\
& \lesssim \| g\cdot \partial g\|_{L^1 L^{\infty}} \sum_j \| \partial k_j \|^2_{L^\infty H^{s-3}}   \nonumber \\
& \lesssim C_0 \mathcal{D} \sum_j \| \partial k_j \|^2_{L^\infty H^{s-3}}.\nonumber
\end{align}
Thus, we can bound the whole right hand side of \eqref{First term decomposition}: 
\begin{equation}\label{Third term}
\sum_j\Big| 2^{2(s-3)j} \int P_j \big( g\cdot \partial g \cdot  \partial k\big) \cdot  \partial k_j \, dx \Big| \lesssim C_0^3 \mathcal{D}^3 + C_0 \mathcal{D} \sum_j \| \partial k_j \|^2_{L^\infty H^{s-3}}.
\end{equation}
Arguing in exactly the same way (replacing the bounds for $\partial g$ by the corresponding bounds for $\omega$ from \eqref{Bootstrap bound}), we can also estimate:
\begin{equation}\label{Fourth term}
\sum_j\Big|  2^{2(s-3)j} \int P_j \big(g\cdot \omega \cdot  \partial k\big) \cdot  \partial k_j \, dx \Big| \lesssim C_0^3 \mathcal{D}^3 + C_0 \mathcal{D} \sum_j \| \partial k_j \|^2_{L^\infty H^{s-3}}.
\end{equation}

\item Let us decompose 
\begin{align}\label{Decomposition d2 g k}
P_j \big( \big( g\cdot \partial^2 g + g\cdot (\partial g)^2\big) \cdot  k\big) =  \big( g\cdot \partial^2 g & + g\cdot (\partial g)^2\big)_{\le j} \cdot  k_j +  \big( g\cdot \partial^2 g + g\cdot (\partial g)^2\big)_j \cdot  k_{\le j} \\
& + \sum_{j'>j} P_j  \big(\big( g\cdot \partial^2 g + g\cdot (\partial g)^2\big)_{j'} \cdot  k_{j'}\big).\nonumber 
\end{align}
 Using (in place of \eqref{Third low high}) the low-high estimate
\begin{align}\label{Low high L74}
\sum_j\Big| 2^{2(s-3)j} & \big(g\cdot \partial^2 g+ g\cdot (\partial g)^2\big)_{\le j} \cdot   k_j \cdot  \partial k_j \, dx \Big| \\
& \lesssim \sum_j \| \big(g\cdot \partial^2 g\big)_{\le j} \|_{L^{\f{7+12\delta_0}{4+12\delta_0}}L^{\f73+4\delta_0}}  \| 2^{(s-3)j} k_j \|_{L^{\f73+4\delta_0}L^{\f{14+24\delta_0}{1+12\delta_0}}} \| 2^{(s-3)j} \partial k_j \|_{L^\infty L^2} 
  \nonumber \\
&\hphantom{ \lesssim \sum_j}
+ \sum_j \|\big(g\cdot (\partial g)^2\big)_{\le j}\|_{L^1 L^3}  \| 2^{(s-3)j} k_j \|_{L^\infty L^6}\| 2^{(s-3)j} \partial k_j \|_{L^\infty L^2}
  \nonumber\\
& \lesssim \sum_j \| g\cdot \partial^2 g\big  \|_{L^{\f{7}{4}}W^{s_1,\f73}} \|2^{(s-3-\f16-2\delta_0)j}k_j\|^{\f6{7+6\delta_0}}_{L^2 L^\infty}  \| 2^{(s-2)j} k_j \|_{L^\infty L^2}^{\f{1+6\delta_0}{7+6\delta_0}} \| 2^{(s-3)j} \partial k_j \|_{L^\infty L^2} 
  \nonumber \\
&\hphantom{ \lesssim \sum_j}
+ \sum_j \|g\|_{L^\infty L^\infty} \|\partial g\|_{L^2 H^1}^2  \|  k_j \|_{L^\infty H^{s-2}}\| \partial k_j \|_{L^\infty H^{s-3}}
  \nonumber\\
& \lesssim \| g\cdot \partial^2 g\big  \|_{L^{\f{7}{4}}W^{s_1,\f73}} \sum_j  2^{-\delta_0 j}\|k_j\|^{\f6{7+6\delta_0}}_{L^2 W^{-\f12+s_0,\infty}}  \| k_j \|_{L^\infty H^{s-2}}^{\f{1+6\delta_0}{7+6\delta_0}} \|  \partial k_j \|_{L^\infty H^{s-3}} 
  \nonumber \\
&\hphantom{ \lesssim \sum_j}
+C_0 \mathcal D \sum_j  \|  k_j \|_{L^\infty H^{s-2}}\| \partial k_j \|_{L^\infty H^{s-3}}
  \nonumber\\
& \lesssim C_0^3 \mathcal{D}^3+ C_0 \mathcal{D}  \sum_j \big(\| \partial k_j \|^2_{L^\infty H^{s-3}} +\| k_j \|^2_{L^\infty H^{s-2}}\big)    \nonumber
\end{align}
and arguing as in the proof of \eqref{Third term} for the rest of the terms in\eqref{Decomposition d2 g k} (but following \eqref{First high low} and \eqref{First high high} in place of \eqref{Second high low} and \eqref{Second high high} for the high-low and high-high interactions, respectively), we infer that the last terms in the right hand side of \eqref{Error energy estimates localized} satisfy:
\begin{equation}
\sum_j\Big|  2^{2(s-3)j} \int P_j \big( \big( g\cdot \partial^2 g + g\cdot (\partial g)^2\big) \cdot  k\big) \cdot  \partial k_j \, dx \Big| \lesssim C_0^3 \mathcal{D}^3 + C_0 \mathcal{D} \sum_j \big(\| \partial k_j \|^2_{L^\infty H^{s-3}} + \| k_j \|^2_{L^\infty H^{s-2}}\big)
\end{equation}
and, similarly,
\begin{equation}\label{Final estimate for P j energy}
\sum_j\Big|  2^{2(s-3)j} \int P_j \big( g\cdot(\partial\omega +\omega^2) \cdot  k\big) \cdot  \partial k_j \, dx \Big| \lesssim C_0^3 \mathcal{D}^3 + C_0 \mathcal{D}  \sum_j \big(\| \partial k_j \|^2_{L^\infty H^{s-3}} +\| k_j \|^2_{L^\infty H^{s-2}}\big).
\end{equation}
\end{itemize}

Combining the above estimates to bound the right hand side of \eqref{Error energy estimates localized}, we obtain:
\begin{equation}\label{FK}
\sum_j \| \partial k_j \|^2_{L^\infty H^{s-3}} \lesssim \mathcal{D}^2 + C_0^3 \mathcal{D}^3 + C_0 \mathcal{D} \sum_j \big( \| \partial k_j \|^2_{L^\infty H^{s-3}} + \| k_j \|^2_{L^\infty H^{s-2}}\big).
\end{equation}
Using the trivial inequalities
\[
\| k_j \|_{L^\infty H^{s-3}} \lesssim \| \partial_0 k_j \|_{L^1 H^{s-3}} + \| k_j|_{x^0=0}\|_{H^{s-3}}
\]
and
\[
\| k_j \|_{L^\infty H^{s-2}} \lesssim \| \partial k_j \|_{L^\infty H^{s-3}} + \| k_j \|_{L^\infty H^{s-3}},
\]
we infer from \eqref{FK}:
\begin{equation}\label{FK 2}
\sum_j \big( \| \partial k_j \|^2_{L^\infty H^{s-3}} + \| k_j \|^2_{L^\infty H^{s-2}}) \lesssim \mathcal{D}^2 + C_0^3 \mathcal{D}^3 + C_0 \mathcal{D} \sum_j \big( \| \partial k_j \|^2_{L^\infty H^{s-3}} + \| k_j \|^2_{L^\infty H^{s-2}}\big).
\end{equation}

The bound \eqref{Energy estimate k lower} now follows immediately from \eqref{FK 2} provided $\epsilon$ has been chosen small enough in terms of $C_0$ (so that the last term of the right hand side can be absorbed into the left hand side). From \eqref{Energy estimate k lower} we immediately deduce that
\begin{equation}\label{First derivatives energy}
 \| \partial k \|_{L^\infty H^{s-3}}+\| k \|_{L^\infty H^{s-2}} \lesssim \mathcal{D}.
\end{equation}

In order to conclude the proof of Lemma \ref{lem:Energy estimates}, it remains to show that
\begin{equation}\label{Second derivatives energy}
\| \partial^2 k \|_{L^\infty H^{s-4}} \lesssim \mathcal{D}.
\end{equation}
Since \eqref{First derivatives energy} already implies that
\[
\| \bar{\partial} \partial k \|_{L^\infty H^{s-4}} \lesssim \mathcal{D}
\]
where $\bar{\partial} \in \{ \partial_1, \partial_2, \partial_3 \}$, we have to show that
\begin{equation}\label{Time derivatives energy}
\| \partial^2_0 k \|_{L^\infty H^{s-4}} \lesssim \mathcal{D}.
\end{equation}
Using the expression \eqref{Equation for d_0^2} for $\partial^2_0 k$ and repeating the steps which led us to proving the bound \eqref{Bound two time derivatives k initial} for $\partial_0^2 k |_{x^0=0}$, we can estimate for any $\tau \in [0,T]$:
\begin{equation}\label{2 time derivatives k}
\| \partial_0^2 k |_{x^0=\tau} \|_{H^{s-4}} \lesssim \| \bar{\partial} \partial k |_{x^0=\tau} \|_{H^{s-4}} + \mathcal D \lesssim \| \partial k \|_{L^\infty H^{s-3}} + \mathcal D \lesssim \mathcal{D}.
\end{equation}
Thus, we infer \eqref{Second derivatives energy}.

\end{proof}

The following lemma provides a Strichartz estimate for equation \eqref{Wave equation k}; the proof uses the results of \cite{Ta} as a black box.

\begin{lemma}\label{lem:Strichartz estimates}
Let $Y:[0,T)\times \mathbb T^3 \rightarrow \mathcal N$ be as in the statement of Proposition \ref{prop:Bootstrap}. Then 
\begin{equation}\label{Strichartz estimate k endpoint}
 \sum_{l=0}^2 \|\partial^l k \|_{L^2 W^{- \f12-l+s_0, \infty}}  \le C \mathcal{D}, 
\end{equation}
where $C>0$ is a constant independent of $C_0$. By interpolating between the energy estimate \eqref{Energy estimate k} and \eqref{Strichartz estimate k endpoint}, we also obtain:
\begin{equation}\label{Strichartz estimate k L4L4}
\sum_{l=0}^2 \Big( \|\partial^l k \|_{L^4 W^{\f{1}{12}-l+s_0,4}} + \|\partial^l k \|_{L^{\f72} W^{-l+s_0,\f{14}3}}\Big) \le C \mathcal{D}.
\end{equation}
\end{lemma}

\begin{proof}
We will adopt the shorthand notation $f_j$ and $f_{\le j}$ for the Littlewood--Paley projections $P_j f$ and $P_{\le j} f$, as we did  in the proof of Lemma \ref{lem:Energy estimates}. 

We will first show that, for any $j\in \mathbb N$,
\begin{equation}\label{Strichartz first derivative}
 \| \partial k_j \|_{L^2 W^{- \f32+s_0+\f12\delta_0, \infty}}  \lesssim \mathcal{D}.
\end{equation}
Recall that the Littlewood-Paley projections $k_j$ of the components of $k$ satisfy 
\begin{equation}\label{Wave equation projection 2}
\square_{g_{\le j}} k_j = F_j, 
\end{equation} 
where
\begin{align}\label{F}
F_j \doteq &  (g-m_0)_j \partial^2  k_{\le j} +\sum_{j' > j}  P_j \big((g-m_0)_{j'} \partial^2  k_{j'}\big)+(g\cdot \partial g)_j \partial k_{\le j} +\sum_{j' > j}  P_j \big((g\cdot \partial g)_{j'} \partial k_{j'}\big)\\
& + P_j \big( ( g\cdot\partial^2 g +  g\cdot(\partial g)^2) \cdot k \big) + P_j \big(g\cdot (\partial\omega+\omega^2) \cdot k\big) + P_j \big( g\cdot \partial g \cdot \partial k \big) + P_j \big( g\cdot\omega \cdot \partial k \big).   \nonumber
\end{align}
Decomposing the terms in the second line of \eqref{F} further into low-high and high-high interactions, we can reexpress $F_j$ as
\begin{equation}\label{F again}
F_j = F^{(j)}_{LH}+F^{(j)}_{HH},
\end{equation}
where
\begin{align*}
F^{(j)}_{LH} \doteq \, & (g-m_0)_j \partial^2  k_{\le j} +(g\cdot \partial g)_j \partial k_{\le j} \\
& + \big( ( g\cdot\partial^2 g)_{\le j} +  (g\cdot(\partial g)^2)_{\le j}\big) \cdot k_j + \big( ( g\cdot\partial^2 g)_{j} +  (g\cdot(\partial g)^2)_{j}\big) \cdot k_{\le j}\\
&+ \big(g\cdot (\partial\omega+\omega^2)\big)_{\le j} \cdot k_{j} +  \big(g\cdot (\partial\omega+\omega^2)\big)_{ j} \cdot k_{\le j}\\
& + \big( g\cdot \partial g\big)_{\le j} \cdot \partial k_j + \big( g\cdot \partial g\big)_{j} \cdot \partial k_{\le j}\\
& + \big( g\cdot\omega\big)_{\le j} \cdot \partial k_j+ \big( g\cdot\omega\big)_{j} \cdot \partial k_{\le j}  
\end{align*}
and
\begin{align*}
F_{HH}^{(j)} \doteq \,  \sum_{j' > j}  \Bigg[  & P_j \big((g-m_0)_{j'} \partial^2  k_{j'}\big)+  P_j \big((g\cdot \partial g)_{j'} \partial k_{j'}\big) \\
& +  P_j \big((g\cdot (\partial\omega+\omega^2))_{j'} \cdot k_{j'}\big)  + P_j \big( (g\cdot\omega)_{j'} \cdot \partial k_{j'} \big)   \Bigg].
\end{align*}

According to the results of \cite{Ta} (see Theorems 1.2 and 1.4 therein and the remark above relation 1.10), the following statement holds for the equation $\square_{g_{\le j}} u = \mathcal F_1 + \mathcal F_2$: For any $(\rho, p ,q)$ satisfying $2\le p,q\le \infty$ and 
\[
\f1p + \f3q = \f32 - \rho, \hphantom{and} \f2p + \f2q \le 1
\]
\emph{with the exception of} $(\rho, p, q) = (1,2,\infty)$, we have
\begin{align}\label{General Strichartz from Tataru 1}
 \| \JapD^{-\rho-\frac{1}{3p}-\frac1{10}\delta_0}\partial P_j u\|_{L^p L^q} \lesssim  & \big( 1+ \|\partial g_{\le j}\|_{L^1L^\infty}\big)^{\frac2{3p}} \|\partial u\|_{L^\infty L^2}\\&  + \big( 1+ \|\partial g_{\le j}\|_{L^1L^\infty}\big)^{-\frac2{3p'}} \|\JapD^{-\f13} \mathcal F_1\|_{L^1 L^2} + \|\JapD^{\rho-\frac{1}{3p}} P_j\mathcal F_2 \|_{L^{p'}L^{q'}}.\nonumber 
\end{align}

\medskip
\noindent \textbf{Remark.} Note that, with respect to the standard Strichartz estimates on $(\mathbb R^{3+1}, \eta)$, \eqref{General Strichartz from Tataru 1} has a loss of $\f1{3p}$ derivatives on the left hand side, while there is a comparative gain of derivatives in the inhomogeneous terms on the right hand side. In what follows, we will actually not be making use of the derivative gain on the right hand side; we will only use the following weaker version of \eqref{General Strichartz from Tataru 1}:
\begin{align}\label{General Strichartz from Tataru}
 \| \JapD^{-\rho-\frac{1}{3p}-\frac1{10}\delta_0}\partial P_j u\|_{L^p L^q} \lesssim  & \big( 1+ \|\partial g_{\le j}\|_{L^1L^\infty}\big)^{\frac2{3p}} \|\partial u\|_{L^\infty L^2}\\&  + \big( 1+ \|\partial g_{\le j}\|_{L^1L^\infty}\big)^{-\frac2{3p'}} \|\mathcal F_1\|_{L^1 L^2} + \|\JapD^{\rho} P_j\mathcal F_2 \|_{L^{p'}L^{q'}}.\nonumber 
\end{align}

\medskip
Setting
$u\doteq k_j$,  $\mathcal F_1 \doteq  F^{(j)}_{LH}$ and $\mathcal F_2 = F_{HH}^{(j)} $,
 in \eqref{General Strichartz from Tataru} we obtain:
\begin{align}\label{Strichartz estimate tataru}
\|2^{-(\rho +\f1{3p}-\f1{10}\delta_0)j} \partial k_j\|_{L^p L^q} \lesssim \big( 1+ & \| \partial g_{\le j}\|_{L^1 L^\infty}\big)^{\f2{3p}}\| \partial k_j \|_{L^\infty L^2} \\
& +\big( 1+ \| \partial g_{\le j}\|_{L^1 L^\infty}\big)^{\f{2(p-1)}{3p}} \|  F_{LH}^{(j)}\|_{L^1 L^2} + \|2^{\rho j} F_{HH}^{(j)} \|_{L^{p'} L^{q'}}, \nonumber
\end{align}
or, equivalently, after multiplying both sides with $2^{(s-3)j}$:
\begin{align}\label{Strichartz estimate tataru 2}
\|2^{(s-3-\rho-\f1{3p}-\f1{10}\delta_0)j} \partial k_j\|_{L^p L^q} \lesssim & \big( 1+ \| \partial g_{\le j}\|_{L^1 L^\infty}\big)^{\f2{3p}}\| 2^{(s-3)j}\partial k_j \|_{L^\infty L^2} \\
& +\big( 1+ \| \partial g_{\le j}\|_{L^1 L^\infty}\big)^{\f{2(p-1)}{3p}} \| 2^{(s-3)j} F^{(j)}_{LH}\|_{L^1 L^2}+ \|2^{(s-3+\rho)j} F_{HH}^{(j)} \|_{L^{p'} L^{q'}}. \nonumber
\end{align}
In view of the bounds \eqref{Bootstrap bound} for $\partial g$ and the energy estimate \eqref{Energy estimate k} for $k$, it can be readily deduced that \eqref{Strichartz first derivative} will follow from \eqref{Strichartz estimate tataru 2} for $p=2, q=\frac{1}{\delta_0}$, provided we show that
\begin{equation}\label{Bound Strichartz F}
\| 2^{(s-3)j} F_{LH}^{(j)}\|_{L^1 L^2} + \|2^{(s-2)j} F_{HH}^{(j)} \|_{L^{2} L^{1}} \lesssim \mathcal{D}.
\end{equation}

In order to establish \eqref{Bound Strichartz F}, we will argue in a similar way as for the derivation of \eqref{FK}, showing that all terms in the expansion \eqref{F} of $F_j$ satisfy the bound \eqref{Bound Strichartz F}.

\begin{itemize}
\item Let us start with the high-high terms (i.e.~$F_{HH}^{(j)}$). The term $P_j ( (g-m_0)_{j'} \partial^2 k_{j'})$ in $F_{HH}^{(j)}$ can be estimated as follows:
\begin{align} \label{Second term F}
\big\| 2^{(s-2)j} & \sum_{j'>j} P_j \big(  (g-m_0)_{j'} \partial^2 k_{j'} \big)\big\|_{L^2 L^1} \\
& 
\lesssim \sum_{j'>j} 2^{(s-2)(j-j')}\big\| 2^{2j'} (g-m_0)_{j'} \cdot 2^{(s-4)j'}\partial^2 k_{j'} \big\|_{L^2 L^1}\nonumber \\
&\lesssim  \sum_{j'>j} 2^{(s-2)(j-j')}\| 2^{2j'} (g-m_0)_{j'}\|_{L^2 L^2}\cdot \|2^{(s-4)j'}\partial^2 k_{j'} \|_{L^\infty L^2} \nonumber \\
& \lesssim \Big(\| \partial^2 g \|^2_{L^2 L^2}+\| \partial^2 k \|^2_{L^\infty H^{s-4}}\Big)\nonumber \\
& \lesssim C_0^2 \mathcal{D}^2. \nonumber 
\end{align}
The rest of the terms in $F^{(j)}_{HH}$, namely $\sum_{j' > j} P_j (( g\cdot\partial g)_{j'} \partial k_{j'})$, $\sum_{j' > j}P_j \big((g\cdot (\partial\omega+\omega^2))_{j'} \cdot k_{j'}\big) $ and  $\sum_{j' > j}P_j \big( (g\cdot\omega)_{j'} \cdot \partial k_{j'} \big)$    can be estimated in exactly the same way, thus leading to 
\begin{equation}\label{Bound F HH}
\|2^{(s-2)j} F_{HH}^{(j)} \|_{L^{2} L^{1}}  \lesssim C_0^2 \mathcal{D}^2.
\end{equation}

\item We will now estimate the terms in $F^{(j)}_{LH}$. For the high-low term $ (g-m_0)_j \partial^2 k_{\le j}$, we have 
\begin{align}\label{First term F}
\| 2^{(s-3)j} (g-m_0)_j \partial^2 k_{\le j} \|_{L^1 L^2}
& \lesssim  \| 2^{(s-\f12-2\delta_0)j} (g-m_0)_j \|_{L^2 L^2} \|  2^{-(\f12+2-2\delta_0) j} \partial^2 k_{\le j}\|_{L^2 L^\infty} \\
& \lesssim  \| g-m_0 \|_{L^2 H^{s-\f12-2\delta_0}} \| \partial^2 k\|_{L^2 W^{-\f52+2\delta_0, \infty}} \nonumber \\
& \lesssim C_0^2 \mathcal{D}^2 \nonumber 
\end{align}
(where, in the last inequality, we used the bootstrap bound \eqref{Bootstrap bound}). The rest of the high-low terms $\big( ( g\cdot\partial^2 g)_{j} +  (g\cdot(\partial g)^2)_{j}\big) \cdot k_{\le j}$, $ \big(g\cdot (\partial\omega+\omega^2)\big)_{ j} \cdot k_{\le j}$, $ \big( g\cdot \partial g\big)_{j} \cdot \partial k_{\le j}$ and $ \big( g\cdot\omega\big)_{j} \cdot \partial k_{\le j}$ can be estimated in exactly the same way.

The low-high term term  $ (g\cdot\partial g)_{\le j} \partial k_j$ can be estimated as follows:
\begin{align}
\| 2^{(s-3)j} ( g\cdot\partial g)_{\le j}  \partial k_j \|_{L^1 L^2}  
& \lesssim  \|  ( g\cdot\partial g)_{\le j}\|_{L^1 L^\infty} \| 2^{(s-3)j} \partial k_j \|_{L^\infty L^2} \\
& \lesssim  \|   g\cdot\partial g\|_{L^1 L^\infty} \| \partial k \|_{L^\infty H^{s-\f32}} \nonumber \\
& \lesssim C_0^2 \mathcal{D}^2. \nonumber
\end{align}
The rest of the low-high terms $\big( ( g\cdot\partial^2 g)_{\le j} +  (g\cdot(\partial g)^2)_{\le j}\big) \cdot k_j $, $ \big(g\cdot (\partial\omega+\omega^2)\big)_{\le j} \cdot k_{j} $, $ \big( g\cdot \partial g\big)_{\le j} \cdot \partial k_j $ and $ \big( g\cdot\omega\big)_{\le j} \cdot \partial k_j$ can be estimated in exactly the same way. Thus, collecting these bounds we infer that
\begin{equation}\label{Bound F LH}
\|2^{(s-3)j} F_{LH}^{(j)} \|_{L^{1} L^{2}}  \lesssim C_0^2 \mathcal{D}^2.
\end{equation}
\end{itemize}
Combining \eqref{Bound F HH} and \eqref{Bound F LH}, we deduce \eqref{Bound Strichartz F} and, therefore, \eqref{Strichartz first derivative}. Moreover, using the low frequency estimate
\[
\| k \|_{L^2 W^{-\f12+s_0+\f14\delta_0,\infty}}  \lesssim \| \partial k \|_{L^2 W^{-\f32+s_0+\f14\delta_0,\infty}} +\| k \|_{L^\infty H^{s-2}}
\]
together with the energy estimate \eqref{Energy estimate k} for $k$, the bound \eqref{Strichartz first derivative} yields:
\begin{equation}\label{Strichartz estimate k lower}
 \| k \|_{L^2 W^{- \f12+s_0+\f14\delta_0, \infty}} + \| \partial k \|_{L^2 W^{- \f32+s_0+\f14 \delta_0, \infty}}
 \lesssim \mathcal{D}. 
\end{equation}

In order to obtain the full estimate \eqref{Strichartz estimate k endpoint} for $\partial^2 k$, 
we will start from \eqref{Strichartz estimate k lower} and use  equation \eqref{Equation for d_0^2} for $k$ as we did when deriving \eqref{Second derivatives energy} from \eqref{First derivatives energy}. Since \eqref{Strichartz estimate k lower} already provides the required estimate for $\bar{\partial} \partial k$, $\bar{\partial} \in \{ \partial_1, \partial_2, \partial_3 \}$, it only remains to show that
\begin{equation}\label{Two derivatives Strichartz}
\| \partial^2_0 k \|_{L^2 W^{- \f52+s_0, \infty}}   \lesssim \mathcal{D},
\end{equation}
Thus, it will be sufficient to estimate the right hand side of the expression \eqref{Equation for d_0^2} for $\partial_0^2 k$, \eqref{Two derivatives Strichartz} in the $L^2 W^{- \f52+s_0, \infty}$ norm:
\begin{itemize}
\item For the first term in the right hand side of \eqref{Equation for d_0^2}, we have:
\begin{align}\label{First term strichartz}
\| g \cdot \bar{\partial}   \partial & k \|_{L^2 W^{- \f52+s_0, \infty}} \\
\lesssim & \sum_j 2^{(-\f52+s_0)j} \| g_{\le j} \bar{\partial} \partial k_j \|_{L^2 L^\infty} +  \sum_j 2^{(-\f52+s_0)j}\| g_{j} \bar{\partial} \partial k_{\le j} \|_{L^2 L^\infty} \nonumber \\  
& + \sum_j \sum_{j'>j} 2^{(-\f52+s_0)j}\| P_j \big( g_{j'} \cdot \bar{\partial} \partial k_{j'}\big) \|_{L^2 L^\infty}  \nonumber\\
\lesssim &  \sum_j 2^{-\f14 \delta_0 j}\|g_{\le j}\|_{L^\infty L^\infty} \| \bar{\partial} \partial k_j \|_{L^2 W^{-\f52+s_0+\f14\delta_0,\infty}} +  \sum_j 2^{-\f14\delta_0 j} \|2^{s_0 j} g_{j}\|_{L^\infty L^\infty} \|2^{(-\f52+\f14 \delta_0)j} \bar{\partial} \partial k_{\le j} \|_{L^2 L^\infty} \nonumber \\  
& + \sum_j \sum_{j'>j} 2^{s_0j}\| P_j \big( g_{j'} \cdot \bar{\partial} \partial k_{j'}\big) \|_{L^2 L^{\f65}}  \nonumber\\
\lesssim &  \|g\|_{L^\infty L^\infty} \| \bar{\partial} \partial k \|_{L^2 W^{-\f52+s_0+\f14\delta_0,\infty}} + \| g\|_{L^\infty W^{s_0,\infty}} \| \bar{\partial} \partial k\|_{L^2 W^{-\f52+\f14 \delta_0,\infty}} \nonumber \\  
& + \sum_j \sum_{j'>j} 2^{s_0j}\| g_{j'} \|_{L^2 L^3} \| \bar{\partial} \partial k_{j'}\|_{L^\infty L^2}  \nonumber\\
\lesssim & \| g \|_{L^\infty H^{s-1}} \| \partial k \|_{L^2 W^{-\f32 + s_0 + \f14\delta_0, \infty}} +  \sum_j \sum_{j'>j} 2^{s_0(j-j')} 2^{-\delta_0 j'} \| 2^{(2+s_0)j'} g_{j'} \|_{L^2 L^2} \|2^{(-\f32+\delta_0)j'} \bar{\partial} \partial k_{j'}\|_{L^\infty L^2} \nonumber \\
\lesssim & \| g \|_{L^\infty H^{s-1}} \| \partial k \|_{L^2 W^{-\f32 + s_0 + \f14\delta_0, \infty}} +  \|g\|_{L^2 H^{2+s_0}} \|\bar{\partial} \partial k\|_{L^\infty H^{-\f32+\delta_0}} \nonumber \\
\lesssim & \mathcal{D}.\nonumber 
\end{align}
\noindent In the last line above, we made use of the estimates and \eqref{Bootstrap bound},  \eqref{Strichartz estimate k lower} and \eqref{Energy estimate k}.

\item For the second term in the right hand side of \eqref{Equation for d_0^2}, we can estimate similarly:
\begin{align}
\| g\cdot  \partial g \cdot \partial & k \|_{L^2 W^{- \f52+s_0, \infty}} \\
\lesssim & \sum_j 2^{(-\f52+s_0)j} \| ( g\cdot\partial g)_{\le j} \partial k_j \|_{L^2 L^\infty} +  \sum_j 2^{(-\f52+s_0)j}\| ( g\cdot\partial g)_{j}  \partial k_{\le j} \|_{L^2 L^\infty} \nonumber \\  
& + \sum_j \sum_{j'>j} 2^{(-\f52+s_0)j}\| P_j \big( ( g\cdot\partial g)_{j'} \cdot \partial k_{j'}\big) \|_{L^2 L^\infty}  \nonumber\\
\lesssim &  \sum_j 2^{-\f14 \delta_0 j}\|2^{-j} ( g\cdot\partial g)_{\le j}\|_{L^\infty L^\infty} \| 2^{-\f32 +s_0 +\f14\delta_0} \partial k_j \|_{L^2 L^\infty} \nonumber \\
& +  \sum_j 2^{-\f14\delta_0 j} \|2^{(-1+s_0) j} ( g\cdot\partial g)_{j}\|_{L^\infty L^\infty} \|2^{(-\f32+\f14 \delta_0)j}  \partial k_{\le j} \|_{L^2 L^\infty} \nonumber \\  
& + \sum_j \sum_{j'>j} 2^{s_0j}\| P_j \big( ( g\cdot\partial g)_{j'} \cdot \partial k_{j'}\big) \|_{L^2 L^{\f65}}  \nonumber\\
\lesssim &  \| g\cdot\partial g\|_{L^\infty W^{-1,\infty}} \| \partial k \|_{L^2 W^{-\f32+s_0+\f14\delta_0,\infty}} + \|  g\cdot\partial g\|_{L^\infty W^{-1+s_0,\infty}} \| \partial k\|_{L^2 W^{-\f32+\f14 \delta_0,\infty}} \nonumber \\  
& + \sum_j \sum_{j'>j} 2^{s_0j}\| ( g\cdot\partial g)_{j'} \|_{L^2 L^3} \| \partial k_{j'}\|_{L^\infty L^2}  \nonumber\\
\lesssim & \|  g\cdot \partial g \|_{L^\infty H^{s-2}} \| \partial k \|_{L^2 W^{-\f32 + s_0 + \f14\delta_0, \infty}} +  \sum_j \sum_{j'>j} 2^{s_0(j-j')} 2^{-\delta_0 j'} \| 2^{(1+s_0)j'} ( g\cdot\partial g)_{j'} \|_{L^2 L^2} \|2^{(-\f12+\delta_0)j'} \partial k_{j'}\|_{L^\infty L^2} \nonumber \\
\lesssim & \|  g\cdot \partial g \|_{L^\infty H^{s-2}} \| \partial k \|_{L^2 W^{-\f32 + s_0 + \f14\delta_0, \infty}} +  \| g\cdot\partial g\|_{L^2 H^{1+s_0}} \| \partial k\|_{L^\infty H^{-\f12+\delta_0}} \nonumber \\
\lesssim & \mathcal{D}\nonumber 
\end{align}
\noindent (where, again, we made use of the estimates and \eqref{Bootstrap bound},  \eqref{Strichartz estimate k lower} and \eqref{Energy estimate k}).

\item The rest of the terms in the right hand side of \eqref{Equation for d_0^2} can be similarly estimated repeating the same steps as above. Thus, we obtain \eqref{Two derivatives Strichartz}.
\end{itemize}

\noindent This completes the proof of Lemma \ref{lem:Strichartz estimates}.
\end{proof}

\subsection{The special structure of the curvature tensors}
In this section, we will uncover a cancellation in the high-high frequency interactions in the expressions  $R=m(e) \cdot k \wedge k$ and $R^\perp = g \cdot k \wedge k$ for the curvature tensors that will allow us to ``pull'' a derivative out of the product $k\wedge k$. In the following section, this structure will be used to control the $L^1 W^{-1, \infty}$ norm of the components of $R_*$ and $R^\perp_*$ (namely the components $R_{\alpha i j k}$ and $R^\perp_{ij}$ with $i,j , k\in \{1,2,3\}$), given the energy and  Strichartz estimates established  for $k$ in the previous section.  

Before proceeding, let us recall the Gauss and Ricci equations (i.e.~\eqref{Gauss} and \eqref{Ricci}):
\begin{align}
R_{\a \b \ga \delta}  &= m(e_{\bA}, e_{\bB}) \Big( k^{\bA}_{\a\ga} k^{\bB}_{\b\delta} - k^{\bA}_{\a\delta} k^{\bB}_{\b\ga} \Big), \label{Gauss 3}\\
R^{\perp \bA \bB}_{\a \b}  &= g^{\ga \delta} \Big( k^{\bB}_{\a\ga} k^{\bA}_{\b\delta}- k^{\bA}_{\a\ga} k^{\bB}_{\b\delta} \Big). \label{Ricci 3}
\end{align}
We will use the following shorthand notation for the antisymmetric product appearing in the right hand sides of the relations above:
\begin{equation}\label{Wedge product}
\big( \chi \wedge\psi \big)^{\bA \bB}_{\a \b \ga \delta} \doteq \chi^{\bB}_{\a\ga} \psi^{\bA}_{\b\delta}- \chi^{\bA}_{\a\delta} \psi^{\bB}_{\b\ga},
\end{equation}
so that the equations \eqref{Gauss 3}--\eqref{Ricci 3} can be expressed as
\begin{align*}
R_{\a\b\ga\delta} &= m_{\bA\bB} \cdot \big( k \wedge k \big)^{\bA\bB}_{\a\b\ga\delta} = m(e) \cdot k \wedge k, \\
 R^{\perp\bA\bB}_{\a\b} & = g^{\ga\delta} \cdot \big( k \wedge k)^{\bA\bB}_{\a\b\ga\delta} = g \cdot k \wedge k
\end{align*}
(here, we made use of the symmetry of $m_{\bA\bB}$ and $g^{\ga\delta}$ in $(\bA, \bB)$ and $(\ga, \delta)$, respectively).

The main result of this section is the following:
\begin{proposition}\label{prop:Derivative k}
For any $j'>2$ and $J'\in \{ -2, \ldots, 2\}$, the following (schematic) relation holds for all $a,b \in \{1,2,3\}$, $\ga,\delta \in \{ 0,1,2,3\}$ and $\bA,\bB \in \{4, \ldots, n\}$ (recall the definition  \eqref{Wedge product}  of the wedge product $\wedge$):
\begin{align}\label{Schematic extraction one derivative wedge product}
(P_{j'} & k \wedge  P_{j'+J'} k)^{\bA\bB}_{ab\ga\delta}  + (P_{j'+J'} k \wedge P_{j'} k)^{\bA\bB}_{ab\ga\delta} \\
=& 
\sum_{(j_1, j_2) \in \big\{ (j',j'+J'),\, (j'+J',j') \big\} } \Bigg[\bar\partial \Big(P_{j_1} (\JapD^{-1}k)\cdot  P_{j_2} (k)\Big) 
\nonumber \\
& \hphantom{= \sum_{(j_1, j_2)=(j',j'+J') \text{ or } (j'+J',j')} \bar{\partial}}  
+P_{j_1}\big(\JapD^{-1}(  g\cdot\partial g \cdot k + \omega\cdot k)\big)\cdot  P_{j_2}k\Bigg], \nonumber
\end{align}
where $\bar{\partial} \in \{ \partial_1, \partial_2, \partial_3\}$, i.e.~it is a spatial derivative (our assumption that $a,b\in \{1,2,3\}$ is crucial for this). Moreover, after contracting with $m_{\bA\bB}$, we can extract one more derivative from the product: For any $j_3\in \mathbb N$,

\begin{align}\label{Schematic extraction two derivatives wedge product}
P_{j_3}(m(e)_{\bA\bB}) & \cdot \Big((P_{j'} k \wedge  P_{j'+J'} k)^{\bA\bB}_{ab\ga\delta}  + (P_{j'+J'} k \wedge P_{j'} k)^{\bA\bB}_{ab\ga\delta} \Big)\\
=&
\sum_{(j_1, j_2) \in \big\{ (j',j'+J'),\, (j'+J',j') \big\} } \Bigg[ \bar\partial \partial \Big\{P_{j_3}(m(e)) \cdot P_{j_1}(\JapD^{-1} k)\cdot P_{j_2}(\JapD^{-1} k)\Big\}  \nonumber \\
& \hphantom{= \sum_{(j_1, j_2)=(j',j'+J') \text{ or } (j'+J',j')} \bar{\partial}}  
+\bar\partial \Big\{P_{j_3}(m(e)) \cdot P_{j_1}\big(\JapD^{-1}(  g\cdot\partial g \cdot k + \omega\cdot k)\big)\cdot  P_{j_2}(\JapD^{-1}k)\Big\}\nonumber \\
& \hphantom{= \sum_{(j_1, j_2)=(j',j'+J') \text{ or } (j'+J',j')} \bar{\partial}}
 + 
\bar\partial \big\{P_{j_3}(m(e))\big\} \cdot P_{j_1}(\JapD^{-1}k)\cdot  P_{j_2}(k)\nonumber \\
& \hphantom{= \sum_{(j_1, j_2)=(j',j'+J') \text{ or } (j'+J',j')} \bar{\partial}}
+ \bar\partial\partial \big\{P_{j_3}(m(e))\big\} \cdot P_{j_1}(\JapD^{-1}k)\cdot  P_{j_2}(\JapD^{-1} k)\nonumber \\
& \hphantom{= \sum_{(j_1, j_2)=(j',j'+J') \text{ or } (j'+J',j')} \bar{\partial}}
+ P_{j_3}(m(e)) \cdot P_{j_1}\big(\JapD^{-1}(  g\cdot\partial g \cdot k + \omega\cdot k)\big)\cdot  P_{j_2}k\Bigg], \nonumber
\end{align}
where $\bar{\partial} \in \{ \partial_1, \partial_2, \partial_3\}$ and $\partial\in \{\partial_0, \bar\partial\}$. When $\gamma,\delta\in \{1,2,3\}$, the $\partial$ in the terms of the right hand side above can be replaced by $\bar\partial$, i.e.
\begin{align}\label{Schematic extraction two spatial derivatives wedge product}
P_{j_3}(m(e)_{\bA\bB}) & \cdot \Big((P_{j'} k \wedge  P_{j'+J'} k)^{\bA\bB}_{abcd}  + (P_{j'+J'} k \wedge P_{j'} k)^{\bA\bB}_{abcd} \Big)\\
=&
\sum_{(j_1, j_2) \in \big\{ (j',j'+J'),\, (j'+J',j') \big\} } \Bigg[ \bar\partial^2 \Big\{P_{j_3}(m(e)) \cdot P_{j_1}(\JapD^{-1} k)\cdot P_{j_2}(\JapD^{-1} k)\Big\}  \nonumber \\
& \hphantom{= \sum_{(j_1, j_2)=(j',j'+J') \text{ or } (j'+J',j')} \bar{\partial}}  
+\bar\partial \Big\{P_{j_3}(m(e)) \cdot P_{j_1}\big(\JapD^{-1}(  g\cdot\partial g \cdot k + \omega\cdot k)\big)\cdot  P_{j_2}(\JapD^{-1}k)\Big\}\nonumber \\
& \hphantom{= \sum_{(j_1, j_2)=(j',j'+J') \text{ or } (j'+J',j')} \bar{\partial}}
 + 
\bar\partial \big\{P_{j_3}(m(e))\big\} \cdot P_{j_1}(\JapD^{-1}k)\cdot  P_{j_2}(k)\nonumber \\
& \hphantom{= \sum_{(j_1, j_2)=(j',j'+J') \text{ or } (j'+J',j')} \bar{\partial}}
+ \bar\partial^2 \big\{P_{j_3}(m(e))\big\} \cdot P_{j_1}(\JapD^{-1}k)\cdot  P_{j_2}(\JapD^{-1} k)\nonumber \\
& \hphantom{= \sum_{(j_1, j_2)=(j',j'+J') \text{ or } (j'+J',j')} \bar{\partial}}
+ P_{j_3}(m(e)) \cdot P_{j_1}\big(\JapD^{-1}(  g\cdot\partial g \cdot k + \omega\cdot k)\big)\cdot  P_{j_2}k\Bigg], \nonumber
\end{align}

As a consequence of \eqref{Schematic extraction one derivative wedge product}, the following bound holds for any $j<j'-2$:
\begin{align}\label{Improved estimate wedge infinity}
\big\| P_j \Big( (P_{j'} & k \wedge  P_{j'+J'} k)^{\bA\bB}_{ab\ga\delta} + (P_{j'+J'} k \wedge P_{j'} k)^{\bA\bB}_{ab\ga\delta}\Big) \big\|_{L^1 L^\infty} \\
& \lesssim  2^{j-(1-\delta_0)j'}  \| P_{j'} k \|_{L^2 L^\infty}\| P_{j'+J'} k \|_{L^2 L^\infty} 
 \hphantom{~} + 2^{2j-(1+2\delta_0) j'} \big( \| g\cdot \partial g\|_{L^2 L^6}+ \|\omega\|_{L^2 L^6}\big) \| k \|^2_{L^4 W^{2\delta_0,4}}. \nonumber
\end{align}
\end{proposition}

\medskip
\noindent \textbf{Remark 1.} For the proof of Theorem \ref{thm:Existence}, we will only make use of the simpler expression \eqref{Schematic extraction one derivative wedge product}. The ``improved''  expression \eqref{Schematic extraction two derivatives wedge product} will be necessary in the proof of Theorem \ref{thm:Uniqueness}, where we will be estimating the difference of solutions to the timelike minimal surface equation in Sobolev spaces of lower regularity compared to those used for Theorem \ref{thm:Existence}.

\medskip
\noindent \textbf{Remark 2.} In the general case when $\alpha, \beta, \gamma, \delta \in\{0,1,2,3\}$ (i.e.~$\alpha,\beta$ are not confined to be spatial, as in the statement of Proposition \ref{prop:Derivative k}), repeating the proof of \eqref{Schematic extraction one derivative wedge product} yields
\begin{align}\label{Schematic extraction one derivative wedge product time indices}
(P_{j'} & k \wedge  P_{j'+J'} k)^{\bA\bB}_{\a\b\ga\delta}  + (P_{j'+J'} k \wedge P_{j'} k)^{\bA\bB}_{\a\b\ga\delta} \\
=& 
\sum_{(j_1, j_2) \in \big\{ (j',j'+J'),\, (j'+J',j') \big\} } \Bigg[\partial \Big(P_{j_1} (\JapD^{-1}k)\cdot  P_{j_2} (k)\Big) 
\nonumber \\
& \hphantom{= \sum_{(j_1, j_2)=(j',j'+J') \text{ or } (j'+J',j')} \bar{\partial}}  
+P_{j_1}\big(\JapD^{-1}(  g\cdot\partial g \cdot k + \omega\cdot k)\big)\cdot  P_{j_2}k\Bigg], \nonumber
\end{align}
i.e.~the derivative pulled outside the product is not necessarily in a spatial direction.

\medskip
\begin{proof}
For the proof of Proposition \ref{prop:Derivative k}, we will make use of the Riesz-type multiplier operators $\mathcal{T}^{(i)}_{j'}$, $i=1,2,3$ and $j\ge 0$ (see Definition \ref{def:T multiplier}). Recall that the operators $\mathcal{T}^{(i)}_{j'}$ are (renormalized) Mikhlin multipliers of order $-1$; in particular, the the bound \eqref{Mikhlin property} implies that:
\begin{equation}\label{Mikhlin}
\| \mathcal{T}^{(i)}_{j'} f(t,\cdot )\|_{W^{s,p}} \lesssim_{s,p} 2^{-j'} \| f(t,\cdot) \|_{W^{s,p}} \hphantom{a} \text{ for any } s\in \mathbb{R}, p\in (1,+\infty)
\end{equation}
and
\begin{equation}\label{Mikhlin infinity}
\| \mathcal{T}^{(i)}_{j'} f(t,\cdot )\|_{W^{s,\infty}} \lesssim_{s} 2^{-j'} \| f(t,\cdot) \|_{W^{s+\delta_0, \infty}} \hphantom{a} \text{ for any } s\in \mathbb{R}.
\end{equation}
Moreover, the operators $\mathcal{T}^{(i)}_{j'}$ satisfy the following divergence identity for any smooth function $f:[0,T]\times \mathbb{T}^3 \rightarrow \mathbb{R}$:
\begin{equation}\label{Identity derivative}
\partial_i \Big( \mathcal{T}^{(i)}_{j'} P_{j'} f \Big) = P_{j'} f.
\end{equation}

The Codazzi equations \eqref{Codazzi} can be reexpressed as
\begin{equation}\label{Codazzi coordinates 2}
\partial_\kappa k^{\bar{A}}_{\lambda\mu} - \partial_\lambda k^{\bar{A}}_{\kappa \mu} =  \partial g \cdot k + \omega\cdot k.
\end{equation}
After applying the multiplier operators $P_{j'}$ and $\mathcal{T}_{j'}^{(i)} P_{j'}$ for $i=1,2,3$, \eqref{Codazzi coordinates 2} yields
\begin{equation}\label{Codazzi coordinates projection 1}
\partial_\kappa P_{j'} (k^{\bar{A}}_{\lambda\mu}) - \partial_\lambda P_{j'} (k^{\bar{A}}_{\kappa \mu})=  P_{j'} \big( \partial g \cdot k + \omega\cdot k\big) 
\end{equation}
and
\begin{equation}\label{Codazzi coordinates projection}
\partial_\kappa \big( \mathcal{T}_{j'}^{(i)} P_{j'} (k^{\bar{A}}_{\lambda\mu}) \big) - \partial_\lambda \big(\mathcal{T}_{j'}^{(i)} P_{j'} (k^{\bar{A}}_{\kappa \mu})\big) =  \mathcal{T}_{j'}^{(i)} P_{j'} \big( \partial g \cdot k + \omega\cdot k\big) 
\end{equation}
Using \eqref{Codazzi coordinates projection} repeatedly, we obtain for $i=1,2,3$:
\begin{align}
\partial_i \big( \mathcal{T}_{j'}^{(i)} P_{j'} (k^{\bar{B}}_{a \ga} & )\big) \cdot  P_{j' +J'} (k^{\bar{A}}_{b\delta}) 
= \partial_a \big( \mathcal{T}_{j'}^{(i)} P_{j'} (k^{\bar{B}}_{i \ga})\big) \cdot  P_{j'+J'} (k^{\bar{A}}_{b\delta}) + \mathcal{T}_{j'}^{(i)} P_{j'}\big( \partial g \cdot k + \omega\cdot k\big)\cdot  P_{j'+J'} k \nonumber \\
& = -\mathcal{T}_{j'}^{(i)} P_{j'} (k^{\bar{B}}_{i \ga})\cdot  \partial_a P_{j' +J'} (k^{\bar{A}}_{b\delta}) 
+\partial_a \Big(\mathcal{T}_{j'}^{(i)} P_{j'} (k^{\bar{B}}_{i \ga})\cdot  P_{j'+J'} (k^{\bar{A}}_{b\delta})\Big)\nonumber\\
& \hphantom{an} + \mathcal{T}_{j'}^{(i)} P_{j'}\big( \partial g \cdot k + \omega\cdot k\big)\cdot  P_{j'+J'}k  \nonumber \\
& = - \mathcal{T}_{j'}^{(i)} P_{j'} (k^{\bar{B}}_{i \ga})\cdot  \partial_b P_{j'+J'}(k^{\bar{A}}_{a\delta}) 
+\partial_a \Big(\mathcal{T}_{j'}^{(i)} P_{j'} (k^{\bar{B}}_{i \ga})\cdot  P_{j'+J'}(k^{\bar{A}}_{b\delta})\Big)\nonumber \\
&\hphantom{an} + \mathcal{T}_{j'}^{(i)} P_{j'}\big( g\cdot \partial g \cdot k + \omega\cdot k\big)\cdot  P_{j'+J'}k+ P_{j'+J'}\big( \partial g \cdot k + \omega\cdot k\big)\cdot  \mathcal{T}_{j'}^{(i)} P_{j'}k \nonumber \\
& = \partial_b \big( \mathcal{T}_{j'}^{(i)} P_{j'} (k^{\bar{B}}_{i \ga})\big)\cdot  P_{j'+J'} (k^{\bar{A}}_{a\delta})
-\partial_b \Big(\mathcal{T}_{j'}^{(i)} P_{j'} (k^{\bar{B}}_{i \ga})\cdot  P_{j'+J'} (k^{\bar{A}}_{a\delta})\Big) 
+\partial_a \Big(\mathcal{T}_{j'}^{(i)} P_{j'} (k^{\bar{B}}_{i \ga})\cdot  P_{j'+J'} (k^{\bar{A}}_{b\delta})\Big)\nonumber \\
&\hphantom{an} +  \mathcal{T}_{j'}^{(i)} P_{j'}\big(  g\cdot\partial g \cdot k + \omega\cdot k\big)\cdot  P_{j'+J'}k+ P_{j'+J'}\big(  g\cdot\partial g \cdot k + \omega\cdot k\big)\cdot  \mathcal{T}_{j'}^{(i)} P_{j'}k  \nonumber \\
& = \partial_i \big(\mathcal{T}_{j'}^{(i)} P_{j'} (k^{\bar{B}}_{b \ga})\big)\cdot  P_{j'+J'} (k^{\bar{A}}_{a\delta})
-\partial_b \Big( \mathcal{T}_{j'}^{(i)} P_{j'} (k^{\bar{B}}_{i \ga})\cdot  P_{j'+J'} (k^{\bar{A}}_{a\delta})\Big) 
+\partial_a \Big(\mathcal{T}_{j'}^{(i)} P_{j'} (k^{\bar{B}}_{i \ga})\cdot  P_{j'+J'} (k^{\bar{A}}_{b\delta})\Big)\nonumber \\
&\hphantom{an} + \mathcal{T}_{j'}^{(i)} P_{j'}\big(  g\cdot\partial g \cdot k + \omega\cdot k\big)\cdot  P_{j'+J'}k+ P_{j'+J'}\big(  g\cdot\partial g \cdot k + \omega\cdot k\big)\cdot  \mathcal{T}_{j'}^{(i)} P_{j'}k \nonumber
\end{align}
and, after moving the first term on the right hand side to the left:
\begin{align}\label{Codazzi k2 1}
\partial_i \big( \mathcal{T}_{j'}^{(i)} P_{j'} (k^{\bar{B}}_{a \ga} ) \big) & \cdot  P_{j'+J'} (k^{\bar{A}}_{b\delta}) 
-   \partial_i \big( \mathcal{T}_{j'}^{(i)} P_{j'} (k^{\bar{B}}_{b \ga})\big) \cdot P_{j'+J'} (k^{\bar{A}}_{a\delta})  \\
=&  -\partial_b \Big( \mathcal{T}_{j'}^{(i)} P_{j'} (k^{\bar{B}}_{i \ga})\cdot  P_{j'+J'} (k^{\bar{A}}_{a\delta})\Big) 
+\partial_a \Big(\mathcal{T}_{j'}^{(i)} P_{j'} (k^{\bar{B}}_{i \ga})\cdot  P_{j'+J'} (k^{\bar{A}}_{b\delta})\Big)\nonumber \\
 & +\mathcal{T}_{j'}^{(i)} P_{j'}\big(  g\cdot\partial g \cdot k + \omega\cdot k\big)\cdot  P_{j'+J'}k+ P_{j'+J'}\big( g\cdot \partial g \cdot k + \omega\cdot k\big)\cdot  \mathcal{T}_{j'}^{(i)} P_{j'}k. \nonumber 
\end{align}
 In view of the identity  \eqref{Identity derivative}, \eqref{Codazzi k2 1} yields:
\begin{align}\label{Codazzi k2}
\ P_{j'} (k^{\bar{B}}_{a \ga} ) & \cdot  P_{j'+J'} (k^{\bar{A}}_{b\delta}) 
-   P_{j'} (k^{\bar{B}}_{b \ga}) \cdot P_{j'+J'} (k^{\bar{A}}_{a\delta})  \\
= &   -\partial_b \Big( \mathcal{T}_{j'}^{(i)} P_{j'} (k^{\bar{B}}_{i \ga})\cdot  P_{j'+J'} (k^{\bar{A}}_{a\delta})\Big) 
+\partial_a \Big(\mathcal{T}_{j'}^{(i)} P_{j'} (k^{\bar{B}}_{i \ga})\cdot  P_{j'+J'} (k^{\bar{A}}_{b\delta})\Big)\nonumber \\
 &  +\mathcal{T}_{j'}^{(i)} P_{j'}\big(  g\cdot\partial g \cdot k + \omega\cdot k\big)\cdot  P_{j'+J'}k+ P_{j'+J'}\big(  g\cdot\partial g \cdot k + \omega\cdot k\big)\cdot  \mathcal{T}_{j'}^{(i)} P_{j'}k. \nonumber 
\end{align}

Going back to the expression \eqref{Wedge product} for $\big(P_{j'} k \wedge P_{j'+J'} k \big)^{\bA\bB}_{ab\ga\delta}$ and  using \eqref{Codazzi k2} twice (once with the roles of $j'$ and $j'+J'$ inverted), we infer:
\begin{align}\label{Total derivative wedge}
\big(P_{j'} k & \wedge P_{j'+J'} k \big)^{\bA\bB}_{ab\ga\delta} + \big(P_{j'+J'} k \wedge P_{j'} k \big)^{\bA\bB}_{ab\ga\delta}  \\
& \doteq  \Big( P_{j'} (k^{\bB}_{a\ga}) P_{j'+J'}  (k^{\bA}_{b\delta}) -  P_{j'} (k^{\bA}_{a\delta}) P_{j'+J'}  (k^{\bB}_{b\ga})\Big) + \Big(  P_{j'+J'} (k^{\bB}_{a\ga}) P_{j'}  (k^{\bA}_{b\delta}) -  P_{j'+J'} (k^{\bA}_{a\delta}) P_{j'}  (k^{\bB}_{b\ga})\Big)   \nonumber  \\
& = \Big( P_{j'} (k^{\bB}_{a\ga}) P_{j'+J'}  (k^{\bA}_{b\delta}) -  P_{j'+J'} (k^{\bA}_{a\delta}) P_{j'}  (k^{\bB}_{b\ga}) \Big) 
+ \Big(  P_{j'+J'} (k^{\bB}_{a\ga}) P_{j'}  (k^{\bA}_{b\delta}) -  P_{j'} (k^{\bA}_{a\delta}) P_{j'+J'}  (k^{\bB}_{b\ga})\Big)   \nonumber  \\
& = \sum_{(j_1, j_2) \in \big\{ (j',j'+J'),\, (j'+J',j') \big\} } \Bigg[-\partial_b \Big( \mathcal{T}_{j_1}^{(i)} P_{j_1} (k^{\bar{B}}_{i \ga})\cdot  P_{j_2} (k^{\bar{A}}_{a\delta})\Big) 
+\partial_a \Big(\mathcal{T}_{j_1}^{(i)} P_{j_1} (k^{\bar{B}}_{i \ga})\cdot  P_{j_2} (k^{\bar{A}}_{b\delta})\Big)\nonumber \\
& \hphantom{= \sum_{(j_1, j_2)=(j',j'+J') \text{ or } (j'+J',j')} \Bigg[ -  }  
+\mathcal{T}_{j_1} P_{j_1}\big(  g\cdot\partial g \cdot k + \omega\cdot k\big)\cdot  P_{j_2}k\Bigg]. \nonumber 
\end{align}
In particular, \eqref{Schematic extraction one derivative wedge product} holds (recalling that $a,b \in \{1,2,3\}$). 

Using the frequency-projected Codazzi equation \eqref{Codazzi coordinates projection 1} repeatedly, we can readily calculate for $(j_1, j_2) \in \big\{ (j',j'+J'),\, (j'+J',j') \big\}$:
\begin{align*}
 \mathcal{T}_{j_1}^{(i)} & P_{j_1} (k^{\bar{B}}_{i \ga})\cdot  P_{j_2} (k^{\bar{A}}_{a\delta})
= \partial_b \Big( \mathcal{T}_{j_1}^{(i)}  P_{j_1} (k^{\bar{B}}_{i \ga})\cdot \partial_\ell \big( \mathcal T_{j_2}^{(\ell)} P_{j_2} (k^{\bar{A}}_{a\delta})\big)
\\
 \stackrel{\hphantom{\eqref{Identity derivative}}}{=} &
 \mathcal{T}_{j_1}^{(i)}  P_{j_1} (k^{\bar{B}}_{i \ga})\cdot \partial_a \big( \mathcal T_{j_2}^{(\ell)} P_{j_2} (k^{\bar{A}}_{\ell \delta})\big)+ \mathcal{T}_{j_1}  P_{j_1} k \cdot \mathcal T_{j_2} P_{j_2}(\partial g\cdot k + \omega \cdot k)
\\
 \stackrel{\hphantom{\eqref{Identity derivative}}}{=} &
 \partial_a \Big(\mathcal{T}_{j_1}^{(i)}  P_{j_1} (k^{\bar{B}}_{i \ga})\cdot  \mathcal T_{j_2}^{(\ell)} P_{j_2} (k^{\bar{A}}_{\ell \delta})\Big) 
- \partial_a \Big(\mathcal{T}_{j_1}^{(i)}  P_{j_1} (k^{\bar{B}}_{i \ga})\Big)\cdot  \mathcal T_{j_2}^{(\ell)} P_{j_2} (k^{\bar{A}}_{\ell \delta})
\\
&  +  \mathcal{T}_{j_1}  P_{j_1} k \cdot \mathcal T_{j_2} P_{j_2}(\partial g\cdot k + \omega \cdot k)
\\
 \stackrel{\hphantom{\eqref{Identity derivative}}}{=} &
\partial_a \Big(\mathcal{T}_{j_1}^{(i)}  P_{j_1} (k^{\bar{B}}_{i \ga})\cdot  \mathcal T_{j_2}^{(\ell)} P_{j_2} (k^{\bar{A}}_{\ell \delta})\Big) 
- \partial_\ga \Big(\mathcal{T}_{j_1}^{(i)}  P_{j_1} (k^{\bar{B}}_{i a})\Big)\cdot  \mathcal T_{j_2}^{(\ell)} P_{j_2} (k^{\bar{A}}_{\ell \delta})
\\
&  +  \mathcal{T}_{j_1}  P_{j_1} k \cdot \mathcal T_{j_2} P_{j_2}(\partial g\cdot k + \omega \cdot k)
\\
 \stackrel{\hphantom{\eqref{Identity derivative}}}{=} & 
\partial_a \Big(\mathcal{T}_{j_1}^{(i)}  P_{j_1} (k^{\bar{B}}_{i \ga})\cdot  \mathcal T_{j_2}^{(\ell)} P_{j_2} (k^{\bar{A}}_{\ell \delta})\Big) 
- \partial_\ga \Big(\mathcal{T}_{j_1}^{(i)}  P_{j_1} (k^{\bar{B}}_{i a})\cdot  \mathcal T_{j_2}^{(\ell)} P_{j_2} (k^{\bar{A}}_{\ell \delta})\Big)
\\
& + (\mathcal{T}_{j_1}^{(i)}  P_{j_1} (k^{\bar{B}}_{i a})\cdot  \partial_{\ga}\Big(\mathcal T_{j_2}^{(\ell)} P_{j_2} (k^{\bar{A}}_{\ell \delta})\Big)
\\
&  +  \mathcal{T}_{j_1}  P_{j_1} k \cdot \mathcal T_{j_2} P_{j_2}(\partial g\cdot k + \omega \cdot k)
\\
 \stackrel{\hphantom{\eqref{Identity derivative}}}{=} &
\partial_a \Big(\mathcal{T}_{j_1}^{(i)}  P_{j_1} (k^{\bar{B}}_{i \ga})\cdot  \mathcal T_{j_2}^{(\ell)} P_{j_2} (k^{\bar{A}}_{\ell \delta})\Big) 
- \partial_\ga \Big(\mathcal{T}_{j_1}^{(i)}  P_{j_1} (k^{\bar{B}}_{i a})\cdot  \mathcal T_{j_2}^{(\ell)} P_{j_2} (k^{\bar{A}}_{\ell \delta})\Big)
\\
& + (\mathcal{T}_{j_1}^{(i)}  P_{j_1} (k^{\bar{B}}_{i a})\cdot  \partial_{\delta}\Big(\mathcal T_{j_2}^{(\ell)} P_{j_2} (k^{\bar{A}}_{\ell \ga})\Big)
\\
&  +  \mathcal{T}_{j_1}  P_{j_1} k \cdot \mathcal T_{j_2} P_{j_2}(\partial g\cdot k + \omega \cdot k)
\\
 \stackrel{\hphantom{\eqref{Identity derivative}}}{=} &
\partial_a \Big(\mathcal{T}_{j_1}^{(i)}  P_{j_1} (k^{\bar{B}}_{i \ga})\cdot  \mathcal T_{j_2}^{(\ell)} P_{j_2} (k^{\bar{A}}_{\ell \delta})\Big) 
- \partial_\ga \Big(\mathcal{T}_{j_1}^{(i)}  P_{j_1} (k^{\bar{B}}_{i a})\cdot  \mathcal T_{j_2}^{(\ell)} P_{j_2} (k^{\bar{A}}_{\ell \delta})\Big)
+ \partial_\delta \Big( (\mathcal{T}_{j_1}^{(i)}  P_{j_1} (k^{\bar{B}}_{i a})\cdot \mathcal T_{j_2}^{(\ell)} P_{j_2} (k^{\bar{A}}_{\ell \ga})\Big) 
\\
& -\partial_\delta \Big(\mathcal{T}_{j_1}^{(i)}  P_{j_1} (k^{\bar{B}}_{i a})\Big)\cdot  \mathcal T_{j_2}^{(\ell)} P_{j_2} (k^{\bar{A}}_{\ell \ga})
\\
&  +  \mathcal{T}_{j_1}  P_{j_1} k \cdot \mathcal T_{j_2} P_{j_2}(\partial g\cdot k + \omega \cdot k)
\\
\stackrel{\eqref{Identity derivative}}{=} & 
\partial_a \Big(\mathcal{T}_{j_1}^{(i)}  P_{j_1} (k^{\bar{B}}_{i \ga})\cdot  \mathcal T_{j_2}^{(\ell)} P_{j_2} (k^{\bar{A}}_{\ell \delta})\Big) 
- \partial_\ga \Big(\mathcal{T}_{j_1}^{(i)}  P_{j_1} (k^{\bar{B}}_{i a})\cdot  \mathcal T_{j_2}^{(\ell)} P_{j_2} (k^{\bar{A}}_{\ell \delta})\Big)
+ \partial_\delta \Big( (\mathcal{T}_{j_1}^{(i)}  P_{j_1} (k^{\bar{B}}_{i a})\cdot \mathcal T_{j_2}^{(\ell)} P_{j_2} (k^{\bar{A}}_{\ell \ga})\Big) 
\\
& -\partial_i \Big(\mathcal{T}_{j_1}^{(i)}  P_{j_1} (k^{\bar{B}}_{\delta a})\Big)\cdot  \mathcal T_{j_2}^{(\ell)} P_{j_2} (k^{\bar{A}}_{\ell \ga})
\\
&  +  \mathcal{T}_{j_1}  P_{j_1} k \cdot \mathcal T_{j_2} P_{j_2}(\partial g\cdot k + \omega \cdot k)
\\
 \stackrel{\hphantom{\eqref{Identity derivative}}}{=} &
\partial_a \Big(\mathcal{T}_{j_1}^{(i)}  P_{j_1} (k^{\bar{B}}_{i \ga})\cdot  \mathcal T_{j_2}^{(\ell)} P_{j_2} (k^{\bar{A}}_{\ell \delta})\Big) 
- \partial_\ga \Big(\mathcal{T}_{j_1}^{(i)}  P_{j_1} (k^{\bar{B}}_{i a})\cdot  \mathcal T_{j_2}^{(\ell)} P_{j_2} (k^{\bar{A}}_{\ell \delta})\Big)
+ \partial_\delta \Big( (\mathcal{T}_{j_1}^{(i)}  P_{j_1} (k^{\bar{B}}_{i a})\cdot \mathcal T_{j_2}^{(\ell)} P_{j_2} (k^{\bar{A}}_{\ell \ga})\Big) 
\\
& -  P_{j_1} (k^{\bar{B}}_{\delta a})\cdot  \mathcal T_{j_2}^{(\ell)} P_{j_2} (k^{\bar{A}}_{\ell \ga})
 +  \mathcal{T}_{j_1}  P_{j_1} k \cdot \mathcal T_{j_2} P_{j_2}(\partial g\cdot k + \omega \cdot k).
\\
\end{align*}
Note that the first term in the last line of the right hand side above is the same as the initial term on the left hand side, but with the roles of $j_1$ and $j_2$  and of $\bA$ and $\bB$ inverted. Moving that term to the left hand side, summing over $(j_1, j_2) \in \big\{ (j',j'+J'),\, (j'+J',j') \big\}$ and contracting with $P_{j_3}(m(e)_{\bA\bB})$, we thus obtain for any $j_3 \in \mathbb N$:
\begin{align*}
&  \sum_{(j_1, j_2) \in \big\{ (j',j'+J'),\, (j'+J',j') \big\} } P_{j_3}(m(e)_{\bA\bB}) \cdot  \mathcal{T}_{j_1}^{(i)}  P_{j_1} (k^{\bar{B}}_{i \ga})\cdot  P_{j_2} (k^{\bar{A}}_{a\delta})\\
& \hphantom{(j_1,j_2)}
= \f12 \sum_{(j_1, j_2) \in \big\{ (j',j'+J'),\, (j'+J',j') \big\} } \Bigg[P_{j_3}(m(e)_{\bA\bB})\Big(
\partial_a \big(\mathcal{T}_{j_1}^{(i)}  P_{j_1} (k^{\bar{B}}_{i \ga})\cdot  \mathcal T_{j_2}^{(\ell)} P_{j_2} (k^{\bar{A}}_{\ell \delta})\big) 
- \partial_\ga \big(\mathcal{T}_{j_1}^{(i)}  P_{j_1} (k^{\bar{B}}_{i a})\cdot  \mathcal T_{j_2}^{(\ell)} P_{j_2} (k^{\bar{A}}_{\ell \delta})\big)\\
& \hphantom{(j_1,j_2)} \hphantom{=\f12 \sum_{(j_1, j_2) \in \big\{ (j',j'+J'),\, (j'+J',j') \big\} } \Bigg[ P_{j_3}(m(e)_{\bA\bB})}
+ \partial_\delta \big( (\mathcal{T}_{j_1}^{(i)}  P_{j_1} (k^{\bar{B}}_{i a})\cdot \mathcal T_{j_2}^{(\ell)} P_{j_2} (k^{\bar{A}}_{\ell \ga})\big)\Big)\\
& \hphantom{(j_1,j_2)}\hphantom{=\f12 \sum_{(j_1, j_2) \in \big\{ (j',j'+J'),\, (j'+J',j') \big\} } \Bigg[}
 + P_{j_3}(m(e))\cdot \mathcal{T}_{j_1}  P_{j_1} k \cdot \mathcal T_{j_2} P_{j_2}(\partial g\cdot k + \omega \cdot k)
\Big].
\end{align*}
Applying a $(-\partial_b)$ derivative to the above and then subtracting the same expression but with the roles of $a$ and $b$ inverted, we obtain:
\begin{align*}
& \sum_{(j_1, j_2) \in \big\{ (j',j'+J'),\, (j'+J',j') \big\} }  P_{j_3}(m(e)_{\bA\bB}) \Bigg[-\partial_b \Big( \mathcal{T}_{j_1}^{(i)} P_{j_1} (k^{\bar{B}}_{i \ga})\cdot  P_{j_2} (k^{\bar{A}}_{a\delta})\Big) 
+\partial_a \Big(\mathcal{T}_{j_1}^{(i)} P_{j_1} (k^{\bar{B}}_{i \ga})\cdot  P_{j_2} (k^{\bar{A}}_{b\delta})\Big)\Bigg]\\
& \hphantom{(j_1,j_2)}= 
\f12 \sum_{(j_1, j_2) \in \big\{ (j',j'+J'),\, (j'+J',j') \big\} } \Bigg[ -\partial_b \Bigg(P_{j_3}(m(e)_{\bA\bB})\Big(
\partial_a \big(\mathcal{T}_{j_1}^{(i)}  P_{j_1} (k^{\bar{B}}_{i \ga})\cdot  \mathcal T_{j_2}^{(\ell)} P_{j_2} (k^{\bar{A}}_{\ell \delta})\big) 
\\
& \hphantom{(j_1,j_2)} \hphantom{(j_1,j_2)} \hphantom{=\f12 \sum_{(j_1, j_2) \in \big\{ (j',j'+J'),\, (j'+J',j') \big\} } \Bigg[ P_{j_3}(m(e)_{\bA\bB})}
- \partial_\ga \big(\mathcal{T}_{j_1}^{(i)}  P_{j_1} (k^{\bar{B}}_{i a})\cdot  \mathcal T_{j_2}^{(\ell)} P_{j_2} (k^{\bar{A}}_{\ell \delta})\big)\\
& \hphantom{(j_1,j_2)} \hphantom{(j_1,j_2)} \hphantom{=\f12 \sum_{(j_1, j_2) \in \big\{ (j',j'+J'),\, (j'+J',j') \big\} } \Bigg[ P_{j_3}(m(e)_{\bA\bB})}
+ \partial_\delta \big( (\mathcal{T}_{j_1}^{(i)}  P_{j_1} (k^{\bar{B}}_{i a})\cdot \mathcal T_{j_2}^{(\ell)} P_{j_2} (k^{\bar{A}}_{\ell \ga})\big)\Big)\Bigg)\\
& \hphantom{(j_1,j_2)} \hphantom{= 
\f12 \sum_{(j_1, j_2) \in \big\{ (j',j'+J'),\, (j'+J',j') \big\} } \Bigg[}
+ \partial_a \Bigg(P_{j_3}(m(e)_{\bA\bB})\Big(
\partial_b \big(\mathcal{T}_{j_1}^{(i)}  P_{j_1} (k^{\bar{B}}_{i \ga})\cdot  \mathcal T_{j_2}^{(\ell)} P_{j_2} (k^{\bar{A}}_{\ell \delta})\big) 
\\
& \hphantom{(j_1,j_2)} \hphantom{(j_1,j_2)} \hphantom{=\f12 \sum_{(j_1, j_2) \in \big\{ (j',j'+J'),\, (j'+J',j') \big\} } \Bigg[ P_{j_3}(m(e)_{\bA\bB})}
- \partial_\ga \big(\mathcal{T}_{j_1}^{(i)}  P_{j_1} (k^{\bar{B}}_{i b})\cdot  \mathcal T_{j_2}^{(\ell)} P_{j_2} (k^{\bar{A}}_{\ell \delta})\big)\\
& \hphantom{(j_1,j_2)} \hphantom{(j_1,j_2)} \hphantom{=\f12 \sum_{(j_1, j_2) \in \big\{ (j',j'+J'),\, (j'+J',j') \big\} } \Bigg[ P_{j_3}(m(e)_{\bA\bB})}
+ \partial_\delta \big( (\mathcal{T}_{j_1}^{(i)}  P_{j_1} (k^{\bar{B}}_{i b})\cdot \mathcal T_{j_2}^{(\ell)} P_{j_2} (k^{\bar{A}}_{\ell \ga})\big)\Big)\Bigg)
\\
& \hphantom{(j_1,j_2)} \hphantom{(j_1,j_2)}\hphantom{=\f12 \sum_{(j_1, j_2) \in \big\{ (j',j'+J'),\, (j'+J',j') \big\} } \Bigg[}
+ \bar\partial \Big(P_{j_3}(m(e))\cdot \mathcal{T}_{j_1}  P_{j_1} k \cdot \mathcal T_{j_2} P_{j_2}(\partial g\cdot k + \omega \cdot k)\Big)
\\
& \hphantom{(j_1,j_2)} \hphantom{(j_1,j_2)}\hphantom{=\f12 \sum_{(j_1, j_2) \in \big\{ (j',j'+J'),\, (j'+J',j') \big\} } \Bigg[}
+ \bar\partial\big(P_{j_3}(m(e)_{\bA\bB})\big) \mathcal T_{j_1} P_{j_1} k\cdot P_{j_2} k 
\Big]
\end{align*}
Returning to \eqref{Total derivative wedge}, contracting with $P_{j_3}(m(e)_{\bA\bB})$ and using the above expression to substitute the first line in the right hand side, we infer \eqref{Schematic extraction two derivatives wedge product}.

Finally, let show that \eqref{Total derivative wedge} implies \eqref{Improved estimate wedge infinity}. Starting from the expression \eqref{Total derivative wedge} and using the bounds \eqref{Mikhlin}--\eqref{Mikhlin infinity} for $\mathcal{T}^{(i)}_{j'}$, we can estimate:
\begin{align}
\| & P_j  \Big( \big(P_{j'} k  \wedge P_{j'+J'} k \big)^{\bA\bB}_{ab\ga\delta} + \big(P_{j'+J'} k \wedge P_{j'} k \big)^{\bA\bB}_{ab\ga\delta} \Big) \|_{L^1 L^\infty}\\
& \lesssim \sum_{(j_1, j_2) \in \big\{ (j',j'+J'),\, (j'+J',j') \big\} } \Bigg( \| P_j \big( \bar{\partial}\big( \mathcal{T}_{j_1}^{(i)} P_{j_1} k \cdot P_{j_2} k\big)\big) \|_{L^1 L^\infty}
  +\| P_j \big( \mathcal{T}_{j_1}^{(i)} P_{j_1}\big(  g\cdot\partial g \cdot k + \omega\cdot k\big)\cdot  P_{j_2}k\big) \|_{L^1 L^\infty} \nonumber \\
& \hphantom{andan \sum_{(j_1, j_2) \in \big\{ (j',j'+J'),\, (j'+J',j') \big\} }} + \| P_j \big( P_{j_1}\big(\partial g \cdot k+ \omega\cdot k\big)\cdot  \mathcal{T}_{j_2}^{(i)} P_{j_2}k\big) \|_{L^1 L^\infty} \Bigg) \nonumber \\
& \lesssim \sum_{(j_1, j_2) \in \big\{ (j',j'+J'),\, (j'+J',j') \big\} } \Bigg( 2^j \|\mathcal{T}^{(i)}_{j_1} P_{j_1}k\|_{L^2 L^\infty} \|P_{j_2} k \|_{L^2 L^\infty} + 2^{2j} \|  P_j \big( \mathcal{T}_{j_1}^{(i)} P_{j_1}\big( g\cdot \partial g \cdot k + \omega\cdot k\big)\cdot  P_{j_2}k\big) \|_{L^1 L^{\f32}} \nonumber \\
&\hphantom{andan \sum_{(j_1, j_2) \in \big\{ (j',j'+J'),\, (j'+J',j') \big\} }} + 2^{2j} \| P_j \big( P_{j_1}\big( g\cdot\partial g \cdot k+ \omega\cdot k\big)\cdot  \mathcal{T}_{j_2}^{(i)} P_{j_2}k\big)   \|_{L^1 L^{\f32}}  \nonumber \\
& \lesssim \sum_{(j_1, j_2) \in \big\{ (j',j'+J'),\, (j'+J',j') \big\} } \Bigg( 2^{j-j'} \| P_{j_1}k\|_{L^2 W^{\delta_0, \infty}} \|P_{j_2} k \|_{L^2 L^\infty} + 2^{2j} \| \mathcal{T}_{j_1}^{(i)} P_{j_1}\big(  g\cdot\partial g \cdot k + \omega\cdot k\big)\|_{L^{\f43} L^{\f{12}5}} \| P_{j_2}k \|_{L^4 L^4} \nonumber \\
& \hphantom{andan \sum_{(j_1, j_2) \in \big\{ (j',j'+J'),\, (j'+J',j') \big\} }} + 2^{2j} \| P_{j_1}\big( g\cdot \partial g \cdot k + \omega\cdot k\big)\|_{L^{\f43} L^{\f{12}5}} \| \mathcal{T}_{j_2}^{(i)} P_{j_2}k \|_{L^4 L^4} \Bigg) \nonumber \\
& \lesssim \sum_{(j_1, j_2) \in \big\{ (j',j'+J'),\, (j'+J',j') \big\} } \Bigg(  2^{j-(1-\delta_0)j'} \| P_{j_1}k\|_{L^2 L^\infty} \| P_{j_2}k\|_{L^2 L^\infty} \nonumber \\
&\hphantom{ \lesssim \sum_{(j_1, j_2) \in \big\{ (j',j'+J'),\, (j'+J',j') \big\} } \Bigg(}
+ 2^{2j-j'} \| P_{j_1}\big(  g\cdot\partial g \cdot k + \omega\cdot k\big)\|_{L^{\f43} L^{\f{12}5}} \| P_{j_2}k \|_{L^4 L^4} \Bigg) \nonumber \\
& \lesssim 2^{j-(1-\delta_0)j'} \| k_{j'}\|_{L^2 L^\infty}\| k_{j'+J'}\|_{L^2 L^\infty} + 2^{2j-(1+2\delta_0)j'} \big( \|  g\cdot\partial g\|_{L^2 L^6} + \|\omega\|_{L^2 L^6}\big) \| k \|^2_{L^4 W^{2\delta_0,4}}, \nonumber
\end{align}
which proves \eqref{Improved estimate wedge infinity}.

\end{proof}

\subsection{Improved bounds for the curvature tensors $R$, $R^\perp$}
In this section, we will use the bounds for $k$ established in Lemmas \ref{lem:Energy estimates} and \ref{lem:Strichartz estimates}, as well as Proposition \ref{prop:Derivative k}, to obtain estimates for the Riemann tensor $R_{\a \b \ga \delta}$ of the metric $g$ and the curvature tensor $R^{\perp \bA \bB}_{\a\b}$ of the normal bundle $NY$.

In particular, we will establish the following result:

\begin{lemma}\label{lem:Riemann estimates}
Let $Y:[0,T)\times \mathbb T\rightarrow \mathcal N$ and $\{ e_{\bA}\}_{\bA=4}^n$ be as in the statement of Proposition \ref{prop:Bootstrap}. Then the Riemann curvature tensor $R$ satisfies
\begin{equation}\label{Estimate R easy}
 \|R \|_{L^2 H^{\f16+s_1+\f14\delta_0}} +\|R \|_{L^{\f74} W^{s_1+\f14\delta_0,\f73}} + \|  R \|_{L^\infty H^{s-3+\delta_0}}  \lesssim C_0^2 \mathcal{D}^2,
\end{equation}
\begin{equation}\label{Estimate R LHH easy}
 \|R^{\natural} \|_{L^2 H^{\f16+s_1+\f14\delta_0}}  +\|R^\natural \|_{L^{\f74} W^{s_1+\f14\delta_0,\f73}}+ \|  R^{\natural} \|_{L^\infty H^{s-3+\delta_0}}  \lesssim C_0^2 \mathcal{D}^2
\end{equation}
and
\begin{equation}\label{Estimate d R easy}
\|  \partial R \|_{L^\infty H^{s-4+\delta_0}} + \|  \partial R^\natural \|_{L^\infty H^{s-4+\delta_0}}  \lesssim C_0^2 \mathcal{D}^2,
\end{equation}
(where $R^{\natural}$ was defined by \eqref{R LHH}), as well as the following estimates:
\begin{itemize}
\item For $i,j\in\{1,2,3\}, \, \ga, \delta \in \{0,1,2,3\}$:
\begin{equation}\label{Estimate R hard}
 \|  R_{ij\ga\delta} \|_{L^1 W^{-1+\delta_0, \infty}}   \lesssim C_0^2 \mathcal{D}^2
\end{equation}
(in view of the symmetries of the Riemann curvature tensor, the above corresponds to precisely those components  of $R$ with at most one index equal to $0$).
\item For $\a,\ga, \delta \in \{0,1,2,3\}$:
\begin{align}\label{Estimate R hard 2}
 \|  R_{\a 0\ga\delta} - R^{\natural}_{\a 0\ga\delta} \|_{L^1 W^{-1+\delta_0, \infty}} & + \|  \partial_0 R_{\a 0\ga\delta} - \partial_0 R^{\natural}_{\a 0\ga\delta} \|_{L^2 H^{-\f56+s_1+\f14\delta_0}}\\
&  +  \|  \partial_0 R_{\a 0\ga\delta} - \partial_0 R^{\natural}_{\a 0\ga\delta} \|_{L^{\f74} W^{-1+s_1+\f14\delta_0,\f73}}\lesssim C_0^2 \mathcal{D}^2.\nonumber 
\end{align}
\end{itemize}
In the above,  the constants implicit in the $\lesssim$ notation are independent of $C_0$ (but depend only on $s$). The normal curvature tensor $R^\perp$ also satisfies the same estimates, i.e.:
\begin{equation}\label{Estimate R perp easy}
  \| R^{\perp} \|_{L^2 H^{\f16+s_1+\f14\delta_0}} +\|R^\perp \|_{L^{\f74} W^{s_1+\f14\delta_0,\f73}}\| + R^{\perp} \|_{L^\infty H^{s-3+\delta_0}}  \lesssim C_0^2 \mathcal{D}^2,
\end{equation}
\begin{equation}\label{Estimate R perp LHH easy}
  \| R^{\perp\natural} \|_{L^2 H^{\f16+s_1+\f14\delta_0}} +\|R^{\perp\natural} \|_{L^{\f74} W^{s_1+\f14\delta_0,\f73}}+ \|  R^{\perp\natural} \|_{L^\infty H^{s-3+\delta_0}}  \lesssim C_0^2 \mathcal{D}^2,
\end{equation}
\begin{equation}\label{Estimate d R perp easy}
\| \partial R^\perp \|_{L^\infty H^{s-4+\delta_0}} + \| \partial R^{\perp\natural} \|_{L^\infty H^{s-4+\delta_0}}  \lesssim C_0^2 \mathcal{D}^2,
\end{equation}
as well as 
\begin{equation}\label{Estimate R perp hard}
\| (R^{\perp})_{ij}^{\bA\bB} \|_{L^1 W^{-1+\delta_0, \infty}}  \lesssim C_0^2 \mathcal{D}^2 \quad \text{for }i,j\in\{1,2,3\}
\end{equation}
and
\begin{align}\label{Estimate R perp hard 2}
\| (R^{\perp})_{\a 0}^{\bA\bB} -(R^{\perp\natural})_{\a 0}^{\bA\bB}  \|_{L^1 W^{-1+\delta_0, \infty}} &+ \|  \partial_0 (R^{\perp})_{\a 0}^{\bA\bB} - \partial_0 (R^{\perp\natural})_{\a 0}^{\bA\bB} \|_{L^2 H^{-\f56+s_1+\f14\delta_0}} \\
& + \|  \partial_0 (R^{\perp})_{\a 0}^{\bA\bB} - \partial_0 (R^{\perp\natural})_{\a 0}^{\bA\bB} \|_{L^{\f74} W^{s_1+\f14\delta_0,\f73}} \lesssim C_0^2 \mathcal{D}^2, \nonumber 
\end{align}
where $R^{\perp\natural}$ was defined by \eqref{R perp LHH}.
\end{lemma}

\begin{proof}
For most of the proof of Lemma \ref{lem:Riemann estimates}, we will use the shorthand notation $f_j$ and $f_{\le j}$ for the Littlewood--Paley projections $P_j f$ and $P_{\le j} f$, respectively, when no other index appears in the same expression for $f$. Recall also our convention \eqref{Wedge product} for the wedge product ($\wedge$) operation, through which  the Gauss and Ricci equations (\eqref{Gauss} and \eqref{Ricci}) can be expressed as
\begin{align}
R_{\a\b\ga\delta} &= m_{\bA\bB} \cdot \big( k \wedge k \big)^{\bA\bB}_{\a\b\ga\delta} = m(e) \cdot k \wedge k, \label{Gauss 2}\\
 R^{\perp\bA\bB}_{\a\b} & = g^{\ga\delta} \cdot \big( k \wedge k)^{\bA\bB}_{\a\b\ga\delta} = g \cdot k \wedge k \label{Ricci 2}
\end{align}
(here, we made use of the symmetry of $m_{\bA\bB}$ and $g^{\ga\delta}$ in $(\bA, \bB)$ and $(\ga, \delta)$, respectively).

We will only show how to prove the bounds \eqref{Estimate R easy}--\eqref{Estimate R hard 2} for $R$, since the analogous bounds \eqref{Estimate R perp easy}--\eqref{Estimate R perp hard 2} for $R^\perp$ follow in exactly the same way. We will make use of the following estimates, which are a consequence of the bootstrap bound \eqref{Bootstrap bound} for $Y$:
\begin{align}\label{First useful bound}
\| k \|_{L^2 W^{s_0 -\f12, \infty}}&  + \| k \|_{L^4 W^{s_0 + \f1{12},4}} + \| k \|_{L^{\f72} W^{s_0,\f{14}3}} + \| k \|_{L^\infty H^{s-2}}\\
& + \| \partial k \|_{L^2 W^{s_0 -\f32, \infty}}  + \| \partial k \|_{L^4 W^{s_0 -1+ \f1{12},4}} + \| \partial k \|_{L^{\f72} W^{s_0-1,\f{14}3}} + \| \partial k \|_{L^\infty H^{s-3}} \lesssim C_0 \mathcal{D}, \nonumber 
\end{align}

\begin{equation}\label{Second useful bound}
\| \partial g \|_{L^2 H^{1+s_1 +\f16}}+ \| \omega \|_{L^2 H^{1+s_1 +\f16}}\lesssim C_0 \mathcal{D},
\end{equation}
and
\begin{equation}
 \sum_{\bA = 4}^{n}\|e_{\bA}- \delta^{(n+1)}_{\bA} \|_{L^\infty H^{s-1}} + \|\partial e\|_{L^\infty H^{s-2}} + \| \partial e \|_{L^2 H^{1+s_1 +\f16}}   \lesssim C_0 \mathcal{D}
\end{equation}
where $\delta^{(n+1)}_{m}$ is the constant $(n+1)$-dimensional vector
\[
\big( \delta^{(n+1)}_m\big)^A =\begin{cases} 1, &\text{ if }A=m,\\ 0, & \text{ if }A \neq m. \end{cases}
\]

Using the functional inequalities from Lemma \ref{lem:Functional inequalities} to estimate products of the form $\|f_1 \cdot f_2\|_{W^{\sigma,p}(\bar\Sigma_\tau)}$, we readily obtain:
\begin{align*}
\| R\|_{L^\infty H^{s-3+\delta_0}}  & =  \| m(e)\cdot k \wedge k \|_{L^\infty H^{s-3+\delta_0}} \\
& \stackrel{Lem~\ref{lem:Functional inequalities}}{\lesssim}
 \|m(e)\|_{L^\infty H^{s-1}} \|k \wedge k \|_{L^\infty H^{s-3+\delta_0}}  \\
& \stackrel{Lem~\ref{lem:Functional inequalities}}{\lesssim}
 \|m(e)\|_{L^\infty H^{s-1}} \|k\|^2_{L^\infty H^{s-2}} \\
& \stackrel{\hphantom{Lem }\eqref{Bootstrap bound}}{\lesssim}
C_0^2 \mathcal D^2,
\end{align*}
\begin{align*}
\| R\|_{L^2 H^{\f16 + s_1 +\f14 \delta_0}}  & = \Bigg( \int_0^T  \| m(e)\cdot k \wedge k \|^2_{H^{s-\f52 -\f74\delta_0}(\bar\Sigma_{\tau})} \, d\tau\Bigg)^{\f12}\\
& \stackrel{Lem~\ref{lem:Functional inequalities}}{\lesssim}
\Bigg( \int_0^T  \| m(e)\|^2_{H^{s-1}(\bar\Sigma_{\tau})} \|k \wedge k \|^2_{H^{s-\f52 -\f74\delta_0}(\bar\Sigma_{\tau})} \, d\tau\Bigg)^{\f12}  \\
&  \stackrel{\hphantom{Lem~\ref{Lem:Functional inequalities}}}{\lesssim}
\|m(e)\|_{L^\infty H^{s-1}} \| k \wedge k\|_{L^2 H^{s-\f52 -\f74\delta_0}}\\
&  \stackrel{\hphantom{Lem~\ref{Lem:Functional inequalities}}}{\lesssim}
\|m(e)\|_{L^\infty H^{s-1}} \cdot \sum_j \Big(  2^{(s-\f52 -\f74\delta_0)j}\| P_j (k_{\le j} \wedge k_j)\|_{L^2 L^2} + 2^{(s-\f52 -\f74\delta_0)j}\| P_j \big(\sum_{j'>j} k_{j'} \wedge k_{j'}\big)\|_{L^2 L^2} \Big) \\
&  \stackrel{\hphantom{Lem~\ref{Lem:Functional inequalities}}}{\lesssim}
\|m(e)\|_{L^\infty H^{s-1}} \cdot \sum_j \Big(  \| 2^{-\f12 j}k_{\le j}\|_{L^2 L^\infty}\| 2^{(s-2-\delta_0)j} k_j\|_{L^\infty L^2} \\
&\hphantom{\stackrel{\hphantom{Lem~\ref{Lem:Functional inequalities}}}{\lesssim}
\|m(e)\|_{L^\infty H^{s-1}} \cdot \sum_j \Big(}
 + 2^{-\f34 \delta_0 j}\sum_{j'>j} 2^{-(s-\f52 -\delta_0)(j'-j)}\| 2^{(s-\f52-\f1{12}-\delta_0)j'}k_{j'}\|_{L^4 L^4} \cdot \| 2^{(\f1{12})j'} k_{j'}\|_{L^4 L^4} \Big) \\
&  \stackrel{\hphantom{Lem~\ref{Lem:Functional inequalities}}}{\lesssim}
\|m(e)\|_{L^\infty H^{s-1}} \cdot \Big(\|k\|_{L^2 W^{-\f12,\infty}}\|k\|_{L^\infty H^{s-2}} + \|k\|^2_{L^4 W^{\f1{12}+s_0,4}}\Big)
\\
& \stackrel{\hphantom{Lem~}\eqref{Bootstrap bound}}{\lesssim}
C_0^2 \mathcal D^2
\end{align*}
and
\begin{align*}
\| R\|_{L^{\f74} W^{s_1 +\f14 \delta_0,\f73}}
&  \stackrel{Lem~\ref{lem:Functional inequalities}}{\lesssim}
\|m(e)\|_{L^\infty H^{s-1}} \| k \wedge k\|_{L^{\f74} W^{s_1 +\f14 \delta_0,\f73}}\\
&  \stackrel{\hphantom{Lem~\ref{Lem:Functional inequalities}}}{\lesssim}
\|m(e)\|_{L^\infty H^{s-1}} \cdot \sum_j \Big(  2^{(s_1+\f14\delta_0)j}\| P_j (k_{\le j} \wedge k_j)\|_{L^{\f74} L^{\f73}} + 2^{(s_1+\f14\delta_0)j}\| P_j \big(\sum_{j'>j} k_{j'} \wedge k_{j'}\big)\|_{L^{\f74} L^{\f73}} \Big) \\
&  \stackrel{\hphantom{Lem~\ref{Lem:Functional inequalities}}}{\lesssim}
\|m(e)\|_{L^\infty H^{s-1}} \cdot \sum_j \Big(  \| k_{\le j}\|_{L^{\f72} L^{\f{14}3}}\| 2^{(s_1+\f14\delta_0)j} k_j\|_{L^{\f72} L^{\f{14}3}} \\
&\hphantom{\stackrel{\hphantom{Lem~\ref{Lem:Functional inequalities}}}{\lesssim}
\|m(e)\|_{L^\infty H^{s-1}} \cdot \sum_j \Big(}
 + 2^{-\f14 \delta_0 j}\sum_{j'>j} 2^{-(s_1 +\f12\delta_0)(j'-j)}\| 2^{(s_1+\f12\delta_0)j'}k_{j'}\|_{L^{\f72} L^{\f{14}3}} \cdot \|  k_{j'}\|_{L^{\f72} L^{\f{14}3}} \Big) \\
&  \stackrel{\hphantom{Lem~\ref{Lem:Functional inequalities}}}{\lesssim}
\|m(e)\|_{L^\infty H^{s-1}} \cdot  \|k\|^2_{L^{\f72} W^{s_0-\f12\delta_0,\f{14}3}}
\\
& \stackrel{\hphantom{Lem~\ref{lem:Functional inequalities}}\eqref{Bootstrap bound}}{\lesssim}
C_0^2 \mathcal D^2.
\end{align*}
Combining the above bounds, we obtain \eqref{Estimate R easy}.

The bound \eqref{Estimate d R easy} for $\partial R$ also follows in a similar way, using the schematic expression
\[
\partial R = \partial e \cdot k \wedge k + m(e) \cdot \partial k \wedge k.
\]
In particular:
\begin{align*}
\|\partial R\|_{L^\infty H^{s-4+\delta_0}}  & \lesssim  \| \partial e\cdot k \wedge k \|_{L^\infty H^{s-4+\delta_0}} +  \| m(e) \cdot \partial k \wedge k \|_{L^\infty H^{s-4+\delta_0}} \\
& \stackrel{Lem~\ref{lem:Functional inequalities}}{\lesssim}
 \| \partial e\|_{L^\infty H^{s-2}} \|k \wedge k \|_{L^\infty H^{s-3+\delta_0}} +  \| m(e)\|_{L^\infty H^{s-1}} \|\partial k \wedge k \|_{L^\infty H^{s-4+\delta_0}}  \\
& \stackrel{Lem~\ref{lem:Functional inequalities}}{\lesssim}
  \| \partial e\|_{L^\infty H^{s-2}} \|k\|^2_{L^\infty H^{s-2}} +  \| m(e)\|_{L^\infty H^{s-1}} \|\partial k\|_{L^\infty H^{s-3}} \|k \|_{L^\infty H^{s-2}}  \\
& \stackrel{\eqref{Bootstrap bound}}{\lesssim}
C_0^2 \mathcal D^2.
\end{align*}

The bounds \eqref{Estimate R LHH easy} and  \eqref{Estimate d R easy} for $R^\natural$ follows by arguing exactly in the same way as for the proof of the analogous bounds for $R$,  since $R^\natural$ has a similar schematic expression as $R$:
\[
R^\natural = \sum_{j}\sum_{j'>j} m(e)_j \cdot  k_{j'} \wedge k_{j'}.
\]

The Littlewood--Paley projection $R_j \doteq P_j R$ can be schematically decomposed into almost orthogonal pieces as follows:
\begin{align}\label{Decomposition R}
R_j = & m(  e)_j \cdot k_{\le j} \wedge k_{\le j} 
+ m(e)_j \cdot \sum_{j_0 \le j}\sum _{j'>j_0}P_{j_0} \big(k_{j'} \wedge k_{j'}\big) \\
&+ m(e)_{\le j} \cdot k_{\le j} \wedge k_{j} 
+ m(e)_{\le j} \cdot  HH_{j} (k \wedge k) \nonumber \\
&+\sum_{j'>j} P_j\big( m(e)_{j'} \cdot k_{\le j'} \wedge k_{j'} \big)
 + \sum_{j'>j} \sum_{j''>j'} P_j \big( m(e)_{j'} \cdot P_{j'} \big( k_{j''} \wedge k_{j''} \big) \big),\nonumber 
\end{align}
where
\begin{equation}\label{HH k}
\big(HH_j (k \wedge k)\big)^{\bA\bB}_{\a\b\ga\delta} \doteq  \sum_{J=-2}^{+2} \sum_{\substack{j',j''>j+J,\\|j''-j'|\le 2}} P_{j+J} \Big( \big(P_{j'} k \wedge P_{j''} k\big)^{\bA\bB}_{\a\b\ga\delta} \Big).
\end{equation}
\begin{remark*}
Let us note that, in the decomposition above, we have grouped similarly looking terms into the same schematic summand; for instance, $\sum_{j'>j} P_j\big( m(e)_{j'} \cdot k_{\le j'} \wedge k_{j'} \big)$ contains any terms in the high-low decomposition of the form $\sum_{j_1 \in J_1}P_j\big( m(e)_{j'} \cdot k_{j_1} \wedge k_{j'+c} \big)$ when $\max\{ i \in J_1 \} \le j$ and $c$ is bounded by an absolute constant.
\end{remark*}
The low-high-high interactions in the decomposition above are captured by $R^\natural$: In view of the definition \eqref{R LHH} of $R^{\natural}$, 
we have
\begin{align}\label{Decomposition R LHH}
R^{\natural}_j  = &  
\sum_{\substack{j_2\ge j_1-2\\ |j_3-j_2|\le 2}} P_j \Big( m(e)_{j_1} k_{j_2} \wedge k_{j_3} \Big)
\\
=&  m(e)_{\le j} \cdot HH_j  \big( k \wedge k \big) + \sum_{j'>j-2} \sum_{j''\ge j'} P_j \big( m(e)_{j'} \cdot P_{j'}(k_{j''}\wedge k_{j''}) \big) \nonumber 
\end{align}
and 
\begin{align}\label{Decomposition R difference}
R_j - R_j^{\natural} 
= & \sum_{j_1, j_2, j_3} P_j \Big( m(e)_{j_1} k_{j_2} \wedge k_{j_3} \Big) - \sum_{\substack{j_2\ge j_1-2\\|j_3-j_2|\le 2}} P_j \Big( m(e)_{j_1} k_{j_2} \wedge k_{j_3} \Big)
\\
 = & \sum_{j_1, j_3 }P_j \Big( m(e)_{j_1} k_{\le j_1-3} \wedge k_{j_3} \Big) + \sum_{\substack{j_2 \ge j_1 -2 \\|j_3-j_2|\ge 3}} P_j \Big( m(e)_{j_1} k_{j_2} \wedge k_{j_3} \Big) \nonumber \\[5pt]
 = &  
m(e)_{\le j} \cdot k_{\le j} \wedge k_j + m(e)_j \cdot k_{\le j}\wedge k_{\le j}  + \sum_{j'>j} P_j \Big( m(e)_{j'} \cdot k_{\le j'} \wedge k_{j'} \Big).    \nonumber
\end{align}
Note that the difference $R_j - R^{\natural}_j$ does \textbf{not} contain any ``low-high-high'' interactions.

In order to prove \eqref{Estimate R hard}, we will establish the $L^1 L^\infty$ estimate
\begin{equation}\label{L1 estimate Rj}
\|2^{(-1 + 2 \delta_0)j}P_j R_{ab\ga\delta} \|_{L^1 L^\infty} \lesssim C_0^2 \mathcal{D}^2 \quad \text{for } a,b\in \{1,2,3\}, \ga, \delta \in \{0,1,2,3\}.
\end{equation}
\begin{remark*}
For the rest of the proof, we will revert to our use of $P_j f$ instead of $f_j$ for denoting the Littlewood--Paley projections of tensor components in the case of expressions that also involve tensorial indices.
\end{remark*}
\noindent  In order to establish \eqref{L1 estimate Rj}, we will have to make use of Proposition  \ref{prop:Derivative k}, which utilizes the special structure of the term $k\wedge k$, in  order to deal with certain parts of of the high-high interaction terms in the decomposition \eqref{Decomposition R}. 

In particular, the bound \eqref{L1 estimate Rj} for each term in the decomposition \eqref{Decomposition R} for $P_j R_{ab\ga\delta}$ is deduced as follows:
\begin{itemize}
\item The first and the third term in the right hand side of \eqref{Decomposition R} can be estimated simultaneously by
\begin{align}\label{First L1}
2^{(-1 + 2 \delta_0)j} \| P_{\le j}m(e) \cdot P_{\le j}k \wedge P_{\le j} k \|_{L^1 L^\infty} 
& \lesssim \| P_{\le j}m(e) \cdot 2^{(-\f12 + \delta_0)j} P_{\le j} k \wedge 2^{(-\f12 + \delta_0)j} P_{\le j} k \|_{L^1 L^\infty} \\
& \lesssim \| P_{\le j} m(e) \|_{L^\infty L^\infty} \| P_{\le j} k\|^2_{L^2 W^{-\f12+\delta_0,\infty}} \nonumber \\
&\lesssim \| m(e) \|_{L^\infty H^{s-1}} \|k\|^2_{L^2 W^{s_0-\f12,\infty}} \nonumber \\
& \lesssim C_0^2 \mathcal{D}^2. \nonumber
\end{align}

\item The second, fifth and sixth term in the right hand side of \eqref{Decomposition R} can be estimated using the following bound which holds for any triplet of indices $j_1, j_2, j_3$ with $j_1 \ge j$:
\begin{align}
2^{(-1 + 2 \delta_0)j} \| P_j \big( P_{j_1}m(e) \cdot & P_{j_2} k \wedge P_{j_3} k \big)\|_{L^1 L^\infty}\\
& \lesssim  2^{(1 + 2 \delta_0)j_1} \|  P_{j_1} m(e) \cdot P_{j_2} k \wedge P_{j_3} k \|_{L^1 L^{\f32}} \nonumber \\
& \lesssim 2^{\delta_0 (j-j_1-j_2-j_3)}\|2^{(1+3\delta_0)j_1} P_{j_1} m(e)\|_{L^2 L^6} \| 2^{\delta_0 j_2} P_{j_2} k \|_{L^4 L^4} \| 2^{\delta_0 j_3} P_{j_3} k \|_{L^4 L^4} \nonumber \\
& \lesssim 2^{\delta_0 (j-j_1-j_2-j_3)} \|m(e)\|_{L^2 W^{1+3\delta_0,6}} \|k\|^2_{L^4 W^{\delta_0,4}} \nonumber \\
& \lesssim 2^{\delta_0 (j-j_1-j_2-j_3)}  \|m(e)\|_{L^2 H^{2+3\delta_0}} \|k\|^2_{L^4 W^{\delta_0,4}} \nonumber \\
& \lesssim 2^{\delta_0 (j-j_1-j_2-j_3)} C_0^2 \mathcal{D}^2. \nonumber 
\end{align}
Note than the last line above we made use of the estimate \eqref{Second useful bound}.

\item For the fourth term in the right hand side of \eqref{Decomposition R} for $P_j R_{ab\ga\delta}$, we will  make use of Proposition \ref{prop:Derivative k}. To this end, we will make use of the fact that the indices  $a,b$ are $\neq 0$ and that the expression \eqref{HH k} for $HH_j (k \wedge k)$ is symmetric with respect to the Littlewood--Paley indices appearing in it, in the sense that it can be reexpressed as follows:
\[
(HH_{j}( k \wedge k )\big)^{\bA\bB}_{a b \ga \delta} = \sum_{J,J'=-2}^2 \sum_{j'> \max\{j+J, j+J+J'\}} P_{j+J} \Big( \big( P_{j'} k \wedge P_{j'+J'} k \big)^{\bA\bB}_{a b \ga \delta} + \big( P_{j'+J'} k \wedge P_{j'} k \big)^{\bA\bB}_{a b \ga \delta} \Big).
\]
In particular, we can estimate
\begin{align}\label{Last L1}
& 2^{(-1 + 2 \delta_0)j}  \| P_{\le j} m(e)_{\bA\bB} \cdot (HH_{j}( k \wedge k )\big)^{\bA\bB}_{a b \ga \delta} \|_{L^1 L^\infty}\\
& \hphantom{2}  \stackrel{\hphantom{\tiny{\text{Prop.} \ref{prop:Derivative k}}}} {\lesssim} 
\| m(e)\|_{L^\infty L^\infty} \sum_{J,J'=-2}^2 \Bigg( \sum_{j'>\max\{j+J,j+J+J'\}}2^{(-1 + 2 \delta_0)j} \times \nonumber\\
& \hphantom{ \hphantom{2}  \stackrel{\hphantom{\tiny{\text{Prop.} \ref{prop:Derivative k}}}} {\lesssim} 
\| m(e)\|_{L^\infty L^\infty} \sum_{J,J'=-2}^2 \Bigg( \sum}
\times \sum_{\bA, \bB \in \{4, \ldots, n\}}\| P_{j+J} \Big(\big( P_{j'} k \wedge P_{j'+J'} k \big)^{\bA\bB}_{a b \ga \delta} + \big( P_{j'+J'} k \wedge P_{j'} k \big)^{\bA\bB}_{a b \ga \delta} \Big) \|_{L^1 L^\infty} \Bigg) \nonumber \\
& \hphantom{2}  \stackrel{\tiny{\text{Prop.} \ref{prop:Derivative k}}} {\lesssim} 
\| m(e)\|_{L^\infty H^{s-1}} \cdot \sum_{J'=-2}^2 \sum_{j'>j} 2^{(-1+2 \delta_0) j}\Big( 2^{j-(1-\delta_0)j'} \| k_{j'}\|_{L^2 L^\infty}\| k_{j'+J'}\|_{L^2 L^\infty} \nonumber \\
&\hphantom{\hphantom{2}  \stackrel{\tiny{\text{Prop.} \ref{prop:Derivative k}}} {\lesssim} 
\| m(e)\|_{L^\infty H^{s-1}} \cdot \sum_{J'=-2}^2 \sum_{j'>j} 2^{(-1+2 \delta_0) j}\Big(}
+ 2^{2j- (1+2\delta_0)j'} \big( \| \partial g\|_{L^2 L^6} + \|\omega\|_{L^2 L^6}\big) \| k \|^2_{L^4 W^{2\delta_0,4}}\Big) \nonumber \\
& \hphantom{2}  \stackrel{\hphantom{\tiny{\text{Prop.} \ref{prop:Derivative k}}}} {\lesssim}  
\| m(e)\|_{L^\infty H^{s-1}} \cdot \sum_{j'>j} 2^{2 \delta_0(j-j')}\Big( \|k\|^2_{L^2 W^{-\f12+2\delta_0,\infty}} + \big( \| \partial g \|_{L^2 L^6}+ \|\omega\|_{L^2 L^6} \big) \| k \|^2_{L^4 W^{2\delta_0, 4}}\Big) \nonumber \\
& \hphantom{2}  \stackrel{\hphantom{\tiny{\text{Prop.} \ref{prop:Derivative k}}}} 
{\lesssim} C_0^2 \mathcal{D}^2. \nonumber
\end{align}
\end{itemize}
This completes the proof of \eqref{L1 estimate Rj}.

Finally, in order to establish the bound \eqref{Estimate R hard 2} (and thus complete the proof of Lemma \ref{lem:Riemann estimates}), it suffices to prove that
\begin{equation}\label{L1 estimate Rj difference}
2^{(-1 + 2 \delta_0)j}\|P_j R -P_j R^\natural\|_{L^1 L^\infty} \lesssim C_0^2 \mathcal{D}^2
\end{equation}
and
\begin{equation}\label{L2 estimate d Rj difference}
2^{(-\f56+s_1+\delta_0)j}\| \partial_0 (P_j R)-\partial_0 (P_j R^{\natural})\|_{L^2 L^2}  \|  +2^{(-1+s_1+\f12\delta_0)j} \|\partial_0 (P_j R)-\partial_0 (P_j R^{\natural}) \|_{L^{\f74} L^{\f73}} \lesssim C_0^2 \mathcal{D}^2.
\end{equation}
As we already remarked, the terms contained in the right hand side of the decomposition \eqref{Decomposition R difference} for $P_j R - P_j R^{\natural}$ are similar to the ones in the decomposition \eqref{Decomposition R} for $P_j R$, with the exception of the ``low-high-high''  term $m(e)_{\le j} \sum_{j'>j} P_j (k_{j'} \wedge k_{j'})$ (which appears in \eqref{Decomposition R} but not in \eqref{Decomposition R difference}). Therefore, the proof of \eqref{L1 estimate Rj difference} follows by repeating the same steps as for the proof of \eqref{L1 estimate Rj} but without the need to appeal to Proposition \ref{prop:Derivative k} for the low-high-high terms. For the proof of \eqref{L2 estimate d Rj difference}, we can similarly compute (recall that $s_1  =s-\f52-\f16-2\delta_0$, $s_0=s_1+\delta_0$):
\begin{align*}
2^{(-\f56+s_1+\delta_0)j} \|& \partial_0 (  P_j R) -\partial_0 (P_j R^{\natural})\|_{L^2 L^2}\\  \lesssim 
 &2^{(-\f56+s_1+\delta_0)j} \Bigg[
\|\partial e_{\le j} \cdot k_{\le j} \wedge k_j\|_{L^2 L^2} + \| m(e)_{\le j} \cdot \partial k_{\le j} \wedge k_j\|_{L^2 L^2} + \| m(e)_{\le j} \cdot k_{\le j} \wedge \partial k_j\|_{L^2 L^2} \\
 &\hphantom{2^{(-\f56+s_1+\delta_0)j} \Bigg[}
+ \|\partial e_j \cdot k_{\le j}\wedge k_{\le j}\|_{L^2 L^2}
+\|m(e)_j \cdot \partial k_{\le j}\wedge k_{\le j}\|_{L^2 L^2}\\
 & +2^{(-\f56+s_1+\delta_0)j} \Bigg[ + \Big\|\sum_{j'>j} P_j \Big( \partial e_{j'} \cdot k_{\le j'} \wedge k_{j'} + m(e)_{j'} \cdot \partial k_{\le j'} \wedge k_{j'} + m(e)_{j'} \cdot k_{\le j'} \wedge \partial k_{j'}\Big)\Big\|_{L^2L^2}   \Bigg]\\
\lesssim 
& 2^{(-1+\f16)j}\|\partial e_{\le j}\|_{L^\infty L^\infty}  \cdot \| k_{\le j}\|_{L^4 L^4}  \cdot 2^{s_0 j}\|k_j\|_{L^4 L^4}\\
& +\|m(e)_{\le j}\|_{L^\infty L^\infty}  \cdot 2^{(-1+\f1{12})j}\| \partial k_{\le j}\|_{L^4 L^4}  \cdot 2^{(s_0 +\f1{12}) j}\|k_j\|_{L^4 L^4}
\\ 
& +\|m(e)_{\le j}\|_{L^\infty L^\infty}  \cdot 2^{(s_0 -1+\f1{12})j}\| k_{\le j} \wedge \partial k_j\|_{L^2 L^{\f{36}{19}}}\\
& +  2^{(s-\f72-\delta_0)j}\|\partial e_j\|_{L^\infty L^\infty} \cdot \| k_{\le j}\|^2_{L^4 L^4} \\
&  + 2^{(s-\f52-\delta_0)j)} \|m(e)_j\|_{L^\infty L^\infty} \cdot 2^{-j}\|\partial k_{\le j}\|_{L^4 L^4} \|k_{\le j} \|_{L^4 L^4} 
\\
&+ \sum_{j'>j} \Big( 2^{(s_0+\f16)j}\|P_j \big( \partial e_{j'} \cdot k_{\le j'} \wedge k_{j'}\big) \|_{L^2 L^{\f65}} + 2^{(s_0+\f16)j}\|P_j \big( m(e)_{j'} \cdot \partial k_{\le j'} \wedge k_{j'} \big)\|_{L^2 L^{\f65}} \\
& \hphantom{+ \sum_{j'>j}++}
+ 2^{(s_0+\f16)j}\| P_j \big( m(e)_{j'} \cdot k_{\le j'} \wedge \partial k_{j'}\big)\|_{L^2 L^{\f65}}\Big)\\
\lesssim &
\|\partial e\|_{L^\infty W^{s-\f72-\delta_0,\infty}} \cdot \|k\|_{L^4 W^{s_0,4}}^2
\\
& + \|m(e)\|_{L^\infty W^{s-\f72-\delta_0,\infty}} \cdot \|\partial k\|_{L^4 W^{s_0-1+\f1{12},4}} \cdot   \| k\|_{L^4 W^{s_0+\f1{12},4}} \\
& + \|m(e)\|_{L^\infty L^\infty} \|k\|_{L^4 L^{\f{18}5}} \|\partial k\|_{L^4 W^{s_0-1+\f1{12},4}}\\
& + \sum_{j'>j}2^{-\delta_0 j'}\Big( \|\partial e_{j'}\|_{L^\infty W^{s_0+\f1{6},3}} \|k\|_{L^4 W^{\delta_0,4}}^2 \\
& \hphantom{+ \sum_{j'>j}2^{-\delta_0 j'}\Big(}
+ \|m(e)_{j'}\|_{L^\infty W^{s_0+1+\f1{6},3}} \|\partial k\|_{L^2 W^{-\f32+\delta_0,\infty}} \|k\|_{L^\infty H^{\f12+\delta_0}} \\
& \hphantom{+ \sum_{j'>j}2^{-\delta_0 j'}\Big(}
+ \|m(e)_{j'}\|_{L^\infty W^{s_0+1+\f1{6},3}} \| k\|_{L^2 W^{-\f12+\delta_0,\infty}} \|\partial k\|_{L^\infty H^{-\f12+\delta_0}}
 \Big)\\
\lesssim &
\|\partial e\|_{L^\infty H^{s-2}} \|k\|^2_{L^4 W^{s_0+\f1{12},4}} + \|m(e)\|_{L^\infty H^{s-1}} \|k\|_{L^4 W^{s_0+\f1{12},4}}\|\partial k\|_{L^4 W^{s_0-1+\f1{12},4}} \\
& + \|m(e)\|_{L^\infty H^{s-1}} \Big( \|k\|_{L^\infty H^{s-2}}\|\partial k\|_{L^2 W^{s_0-\f32,\infty}} 
+\|\partial k\|_{L^\infty H^{s-3}}\|k\|_{L^2 W^{s_0-\f12,\infty}} \Big)\\
\lesssim & C_0^2 \mathcal D^2
\end{align*}
and, similarly,
\begin{align*}
2^{(-1+s_1+\f12\delta_0)j} \|& \partial_0 (  P_j R) -\partial_0 (P_j R^{\natural})\|_{L^{\f74} L^{\f73}}\\  \lesssim 
 &2^{(-1+s_1+\f12\delta_0)j} \Bigg[
\|\partial e_{\le j} \cdot k_{\le j} \wedge k_j\|_{L^{\f74} L^{\f73}} + \| m(e)_{\le j} \cdot \partial k_{\le j} \wedge k_j\|_{L^{\f74} L^{\f73}} + \| m(e)_{\le j} \cdot k_{\le j} \wedge \partial k_j\|_{L^{\f74} L^{\f73}} \\
 &\hphantom{2^{(-\f56+s_1+\delta_0)j} \Bigg[}
+ \|\partial e_j \cdot k_{\le j}\wedge k_{\le j}\|_{L^{\f74} L^{\f73}}
+\|m(e)_j \cdot \partial k_{\le j}\wedge k_{\le j}\|_{L^{\f74} L^{\f73}}\\
 &+2^{(-1+s_1+\f12\delta_0)j} \Bigg[ + \Big\|\sum_{j'>j} P_j \Big( \partial e_{j'} \cdot k_{\le j'} \wedge k_{j'} + m(e)_{j'} \cdot \partial k_{\le j'} \wedge k_{j'} + m(e)_{j'} \cdot k_{\le j'} \wedge \partial k_{j'}\Big)\Big\|_{L^{\f74} L^{\f73}}   \Bigg]\\
\lesssim 
& 2^{-j}\|\partial e_{\le j}\|_{L^\infty L^\infty}  \cdot \| k_{\le j}\|_{L^{\f72} L^{\f{14}3}}  \cdot 2^{(s_0-\f12\delta_0) j}\|k_j\|_{L^{\f72} L^{\f{14}3}}\\
& +\|m(e)_{\le j}\|_{L^\infty L^\infty}  \cdot 2^{-j}\| \partial k_{\le j}\|_{L^{\f72} L^{\f{14}3}}  \cdot 2^{(s_0 -\f12\delta_0) j}\|k_j\|_{L^{\f72} L^{\f{14}3}}
\\ 
& +\|m(e)_{\le j}\|_{L^\infty L^\infty}  \cdot \| k_{\le j}\|_{L^{\f72} L^{\f{14}3}} \cdot 2^{(-1+s_0-\f12\delta_0)j}\| \partial k_j\|_{L^{\f72} L^{\f{14}3}}\\
& +  2^{(s-\f72-\delta_0)j}\|\partial e_j\|_{L^\infty L^\infty} \cdot \| k_{\le j}\|^2_{L^{\f72} L^{\f{14}3}} \\
&  + 2^{(s-\f52-\delta_0)j)} \|m(e)_j\|_{L^\infty L^\infty} \cdot 2^{-j}\|\partial k_{\le j}\|_{L^{\f72} L^{\f{14}3}} \|k_{\le j} \|_{L^{\f72} L^{\f{14}3}} 
\\
&+ \sum_{j'>j} \Big( 2^{(s_1+\f12\delta_0)j}\|P_j \big( \partial e_{j'} \cdot k_{\le j'} \wedge k_{j'}\big) \|_{L^{\f74} L^{\f{21}{16}}} + 2^{(s_1+\f12\delta_0)j}\|P_j \big( m(e)_{j'} \cdot \partial k_{\le j'} \wedge k_{j'} \big)\|_{L^{\f74} L^{\f{21}{16}}} \\
& \hphantom{+ \sum_{j'>j}++}
+ 2^{(s_1+\f12\delta_0)j}\| P_j \big( m(e)_{j'} \cdot k_{\le j'} \wedge \partial k_{j'}\big)\|_{L^{\f74} L^{\f{21}{16}}}\Big)\\
\lesssim &
\|\partial e\|_{L^\infty W^{(s-\f72-\delta_0)j,\infty}} \cdot \|k\|_{L^{\f72} W^{s_0-\f12\delta_0,\f{14}3}}^2
\\
& + \|m(e)\|_{L^\infty L^\infty} \cdot \|\partial k\|_{L^{\f72} W^{-1+s_0-\f12\delta_0,\f{14}3}} \cdot   \| k\|_{L^{\f72} W^{s_0-\f12\delta_0,\f{14}3}} \\
& + \sum_{j'>j}2^{-\delta_0 j'}\Big( \|\partial e_{j'}\|_{L^\infty W^{s-\f52-\delta_0,3}} \|k\|_{L^{\f72} L^{\f{14}3}}^2 \\
& \hphantom{+ \sum_{j'>j}2^{-\delta_0 j'}\Big(}
+ \|m(e)_{j'}\|_{L^\infty W^{s-\f32-\delta_0,3}} \|\partial k\|_{L^{\f72} W^{-1,\f{14}3}} \|k\|_{L^{\f72} L^{\f{14}3}} \\
& \hphantom{+ \sum_{j'>j}2^{-\delta_0 j'}\Big(}
+ \|m(e)_{j'}\|_{L^\infty W^{s-\f32-\delta_0,3}} \| k\|_{L^{\f72} L^{\f{14}3}} \|\partial k\|_{L^{\f72} W^{-1,\f{14}3}}
 \Big)\\
\lesssim &
\|\partial e\|_{L^\infty H^{s-2}} \|k\|^2_{L^{\f72} W^{s_0-\f12\delta_0,\f{14}3}} + \|m(e)\|_{L^\infty H^{s-1}} \|k\|_{L^{\f72} W^{s_0-\f12\delta_0,\f{14}3}}\|\partial k\|_{L^{\f72} W^{s_0-1-\f12\delta_0,\f{14}3}} \\
\lesssim & C_0^2 \mathcal D^2.
\end{align*}
Therefore, \eqref{L2 estimate d Rj difference} holds.

The bounds \eqref{Estimate R perp easy}--\eqref{Estimate R perp hard 2} for $R^\perp$ follow in exactly the same way as \eqref{Estimate R easy}--\eqref{Estimate R hard 2} . Thus, the proof of Lemma \ref{lem:Riemann estimates} is complete.

\end{proof}

\subsection{Bounds for $\mathcal F^\natural$ and $\mathcal F_\perp$}
In this section, we will establish a number of estimates for the source functions $\tilde{\mathcal F}^\natural$ and $\tilde{\mathcal F}_\perp$ appearing in the right hand side of equations \eqref{Mean curvature condition} and \eqref{Divergence condition frame} defining the balanced gauge condition. To this end, we will use  the bounds for $k$ provided by Lemmas \ref{lem:Energy estimates} and \ref{lem:Strichartz estimates}, as well as the estimates established in the previous section regarding the  curvature tensors $R$ and $R^\perp$.

\begin{lemma}\label{lem:Bounds F natural}
The following estimates hold for $\tilde{\mathcal F}^{\natural}$:
\begin{equation}\label{Bounds F natural easy}
\| \partial \tilde{\mathcal F}^{\natural}  \|_{L^\infty H^{s-3}} + \| \partial \tilde{\mathcal F}^{\natural}  \|_{L^2 H^{\f16+s_1}} + \| \partial \tilde{\mathcal F}^{\natural}  \|_{L^{\f74} W^{s_1+\f14\delta_0,\f73}}+ \| \tilde{\mathcal F}^{\natural}  \|_{L^1 W^{s_0, \infty}} \lesssim \mathcal D
\end{equation}
and
\begin{equation} \label{Bound F natural L1Linfty}
 \sum_{c,d=1}^3 \| \partial_0 \tilde{\mathcal F}^{\natural}_{cd} - R^{\natural}_{c0d0}  \|_{L^1 W^{-1+s_0, \infty}} \lesssim \mathcal D.
 \end{equation}
Moreover, 
\begin{equation}\label{Bound F natural higher order}
\sum_{c,d=1}^3 \Big( \| \partial_0 \big( \partial_0 \tilde{\mathcal F}^{\natural}_{cd} - R^{\natural}_{c0d0} \big) \|_{L^\infty H^{s-4}} + \| \partial_0 \big( \partial_0 \tilde{\mathcal F}^{\natural}_{cd} - R^{\natural}_{c0d0} \big) \|_{L^2 H^{-\f56+s_1}} +  \| \partial_0 \big( \partial_0 \tilde{\mathcal F}^{\natural}_{cd} - R^{\natural}_{c0d0} \big) \|_{L^{\f74} W^{-1+s_1+\f14\delta_0,\f73}}\Big) \lesssim \mathcal D.
\end{equation}
\end{lemma}

\begin{proof}
Recall that $\tilde{\mathcal F}^\natural$ was defined by the relation:
\begin{equation}\label{test relationF natural}
\tilde{\mathcal F}^{\natural}_{\a\b}(x^0,\bar x) \doteq \mathcal F^{\natural}_{\a\b}(x^0,\bar x) - \mathbb E \big[\mathcal F^{\natural}_{\a\b}|_{x^0=0}\big](x^0,\bar x),
\end{equation}
where $\mathcal F^{\natural}$ is given by the expression \eqref{F natural seed} and $\mathbb E$ is the extension operator defined by \eqref{Extension operator}. In view of Lemma \ref{lem:Estimates time operators} for the operator $\mathbb E$, we can readily estimate:
\begin{align}\label{First  bound Extension F natural}
\big\| & \partial \big( \mathbb E \big[\mathcal F^{\natural}|_{x^0=0}\big] \big)  \big\|_{L^\infty H^{s-3}} 
+ \big\| \partial \big( \mathbb E \big[\mathcal F^{\natural}|_{x^0=0}\big] \big)  \big\|_{L^2 H^{\f16+s_1}} \\
&+  \big\| \partial \big( \mathbb E \big[\mathcal F^{\natural}|_{x^0=0}\big] \big)  \big\|_{L^{\f74} W^{s_1+\f14\delta_0,\f73}}  + \big\| \partial \big( \mathbb E \big[\mathcal F^{\natural}   |_{x^0=0}\big] \big)  \big\|_{L^1 W^{-1+s_0, \infty}}
 \lesssim \| \mathcal F^{\natural}|_{x^0=0} \|_{H^{s-2}}.   \nonumber 
\end{align}
For the initial data term $\mathcal F^{\natural}|_{x^0=0}$, we can estimate:
\begin{align}
\| \mathcal F^{\natural}|_{x^0=0}\|_{H^{s-2}} 
& \stackrel{\hphantom{\eqref{Codazzi}}}{\lesssim} 
\sum_j  2^{(s-2)j}\| P_j \mathcal F^{\natural} |_{x^0=0}\|_{L^2}   \\
&  \stackrel{\hphantom{\eqref{Codazzi}}}{\lesssim}
 \sum_j \sum_{\bar j} \sum_{\substack{j_1,j_2 \ge \{\max j, \bar j \}-2,\\|j_1-j_2|\le 2}}
2^{(s-2)j} \Big\| P_j \Big(  P_{\bar j} m(e) \cdot P_{j_1} k \cdot P_{j_2} \mathcal T^{(b)}_{j_2}  k \Big) |_{x^0=0} \Big\|_{L^2}      \nonumber \\
&  \stackrel{\hphantom{\eqref{Codazzi}}}{\lesssim}
 \sum_j \sum_{\bar j} \sum_{\substack{j_1,j_2 \ge \{\max j, \bar j \}-2,\\|j_1-j_2|\le 2}}
2^{(s-\f12)j} \Big\| P_j \Big(  P_{\bar j} m(e) \cdot P_{j_1} k \cdot P_{j_2} \mathcal T^{(b)}_{j_2}  k \Big) |_{x^0=0} \Big\|_{L^1}      \nonumber \\
&  \stackrel{\hphantom{\eqref{Codazzi}}}{\lesssim}
 \sum_j \sum_{\bar j} \sum_{\substack{j_1,j_2 \ge \{\max j, \bar j \}-2,\\|j_1-j_2|\le 2}}
2^{(s-\f12)(j-j_1)+\f32(j_1-j_2)-\delta_0 (\bar j+j_2)}\cdot 2^{\delta_0 \bar j}\| P_{\bar j} m(e)|_{x^0=0}\|_{L^\infty} \nonumber \\
& \hphantom{\stackrel{\hphantom{\eqref{Codazzi}}}{\lesssim}
 \sum_j \sum_{\bar j} \sum_{\substack{j_1,j_2 \ge \{\max j, \bar j \}-2,\\|j_1-j_2|\le 2}}
2}
 \times  2^{(s-2)j_1}\| P_{j_1} k|_{x^0=0}\|_{L^2} \cdot  2^{(\f32+\delta_0)j_2}\| P_{j_2} \mathcal T^{(b)}_{j_2}  k  |_{x^0=0} \|_{L^2}      \nonumber \\
& \stackrel{\eqref{Mikhlin property}}{\lesssim} 
\|m(e)|_{x^0=0}\|_{L^\infty} \| k |_{x^0=0}\|_{H^{s-2}}^2   \nonumber \\
& \stackrel{\hphantom{\eqref{Codazzi}}}{\lesssim} \mathcal D^2,  \nonumber 
\end{align}
where, in the last line above, we made use of the bounds provided by Proposition \ref{prop: Bounds S0} for $e|_{x^0=0}$ and $k|_{x^0=0}$. Thus, returning to \eqref{First bound Extension F natural}, we obtain:
\begin{align}\label{Bound Extension F natural}
\big\| \partial \big( & \mathbb E \big[\mathcal F^{\natural}|_{x^0=0}\big] \big)  \big\|_{L^\infty H^{s-3}} 
+  \big\| \partial \big( \mathbb E \big[\mathcal F^{\natural}|_{x^0=0}\big] \big)  \big\|_{L^2 H^{\f16+s_1}} \\
& + \big\| \partial \big( \mathbb E \big[\mathcal F^{\natural}|_{x^0=0}\big] \big)  \big\|_{L^{\f74} W^{s_1+\f14\delta_0,\f73}}+ \big\| \partial \big( \mathbb E \big[\mathcal F^{\natural}   |_{x^0=0}\big] \big)  \big\|_{L^1 W^{-1+s_0, \infty}}
 \lesssim \mathcal D^2. \nonumber
\end{align}

In view of the bound \eqref{Bound Extension F natural} and the expression \eqref{test relationF natural} for $\tilde{\mathcal F}^{\natural}$, in order to prove the bounds \eqref{Bounds F natural easy} and \eqref{Bound F natural L1Linfty} it suffices to show that
\begin{equation}\label{Bounds seed F natural easy}
\| \partial \mathcal F^{\natural}  \|_{L^\infty H^{s-3}} + \| \partial \mathcal F^{\natural}  \|_{L^2 H^{\f16+s_1}} + \| \partial \mathcal F^{\natural}  \|_{L^{\f74} W^{s_1+\f14\delta_0,\f73}}+ \| \mathcal F^{\natural}  \|_{L^1 W^{s_0, \infty}} \lesssim \mathcal D
\end{equation}
and
\begin{equation} \label{Bound seed F natural L1Linfty}
 \sum_{c,d=1}^3 \| \partial_0 \mathcal F^{\natural}_{cd} - R^{\natural}_{c0d0}  \|_{L^1 W^{-1+s_0, \infty}} \lesssim \mathcal D
 \end{equation}
 We can readily calculate (using the expression \eqref{F natural seed}):
 \begin{align}\label{Seed F natural LinftyL2}
\| \partial & \mathcal F^{\natural}\|_{L^\infty H^{s-3}} 
 \stackrel{\hphantom{\eqref{Codazzi}}}{\lesssim} 
\sum_j  2^{(s-3)j}\|\partial P_j \mathcal F^{\natural} \|_{L^\infty L^2}   \\
&  \stackrel{\hphantom{\eqref{Codazzi}}}{\lesssim}
 \sum_j \sum_{\bar j} \sum_{\substack{j_1,j_2 \ge \{\max j, \bar j \}-2,\\|j_1-j_2|\le 2}}
2^{(s-3)j} \Big\| \partial P_j \Big(  P_{\bar j} m(e) \cdot P_{j_1} k \cdot P_{j_2} \mathcal T^{(b)}_{j_2}  k \Big)  \Big\|_{L^\infty L^2}      \nonumber \\
&  \stackrel{\hphantom{\eqref{Codazzi}}}{\lesssim}
 \sum_j \sum_{\bar j} \sum_{\substack{j_1,j_2 \ge \{\max j, \bar j \}-2,\\|j_1-j_2|\le 2}}\Bigg( 
2^{(s-3)j} \Big\| P_j \Big(  P_{\bar j} \partial e \cdot P_{j_1} k \cdot P_{j_2} \mathcal T^{(b)}_{j_2}  k \Big)  \Big\|_{L^\infty L^2}      \nonumber \\
& \hphantom{\stackrel{\hphantom{\eqref{Codazzi}}}{\lesssim}
 \sum_j \sum_{\bar j} \sum_{\substack{j_1,j_2 \ge \{\max j, \bar j \}-2,\\|j_1-j_2|\le 2}}\Bigg( }
 +2^{(s-3)j} \Big\| P_j \Big(  P_{\bar j} m(e) \cdot P_{j_1} \partial k \cdot P_{j_2} \mathcal T^{(b)}_{j_2}  k \Big)  \Big\|_{L^\infty L^2}     \nonumber \\
 & \hphantom{\stackrel{\hphantom{\eqref{Codazzi}}}{\lesssim}
 \sum_j \sum_{\bar j} \sum_{\substack{j_1,j_2 \ge \{\max j, \bar j \}-2,\\|j_1-j_2|\le 2}}\Bigg( }
 +2^{(s-3)j} \Big\| P_j \Big(  P_{\bar j} m(e) \cdot P_{j_1}  k \cdot P_{j_2} \mathcal T^{(b)}_{j_2}  \partial k \Big) \Big\|_{L^\infty L^2}  \Bigg)   \nonumber \\
&  \stackrel{\hphantom{\eqref{Codazzi}}}{\lesssim}
 \sum_j \sum_{\bar j} \sum_{\substack{j_1,j_2 \ge \{\max j, \bar j \}-2,\\|j_1-j_2|\le 2}}\Bigg( 
2^{(s-\f52)j} \Big\| P_j \Big(  P_{\bar j} \partial e \cdot P_{j_1} k \cdot P_{j_2} \mathcal T^{(b)}_{j_2}  k \Big)  \Big\|_{L^\infty L^{\f32}}      \nonumber \\
& \hphantom{\stackrel{\hphantom{\eqref{Codazzi}}}{\lesssim}
 \sum_j \sum_{\bar j} \sum_{\substack{j_1,j_2 \ge \{\max j, \bar j \}-2,\\|j_1-j_2|\le 2}}\Bigg( }
 +2^{(s-\f32)j} \Big\| P_j \Big(  P_{\bar j} m(e) \cdot P_{j_1} \partial k \cdot P_{j_2} \mathcal T^{(b)}_{j_2}  k \Big)  \Big\|_{L^\infty L^1}     \nonumber \\
 & \hphantom{\stackrel{\hphantom{\eqref{Codazzi}}}{\lesssim}
 \sum_j \sum_{\bar j} \sum_{\substack{j_1,j_2 \ge \{\max j, \bar j \}-2,\\|j_1-j_2|\le 2}}\Bigg( }
 +2^{(s-\f32)j} \Big\| P_j \Big(  P_{\bar j} m(e) \cdot P_{j_1}  k \cdot P_{j_2} \mathcal T^{(b)}_{j_2}  \partial k \Big) \Big\|_{L^\infty L^1}  \Bigg)   \nonumber \\
&  \stackrel{\hphantom{\eqref{Codazzi}}}{\lesssim}
 \sum_j \sum_{\bar j} \sum_{\substack{j_1,j_2 \ge \{\max j, \bar j \}-2,\\|j_1-j_2|\le 2}}\Big( 
2^{(s-\f52)(j-j_2)+(j_1-j_2)  -\delta_0(\bar j + 2j_1-j_2)}  \nonumber \\
& \hphantom{\stackrel{\hphantom{\eqref{Codazzi}}}{\lesssim}
 \sum_j \sum_{\bar j} \sum_{\substack{j_1,j_2 \ge \{\max j, \bar j \}-2,\\|j_1-j_2|\le 2}}\Bigg( 222}
 \times 2^{\delta_0 \bar j} \| P_{\bar j} \partial e\|_{L^\infty L^3} \cdot 2^{(-1+2\delta_0)j_1} \| P_{j_1} k\|_{L^\infty L^\infty} \cdot 2^{(s-\f32-\delta_0)j_2}\| P_{j_2} \mathcal T^{(b)}_{j_2}  k \|_{L^\infty L^3}  \Big)    \nonumber \\
 & \hphantom{\stackrel{\hphantom{\eqref{Codazzi}}}{\lesssim}}
 +\sum_j \sum_{\bar j} \sum_{\substack{j_1,j_2 \ge \{\max j, \bar j \}-2,\\|j_1-j_2|\le 2}}\Big( 
 +2^{(s-\f32)(j-j_1) +\f32(j_1-j_2)- \delta_0 (\bar j +j_2)}    \nonumber \\
& \hphantom{\stackrel{\hphantom{\eqref{Codazzi}}}{\lesssim}
 \sum_j \sum_{\bar j} \sum_{\substack{j_1,j_2 \ge \{\max j, \bar j \}-2,\\|j_1-j_2|\le 2}}\Bigg( 222}  
  \times 2^{\delta_0 \bar j}\| P_{\bar j} m(e)\|_{L^\infty L^\infty} \cdot 2^{(s-3)j_1}\| P_{j_1} \partial k\|_{L^\infty L^2} \cdot 2^{(\f32+\delta_0)j_2} \| P_{j_2} \mathcal T^{(b)}_{j_2}  k \|_{L^\infty L^2} \Big)    \nonumber \\
 & \hphantom{\stackrel{\hphantom{\eqref{Codazzi}}}{\lesssim}}
 +\sum_j \sum_{\bar j} \sum_{\substack{j_1,j_2 \ge \{\max j, \bar j \}-2,\\|j_1-j_2|\le 2}}\Big( 
 +2^{(s-\f32)(j-j_1) +\f12(j_1-j_2)-\delta_0(\bar j + j_2)} \nonumber \\  
& \hphantom{\stackrel{\hphantom{\eqref{Codazzi}}}{\lesssim}
 \sum_j \sum_{\bar j} \sum_{\substack{j_1,j_2 \ge \{\max j, \bar j \}-2,\\|j_1-j_2|\le 2}}\Bigg( 222}
 \times 2^{\delta_0 \bar j} \| P_{\bar j} m(e)\|_{L^\infty L^\infty} \cdot 2^{(s-2)j_1}\| P_{j_1}  k\|_{L^\infty L^2} \cdot 2^{(\f12+\delta_0) j_2 }\| P_{j_2} \mathcal T^{(b)}_{j_2}  \partial k \|_{L^\infty L^2}  \Big)   \nonumber \\
& \stackrel{\eqref{Mikhlin property}}{\lesssim}
\| \partial e\|_{L^\infty W^{\delta_0,3}} \|k\|_{L^\infty W^{-1+2\delta_0,\infty}} \| k\|_{L^\infty W^{s-\f52-\delta_0,3}} 
+ \|m(e)\|_{L^\infty W^{\delta_0,\infty}} \|\partial k\|_{L^\infty H^{s-3}} \| k \|_{L^\infty H^{\f12+\delta_0}}  \nonumber \\
&\hphantom{\sum\sum}
+ \|m(e)\|_{L^\infty W^{\delta_0,\infty}} \|k\|_{L^\infty H^{s-2}} \|\partial k \|_{L^\infty H^{-\f12+\delta_0}}    \nonumber \\
&  \stackrel{\hphantom{\eqref{Codazzi}}}{\lesssim}
 \| \partial e \|_{L^\infty H^{s-2}} \|k\|^2_{L^\infty H^{s-2}} + \|m(e)\|_{L^\infty H^{s-1}} \|k\|_{L^\infty H^{s-2}} \|\partial k\|_{L^\infty H^{s-3}} \nonumber \\
& \stackrel{\hphantom{\eqref{Codazzi}}}{\lesssim} C_0^2 \mathcal D^2,  \nonumber 
\end{align}
where, in the last line above, we made use of the bootstrap bound \eqref{Bootstrap bound}.

Similarly, we can estimate:
 \begin{align}\label{Seed F natural L2L2}
\| \partial & \mathcal F^{\natural}\|_{L^2 H^{\f16+s_1}} 
 \stackrel{\hphantom{\eqref{Codazzi}}}{\lesssim} 
\sum_j  2^{(\f16+s_1)j}\|\partial P_j \mathcal F^{\natural} \|_{L^2 L^2}   \\
&  \stackrel{\hphantom{\eqref{Codazzi}}}{\lesssim}
 \sum_j \sum_{\bar j} \sum_{\substack{j_1,j_2 \ge \{\max j, \bar j \}-2,\\|j_1-j_2|\le 2}}
2^{(\f16+s_1)j} \Big\| \partial P_j \Big(  P_{\bar j} m(e) \cdot P_{j_1} k \cdot P_{j_2} \mathcal T^{(b)}_{j_2}  k \Big)  \Big\|_{L^2 L^2}      \nonumber \\
&  \stackrel{\hphantom{\eqref{Codazzi}}}{\lesssim}
 \sum_j \sum_{\bar j} \sum_{\substack{j_1,j_2 \ge \{\max j, \bar j \}-2,\\|j_1-j_2|\le 2}}\Bigg( 
2^{(\f16+s_1)j} \Big\| P_j \Big(  P_{\bar j} \partial e \cdot P_{j_1} k \cdot P_{j_2} \mathcal T^{(b)}_{j_2}  k \Big)  \Big\|_{L^2 L^2}      \nonumber \\
& \hphantom{\stackrel{\hphantom{\eqref{Codazzi}}}{\lesssim}
 \sum_j \sum_{\bar j} \sum_{\substack{j_1,j_2 \ge \{\max j, \bar j \}-2,\\|j_1-j_2|\le 2}}\Bigg( }
 +2^{(\f16+s_1)j} \Big\| P_j \Big(  P_{\bar j} m(e) \cdot P_{j_1} \partial k \cdot P_{j_2} \mathcal T^{(b)}_{j_2}  k \Big)  \Big\|_{L^2 L^2}     \nonumber \\
 & \hphantom{\stackrel{\hphantom{\eqref{Codazzi}}}{\lesssim}
 \sum_j \sum_{\bar j} \sum_{\substack{j_1,j_2 \ge \{\max j, \bar j \}-2,\\|j_1-j_2|\le 2}}\Bigg( }
 +2^{(\f16+s_1)j} \Big\| P_j \Big(  P_{\bar j} m(e) \cdot P_{j_1}  k \cdot P_{j_2} \mathcal T^{(b)}_{j_2}  \partial k \Big) \Big\|_{L^2 L^2}  \Bigg)   \nonumber \\
&  \stackrel{\hphantom{\eqref{Codazzi}}}{\lesssim}
 \sum_j \sum_{\bar j} \sum_{\substack{j_1,j_2 \ge \{\max j, \bar j \}-2,\\|j_1-j_2|\le 2}}\Big( 
2^{(\f16+s_1)(j-j_1)+(j_1-j_2)  -\delta_0(\bar j + j_2)}  \nonumber \\
& \hphantom{\stackrel{\hphantom{\eqref{Codazzi}}}{\lesssim}
 \sum_j \sum_{\bar j} \sum_{\substack{j_1,j_2 \ge \{\max j, \bar j \}-2,\\|j_1-j_2|\le 2}}\Bigg( 222}
 \times 2^{\delta_0 \bar j} \| P_{\bar j} \partial e\|_{L^2 L^6} \cdot 2^{(s_1+\f16)j_1} \| P_{j_1} k\|_{L^\infty L^3} \cdot 2^{\delta_0 j_2}\| P_{j_2} \mathcal T^{(b)}_{j_2}  k \|_{L^\infty L^\infty}  \Big)    \nonumber \\
 & \hphantom{\stackrel{\hphantom{\eqref{Codazzi}}}{\lesssim}}
 +\sum_j \sum_{\bar j} \sum_{\substack{j_1,j_2 \ge \{\max j, \bar j \}-2,\\|j_1-j_2|\le 2}}\Big( 
 +2^{(\f16+s_1)(j-j_1) +\f12(j_1-j_2)- \delta_0 (\bar j +j_1)}    \nonumber \\
& \hphantom{\stackrel{\hphantom{\eqref{Codazzi}}}{\lesssim}
 \sum_j \sum_{\bar j} \sum_{\substack{j_1,j_2 \ge \{\max j, \bar j \}-2,\\|j_1-j_2|\le 2}}\Bigg( 222}  
  \times 2^{\delta_0 \bar j}\| P_{\bar j} m(e)\|_{L^\infty L^\infty} \cdot 2^{(-\f12+\f16+s_1+\delta_0)j_1}\| P_{j_1} \partial k\|_{L^\infty L^2} \cdot 2^{\f12 j_2} \| P_{j_2} \mathcal T^{(b)}_{j_2}  k \|_{L^2 L^\infty} \Big)    \nonumber \\
 & \hphantom{\stackrel{\hphantom{\eqref{Codazzi}}}{\lesssim}}
 +\sum_j \sum_{\bar j} \sum_{\substack{j_1,j_2 \ge \{\max j, \bar j \}-2,\\|j_1-j_2|\le 2}}\Big( 
 +2^{(s-\f32)(j-j_1) +\f12(j_1-j_2)-\delta_0(\bar j + j_2)} \nonumber \\  
& \hphantom{\stackrel{\hphantom{\eqref{Codazzi}}}{\lesssim}
 \sum_j \sum_{\bar j} \sum_{\substack{j_1,j_2 \ge \{\max j, \bar j \}-2,\\|j_1-j_2|\le 2}}\Bigg( 222}
 \times 2^{\delta_0 \bar j} \| P_{\bar j} m(e)\|_{L^\infty L^\infty} \cdot 2^{-\f12 j_1}\| P_{j_1}  k\|_{L^2 L^\infty} \cdot 2^{(\f12+\f16+s_1+\delta_0) j_2 }\| P_{j_2} \mathcal T^{(b)}_{j_2}  \partial k \|_{L^\infty L^2}  \Big)   \nonumber \\
& \stackrel{\eqref{Mikhlin property}}{\lesssim}
\| \partial e\|_{L^2 W^{\delta_0,6}} \|k\|_{L^\infty W^{\f16+s_1,3}} \| k\|_{L^\infty W^{-1+\delta_0,\infty}} 
+ \|m(e)\|_{L^\infty W^{\delta_0,\infty}} \|\partial k\|_{L^\infty H^{s-3}} \| k \|_{L^2 W^{-\f12+\delta_0,\infty}}  \nonumber \\
&\hphantom{\sum\sum}
+ \|m(e)\|_{L^\infty W^{\delta_0,\infty}} \|k\|_{L^2 W^{-\f12,\infty}} \|\partial k \|_{L^\infty H^{s-3}}    \nonumber \\
&  \stackrel{\hphantom{\eqref{Codazzi}}}{\lesssim}
 \| \partial e \|_{L^2 H^{1+\f16}} \|k\|^2_{L^\infty H^{s-2}} + \|m(e)\|_{L^\infty H^{s-1}} \|k\|_{L^2 W^{-\f12+s_0,\infty}} \|\partial k\|_{L^\infty H^{s-3}} \nonumber \\
& \stackrel{\hphantom{\eqref{Codazzi}}}{\lesssim} C_0^2 \mathcal D^2,  \nonumber 
\end{align}
where, again, in the last line above, we made use of the bootstrap bound \eqref{Bootstrap bound}. Similarly:
 \begin{align}\label{Seed F natural L2-L2+}
\| \partial & \mathcal F^{\natural}\|_{L^{\f74} W^{s_1+\f14\delta_0,\f73}} 
 \stackrel{\hphantom{\eqref{Codazzi}}}{\lesssim} 
\sum_j  2^{(s_1+\f14\delta_0) j}\|\partial P_j \mathcal F^{\natural} \|_{L^{\f74} L^{\f73}}   \\
&  \stackrel{\hphantom{\eqref{Codazzi}}}{\lesssim}
 \sum_j \sum_{\bar j} \sum_{\substack{j_1,j_2 \ge \{\max j, \bar j \}-2,\\|j_1-j_2|\le 2}}
2^{(s_1+\f14\delta_0) j} \Big\| \partial P_j \Big(  P_{\bar j} m(e) \cdot P_{j_1} k \cdot P_{j_2} \mathcal T^{(b)}_{j_2}  k \Big)  \Big\|_{L^{\f74} L^{\f73}}      \nonumber \\
&  \stackrel{\hphantom{\eqref{Codazzi}}}{\lesssim}
 \sum_j \sum_{\bar j} \sum_{\substack{j_1,j_2 \ge \{\max j, \bar j \}-2,\\|j_1-j_2|\le 2}}\Bigg( 
2^{(s_1 +\f14\delta_0)j} \Big\| P_j \Big(  P_{\bar j} \partial e \cdot P_{j_1} k \cdot P_{j_2} \mathcal T^{(b)}_{j_2}  k \Big)  \Big\|_{L^{\f74} L^{\f73}}      \nonumber \\
& \hphantom{\stackrel{\hphantom{\eqref{Codazzi}}}{\lesssim}
 \sum_j \sum_{\bar j} \sum_{\substack{j_1,j_2 \ge \{\max j, \bar j \}-2,\\|j_1-j_2|\le 2}}\Bigg( }
 +2^{(s_1+\f14\delta_0) j} \Big\| P_j \Big(  P_{\bar j} m(e) \cdot P_{j_1} \partial k \cdot P_{j_2} \mathcal T^{(b)}_{j_2}  k \Big)  \Big\|_{L^{\f74} L^{\f73}}     \nonumber \\
 & \hphantom{\stackrel{\hphantom{\eqref{Codazzi}}}{\lesssim}
 \sum_j \sum_{\bar j} \sum_{\substack{j_1,j_2 \ge \{\max j, \bar j \}-2,\\|j_1-j_2|\le 2}}\Bigg( }
 +2^{(s_1+\f14\delta_0) j} \Big\| P_j \Big(  P_{\bar j} m(e) \cdot P_{j_1}  k \cdot P_{j_2} \mathcal T^{(b)}_{j_2}  \partial k \Big) \Big\|_{L^{\f74} L^{\f73}}  \Bigg)   \nonumber \\
&  \stackrel{\hphantom{\eqref{Codazzi}}}{\lesssim}
 \sum_j \sum_{\bar j} \sum_{\substack{j_1,j_2 \ge \{\max j, \bar j \}-2,\\|j_1-j_2|\le 2}}\Big( 
2^{(s_1+\f14\delta_0) (j-j_1)+(\bar j-j_2)  -\delta_0(\bar j + j_2)}  \nonumber \\
& \hphantom{\stackrel{\hphantom{\eqref{Codazzi}}}{\lesssim}
 \sum_j \sum_{\bar j} \sum_{\substack{j_1,j_2 \ge \{\max j, \bar j \}-2,\\|j_1-j_2|\le 2}}\Bigg( 222}
 \times 2^{(-1+\delta_0) \bar j} \| P_{\bar j} \partial e\|_{L^\infty L^\infty} \cdot 2^{(s_1+\f14\delta_0) j_1} \| P_{j_1} k\|_{L^{\f72} L^{\f{14}3}} \cdot 2^{(1+\delta_0) j_2}\| P_{j_2} \mathcal T^{(b)}_{j_2}  k \|_{L^{\f72} L^{\f{14}3}}  \Big)    \nonumber \\
 & \hphantom{\stackrel{\hphantom{\eqref{Codazzi}}}{\lesssim}}
 +\sum_j \sum_{\bar j} \sum_{\substack{j_1,j_2 \ge \{\max j, \bar j \}-2,\\|j_1-j_2|\le 2}}\Big( 
 +2^{(s_1+\f14\delta_0)(j-j_1) +(j_1-j_2)- \delta_0 (\bar j +j_2)}    \nonumber \\
& \hphantom{\stackrel{\hphantom{\eqref{Codazzi}}}{\lesssim}
 \sum_j \sum_{\bar j} \sum_{\substack{j_1,j_2 \ge \{\max j, \bar j \}-2,\\|j_1-j_2|\le 2}}\Bigg( 222}  
  \times 2^{\delta_0 \bar j}\| P_{\bar j} m(e)\|_{L^\infty L^\infty} \cdot 2^{(-1+s_1+\f14\delta_0)j_1}\| P_{j_1} \partial k\|_{L^{\f72} L^{\f{14}3}} \cdot 2^{(1+\delta_0) j_2} \| P_{j_2} \mathcal T^{(b)}_{j_2}  k \|_{L^{\f72} L^{\f{14}3}} \Big)    \nonumber \\
 & \hphantom{\stackrel{\hphantom{\eqref{Codazzi}}}{\lesssim}}
 +\sum_j \sum_{\bar j} \sum_{\substack{j_1,j_2 \ge \{\max j, \bar j \}-2,\\|j_1-j_2|\le 2}}\Big( 
 +2^{(s_1+\f14\delta_0)(j-j_1) -\delta_0(\bar j + j_2)} \nonumber \\  
& \hphantom{\stackrel{\hphantom{\eqref{Codazzi}}}{\lesssim}
 \sum_j \sum_{\bar j} \sum_{\substack{j_1,j_2 \ge \{\max j, \bar j \}-2,\\|j_1-j_2|\le 2}}\Bigg( 222}
 \times 2^{\delta_0 \bar j} \| P_{\bar j} m(e)\|_{L^\infty L^\infty} \cdot 2^{(s_1+\f14\delta_0) j_1}\| P_{j_1}  k\|_{L^{\f72} L^{\f{14}3}} \cdot 2^{\delta_0 j_2 }\| P_{j_2} \mathcal T^{(b)}_{j_2}  \partial k \|_{L^{\f72} L^{\f{14}3}}  \Big)   \nonumber \\
& \stackrel{\eqref{Mikhlin property}}{\lesssim}
\| \partial e\|_{L^\infty W^{-1+\delta_0,\infty}} \|k\|^2_{L^{\f72} W^{s_1+\f12\delta_0,\f{14}3}} 
+ \|m(e)\|_{L^\infty W^{\delta_0,\infty}} \|\partial k\|_{L^{\f72} W^{-1+s_1+\f12\delta_0,\f{14}3}} \| k \|_{L^{\f72} W^{s_1,\f{14}3}}  \nonumber \\
&\hphantom{\sum\sum}
+ \|m(e)\|_{L^\infty W^{\delta_0,\infty}} \|k\|_{L^{\f72} W^{s_1+\f12\delta_0,\f{14}3}} \|\partial k \|_{L^{\f72} W^{-1+s_1,\f{14}3}}    \nonumber \\
& \stackrel{\hphantom{\eqref{Codazzi}}}{\lesssim} C_0^2 \mathcal D^2.  \nonumber 
\end{align}
We also have: 
 \begin{align}\label{Seed F natural L1Linfty}
\|  \mathcal F^{\natural} & \|_{L^1 W^{s_0, \infty}} 
 \stackrel{\hphantom{\eqref{Codazzi}}}{\lesssim} 
\sum_j  2^{s_0 j}\| P_j \mathcal F^{\natural} \|_{L^1 L^\infty}   \\
&  \stackrel{\hphantom{\eqref{Codazzi}}}{\lesssim}
 \sum_j \sum_{\bar j} \sum_{\substack{j_1,j_2 \ge \{\max j, \bar j \}-2,\\|j_1-j_2|\le 2}}
2^{s_0 j} \Big\| \partial P_j \Big(  P_{\bar j} m(e) \cdot P_{j_1} k \cdot P_{j_2} \mathcal T^{(b)}_{j_2}  k \Big)  \Big\|_{L^1 L^\infty}      \nonumber \\
&  \stackrel{\hphantom{\eqref{Codazzi}}}{\lesssim}
 \sum_j \sum_{\bar j} \sum_{\substack{j_1,j_2 \ge \{\max j, \bar j \}-2,\\|j_1-j_2|\le 2}}
2^{s_0 (j-j_1) + \f12(j_1-j_2)-\delta_0 (\bar j+2j_2-j_1)}   \nonumber \\
& \hphantom{  \stackrel{\hphantom{\eqref{Codazzi}}}{\lesssim}
 \sum_j \sum_{\bar j} \sum_{\substack{j_1,j_2 \ge \{\max j, \bar j \}-2,\\|j_1-j_2|\le 2}} \sum}
\times 2^{\delta_0 \bar j} \|  P_{\bar j} m(e) \|_{L^\infty L^\infty} \cdot 2^{(-\f12+s_0-\delta_0) j_1}\| P_{j_1} k\|_{L^2 L^\infty} \cdot 2^{(\f12+2\delta_0) j_2} \| P_{j_2} \mathcal T^{(b)}_{j_2}  k \|_{L^2 L^\infty}      \nonumber \\
&  \stackrel{\eqref{Mikhlin property}}{\lesssim}
 \|m(e)\|_{L^\infty W^{\delta_0,\infty}} \|k\|_{L^2 W^{-\f12+s_0-\delta_0,\infty}} \|k\|_{L^2 W^{-\f12+2\delta_0,\infty}}  \nonumber \\
&  \stackrel{\hphantom{\eqref{Mikhlin property}}}{\lesssim}
 \|m(e)\|_{L^\infty H^{s-1}} \|k\|^2_{L^2 W^{-\f12+s_0,\infty}}   \nonumber \\
  &  \stackrel{\hphantom{\eqref{Mikhlin property}}}{\lesssim}
  \mathcal C_0^2 D^2.   \nonumber
\end{align}
Combining \eqref{Seed F natural LinftyL2}--\eqref{Seed F natural L1Linfty}, we obtain \eqref{Bounds seed F natural easy}.

Establishing the bound \eqref{Bound seed F natural L1Linfty} requires exploiting the cancellations emerging in the expression for $\partial_0 \mathcal F^{\natural}_{cd}-R^{\natural}_{c0d0}$. In particular, using Lemma \ref{lem:Cancellations}, we obtain the following schematic expression for $P_j \Big( \partial_0 \mathcal F^{\natural}_{cd}  -  R^{\natural}_{c0d0} \Big)$:
\begin{align*}
P_j \Big( \partial_0 \mathcal F^{\natural}_{cd}  -  R^{\natural}_{c0d0}  \Big)  = 
 &  
     \sum_{\substack{j_2 > j_1-2,\\|j_2-j_3|\le 2}} P_j \Big(D \big( P_{j_1} (m(e)) \cdot  P_{j_2} k \cdot P_{j_3} (\JapD^{-1} k)\big) \Big)\\
&  +   \sum_{\substack{j_2 > j_1-2,\\|j_2-j_3|\le 2}} P_j \big( P_{j_1} (\partial e)  \cdot P_{j_2} k \cdot P_{j_3} (\JapD^{-1} k) \big)  \nonumber  \\
& +   \sum_{\substack{j_2 > j_1-2,\\|j_2-j_3|\le 2}} P_j \big( P_{j_1} (m(e)) \cdot (P_{j_2} ( g\cdot\partial g \cdot k + \omega \cdot k) )\cdot P_{j_3} (\JapD^{-1} k) \big).   \nonumber
\end{align*}
Therefore, we can estimate for any $c,d \in \{1,2,3\}$:
\begin{align}
\| \partial_0 \mathcal F^{\natural}_{cd} & - R^{\natural}_{c0d0}  \|_{L^1 W^{-1+s_0, \infty}}
\lesssim 
\sum_j 2^{(-1+s_0) j} \Big\| P_j \Big( \partial_0 \mathcal F^{\natural}_{cd}  -  R^{\natural}_{c0d0}  \Big) \Big\|_{L^1 L^\infty } \\
& \lesssim
\sum_j  \sum_{\substack{j_2 > j_1-2,\\|j_2-j_3|\le 2}}
\Bigg( 2^{(-1+s_0) j} \Big\|   P_j \Big(|D| \big( P_{j_1} (m(e)) \cdot  P_{j_2} k \cdot P_{j_3} (\JapD^{-1} k)\big) \Big)   \Big\|_{L^1 L^\infty }   \nonumber \\
& \hphantom{\lesssim
\sum_j  \sum_{\substack{j_2 > j_1-2,\\|j_2-j_3|\le 2}} 
\Bigg( }
+2^{(-1+s_0) j} \Big\|   P_j \big( P_{j_1} (\partial e)  \cdot P_{j_2} k \cdot P_{j_3} (\JapD^{-1} k) \big)    \Big\|_{L^1 L^\infty }  \nonumber \\
& \hphantom{\lesssim
\sum_j  \sum_{\substack{j_2 > j_1-2,\\|j_2-j_3|\le 2}}
\Bigg( }
+2^{(-1+s_0) j} \Big\|    P_j \big( P_{j_1} (m(e)) \cdot (P_{j_2} ( g\cdot\partial g \cdot k + \omega \cdot k) )\cdot P_{j_3} (\JapD^{-1} k) \big)    \Big\|_{L^1 L^\infty }  \Bigg) \nonumber \\
& \lesssim
\sum_j  \sum_{\substack{j_2 > \max\{j_1-2,j-8\},\\|j_2-j_3|\le 2}}
\Bigg( 2^{s_0 j} \Big\|   P_j \big(P_{j_1} (m(e)) \cdot P_{j_2} k \cdot P_{j_3} (\JapD^{-1} k)\big)   \Big\|_{L^1 L^\infty }   \nonumber \\
& \hphantom{\lesssim
\sum_j  \sum_{\substack{j_2 > j_1-2,\\|j_2-j_3|\le 2}} 
\Bigg( }
+2^{s_0 j} \Big\|   P_j \big( P_{j_1} (\partial e)  \cdot P_{j_2} k \cdot P_{j_3} (\JapD^{-1} k) \big)    \Big\|_{L^1 L^3 }  \nonumber \\
& \hphantom{\lesssim
\sum_j  \sum_{\substack{j_2 > j_1-2,\\|j_2-j_3|\le 2}}
\Bigg( }
+2^{(1+s_0) j} \Big\|    P_j \big( P_{j_1} (m(e)) \cdot (P_{j_2} ( g\cdot\partial g \cdot k + \omega \cdot k) )\cdot P_{j_3} (\JapD^{-1} k) \big)    \Big\|_{L^1 L^{\f32} }  \Bigg) \nonumber \\
& \lesssim
\sum_j  \sum_{\substack{j_2 > \max\{j_1-2,j-8\},\\|j_2-j_3|\le 2}}
\Bigg( 2^{s_0 (j-j_2) +\f12(j_3-j_2) -\delta_0(j_1+j_3)}    \nonumber \\
& \hphantom{\lesssim
\sum_j  \sum_{\substack{j_2 > \max\{j_1-2,j-8\},\\|j_2-j_3|\le 2}}
\Bigg( 2}
\times 2^{\delta_0 j_1} \| P_{j_1} (m(e))\|_{L^\infty L^\infty} \cdot 2^{(-\f12+s_0)j_2} \| P_{j_2} k\|_{L^2 L^\infty} \cdot 2^{(-\f12+\delta_0) j_3} \| P_{j_3} k\|_{L^2 L^\infty}   \nonumber \\
& \hphantom{\lesssim
\sum_j  \sum_{\substack{j_2 > j_1-2,\\|j_2-j_3|\le 2}} 
\Bigg( }
+2^{s_0 (j-j_2) +\f12 (j_2-j_3)-\delta_0 j_1-\delta_0 j_3}    \nonumber \\
& \hphantom{\lesssim
\sum_j  \sum_{\substack{j_2 > \max\{j_1-2,j-8\},\\|j_2-j_3|\le 2}}
\Bigg( 2}
\times 2^{\delta_0j_1}\| P_{j_1} (\partial e)\|_{L^\infty L^3}  \cdot 2^{(-\f12+s_0)j_2}\| P_{j_2} k\|_{L^2 L^\infty} \cdot 2^{(-\f12+\delta_0)j_3}\| P_{j_3}  k\|_{L^2 L^\infty}   \nonumber \\
& \hphantom{\lesssim
\sum_j  \sum_{\substack{j_2 > j_1-2,\\|j_2-j_3|\le 2}}
\Bigg( }
+2^{(1+s_0)( j-j_3)-\delta_0 (j_1+j_3)}     \nonumber \\
& \hphantom{\lesssim
\sum_j  \sum_{\substack{j_2 > \max\{j_1-2,j-8\},\\|j_2-j_3|\le 2}}
\Bigg( 2}
\times  2^{\delta_0 j_1} \| P_{j_1} (m(e))\|_{L^\infty L^\infty} \cdot \| P_{j_2} ( g\cdot\partial g \cdot k + \omega \cdot k) \|_{L^{\f43}L^{\f{12}5}}\cdot 2^{(s_0+\delta_0) j_3}\| P_{j_3}  k \|_{L^4 L^4}  \Bigg) \nonumber \\
&\lesssim
\| m(e)\|_{L^\infty W^{\delta_0, \infty}} \|k\|_{L^2 W^{-\f12+s_0,\infty}} \|k\|_{L^2 W^{-\f12+\delta_0,\infty}}  \nonumber \\
& \hphantom{\lesssim \sum} +  \| \partial e\|_{L^\infty W^{\delta_0,3}} \|k\|_{L^2 W^{-\f12+s_0,\infty}} \|k\|_{L^2 W^{-\f12+\delta_0,\infty}}  \nonumber \\
& \hphantom{\lesssim \sum} + \| m(e)\|_{L^\infty W^{\delta_0, \infty}} \big( \| g\cdot\partial g \|_{L^2 L^6} + \| \omega \|_{L^2 L^6} \big) \| k \|_{L^4 L^4} \| k\|_{L^4 W^{s_0+\delta_0,4}}   \nonumber \\
&\lesssim \big(\|e\|_{L^\infty H^{s-1}} + \|\partial e\|_{L^\infty H^{s-2}} \big) \big( \|  g\cdot\partial g \|_{L^2 H^1} + \| \omega \|_{L^2 H^1}\big) \big( \| k\|_{L^2 W^{-\f12+s_0,\infty}}^2 +\| k\|_{L^4 W^{\f1{12},4}}^2 \big)   \nonumber \\
& \lesssim C_0^2 \mathcal D^2 \nonumber
\end{align}
(in the above, while passing from the second to the third inequality, we used the fact that, when $j_1, j_2, j_3 <j-4$, the projection $P_j (P_{j_1} f_1 \cdot P_{j_2} f_2 \cdot P_{j_3} f_3)$ vanishes identically). The bound \eqref{Bound seed F natural L1Linfty} now follows after summing over $c,d \in \{1,2,3\}$. 

The higher order bound \eqref{Bound F natural higher order} follows by arguing in a similar way: Starting from the schematic expression (obtained by differentiating \eqref{Calculation F natural difference})
\begin{align}\label{Expression higher order F natural}
P_j \Big( \partial_0 \big( \partial_0 \mathcal F^{\natural}_{cd}  -  R^{\natural}_{c0d0}  \big)\Big)  = 
 &  
     \sum_{\substack{j_2 > j_1-2,\\|j_2-j_3|\le 2}} \partial_0 P_j \Big(D \big( P_{j_1} (m(e)) \cdot  P_{j_2} k \cdot P_{j_3} (\JapD^{-1} k)\big) \Big)\\
&  +   \sum_{\substack{j_2 > j_1-2,\\|j_2-j_3|\le 2}} \partial_0 P_j \big( P_{j_1} (\partial e)  \cdot P_{j_2} k \cdot P_{j_3} (\JapD^{-1} k) \big)  \nonumber  \\
& +   \sum_{\substack{j_2 > j_1-2,\\|j_2-j_3|\le 2}} \partial_0 P_j \big( P_{j_1} (m(e)) \cdot (P_{j_2} ( g\cdot\partial g \cdot k + \omega \cdot k) )\cdot P_{j_3} (\JapD^{-1} k) \big), \nonumber
\end{align}
we can estimate each term in the right hand side of \eqref{Expression higher order F natural} as follows:
\begin{itemize}
\item For the first term in the right hand side of \eqref{Expression higher order F natural}, we can bound (using the functional inequalities of Lemma \ref{lem:Functional inequalities}):
\begin{align*}
\sum_j \big\|   \sum_{\substack{j_2 > j_1-2,\\|j_2-j_3|\le 2}} & \partial_0 P_j \Big(|D| \big( P_{j_1} (m(e)) \cdot  P_{j_2} k \cdot P_{j_3} (\JapD^{-1} k)\big) \Big)   \big\|_{L^\infty H^{s-4+\delta_0}} \\
& \lesssim  \sum_j \big\|   \sum_{\substack{j_2 > j_1-2,\\|j_2-j_3|\le 2}} \partial_0 P_j \Big( \big( P_{j_1} (m(e)) \cdot  P_{j_2} k \cdot P_{j_3} (\JapD^{-1} k)\big) \Big)   \big\|_{L^\infty H^{s-3+\delta_0}} \\
& \lesssim \|\partial e \|_{L^\infty H^{s-2}} \|k\|_{L^\infty H^{s-2}}\|\JapD^{-1} k\|_{L^\infty H^{s-1}} \\
& \hphantom{\lesssim}
+ \|m(e)\|_{L^\infty H^{s-1}} \|\partial k \|_{L^\infty H^{s-3}} \|\JapD^{-1} k\|_{L^\infty H^{s-1}} \\
& \hphantom{\lesssim}
+ \|m(e)\|_{L^\infty H^{s-1}} \|k \|_{L^\infty H^{s-2}} \|\JapD^{-1}\partial k\|_{L^\infty H^{s-2}}  \\
& \lesssim C_0^2 \mathcal D^2.
\end{align*}
Similarly:
\begin{align*}
\sum_j \big\|   \sum_{\substack{j_2 > j_1-2,\\|j_2-j_3|\le 2}} & \partial_0 P_j \Big(|D| \big( P_{j_1} (m(e)) \cdot  P_{j_2} k \cdot P_{j_3} (\JapD^{-1} k)\big) \Big)   \big\|_{L^2 H^{-\f56+s_1}} \\
& \lesssim  \sum_j \big\|   \sum_{\substack{j_2 > j_1-2,\\|j_2-j_3|\le 2}} \partial_0 P_j \Big( \big( P_{j_1} (m(e)) \cdot  P_{j_2} k \cdot P_{j_3} (\JapD^{-1} k)\big) \Big)   \big\|_{L^2 H^{\f16+s_1}} \\
& \lesssim \|\partial e \|_{L^\infty W^{-1,\infty}} \|k\|^2_{L^4 W^{s_0+\f1{12},4}} 
+ \|m(e)\|_{L^\infty L^\infty} \|\partial k \|_{L^4 W^{s_0-1+\f1{12},4}} \| k \|_{L^4 W^{s_0+\f1{12},4}} \\
& \lesssim \|\partial e\|_{L^\infty H^{s-2}}  \|k\|^2_{L^4 W^{s_0+\f1{12},4}} 
+ \|m(e)\|_{L^\infty H^{s-1}} \|\partial k \|_{L^4 W^{s_0-1+\f1{12},4}} \| k \|_{L^4 W^{s_0+\f1{12},4}} \\ 
& \lesssim C_0^2 \mathcal D^2
\end{align*}
and
\begin{align*}
\sum_j \big\|   \sum_{\substack{j_2 > j_1-2,\\|j_2-j_3|\le 2}} & \partial_0 P_j \Big(|D| \big( P_{j_1} (m(e)) \cdot  P_{j_2} k \cdot P_{j_3} (\JapD^{-1} k)\big) \Big)   \big\|_{L^{\f74} W^{-1+s_1+\f14\delta_0,\f73}} \\
& \lesssim  \sum_j \big\|   \sum_{\substack{j_2 > j_1-2,\\|j_2-j_3|\le 2}} \partial_0 P_j \Big( \big( P_{j_1} (m(e)) \cdot  P_{j_2} k \cdot P_{j_3} (\JapD^{-1} k)\big) \Big)   \big\|_{L^{\f74} W^{s_1+\f14\delta_0,\f73}} \\
& \lesssim \|\partial e \|_{L^\infty W^{-1,\infty}} \|k\|^2_{L^{\f72} W^{s_0,\f{14}3}} 
+ \|m(e)\|_{L^\infty L^\infty} \|\partial k \|_{L^{\f72} W^{-1+s_0,\f{14}3}} \| k \|_{L^{\f72} W^{s_0,\f{14}3}} \\
& \lesssim \|\partial e\|_{L^\infty H^{s-2}}  \|k\|^2_{L^{\f72} W^{s_0,\f{14}3}} 
+ \|m(e)\|_{L^\infty H^{s-1}} \|\partial k \|_{L^{\f72} W^{-1+s_0,\f{14}3}} \| k \|_{L^{\f72} W^{s_0,\f{14}3}} \\ 
& \lesssim C_0^2 \mathcal D^2.
\end{align*}

\smallskip
\item For the second term in the right hand side of \eqref{Expression higher order F natural}, we can bound:
\begin{align*}
\sum_j \big\|   \sum_{\substack{j_2 > j_1-2,\\|j_2-j_3|\le 2}} & \partial_0 P_j \Big( P_{j_1} ( \partial e) \cdot  P_{j_2} k \cdot P_{j_3} (\JapD^{-1} k) \Big)   \big\|_{L^\infty H^{s-4+\delta_0}} \\
& \lesssim \|\partial^2 e \|_{L^\infty H^{s-3}} \|k\|_{L^\infty H^{s-2}}\|\JapD^{-1} k\|_{L^\infty H^{s-1}} \\
& \hphantom{\lesssim}
+ \|\partial e\|_{L^\infty H^{s-2}} \|\partial k \|_{L^\infty H^{s-3}} \|\JapD^{-1} k\|_{L^\infty H^{s-1}} \\
& \hphantom{\lesssim}
+ \|\partial e\|_{L^\infty H^{s-2}} \|k \|_{L^\infty H^{s-2}} \|\JapD^{-1}\partial k\|_{L^\infty H^{s-2}}  \\
& \lesssim C_0^2 \mathcal D^2,
\end{align*}

\smallskip
\begin{align*}
\sum_j \big\|   \sum_{\substack{j_2 > j_1-2,\\|j_2-j_3|\le 2}} & \partial_0 P_j \Big( P_{j_1} (\partial e) \cdot  P_{j_2} k \cdot P_{j_3} (\JapD^{-1} k) \Big)   \big\|_{L^2 H^{-\f56+s_1}}
\\ 
& \lesssim
\sum_j \big\|   \sum_{\substack{j_2 > j_1-2,\\|j_2-j_3|\le 2}} \partial_0 P_j \Big( P_{j_1} (\partial e) \cdot  P_{j_2} k \cdot P_{j_3} (\JapD^{-1} k) \Big)   \big\|_{L^2 W^{\f16+s_1,\f65}} \\ 
& \lesssim \|\partial^2 e \|_{L^\infty W^{-1,3}} \|k\|^2_{L^4 W^{s_0+\f1{12},4}} 
+ \|\partial e\|_{L^\infty L^3} \|\partial k \|_{L^4 W^{s_0-1+\f1{12},4}} \| k \|_{L^4 W^{s_0+\f1{12},4}} \\
& \lesssim \|\partial^2 e\|_{L^\infty H^{s-3}}  \|k\|^2_{L^4 W^{s_0+\f1{12},4}} 
+ \|\partial e\|_{L^\infty H^{s-2}} \|\partial k \|_{L^4 W^{s_0-1+\f1{12},4}} \| k \|_{L^4 W^{s_0+\f1{12},4}} \\ 
& \lesssim C_0^2 \mathcal D^2
\end{align*}
and
\begin{align*}
\sum_j \big\|   \sum_{\substack{j_2 > j_1-2,\\|j_2-j_3|\le 2}} & \partial_0 P_j \Big( P_{j_1} (\partial e) \cdot  P_{j_2} k \cdot P_{j_3} (\JapD^{-1} k) \Big)   \big\|_{L^{\f74} W^{-1+s_1+\f14\delta_0,\f73}}
\\ 
& \lesssim
\sum_j \big\|   \sum_{\substack{j_2 > j_1-2,\\|j_2-j_3|\le 2}} \partial_0 P_j \Big( P_{j_1} (\partial e) \cdot  P_{j_2} k \cdot P_{j_3} (\JapD^{-1} k) \Big)   \big\|_{L^{\f74} W^{s_1+\f14\delta_0,\f{21}{16}}} \\ 
& \lesssim \|\partial^2 e \|_{L^\infty W^{-1,3}} \|k\|^2_{L^{\f72} W^{s_0,\f{14}3}} 
+ \|\partial e\|_{L^\infty L^3} \|\partial k \|_{L^{\f72} W^{-1+s_0,\f{14}3}} \| k \|_{L^{\f72} W^{s_0,\f{14}3}} \\
& \lesssim \|\partial^2 e\|_{L^\infty H^{s-3}}  \|k\|^2|_{L^{\f72} W^{s_0,\f{14}3}}
+ \|\partial e\|_{L^\infty H^{s-2}} \|\partial k \|_{L^{\f72} W^{-1+s_0,\f{14}3}} \| k \|_{L^{\f72} W^{s_0,\f{14}3}} \\ 
& \lesssim C_0^2 \mathcal D^2.
\end{align*}

\smallskip
\item For the third term in the right hand side of \eqref{Expression higher order F natural}, we can bound:
\begin{align*}
\sum_j \big\|   \sum_{\substack{j_2 > j_1-2,\\|j_2-j_3|\le 2}} & \partial_0 P_j \Big( P_{j_1} ( m(e)) \cdot  P_{j_2} \big( ( g\cdot\partial g + \omega)\cdot k\big) \cdot P_{j_3} (^{-1} k) \Big)   \big\|_{L^\infty H^{s-4+\delta_0}} \\
& \lesssim \|\partial e \|_{L^\infty H^{s-2}} \| ( g\cdot\partial g + \omega) k\|_{L^\infty H^{s-3}}\|\JapD^{-1} k\|_{L^\infty H^{s-1}} \\
& \hphantom{\lesssim}
+ \| m(e)\|_{L^\infty H^{s-1}} \|( g\cdot\partial^2 g + \partial \omega) k \|_{L^\infty H^{s-4}} \|\JapD^{-1} k\|_{L^\infty H^{s-1}} \\
& \hphantom{\lesssim}
+ \| m(e)\|_{L^\infty H^{s-1}} \|( g\cdot\partial g + \omega)\partial k \|_{L^\infty H^{s-4}} \|\JapD^{-1} k\|_{L^\infty H^{s-1}} \\
& \hphantom{\lesssim}
+ \|m(e)\|_{L^\infty H^{s-1}} \|( g\cdot\partial g + \omega)k \|_{L^\infty H^{s-3}} \|\JapD^{-1}\partial k\|_{L^\infty H^{s-2}}  \\
& \lesssim \|\partial e \|_{L^\infty H^{s-2}} \big(\| g\cdot\partial g\|_{L^\infty H^{s-2}} + \|\omega\|_{L^\infty H^{s-2}}\big) \|k\|_{L^\infty H^{s-2}} \|\JapD^{-1} k\|_{L^\infty H^{s-1}} \\
& \hphantom{\lesssim}
+ \| m(e)\|_{L^\infty H^{s-1}} \big(\| g\cdot\partial^2 g\|_{L^\infty H^{s-3}} + \|\partial \omega\|_{L^\infty H^{s-3}}\big) \|k\|_{L^\infty H^{s-2}} \|\JapD^{-1} k\|_{L^\infty H^{s-1}} \\
& \hphantom{\lesssim}
+ \| m(e)\|_{L^\infty H^{s-1}} \big(\| g\cdot\partial g\|_{L^\infty H^{s-2}} + \|\omega\|_{L^\infty H^{s-2}}\big) \|\partial k\|_{L^\infty H^{s-3}} \|\JapD^{-1} k\|_{L^\infty H^{s-1}} \\
& \hphantom{\lesssim}
+ \|m(e)\|_{L^\infty H^{s-1}} \big(\| g\cdot\partial g\|_{L^\infty H^{s-2}} + \|\omega\|_{L^\infty H^{s-2}}\big) \|k\|_{L^\infty H^{s-2}} \|\JapD^{-1}\partial k\|_{L^\infty H^{s-2}}  \\
& \lesssim C_0^2 \mathcal D^2,
\end{align*}

\smallskip
\begin{align*}
\sum_j \big\|    \sum_{\substack{j_2 > j_1-2,\\|j_2-j_3|\le 2}} & \partial_0 P_j \Big( P_{j_1} ( m(e)) \cdot  P_{j_2} \big( ( g\cdot\partial g + \omega)\cdot k\big) \cdot P_{j_3} (\JapD^{-1} k) \Big)     \big\|_{L^2 H^{-\f56+s_1}}
\\ 
& \lesssim
\sum_j \big\|    \sum_{\substack{j_2 > j_1-2,\\|j_2-j_3|\le 2}} \partial_0 P_j \Big( P_{j_1} ( m(e)) \cdot  P_{j_2} \big( ( g\cdot\partial g + \omega)\cdot k\big) \cdot P_{j_3} (\JapD^{-1} k) \Big)     \big\|_{L^2 W^{\f16+s_1,\f65}} \\ 
& \lesssim \|\partial e \|_{L^\infty L^3} \|( g\cdot\partial g + \omega)\cdot k\|_{L^\infty H^{s-3}} \| \JapD^{-1}k\|_{L^2 W^{\f12+\delta_0,\infty}} \\
&\hphantom{\lesssim}
+ \|m(e)\|_{L^\infty L^\infty} \|( g\cdot\partial^2 g + \partial \omega)\cdot k\|_{L^2 W^{\f16,\f65}} \| \JapD^{-1}k\|_{L^\infty W^{s_1+\delta_0,\infty}} \\
&\hphantom{\lesssim}
+ \|m(e)\|_{L^\infty L^\infty} \|( g\cdot\partial g + \omega)\cdot \partial k\|_{L^2 H^{-\f56+s_1}} \| \JapD^{-1}k\|_{L^\infty W^{1+\delta_0,3}} \\
& \hphantom{\lesssim}
+ \|m(e)\|_{L^\infty L^\infty}  \|( g\cdot\partial g + \omega)\cdot k\|_{L^2 H^{\f16+s_1}} \|\partial \JapD^{-1}k\|_{L^\infty W^{\delta_0,3}} \\
&\lesssim 
\|\partial e\|_{L^\infty H^{s-2}} \big(\| g\cdot\partial g\|_{L^\infty H^{s-2}} + \|\omega\|_{L^\infty H^{s-2}}\big) \|k\|_{L^\infty H^{s-2}} \|k\|_{L^2 W^{-\f12+\delta_0,\infty}} \\
& \hphantom{\lesssim}
+\|m(e)\|_{L^\infty H^{s-1}}  \big(\| g\cdot\partial^2 g\|_{L^2 H^{\f16}} + \|\partial \omega\|_{L^2 H^{\f16}} \big) \|k\|_{L^\infty W^{\f16,3}} \|k\|_{L^\infty H^{s-2}}\\
& \hphantom{\lesssim}
+\|m(e)\|_{L^\infty H^{s-1}}  \big(\| g\cdot\partial g\|_{L^2 H^{s_1+\f76}} + \| \omega\|_{L^2 H^{s_1+\f76}} \big) \|\partial k\|_{L^\infty H^{s-3}} \|k\|_{L^\infty H^{s-2}}\\
& \hphantom{\lesssim}
+\|m(e)\|_{L^\infty H^{s-1}}  \big(\| g\cdot\partial g\|_{L^2 H^{s_1+\f76}} + \| \omega\|_{L^2 H^{s_1+\f76}} \big) \| k\|_{L^\infty H^{s-2}}^2\\
& \lesssim C_0^2 \mathcal D^2.
\end{align*}
Using Lemma \ref{lem:Functional inequalities} to estimate $\| (g\cdot \partial g+\omega) \cdot k\|_{L^{\f72} W^{-1,\f{14}3}}\lesssim\|g\cdot \partial g +\omega\|_{L^\infty H^{s-2}} \|k\|_{L^{\f72} W^{\f{14}3}}$,  we also have:
\begin{align*}
\sum_j \big\|    \sum_{\substack{j_2 > j_1-2,\\|j_2-j_3|\le 2}} & \partial_0 P_j \Big( P_{j_1} ( m(e)) \cdot  P_{j_2} \big( ( g\cdot\partial g + \omega)\cdot k\big) \cdot P_{j_3} (\JapD^{-1} k) \Big)     \big\|_{L^{\f74} W^{-1+s_1+\f14\delta_0,\f73}}
\\ 
& \lesssim
\sum_j \big\|    \sum_{\substack{j_2 > j_1-2,\\|j_2-j_3|\le 2}} \partial_0 P_j \Big( P_{j_1} ( m(e)) \cdot  P_{j_2} \big( ( g\cdot\partial g + \omega)\cdot k\big) \cdot P_{j_3} (\JapD^{-1} k) \Big)     \big\|_{L^{\f74} W^{s_1+\f14\delta_0,\f{21}{16}}} \\ 
& \lesssim \|\partial e \|_{L^\infty L^3} \|( g\cdot\partial g + \omega)\cdot k\|_{L^{\f72} W^{-1,\f{14}3}} \| \JapD^{-1}k\|_{L^{\f72} W^{1+s_0,\f{14}3}} \\
&\hphantom{\lesssim}
+ \|m(e)\|_{L^\infty L^\infty} \|( g\cdot\partial^2 g + \partial \omega)\cdot k\|_{L^{\f74} W^{\delta_0,\f{21}{16}}} \| \JapD^{-1}k\|_{L^\infty W^{s_1,\infty}} \\
&\hphantom{\lesssim}
+ \|m(e)\|_{L^\infty L^\infty} \|( g\cdot\partial g + \omega)\cdot \partial k\|_{L^{\f74} W^{-1+\delta_0,\f73}} \| \JapD^{-1}k\|_{L^\infty W^{1+s_1,3}} \\
& \hphantom{\lesssim}
+ \|m(e)\|_{L^\infty L^\infty}  \|( g\cdot\partial g + \omega)\cdot k\|_{L^{\f74} W^{s_1,\f73}} \|\partial \JapD^{-1}k\|_{L^\infty W^{s_0,3}} \\
&\lesssim 
\|\partial e\|_{L^\infty H^{s-2}} \big(\| g\cdot\partial g\|_{L^\infty H^{s-2}} + \|\omega\|_{L^\infty H^{s-2}}\big) \|k\|^2_{L^{\f72} W^{s_0,\f{14}3}} \\
& \hphantom{\lesssim}
+\|m(e)\|_{L^\infty H^{s-1}}  \big(\| g\cdot\partial^2 g\|_{L^{\f74} W^{\delta_0,\f73}} + \|\partial \omega\|_{L^{\f74} W^{\delta_0,\f73}} \big) \|k\|_{L^\infty W^{\delta_0,3}} \| k\|_{L^{\infty} W^{-1+s_1,\infty}}\\
& \hphantom{\lesssim}
+\|m(e)\|_{L^\infty H^{s-1}}  \big(\| g\cdot\partial g\|_{L^{\f74} W^{1+\delta_0,\f73}} + \| \omega\|_{L^{\f74} W^{1+\delta_0, \f73}} \big) \|\partial k\|_{L^\infty W^{-1+\delta_0,3}}  \| k\|_{L^{\infty} W^{s_1,3}}\\
& \hphantom{\lesssim}
+\|m(e)\|_{L^\infty H^{s-1}}  \big(\| g\cdot\partial g\|_{L^{\f74} W^{\delta_0, \f1{\f37-\f13}}} + \| \omega\|_{L^{\f74} W^{\delta_0, \f1{\f37-\f13}}} \big) \|k\|_{L^\infty W^{\delta_0, 3}}\| \partial k\|_{L^\infty W^{-1+s_0, 3}}\\
& \lesssim C_0^2 \mathcal D^2.
\end{align*}
In the above, we made use of the functional inequalities:
\[
\|f_1 \cdot f_2 \|_{L^{\f74} W^{\delta_0, \f73}} \lesssim \|f_1\|_{L^{\f74} W^{\delta_0, \f1{\f37-\f13}}} \|f_2\|_{L^\infty W^{\delta_0, 3}}.
\]
and
\[
\|f_1 \cdot f_2 \|_{L^{\f74} W^{-1+\delta_0, \f73}} \lesssim \|f_1\|_{L^{\f74} W^{1+\delta_0, \f73}} \|f_2\|_{L^\infty W^{-1+\delta_0, 3}}
\]
(readily derived via a high-low decomposition of the product).

\end{itemize}
Combining the above bounds, we infer that:
\[
\Big\|  \partial_0 \big( \partial_0 \mathcal F^{\natural}_{cd}  -  R^{\natural}_{c0d0}  \big)\Big\|_{L^\infty H^{s-4+\delta_0}} +
\Big\| \partial_0 \big( \partial_0 \mathcal F^{\natural}_{cd}  -  R^{\natural}_{c0d0}  \big)\Big\|_{L^2 H^{-\f56+s_1}} +
\Big\| \partial_0 \big( \partial_0 \mathcal F^{\natural}_{cd}  -  R^{\natural}_{c0d0}  \big)\Big\|_{L^{\f74} W^{-1+s_1+\f14\delta_0, \f73}}\lesssim  C_0^2 \mathcal D^2.
\]
The desired estimate  \eqref{Bound F natural higher order}  for $ \partial_0^2  \tilde{\mathcal F}^{\natural}_{cd} -\partial_0 R^{\natural}_{c0d0}$ follows from the above bound, combined with the estimate
\[
\|\partial_0^2 \mathbb E\big(\mathcal F^\natural |_{x^0=0} \big)\|_{L^\infty H^{s-4}} +
\|\partial_0^2 \mathbb E\big(\mathcal F^\natural |_{x^0=0} \big)\|_{L^2 H^{-\f56+s_1}}++
\|\partial_0^2 \mathbb E\big(\mathcal F^\natural |_{x^0=0} \big)\|_{L^{\f74} W^{-1+s_1,\f73}}
\lesssim \| \mathcal F^\natural |_{x^0=0}\|_{H^{s-2}}\lesssim \mathcal D 
\]
following from Lemma \ref{lem:Estimates time operators} (recall also that $\partial_0^2 \mathbb E\big(\mathcal F^\natural |_{x^0=0} \big) = \partial_0^2 \tilde{\mathcal F}^\natural - \partial_0^2 \mathcal F^\natural$).

\end{proof}

\begin{lemma}\label{lem:Bounds F perp}
The following bounds hold for $\tilde{\mathcal F}_\perp$:
\begin{equation}\label{Bounds F perp easy}
\| \partial \tilde{\mathcal F}_\perp  \|_{L^\infty H^{s-4}} + \| \partial \tilde{\mathcal F}_\perp  \|_{L^2 H^{-\f56+s_1}} + \| \partial \tilde{\mathcal F}_\perp  \|_{L^{\f74} W^{-1+s_1+\f14\delta_0,\f73}}+ \| \tilde{\mathcal F}_\perp  \|_{L^1 W^{-1+s_0, \infty}}  \lesssim \mathcal D
\end{equation}
 and 
 \begin{equation} \label{Bound F perp L1Linfty}
 \sum_{\bA,\bB=4}^n \| \partial_{0} (\tilde{\mathcal F}_{\perp})^{\bA\bB} +   \partial_c  (\bar g^{cd} R^{\perp\natural})_{d0}^{\bA\bB}   \|_{L^1 W^{-2+s_0, \infty}} \lesssim \mathcal D.
 \end{equation}
\end{lemma}

\begin{proof}

The proof of Lemma \ref{lem:Bounds F perp} will follow very closely that of Lemma \ref{lem:Bounds F natural} above.  Recall that $\tilde{\mathcal F}_\perp$ is given by
\begin{equation}\label{F perp again}
\tilde{\mathcal F} _{\perp}^{\bA\bB}(x^0,\bar x) \doteq \mathcal F_{\perp}^{\bA\bB}(x^0,\bar x) - \mathbb E \big[ \mathcal F_{\perp}^{\bA\bB}|_{x^0=0}\big](x^0,\bar x),
\end{equation}
where $\mathcal F_\perp$ was defined by \eqref{F perp seed}.

Arguing similarly as for the proof of \eqref{Bound Extension F natural}, we can readily estimate:
\begin{align}\label{Bound Extension F perp}
\big\| \partial \big( & \mathbb E \big[ \mathcal F_\perp|_{x^0=0}\big] \big)  \big\|_{L^\infty H^{s-4}} 
+  \big\| \partial \big( \mathbb E \big[ \mathcal F_\perp|_{x^0=0}\big] \big)  \big\|_{L^2 H^{-\f56+s_1}}\\
&  + \big\| \partial \big( \mathbb E \big[ \mathcal F_\perp|_{x^0=0}\big] \big)  \big\|_{L^{\f74} W^{-1+s_1+\f14\delta_0,\f73}}
+ \big\| \partial \big( \mathbb E \big[ \mathcal F_\perp|_{x^0=0}\big] \big)  \big\|_{L^1 W^{-2+s_0, \infty}} 
 \lesssim \mathcal D^2. \nonumber
\end{align}
Thus, in order to establish \eqref{Bounds F perp easy}--\eqref{Bound F perp L1Linfty} for $\tilde{\mathcal F}_\perp$, it suffices to establish the same bounds for $\mathcal F_\perp$, i.e.
\begin{equation}\label{Bounds F perp easy n}
\| \partial \mathcal F_\perp  \|_{L^\infty H^{s-4}} + \| \partial \mathcal F_\perp  \|_{L^2 H^{-\f56+s_1}} + \| \partial \mathcal F_\perp  \|_{L^{\f74} W^{-1+s_1+\f14\delta_0,\f73}}+ \| \mathcal F_\perp  \|_{L^1 W^{-1+s_0, \infty}}  \lesssim \mathcal D
\end{equation}
 and 
 \begin{equation} \label{Bound F perp L1Linfty n}
 \sum_{\bA,\bB=4}^n \| \partial_{0} (\mathcal F_{\perp})^{\bA\bB} +   \partial_c  (\bar g^{cd} R^{\perp\natural})_{d0}^{\bA\bB}   \|_{L^1 W^{-2+s_0, \infty}} \lesssim \mathcal D.
 \end{equation}

Using the schematic expression
\[
 \mathcal F_{\perp} = \bar\partial  \Big( g \cdot  \sum_{\bar j} \sum_{\substack{j_1,j_2 \ge \bar j - 2\\|j_1-j_2|\le 2}} \big( P_{\bar j} g \cdot P_{j_1} k \cdot \JapD^{-1}P_{j_2}k \big) \Big)
\]
and using exactly the same arguments as in the proof of the bound 
 \eqref{Bounds seed F natural easy} for $\mathcal F^{\natural}$ (where we only used the fact that $\mathcal F^\natural$ can be expressed schematically as $m(e) \cdot k \cdot \JapD^{-1} k$), we obtain the bound \eqref{Bounds F perp easy n}.

In order to establish the $L^1 L^\infty$-type bound \eqref{Bound F perp L1Linfty n}, we will have to exploit the structure of $\partial_0 \mathcal F_\perp$. Using the expression \eqref{Computation F perp}, we have:
\begin{align*}
P_j \Big( \partial_{0} & (\mathcal F_{\perp})^{\bA\bB} +   \partial_c  (\bar g^{cd} R^{\perp\natural})_{d0}^{\bA\bB} \Big)\\
& = \sum_{\substack{j_1,j_2,j_3,j_4:\\|j_3-j_4|\le 2}}  P_j \Bigg[ 
D^2 \Big(P_{j_1} g \cdot P_{j_2} g \cdot P_{j_3} k \cdot P_{j_4}(\JapD^{-1} k) \Big)   \nonumber \\
& \hphantom{ \sum_{\substack{j_1,j_2,j_3,j_4:\\|j_3-j_4|\le 2}}  P_j \Bigg[  D}
 +D \Big(P_{j_1} (\partial g + \partial e) \cdot P_{j_2} g \cdot P_{j_3} k \cdot P_{j_4}(\JapD^{-1} k) \Big) \Bigg].   \nonumber
\end{align*}
Therefore, we can bound:
\begin{align*}
\big\| \partial_{0} & (\mathcal F_{\perp})^{\bA\bB} +   \partial_c  (\bar g^{cd} R^{\perp\natural})_{d0}^{\bA\bB} \big\|_{L^1 W^{-2+s_0,\infty}} \\
& \lesssim 
\sum_j \sum_{\substack{j_1,j_2,j_3,j_4:\\|j_3-j_4|\le 2}}\Bigg[ 
2^{(-2+s_0)j} \Big\| P_j \Big( D^2 \Big(P_{j_1} g \cdot P_{j_2} g \cdot P_{j_3} k \cdot P_{j_4}(\JapD^{-1} k) \Big)\Big) \Big\|_{L^1L^\infty}   \nonumber \\
& \hphantom{ \sum_j \sum_{\substack{j_1,j_2,j_3,j_4:\\|j_3-j_4|\le 2}} 
 P_j \Bigg[  D}
 + 2^{(-2+s_0)j} \Big\| P_j \Big( D \big(P_{j_1} (\partial g + \partial e) \cdot P_{j_2} g \cdot P_{j_3} k \cdot P_{j_4}(\JapD^{-1} k) \big)\Big) \Big\|_{L^1 L^\infty} \Bigg]  \\
 & \lesssim 
\sum_j \sum_{\substack{j_1,j_2,j_3,j_4:\\|j_3-j_4|\le 2}}\Bigg[ 
2^{s_0 j} \Big\| P_j \Big(P_{j_1} g \cdot P_{j_2} g \cdot P_{j_3} k \cdot P_{j_4}(\JapD^{-1} k) \Big) \Big\|_{L^1L^\infty}   \nonumber \\
& \hphantom{ \sum_j \sum_{\substack{j_1,j_2,j_3,j_4:\\|j_3-j_4|\le 2}} 
 P_j \Bigg[  D}
 + 2^{s_0 j} \Big\| P_j \Big(  P_{j_1} (\partial g + \partial e) \cdot P_{j_2} g \cdot P_{j_3} k \cdot P_{j_4}(\JapD^{-1} k) \Big) \Big\|_{L^1 L^3} \Bigg]  \\
 & \lesssim \|g\|_{L^\infty H^{s-1}} \| k\|^2_{L^2 W^{-\f12+s_0,\infty}} + \|g\|_{L^\infty H^{s-1}} \big(\|\partial g\|_{L^2 W^{s_0,6}} + \|\partial e\|_{L^2 W^{s_0,6}} \big) \| k\|_{L^2 W^{-\f12+s_0,\infty}}\| k\|_{L^\infty W^{-\f12+s_0,6}} \\
 &   \lesssim \|g\|_{L^\infty H^{s-1}} \| k\|^2_{L^2 W^{-\f12+s_0,\infty}} + \|g\|_{L^\infty H^{s-1}} \big(\|\partial g\|_{L^2 H^{\f76}} + \|\partial e\|_{L^2 H^{\f76}} \big) \| k\|_{L^2 W^{-\f12+s_0,\infty}}\| k\|_{L^\infty H^{s-2}} \\
 & \lesssim C_0^2 \mathcal D^2,
\end{align*}
thus obtaining \eqref{Bound F perp L1Linfty} (provided $\epsilon$ is small enough in terms of $C_0$).

\end{proof}

\subsection{Closing the estimates for $g$ and $\omega$}\label{subsec:Closing bootstrap metric}
In this section, we will show that the balanced gauge condition ensures that the components of the metric $g$ and the connection 1-form $\omega$ satisfy nearly optimal regularity bounds compared to those satisfied by the curvature tensors $R$ and $R^\perp$. In particular, we will show that the estimates obtained in the previous section for $R$ and $R^\perp$ allow us to infer the bounds for $g$ and $\omega$ which are needed to close the bootstrap estimate \eqref{Bootstrap conclusion}. To this end, we will make use of the fact that, in the balanced gauge, the components of $g$ and $\omega$ satisfy the parabolic-elliptic system derived in Lemma \ref{lem:Parabolic elliptic system}.

Throughout this section, we will assume that $Y:[0,T)\times \mathbb T^3 \rightarrow \mathcal N$ has been fixed as in the statement of Proposition \ref{prop:Bootstrap}.

\medskip
\subsubsection{\textbf{Estimates for the lapse $N$.}}
The following result is the first step towards improving the bootstrap assumption \eqref{Bootstrap bound} for the coefficients of the metric $g$:

\begin{lemma}\label{lem:Estimates N}
The lapse function $N$ satisfies the following estimate for any $p\in (1,+\infty)$:
\begin{align}\label{Bound d N}
\|\partial N \|_{L^\infty H^{s-2}} + \|\partial N\|_{L^2 H^{\f76 + s_1}} & +\|\partial N\|_{L^{\f74} B^{1+s_1+\f18\delta_0}_{\f73, \f{7}4}} + \| \partial N\|_{L^1 B^{\f12 \delta_0}_{p,1}} \\
&  + \| N-1\|_{L^\infty L^\infty} \le C_{(p)} \big( \mathcal D + \mathcal D^{\f12} \|\|\partial g \|_{L^{\f74} B^{1+s_1+\f18\delta_0}_{\f73, \f{7}4}}+  \partial g\|_{L^1 B^{\f12 \delta_0}_{p,1}}\big), \nonumber 
\end{align}
where $C_{(p)}>0$ is a constant depending only on $p,s$  and which is independent of $C_0$.
\end{lemma}

\begin{proof}
The function $N$ satisfies the following parabolic equation:
\begin{equation}\label{Equation lapse again}
\partial_0 (N-1) + |D| (N-1) =  \mathcal G_N, 
\end{equation}
where 
\[
\mathcal G_N = \mathfrak F_{N} + \mathcal E_{N},
\]
\[
\mathfrak F_N \doteq |D|  \Delta_{\bar g} ^{-1} \Bigg[ \f1N \bar{g}^{ij}\big( \partial_0  \tilde{\mathcal{F}}^{\natural}_{ij} - R_{i0j0}\big)    -   \fint_{(\bar\Sigma_{\tau},\bar g)} \Big\{ \f1N \bar{g}^{ij}\big( \partial_0  \tilde{\mathcal{F}}^{\natural}_{ij} - R_{i0j0}\big) \Big\} \Bigg] 
\]
and
\[
\mathcal E_N =\JapD^{-1} \Big( g \cdot \partial g \cdot \tilde{\mathcal{F}}^{\natural} + (g-m_0) \cdot D\partial g +g \cdot R_{*} + g \cdot \partial g \cdot \partial g  \Big),
\]
where $R_*$ denotes any component of the Riemann tensor $R$ with at most one non zero index (see Lemma \ref{lem:Parabolic elliptic system}).
Using the facts that
\begin{itemize}
\item $\fint_{(\bar\Sigma_{\tau},g_{\mathbb T^3})} N=1$ (in view of condition \eqref{Mean curvature condition}) and
\item $\| (N-1)|_{x^0=0} \|_{H^{s-1}} \lesssim \mathcal D $ (as a consequence of Proposition \ref{prop: Bounds S0}),
\end{itemize}
The bound \eqref{Bound d N} will follow immediately from Lemma \ref{lem:Parabolic estimates model} from the Appendix,  provided
\begin{equation}\label{Necessary bound lapse}
\| \mathcal G_N  \|_{L^\infty H^{s-2}} + \| \mathcal G_N  \|_{L^2 H^{\f76+s_1}} +\|\mathcal G_N \|_{L^{\f74} B^{1+s_1+\f18\delta_0}_{\f73, \f{7}4}} +  \| \mathcal G_N  \|_{L^1 B^{\f12 \delta_0}_{p,1}} \lesssim_p \mathcal D + \mathcal D^{\f12} \Big(\|\partial g \|_{L^{\f74} B^{1+s_1+\f18\delta_0}_{\f73, \f{7}4}} + \| \partial g \|_{L^1 B^{\f12 \delta_0}_{p,1}}\Big).
\end{equation}
The bound  \eqref{Necessary bound lapse} for the term $\mathfrak F_N$ follows by combining the estimates for $R$ and $R^\natural$ of Lemma \ref{lem:Riemann estimates} with those for $\partial \mathcal F^{\natural}$ from Lemma \ref{lem:Bounds F natural} (without the need for the extra term $\mathcal D^{\f12}\Big(\|\partial g \|_{L^{\f74} B^{1+s_1+\f18\delta_0}_{\f73, \f{7}4}}+ \| \partial g \|_{L^1 B^{\f12 \delta_0}_{p,1}}\Big)$ in the right hand side); the same bound for the term $\mathcal E_N$ follows from Lemma  \ref{lem:Bound EN} below.

\end{proof}

\begin{lemma}\label{lem:Bound EN}
Let $\mathcal E_N$ be a collection of terms of the form
\begin{equation}\label{Expression EN2}
\mathcal E_N =\JapD^{-1} \Big( g \cdot \partial g \cdot \tilde{\mathcal{F}}^{\natural} + (g-m_0) \cdot D\partial g + g \cdot \partial g \cdot \partial g    +   g \cdot R_{*}  \Big),
\end{equation}
where $m_0= -(dx^0)^2+\sum_{l=1}^3 (dx^l)^2$ and $R_*$ denotes any component of the Riemann tensor $R$ with at most one non zero index.
Then, for any $p\in (1,+\infty)$, we can estimate:
\begin{align}\label{Necessary bound lapse 2}
\| \mathcal E_N  \|_{L^\infty H^{s-2}} +&  \| \mathcal E_N \|_{L^2 H^{\f76+s_1}}  + \| \mathcal E_N \|_{L^{\f74} B^{1+s_1+\f18\delta_0}_{\f73, \f{7}4}}+  \| \mathcal E_N  \|_{L^1 B^{\f12 \delta_0}_{p,1}} \\
& \lesssim_p \mathcal D + \mathcal D^{\f12} \Big(\|\partial g \|_{L^{\f74} B^{1+s_1+\f18\delta_0}_{\f73, \f{7}4}}+ \| \partial g \|_{L^1 B^{\f12 \delta_0}_{p,1}}\Big).\nonumber 
\end{align}
\end{lemma}

\begin{proof}
We will treat each summand in the expression \eqref{Expression EN2} separately.

\smallskip 
\noindent \textbf{The term $g \cdot \partial g \cdot \tilde{\mathcal F}^\natural$.} 
We can readily estimate:
\begin{align*}
\| g \cdot \partial g \cdot  \tilde{\mathcal F}^\natural \|^2_{L^\infty H^{s-3}} & 
\lesssim \sup_{\tau \in [0,T)}\sum_j   2^{2(s-3)j} \| P_j \big( g \cdot \partial g \cdot  \tilde{\mathcal F}^\natural \big)|_{x^0=\tau} \|^2_{L^2} \\
& \lesssim \sup_{\tau \in [0,T)} \sum_j \Bigg(  \sum_{\max\{j_1,j_2,j_3\}>j-8} 2^{(s-3)j} \| P_j \big( P_{j_1} g \cdot P_{j_2} \partial g \cdot  P_{j_3} \tilde{\mathcal F}^\natural \big)|_{x^0=\tau} \|_{L^{2}}  \Bigg)^2 \nonumber \\
& \lesssim \sup_{\tau \in [0,T)} \sum_j  \Bigg( \sum_{\substack{j_1>j-8\\j_2 < j_1}} 2^{(s-2)j} \| P_j \big( P_{j_1} g \cdot P_{j_2} \partial g \cdot  P_{j_3} \tilde{\mathcal F}^\natural \big)|_{x^0=\tau} \|_{L^{\f65}}\Bigg)^2   \nonumber \\
& \hphantom{\lesssim} 
+ \sup_{\tau \in [0,T)} \sum_j   \Bigg(   \sum_{\max\{j_2,j_3\}>j-8} 2^{(s-\f52)j} \| P_j \big( P_{j_1} g \cdot P_{j_2} \partial g \cdot  P_{j_3} \tilde{\mathcal F}^\natural \big)|_{x^0=\tau} \|_{ L^{\f32}}\Bigg)^2   \nonumber \\[5pt]
&  \lesssim \sup_{\tau \in [0,T)} \sum_j  \Bigg( \sum_{\substack{j_1>j-8\\j_2 < j_1}} 
2^{(s-1)(j-j_1) -(j-j_2) -\delta_0 (j_2+j_3)}  \nonumber \\
& \hphantom{\lesssim \sum_j  \sum_{>j-8} 
2} 
\times 2^{(s-1) j_1}\|P_{j_1} g|_{x^0=\tau}\|_{L^2} \cdot 2^{(-1+\delta_0) j_2}\|P_{j_2} \partial g\|_{L^\infty L^\infty} \cdot  2^{\delta_0j_3}\|P_{j_3} \tilde{\mathcal F}^\natural  \|_{L^\infty L^3} \Bigg)^2   \nonumber \\[5pt]
&  + \sup_{\tau \in [0,T)} \sum_j  \Bigg( \sum_{\max\{j_2,j_3\}>j-8} 
2^{(s-\f52)(j-j_2-j_3)-\delta_0 j_1}   \nonumber \\
& \hphantom{\lesssim \sum_j  \sum_{\max\{j_2,j_3\}>j-8} 
2} 
\times 2^{\delta_0 j_1}\|P_{j_1} g\|_{L^\infty L^\infty} \cdot 2^{(s-\f52)j_2}\|P_{j_2} \partial g|_{x^0=\tau}\|_{ L^3} \cdot  2^{(s-\f52)j_3}\|P_{j_3} \tilde{\mathcal F}^\natural |_{x^0=\tau} \|_{L^3} \Bigg)^2  \nonumber \\[5pt]
& \lesssim \| g \|^2_{L^\infty H^{s-1}} \| \partial g\|^2_{L^\infty W^{-1+\delta_0, \infty}} \| \tilde{\mathcal F}^{\natural} \|^2_{L^\infty W^{\delta_0,3}}   \nonumber \\
& \hphantom{\lesssim} \quad \quad \quad\quad + \| g \|^2_{L^\infty W^{\delta_0,\infty}} \| \partial g\|^2_{L^\infty W^{s-\f52, 3}} \| \tilde{\mathcal F}^{\natural} \|^2_{L^\infty W^{s-\f52,3}} \nonumber \\
& \lesssim    \| g \|^2_{L^\infty H^{s-1}} \| \partial g\|^2_{L^\infty H^{s-2}} \| \tilde{\mathcal F}^{\natural} \|^2_{L^\infty H^{s-2}} \nonumber \\
& \lesssim C_0^4 \mathcal D^4,
\end{align*}
where, in the last line above, we made use of the bootstrap bound \eqref{Bootstrap bound} for $g$, $\partial g$ and the bound \eqref{Bounds F natural easy} for $\tilde{\mathcal F}^\natural$.

Similarly, we can bound
\begin{align*}
\| g \cdot \partial g \cdot  \tilde{\mathcal F}^\natural \|^2_{L^2 H^{\f16+s_1}} & 
\lesssim \sum_j   2^{2(\f16+s_1)j} \| P_j \big( g \cdot \partial g \cdot  \tilde{\mathcal F}^\natural \big) \|^2_{L^2 L^2} \\
& \lesssim \sum_j  \Bigg( \sum_{\max\{j_1,j_2,j_3\}>j-8} 2^{(\f16+s_1)j} \| P_j \big( P_{j_1} g \cdot P_{j_2} \partial g \cdot  P_{j_3} \tilde{\mathcal F}^\natural \big) \|_{L^2 L^{2}} \Bigg)^2  \nonumber \\
& \lesssim \sum_j   \Bigg(  \sum_{\substack{j_1>j-8\\j_2 < j_1}} 2^{(\f16+s_1)j} \| P_j \big( P_{j_1} g \cdot P_{j_2} \partial g \cdot  P_{j_3} \tilde{\mathcal F}^\natural \big) \|_{L^2 L^{2}}\Bigg)^2   \nonumber \\
& \hphantom{\lesssim} 
+ \sum_j  \Bigg(   \sum_{\max\{j_2,j_3\}>j-8} 2^{(\f16+s_1)j} \| P_j \big( P_{j_1} g \cdot P_{j_2} \partial g \cdot  P_{j_3} \tilde{\mathcal F}^\natural \big) \|_{L^2 L^2} \Bigg)^2  \nonumber \\[5pt]
&  \lesssim \sum_j \Bigg( \sum_{\substack{j_1>j-8\\j_2 < j_1}} 
2^{(\f16+s_1)(j-j_1) -(j-j_2) -\delta_0 (j_1+j_2+j_3)}  \nonumber \\
& \hphantom{\lesssim \sum_j  \sum_{>j-8} 
2} 
\times 2^{(1+\f16+s_1+\delta_0) j_1}\|P_{j_1} g\|_{L^\infty L^3} \cdot 2^{(-1+\delta_0) j_2}\|P_{j_2} \partial g\|_{L^\infty L^\infty} \cdot  2^{\delta_0j_3}\|P_{j_3} \tilde{\mathcal F}^\natural  \|_{L^2 L^6}  \Bigg)^2 \nonumber \\[5pt]
&  + \sum_j  \Bigg( \sum_{\max\{j_2,j_3\}>j-8} 
2^{(\f16+s_1)(j-j_2-j_3)-\delta_0( j_1+j_2)}   \nonumber \\
& \hphantom{\lesssim \sum_j  \sum_{\max\{j_2,j_3\}>j-8} 
2} 
\times 2^{\delta_0 j_1}\|P_{j_1} g\|_{L^\infty L^\infty} \cdot 2^{(\f16+s_1+\delta_0)j_2}\|P_{j_2} \partial g\|_{L^\infty L^3} \cdot  2^{(\f16+s_1) j_3}\|P_{j_3} \tilde{\mathcal F}^\natural  \|_{L^2 L^6}  \Bigg)^2 \nonumber \\[5pt]
& \lesssim \| g \|^2_{L^\infty W^{\f76+s_1+\delta_0,3}} \| \partial g\|^2_{L^\infty W^{-1+\delta_0, \infty}} \| \tilde{\mathcal F}^{\natural} \|^2_{L^2 W^{\delta_0,6}}   \nonumber \\
& \hphantom{\lesssim} \quad \quad \quad\quad + \| g \|^2_{L^\infty W^{\delta_0,\infty}} \| \partial g\|^2_{L^\infty W^{\f16+s_1+\delta_0, 3}} \| \tilde{\mathcal F}^{\natural} \|^2_{L^2 W^{\f16+s_1,6}} \nonumber \\
& \lesssim    \| g \|^2_{L^\infty H^{s-1}} \| \partial g\|^2_{L^\infty H^{s-2}} \| \tilde{\mathcal F}^{\natural} \|^2_{L^2 H^{\f76+s_1}} \nonumber \\
& \lesssim C_0^4 \mathcal D^4,  \nonumber
\end{align*}
where, again, in the last step above, we made use of the bootstrap bound \eqref{Bootstrap bound} for $g$, $\partial g$ and the bound \eqref{Bounds F natural easy} for $\tilde{\mathcal F}^\natural$. The bound
\[
\|g \cdot \partial g \cdot  \tilde{\mathcal F}^\natural \|_{L^{\f74} B^{s_1+\f18\delta_0}_{\f73, \f{7}4}} \lesssim C_0^2 \mathcal D^2
\]
follows in a completely analogous way (using the estimate for $\|\tilde{\mathcal F}^\natural\|_{L^{\f74} W^{1+s_1+\f14\delta_0,\f73}}$ from \eqref{Bounds F natural easy}). We can also estimate:
\begin{align*}
\| g \cdot \partial g \cdot  \tilde{\mathcal F}^\natural \|_{L^1 B^{-1+\f12 \delta_0}_{\infty,1}} & 
\lesssim \sum_j   2^{(-1+\f12\delta_0)j} \| P_j \big( g \cdot \partial g \cdot  \tilde{\mathcal F}^\natural \big) \|_{L^1 L^\infty} \\
& \lesssim \sum_j  \sum_{\max\{j_1,j_2,j_3\}>j-8} 2^{(-1+\f12\delta_0)j} \| P_j \big( P_{j_1} g \cdot P_{j_2} \partial g \cdot  P_{j_3} \tilde{\mathcal F}^\natural \big) \|_{L^1 L^\infty}  \nonumber \\
& \lesssim \sum_j    \sum_{\substack{j_1>j-8\\j_2<j_1}} 2^{\f12\delta_0 j} \| P_j \big( P_{j_1} g \cdot P_{j_2} \partial g \cdot  P_{j_3} \tilde{\mathcal F}^\natural \big) \|_{L^1 L^3}   \nonumber \\
& \hphantom{\lesssim} 
+ \sum_j    \sum_{\max\{j_2,j_3\}>j-8} 2^{\f12\delta_0j} \| P_j \big( P_{j_1} g \cdot P_{j_2} \partial g \cdot  P_{j_3} \tilde{\mathcal F}^\natural \big) \|_{L^1 L^3} \nonumber \\[5pt]
&  \lesssim \sum_j \sum_{\substack{j_1>j-8\\j_2<j_1}} 
2^{-(j_1-j_2)-\f12\delta_0 (j_1+j_2+j_3-j)}  \nonumber \\
& \hphantom{\lesssim \sum_j  \sum_{>j-8} 
2} 
\times 2^{(1+\f12 \delta_0) j_1}\|P_{j_1} g\|_{L^2 L^6} \cdot 2^{(-1+\f12\delta_0) j_2}\|P_{j_2} \partial g\|_{L^\infty L^\infty} \cdot  2^{\f12\delta_0 j_3}\|P_{j_3} \tilde{\mathcal F}^\natural  \|_{L^2 L^6}   \nonumber \\[5pt]
&  +\sum_{\max\{j_2,j_3\}>j-8}  
2^{-\f12\delta_0(j_2+j_3-j)-(j-j_3)-\f12 \delta_0( j_1+j_2+j_3)}   \nonumber \\
& \hphantom{\lesssim \sum_j  \sum_{\max\{j_2,j_3\}>j-8} 
2} 
\times 2^{\f12 \delta_0 j_1}\|P_{j_1} g\|_{L^\infty L^\infty} \cdot 2^{\delta_0 j_2}\|P_{j_2} \partial g\|_{L^2 L^6} \cdot  2^{\delta_0 j_3}\|P_{j_3} \tilde{\mathcal F}^\natural  \|_{L^2 L^6}  \nonumber \\[5pt]
& \lesssim \| g \|_{L^2 W^{1+\f12\delta_0,6}} \| \partial g\|_{L^\infty W^{-1+\delta_0, \infty}} \| \tilde{\mathcal F}^{\natural} \|^2_{L^2 W^{\f12\delta_0,6}}   \nonumber \\
& \hphantom{\lesssim} \quad \quad \quad\quad + \| g \|_{L^\infty W^{\f12\delta_0,\infty}} \| \partial g\|_{L^2 W^{\delta_0, 6}} \| \tilde{\mathcal F}^{\natural} \|_{L^2 W^{\delta_0,6}} \nonumber \\
& \lesssim    \big( \| g \|_{L^\infty H^{s-1}} +\| \partial g \|_{L^2 H^{\f76}}\big) \big( \| \partial g\|_{L^\infty H^{s-2}}+\| \partial g \|_{L^2 H^{\f76}}\big)  \| \tilde{\mathcal F}^{\natural} \|^2_{L^2 H^{\f76}} \nonumber \\
& \lesssim C_0^2 \mathcal D^2.  \nonumber 
\end{align*}
Combining the above bounds, we obtain, provided $\epsilon$ is sufficiently small (so that $C_0 \mathcal D^{\f12} \ll 1$):
\begin{equation}\label{Bound for dg F natural}
\| g \cdot \partial g \cdot  \tilde{\mathcal F}^\natural \|_{L^\infty H^{s-3}}  + \| g \cdot \partial g \cdot  \tilde{\mathcal F}^\natural \|_{L^2 H^{\f16+s_1}}  + \| g \cdot \partial g \cdot  \tilde{\mathcal F}^\natural \|_{L^1 B^{-1+\f12 \delta_0}_{\infty,1}} \le \mathcal D.
\end{equation}

\smallskip 
\noindent \textbf{The term $(g-m_0) \cdot D\partial g$.} 
We can readily estimate:
\begin{align*}
\| (g-m_0) \cdot D \partial g  \|^2_{L^\infty H^{s-3}} & 
\lesssim \sup_{\tau \in [0,T)}\sum_j   2^{2(s-3)j} \| P_j \big( (g-m_0) \cdot D \partial g \big)|_{x^0=\tau} \|^2_{L^2} \\
& \lesssim \sup_{\tau \in [0,T)} \sum_j \Bigg(  2^{2(s-3)j} \| P_{\le j}(g-m_0) \cdot P_j D\partial g |_{x^0=\tau} \|^2_{L^2}  \nonumber \\
& \hphantom{\lesssim \sup_{\tau \in [0,T)} \sum_j \Bigg(  }
+  2^{2(s-3)j} \| P_j(g-m_0) \cdot P_{\le j} D\partial g |_{x^0=\tau} \|^2_{L^2}   \nonumber \\
& \hphantom{\lesssim \sup_{\tau \in [0,T)} \sum_j \Bigg(  }
+ \Big(\sum_{j'>j} 2^{(s-3)j} \|P_j \big( P_{j'}(g-m_0) \cdot P_{j'} D\partial g\big)|_{x^0=\tau}\|_{L^2} \Big)^2 \Bigg) \nonumber \\
& \lesssim \sup_{\tau \in [0,T)} \sum_j \Bigg(   \| g-m_0\|^2_{L^\infty L^\infty} \cdot 2^{2(s-3)j} \| P_j D\partial g |_{x^0=\tau} \|^2_{L^2} \nonumber \\
& \hphantom{\lesssim \sup_{\tau \in [0,T)} \sum_j \Bigg(  }
+  2^{2(s-1)j} \| P_j (g-m_0)|_{x^0=\tau}\|_{L^2} \cdot 2^{-4j} \| P_{\le j} D\partial g \|^2_{L^\infty L^\infty}   \nonumber \\
& \hphantom{\lesssim \sup_{\tau \in [0,T)} \sum_j \Bigg(  }
+ \Big(\sum_{j'>j} 2^{(s-\f52)j} \|P_j \big( P_{j'} (g-m_0) \cdot P_{j'} D\partial g\big)|_{x^0=\tau}\|_{L^{\f32}} \Big)^2 \Bigg) \nonumber \\
& \lesssim \| g-m_0\|^2_{L^\infty L^\infty} \| D\partial g \|^2_{L^\infty H^{s-3}} +  \|g-m_0\|_{L^\infty H^{s-1}} \cdot  \| D\partial g \|^2_{L^\infty W^{-2,\infty}}  + \nonumber \\
& \hphantom{\lesssim }
+ \sum_j \Big(\sum_{j'>j} 2^{(s-\f52)(j-j')-\delta_0 j'} 2^{(s-\f32)j'}\| P_{j'} (g-m_0)\|_{L^\infty L^3} \cdot 2^{(-1+\delta_0)j'}\|P_{j'} D\partial g\|_{L^\infty L^3} \Big)^2 \nonumber \\
& \lesssim \| g-m_0 \|^2_{L^\infty H^{s-1}} \| \partial g \|^2_{L^\infty H^{s-2}} \nonumber \\
& \lesssim C^2_0 \mathcal D^4 \nonumber 
\end{align*}
(where, in the last line above, we made use of the bootstrap estimate \eqref{Bootstrap bound}) and
\begin{align*}
\| (g-m_0) \cdot D \partial g  \|^2_{L^2 H^{\f16+s_1}} & 
\lesssim \sum_j   2^{2(\f16+s_1)j} \| P_j \big( (g-m_0) \cdot D \partial g \big) \|^2_{L^2 L^2} \\
& \lesssim  \sum_j \Bigg(  2^{2(\f16+s_1)j} \| P_{\le j}(g-m_0) \cdot P_j D\partial g  \|^2_{L^2 L^2} \nonumber \\
& \hphantom{\lesssim \sum_j \Bigg(  }
+  2^{2(\f16+s_1)j} \| P_j(g-m_0) \cdot P_{\le j} D\partial g  \|^2_{L^2 L^2}  \nonumber \\
& \hphantom{\lesssim  \sum_j \Bigg(  }
+ \Big(\sum_{j'>j} 2^{(\f16+s_1)j} \|P_j \big( P_{j'}(g-m_0) \cdot P_{j'} D\partial g\big)\|_{L^2 L^2} \Big)^2 \Bigg) \nonumber \\
& \lesssim  \sum_j \Bigg(  \| P_{\le j}(g-m_0)\|^2_{L^\infty L^\infty} \cdot 2^{2(\f16+s_1)j}  \| P_j D\partial g  \|^2_{L^2 L^2} \nonumber \\
& \hphantom{\lesssim  \sum_j \Bigg(  }
+  2^{2(\f16+s_1)j} \| P_j(g-m_0)\|^2_{L^\infty L^\infty} \cdot \| P_{\le j} D\partial g  \|^2_{L^2 L^2}  \nonumber \\
& \hphantom{\lesssim  \sum_j \Bigg(  }
+ \Big(\sum_{j'>j} 2^{(\f16+s_1)(j-j')-\delta_0 j'} 2^{\delta_0 j'}\|P_{j'}(g-m_0)\|_{L^\infty L^\infty} \cdot 2^{(\f16+s_1)j'}\| P_{j'} D\partial g\|_{L^2 L^2} \Big)^2 \Bigg) \nonumber \\
& \lesssim \| g-m_0 \|^2_{L^\infty H^{s-1}} \| \partial g \|^2_{L^2 H^{\f76+s_1}}   \nonumber \\
& \lesssim C^2_0 \mathcal D^4.   \nonumber 
\end{align*}
Following the same steps, we also have:
\[
\|(g-m_0)\cdot D\partial g\|_{L^{\f74} B^{s_1+\f18\delta_0}_{\f73, \f{7}4}} \lesssim  C_0\mathcal D \|\partial g \|_{L^{\f74} B^{1+s_1+\f18\delta_0}_{\f73, \f{7}4}} + C_0^2 \mathcal D^2.
\]
Moreover, using the schematic relation
\[
(g-m_0) \cdot D\partial g = D \big( (g-m_0) \cdot \partial g \big) + \partial g \cdot \partial g,
\]
together with the bounds
\begin{align*}
\| D ( f_1 \cdot f_2 \big) \|_{L^1 B^{-1+\f12 \delta_0}_{p,1}} & \lesssim \|  f_1 \cdot f_2 \|_{L^1B^{\f12 \delta_0}_{p,1}} \\
& \lesssim  \|  f_1\|_{L^\infty L^\infty} \cdot \| f_2 \|_{L^1 B^{\f12 \delta_0}_{p,1}} + \|f_1 \|_{L^1 B^{1+\f12 \delta_0}_{p,1}} \| f_2\|_{L^\infty W^{-1,\infty}}
\end{align*}
and 
\begin{align*}
\| f_1 \cdot f_2 \|_{L^1 B^{-1+\f12 \delta_0}_{p,1}} & \lesssim \| f_1 \cdot f_2 \|_{L^1 B^0_{\max\{1,\f{3}{(1-\f12\delta_0)+\f3p}\},1}}\\
&  \lesssim \| f_1\|_{L^2 H^{1+\f16}} \| f_2\|_{L^2 H^{1+\f16}},
\end{align*}
we can readily show that, for any $p\in (1,+\infty)$:
\begin{align*}
\| (g-m_0) \cdot D \partial g&   \|_{L^1 B^{-1+\f12 \delta_0}_{p,1}} \\
&
\lesssim \| D \big( (g-m_0) \cdot \partial g \big)  \|_{L^1 B^{-1+\f12 \delta_0}_{p,1}} + \|\partial g \cdot  \partial g  \|_{L^1 B^{-1+\f12 \delta_0}_{p,1}}   \\
& \lesssim \| g-m_0\|_{L^\infty L^\infty} \| \partial g  \|_{L^1 B^{\f12 \delta_0}_{p,1}} +  \|g-m_0 \|_{L^1 B^{1+\f12 \delta_0}_{p,1}} \| \partial g \|_{L^\infty W^{-1,\infty}} + \| \partial g \|^2_{L^2 H^{1+\f16}}  \\
& \lesssim \| g-m_0\|_{L^\infty L^\infty} \| \partial g  \|_{L^1 B^{\f12 \delta_0}_{p,1}} +  \|g-m_0 \|_{L^1 B^{1+\f12 \delta_0}_{p,1}} \| \partial g \|_{L^\infty H^{s-2}} + \| \partial g \|^2_{L^2 H^{1+\f16}} \\
& \lesssim C_0 \mathcal D \| \partial g \|_{L^1 B^{\f12 \delta_0}_{p,1}}+   C_0^2 \mathcal D^2 
\end{align*}
(where, in the last line above, we used the bootstrap bound \eqref{Bootstrap bound}.

Combining the above bounds, we obtain for any $p\in (1,+\infty)$, provided $\epsilon$ is sufficiently small so that $C_0 \mathcal D^{\f12} \ll 1$:
\begin{align}\label{Bound for g Ddg}
\| (g-m_0) \cdot D \partial g  & \|_{L^\infty H^{s-3}} + \| (g-m_0) \cdot D \partial g  \|_{L^2 H^{\f16+s_1}} \\
& 
 + \| (g-m_0) \cdot D \partial g  \|_{L^1 B^{-1+\f12 \delta_0}_{p,1}} \le \mathcal D + \mathcal D^{\f12}\Big(\|\partial g \|_{L^{\f74} B^{1+s_1+\f18\delta_0}_{\f73, \f{7}4}}+  \| \partial g  \|_{L^1 B^{\f12 \delta_0}_{p,1}}\Big).    \nonumber
\end{align}

\smallskip 
\noindent \textbf{The term $g \cdot \partial g \cdot \partial g$.}
The bootstrap assumption \eqref{Bootstrap bound} implies that the term $\partial g$ satisfies similar $L^\infty L^2$, $L^2 L^2$ and $L^{\f74} L^{\f73}$--type estimates as the term $\tilde{\mathcal F}^{\natural}$, i.e.
\[
\| \partial g \|_{L^\infty H^{s-2}} + \| \partial g\|_{L^2 H^{\f76+s_1}} + \|\partial g \|_{L^{\f74} B^{1+s_1+\f18\delta_0}_{\f73, \f{7}4}} \lesssim C_0 \mathcal D
\]
(compare the above with \eqref{Bounds F natural easy}) Therefore, by arguing exactly as in the case of the term $g \cdot \partial g \cdot \tilde{\mathcal F}^\natural$, we obtain:
\[
\| g \cdot \partial g \cdot \partial g\|_{L^\infty H^{s-3}}  + \| g \cdot \partial g \cdot \partial g\|_{L^2 H^{\f16+s_1}} + \|g\cdot \partial g \cdot \partial g \|_{L^{\f74} B^{s_1+\f18\delta_0}_{\f73, \f{7}4}} \lesssim C_0^2 \mathcal D^2.
\]
In order to obtain the $L^1 L^\infty$-type bound for $g \cdot \partial g \cdot \partial g$, we argue as follows:
\begin{align*}
\| g \cdot \partial g  \cdot \partial g\|_{L^1 B^{-1+\f12 \delta_0}_{\infty,1}} & 
\lesssim     \| g \cdot \partial g  \cdot \partial g\|_{L^1 B^0_{\f{3}{(1-\f12\delta_0)},1}} \\
& \lesssim     \| g\|_{L^\infty L^\infty} \| \partial g \|^2_{L^2 B^0_{\f{6}{(1-\f12\delta_0)},1}} \\
& \lesssim \|g\|_{L^\infty L^\infty} \| \partial g \|^2_{L^2 H^{1+\f16}} \\
& \lesssim C_0^2 \mathcal D^2.
\end{align*}
Combining the above bounds, we obtain for any $p\in (1,+\infty)$, provided $C_0 \mathcal D^{\f12} \ll 1$:
\begin{align}\label{Bound for dgdg}
\| g \cdot  & \partial g  \cdot \partial g  \|_{L^\infty H^{s-3}}  + \| g \cdot \partial g  \cdot \partial g  \|_{L^2 H^{\f16+s_1}} \\
& 
+\|g\cdot \partial g \cdot \partial g \|_{L^{\f74} B^{s_1+\f18\delta_0}_{\f73, \f{7}4}} + \| g \cdot \partial g  \cdot \partial g  \|_{L^1 B^{-1+\f12 \delta_0}_{\infty,1}} \le \mathcal D.    \nonumber
\end{align}

\smallskip 
\noindent \textbf{The term $g \cdot R_*$.}  
The bounds for $R$ from Lemma \ref{lem:Riemann estimates} imply that
\begin{equation}\label{Bound R_*}
\| R_* \|_{L^\infty H^{s-3+\delta_0}} + \| R_*\|_{L^2 H^{\f16+s_1+\f14\delta_0}}+ \|R_*\|_{L^{\f74} W^{s_1+\f14\delta_0, \f73}} + \|R_*\|_{L^1 W^{-1+\delta_0,\infty}} \lesssim C_0^2 \mathcal D^2.
\end{equation}
Therefore, we can readily estimate:
\begin{align*}
\| g \cdot R_* \|_{L^\infty H^{s-3}} & \lesssim  \| g \|_{L^\infty H^{s-1}} \|R_* \|_{L^\infty H^{s-3}} \\
& \lesssim C_0^2 \mathcal D^2,
\end{align*}

\begin{align*}
\| g \cdot R_* \|_{L^2 H^{\f16+s_1}} & 
\lesssim \| g \|_{L^\infty L^\infty} \| R_*\|_{L^2 H^{\f16+s_1}} + \| g\|_{L^\infty W^{\f16+s_1,\infty}} \|R_{*}\|_{L^2 L^2} \\
& \lesssim \|g \|_{L^\infty H^{s-1}} \| R_*\|_{L^2 H^{\f16+s_1}}\\
& \lesssim C_0^2 \mathcal D^2,
\end{align*}
\[
\| g \cdot R_*\|_{L^{\f74} W^{s_1+\f14\delta_0, \f73}} \lesssim C_0^2 \mathcal D^2
\]
and
\begin{align*}
\| g \cdot R_* \|_{L^1 B^{-1+\f12 \delta_0}_{\infty,1}} & 
\lesssim \sum_j 2^{(-1+\f12\delta_0)j} \| P_j (g \cdot R_*) \|_{L^1 L^\infty}\\
& \lesssim \sum_j \Bigg( 2^{(-1+\f12\delta_0)j} \| P_j g \cdot P_{\le j} R_* \|_{L^1 L^\infty} 
+ 2^{(-1+\f12\delta_0)j} \| P_{\le j} g \cdot P_{ j} R_* \|_{L^1 L^\infty}  \\
& \hphantom{\sum_j \Bigg( }
+\sum_{j'>j} 2^{\f12\delta_0j} \| P_j ( P_{j'} g \cdot P_{j'} R_*) \|_{L^1 L^3} \Bigg) \\
& \lesssim \sum_j \Bigg( \| P_j g\|_{L^\infty L^\infty} \cdot 2^{(-1 +\f12\delta_0) j}\| P_{\le j} R_* \|_{L^1 L^\infty} \\
& \hphantom{\sum_j \Bigg(}
+  \| P_{\le j} g\|_{L^\infty L^\infty} \cdot 2^{(-1+\f12\delta_0)j}\| P_j R_* \|_{L^1 L^\infty}  \\
& \hphantom{\sum_j \Bigg( }
+\sum_{j'>j} 2^{\f12\delta_0 (j-j')-\f12\delta_0 j'} 2^{ j'}\| P_{j'} g\|_{L^2 L^6} \cdot 2^{(-1+\delta_0)j'}\| P_{j'} R_* \|_{L^2 L^6} \Bigg) \\
& \lesssim \| g \|_{L^\infty W^{\delta_0,\infty}} \|R_* \|_{L^1 W^{-1+\delta_0,\infty}} + \| g\|_{L^2 W^{1+\delta_0,6}} \|R_*\|_{L^2 W^{-1+\delta_0,6}} \\
& \lesssim \| g \|_{L^\infty H^{s-1}} \|R\|_{L^1 W^{-1+\delta_0,\infty}}  + \| g\|_{L^2 H^{2+\f16 +\delta_0}} \|R_*\|_{L^2 H^{\delta_0}} \\
& \lesssim C_0^2 \mathcal D^2.
\end{align*}
Combining the above bounds, we obtain (provided $C_0 \mathcal D^{\f12} \ll 1$):
\begin{equation}\label{Bound for g R*}
\| g \cdot R_* \|_{L^\infty H^{s-3}}+\| g \cdot R_* \|_{L^2 H^{\f16+s_1}}+ \|g\cdot R_*\|_{L^{\f74} B^{s_1+\f18\delta_0}_{\f73,\f74}}+ \| g \cdot R_* \|_{L^1 B^{-1+\f12 \delta_0}_{\infty,1}} \le \mathcal D.
\end{equation}

\medskip
\noindent Adding \eqref{Bound for dg F natural}, \eqref{Bound for g Ddg}, \eqref{Bound for dgdg} and \eqref{Bound for g R*}, we finally obtain \eqref{Necessary bound lapse 2}.

\end{proof}

\medskip
\subsubsection{\textbf{Estimates for the second fundamental form $h$.}}
We will now proceed to obtain the following estimates for the second fundamental form  $h$ of the foliation $\bar\Sigma_\tau$ using the fact that $h$ satisfies an elliptic equation along $\bar\Sigma_{\tau}$ (see Lemma \ref{lem:Parabolic elliptic system}):
\begin{lemma}\label{lem:Estimates h}
The tensor $h$ satisfies for any $p\in (1,+\infty)$:
\begin{align}\label{Bound h}
\|\bar\partial h \|_{L^\infty H^{s-3}} + \|\bar\partial h\|_{L^2 H^{\f16 + s_1}}& + \|\bar \partial h\|_{L^{\f74} B^{s_1+\f18\delta_0}_{\f73,\f74}}+ \|\bar\partial h\|_{L^1 B^{-1+\f12 \delta_0}_{p,1}}\\
&    \le C_{(p)} \Big( \mathcal D + \mathcal D^{\f12} \big(\|\partial g\|_{L^{\f74} B^{1+s_1+\f18\delta_0}_{\f73,\f74}}+ \| \partial g\|_{L^1 B^{\f12 \delta_0}_{p,1}}\big)\Big),
\end{align}
where $C_{(p)}>0$ is a constant depending only on $p,s$  and which is independent of $C_0$.
\end{lemma}

\begin{proof}
In view of Lemma \ref{lem:Parabolic elliptic system}, $h$ satisfies the following elliptic equation:
\begin{equation}\label{Equation h again}
\Delta_{\bar g} h_{ij} =  \bar\nabla^2_{ij}\big( \Delta_{\bar g} |D|^{-1}(N-1) \big) + \mathcal E_h,
\end{equation}
where
\[
\mathcal E_h = D^2 (g\cdot \tilde{\mathcal{F}}^{\natural} ) + D(g \cdot R_*) + g \cdot \partial g \cdot R_*.
\]
Therefore, the elliptic estimates from Lemma \ref{lem:Elliptic estimates model}, together with the bound
\[
\int_0^T \| g-m_0\|_{H^{s-\f12}(\bar\Sigma_\tau)}\cdot \|\bar\partial h\|_{W^{-1,\f{6}{1-\delta_0}}(\bar\Sigma_\tau)} \,d\tau \lesssim \|g-m_0\|_{L^2 H^{2+\f16+s_1}} \|\partial^2 g\|_{L^2 L^2} \lesssim C_0^2 \mathcal D^2
\]
obtained from the bootstrap assumption \eqref{Bootstrap bound} (and which can be used to control the last term in the right hand side of \eqref{Very low regularity elliptic estimate} for $f=h_{ij}$, $\bar s= s-1$ and $\sigma' = -2+\f12\delta_0$), imply that \eqref{Bound h} will follow once we establish the following bounds for the right hand side of \eqref{Equation h again}:
\begin{align}\label{Bound to show N for elliptic}
\|\bar\nabla^2_{ij}\big( & \Delta_{\bar g} |D|^{-1}(N-1) \big)  \|_{L^\infty H^{s-4}} + \|\bar\nabla^2_{ij}\big( \Delta_{\bar g} |D|^{-1}(N-1) \big)\|_{L^2 H^{-\f56 + s_1}}
\\
+&\|\bar\nabla^2_{ij} \big(\Delta_{\bar g}|D|^{-1} (N-1)\big)\|_{L^{\f74} B^{-1+s_1+\f18\delta_0}_{\f73,\f74}} + \|\bar\nabla^2_{ij}\big( \Delta_{\bar g} |D|^{-1}(N-1) \big)\|_{L^1 B^{-2+\f12 \delta_0}_{p,1}}  \nonumber \\
& \le C_{(p)} \Big( \mathcal D + \mathcal D^{\f12}\big( \|\partial g\|_{L^{\f74} B^{1+s_1+\f18\delta_0}_{\f73,\f74}}+\| \partial g\|_{L^1 B^{\f12 \delta_0}_{p,1}}\big)\Big)\nonumber
\end{align}
and
\begin{equation}\label{Bound to show Eh for elliptic}
\|\mathcal E_h \|_{L^\infty H^{s-4}} + \|\mathcal E_h\|_{L^2 H^{-\f56 + s_1}} + \|\mathcal E_h\|_{L^{\f74} B^{-1+s_1+\f18\delta_0}_{\f73,\f74}}+ \|\mathcal E_h\|_{L^1 B^{-2+\f12 \delta_0}_{p,1}}   \le C \mathcal D.
\end{equation}

The bound \eqref{Bound to show N for elliptic} follows directly from the estimates \eqref{Bound d N} for $N$ (together with the bootstrap assumption \eqref{Bootstrap bound} for the Riemannian metric $\bar g$). The bound \eqref{Bound to show Eh for elliptic} will follow once we establish a similar bound for each one of the summands $D^2 (g\cdot \tilde{\mathcal{F}}^{\natural} )$,  $D(g \cdot R_*)$ and $g \cdot \partial g \cdot R_*$ of $\mathcal E_h$; we will use very similar arguments as for the proof of Lemma \ref{lem:Bound EN} (and, thus, we will omit some details).
\begin{itemize}
\item For the term $D^2 (g\cdot \tilde{\mathcal{F}}^{\natural} )$, we obtain by combining the bootstrap bounds \eqref{Bootstrap bound} for $g$ with the estimates $\eqref{Bounds F natural easy}$ for $\tilde{\mathcal{F}}^{\natural}$:
\begin{align}
\|D^2 (g\cdot \tilde{\mathcal{F}}^{\natural} )&  \|_{L^\infty H^{s-4}} + \| D^2 (g\cdot \tilde{\mathcal{F}}^{\natural} )\|_{L^2 H^{-\f56 + s_1}} + \|D^2 (g\cdot \tilde{\mathcal F}^\natural)\|_{L^{\f74} B^{-1+s_1+\f18\delta_0}_{\f73,\f74}} + \| D^2 (g\cdot \tilde{\mathcal{F}}^{\natural} )\|_{L^1 B^{-2+\f12 \delta_0}_{p,1}} \nonumber \\
& \lesssim
\|\ g\cdot \tilde{\mathcal{F}}^{\natural}  \|_{L^\infty H^{s-2}} + \| g\cdot \tilde{\mathcal{F}}^{\natural} \|_{L^2 H^{\f76 + s_1}}+\|g\cdot \\tilde{mathcal F}^\natural\|_{L^{\f74} B^{1+s_1+\f18\delta_0}_{\f73,\f74}}+ \|g\cdot \tilde{\mathcal{F}}^{\natural} \|_{L^1 B^{\f12 \delta_0}_{p,1}} \nonumber \\
&\lesssim \mathcal D. \label{Bound D g F natural}
\end{align}

\smallskip
\item For the term $D(g \cdot R_*)$, we can show using the bound \eqref{Bound for g R*} established already in the proof of Lemma \ref{lem:Estimates N}:
\begin{align}
\|D (g\cdot R_* )&  \|_{L^\infty H^{s-4}} + \| D (g\cdot R_* )\|_{L^2 H^{-\f56 + s_1}} \|D(g\cdot R_*)\|_{L^{\f74} B^{-1+s_1+\f18\delta_0}_{\f73,\f74}}+ \| D (g\cdot R_* )\|_{L^1 B^{-2+\f12 \delta_0}_{p,1}} \nonumber \\
& \lesssim
\|\ g\cdot R_*   \|_{L^\infty H^{s-3}} + \| g\cdot R_*  \|_{L^2 H^{\f16 + s_1}} +\|g\cdot R_*\|_{L^{\f74} B^{s_1+\f18\delta_0}_{\f73,\f74}}+ \|g\cdot R_*  \|_{L^1 B^{-1+\f12 \delta_0}_{p,1}} \nonumber \\
&\lesssim \mathcal D. \label{Bound D g R}
\end{align}

\smallskip
\item For the term $g \cdot \partial g \cdot R_*$, we can similarly argue (using the functional inequalities of Lemma \ref{lem:Functional inequalities}):
\begin{align*}
\| g \cdot \partial g \cdot R_* \|_{L^\infty H^{s-4}} 
& \lesssim  \| g \|_{L^\infty H^{s-1}} \|\partial g\|_{L^\infty H^{s-2}} \|R_* \|_{L^\infty H^{s-3}} \\
& \lesssim C_0^2 \mathcal D^2,
\end{align*}

\begin{align*}
\| g \cdot \partial g \cdot R_* \|_{L^2 H^{-\f56+s_1}} 
& =   \| g \cdot \partial g \cdot R_* \|_{L^2 H^{s-\f72-2\delta_0}} \\
& = \Big( \int_{0}^T \| g \cdot \partial g \cdot R_* \|^2_{H^{s-\f72-2\delta_0}(\bar\Sigma_{\tau})} \, d\tau\Big)^{\f12}\\
& \lesssim \Big( \int_{0}^T \| g \cdot \partial g\|^2_{H^{s-\f32-2\delta_0}(\bar\Sigma_{\tau})} \| R_* \|^2_{H^{s-\f52-2\delta_0}(\bar\Sigma_{\tau})} \, d\tau\Big)^{\f12} \\
& \lesssim \Big( \int_{0}^T \| g\|^2_{H^{s-1}(\bar\Sigma_{\tau})} \|\partial g\|^2_{H^{s-2}(\bar\Sigma_{\tau})} \| R_* \|^2_{H^{\f16+s_1}(\bar\Sigma_{\tau})} \, d\tau\Big)^{\f12}\\
& 
\lesssim \| g \|_{L^\infty H^{s-1}} \|\partial g\|_{L^\infty H^{s-2}} \| R_*\|_{L^2 H^{\f16+s_1}} \\
& \lesssim C_0^2 \mathcal D^2,
\end{align*}
\begin{align*}
\|g\cdot \partial g \cdot R_*\|_{L^{\f74} B^{-1+s_1+\f18\delta_0}_{\f73,\f74}} 
& \lesssim \|g\cdot \partial g \cdot R_*\|_{L^{\f74}W^{-1+s_1+\f3{16}\delta_0,\f73}}\\
& \lesssim \|g\|_{L^\infty H^{s-1}} \|\partial g\|_{L^\infty H^{s-2}} \|R_*\|_{L^{\f74} W^{s_1+\f14\delta_0,\f73}}\\
& \lesssim C_0^2 \mathcal D^2
\end{align*}
and
\begin{align*}
\| g \cdot \partial g \cdot R_* \|_{L^1 B^{-2+\f12 \delta_0}_{\infty,1}} & 
\lesssim \sum_j 2^{(-2+\f12\delta_0)j} \| P_j (g \cdot \partial g \cdot R_*) \|_{L^1 L^\infty}\\
& \lesssim \sum_j \Bigg( 2^{(-2+\f12\delta_0)j} \| P_j (g\cdot \partial g) \cdot P_{\le j} R_* \|_{L^1 L^\infty} 
+ 2^{(-2+\f12\delta_0)j} \| P_{\le j} (g\cdot \partial g) \cdot P_{ j} R_* \|_{L^1 L^\infty}  \\
& \hphantom{\sum_j \Bigg( }
+ \sum_{j'>j} 2^{\f12\delta_0j} \| P_j \big( P_{j'} (g\cdot \partial g) \cdot P_{j'} R_* \big) \|_{L^1 L^{\f32}} \Bigg) \\
& \lesssim \sum_j \Bigg( \| P_j (g\cdot \partial g)\|_{L^\infty W^{-1,\infty}} \cdot 2^{(-1 +\f12\delta_0) j}\| P_{\le j} R_* \|_{L^1 L^\infty} \\
& \hphantom{\sum_j \Bigg(}
+  \| P_{\le j} (g\cdot \partial g)\|_{L^\infty W^{-1,\infty}} \cdot 2^{(-1+\f12\delta_0)j}\| P_j R_* \|_{L^1 L^\infty}  \\
& \hphantom{\sum_j \Bigg( }
+\sum_{j'>j} 2^{\f12\delta_0(j-j')-\f12\delta_0 j'} 2^{ j'}\| P_{j'} (g\cdot \partial g)\|_{L^2 L^2} \cdot 2^{(-1+\delta_0)j'}\| P_{j'} R_* \|_{L^2 L^6} \Bigg) \\
& \lesssim \| g\cdot \partial g \|_{L^\infty W^{-1+\delta_0,\infty}} \|R_* \|_{L^1 W^{-1+\delta_0,\infty}} + \| g\cdot \partial g\|_{L^2 H^{1+\delta_0}} \|R_*\|_{L^2 W^{-1+\delta_0,6}} \\
& \lesssim \| g \|_{L^\infty H^{s-1}} \|\partial g\|_{L^\infty H^{s-2}} \|R_*\|_{L^1 W^{-1+\delta_0,\infty}}  + \|g\|_{L^\infty H^{s-1}}\| \partial^2 g\|_{L^2 H^{\f16 +\delta_0}} \|R_*\|_{L^2 H^{\delta_0}} \\
& \lesssim C_0^2 \mathcal D^2.
\end{align*}
Combining the above bounds, we obtain (provided $C_0 \mathcal D^{\f12} \ll 1$):
\begin{equation}\label{Bound g dg R}
\|g \cdot \partial g \cdot R_*  \|_{L^\infty H^{s-4}} + \| g \cdot \partial g \cdot R_*\|_{L^2 H^{-\f56 + s_1}}+ \|g\cdot \partial g \cdot R_*\|_{L^{\f74} B^{-1+s_1+\f18\delta_0}_{\f73,\f74}}+ \| g \cdot \partial g \cdot R_*\|_{L^1 B^{-2+\f12 \delta_0}_{p,1}} \lesssim \mathcal D.
\end{equation}
\end{itemize}
The bounds \eqref{Bound D g F natural}--\eqref{Bound g dg R} finally yield \eqref{Bound to show Eh for elliptic}.

\end{proof}

\medskip
\subsubsection{\textbf{Estimates for the shift $\beta$.}}
We will obtain the necessary estimates for the shift vector field $\beta$ using similar arguments as for the proof of the analogous estimates for $N$ (see Lemma \ref{lem:Estimates N}); both quantities satisfy a similar parabolic equation (see Lemma \ref{lem:Parabolic elliptic system}), a crucial difference being that the equation for $\beta$ contains a linear top order term involving $h$. 

\begin{lemma}\label{lem:Estimates b}
The shift vector field $\beta$ satisfies the following estimate for any $p\in (1,+\infty)$:
\begin{align}\label{Bound d b}
\|\partial \beta \|_{L^\infty H^{s-2}} + \|\partial \b\|_{L^2 H^{\f76 + s_1}}+&   \|\partial \beta \|_{L^{\f74} B^{1+ s_1+\f18\delta_0}_{\f73,\f74}}+ \| \partial \b\|_{L^1 B^{\f12 \delta_0}_{p,1}}\\
&   + \| \b\|_{L^\infty L^\infty} \le C_{(p)} \Big( \mathcal D + \mathcal D^{\f12} \big(\|\partial g\|_{L^{\f74} B^{1+s_1+\f18\delta_0}_{\f73,\f74}}+\| \partial g\|_{L^1 B^{\f12 \delta_0}_{p,1}}\big)\Big),\nonumber 
\end{align}
where $C_{(p)}>0$ is a constant depending only on $p,s$  and which is independent of $C_0$.
\end{lemma}

\begin{proof}
The function $\b$ satisfies the following parabolic equation (see Lemma \ref{lem:Parabolic elliptic system}):
\begin{equation}\label{Equation shift again}
\partial_0 \b + |D|\b = \JapD^{-1}(g\cdot Dh) + \mathcal E^{\prime}_{\beta},
\end{equation}
where
\[
\mathcal E^{\prime}_{\beta} =\JapD^{-1} \Big(  (g-m_0) \cdot D\partial g + g \cdot R_*  +  g \cdot \partial g \cdot \partial g \Big).
\]
Using the facts that
\begin{itemize}
\item $\fint_{(\bar\Sigma_{\tau},g_{\mathbb T^3})} \beta=0$ (in view of condition \eqref{Harmonic condition}) and
\item $\| \beta |_{x^0=0} \|_{H^{s-1}} \lesssim \mathcal D $ (as a consequence of Proposition \ref{prop: Bounds S0}),
\end{itemize}
the bound \eqref{Bound d b} will follow immediately from Lemma \ref{lem:Parabolic estimates model} from the Appendix,  provided
\begin{align}\label{Necessary bound shift h}
\|\JapD^{-1} (g \cdot Dh)  & \|_{L^\infty H^{s-2}} + \| \JapD^{-1} (g \cdot Dh)  \|_{L^2 H^{\f76+s_1}}+ \| \JapD^{-1} (g \cdot Dh)  \|_{L^{\f74} B^{1+s_1+\f18\delta_0}_{\f73,\f74}} \\
&+  \| \JapD^{-1} (g \cdot Dh)  \|_{L^1 B^{\f12 \delta_0}_{p,1}} \lesssim_p \mathcal D + \mathcal D^{\f12}\big(\|\partial g\|_{L^{\f74} B^{1+s_1+\f18\delta_0}_{\f73,\f74}} \| \partial g \|_{L^1 B^{\f12 \delta_0}_{p,1}}\big) \nonumber 
\end{align}
and
\begin{align}\label{Necessary bound shift error}
\| \mathcal E^{\prime}_\beta  \|_{L^\infty H^{s-2}} &+ \| \mathcal E^{\prime}_\beta  \|_{L^2 H^{\f76+s_1}}+ \|\mathcal E^\prime_\beta\|_{L^{\f74} B^{1+s_1+\f18\delta_0}_{\f73,\f74}} \\&+  \| \mathcal E^{\prime}_\beta  \|_{L^1 B^{\f12 \delta_0}_{p,1}} \lesssim_p \mathcal D + \mathcal D^{\f12} \big(\|\partial g \|_{L^{\f74} B^{1+s_1+\f18\delta_0}_{\f73,\f74}}+\| \partial g \|_{L^1 B^{\f12 \delta_0}_{p,1}}\big).\nonumber
\end{align}
The bound \eqref{Necessary bound shift h} follows directly from the estimate \eqref{Bound h} for $h$, together with the bootstrap bounds \eqref{Bootstrap bound} for $g$. The bound \eqref{Necessary bound shift error} for $\mathcal E^{\prime}_\beta$ follows by arguing exactly as for the proof of the bound \eqref{Necessary bound lapse 2} for $\mathcal E_N$.

\end{proof}

\medskip
\subsubsection{\textbf{Estimates for the induced metric $\bar g$.}}
Let us turn our attention to the components of the Riemannian metric $\bar g$ on the slices $\bar\Sigma_{\tau}$. We will obtain the following estimates using the fact that $\bar g$ satisfies an elliptic system along $\bar\Sigma_{\tau}$ (see Lemma \ref{lem:Parabolic elliptic system}):
\begin{lemma}\label{lem:Estimates bar g}
The Riemannian metric $\bar g$ satisfies for any $p\in (1,+\infty)$:
\begin{align}\label{Bound bar g}
\|\partial  \bar g \|_{L^\infty H^{s-2}} &  + \|\partial \bar g\|_{L^2 H^{\f76 + s_1}}+ \|\partial \bar g \|_{L^{\f74} B^{1+s_1+\f18\delta_0}_{\f73,\f74}}\\
& + \|\partial \bar g\|_{L^1 B^{\f12 \delta_0}_{p,1}}  + \|\bar g - \delta_E\|_{L^\infty L^\infty} \le C_{(p)} \Big( \mathcal D + \mathcal D^{\f12} \big(\|\partial g \|_{L^{\f74} B^{1+s_1+\f18\delta_0}_{\f73,\f74}}+\| \partial g\|_{L^1 B^{\f12 \delta_0}_{p,1}}\big)\Big),\nonumber 
\end{align}
where $(\delta_E)_{ij} = \delta_{ij}$ are the components of the Euclidean metric in Cartesian coordinates and $C_{(p)}>0$ is a constant depending only on $p,s$  and which is independent of $C_0$.
\end{lemma}

\begin{proof}
According to Lemma \ref{lem:Parabolic elliptic system}, the components of $\bar g$ and $\partial_0 \bar g$ satisfy the system
\begin{align}
\bar{g}^{kl} \partial_k \partial_l (\bar g_{ij}) & = g \cdot D \partial \beta + g\cdot R_{**} + g\cdot \partial g \cdot \partial g,\label{Elliptic bar g again}\\
\partial_0 \bar g & = g \cdot h + g \cdot D \beta + (g-m_0)\cdot \partial g.   \label{Dt bar g again}
\end{align}
In view of our bootstrap assumption \eqref{Bootstrap bound}, the elliptic operator $\bar{g}^{kl} \partial_k \partial_l $ satisfies the requirements of Lemma \ref{lem:Elliptic estimates model}; thus, using the elliptic estimates from Lemma \ref{lem:Elliptic estimates model} to estimate the spatial derivatives $\bar\partial \bar g$ and equation \eqref{Dt bar g again} to estimate $\partial_0 \bar g$ directly, we obtain
\begin{align}\label{Auxiliary bound d bar g}
\|\partial \bar g \|_{\mathfrak N_p} \lesssim_p &  \Big\| \JapD^{-1}\Big( g \cdot D \partial \beta + g\cdot R_* + g\cdot \partial g \cdot \partial g \Big) \Big\|_{\mathfrak N_p} \\
& + \Big\| g \cdot h + g \cdot D \beta + (g-m_0)\cdot \partial g  \Big\|_{\mathfrak N_p},   \nonumber
\end{align}
where the norm $\|\cdot\|_{\mathfrak N_p}$ is defined by
\[
\| f \|_{\mathfrak N_p} \doteq \|f \|_{L^\infty H^{s-2}} + \|f\|_{L^2 H^{\f76 + s_1}}+\|f\|_{L^{\f74} B^{1+s_1+\f18\delta_0}_{\f73,\f74}}+ \|f\|_{L^1 B^{\f12 \delta_0}_{p,1}}.
\]

The bound \eqref{Bound d b} for $\partial \beta$ together with the bootstrap bound \eqref{Bootstrap bound} for $g$ 
imply that
\[
\big\|   \JapD^{-1}(g \cdot D \partial \beta)\big\|_{\mathfrak N_p}  + \big\| g \cdot h\big\|_{\mathfrak N_p} + \big\|g \cdot D \beta  \big\|_{\mathfrak N_p} \lesssim_p \mathcal D + \mathcal D^{\f12}\big( \|\partial g\|_{L^{\f74} B^{1+s_1+\f18\delta_0}_{\f73,\f74}}+\|\partial g\|_{L^1 W^{\f12\delta_0,p}} \big)+ \| h \|_{L^\infty L^2}.
\]
Moreover, the bounds \eqref{Bound for g R*} for $g\cdot R_*$, \eqref{Bound for dgdg} for $g\cdot \partial g \cdot \partial g$ and \eqref{Bound for g Ddg} for $(g-m_0)\cdot \partial g$ can be combined into
\[
\big\|   \JapD^{-1}(g \cdot R_*)\big\|_{\mathfrak N_p} + \big\| |D|^{-1}(g \cdot \partial g \cdot \partial g)\big\|_{\mathfrak N_p} + \big\| (g-m_0)\cdot \partial g\big\|_{\mathfrak N_p} \lesssim_p \mathcal D + \mathcal D^{\f12} \big(\|\partial g\|_{L^{\f74} B^{1+s_1+\f18\delta_0}_{\f73,\f74}}+\|\partial g\|_{L^1 W^{\f12\delta_0,p}}\big).
\]
Combining the above bounds, we infer from \eqref{Auxiliary bound d bar g}:
\begin{equation}\label{First bound d bar g}
\|\partial \bar g \|_{\mathfrak N_p} \lesssim_p \mathcal D + \mathcal D^{\f12} \big(\|\partial g\|_{L^{\f74} B^{1+s_1+\f18\delta_0}_{\f73,\f74}}+\|\partial g\|_{L^1 B^{\f12 \delta_0}_{p,1}} \big)+ \| h \|_{L^\infty L^2}.
\end{equation}
Adding to \eqref{First bound d bar g} the zero-frequency estimate
\begin{align*}
\|\bar g - \delta_{E} \|_{L^\infty L^\infty} &  \lesssim \|\partial \bar g \|_{L^\infty H^{\f12+\delta_0}} + \|\bar g - \delta_{E}\|_{L^\infty L^2} \\
& \lesssim \|\partial \bar g \|_{\mathfrak N_p} + \|\bar g - \delta_{E}\|_{L^\infty L^2}, 
\end{align*}
we infer
\begin{align}\label{Second bound d bar g}
\|\partial  \bar g \|_{L^\infty H^{s-2}} & + \|\partial \bar g\|_{L^2 H^{\f76 + s_1}}+ \|\partial \bar g\|_{L^1 B^{\f12 \delta_0}_{p,1}}  + \|\bar g - \delta_E\|_{L^\infty L^\infty} \\
&  \lesssim_p \mathcal D + \mathcal D^{\f12}\big(\partial g\|_{L^{\f74} B^{1+s_1+\f18\delta_0}_{\f73,\f74}}+ \|\partial g\|_{L^1 B^{\f12 \delta_0}_{p,1}}\big) + \| h \|_{L^\infty L^2} + \|\bar g - \delta_{E}\|_{L^\infty L^2}.
\end{align}

In order to obtain the estimate \eqref{Bound bar g} for $\partial g$ and complete the proof of the Lemma \ref{lem:Estimates bar g}, it suffices to establish the following low frequency estimate for the last two terms in the right hand side of \eqref{Second bound d bar g}:
\begin{equation}\label{To show zero frequency estimate g h}
\| h \|_{L^\infty L^2} + \|\bar g - \delta_{E}\|_{L^\infty L^2} \lesssim \mathcal D.
\end{equation}
In view of the fact that, as a consequence of Proposition \ref{prop: Bounds S0}, the above bound is satisfied at $x^0=0$, i.e.
\[
\| h|_{x^0=0} \|_{L^2} + \|(\bar g - \delta_{E})|_{x^0=0}\|_{L^2} \lesssim \mathcal D
\]
(see \eqref{Bound metric initial}), we can estimate:
\begin{align}\label{Low frequency estimate h bar g}
\|\bar g - \delta_{E}\|_{L^\infty L^2} &\lesssim \mathcal D + \|\partial_0 \bar g\|_{L^1 L^2}\\
\| h \|_{L^\infty L^2} &\lesssim \mathcal D + \| \partial_0 h \|_{L^1 L^2}. \nonumber 
\end{align}
Using the expressions
\begin{align*}
\partial_0 \bar g & = g \cdot h + g \cdot D\beta + (g-m_0) \cdot \partial g,\\
\partial_0 h & = D^2 N + g \cdot \partial g \cdot \partial g + g \cdot R 
\end{align*}
(following from \eqref{Variation metric}--\eqref{Variation second fundamental form}), we can immediately bound (using the estimates \eqref{Bound d N}, \eqref{Bound d b} and \eqref{Estimate R easy} established already for $N$, $\beta$ and $R$, respectively, as well as the bootstrap bound \eqref{Bootstrap bound} for $g$):
\begin{align*}
 \| \partial_0 h \|_{L^1 L^2} &  \lesssim  \mathcal D,\\
\|\partial_0 \bar g\|_{L^1 L^2} & \lesssim \|h\|_{L^1 L^2} + \mathcal D.
\end{align*}
Returning to \eqref{Low frequency estimate h bar g} and using the trivial bound $\|h\|_{L^1 L^2}  \lesssim \|h\|_{L^\infty L^2} $, we obtain \eqref{To show zero frequency estimate g h}. Using \eqref{To show zero frequency estimate g h} to bound the last two terms in the right hand side of \eqref{Second bound d bar g}, we obtain \eqref{Bound bar g}.

\end{proof}

\medskip
\subsubsection{\textbf{Combining the bounds for $g$ and $\partial g$.}}
The following result provides an improvement of the bootstrap bound \eqref{Bootstrap bound} for the components of $g$ and $\partial g$:

\begin{lemma}\label{Estimate g and dg}
The components of the metric $g$ satisfy the bound
\begin{equation}\label{Bound g and d g}
\|\partial  g \|_{L^\infty H^{s-2}} + \|\partial  g\|_{L^2 H^{\f76 + s_1}}+\|\partial g\|_{L^{\f74} W^{1+s_1,\f73}}+ \|\partial  g\|_{L^1 W^{\f14\delta_0,\infty}}  + \|g - m_0\|_{L^\infty L^\infty} \le C \mathcal D
\end{equation}
for a constant $C>0$ independent of $C_0$.
\end{lemma}

\begin{proof}
Combining the estimates \eqref{Bound d N}, \eqref{Bound d b} and \eqref{Bound bar g}, we infer that, for any $p\in (1,+\infty)$:
\begin{align*}
\|\partial  g \|_{L^\infty H^{s-2}} + \|\partial  g\|_{L^2 H^{\f76 + s_1}} &+\|\partial g\|_{L^{\f74} B^{1+s_1+\f18\delta_0}_{\f73,\f74}}+ \|\partial  g\|_{L^1 B^{\f12 \delta_0}_{p,1}}  + \|g - m_0\|_{L^\infty L^\infty}\\& \lesssim_p \Big( \mathcal D + \mathcal D^{\f12}\big(\|\partial g\|_{L^{\f74} B^{1+s_1+\f18\delta_0}_{\f73,\f74}}+ \| \partial g\|_{L^1B^{\f12 \delta_0}_{p,1}}\big)\Big).
\end{align*}
Fixing $p=\frac{14}{\delta_0}$ and assuming that $\epsilon$ (and thus $\mathcal D$) in \eqref{D} is small enough in terms of $s$ so that the last term in the right hand side above can be absorbed into the left hand side, we infer \eqref{Bound g and d g}.
\end{proof}

\medskip
\subsubsection{\textbf{Estimates for $\partial^2 g$.}}
We are now in a position where, using the relations provided by Lemma \ref{lem:Parabolic elliptic system}, we can improve the bootstrap estimates \eqref{Bootstrap bound} for $\partial^2 g$:

\begin{lemma}\label{Estimate d2 g}
The second order derivatives $\partial^2 g$ of the metric $g$ satisfy the bound
\begin{equation}\label{Bound d2 g}
\|\partial^2  g \|_{L^\infty H^{s-3}} + \|\partial^2  g\|_{L^2 H^{\f16 + s_1}}+\|\partial^2 g\|_{L^{\f74} W^{s_1,\f73}}\le C \mathcal D
\end{equation}
for a constant $C>0$ independent of $C_0$.
\end{lemma}

\begin{proof}
The bound \eqref{Bound g and d g} implies \eqref{Bound d2 g} for terms of the form $\bar\partial \partial g$, i.e.
\begin{equation}\label{Bound easy Ddg}
\|\bar\partial \partial  g \|_{L^\infty H^{s-3}} + \|\bar\partial \partial  g\|_{L^2 H^{\f16 + s_1}}+\|\bar\partial\partial g\|_{L^{\f74} W^{s_1,\f73}}\lesssim \mathcal D.
\end{equation}
The relations \eqref{Equation dt2 lapse}--\eqref{Equation dt2 bar g} for $\partial_0^2 N$, $\partial_0^2 \beta$ and $\partial_0^2 \bar g$ can be schematically reexpressed as
\begin{equation}\label{Schematic dt2 g}
\partial_0^2 g = D\partial g +\JapD^{-1}\big( g \cdot \mathfrak F^{(1)}_{\partial_0^2 g}\big)+\mathfrak F^{(2)}_{\partial_0^2 g} + \mathcal E_{\partial_0^2 g},
\end{equation}
where
\[
\big(\mathfrak F^{(1)}_{\partial_0^2 g}\big)_{cd} = \partial^2_0 \tilde{\mathcal F}^\natural_{cd} - \partial_0 R_{c0d0},
\]
\[
\mathfrak F_{\partial_0^2 g} = \JapD^{-1}\Big( g\cdot \partial g \cdot \partial \tilde{\mathcal F}^\natural + \partial g \cdot \partial g \cdot \tilde{\mathcal F}^\natural  \Big)
\]
and
\begin{align*}
\mathcal E_{\partial_0^2 g} = & |D|^{-1} \Big((g-m_0)\cdot D\partial^2 g+g \cdot \partial g \cdot \partial^2 g + g\cdot D^2 \partial g + g\cdot \partial R_* + \partial g \cdot \partial g \cdot \partial g \Big) \\
& +g\cdot D\partial g + g \cdot R + g\cdot \partial g \cdot \partial g.
\end{align*}
Therefore, \eqref{Bound easy Ddg} and \eqref{Schematic dt2 g} imply that
\begin{align}\label{Almost there d2 g}
\|\partial^2  g \|_{L^\infty H^{s-3}} + \| \partial^2 g\|_{L^2 H^{\f16 + s_1}}
+\|\partial^2 g\|_{L^{\f74} W^{s_1,\f73}}\lesssim \mathcal D & + \|\mathfrak F^{(1)}_{\partial_0^2 g}  \|_{L^\infty H^{s-4}} + \| \mathfrak F^{(1)}_{\partial_0^2 g} \|_{L^2 H^{-\f56 + s_1}} +\|\mathfrak F^{(1)}_{\partial_0^2 g}\|_{L^{\f74} W^{-1+s_1,\f73}} \\
& + \|\mathfrak F^{(2)}_{\partial_0^2 g}  \|_{L^\infty H^{s-3}} + \| \mathfrak F^{(2)}_{\partial_0^2 g} \|_{L^2 H^{\f16 + s_1}}+\|\mathfrak F^{(2)}_{\partial_0^2 g}\|_{L^{\f74} W^{s_1,\f73}}  \nonumber \\
& + \|\mathcal E_{\partial_0^2 g} \|_{L^\infty H^{s-3}} + \| \mathcal E_{\partial_0^2 g}\|_{L^2 H^{\f16 + s_1}}+\|\mathfrak E_{\partial_0^2 g}\|_{L^{\f74} W^{s_1,\f73}}.\nonumber
\end{align}

We will estimate the terms in the right hand side of \eqref{Almost there d2 g} as follows:
\begin{itemize}
\item The terms in the first line of the right hand side of \eqref{Almost there d2 g} have also been already estimated in Lemmas \ref{lem:Bounds F natural} and \ref{lem:Riemann estimates} (see \eqref{Bound F natural higher order} and \eqref{Estimate R hard 2}):
\[
\|\mathfrak F^{(1)}_{\partial_0^2 g}  \|_{L^\infty H^{s-4}} + \| \mathfrak F^{(1)}_{\partial_0^2 g} \|_{L^2 H^{-\f56 + s_1}}+\| \mathfrak F^{(1)}_{\partial_0^2 g}\|_{L^{\f74} W^{-1+s_1,\f73}} \lesssim \mathcal D.
\]
\item In view of the estimates \eqref{Bounds F natural easy} for $\partial \tilde{\mathcal F}^\natural$ and $\tilde{\mathcal F}^\natural $ as well as the bound \eqref{Bound g and d g} for $g$ and $\partial g$, we can readily bound:
\[
\|\mathfrak F^{(2)}_{\partial_0^2 g}  \|_{L^\infty H^{s-3}} + \| \mathfrak F^{(2)}_{\partial_0^2 g} \|_{L^2 H^{\f16 + s_1}}+\| \mathfrak F^{(2)}_{\partial_0^2 g}\|_{L^{\f74} W^{s_1,\f73}}  \lesssim \mathcal D.
\]

\item Using the bound \eqref{Bound g and d g} for $g$ and $\partial g$, the bound \eqref{Estimate R easy} for $R$ and the expression (following from the Bianchi identity for $R$)
\[
\partial R_* = \bar\partial R + g \cdot \partial g \cdot R,
\] 
the functional inequalities from Lemma \ref{lem:Functional inequalities} imply that
\begin{align*}
\|\mathcal E_{\partial_0^2 g}  \|_{L^\infty H^{s-3}} &+ \| \mathcal E_{\partial_0^2 g} \|_{L^2 H^{\f16 + s_1}} +\|\mathcal E_{\partial_0^2 g} \|_{L^{\f74} W^{s_1,\f73}}\\& \lesssim \mathcal D+ \mathcal D \Big( \| \partial^2 g  \|_{L^\infty H^{s-3}} + \| \partial^2 g \|_{L^2 H^{\f16 + s_1}} +\|\partial^2 g\|_{L^{\f74} W^{s_1,\f73}} \Big).
\end{align*}
\end{itemize}

Combining the above bounds, we therefore obtain
\[
\|\partial^2  g \|_{L^\infty H^{s-3}} + \|\partial^2  g\|_{L^2 H^{\f16 + s_1}}+\|\partial^2 g\|_{L^{\f74} W^{s_1,\f73}}\lesssim \mathcal D + \mathcal D  \Big( \| \partial^2 g  \|_{L^\infty H^{s-3}} + \| \partial^2 g \|_{L^2 H^{\f16 + s_1}}+\|\partial^2 g\|_{L^{\f74} W^{s_1,\f73}}  \Big)
\]
in which, provided $\epsilon$ (and thus $\mathcal D$) in \eqref{D} is small enough in terms of $s$, we can absorb the last term in the right hand side into the left hand side, thus obtaining \eqref{Bound d2 g}.

\end{proof}

\medskip
\subsubsection{\textbf{Estimates for the connection coefficients $\omega$.}}
We are now able to improve the bootstrap assumption \eqref{Bootstrap bound} for the connection coefficients $\omega$ and their derivatives:

\begin{lemma}\label{lem:Estimates omega}
The connection coefficients $\omega$ satisfy:
\begin{align}\label{Bound omega d omega}
\|\partial \omega\|_{L^\infty H^{s-3}} + \|\partial \omega\|_{L^2 H^{\f16 +s_1}} 
+\|\partial\omega\|_{L^{\f74} W^{s_1,\f73}}+ \|\omega\|_{L^1 W^{\f14\delta_0,\infty}} \le C \mathcal D
\end{align}
for some constant $C>0$ independent of $C_0$.
\end{lemma}

\begin{proof}
The temporal components $\omega_0$ of $\omega$ satisfy the parabolic equation (see Lemma \ref{lem:Parabolic elliptic system}):
\begin{equation}\label{Equation omega 0 again}
\partial_0   \omega_{0\bB}^{\bA} +|D| \omega_{0\bB}^{\bA} = \big( \mathfrak F_{\omega_0}\big)^{\bA}_{\bB}  +  \mathcal E_{\omega_0},
\end{equation}
where
\[
\big( \mathfrak F_{\omega_0}\big)^{\bA}_{\bB} \doteq 
 |D|\Delta_{\bar g}^{-1} \Bigg[  \Big(  m_{\bB\bC} \partial_0 \tilde{\mathcal{F}}_{\perp}^{\bA\bC} -\fint_{(\bar\Sigma_{\tau}, \bar g)} \big( m_{\bB\bC} \partial_0 \tilde{\mathcal{F}}_{\perp}^{\bA\bC} \Big) 
 +\bar g^{ij} \bar{\nabla}_i (R^{\perp})_{j0\hphantom{\bA}\bB}^{\hphantom{\a\b}\bA} \Bigg]
\]
and
\begin{align*}
 \mathcal E_{\omega_0} = & 
 \JapD^{-1} \big( \partial g \cdot D \omega\big) + \JapD^{-1} \big( D\partial g \cdot g\cdot \omega \big) +  \JapD^{-1} \big( \partial g \cdot \omega \cdot \omega \big)\\
&  + g\cdot \omega \cdot \omega + \JapD^{-1} \big(\partial e \cdot \tilde{\mathcal F}_{\perp} + \partial g \cdot \tilde{\mathcal F}_{\perp} \big).
\end{align*}
Using the facts that
\begin{itemize}
\item $\fint_{(\bar\Sigma_{\tau}, g_{\mathbb T^3})} \omega_0 =0$ (in view of condition \eqref{Divergence condition frame}) and
\item $\| \omega_0 |_{x^0=0} \|_{H^{s-2}} \lesssim \mathcal D $ (as a consequence of Proposition \ref{prop: Bounds S0}),
\end{itemize}
the estimates provided by Lemma \ref{lem:Parabolic estimates model} from the Appendix imply that, for any $p\in (1,+\infty)$:
\begin{align*}
\|\partial \omega_0\|_{L^\infty H^{s-3}} + \|\partial \omega_0 & \|_{L^2 H^{\f16 +s_1}} 
+\|\partial\omega_0\|_{L^{\f74} W^{s_1,\f73}}+ \|\omega_0\|_{L^1 B^{ \f34\delta_0}_{p,1}}\\
 \lesssim_p \mathcal D & + \|\mathfrak F_{\omega_0}\|_{L^\infty H^{s-3}} + \|\mathfrak F_{\omega_0}\|_{L^2 H^{\f16 +s_1}} +\|\mathfrak F_{\omega_0}\|_{L^{\f74} W^{s_1,\f73}}
+ \|\mathfrak F_{\omega_0}\|_{L^1 B^{-1+ \f34\delta_0}_{p,1}}  \nonumber    \\
& +\|\mathcal E_{\omega_0}\|_{L^\infty H^{s-3}} + \|\mathcal E_{\omega_0}\|_{L^2 H^{\f16 +s_1}} +\|\mathcal E_{\omega_0}\|_{L^{\f74} W^{s_1,\f73}}
+ \|\mathcal E_{\omega_0}\|_{L^1 B^{-1+\f34\delta_0}_{p,1}}.    \nonumber 
\end{align*}
Therefore, fixing $p$ large enough in term $\delta_0$, we have:
\begin{align}\label{Bound to show omega 0}
\|\partial \omega_0\|_{L^\infty H^{s-3}} + \|\partial \omega_0 & \|_{L^2 H^{\f16 +s_1}} +\|\partial\omega_0\|_{L^{\f74} W^{s_1,\f73}}
+ \|\omega_0\|_{L^1 W^{ \f12\delta_0, \infty}}\\
 \lesssim_p \mathcal D & + \|\mathfrak F_{\omega_0}\|_{L^\infty H^{s-3}} + \|\mathfrak F_{\omega_0}\|_{L^2 H^{\f16 +s_1}} +\|\mathfrak F_{\omega_0}\|_{L^{\f74} W^{s_1,\f73}}
+ \|\mathfrak F_{\omega_0}\|_{L^1 W^{ \delta_0, \infty}}  \nonumber    \\
& +\|\mathcal E_{\omega_0}\|_{L^\infty H^{s-3}} + \|\mathcal E_{\omega_0}\|_{L^2 H^{\f16 +s_1}} +\|\mathcal E_{\omega_0}\|_{L^{\f74} W^{s_1,\f73}}
+ \|\mathcal E_{\omega_0}\|_{L^1 W^{ \delta_0, \infty}}.    \nonumber 
\end{align}

The term $\mathfrak F_{\omega_0}$ in the right hand side of \eqref{Bound to show omega 0} has already been estimated in Lemma \ref{lem:Bounds F perp}: As a consequence of the  bound \eqref{Bound F perp L1Linfty}, together with the bootstrap assumption \eqref{Bootstrap bound} for $g$, we can estimate
\begin{equation}\label{Estimate to show F omega 0}
 \|\mathfrak F_{\omega_0}\|_{L^\infty H^{s-3}} + \|\mathfrak F_{\omega_0}\|_{L^2 H^{\f16 +s_1}} +\|\mathfrak F_{\omega_0}\|_{L^{\f74} W^{s_1,\f73}}
+ \|\mathfrak F_{\omega_0}\|_{L^1 W^{-1+\delta_0,\infty}} \lesssim \mathcal D.
\end{equation}
For the $\mathcal E_{\omega_0}$ term in the right hand side of \eqref{Bound to show omega 0}, we will argue as follows:
\begin{itemize}
\item Using the functional inequalities from Lemma \ref{lem:Functional inequalities}, we can readily estimate:
\begin{align*}
\| \JapD^{-1} \big( \partial g \cdot D \omega\big) & \|_{L^\infty H^{s-3}} +\| \JapD^{-1} \big( D\partial g \cdot g\cdot \omega \big) \|_{L^\infty H^{s-3}} + \|  \JapD^{-1} \big( \partial g \cdot \omega \cdot \omega \big)\|_{L^\infty H^{s-3}} \\
 & \hphantom{\|_{L^\infty H^{s-3}} +\| |D|^{-1} \big(}
+\| g\cdot \omega \cdot \omega \|_{L^\infty H^{s-3}}  + \|\JapD^{-1} \big(\partial e \cdot \tilde{\mathcal F}_{\perp} + \partial g \cdot \tilde{\mathcal F}_{\perp} \big) \|_{L^\infty H^{s-3}} \\
\lesssim & \|\partial g\|_{L^\infty H^{s-2}} \|D\omega\|_{L^\infty H^{s-3}}
+ \|D\partial g\|_{L^\infty H^{s-3}} \|g\|_{L^\infty H^{s-1}} \|\omega\|_{L^\infty H^{s-2}} + \|\partial g\|_{L^\infty H^{s-2}} \|\omega\|^2_{L^\infty H^{s-2}} \\
& \hphantom{\|_{L^\infty H^{s-3}} +\| |D|^{-1} \big(}
+ \|g\|_{L^\infty H^{s-1}} \|\omega\|^2_{L^\infty H^{s-2}} + \big( \|\partial e \|_{L^\infty H^{s-2}} + \|\partial g \|_{L^\infty H^{s-2}} \big) \| \tilde{\mathcal F}_{\perp}\|_{L^\infty H^{s-3}} \\
\lesssim & C_0^2 \mathcal D^2,
\end{align*}
where, in passing to the last line above, we used the estimate \eqref{Bound g and d g} for $g$ and $\partial g$, the bootstrap assumption \eqref{Bootstrap bound} for $\omega$ and $\partial e$, as well as  the estimates \eqref{Bounds F perp easy} for $\tilde{\mathcal F}_\perp$.

\item Similarly, we can estimate (recall that $\f16 + s_1 = s-\f52-2\delta_0$):
\begin{align*}
\| \JapD^{-1} \big( \partial g \cdot & D \omega\big)  \|_{L^2 H^{\f16+s_1}} +\| \JapD^{-1} \big( D\partial g \cdot g\cdot \omega \big) \|_{L^2 H^{\f16+s_1}} + \|  \JapD^{-1} \big( \partial g \cdot \omega \cdot \omega \big)\|_{L^2 H^{\f16+s_1}} \\
 & \hphantom{\|_{L^\infty H^{s-3}} +\|}
+\| g\cdot \omega \cdot \omega \|_{L^2 H^{\f16+s_1}}  + \|\JapD^{-1} \big(\partial e \cdot \tilde{\mathcal F}_{\perp} + \partial g \cdot \tilde{\mathcal F}_{\perp} \big) \|_{L^2 H^{\f16+s_1}} \\
\lesssim & \|\partial g\|_{L^2 H^{s-\f32-2\delta_0}} \|D\omega\|_{L^\infty H^{s-3}}
+ \|D\partial g\|_{L^\infty H^{s-3}} \|g\|_{L^\infty H^{s-1}} \|\omega\|_{L^2 H^{s-\f32-2\delta_0}} + \|\partial g\|_{L^2 H^{s-\f32-2\delta_0}} \|\omega\|^2_{L^\infty H^{s-2}} \\
& \hphantom{\|_{L^\infty H^{s-3}} +\|}
+ \|g\|_{L^\infty H^{s-1}} \|\omega\|_{L^\infty H^{s-2}} \|\omega\|_{L^2 H^{s-\f32-2\delta_0}}+ 
\big( \|\partial e \|_{L^\infty H^{s-2}} + \|\partial g \|_{L^\infty H^{s-2}} \big) \| \tilde{\mathcal F}_{\perp}\|_{L^2 H^{s-\f52-2\delta_0}} \\
\lesssim & C_0^2 \mathcal D^2,
\end{align*}
where, again, in passing to the last line above we made use of the estimate \eqref{Bound g and d g} for $g$ and $\partial g$, the bootstrap assumption \eqref{Bootstrap bound} for $\omega$ and $\partial e$, as well as  the estimates \eqref{Bounds F perp easy} for $\tilde{\mathcal F}_\perp$. Similarly, we also have:
\begin{align*}
\| \JapD^{-1} \big( \partial g \cdot & D \omega\big)  \|_{L^{\f74} W^{s_1,\f73}} +\| \JapD^{-1} \big( D\partial g \cdot g\cdot \omega \big) \|_{L^{\f74} W^{s_1,\f73}} + \|  \JapD^{-1} \big( \partial g \cdot \omega \cdot \omega \big)\|_{L^{\f74} W^{s_1,\f73}} \\
 & \hphantom{\|_{L^\infty H^{s-3}} +\|}
+\| g\cdot \omega \cdot \omega \|_{L^{\f74} W^{s_1,\f73}}  + \|\JapD^{-1} \big(\partial e \cdot \tilde{\mathcal F}_{\perp} + \partial g \cdot \tilde{\mathcal F}_{\perp} \big) \|_{L^{\f74} W^{s_1,\f73}} \\
\lesssim & \|\partial g\|_{L^\infty H^{s-2}} \|D\omega\|_{L^{\f74} W^{s_1,\f73}}
+ \|D\partial g\|_{L^{\f74} W^{s_1,\f73}} \|g\|_{L^\infty H^{s-1}} \|\omega\|_{L^\infty H^{s-2}} + \|\partial g\|_{L^{\f74} W^{1+s_1,\f73}} \|\omega\|^2_{L^\infty H^{s-2}} \\
& \hphantom{\|_{L^\infty H^{s-3}} +\|}
+ \|g\|_{L^\infty H^{s-1}} \|\omega\|_{L^\infty H^{s-2}} \|\omega\|_{L^{\f74} W^{1+s_1,\f73}}+ 
\big( \|\partial e \|_{L^\infty H^{s-2}} + \|\partial g \|_{L^\infty H^{s-2}} \big) \| \tilde{\mathcal F}_{\perp}\|_{L^{\f74} W^{s_1,\f73}} \\
\lesssim & C_0^2 \mathcal D^2,
\end{align*}

\item Similarly:
\begin{align*}
\| \JapD^{-1} \big( \partial g \cdot & D \omega\big)  \|_{L^1 W^{-1+\delta_0, \infty}} +\| \JapD^{-1} \big( D\partial g \cdot g\cdot \omega \big) \|_{L^1 W^{-1+\delta_0, \infty}} + \|  \JapD^{-1} \big( \partial g \cdot \omega \cdot \omega \big)\|_{L^1 W^{-1+\delta_0, \infty}} \\
 & \hphantom{\|_{L^\infty H^{s-3}} +\|}
+\| g\cdot \omega \cdot \omega \|_{L^1 W^{-1+\delta_0, \infty}}  + \|\JapD^{-1} \big(\partial e \cdot \tilde{\mathcal F}_{\perp} + \partial g \cdot \tilde{\mathcal F}_{\perp} \big) \|_{L^1 W^{-1+\delta_0, \infty}} \\
\lesssim & \|\partial g \cdot  D \omega  \|_{L^1 W^{2\delta_0, \f32}} +\|  D\partial g \cdot g\cdot \omega  \|_{L^1 W^{2\delta_0, \f32}} + \|  \partial g \cdot \omega \cdot \omega \|_{L^1 W^{2\delta_0, \f32}} \\
 & \hphantom{\|_{L^\infty H^{s-3}} +\|}
+\| g\cdot \omega \cdot \omega \|_{L^1 W^{2\delta_0, 3}}  + \|\partial e \cdot \tilde{\mathcal F}_{\perp} + \partial g \cdot \tilde{\mathcal F}_{\perp}  \|_{L^1 W^{2\delta_0, \f32}} \\
\lesssim & \|\partial g\|_{L^2 W^{2\delta_0, 6}} \|D\omega\|_{L^2 H^{2\delta_0}}
+ \|D\partial g\|_{L^2 H^{2\delta_0}} \|g\|_{L^\infty H^{\f32+2\delta_0}} \|\omega\|_{L^2 W^{2\delta_0, 6}} + \|\partial g\|_{L^2 W^{2\delta_0, 6}} \|\omega\|_{L^2 W^{2\delta_0, 6}}\|\omega\|_{L^\infty H^{2\delta_0}} \\
& \hphantom{\|_{L^\infty H^{s-3}} +\|}
+ \|g\|_{L^\infty H^{\f32+2\delta_0}} \|\omega\|^2_{L^2 W^{2\delta_0, 6}}+ 
\big( \|\partial e \|_{L^2 W^{2\delta_0, 6}} + \|\partial g \|_{L^2 W^{2\delta_0, 6}} \big) \| \tilde{\mathcal F}_{\perp}\|_{L^2 H^{2\delta_0}} \\
\lesssim & \|\partial g\|_{L^2 H^{\f76+s_1}} \|D\omega\|_{L^2 H^{\f16+s_1}}
+ \|D\partial g\|_{L^2 H^{\f16+s_1}} \|g\|_{L^\infty H^{s-1}} \|\omega\|_{L^2 H^{\f76+s_1}} + \|\partial g\|_{L^2 H^{\f76+s_1}} \|\omega\|_{L^2 H^{\f76+s_1}}\|\omega\|_{L^\infty H^{s-2}} \\
& \hphantom{\|_{L^\infty H^{s-3}} +\|}
+ \|g\|_{L^\infty H^{s-1}} \|\omega\|^2_{L^2 H^{\f76+s_1}}+ 
\big( \|\partial e \|_{L^2 H^{\f76+s_1}} + \|\partial g \|_{L^2 H^{\f76+s_1}} \big) \| \tilde{\mathcal F}_{\perp}\|_{L^2 H^{\f16+s_1}} \\
\lesssim & C_0^2 \mathcal D^2.
\end{align*}
\end{itemize}
Combining these estimates, we obtain:
\begin{equation}\label{Estimate to show E omega 0}
 \|\mathcal E_{\omega_0}\|_{L^\infty H^{s-3}} + \|\mathcal E_{\omega_0}\|_{L^2 H^{\f16 +s_1}} +\|\mathcal E_{\omega_0}\|_{L^{\f74} W^{s_1,\f73}}
+ \|\mathcal E_{\omega_0}\|_{L^1 W^{-1+\delta_0,\infty}} \lesssim C_0^2 \mathcal D^2 \lesssim \mathcal D
\end{equation}
(provided $\epsilon$ in \eqref{D} is sufficiently small in terms of $C_0$). Thus, returning to \eqref{Bound to show omega 0} and using \eqref{Estimate to show E omega 0} and \eqref{Estimate to show F omega 0}, we infer:
\begin{equation}\label{Bound omega 0 final}
\|\partial \omega_0\|_{L^\infty H^{s-3}} + \|\partial \omega_0  \|_{L^2 H^{\f16 +s_1}} +\|\partial\omega_0\|_{L^{\f74} W^{s_1,\f73}}
+ \|\omega_0\|_{L^1 W^{\f12\delta_0,\infty} }
 \lesssim \mathcal D.
\end{equation}

According to Lemma \ref{lem:Parabolic elliptic system}, the spatial components $\bar\omega$ of $\omega$ (i.e. $\omega_{i}^{\bA}$  for $i=1,2,3$)  satisfy the elliptic system
\begin{equation}\label{Equation omega i}
  \Delta_{\bar g} \bar\omega = - \bar\nabla \big( \Delta_{\bar g}|D|^{-1} \omega_0 \big)+  \mathcal E_{\bar\omega},
\end{equation}
where
\begin{align*}
\mathcal E_{\bar\omega} = & D \big( g \cdot R^{\perp}_* \big) 
+ g\cdot   \partial g \cdot R^{\perp}_* 
 +D\big( m(e) \cdot \tilde{\mathcal F}_\perp \big) \\
& + D( \omega \cdot \omega) +  \partial g \cdot \omega \cdot \omega + g \cdot \omega \cdot R_*.
\end{align*}
Using the elliptic estimates from Lemma \ref{lem:Elliptic estimates model} together with the bootstrap assumption \eqref{Bootstrap bound} for $\bar g$
and the bound
\[
\int_0^T \| g-m_0\|_{H^{s-\f12}(\bar\Sigma_\tau)}\cdot \|\bar\partial \bar\omega\|_{W^{-1,\f{6}{1-\delta_0}}(\bar\Sigma_\tau)} \,d\tau \lesssim \|g-m_0\|_{L^2 H^{2+\f16+s_1}} \|\partial \omega\|_{L^2 L^2} \lesssim C_0^2 \mathcal D^2
\]
(obtained from the bootstrap assumption \eqref{Bootstrap bound} and used to control the last term in the right hand side of \eqref{Very low regularity elliptic estimate} for $f=\bar\omega$, $\bar s= s-1$ and $\sigma = -2+\f12\delta_0$), we therefore obtain for any $p\in (1,+\infty)$:
\begin{align*}
\|\bar\partial \bar\omega  \|_{L^\infty H^{s-3}} & + \|\bar\partial \bar\omega\|_{L^2 H^{\f16+s_1}} + \|\bar\partial\bar\omega\|_{L^{\f74} W^{s_1,\f73}}+ \|\bar\partial \bar\omega\|_{L^1 B^{-1+\f12 \delta_0}_{p,1}} \\
\lesssim _{p} &  \|\omega_0 \|_{L^\infty H^{s-2}} + \|\omega_0 \|_{L^2 H^{\f76+s_1}}+\|\omega_0\|_{L^{\f74} W^{1+s_1,\f73}} +  \|\omega_0 \|_{L^1 B^{\f12 \delta_0}_{p,1}}    \nonumber \\
& + \|\mathcal E_{\bar\omega} \|_{L^\infty H^{s-4}} + \|\mathcal E_{\bar\omega} \|_{L^2 H^{-\f56+s_1}} + \|\mathcal E_{\bar\omega}\|_{L^{\f74} W^{-1+s_1,\f73}} \|\mathcal E_{\bar\omega}\|_{L^1 B^{-2+\f12 \delta_0}_{p,1}},  \nonumber
\end{align*}
where $\bar\partial$ denotes differentiation in the spatial coordinates. Fixing a sufficiently large $p$ (in terms of $\delta_0$), the above estimate yields:
\begin{align}\label{Preliminary bound bar omega}
\|\bar\partial \bar\omega  \|_{L^\infty H^{s-3}} & + \|\bar\partial \bar\omega\|_{L^2 H^{\f16+s_1}} + \|\bar\partial\bar\omega\|_{L^{\f74} W^{s_1,\f73}}+\|\bar\partial \bar\omega\|_{L^1 W^{-1+\f14\delta_0,\infty}} \\
\lesssim  &  \|\omega_0 \|_{L^\infty H^{s-2}} + \|\omega_0 \|_{L^2 H^{\f76+s_1}} +\|\omega_0\|_{L^{\f74} W^{-1+s_1,\f73}}+  \|\omega_0 \|_{L^1 W^{\f12 \delta_0,\infty}}    \nonumber \\
& + \|\mathcal E_{\bar\omega} \|_{L^\infty H^{s-4}} + \|\mathcal E_{\bar\omega} \|_{L^2 H^{-\f56+s_1}}+\|\mathcal E_{\bar\omega}\|_{L^{\f74} W^{-1+s_1,\f73}} +  \|\mathcal E_{\bar\omega}\|_{L^1 W^{-2+\f12 \delta_0,\infty}},  \nonumber
\end{align}

For the terms involving $\omega_0$ in the right hand side of \eqref{Preliminary bound bar omega}, the bound \eqref{Bound omega 0 final} implies that
\[
\|\omega_0 \|_{L^\infty H^{s-2}} + \|\omega_0 \|_{L^2 H^{\f76+s_1}} +\|\omega_0\|_{L^{\f74} W^{1+s_1,\f73}}+  \|\omega_0 \|_{L^1 W^{\f12 \delta_0,\infty}} \lesssim \mathcal D.
\]
For the terms involving $\mathcal E_{\bar\omega}$, on the other hand, we will argue similarly as for the proof of \eqref{Estimate to show E omega 0}:
\begin{itemize}
\item Using the functional inequalities of Lemma \ref{lem:Functional inequalities}, we have:
\begin{align*}
\|D \big( g \cdot R^{\perp}_* \big)&  \|_{L^\infty H^{s-4}}
+ \|g\cdot   \partial g \cdot R^{\perp}_*  \|_{L^\infty H^{s-4}} 
 + \| D\big( m(e) \cdot \tilde{\mathcal F}_\perp \big) \|_{L^\infty H^{s-4}}  \\
& + \|D( \omega \cdot \omega)  \|_{L^\infty H^{s-4}} + \| \partial g \cdot \omega \cdot \omega  \|_{L^\infty H^{s-4}} + \| g \cdot \omega \cdot R_* \|_{L^\infty H^{s-4}}\\
\lesssim &
\|g\|_{L^\infty H^{s-1}} \|R^{\perp}_* \|_{L^\infty H^{s-3}}
+ \|g\|_{L^\infty H^{s-1}} \| \partial g\|_{L^\infty H^{s-2}} \|R^{\perp}_*  \|_{L^\infty H^{s-3}}
 + \| m(e)\|_{L^\infty H^{s-1}} \| \tilde{\mathcal F}_\perp \|_{L^\infty H^{s-3}}  \\
& + \|\omega \|^2_{L^\infty H^{s-2}}  + \| \partial g\|_{L^\infty H^{s-2}}\|\omega \|^2_{L^\infty H^{s-2}} + \| g\|_{L^\infty H^{s-1}} \|\omega \|_{L^\infty H^{s-2}} \|R_* \|_{L^\infty H^{s-3}}\\
\lesssim &  C_0^2 \mathcal D^2,
\end{align*}
where, in passing to the last line above, we made use of the estimate \eqref{Bound g and d g} for $g$ and $\partial g$, \eqref{Estimate R perp easy} for $R^\perp$, the bootstrap assumption \eqref{Bootstrap bound} for $\omega$ and $\partial e$, as well as  the estimates \eqref{Bounds F perp easy} for $\tilde{\mathcal F}_\perp$.

\item Similarly, we can estimate (recalling that $-\f56+s_1 = s-\f72-2\delta_0$):
\begin{align*}
\|D \big( g \cdot R^{\perp}_* \big)&  \|_{L^2 H^{-\f56+s_1}}
+ \|g\cdot   \partial g \cdot R^{\perp}_*  \|_{L^2 H^{-\f56+s_1}} 
 + \| D\big( m(e) \cdot \tilde{\mathcal F}_\perp \big) \|_{L^2 H^{-\f56+s_1}}  \\
& + \|D( \omega \cdot \omega)  \|_{L^2 H^{-\f56+s_1}} + \| \partial g \cdot \omega \cdot \omega  \|_{L^2 H^{-\f56+s_1}} + \| g \cdot \omega \cdot R_* \|_{L^2 H^{-\f56+s_1}}\\
\lesssim &
\|g\|_{L^\infty H^{s-1}} \|R^{\perp}_* \|_{L^2 H^{\f16+s_1}}
+ \|g\|_{L^\infty H^{s-1}} \| \partial g\|_{L^2 H^{\f76+s_1}} \|R^{\perp}_*  \|_{L^\infty H^{s-3}}
 + \| m(e)\|_{L^\infty H^{s-1}} \| \tilde{\mathcal F}_\perp \|_{L^2 H^{\f16+s_1}}  \\
& + \|\omega \|_{L^\infty H^{s-2}}  \|\omega \|_{L^2 H^{\f76+s_1}} + \| \partial g\|_{L^2 H^{\f76+s_1}}\|\omega \|^2_{L^\infty H^{s-2}} + \| g\|_{L^\infty H^{s-1}} \|\omega \|_{L^\infty H^{s-2}} \|R_* \|_{L^2 H^{\f16+s_1}}\\
\lesssim &  C_0^2 \mathcal D^2
\end{align*}
and
\begin{align*}
\|D \big( g \cdot R^{\perp}_* \big)&  \|_{L^{\f74} W^{-1+s_1,\f73}}
+ \|g\cdot   \partial g \cdot R^{\perp}_*  \|_{L^{\f74} W^{-1+s_1,\f73}} 
 + \| D\big( m(e) \cdot \tilde{\mathcal F}_\perp \big) \|_{L^{\f74} W^{-1+s_1,\f73}}  \\
& + \|D( \omega \cdot \omega) \|_{L^{\f74} W^{-1+s_1,\f73}} + \| \partial g \cdot \omega \cdot \omega  \|_{L^{\f74} W^{-1+s_1,\f73}} + \| g \cdot \omega \cdot R_* \|_{L^{\f74} W^{-1+s_1,\f73}}\\
\lesssim &
\|g\|_{L^\infty H^{s-1}} \|R^{\perp}_* \|_{L^{\f74} W^{s_1,\f73}}
+ \|g\|_{L^\infty H^{s-1}} \| \partial g\|_{L^{\f74} W^{1+s_1,\f73}} \|R^{\perp}_*  \|_{L^\infty H^{s-3}}
 + \| m(e)\|_{L^\infty H^{s-1}} \| \tilde{\mathcal F}_\perp \|_{L^{\f74} W^{s_1,\f73}}  \\
& + \|\omega \|_{L^\infty H^{s-2}}  \|\omega \|_{L^{\f74} W^{1+s_1,\f73}} + \| \partial g\|_{L^{\f74} W^{1+s_1,\f73}}\|\omega \|^2_{L^\infty H^{s-2}} + \| g\|_{L^\infty H^{s-1}} \|\omega \|_{L^\infty H^{s-2}} \|R_* \|_{L^{\f74} W^{1+s_1,\f73}}\\
\lesssim &  C_0^2 \mathcal D^2
\end{align*}
\item Using the fact that any two smooth functions $h_1, h_2 :\mathbb T^3 \rightarrow \mathbb R$ satisfy
\begin{align*}
\| h_1 \cdot h_2 \|_{W^{-1+\delta_0,\infty}} 
\lesssim & \sum_j \Big(  2^{(-1+\delta_0)j}\| P_{\le j} h_1 \cdot P_j h_2\|_{L^\infty}
+  2^{(-1+\f12\delta_0)j} \| P_{j} h_1 \cdot P_{\le j} h_2\|_{L^\infty}\\
& \hphantom{\sum_j \Big(}
+\sum_{j'>j} \big( 2^{(-1+\delta_0)j}\| P_j \big(P_{j'} h_1 \cdot P_{j'} h_2\big)\|_{L^\infty}\big) \Big)
\\
\lesssim &   \|h_1\|_{L^\infty} \cdot \sum_j 2^{(-1+\delta_0)j}\|P_j h_2\|_{L^\infty}
+  \|h_2\|_{W^{-\f12,\infty}} \sum_j 2^{\delta_0 j}\| P_{j} h_1\|_{L^\infty}\\
& \hphantom{\sum_j \Big(}
+\sum_{j'>j} \big( 2^{\delta_0 j}\| P_j \big(P_{j'} h_1 \cdot P_{j'} h_2\big)\|_{L^3}\big) \Big)
\\
\lesssim & 
\|h_1\|_{W^{2\delta_0,\infty}}\| h_2\|_{W^{-1+\delta_0,\infty}} + \|h_1\|_{W^{1+\delta_0,3}} \|h_2\|_{W^{-1+\delta_0,\infty}}  \\
\lesssim & \|h_1\|_{H^{\f32+3\delta_0}} \|h_2\|_{W^{-1+\delta_0,\infty}},
\end{align*}
we can also bound:
\begin{align*}
\|D \big( g \cdot R^{\perp}_* \big)&  \|_{L^1 W^{-2+\f12\delta_0,\infty}}
+ \|g\cdot   \partial g \cdot R^{\perp}_*  \|_{L^1 W^{-2+\f12\delta_0,\infty}}
 + \| D\big( m(e) \cdot \tilde{\mathcal F}_\perp \big) \|_{L^1 W^{-2+\f12\delta_0,\infty}}  \\
& + \|D( \omega \cdot \omega)\|_{L^1 W^{-2+\f12\delta_0,\infty}} + \| \partial g \cdot \omega \cdot \omega \|_{L^1 W^{-2+\f12\delta_0,\infty}} + \| g \cdot \omega \cdot R_* \|_{L^1 W^{-2+\f12\delta_0,\infty}}\\
\lesssim & \|g \cdot R^{\perp}_* \|_{L^1 W^{-1+\delta_0,\infty}}
+ \|g\cdot   \partial g \cdot R^{\perp}_*  \|_{L^1 W^{\delta_0,\f32}}
 + \| m(e) \cdot \tilde{\mathcal F}_\perp \|_{L^1 W^{-1+\delta_0,\infty}}  \\
& + \| \omega \cdot \omega\|_{L^1 W^{\delta_0,3}} + \| \partial g \cdot \omega \cdot \omega \|_{L^1 W^{\delta_0,\f32}} + \| g \cdot \omega \cdot R_* \|_{L^1 W^{\delta_0,\f32}}\\ 
\lesssim &
\|g\|_{L^\infty H^{s-1}} \|R^{\perp}_* \|_{L^1 W^{-1+\delta_0,\infty}}
+ \|g\|_{L^\infty H^{s-1}} \| \partial g\|_{L^2 W^{\delta_0, 6}} \|R^{\perp}_*  \|_{L^2 H^{\delta_0}}
 + \| m(e)\|_{L^\infty H^{s-1}} \| \tilde{\mathcal F}_\perp \|_{L^1 W^{-1+\delta_0,\infty}}  \\
& + \|\omega \|^2_{L^\infty W^{\delta_0, 6}}  + 
\| \partial g\|_{L^2 W^{\delta_0,6}}\|\omega \|_{L^2 W^{\delta_0,6}}\|\omega\|_{L^\infty W^{\delta_0, 3}} + \| g\|_{L^\infty H^{s-1}} \|\omega \|_{L^2 W^{\delta_0,6}} \|R_* \|_{L^2 H^{\delta_0}}\\
\lesssim & 
\|g\|_{L^\infty H^{s-1}} \|R^{\perp}_* \|_{L^1 W^{-1+\delta_0,\infty}}
+ \|g\|_{L^\infty H^{s-1}} \| \partial g\|_{L^2 H^{\f76+s_1}} \|R^{\perp}_*  \|_{L^2 H^{\f16+s_1}}
 + \| m(e)\|_{L^\infty H^{s-1}} \| \tilde{\mathcal F}_\perp \|_{L^1 W^{-1+\delta_0,\infty}}  \\
& + \|\omega \|^2_{L^\infty H^{\f76+s_1}}  + 
\| \partial g\|_{L^2 H^{\f76+s_1}}\|\omega \|_{L^2 H^{\f76+s_1}}\|\omega\|_{L^\infty H^{s-2}} + \| g\|_{L^\infty H^{s-1}} \|\omega \|_{L^2 H^{\f76+s_1}} \|R_* \|_{L^2 H^{\f16+s_1}}\\
\lesssim &  C_0^2 \mathcal D^2
\end{align*}
(note that, in this case, we also used the $L^1 L^\infty$ bound \eqref{Estimate R perp hard} for $R^\perp_*$).
\end{itemize}
Combining the above bounds, we obtain:
\[
\|\mathcal E_{\bar \omega} \|_{L^\infty H^{s-4}} + \|\mathcal E_{\bar \omega}  \|_{L^2 H^{-\f56+s_1}} +\|\mathcal E_{\bar \omega}\|_{L^{\f74} W^{-1+s_1,\f73}}+  \|\mathcal E_{\bar \omega}  \|_{L^1 W^{-2+\f12 \delta_0,\infty}} \lesssim C_0^2 \mathcal D^2 \lesssim \mathcal D.
\]
Returning to \eqref{Preliminary bound bar omega}, we thus infer:
\begin{equation}\label{Second reliminary bound bar omega}
\|\bar\partial \bar\omega  \|_{L^\infty H^{s-3}} + \|\bar\partial \bar\omega\|_{L^2 H^{\f16+s_1}}+\bar\partial\bar\omega\|_{L^{\f74} W^{s_1,\f73}} + \|\bar\partial \bar\omega\|_{L^1 W^{-1+\f14\delta_0,\infty}} \lesssim \mathcal D.
\end{equation}

Using \eqref{Normal curvature coordinates} to express
\[
\partial_0 \bar\omega = \bar\partial \omega_0 + R^\perp + \omega\cdot \omega,
\]
we can estimate using the bounds \eqref{Bound omega 0 final} for  $\omega_0$, \eqref{Estimate R perp easy} for $R^\perp$ and \eqref{Bootstrap bound} for $\omega\cdot \omega$:
\begin{equation}\label{Dt bar omega final bound}
\|\partial_0 \bar\omega  \|_{L^\infty H^{s-3}} + \|\partial_0 \bar\omega\|_{L^2 H^{\f16+s_1}} +\|\partial_0 \bar\omega\|_{L^{\f74} W^{s_1,\f73}}\lesssim \mathcal D.
\end{equation}
Adding \eqref{Dt bar omega final bound} to \eqref{Second reliminary bound bar omega}, we obtain:
\begin{equation}\label{Third reliminary bound bar omega}
\|\partial \bar\omega  \|_{L^\infty H^{s-3}} + \|\partial \bar\omega\|_{L^2 H^{\f16+s_1}} +\|\partial \bar\omega\|_{L^{\f74} W^{s_1,\f73}}+ \|\bar\partial \bar\omega\|_{L^1 W^{-1+\f14\delta_0,\infty}} \lesssim \mathcal D.
\end{equation}

Using the trivial low frequency estimate
\begin{align*}
\|\bar\omega\|_{L^1 W^{\f14\delta_0,\infty}} \lesssim & 
\|\bar\partial \bar\omega\|_{L^1 W^{-1+\f14\delta_0,\infty}} + \|\bar\omega\|_{L^1 L^2} \\
\lesssim & \|\bar\partial \bar\omega\|_{L^1 W^{-1+\f14\delta_0,\infty}} 
+ \|\partial_0 \bar\omega\|_{L^2 L^2} + \| \bar\omega|_{x^0=0}\|_{L^2} \\
 \lesssim & \mathcal D
\end{align*}
(where, in passing to the last line above, we used the bounds \eqref{Third reliminary bound bar omega} for $\bar\partial \bar\omega$, \eqref{Dt bar omega final bound} for $\partial_0 \bar\omega$ and the bound provided by Proposition \ref{prop: Bounds S0} for $\bar\omega|_{x^0=0}$), we finally obtain from \eqref{Third reliminary bound bar omega}: 
\begin{equation}\label{Final bound bar omega}
\|\partial \bar\omega  \|_{L^\infty H^{s-3}} + \|\partial \bar\omega\|_{L^2 H^{\f16+s_1}} + \|\partial \bar\omega\|_{L^{\f74} W^{s_1,\f73}}+\| \bar\omega\|_{L^1 W^{-1+\f14\delta_0,\infty}} \lesssim \mathcal D.
\end{equation}
Combining \eqref{Bound omega 0 final} and \eqref{Final bound bar omega}, we deduce \eqref{Bound omega d omega}.

\end{proof}

\medskip
\subsection{Completing the proof of Proposition \ref{prop:Bootstrap}} \label{subsec:End of bootstrap}
We will now proceed to collect all the estimates that we established in the previous sections and show that, provided $\epsilon$ is small enough compared to $C_0$, we have
\begin{align}\label{Bootstrap conclusion almost}
\mathcal{Q}_g  + \mathcal{Q}_k  + \mathcal{Q}_\perp + \| Y - Y_0 \|_{L^\infty L^\infty} + \| \partial Y - \partial Y_0 \|_{L^\infty H^{s-1}} + \|\partial^2 Y\|_{L^4 W^{\f1{12},4}}+ \|\partial^2 Y\|_{L^{\f72} L^{\f{14}3}}\le C \mathcal{D} 
\end{align}
for a constant $C>0$ independent of $C_0$; this will imply the improved bootstrap estimate \eqref{Bootstrap conclusion}, thus completing the proof of Proposition \ref{prop:Bootstrap}.

In particular, combining the bounds:
\begin{itemize}
\item \eqref{Energy estimate k}, \eqref{Strichartz estimate k endpoint} and \eqref{Strichartz estimate k L4L4} for $k$, 
\item \eqref{Bound g and d g} \eqref{Estimate d2 g} for $g$ and
\item \eqref{Bound omega d omega} for $\omega$,
\end{itemize}
we obtain:
\begin{equation}\label{Final bound bootstrap almost}
 \mathcal{Q}_k   + \mathcal{Q}_g  + \|\partial \omega \|_{L^{\infty} H^{s-3}} +\| \partial \omega \|_{L^2 H^{ \f16+s_1}} +\|\partial \omega\|_{L^{\f74} W^{s_1,\f73}}+ \| \omega \|_{L^1 W^{\f14\delta_0,\infty}} \le C \mathcal D.
\end{equation}
Thus, in order to establish \eqref{Bootstrap conclusion almost}, it remains to show that
\begin{equation}\label{Almost bound frame final}
\sum_{\bA=4}^{n} \Big(\|e_{\bA} - \delta_{\bA}\|_{L^\infty L^\infty} + \|\partial^2 e_{\bA} \|_{L^\infty H^{s-3}}+\| \partial^2 e_{\bA} \|_{L^2 H^{\f16+s_1}}+\|\partial^2 e_{\bA}\|_{L^{\f74} W^{s_1,\f73}}\Big) \le C \mathcal D
\end{equation}
and
\begin{equation}\label{Bound Y final}
\| Y - Y_0 \|_{L^\infty L^\infty}  + \| \partial Y - \partial Y_0 \|_{L^\infty H^{s-1}} + \|\partial^2 Y\|_{L^4 W^{\f1{12},4}}+\|\partial^2 Y\|_{L^{\f72} L^{\f{14}3}}\le C \mathcal{D}.
\end{equation}
Note that, in view of the low frequency bound
\begin{align*}
\|e_{\bA} - \delta_{\bA}\|_{L^\infty L^\infty} 
\lesssim & \|\partial e_{\bar A}\|_{L^\infty H^{s-2}} + \|e_{\bA} - \delta_{\bA}\|_{L^\infty L^2} \\
\lesssim & \|\partial e_{\bar A}\|_{L^\infty H^{s-2}} + \|\partial_0 e_{\bA}\|_{L^\infty L^2} + \|(e_{\bA} - \delta_{\bA})|_{x^0=0}\|_{L^2} \\
\lesssim &  \|\partial^2 e_{\bar A}\|_{L^\infty H^{s-3}} + \mathcal D
\end{align*}
(the last line being a consequence of the estimates provided by Proposition \ref{prop: Bounds S0} for $e_{\bA}|_{x^0=0}$), the bound \eqref{Almost bound frame final} will follow from the simpler estimate
\begin{equation}\label{Bound frame final}
\sum_{\bA=4}^{n} \Big(\|\partial^2 e_{\bA} \|_{L^\infty H^{s-3}}+\| \partial^2 e_{\bA} \|_{L^2 H^{\f16+s_1}}+\|\partial^2 e_{\bA}\|_{L^{\f74} W^{s_1,\f73}}\Big) \le C \mathcal D
\end{equation}

The bound \eqref{Bound frame final} follows readily from \eqref{Final bound bootstrap almost} and the bootstrap assumption \eqref{Bootstrap bound} for $\partial Y$ using the relation \eqref{Relation e Omega k} for $\partial e$, which can be expressed schematically as
\[
\partial e = \omega \cdot e + g \cdot m(e) \cdot k \cdot \partial Y.
\]

Combining the relations \eqref{Projection Christoffel symbols} and \eqref{Second fundamental form}, we obtain:
\begin{align*}
\partial_\a\partial_\b Y^A & = (\Pi)_B^{\ga}\partial_\a\partial_\b Y^B \partial_{\ga} Y^A +(\Pi^\perp)_{B}^{\bA}\partial_\a\partial_\b Y^B e_{\bA}^A \\
& = \Gamma^{\ga}_{\a\b} \partial_{\ga} Y^A + k^{\bA}_{\a\b} e_{\bA}^A.
\end{align*}
Therefore, using the bound \eqref{Final bound bootstrap almost} for $\Gamma\simeq g\cdot \partial g$ and $k$,  as well as the bootstrap assumption \eqref{Bootstrap bound} for $\partial Y$ and $e$, we obtain:
\begin{align}\label{L4L4 bound Y final}
\| \partial^2 Y\|_{L^4 W^{\f1{12},4}} & \le C (1+C_0 \mathcal D) \Big( \| \partial g\|_{L^4 W^{\f1{12},4}} + \| k \|_{L^4 W^{\f1{12},4}} \Big) \\
& \le C (1+C_0 \mathcal D) \Big( \| \partial g\|^{\f12}_{L^2 W^{\f1{12},6}}\| \partial g\|^{\f12}_{L^\infty W^{\f1{12},3}} + \| k \|_{L^4 W^{\f1{12},4}} \Big) \nonumber \\
& \le C (1+C_0 \mathcal D) \Big( \| \partial g\|^{\f12}_{L^2 H^{1+\f1{12}}}\| \partial g\|^{\f12}_{L^\infty H^{s-2}} + \| k \|_{L^4 W^{\f1{12},4}} \Big) \nonumber \\
& \le C (1+C_0 \mathcal D) \mathcal D \nonumber \\
& \le C \mathcal D
\end{align}
(assuming, as earlier, that $\epsilon$ is sufficiently small in terms of $C_0$ so that $C_0 \mathcal D \ll 1$) and, similarly:
\begin{align}\label{LinftyL2 bound Y final}
\| \partial^2 Y\|_{L^\infty H^{s-1}} & \le C (1+C_0 \mathcal D) \Big( \| \partial g\|_{L^\infty H^{s-1}} + \| k \|_{L^\infty H^{s-1}} \Big) \\
& \le C (1+C_0 \mathcal D) \mathcal D \nonumber \\
& \le C \mathcal D \nonumber
\end{align}
and
\begin{equation}\label{L4-L4+ bound Y final}
\|\partial^2 Y\|_{L^{\f72} L^{\f{14}3}}\le C\mathcal D.
\end{equation}
Moreover, after integrating from the initial data and using the bound \eqref{LinftyL2 bound Y final} for $\partial_0 (\partial Y)$ and the initial data bound \eqref{Bound Y initial} for $(Y-Y_0)|_{x^0=0}$ and $(\partial Y- \partial Y_0)|_{x^0=0}$, 
we can trivially get the following low-frequency estimate:
\begin{align}\label{Low frequency bound Y final}
\| \partial Y- \partial Y_0\|_{L^\infty L^2}&  + \| Y-Y_0\|_{L^\infty L^2} \\
& \le C \cdot \Bigg( \|  (\partial Y- \partial Y_0)|_{x^0=0}\|_{L^\infty L^2} + \| (Y-Y_0)|_{x^0=0}\|_{L^\infty L^2} +\| \partial^2 Y\|_{L^1 L^2}  \Bigg)   \nonumber \\
& \le C \mathcal D.   \nonumber
\end{align}
Combining the bounds \eqref{L4L4 bound Y final}, \eqref{LinftyL2 bound Y final}, \eqref{L4-L4+ bound Y final} and \eqref{Low frequency bound Y final}, we obtain \eqref{Bound Y final}. Thus, the proof of Proposition \ref{prop:Bootstrap} is complete.

\section{Steps 2 and 6: Local existence of a gauge satisfying the balanced condition}\label{sec:Existence gauge}

In this section, we will establish that any smooth immersion $Y:[0,T]\times \mathbb T^3 \rightarrow \mathcal N$ extending an initial data set $(\bY,\bn)$ satisfying \eqref{D} admits (at least locally in time) a gauge $\mathcal G = (x,e)$ obeying the balanced condition and satisfying the initial bounds \eqref{Diffeomorphism initial}--\eqref{Bound slice initial}. In particular, Proposition \ref{prop:Local existence gauge} below will cover both Steps 2 and 6 in our outline of the proof of Theorem \ref{thm:Existence} (see Section \ref{sec:Continuity}). 

\begin{remark*}
The results of this section will apply to immersions $Y$ which are not necessarily solutions of \eqref{Minimal surface equation}.
\end{remark*}

The following auxiliary result will provide a convenient frame $\{e^\prime_{\bA'}\}_{\bA'=4}^n$ for the normal bundle of the immersion $Y$ that will simplify our choice of initial data for the gauge transformation along $\{x^0=0\}$.

\begin{lemma}\label{lem:Initial auxiliary frame}
Let $(\bY,\bn)$ be as in the statement of Theorem \ref{thm:Existence} and let $Y:[0,T]\times \mathbb T^3 \rightarrow \mathcal N$ be a smooth immersion such that 
\begin{equation}\label{Immersion initially agrees}
Y|_{x^0=0} = \bY, \quad \partial_0 Y|_{x^0=0} \in \mathrm{span} \big\{\bn, \partial_i \bY, \, i=1,2,3 \big\}.
\end{equation}
Then, there exists a smooth frame $\{e^\prime_{\bA'}\}_{\bA'=4}^n$ for the normal bundle $NY$ of the immersion satisfying on $\{x^0=0\}$:
\begin{equation}
\sum_{\bA'=4}^n \big\| e^\prime_{\bA'}|_{x^0=0} - \delta_{\bA'} \big\|_{H^{s-1}} \le C \mathcal D,
\end{equation}
where $C>0$ is a constant depending only on $s$; in the above, $\delta_{\bA}$ is the constant vector with components
\[
\delta_{\bA}^A = 
\begin{cases}
1, \quad A=\bA,\\
0, \quad A\neq \bA.
\end{cases}
\]
\end{lemma}

\begin{remark*}
In fact, the proof of Lemma \ref{lem:Initial auxiliary frame} below shows that if the immersion $Y$ satisfies the additional bound
\[
\|Y-Y_0\|_{L^\infty H^{s}} \le \epsilon_1 \ll 1,
\]
 then $e'_{A'}$ can be chosen so that
\[
\| (e'_{\bA'})^A - \delta_{\bA'}^A \|_{L^\infty H^{s-1}} \lesssim \epsilon_1.
\]
\end{remark*}

\begin{proof}
It suffices to construct the frame $\{e'_{\bA'}\}_{\bA'=4}^n$ along the slice $\{x^0=0\}$, since, then, any smooth extension of this frame will satisfy the statement of the lemma. To this end, let us choose $e'_{\bA'}: \mathbb T^3 \rightarrow \mathbb R^{n+1}$ by the relation:
\[
(e'_{\bA'})^A|_{x^0=0} = \delta_{\bA'}^A - m_{\bA' B} \bn^B \bn^A -  m_{\bA' B} \bar g^{ij} \partial_i \bY^B \partial_j \bY^A.
\]
Note that the vector fields $e'_{\bA'}$ defined above are indeed orthogonal to both $\bn$ and $\partial_i \bY$, $i=1,2,3$, since $\bar g_{ij} = m\big( \partial_i \bY, \partial_j \bY \big)$). Thus, \eqref{Immersion initially agrees} implies that $e'_{\bA'}$ belongs to the normal bundle of the immersion $Y$ at $\{x^0=0\}$. Moreover, in view of our assumption that 
\[
\| (\bY,\bn)-(\bY_0,\bn_0)\|_{H^s \times H^{s-1}} \le \mathcal D,
\]
we can readily calculate that 
\begin{equation}\label{Auxiliary frame bound}
\| (e'_{\bA'})^A|_{x^0=0} - \delta_{\bA'}^A \|_{H^{s-1}} \lesssim \mathcal D.
\end{equation}

\end{proof}

The main result of this section is the following:
\begin{proposition}\label{prop:Local existence gauge}
For $s>\f52+\f16$ and $\epsilon>0$ small enough in terms of $s$, let $(\bY,\bn)$ be as in the statement of Theorem \ref{thm:Existence} (i.e. satisfying \eqref{D}) and let $Y:[0,T]\times \mathbb T^3 \rightarrow \mathcal N$ be a smooth immersion such that 
\begin{equation}\label{Immersion initially agrees 2}
Y|_{x^0=0} = \bY, \quad \partial_0 Y|_{x^0=0} \in \mathrm{span} \big\{\bn, \partial_i \bY, \, i=1,2,3 \big\}
\end{equation}
and\footnote{Note that, in the case when $Y$ solves the minimal surface equation\eqref{Minimal surface equation}, the bound \eqref{Initial bound k general Y} follows from our assumption that $(\bY, \bn)$ satisfy \eqref{D}; see Proposition \eqref{prop: Bounds S0}.}
\begin{equation}\label{Initial bound k general Y}
\| \partial^2 Y|_{x^0=0}\|_{H^{s-2}} + \|Y|_{x^0=0}-Y_0|_{x^0=0}\| <\epsilon.
\end{equation}
Then, there exists a $T'>0$, a diffeomorphism $\Psi:[0,T']\times \mathbb T^3 \rightarrow \mathcal V \subset [0,T]\times \mathbb T^3$ fixing $\{0\}\times \mathbb T^3$ and a frame $\{e_{\bA}\}_{\bA=4}^n$ for the normal bundle $N (Y\circ \Psi)$ of the immersion such that:
\begin{itemize}
\item Along the initial slice $\{0\}\times \mathbb T^3$, the diffeomorphism $\Psi$ satisfies
\begin{equation}\label{Diffeomorphism initial prop}
\Psi|_{x^0=0}= \mathrm{\text{Id}}.
\end{equation}

\smallskip
\item The gauge $\mathcal G = (\Psi,e)$ for the immersion $Y:\mathcal V \rightarrow \mathcal N$ satisfies the \textbf{balanced gauge condition} (see Definition \ref{def:Balanced gauge}).

\smallskip
\item Along the initial slice $\{0\}\times \mathbb T^3$, the frame $\{e_{\bA}\}_{\bA=4}^n$ for $N (Y \circ \Psi)$ and the differential $\Psi^*$ of the diffeomorphism $\Psi$ satisfy the bound:
\begin{align}\label{Bound slice initial prop}
\sum_{A=0}^n \sum_{\bA=4}^{n} \Big( \big\| e^A_{\bA}|_{x^0=0} -\delta^A_{\bA}\big\|_{H^{s-1}(\mathbb T^3)} 
& + \big\| \nabla^\perp_0 e^{A}_{\bA}|_{x^0=0} \big\|_{H^{s-2}(\mathbb T^3)}\Big) \\
&  + \sum_{\a,\b=0}^3 \big\| \partial_{\a} \Psi^{\b}|_{x^0=0} - \delta_{\a}^{\b} \big\|_{H^{s-1}(\mathbb T^3)} \le  C \mathcal D  \nonumber 
\end{align}
for some constant $C>0$ depending only on $s$. 
\end{itemize}

Assume, in addition, that the immersion $Y$ equipped with a given frame $\{e'_{\bA'}\}_{\bA'=4}^n$ for the normal bundle $NY$ as in the statement of Lemma \ref{lem:Initial auxiliary frame} satisfies the following bound:
\begin{equation}\label{Additional bound immersion}
\mathcal Q_g + \mathcal Q_k + \mathcal Q_\perp + \|Y-Y_0\|_{L^\infty L^\infty} + \|\partial^2 Y\|_{L^4 W^{\f1{12},4}} +\|\partial^2 Y\|_{L^{\f74} L^{\f73}}+ \|\partial^2 Y\|_{L^\infty H^{s-2}} \le \epsilon_2
\end{equation}
where  is the induced metric on $[0,T]\times \mathbb T^3$ and $k$ the associated second fundamental form (see Definition \ref{def:Norms 1} for the definition of the quantities $\mathcal Q_g, \mathcal Q_k$ and $\mathcal Q_\perp$).
Then, provided $\epsilon_2$ is small enough in terms of $s$, the time of existence $T'\le 1$ of the gauge transformation (appearing in the domain $[0,T']\times \mathbb T^3$ of $(\Psi, U)$) depends only on $s$; moreover, the pair $(\Psi, U)$ satisfies the bound
\begin{align}\label{Bound under extra regularity assumptions}
\|\partial \Psi & \|_{L^\infty L^\infty} + \sum_{l=0}^2 \Big( \|\partial^{1+l} \Psi\|_{L^\infty H^{s-1-l}}  + \|\partial^{1+l} \Psi\|_{L^2 H^{2-l+\f16+s_1}} \\
& \hphantom{ \|_{L^\infty L^\infty} + \sum_{l=0}^2 \Big( \|\partial^{1+l}} 
+\|\partial^{1+l} \Psi\|_{L^{\f74} W^{2-l+s_1,\f73}}+ \|\partial^{1+l} \Psi\|_{L^1 W^{1-l+\f18 \delta_0,\infty}} \Big)\nonumber  \\
& +  \|U-\mathbb I\|_{L^\infty L^\infty} + \sum_{l=1}^2\Big( \|\partial^l U\|_{L^\infty H^{s-1-l}}  + \|\partial^l U\|_{L^2 H^{2-l+\f16+s_1}} \nonumber \\
&\hphantom{+  \|U-\mathbb I\|_{L^\infty L^\infty} + \sum_{l=1}^2\Big( \|\partial^l}
+\|\partial^l U\|_{L^{\f74} W^{2-l+s_1,\f73}}+ \|\partial^l U\|_{L^1 W^{1-l+\f18\delta_0,\infty}} \Big) 
\lesssim_{s} \epsilon_2 \nonumber
\end{align}
(where $\delta_0, s_1$ are chosen as  in Definition \ref{def:Small constants}).
\end{proposition}

\begin{remark*}
The second part of Proposition \ref{prop:Local existence gauge} (i.e.~the statement about immersions $Y$ satisfying the additional bound \eqref{Additional bound immersion}) is not needed in Theorem \ref{thm:Existence}, since both Steps 2 and 6 of the proof do not require a quantitative estimate for the time of existence of the gauge transformation. This part is only relevant for the proof of Theorem  \ref{thm:Uniqueness} on the geometric uniqueness of rough solutions   to \eqref{Minimal surface equation}.
\end{remark*}

\begin{proof}
Before proceeding to the proof of Proposition\ref{prop:Local existence gauge}, we will first need to set up some notations regarding certain operators acting on $C^\infty (\mathbb T^3)$ that will be used in the proof. To this end, we will assume that the parameters $s$ and $\delta_0, s_1, s_0$ are as in Definition \ref{def:Small constants}.

\medskip
\noindent \textbf{Pseudodifferential operators and norms on $\mathbb T^3$.}
Let $\bar g$ be a Riemannian metric on $\mathbb T^3$. For any $f\in C^\infty \big(\mathbb T^3\big)$, we will define
\begin{align}
\mathcal P_{\bar g} f & = f - \frac{1}{\mathrm{Vol}_{\bar g}(\mathbb T^3)}\int_{\mathbb T^3} f \, \mathrm{dVol}_{\bar g}, \label{Projection operator away 0 g}\\
\mathcal P_{\mathbb T^3} f  & = \mathcal P_{g_{\mathbb T^3}} f,
\end{align}
where $g_{\mathbb T^3}$ is the flat metric on $\mathbb T^3$. We wil also denote, as usual, with $\Delta_{\bar g}$ the Laplace--Beltrami operator on $(\mathbb T^3, \bar g)$ acting on functions, i.e.
\[
\Delta_{\bar g} = \frac{1}{\sqrt{\det(\bar g)}} \partial_i \big( \sqrt{\det(\bar g)} \bar g^{ij} \partial_j \big)
\]
(note that $\mathrm{Im}(\Delta_{\bar g}) = \mathrm{Im}(\mathcal P_{\bar g})$). 
We will also define for any $f_1, f_2 \in C^\infty(\mathbb T^3)$:
\[
\mathfrak D^{(-1)} [f] \doteq |D|\Delta_{\bar g}^{-1} (\mathcal P_{\bar g} f\big)
\]
and
\[
\mathfrak I[f_1, f_2]  \doteq \mathfrak D^{(-1)} [f_1 \cdot f_2] -  \big(\mathfrak D^{(-1)} [f_1]\big)\cdot f_2,
\]
where $\bar\Gamma^k_{ij}$ are the Christoffel symbols of $\bar g$ and 
\[
\widehat{|D|f}(\xi) = |\xi| \hat f(\xi).
\]
Note that
\[
\mathcal P_{\mathbb T^3} |D| = |D| \quad \text{and} \quad \mathcal P_{\bar g} \Delta_{\bar g} = \Delta_{\bar g}.
\]

Let us now consider functions defined on $[0,\tau]\times \mathbb T^3$ for some $\tau\in(0,1)$ to be determined later. Given any $p\in (1,+\infty)$, we can define the norm $\|\cdot \|_{\mathfrak V_p}$ for functions on $[0,\tau]\times \mathbb T^3$ by the relation
\begin{equation}\label{V norm}
\|f \|_{\mathfrak V_p} \doteq \|f\|_{L^\infty H^{s-1}} + \|f\|_{L^2 H^{2+s_1}} 
+\|f\|_{L^{\f74} W^{2+s_1,\f73}}+\| f\|_{L^1 B^{1+\f14\delta_0}_{p,1}}
\end{equation}
(where $\delta_0 = \delta_0(s)$ is chosen as  in Definition \ref{def:Small constants}). Note that 
\[
\|f_1 \cdot f_2\|_{\mathfrak V_p} \lesssim_{p,s} \|f_1 \|_{\mathfrak V_p}\|f_2\|_{\mathfrak V_p}.
\]\
Assuming that $\bar g$ satisfies the smallness bound
\begin{equation}\label{Smallness g bar}
c[\bar g] \doteq \big\|\bar g_{ij} - \delta_{ij} \big\|_{L^\infty H^{s-1}} + \big\| \partial_k \bar g_{ij}\big\|_{L^2 H^{s-\f32-2\delta_0}} \le \epsilon_0
\end{equation}
for some constant $\epsilon_0 \in (0,1)$ small enough in terms of $p,s$, the elliptic estimates for $\Delta_{\bar g}$ from Lemma \ref{lem:Elliptic estimates model} imply that the operator $D^{(-1)}$ satisfies
\begin{equation}\label{Bound pseudodifferential operators}
 \|\JapD \mathfrak D^{(-1)} [f]\|_{\mathfrak V_p} \lesssim_{p,s} \|f\|_{\mathfrak V_p}.
\end{equation}
For this reason, we will sometimes write schematically $\mathfrak D^{(-1)} \sim \JapD^{-1}$. 

Let us use the decomposition
\[
\Delta_{\bar g} = \Delta_{\mathbb T^3} + (\bar g^{ij} - \delta^{ij})\partial_i \partial_j - \bar\Gamma_{ij}^k \partial_k
\]
to express
\begin{equation}\label{Operator O}
\mathfrak D^{(-1)} = -|D|^{-1} \mathcal P_{\mathbb T^3} + \JapD^{-1}\mathcal O_{\bar g}
\end{equation}
for some operator $\mathcal O_{\bar g}$ satisfying
\[
\|\mathcal O_{\bar g} f\|_{\mathfrak V_p} \lesssim_{s,p} c[\bar g] \|f\|_{\mathfrak V_p}
\]
(where $c[\bar g]$ was defined in \eqref{Smallness g bar}). Then, using a paraproduct decomposition of $f_1\cdot f_2$ into high-low interactions, together with the trivial observation about the Fourier transform of functions on $\mathbb T^3$:
\begin{align*}
\mathcal F \big[|D|^{-1} (f_1 \cdot f_2)\big] (\xi) - & \mathcal F \big[ \big(|D|^{-1} f_1\big) \cdot f_2\big](\xi) \\
& = \sum_{\xi_2 \in \mathbb Z^3} m(\xi,\xi_2)f_1(\xi-\xi_2) f_2(\xi_2)\quad \text{with} \quad |m(\xi,\xi_2)| \sim \f{|\xi_2|}{|\xi||\xi-\xi_2|}
\end{align*}
we can also estimate (arguing similarly as for the proof of Lemma \ref{lem:Elliptic estimates model}):
\begin{align}\label{Bound commutator operator}
\|\mathfrak I[f_1, f_2]\|_{\mathfrak V_p} \lesssim_{s,p} \big\| & \JapD^{-2}( f_1\cdot |D|f_2)\big\|_{\mathfrak V_p} \\
&+ c[\bar g]\Big(\big\| \JapD^{-1} (f_1 \cdot f_2)\big\|_{\mathfrak V_p}  +\| \JapD^{-1}f_1\|_{\mathfrak V_p} \cdot \|f_2\|_{\mathfrak V_p} \Big). \nonumber 
\end{align}

\medskip
\noindent \textbf{The geometric setup.} Let us now return to the setting of Proposition \ref{prop:Local existence gauge} and let us set
\[
\mathcal M = [0,T]\times \mathbb T^3.
\]
Let $g$ and $k$ be the induced metric and second fundamental form, respectively, for the immersion $Y:\mathcal M \rightarrow \mathcal N$ (see Section \ref{sec:Geometry}). Let also  $\{e^\prime_{\bA'}\}_{\bA'=4}^n$ be the frame for $NY$ provided by Lemma \ref{lem:Initial auxiliary frame}. We will sometimes use the superscript $^{(0)}$ to denote quantities on $\mathcal M$ associated to $g,k$ and $e^\prime$ expressed in the product coordinate system of $\mathcal M = [0,T]\times \mathbb T^3$. By possibly shrinking the time interval $[0,T]$, we will assume without loss of generality that the hypersurfaces 
\[
\Sigma^{(0)}_{t} = \{t\}\times \mathbb T^3 \subset  \mathcal M
\]
are spacelike. We will denote with $(\Gamma^{(0)})^k_{ij}$ the Christoffel symbols of $g$ in the product coordinates, while $\hat n^{(0)}$ will denote the timelike unit normal to $\Sigma^{(0)}$ and $\bar g^{(0)}$ the induced Riemannian metric on $\Sigma^{(0)}$.

For a $\tau>0$ to be determined later, let $\Psi:[0,\tau]\times \mathbb T^3 \rightarrow \mathcal V \subset \mathcal M$. Setting
\[
\bar\Sigma_t = \{t\}\times \mathbb T^3,
\]
we will assume that $\Psi$ satisfies
\[
\Psi \big(\bar\Sigma_0\big)=\Sigma^{(0)}_0
\]
 and
\[
\Psi\big(\bar\Sigma_t \big) \, \text{ is a spacelike hypersurface of }\, (\mathcal M, g) \, \text{ for all }\, t\in [0,\tau].
\]
 Let also $e_{\bA}^A: [0,\tau]\times \mathbb T^3 \rightarrow \mathbb R$ for $\bA=4, \ldots, n$ and $A=0, \ldots, n$ be functions so that the vector fields $\{e_{\bA}\}_{\bA=4}^n$ constitute a frame for the normal bundle $ N (Y\circ \Psi)$. Let $e_{\bA}$ and $e_{\bA'}^{\prime}\circ \Psi$ be related by
\[
e_{\bA} = U_{\bA}^{\bA'} \cdot \big(e_{\bA'}^{\prime}\circ \Psi\big)
\]
for some $U:[0,\tau]\times \mathbb T^3 \rightarrow GL_{n-3}$ (for the rest of this proof, we will be using primed indices $\bA'$ for components associated to the $e_{\bA'}^{\prime A}$ frame). Then, the connection coefficients $\omega$ and $\omega'$ associated to $e_{\bA}$ and $e^\prime_{\bA'}\circ \Psi$, respectively, are related by the transformation formula:
\begin{equation}\label{Transformation connection coefficients}
\partial_\alpha U_{\bA}^{\bA'} = \omega_{\alpha \bA}^{\bB} \cdot U_{\bB}^{\bA'} -  \omega_{\alpha \bB'}^{\prime \bA'} \cdot U_{\bA}^{\bB'}.
\end{equation}

The induced metric and second fundamental form of the immersion $Y\circ\Psi:[0,\tau] \times  \mathbb T^3 \rightarrow \mathcal N$ are
\begin{equation}\label{Transformation metric and k}
\Psi^* g_{\alpha \beta} = \big(g_{\gamma\delta}\circ \Psi\big) \cdot  \partial_\alpha \Psi^\gamma \cdot  \partial_\beta\Psi^\delta   \quad \text{and} \quad   \Psi^* k^{\bA}_{\alpha \beta} = U_{\bA'}^{\bA} \cdot \big( k^{\bA'}_{\gamma\delta}\circ \Psi \big) \cdot \partial_\alpha \Psi^\gamma \cdot  \partial_\beta\Psi^\delta.
\end{equation}
 Considering the spacelike foliation $\Psi(\bar\Sigma_t)$ of $\mathcal V$, the tangent space of $\Psi(\bar\Sigma_t)$ is spanned by $\{\partial_i \Psi\}_{i=1}^3$. If we denote by $\bar g$ the (Riemannian) metric induced by $g$ on $\Psi(\bar\Sigma_t)$ (or, equivalently, by $\Psi_* g$ on $\bar\Sigma_t$), then,
\[
\bar g_{ij} = g(\partial_i \Psi, \partial_j \Psi) = \big(g_{\gamma\delta}\circ \Psi\big) \cdot  \partial_i \Psi^\gamma \cdot  \partial_j \Psi^\delta.
\]
Let $\hat{n}:[0,\tau]$ be the future directed unit normal vector field to $\bar\Sigma_t$ with respect to $\Psi_* g$; we will use the same notation $\hat n$ to also denote the push-forward vector field $\Psi^* \hat n$ which is normal to $\Psi(\bar\Sigma_t)$. Note that the components $\hat n^{\alpha}$ of $\hat n$ can be explicitly expressed as an algebraic function of $g\circ \Psi$ and $\partial \Psi$; for this reason, we will frequently use the notation $\hat n = \hat n[\Psi, \partial\Psi]$. Note that $\hat n =\hat n^{(0)}$ in the special case when $\Psi=\mathrm{Id}$.

 Let $\Pi^{\top}_{\Psi}$ denote the orthogonal projection onto the tangent space of $\bar\Sigma_t$ i.e., for any $Z \in \Gamma (T_{\Psi(\bar\Sigma_t)} \mathcal V)$, $\Pi^{\top}_{\Psi}(Z):\{t\}\times \mathbb T^3 \rightarrow \mathbb R^3$ is defined by
\[
Z - \big(\Pi^{\top}_{\Psi}(Z)\big)^i \partial_i \Psi \, \parallel \, \hat n.
\]
Note that any vector field $Z \in \Gamma (T_{\bar\Sigma_t} \mathcal V)$ can be decomposed as
\[
Z^\alpha = -\big( g_{\beta\gamma} Z^\beta \hat n^\gamma \big) \hat n^{\alpha} \, + \, \big(\Pi^{\top}_{\Psi}\big)^i_\beta Z^\beta \partial_i \Psi^\alpha.
\]
In view of the fact that the Christoffel symbols $\bar\Gamma^k_{ij}$ of $\bar g$ and the second fundamental form $h = -g(\nabla_{\cdot} \hat n, \cdot)$ of $\Psi(\bar\Sigma_s)$ satisfy
\[
\bar\Gamma_{ij}^k = \big(\Pi^{\top}_{\Psi}(\nabla_i \nabla_j \Psi)\big)^k
= \big(\Pi^{\top}_{\Psi}(\partial_i \partial_j \Psi)\big)^k - \big(\Pi^{\top}_{\Psi}(\big( \Gamma^{(0)}\big)_{ij}^{\lambda} \partial_\lambda \Psi)\big)^k
\]
and
\[
 h_{ij} = g_{\beta\gamma} \cdot \big( \nabla_i \nabla_j \Psi^\beta\big) \cdot \hat n^{\gamma}  = g_{\beta\gamma} \partial_i \partial_j \Psi^\beta\hat n^{\gamma} - g_{\beta\gamma} \big(\Gamma^{(0)}\big)^{\lambda}_{ij} \partial_\lambda \Psi^\beta\hat n^{\gamma} ,
\]
the above decomposition yields the following formula for the map $\Psi$:
\[
\partial_i \partial_j \Psi^{\alpha} - \big(\Gamma^{(0)}\big)^{\lambda}_{ij} \partial_\lambda \Psi^\alpha  = -h_{ij} \hat n^\alpha + \bar\Gamma_{ij}^k \partial_k \Psi^\alpha.
\]
As a consequence:
\begin{equation}\label{Formula Laplacian Psi}
\bar g^{ij} \partial_i \partial_j \Psi^{\alpha} = -\tr_{\bar g} h \, \hat n^\alpha + \bar g^{ij} \bar\Gamma_{ij}^k \partial_k \Psi^\alpha + \bar g^{ij} \big(\Gamma^{(0)}\big)^{\lambda}_{ij} \partial_\lambda \Psi^\alpha.
\end{equation}
Adopting the notation of Section \ref{sec:Gauge} regarding the $3+1$ decomposition of the metric $g$ with respect to the foliation $\Psi(\bar\Sigma_t)$, we can also express
\begin{equation}\label{Dt derivative Psi}
\partial_0 \Psi^\alpha = N \hat n^\alpha + \beta^k \partial_k \Psi^\alpha,
\end{equation}
where $N$, $\beta$ are the lapse function and shift vector field, respectively.

\medskip
\noindent \textbf{The initial value problem for $\Psi$.} We will now consider the case when the gauge $\mathcal G = (\Psi,e)$ on $\mathcal V$ satisfies the balanced conditions \eqref{Mean curvature condition}--\eqref{Divergence condition frame}, subject to the initial conditions
\begin{equation}\label{Initial conditions transformation}
\Psi^0(0,x) = 0, \quad \Psi^i (0,x) = x^i \quad \text{and}\quad  U_{\bA}^{\bA'}(0,x)=\delta_{\bA}^{\bA'}, \quad x\in \mathbb T^3.
\end{equation}
The conditions \eqref{Mean curvature condition}--\eqref{Harmonic condition} take the form:
\begin{equation}\label{Mean curvature condition again} 
\begin{cases}
\tr_{\bar g} h   = \Delta_{\bar{g}} |D|^{-1} (N-1) - \f1N \bar{g}^{ij} \tilde{\mathcal{F}}^{\natural}_{ij} +\displaystyle\fint_{(\bar{\Sigma}_{t}, \bar g)}\Big( \tr_{\bar g} h + \f1N \bar{g}^{ij}  \tilde{\mathcal{F}}^{\natural}_{ij}\Big),\\[6pt] 
\displaystyle\fint_{(\bar{\Sigma}_{t},  g_{\mathbb T^3})} N = 1
\end{cases}
\end{equation}
and 
\begin{equation}\label{Harmonic condition again}
\begin{cases}
\bar g^{ij} \bar\Gamma^k_{ij} = - \Delta_{\bar g}|D|^{-1} \beta^k, \\[6pt]
\displaystyle\fint_{(\bar{\Sigma}_{t},  g_{\mathbb T^3})} \beta^k = 0,
\end{cases}
\quad k=1,2,3,
\end{equation}
where 
\[
\tilde{\mathcal{F}}^{\natural}_{ij} = P^{\natural}[m(e)_{\bA\bB}, \Psi_* k^{\bB}_{ij}, \mathcal{T}^{(l)} (\Psi_* k^{\bA}_{0 l})] - \mathbb E \Big[ P^{\natural}[m(e)_{\bA\bB}, \Psi_* k^{\bB}_{ij}, \mathcal{T}^{(l)} (\Psi_* k^{\bA}_{0 l})]\Big|_{x^0=0}\Big].
\]
In view of the relations \eqref{Dt derivative Psi} and \eqref{Formula Laplacian Psi} for $\Psi$, the conditions \eqref{Mean curvature condition again}--\eqref{Harmonic condition again} yield the following identity:
\begin{align*}
\partial_0 \Psi^\alpha
& = (N -1) \hat n^\alpha + \beta^k \partial_k \Psi^\alpha + \hat n^\alpha \\
& = \Big[ \mathfrak D^{(-1)}\Big( \tr_{\bar g} h + \f{1}{N}\bar{g}^{ij} \tilde{\mathcal F^\natural}_{ij}\Big)\Big] \hat n^\alpha  - \Big[ \mathfrak D^{(-1)} \big( \bar g^{ij} \bar\Gamma^k_{ij}\big)\Big] \partial_k \Psi^\alpha   + \hat n^\alpha \\
& = \mathfrak D^{(-1)} \Bigg[ \tr_{\bar g} h \hat n^\alpha - \bar g^{ij} \bar\Gamma^k_{ij} \partial_k \Psi^\alpha \Bigg] 
- \mathfrak I\big[\tr_{\bar g} h,\, \hat n^\alpha\big] +\mathfrak I\big[ \bar g^{ij} \bar\Gamma^k_{ij},\, \partial_k \Psi^\alpha \big]
 +   \Big[ \mathfrak D^{(-1)}\Big(\f{1}{N}\bar{g}^{ij} \tilde{\mathcal F^\natural}_{ij}\Big)\Big] \hat n^\alpha + \hat n^\alpha \\
& = \mathfrak D^{(-1)} \Bigg[ -\bar g^{ij} \partial_i \partial_j \Psi^{\alpha}  +  \bar g^{ij} \big(\Gamma^{(0)}\big)^{\lambda}_{ij} \partial_\lambda \Psi^\alpha \Bigg] \\
& \hphantom{= \mathfrak D^{(-1)} \Bigg[ -\Delta}
- \mathfrak I\big[\tr_{\bar g} h,\, \hat n^\alpha\big] +\mathfrak I\big[ \bar g^{ij} \bar\Gamma^k_{ij},\, \partial_k \Psi^\alpha \big]
 +   \Big[ \mathfrak D^{(-1)}\Big(\f{1}{N}\bar{g}^{ij} \tilde{\mathcal F^\natural}_{ij}\Big)\Big] \hat n^\alpha + \hat n^\alpha \\
& = \mathfrak D^{(-1)} \Bigg[ -\Delta_{\bar g} \Psi^{\alpha} -\bar g^{ij} \Big( \bar\Gamma_{ij}^k \partial_k \Psi^\alpha -\big(\Gamma^{(0)}\big)^{\lambda}_{ij}\partial_\lambda \Psi^\alpha \Big)  \Bigg] \\
& \hphantom{= \mathfrak D^{(-1)} \Bigg[ -\Delta}
- \mathfrak I\big[\tr_{\bar g} h,\, \hat n^\alpha\big] +\mathfrak I\big[ \bar g^{ij} \bar\Gamma^k_{ij},\, \partial_k \Psi^\alpha \big]
 +   \Big[ \mathfrak D^{(-1)}\Big(\f{1}{N}\bar{g}^{ij} \tilde{\mathcal F^\natural}_{ij}\Big)\Big] \hat n^\alpha + \hat n^\alpha \\
& = -|D| \Psi^\alpha -\mathfrak D^{(-1)} \Bigg[\bar g^{ij} \Big( \bar\Gamma_{ij}^k \partial_k \Psi^\alpha -\big(\Gamma^{(0)}\big)^{\lambda}_{ij}\partial_\lambda \Psi^\alpha \Big)  \Bigg] \\
& \hphantom{= \mathfrak D^{(-1)} \Bigg[ -\Delta}
- \mathfrak I\big[\tr_{\bar g} h,\, \hat n^\alpha\big] +\mathfrak I\big[ \bar g^{ij} \bar\Gamma^k_{ij},\, \partial_k \Psi^\alpha \big]
 +   \Big[ \mathfrak D^{(-1)}\Big(\f{1}{N}\bar{g}^{ij} \tilde{\mathcal F^\natural}_{ij}\Big)\Big] \hat n^\alpha + \hat n^\alpha.
\end{align*}
As a result, if we use the following ansatz for $\Psi(x)$, $x\in [0,\tau]\times \mathbb T^3$:
\begin{equation}\label{Decomposition Psi}
\Psi^\alpha (x) = x^\alpha +\JapD^{-1}\Phi^{\alpha}(x)
\end{equation}
for some $\Phi:[0,\tau]\times \mathbb T^3 \rightarrow \mathbb R^{3+1}$, then the gauge conditions \eqref{Mean curvature condition again}--\eqref{Harmonic condition again} together with the initial condition $\Psi|_{x^0=0} = \mathrm{Id}$ are formally equivalent to the statement that $\Phi$ satisfies the following non-linear parabolic initial value problem:
\begin{equation}\label{Parabolic IVP Phi}
\begin{cases}
\partial_0 \Phi^\alpha + |D| \Phi^\alpha = \JapD\mathscr N^\alpha[\Phi, |D| \Phi; g, \tilde{\mathcal F}^\natural],\\
\Phi^\alpha|_{x^0=0} = 0,
\end{cases}
\end{equation}
where $\mathscr N[\Phi, |D| \Phi; g, \tilde{\mathcal F}^\natural]$ takes the following form in terms of $\Psi = x + \JapD^{-1} \Phi$:
\begin{align*}
\mathscr N^\alpha[\Phi, |D| \Phi; g, \tilde{\mathcal F}^\natural]
\doteq & 
 -\mathfrak D^{(-1)} \Bigg[\bar g^{ij} \Big( \bar\Gamma_{ij}^k \partial_k \Psi^\alpha -\big(\Gamma^{(0)}\big)^{\lambda}_{ij}\partial_\lambda \Psi^\alpha \Big)  \Bigg] 
- \mathfrak I\big[\tr_{\bar g} h,\, \hat n^\alpha\big] +\mathfrak I\big[ \bar g^{ij} \bar\Gamma^k_{ij},\, \partial_k \Psi^\alpha \big]\\
 &\hphantom{\sum}
 +   \Big[ \mathfrak D^{(-1)}\Big(\f{1}{N}\bar{g}^{ij} \tilde{\mathcal F}^\natural_{ij}\Big)\Big] \hat n^\alpha + \hat n^\alpha-\delta_0^\alpha.
\end{align*}
Let us list here a schematic expression of the terms on the right hand side above in terms of differential operators applied to $\Psi$ and the normal frame transformation matrix $U_{\bA}^{\bA'}$:
\begin{itemize}
\item The induced Riemannian metric $\bar g$ and the timelike normal $\hat n$ satisfy
\[
\bar g \sim g[\Psi] \cdot \bar\partial \Psi \cdot \bar\partial \Psi \quad \text{and} \quad \hat n = \hat n[\Psi, \partial \Psi]
\]
\item The Christoffel symbols $\bar\Gamma_{ij}^k$ and the second fundamental form $h_{ij}$ satisfy
\[
\bar\Gamma = \mathcal G_1 \big[ g [\Psi], \bar\partial \Psi\big] \cdot  \bar \partial^2 \Psi + (\partial g)[\Psi] \cdot \mathcal G_2 [\partial \Psi], \quad h = \mathcal G_3 \big[g[\Psi], \bar\partial \Psi\big] \cdot  \bar \partial^2 \Psi + (\partial g)[\Psi] \cdot \mathcal  G_4 [\partial \Psi]
\]
\item The background Christoffel symbols $\big(\bar\Gamma^{(0)}\big)_{ij}^\lambda$ satisfy
\[
\big(\bar\Gamma^{(0)}\big)_{ij}^\lambda = (\partial g) [\Psi] 
\] 
\item The term $\tilde{\mathcal F}^\natural$ satisfies
\[
\tilde{\mathcal F}^\natural = \mathbb P^{\natural} \big[ U e^\prime, \, \big( U \cdot k \cdot \bar\partial \Psi)^2 \big) ,\, \JapD^{-1}\big( U \cdot k \cdot \bar\partial \Psi)^2 \big)\big].
\]
\end{itemize}

Let us introduce the following quantity measuring the size of the background tensors $g,k$ and the background normal frame components $e^{\prime A}_{\bA'}$ for any $1\ll p <+\infty$:
\[
\mathcal C[Y] \doteq \|g-m_0\|_{\mathfrak V_p} + \sum_{A=0}^n\sum_{\bA'=4}^n\|e^{\prime A}_{\bA'} - \delta^A_{\bA'}\|_{\mathfrak V_p} + \|k\|_{L^\infty H^{s-2}} + \|k\|_{L^2 W^{-\f12+s_0,+\infty}}, 
\]
where the norm $\mathfrak V_p$ was defined by \eqref{V norm}. Then, assuming a priori that 
\begin{equation}\label{A priori smallness existence}
\mathcal C[Y] + \| \JapD^{-1}\partial \Phi \|_{\mathfrak V_p} + \| U - \mathrm{I}\|_{\mathfrak V_p} < \epsilon_1
\end{equation}
where $\epsilon_1>0$ is a fixed (small) constant depending only on $p,s$, the parabolic estimates of Lemma \ref{lem:Parabolic estimates model} and the estimates \eqref{Bound pseudodifferential operators} and \eqref{Bound commutator operator} for the operators $\mathfrak D^{(-1)}$ and $\mathfrak I[\cdot, \cdot]$ imply that any solution $\Psi$ of the initial value problem \eqref{Parabolic IVP Phi} satisfies the following \emph{a priori} bound:
\begin{equation}\label{A priori bound Phi}
\| \JapD^{-1}  \partial \Phi \|_{\mathfrak V_p} \lesssim_{p,s} \mathcal C[Y] + \epsilon_1 \Big( \big( \| \JapD^{-1}\partial \Phi \|_{\mathfrak V_p} + \| U - \mathrm{I}\|_{\mathfrak V_p}  \Big).
\end{equation}

\medskip
\noindent \textbf{The initial value problem for $U$.}
Let us assume that the transformation matrix $U^{\bA'}_{\bA}$ has been chosen so that the gauge condition \eqref{Divergence condition frame} is satisfied; recall that this condition takes the form
\begin{equation}\label{Divergence condition frame again}
\begin{cases}
\delta_{\bar{g}} \omega_{\bar B}^{\bar A} \doteq \bar{g}^{ij} \bar{\nabla}_i \omega_{j \bar B}^{\bar A}= 
- \Delta_{\bar g}|D|^{-1}  \omega^{\bA}_{0 \bB} + 
  \Bigg( m(e)_{\bB\bC} \tilde{\mathcal{F}}_{\perp}^{\bA\bC} - 
\displaystyle\fint_{\bar\Sigma_{t}}\Big(m(e)_{\bB\bC} \tilde{\mathcal{F}}_{\perp}^{\bA\bC}\Big)\Bigg), \\[6pt]
\displaystyle\fint_{(\bar{\Sigma}_t, g_{\mathbb T^3})} \omega^{\bA}_{0 \bB} = 0,
\end{cases}
\end{equation}
where, in our case, 
\[
\tilde{\mathcal F}_\perp^{\bB\bC} = -  \partial_i\Big( \bar g^{ij}   \mathbb P^{\natural} [\Psi_*g^{\ga\delta}, \Psi_*k^{\bB}_{j\ga}, \mathcal T^{(b)} \Psi_* k^{\bA}_{b\delta}]\Big) + \mathbb E \Bigg[  \partial_i\Big( \bar g^{ij}   \mathbb P^{\natural} [\Psi_*g^{\ga\delta}, \Psi_*k^{\bB}_{j\ga}, \mathcal T^{(b)} \Psi_* k^{\bA}_{b\delta}]\Big)\Big|_{x^0=0}\Bigg].
\] 
Considering the spatial divergence of equation \eqref{Transformation connection coefficients} for the transformation matrix $U_{\bA}^{\bA'}$, \eqref{Divergence condition frame again} implies that
\begin{align}\label{Equation for U}
\Delta_{\bar g} U_{\bA}^{\bA'} =
& \delta_{\bar g} \omega_{\bA}^{\bB} \cdot U_{\bB}^{\bA'} + \bar g^{ij} \omega_{i\bA}^{\bB} \cdot \partial_j U_{\bB}^{\bA'}  -  \delta_{\bar g} \omega_{\bB'}^{\prime\bA'} \cdot U_{\bA}^{\bB'} - \bar g^{ij}\omega_{ i\bB'}^{\prime\bA'} \cdot \partial_jU_{\bA}^{\bB'} \\
= & -U_{\bB}^{\bA'}  \cdot \Delta_{\bar g}|D|^{-1}\mathcal P_{\mathbb T^3}\Big( (U^{-1})^{\bB}_{\bC'} \cdot \partial_0 U_{\bA}^{\bC'} + (U^{-1})^{\bB}_{\bB'} \cdot \omega^{\prime \bB'}_{0 \bC'} \cdot U^{\bC'}_{\bA}  \Big) 
\\
& \hphantom{\sum\sum}
+ U^{\bA'}_{\bB} \cdot \Bigg( m(e)_{\bA\bC} \tilde{\mathcal{F}}_{\perp}^{\bB\bC} - 
\displaystyle\fint_{(\bar\Sigma_{t}, \bar g)}\Big(m(e)_{\bA\bC} \tilde{\mathcal{F}}_{\perp}^{\bB\bC}\Big)\Bigg)  \nonumber
\\
 & \hphantom{\sum\sum}
+ \bar g^{ij} \omega_{i\bA}^{\bB} \cdot \partial_j U_{\bB}^{\bA'}  -  \delta_{\bar g} \omega_{\bB'}^{\prime\bA'} \cdot U_{\bA}^{\bB'} - \bar g^{ij}\omega_{ i\bB'}^{\prime\bA'} \cdot \partial_jU_{\bA}^{\bB'}  \nonumber 
\end{align}
and
\begin{equation}\label{Zero average condition U}
\fint_{(\bar\Sigma_t, g_{\mathbb T^3})} \Big( (U^{-1})^{\bA}_{\bC'} \cdot \partial_0 U_{\bB}^{\bC'} + (U^{-1})^{\bA}_{\bB'} \cdot \omega^{\prime \bB'}_{0 \bC'} \cdot U^{\bC'}_{\bB}  \Big)  = 0 \quad \text{for} \quad t\in [0,\tau],
\end{equation}
where $\omega^\prime_i$ are simply the pull-backs $\Psi_* \omega^\prime_i$ of the connection 1-forms associated to the frame $e^\prime$.

We will use the following ansatz for the transformation matrix $U$:
\begin{equation}\label{Decomposition U}
U^{\bA'}_{\bA}(x^0;\bar x) = \mathbb{I} + \mathbb M^{\bA'}_{\bA}(x^0) + \frac{|D|}{1+|D|} V^{\bA'}_{\bA}(x^0;\bar x)
\end{equation}
(where we used the notation $x=(x^0; \bar x)$, $x^0\in [0,\tau]$, $\bar x \in \mathbb T^3$).  Note that, since $ \mathbb M^{\bA'}_{\bA}(x^0)$ is constant in the spatial directions,
\[
\mathcal P_{\mathbb T^3} U^{\bA'}_{\bA} = \frac{|D|}{1+|D|}  V^{\bA'}_{\bA}. 
\] 
The matrix $\mathbb M$ introduced in \eqref{Decomposition U} will be used to guarantee the zero-average condition \eqref{Zero average condition U} (which is equivalent to $\fint_{\mathbb T^3} \omega_0 =0$).

We can now readily see that the gauge condition \eqref{Divergence condition frame again} together with the initial condition $U|_{x^0=0} = \mathbb{I}$ is formally equivalent to the following initial value problems for $\mathbb M^{\bA'}_{\bA}(x^0)$ and $V^{\bA'}_{\bA}(x^0;\bar x)$:
\begin{equation}\label{IVP V}
\begin{cases}
\partial_0 V + |D| V = \mathcal P_{\mathbb T^3} \Bigg[(1+|D|) \Delta_{\bar g}^{-1} \Big(\mathcal P_{\bar g} \mathscr G [V, \partial V, \mathbb M, \frac{d\mathbb M}{dx^0};\bar g, \omega^\prime, \tilde{\mathcal F}_{\perp}]\Big)\Bigg],\\[6pt]
V|_{x^0=0} = 0
\end{cases}
\end{equation}
and
\begin{equation}\label{ODE M}
\begin{cases}
\frac{d \mathbb M^{\bA'}_{\bA}}{dx^0}(x^0)  = -
\Big(\fint_{(\bar\Sigma_{x^0},g_{\mathbb T^3})} \delta^{\bC'}_{\bA}\omega^{\prime\bB'}_{0\bA'} \Big)+ \mathscr{F}^{\bA'}_{\bA}[V, \mathbb M;\omega_0^\prime](x^0),\\[6pt]
\mathbb M^{\bA'}_{\bA}(0) = 0,
\end{cases}
\end{equation}
where
\begin{align*}
\mathscr G^{\bA'}_{\bA} [V, & \partial V, \mathbb M, \frac{d\mathbb M}{dx^0};\bar g, \omega^\prime, \tilde{\mathcal F}_{\perp}]\\
= & - U^{\bA'}_{\bB}\Delta_{\bar g}|D|^{-1}\mathcal P_{\mathbb T^3}\Big( (U^{-1})^{\bB}_{\bC'} \cdot \frac{d\mathbb M_{\bA}^{\bC'}}{dx^0} + \big( (U^{-1})^{\bB}_{\bC'} - \delta^{\bB}_{\bC'}\big) \partial_0 \big(\frac{|D|}{1+|D|}  V _{\bA}^{\bC'}\big)  + (U^{-1})^{\bB}_{\bB'} \cdot \omega^{\prime \bB'}_{0 \bC'} \cdot U^{\bC'}_{\bA}  \Big) 
\\
& \hphantom{\sum\sum}
- \big(U^{\bA'}_{\bB}- \delta^{\bA'}_{\bB}\big)\Delta_{\bar g}|D|^{-1}\mathcal P_{\mathbb T^3}\Big( (U^{-1})^{\bB}_{\bC'} \cdot \partial_0 U_{\bA}^{\bC'}  + (U^{-1})^{\bB}_{\bB'} \cdot \omega^{\prime \bB'}_{0 \bC'} \cdot U^{\bC'}_{\bA}  \Big) 
\\
& \hphantom{\sum\sum}
+ U^{\bA'}_{\bB} \cdot \Bigg( m(e)_{\bA\bC} \tilde{\mathcal{F}}_{\perp}^{\bB\bC} - 
\displaystyle\fint_{(\bar\Sigma_{t}, \bar g)}\Big(m(e)_{\bA\bC} \tilde{\mathcal{F}}_{\perp}^{\bB\bC}\Big)\Bigg)  \nonumber
\\
 & \hphantom{\sum\sum}
+ \bar g^{ij} \partial_i U_{\bB}^{\bA'}\cdot \Big( (U^{-1})^{\bB}_{\bC'} \cdot \partial_j U_{\bA}^{\bC'} + (U^{-1})^{\bB}_{\bB'} \cdot \omega^{\prime \bB'}_{j \bC'} \cdot U^{\bC'}_{\bA}  \Big)  
\\
 &\hphantom{\sum\sum}
-  \delta_{\bar g} \omega_{\bB'}^{\prime\bA'} \cdot U_{\bA}^{\bB'} - \bar g^{ij}\omega_{ i\bB'}^{\prime\bA'} \cdot \partial_jU_{\bA}^{\bB'}  \nonumber 
\end{align*}
and
\begin{align*}
\mathscr{F}^{\bA'}_{\bA}[V, & \mathbb M;\omega_0^\prime](x^0) \\
& = 
-\Big[\Big(\fint_{(\bar\Sigma_{x^0},g_{\mathbb T^3})}(U^{-1})\Big)^{-1} \Big]^{\bA'}_{\bB} \Bigg[\fint_{(\bar\Sigma_{x^0}, g_{\mathbb T^3})} \Big( (U^{-1})^{\bB}_{\bC'} \cdot \partial_0 (\frac{|D|}{1+|D|} V_{\bA}^{\bC'}) + \big( (U^{-1})^{\bB}_{\bB'} \cdot U^{\bC'}_{\bA}- \delta_{\bB'}^{\bB} \delta_{\bA}^{\bC'}\big)\cdot \omega^{\prime \bB'}_{0 \bC'}   \Big)\Bigg] \\
& \hphantom{\sum\sum}
+ \Bigg(\Big[\Big(\fint_{(\bar\Sigma_{x^0},g_{\mathbb T^3})}(U^{-1})\Big)^{-1} \Big]^{\bA'}_{\bB}-  \delta^{\bA'}_{\bB}\Bigg) \Big(\fint_{(\bar\Sigma_{x^0},g_{\mathbb T^3})} \delta^{\bB}_{\bC'}\delta^{\bC'}_{\bA}\omega^{\prime\bB'}_{0\bC'} \Big).
\end{align*}

Let us assume \emph{a priori} that \eqref{A priori smallness existence} holds for some $\epsilon_1>0$ small enough in terms of $p,s$. Applying the parabolic estimates of Lemma \ref{lem:Parabolic estimates model} for the initial value problem \eqref{IVP V} for $V$ and  a straightforward $L^1_{x^0}$ estimate for the ODE \eqref{ODE M} satisfied by $\mathbb M$, using also:
\begin{itemize}
\item The ansatz \eqref{Decomposition U} for $U$ and \eqref{Decomposition Psi} for $\Psi$ for the terms on the right hand sides of \eqref{IVP V}--\eqref{ODE M},
\item The schematic expression
\[
\Psi^* \omega^\prime = \partial \Psi \cdot \partial e',
\]
for $\omega'$,
\end{itemize}
we readily obtain the following \emph{a priori} bounds for $V,\mathbb M$:
\begin{equation}\label{A priori bound V}
\| \JapD^{-1}  \partial V \|_{\mathfrak V_p} \lesssim_{p,s} \mathcal C[Y] + \epsilon_1 \Big( \big( \| \JapD^{-1}\partial \Phi \|_{\mathfrak V_p} + \|\JapD^{-1}\partial U\|_{\mathfrak V_p}+ \| U - \mathrm{I}\|_{L^\infty L^\infty}  \Big).
\end{equation}
(where $\|\cdot \|_{\mathfrak V_p}$ was defined in \eqref{V norm}) and 
\begin{equation}\label{A priori bound M}
\sup_{x^0\in [0,\tau]} \Big( \Big| \frac{d\mathbb M}{dx^0}(x^0)\Big| + \big| \mathbb M(x^0) \big|\Big) \lesssim_{p,s}  \mathcal C[Y] + \epsilon_1 \Big( \big( \| \JapD^{-1}\partial \Phi \|_{\mathfrak V_p} + \|\JapD^{-1}\partial U\|_{\mathfrak V_p}+ \| U - \mathrm{I}\|_{L^\infty L^\infty}  \Big).
\end{equation}

\medskip
\noindent \textbf{Existence of the gauge transformation $(\Psi, U)$.} Let us consider the combined initial value problem \eqref{Parabolic IVP Phi}, \eqref{IVP V} and \eqref{ODE M} for $(\Phi, V, \mathbb M)$; note that this is a quasilinear parabolic system for $(\Phi,V)$ coupled with an ODE for $\mathbb M$. In view of the a priori bounds \eqref{A priori bound Phi}, \eqref{A priori bound V} and \eqref{A priori bound M}, it can be inferred using standard arguments (e.g.~using an iteration scheme) that, if $T_0$ is chosen so that the immersion $Y$ and the frame $e'$ satisfy
\begin{equation}\label{Smallness assumption for existence}
c[\bar g] < \epsilon_0 \quad \text{and} \quad C[Y] <  \epsilon_0^\prime \quad \text{when restricted to } \, [0,T_0]\times \mathbb T^3
\end{equation}
(for some small constants $\epsilon_0, \epsilon_0^\prime$ depending only on $p,s$), then there exists a unique smooth solution $(\Phi, V, \mathbb M)$ of \eqref{Parabolic IVP Phi}, \eqref{IVP V} and \eqref{ODE M} on a time interval $0\le t \le T'$ with $T'$ depending only on the initial data parameter $\epsilon$ (see \eqref{D}), the background parameter $T_0$ and the precise choice of $\epsilon_0, \epsilon^\prime_0$ above and such that
\begin{equation}\label{Final estimate gauge transformation}
\| \JapD^{-1}\partial \Phi \|_{\mathfrak V_p} + \| V - \mathrm{I}\|_{\mathfrak V_p}
+\sup_{x^0\in [0,\tau]} \Big( \Big| \frac{d\mathbb M}{dx^0}(x^0)\Big| + \big| \mathbb M(x^0) \big|\Big) \lesssim_{p,s}   C[Y] +\epsilon
\end{equation}
(note, in particular, that the above bound implies that the gauge transformation $(\Psi, U)$ defined from $(\Phi, V, \mathbb M)$ through the relations \eqref{Decomposition Psi} and \eqref{Decomposition U} satisfies $\|\partial \Psi\|_{L^\infty L^\infty} \ll 1$, i.e.~the foliation $\Psi(\Sigma_{t})$, $t\in [0,T']$ is strictly spacelike). 

In the case when $Y$ and $e'$ satisfy the additional bound \eqref{Additional bound immersion}, then \eqref{Smallness assumption for existence} is satisfied for $T_0 = T$ (provided $\epsilon_2$ in \eqref{Additional bound immersion} is sufficiently small in terms $p,s$). Thus, in this case, fixing $p$ large enough in terms of $\delta_0(s)$, we infer that the time of existence $T'$ of the gauge transformation depends only on $s$. Moreover, commuting equations  \eqref{Parabolic IVP Phi}, \eqref{IVP V} and \eqref{ODE M}  twice with $\partial$ and repeating the steps that lead to the lower order estimates \eqref{A priori bound Phi}, \eqref{A priori bound V} and \eqref{A priori bound M}, we obtain
\begin{align}\label{Commuted bound existence gauge}
\sum_{l=0}^2 \Big(\|\JapD^{-l} & \partial^{1+l} \Psi  \|_{\mathfrak V_p}  + \|\JapD^{-l}\partial^l (U-\mathbb I)\|_{\mathfrak V_p}\Big)\\
&  \lesssim_{p,s} 
 \sum_{l=0}^2 \Big( \|\JapD^{-l}\partial^l (g-m_0)\|_{\mathfrak V_p} + \|\JapD^{-l}\partial^l (e'_{\bA} - \delta_{\bA})\|_{\mathfrak V_p}\nonumber \\
 & \hphantom{\lesssim_{p,s} 
 \sum_{l=0}^2 \Big( \|(1+|D|}
+ \|\partial^l k\|_{L^\infty H^{s-2-l}} + \|\partial^l k\|_{L^2 W^{-\f12-l+s_0,+\infty}}\Big) . \nonumber 
\end{align}
Therefore, using our assumption that $g$, $k$ and $e'$ satisfy \eqref{Additional bound immersion}, we infer  \eqref{Bound under extra regularity assumptions}.

\end{proof}

\section{The final step: Persistence of regularity}\label{sec:Persistence of regularity}
In this section, we will establish the persistence of regularity in the balanced gauge  for solutions $Y$ to the minimal surface equation \eqref{Minimal surface equation} satisfying the bootstrap assumption of Proposition \ref{prop:Bootstrap}. This result will amount to Step 5 in the outline of the proof of Theorem \ref{thm:Existence} presented in Section \ref{sec:Continuity}. Combined with the results of Sections \ref{sec: Bounds S0}--\ref{sec:Existence gauge}, this will complete the proof of Theorem \ref{thm:Existence}.

The main result of this section will be the following:

\begin{lemma}\label{lem:Regularity}
Let $Y: [0,T)\times \mathbb{T}^3 \rightarrow \mathcal{N}$ be a $C^\infty$ development of a smooth initial data pair $(\bY, \bn)$ as in the statement of Theorem \ref{thm:Existence} and $\{e_{\bA}\}_{\bA=4}^n$ be a $C^\infty$ frame for the normal bundle $NY$. 
Assume also that the following conditions hold:
\begin{itemize}
\item The gauge $\mathcal G = (\text{Id}, e)$ for the immersion $Y$ satisfies the balanced gauge condition.
\item The immersion $Y$ satisfies
 \begin{align}\label{Bootstrap bound again}
 \mathcal{Q}_g  + \mathcal{Q}_k  + \mathcal{Q}_\perp + \| Y - Y_0 \|_{L^\infty L^\infty}
 + \| \partial Y - \partial Y_0 \|_{L^\infty H^{s-1}} \le \epsilon_1 
\end{align}
for some constant $\epsilon_1>0$ sufficiently small in terms of $s$ (the quantities $\mathcal Q_g$, $\mathcal Q_k$ and $\mathcal Q_\perp$ were introduced in Definition \ref{def:Norms 1}).
\end{itemize}
Then, the function $Y$ extends on the whole of  $[0,T]\times \mathbb T^3$ as a $C^\infty$ immersion.
\end{lemma}

\begin{proof}
In view of the fact that that higher order derivatives of $Y$ can be computed in terms of derivatives of $k$, $g$ and $e$ via the formula
\begin{align}\label{Schematic expression d2 Y again}
\partial_\a\partial_\b Y^A & = (\Pi)_B^{\ga}\partial_\a\partial_\b Y^B \partial_{\ga} Y^A +(\Pi^\perp)_{B}^{\bA}\partial_\a\partial_\b Y^B e_{\bA}^A \\
& = \Gamma^{\ga}_{\a\b} \partial_{\ga} Y^A + k^{\bA}_{\a\b} e_{\bA}^A,\nonumber
\end{align}
it suffices to show that the components of $k$, $g$ and $e$ can be extended as $C^\infty$ functions across $\{T\}\times \mathbb T^3$. To this end, let us define for any integer $m \ge 1$ and any $\tau \in [0,T)$:
\begin{align*}
\mathscr B_{(m)}(\tau) \doteq  & \|\partial^{m+1} k\|_{L^\infty_{[0,\tau)}H^{-1+\f14\delta_0}} + \|\partial^{m+1} k\|_{L^{\f73}_{[0,\tau)}W^{-2+\f14\delta_0,14}} \\
& + \|\partial^{m+1} g\|_{L^\infty_{[0,\tau)} H^{\f18\delta_0}} + \|\partial^{m+1} g\|_{L^4_{[0,\tau)}W^{\f18\delta_0,\f{12}5}} \\
& +\|\partial^m \omega\|_{L^\infty_{[0,\tau)} H^{\f18\delta_0}} + \|\partial^m \omega\|_{L^4_{[0,\tau)}W^{\f18\delta_0,\f{12}5}}\\
& +\| \partial^{m+1} e \|_{L^\infty_{[0,\tau)} H^{\f18\delta_0}}, 
\end{align*}
where $L^p_{[0,\tau)} W^{s,q}$ denotes the $L^p W^{s,q}$ mixed Sobolev space over the domain $[0,\tau) \times \mathbb T ^3$ and $\delta_0=\delta_0(s)$ is fixed as in Definition \ref{def:Norms 1}, so that, in particular
\[
0< \delta_0 < \min\big\{ \f1{10}(s - \f52 -\f16), \f1{100}\big\}.
\]
Then, the proof of Lemma \ref{lem:Regularity} will follow once we show that, for any $m\ge 1$:
\begin{equation}\label{Bound to show regularity}
\sup_{\tau \in [0,T)} \mathscr B_{(m)}(\tau) <+\infty.
\end{equation}
We will only establish \eqref{Bound to show regularity} for $m=1$, since the bound for higher $m$ follows by arguing inductively on $m$  and using the same reasoning (together with the inductive bound on $\mathscr B_{(m)}$) after commuting the equations with one additional coordinate derivative at each step of the induction.\footnote{Note that the bound  \eqref{Bound to show regularity} for $m=1$ already implies uniform $L^\infty H^{3+\f18\delta_0}\times L^\infty H^{2+\f18\delta_0}$ bounds for $(Y,\partial Y)$ on $\tau \in (0,T)$. From this, one could alternatively infer the propagation of higher order regularity using already established well-posedness results for equation \eqref{Minimal surface equation}.} By a slight abuse of notation, we will also denote
\[
\mathscr B_{(m)}(0) \doteq  \|\partial^{m+1} k|_{x^0=0}\|_{H^{-1+\f14\delta_0}} 
+ \|\partial^{m+1} g|_{x^0=0}\|_{H^{\f18\delta_0}} 
 +\|\partial^m \omega|_{x^0=0}\|_{H^{\f18\delta_0}}+\| \partial^{m+1} e|_{x^0=0} \|_{H^{\f18\delta_0}}.
\]

In order to establish \eqref{Bound to show regularity} for $m=1$, we will start from the wave equation \eqref{Covariant wave equation k} for $k$, which, when expressed as a wave equation for the components of $k$ in the gauge $\mathcal G$, takes the following schematic form:
\begin{equation}\label{Wave equation k once more}
\square_g k = g \cdot\partial g \cdot \partial k + g\cdot \omega \cdot \partial k+  (g\cdot \partial^2 g + g\cdot \partial g \cdot \partial g) \cdot k + g\cdot (\partial\omega + \omega \cdot \omega)\cdot k.
\end{equation}
Our assumption that the immersion $Y$ satisfies the bound \eqref{Bootstrap bound again} allows us to obtain Strichartz estimates (with a loss of $1/6$ derivatives) for equation \eqref{Wave equation k once more}: Arguing exactly as in the proof of Lemmas \ref{lem:Energy estimates} and \ref{lem:Strichartz estimates}, we obtain for any $\tau \in [0,T)$:
\begin{align}
\| \partial^2 k & \|_{L^\infty_{[0,\tau)} H^{-1+\f14\delta_0}} + \|  \partial^2 k \|_{L^{\f73}_{[0,\tau)} W^{-2+\f14\delta_0,14}}  \nonumber \\
& \lesssim \| \partial k |_{x^0=0}\|_{H^{\f14\delta_0}} + \Big\|   g\cdot \partial g\cdot  \partial k + g\cdot \omega \cdot \partial k+  (g\cdot\partial^2 g + g\cdot \partial g \cdot \partial g) \cdot k + g\cdot (\partial \omega + \omega \cdot \omega)\cdot k  \Big\|_{L^1_{[0,\tau)} H^{\f14\delta_0}}   \nonumber \\
& \lesssim \mathscr B_{(1)}(0) + \|g\|_{L^\infty W^{\f14\delta_0, \infty}} \cdot \Bigg\{\Big( \|\partial g\|_{L^1 W^{\f14\delta_0,\infty}} + \|\omega\|_{L^1W^{\f14\delta_0,\infty}} \Big)
\| \partial k \|_{L^\infty_{[0,\tau)} H^{\f14\delta_0}} \nonumber \\
&\hphantom{\lesssim \epsilon + (1+\|g\|} 
+ \Big( \|\partial^2 g\|_{L^{\f74} W^{\f14\delta_0,\f73}} + \|\partial \omega\|_{L^{\f74} W^{\f14\delta_0,\f73}} + \|\partial g\|_{L^\infty W^{\f14\delta_0,3}} \|\partial g\|_{L^{\f74} W^{\f14\delta_0,\f{21}2}} +\|\omega\|_{L^\infty W^{\f14\delta_0,3}} \|\omega\|_{L^{\f74} W^{\f14\delta_0,\f{21}2}} \Big) \| k \|_{L^{\f73}_{[0,\tau)} W^{\f14\delta_0,14}}
\Bigg\} \nonumber \\
& \lesssim \mathscr B_{(1)}(0) + \epsilon_1 \Big( \| \partial k \|_{L^\infty_{[0,\tau)} H^{\f14\delta_0}} +  \| k \|_{L^{\f73}_{[0,\tau)} W^{\f14\delta_0,14}} \Big) \nonumber  \\
& \lesssim \mathscr B_{(1)}(0) + \epsilon_1 \mathscr B_{(1)}(\tau),   \label{Bound k regularity}
\end{align}
where, in the last line above, we made use of the bound \eqref{Bootstrap bound again} for the coefficients of $k$ and $\partial k$ in the right hand side, as well as the fact that$H^{s-2} \hookrightarrow W^{\f14\delta_0,3}$ and $W^{1+s_1, \f73}\hookrightarrow W^{\f14\delta_0,\f{21}2}$. 

The relations \eqref{Gauss} and \eqref{Ricci} for $R$ and $R^\perp$, respectively, can be schematically expressed as
\[
R = m(e)\cdot  k \cdot k \quad \text{and} \quad R^\perp = g \cdot k \cdot k.
\]
Therefore, we can estimate:
\begin{align*}
\| \partial R \|_{L^\infty_{[0,\tau)} W^{\f14\delta_0,\f65}} 
& \lesssim \| \partial e \cdot k \cdot k \|_{L^\infty_{[0,\tau)} W^{\f14\delta_0,\f65}} + \| m(e) \cdot k \cdot \partial k \|_{L^\infty_{[0,\tau)} W^{\f14\delta_0,\f65}} \\
& \lesssim  \| \partial e\|_{L^\infty_{[0,\tau)} W^{\f14\delta_0,3}}\cdot \| k \|_{L^\infty_{[0,\tau)} W^{\f14\delta_0,3}} \cdot \| k \|_{L^\infty_{[0,\tau)} W^{\f14\delta_0,6}} \\
&\hphantom{ \lesssim  \| \partial e\|_{L^\infty_{[0,\tau)} W^{\f14\delta_0,3}}}
+ \| m(e)\|_{L^\infty_{[0,\tau)} W^{\f14\delta_0,\infty}} \| k\|_{L^\infty_{[0,\tau)} W^{\f14\delta_0,3}} \| \partial k \|_{L^\infty_{[0,\tau)} H^{\f14\delta_0}} \\
& \lesssim \big( \| m(e)\|_{L^\infty W^{\f14\delta_0,\infty}}  + \| \partial e\|_{L^\infty H^{s-2}}  \big) \cdot \| k \|_{L^\infty H^{s-2}}  \cdot  \| \partial k \|_{L^\infty_{[0,\tau)} H^{\f14\delta_0}} \\
& \lesssim  \mathscr B_{(1)}(0) + \epsilon_1 \mathscr B_{(1)}(\tau) 
\end{align*}
(where, in passing to the last line above, we used \eqref{Bootstrap bound again} to control $\big( \| m(e)\|_{L^\infty W^{\f14\delta_0,\infty}}  + \| \partial e\|_{L^\infty H^{s-2}}  \big) \cdot \| k \|_{L^\infty H^{s-2}}$ and \eqref{Bound k regularity} for the term $\| \partial k \|_{L^\infty_{[0,\tau)} L^2} $) and
\begin{align*}
\| \partial R \|_{L^4_{[0,\tau)} W^{\f14\delta_0,\f43}} 
& \lesssim \| \partial e \cdot k \cdot k \|_{L^4_{[0,\tau)} W^{\f14\delta_0,\f43}} + \| m(e) \cdot k \cdot \partial k \|_{L^4_{[0,\tau)} W^{\f14\delta_0,\f43}} \\
& \lesssim  \| \partial e\|_{L^4_{[0,\tau)} W^{\f14\delta_0,4}}\cdot \|k\|_{L^\infty_{[0,\tau)} W^{\f14\delta_0,3}} \cdot\| k \|_{L^\infty_{[0,\tau)} W^{\f14\delta_0,6}}\\
&\hphantom{ \lesssim  \| \partial e\|_{L^4_{[0,\tau)} W^{\f14\delta_0,4}}}
 + \| m(e)\|_{L^\infty_{[0,\tau)} W^{\f14\delta_0,\infty}} \cdot \| k\|_{L^4_{[0,\tau)} W^{\f14\delta_0,4}} \cdot \| \partial k \|_{L^\infty_{[0,\tau)} H^{\f14\delta_0}} \\
& \lesssim  \Big(\big(\| \partial e\|_{L^2 W^{\f14\delta_0,6}} + \|\partial e\|_{L^\infty W^{\f14\delta_0,3}}\big) \|k\|_{L^\infty H^{s-2}}  \Big) \|\partial  k \|_{L^\infty_{[0,\tau)} H^{\f14\delta_0}}\\
&\hphantom{ \lesssim  \| \partial e\|_{L^4_{[0,\tau)} W^{\f14\delta_0,4}}}
+ \Big( \| m(e)\|_{L^\infty W^{\f14\delta_0,\infty}} \| k\|_{L^4 W^{\f14\delta_0,4}} \Big)\| \partial k \|_{L^\infty_{[0,\tau)} H^{\f14\delta_0}} \\
& \lesssim \mathscr B_{(1)}(0) + \epsilon_1 \mathscr B_{(1)}(\tau).
\end{align*}
Similarly,
\[
\| \partial R^{\perp} \|_{L^\infty W^{\f14\delta_0,\f65}}  + \| \partial R^\perp \|_{L^4 W^{\f14\delta_0,\f43}}  \lesssim \mathscr B_{(1)}(0) + \epsilon_1 \mathscr B_{(1)}(\tau)  .
\]
Combining the above bounds, we obtain
\begin{equation}\label{Bounds curvature regularity}
\| \partial R \|_{L^\infty W^{\f14\delta_0,\f65}}   + \| \partial R \|_{L^4 W^{\f14\delta_0,\f43}} + \| \partial R^\perp \|_{L^\infty W^{\f14\delta_0,\f65}}   + \| \partial R^\perp \|_{L^4 W^{\f14\delta_0,\f43}} \lesssim \mathscr B_{(1)}(0) + \epsilon_1 \mathscr B_{(1)}(\tau) .
\end{equation}

The quantities $\tilde{\mathcal F}^\natural$ and $\tilde{\mathcal F}_\perp$ (see \eqref{F natural} and \eqref{F perp}, respectively, for their definition) satisfy the following schematic relations:
\begin{align*}
\tilde{\mathcal F}^{\natural} & = \sum_{\bar j }\sum_{j' \ge \bar j} \Bigg(  P_{\bar j} m(e) \cdot P_{j'} k \cdot P_{j'}(|D|^{-1} k) - \mathbb E \big[  \big( P_{\bar j} m(e) \cdot P_{j'}k \cdot P_{j'}(|D|^{-1} k)\big)|_{x^0=0}\big] \Bigg), \\
\tilde{\mathcal F}_\perp & = \sum_{\bar j }\sum_{j' \ge \bar j} \Bigg( \bar \partial \Big( g \cdot P_{\bar j} m(e) \cdot P_{j'} k \cdot P_{j'}(|D|^{-1} k) \Big)- \mathbb E \big[  \bar \partial \big( g \cdot P_{\bar j} m(e)) \cdot P_{j'}k \cdot P_{j'}(|D|^{-1} k)\big)|_{x^0=0}\big] \Bigg),
\end{align*}
where $\bar\partial$ denotes a derivative in a spatial direction.
Arguing similarly as for the proof of \eqref{Bounds curvature regularity} and using the bound \eqref{Extension bound} for the extension operator $\mathbb E$ and the fact that $j'\ge \bar j$ in the above summands, we can readily show that, for any $\tau \in [0,T)$:
\begin{align*}
\| \partial^2 \tilde{\mathcal F}^{\natural} \|_{L^\infty_{[0,\tau)} W^{\f14\delta_0,\f65}} 
& \lesssim \| \partial^2 e \|_{L^\infty_{[0,\tau)} L^{2} } \| |D|^{-\f12} k \|^2_{L^\infty W^{\f14\delta_0,6}}  \\
& \hphantom{\lesssim \| \partial^2 e  }
+ \| \partial e \|_{L^\infty L^3}  \| \partial |D|^{-1} k \|_{L^\infty W^{\f14\delta_0,3}}\| k \|_{L^\infty_{[0,\tau)} W^{\f14\delta_0,6}} \\
& \hphantom{\lesssim \| \partial^2 e}
+ \| m(e)\|_{L^\infty L^\infty}  \| k\|_{L^\infty W^{\f14\delta_0,3}}  \| \partial^2 |D|^{-1} k\|_{L^\infty_{[0,\tau)} H^{\f14\delta_0}}\\
& \lesssim \mathscr B_{(1)} (0)+ \epsilon_1  \mathscr B_{(1)}(\tau) 
\end{align*}
(where, in passing to the last line above, we made use of the bounds   \eqref{Bootstrap bound again} and \eqref{Bound k regularity}) and
\begin{align*}
\| \partial^2 \tilde{\mathcal F}^{\natural} \|_{L^4_{[0,\tau)} W^{\f14\delta_0,\f43}} 
& \lesssim \| \partial^2 e \|_{L^\infty_{[0,\tau)} L^{2} } \|k\|_{L^4 W^{\f14\delta_0,4}} \| |D|^{-1} k \|_{L^\infty W^{\f14\delta_0,\infty}}  \\
& \hphantom{\lesssim \| \partial^2 e }
+ \| \partial e \|_{L^{\infty} L^{3}}  \| \partial |D|^{-1} k \|_{L^4 W^{\f14\delta_0,4}}\| k \|_{L^{\infty}_{[0,\tau)} W^{\f14\delta_0,6}} \\
& \hphantom{\lesssim \| \partial^2 e }
+ \| m(e)\|_{L^\infty L^\infty}  \| k\|_{L^4 W^{\f14\delta_0,4}}  \| \partial^2 |D|^{-1} k\|_{L^\infty_{[0,\tau)} H^{\f14\delta_0}}\\
& \lesssim \mathscr B_{(1)} (0)+ \epsilon_1  \mathscr B_{(1)}(\tau) 
\end{align*}
Similarly,
\[
\| \partial \tilde{\mathcal F}_\perp\|_{L^\infty_{[0,\tau)} W^{\f14\delta_0,\f65}}+\| \partial \tilde{\mathcal F}_\perp\|_{L^4_{[0,\tau)} W^{\f14\delta_0,\f43}}  \lesssim  \mathscr B_{(1)} (0)+ \epsilon_1  \mathscr B_{(1)}(\tau).
\]
Combining the above bounds, we obtain  for any $\tau \in [0,T)$:
\begin{align}\label{Bounds F reqularity}
 \| \partial^2 \tilde{\mathcal F}^\natural \|_{L^\infty_{[0,\tau)} W^{\f14\delta_0,\f65}}  
+ \| \partial^2 \tilde{\mathcal F}^\natural \|_{L^4_{[0,\tau)} W^{\f14\delta_0,\f43}}  
 + \| \partial \tilde{\mathcal F}_\perp \|_{L^\infty_{[0,\tau)} W^{\f14\delta_0,\f65}}  
+ \| \partial \tilde{\mathcal F}_\perp \|_{L^4_{[0,\tau)} W^{\f14\delta_0,\f43}}   \lesssim    \mathscr B_{(1)} (0)+ \epsilon_1  \mathscr B_{(1)}(\tau). 
\end{align}

The components of the metric $g$ and the connection 1-form $\omega$ satisfy the parabolic-elliptic system \eqref{Equation lapse}--\eqref{Equation omega i}. Therefore, using the bound \eqref{Bounds F reqularity} for $\tilde{\mathcal F}^\natural$ and $\tilde{\mathcal F}_\perp$ and the estimates from Lemmas \ref{lem:Parabolic estimates model}--\ref{lem:Elliptic estimates model}, we can readily obtain following the same line of arguments as in Section \ref{subsec:Closing bootstrap metric}:
\begin{equation}\label{Estimates g regularity}
\|\partial^2 g\|_{L^\infty_{[0,\tau)} H^{\f18\delta_0}} + \|\partial^2 g\|_{L^4_{[0,\tau)}W^{\f18\delta_0,\f{12}5}} \lesssim 
B_{(1)} (0)+ \epsilon_1  \mathscr B_{(1)}(\tau)
\end{equation}
and
\begin{equation}\label{Estimates omega regularity}
\|\partial \omega\|_{L^\infty_{[0,\tau)} H^{\f18\delta_0}} + \|\partial \omega\|_{L^4_{[0,\tau)}W^{\f18\delta_0,\f{12}5}} \lesssim
 B_{(1)} (0)+ \epsilon_1  \mathscr B_{(1)}(\tau)
\end{equation}
(note the $\f18\delta_0$-loss compared to the analogous estimates for $R$ and $R^\perp$ in \eqref{Bounds curvature regularity}).

In view of the relation \eqref{Relation e Omega k} between $\partial e$, $\omega$ and $k$ and the assumption \eqref{Bootstrap bound again}, we can estimate
\begin{align}\label{Bound e regularity infty}
\| \partial^2 e \|_{L^\infty_{[0,\tau)} H^{\f18\delta_0}} 
\lesssim  &  \| e \|_{L^\infty W^{\f18\delta_0,\infty} }\| \partial \omega \|_{L^\infty_{[0,\tau)} H^{\f18\delta_0}} + \|\partial e\|_{L^\infty W^{\f18\delta_0,3} } \|\omega\|_{L^\infty_{[0,\tau)} W^{\f18\delta_0,6}} \\
& + \| g \cdot m(e) \cdot \partial Y \|_{L^\infty W^{\f18\delta_0,\infty}} \| \partial k \|_{L^\infty_{[0,\tau)} H^{\f18\delta_0}} 
+ \| m(e) \cdot \partial Y \|_{L^\infty W^{\f18\delta_0,\infty}} \|\partial g \|_{L^\infty W^{\f18\delta_0,3}} \|k\|_{L^\infty_{[0,\tau)} W^{\f18\delta_0,6}} \nonumber \\
& + \|g \cdot m(e)\|_{L^\infty W^{\f18\delta_0,\infty}} \|\partial^2 Y\|_{L^\infty W^{\f18\delta_0,3}} \|k\|_{L^\infty_{[0,\tau)} W^{\f18\delta_0,6}}\nonumber\\
\lesssim &  B_{(1)} (0)+ \epsilon_1  \mathscr B_{(1)}(\tau)
\end{align}
(where, in the last line above, we made use of \eqref{Bootstrap bound again}, \eqref{Bound k regularity} and \eqref{Estimates omega regularity}, as well as the expression \eqref{Schematic expression d2 Y again} for $\partial^2 Y$).

Combining the bounds \eqref{Bound k regularity}, \eqref{Estimates g regularity}, \eqref{Estimates omega regularity} and \eqref{Bound e regularity infty}, we obtain for any $\tau \in [0,T)$:
\[
\mathscr B_{(1)} (\tau) \lesssim \mathscr B_{(1)} (0) + \epsilon_1 \mathscr B_{(1)}(\tau) ,
\]
from which we can immediately infer that 
\[
\sup_{\tau \in [0,T)} \mathscr B_{(1)}(\tau) <+\infty,
\]
 provided $\epsilon_1$ is smaller than some constant depending only on $s$. 

The analogous estimate for $ \mathscr B_{(m)}(\tau) $ for $m>1$ follows in a similar way, after commuting the equations for $k$, $g$ and $\omega$ with coordinate derivatives and arguing inductively on $m$.

\end{proof}

\section{Geometric uniqueness: Proof of Theorem \ref{thm:Uniqueness}}\label{sec:Uniqueness}

In this section, we will establish the geometric uniqueness of the solutions constructed in Theorem \ref{thm:Existence} within the class of (not necessarily smooth) developments satisfying the bounds \eqref{Bound Y uniqueness}--\eqref{Bound omega uniqueness}. 

Let $Y:[0,T]\times \mathbb T^3 \rightarrow \mathcal N$ and $Y': [0,T']\times \mathbb T^3 \rightarrow \mathcal N$ be two immersions as in the statement of Theorem \ref{thm:Uniqueness}. The bounds \eqref{Bound Y uniqueness}--\eqref{Bound omega uniqueness} for $Y,Y'$ allow us to apply Proposition \ref{prop:Local existence gauge} for both $Y$ and $Y'$, with the given normal frames $e$ and $e'$ already satisfying the conditions of Lemma \ref{lem:Initial auxiliary frame} (in view of the bound \eqref{Bound omega uniqueness} for $e$ and $e'$).\footnote{Even though the statement of Proposition \ref{prop:Local existence gauge} assumes a smooth immersion, they do not assume that said immersion is a solution of \eqref{Minimal surface equation}; therefore, in our case, we can obtain the result of Proposition \ref{prop:Local existence gauge} for $Y$ and $Y'$ by applying the same proposition to a sequence of smooth approximations of $Y$ and $Y'$ in the topology defined by the bounds \eqref{Bound Y uniqueness}--\eqref{Bound omega uniqueness}.} Thus, we infer that there exist open neighborhoods $\mathcal U \subset [0,T]$ and $\mathcal U' \subset [0,T']$ of $\{0\}\times \mathbb T^3$ and gauges $\mathcal G = (\Psi,\tilde e)$ for $Y|_{\mathcal U}$ and $\mathcal G' = (\Psi',\tilde e')$ for $Y'|_{\mathcal U'}$ such that:
\begin{itemize}
\item The gauges $\mathcal G$ and $\mathcal G'$ satisfy the balanced condition. 
\item In view of the bounds \eqref{Bound Y uniqueness}--\eqref{Bound omega uniqueness}, Proposition \ref{prop:Local existence gauge} implies that the gauge $\mathcal G$ satisfies
\begin{align}\label{Bound under extra regularity assumptions uniqueness}
 \|\partial \Psi & \|_{L^\infty L^\infty} +  \sum_{l=0}^2\Big( \|\partial^{1+l} \Psi  \|_{L^\infty H^{s-1-l}}  + \|\partial^{1+l} \Psi\|_{L^2 H^{\f16+2-l+s_1}}\\
&\hphantom{\|_{L^\infty L^\infty} +  \sum_{l=0}^2\Big( \|\partial^{1+l}}
+\|\partial^{1+l} \Psi\|_{L^{\f74} W^{2-l+s_1, \f73}}+ \|\partial^{1+l} \Psi\|_{L^1 W^{1-l+\f18 \delta_0,\infty}} \nonumber \\
& + \sum_{\bA=4}^n \sum_{l=1}^2\Big( \|\partial^l (\tilde e_{\bA} - \delta_{\bA})\|_{L^\infty H^{s-1-l}}  + \|\partial^l(\tilde e_{\bA} - \delta_{\bA})\|_{L^2 H^{2-l+\f16+s_1}} \nonumber\\
&\hphantom{+ \sum_{\bA=4}^n \sum_{l=1}^2 \|\partial^l}
 + \|\partial^l(\tilde e_{\bA} - \delta_{\bA})\|_{L^{\f74} W^{2-l+s_1,\f73}}+ \|\partial^l (\tilde e_{\bA} - \delta_{\bA})\|_{L^1 W^{1-l+\f18\delta_0,\infty}}\Big)
\lesssim \epsilon \nonumber
\end{align}
and similarly for $\mathcal G'$.
\end{itemize}
Therefore, Theorem \ref{thm:Uniqueness} will follow as a corollary of the following result:

\begin{proposition}\label{prop:Uniqueness parabolic}
Let $s>\f52+\f16$ and let $Y^{(1)}:[0,T^{(1)}]\times \mathbb T^3$ and $Y^{(2)}:[0,T^{(2)}]\times \mathbb T^3$ be two solutions of the minimal surface equation \eqref{Minimal surface equation} which are developments of the same initial data set $(\bY,\bn)$ as in the statement of Theorem \ref{thm:Uniqueness}. Let also $\{e^{(1)}_{\bA}\}_{\bA=4}^n$ and $\{e^{(2)}_{\bA}\}_{\bA=4}^n$ be frames for the normal bundles of $Y^{(1)}$ and $Y^{(2)}$, respectively, such that the following conditions are satisfied:
\begin{enumerate}
\item The gauges $\mathcal G^{(1)} = \{\mathrm{\text{Id}}, e^{(1)}\}$ and $\mathcal G^{(2)} = \{\mathrm{\text{Id}}, e^{(2)}\}$ for $Y^{(1)}$ and $Y^{(2)}$, respectively, satisfy the balanced condition (see Definition \ref{def:Balanced gauge}).
\item Along the initial slice $\{0\}\times \mathbb T^3$, we have $e^{(1)}_{\bA} = e^{(2)}_{\bA}$ for all $\bA=4, \ldots, n$.
\item The immersions $Y^{(1)}$ and $Y^{(2)}$ satisfy the bounds \eqref{Bound Y uniqueness}--\eqref{Bound omega uniqueness} for some $\epsilon>0$ small enough in terms of $s$.
\end{enumerate}
Then $Y^{(1)}=Y^{(2)}$ and $e^{(1)}=e^{(2)}$ on $[0,\min\{T^{(1)},T^{(2)}\}]\times \mathbb T^3$.
\end{proposition}

The rest of this section we will be devoted to the proof of Proposition \ref{prop:Uniqueness parabolic}.

\bigskip
\noindent \textbf{Remark 1.}
For the rest of this section, we will adopt some notational conventions in addition to the ones outlined in the previous sections:

We will use the superscript $\dot{~}$ to denote the difference between quantities associated to the immersions $Y^{(1)}$ and $Y^{(2)}$; for instance, we will denote
\[
 \dot k = k^{(1)}-k^{(2)}, \quad \dot g = g^{(1)}-g^{(2)}, \quad \dot\omega = \omega^{(1)}-\omega^{(2)}, \quad \text{etc.}
\]
We will also use the notation $Y, g, k, \omega$~(i.e.~without the superscripts $\,^{(1)}, \,^{(2)}$) to denote either of $Y^{(1)}, g^{(1)}, k^{(1)}, \omega^{(1)}$ or $Y^{(2)}, g^{(2)}, k^{(2)}, \omega^{(2)}$ when no confusion arises (and similarly for other objects associated to the immersions $Y^{(1)}$ and $Y^{(2)}$. 
We will also set 
\[
T \doteq \min\big\{ T^{(1)}, T^{(2)} \big\}.
\]

\medskip
\noindent \textbf{Remark 2.} 
In view of our assumption that $Y^{(1)}$ and $Y^{(2)}$ are developments of the same initial data pair $(\bY, \bn)$ at $\{x^0=0\}$ and $e|_{x^0=0} = e'|_{x^0=0}$, we can calculate using the minimal surface equation \eqref{Minimal surface equation} and the fact that both  $\mathcal G^{(1)} = \{\mathrm{\text{Id}}, e^{(1)}\}$ and $\mathcal G^{(2)} = \{\mathrm{\text{Id}}, e^{(2)}\}$  satisfy the balanced gauge condition:
\begin{align}\label{Same initial data}
\partial^j \dot Y|_{x^0=0} = 0 \, \text{for } j=1,2, \quad \partial^l \dot k|_{x^0=0} = 0 \, \text{for } l=0,1,2, \\
\partial^l \dot g|_{x^0=0} = 0 \, \text{for } l=0,1,2, \quad \partial^j \dot \omega|_{x^0=0} = 0  \, \text{for } j=1,2  \nonumber
\end{align}
(calculating the above quantities in terms of $(\bY, \bn)$ can be done using the process carried out in the proof of Proposition \ref{prop: Bounds S0}).

\subsection{Equations satisfied by $(k,g,\omega)$: Revisited}
In this section, we will  carefully derive the frequency projected versions of the equations satisfied by the tensor fields $(k, g, \omega)$ and $(\dot k, \dot g, \dot \omega)$. 

For the tensor fields $(k, g, \omega)$, we will perform this process in a different way than in the proof of Theorem \ref{thm:Existence}, since we will need to then estimate the differences $(\dot k , \dot g, \dot\omega)$ in spaces of lower regularity compared to $(k, g, \omega)$ and, thus, the high-high frequency interactions in the equations for  $(\dot k , \dot g, \dot\omega)$ will pose a bigger hurdle. In order to overcome the issues related to the high-high interactions, we will start with the first order equations implied by \eqref{Minimal surface equation} for $k$ and the balanced gauge condition (see Definition \ref{def:Balanced gauge}) for $(g,\omega)$, apply a Littlewood--Paley projection and \emph{then} differentiating in order to end up with the second order equations (as opposed to applying the frequency projection on the second order equation directly, as we did, for instance, in deriving equation \eqref{Wave equation projection} for $P_j k$).

Recall that $k$ satisfies the minimal surface equation \eqref{Minimal surface equation}, namely
\begin{equation*}
g^{\a\b} k^{\bA}_{\a\b} = 0,  
\end{equation*}
together with the Codazzi equations
\[
\partial_\a k^{\bA}_{\b\ga} = \partial_\b k^{\bA}_{\a\ga} + (g \cdot \partial g + \omega) \cdot k.
\]
For any $j\in \mathbb N$, we therefore infer that, schematically:
\begin{equation}\label{Minimal surface equation j}
P_{\le j} g^{\a\b} \cdot P_j k^{\bA}_{\a\b} = P_j g \cdot P_{\le j} k + P_j \Big( \sum_{j'>j} P_{j'} g \cdot P_{j'}k\Big)
\end{equation}
and 
\begin{equation}\label{Codazzi j}
\partial_\a P_j k^{\bA}_{\b\ga} = \partial_\b P_j k^{\bA}_{\a\ga} + P_j \big((g \cdot \partial g + \omega) \cdot k\big).
\end{equation}
Differentiating \eqref{Codazzi j} and contracting with $P_{\le j} g$, we therefore infer
\begin{align*}
P_{\le j} g^{\a\delta} \partial_\a \partial_{\delta}  P_j k^{\bA}_{\b\ga} 
& \stackrel{\hphantom{\eqref{Codazzi j}}}{=}
 P_{\le j} g^{\a\delta} \partial_{\delta} \partial_\b P_j k^{\bA}_{\a\ga} + P_{\le j} g \cdot \partial \Big(P_j \big((g \cdot \partial g + \omega) \cdot k\big)\Big) 
\\
& \stackrel{\hphantom{\eqref{Codazzi j}}}{=}
 \partial_\b \big(P_{\le j} g^{\a\delta} \partial_{\delta} P_j k^{\bA}_{\a\ga} \big) + (\partial P_{\le j} g)\cdot \partial P_j k + P_{\le j} g \cdot \partial \Big(P_j \big((g \cdot \partial g + \omega) \cdot k\big)\Big) 
\\
& \stackrel{\eqref{Codazzi j}}{=}
 \partial_\b \big(P_{\le j} g^{\a\delta} \partial_{\ga} P_j k^{\bA}_{\a\delta} \big) + (\partial P_{\le j} g)\cdot \partial P_j k + P_{\le j} g \cdot \partial \Big(P_j \big((g \cdot \partial g + \omega) \cdot k\big)\Big) 
\\
& \stackrel{\hphantom{\eqref{Codazzi j}}}{=}
 \partial_\b  \partial_{\ga}\big(P_{\le j} g^{\a\delta} P_j k^{\bA}_{\a\delta} \big)
+\partial \big( \partial P_{\le j} g \cdot P_j k\big) + (\partial P_{\le j} g)\cdot \partial P_j k + P_{\le j} g \cdot \partial \Big(P_j \big((g \cdot \partial g + \omega) \cdot k\big)\Big) 
\\
& \stackrel{\eqref{Minimal surface equation j}}{=}
\partial^2 \big(P_{j} g P_{\le j} k\big)+ \partial^2 \Big[P_j\Big(\sum_{j'>j} P_{j'} g \cdot P_{j'} k \Big) \Big]
+\partial \big( \partial P_{\le j} g \cdot P_j k\big) \\
& \hphantom{=\sum\sum}+ (\partial P_{\le j} g)\cdot \partial P_j k + P_{\le j} g \cdot \partial \Big(P_j \big((g \cdot \partial g + \omega) \cdot k\big)\Big),
\end{align*}
from which we infer that, using the shorthand notation $f_j$ for $P_j f$: 
\begin{align}\label{Wave equation k j careful}
\square_{g_{\le j}} k_j = & g_{\le j} \cdot \partial g_{\le j} \cdot \partial k_{j}   
+ \partial^2 \big(g_j \cdot k_{\le j}\big)+ \partial^2 \Big[P_j\Big(\sum_{j'>j} g_{j'} \cdot k_{j'} \Big) \Big]
 \\
& \hphantom{=\sum\sum}
+\partial \big( \partial g_{\le j} \cdot k_{j}\big) + g_{\le j} \cdot \partial \Big(P_j \big((g \cdot \partial g + \omega) \cdot k\big)\Big)
\nonumber\\
= & g_{\le j} \cdot \partial g_{\le j} \cdot \partial k_{j}   
+ \partial^2 \big(g_j \cdot k_{\le j}\big)+ \partial \big( \partial g_{\le j} \cdot k_{j}\big) 
+ \partial(g_{\le j}) \cdot P_j \big((g \cdot \partial g + \omega) \cdot k\big)
 \nonumber \\
& \hphantom{=\sum\sum}
+ \partial^2 \Big[P_j\Big(\sum_{j'>j} g_{j'} \cdot k_{j'} \Big) \Big] + \partial \Big(g_{\le j} \cdot P_j \big((g \cdot \partial g + \omega) \cdot k\big)\Big)
\nonumber\\
= & g_{\le j} \cdot \partial g_{\le j} \cdot \partial k_{j}   
+ \partial \bar\partial \big(g_j \cdot k_{\le j}\big)+ \bar \partial \big( \partial g_{\le j} \cdot k_{j}\big) 
+ \partial(g_{\le j}) \cdot P_j \big((g \cdot \partial g + \omega) \cdot k\big)
 \nonumber \\
& \hphantom{=\sum\sum}
+ \bar\partial\partial \Big[P_j\Big(\sum_{j'>j} g_{j'} \cdot k_{j'} \Big) \Big] + \bar\partial \Big(g_{\le j} \cdot P_j \big((g \cdot \partial g + \omega) \cdot k\big)\Big)
\nonumber\\
& \hphantom{=}
+ \partial_0 \Big\{\partial \Big[P_j\Big(\sum_{j'>j} g_{j'} \cdot k_{j'} \Big) \Big] + g_{\le j} \cdot P_j \big((g \cdot \partial g + \omega) \cdot k\big)+ \partial\big(g_j \cdot k_{\le j}\big)+ \partial g_{\le j} \cdot k_{j} \Big\},
\nonumber
\end{align}
where $\square_{g_{\le j}}$ is applied component-wise to $k_j$ on the left hand side above and $\bar\partial \in \{\partial_1, \partial_2, \partial_3\}$ denotes a derivative in a spatial direction. Note that, in \eqref{Wave equation k j careful}, all the top order high-high interaction terms have a total derivative outside of the corresponding product (as opposed to \eqref{Wave equation projection}); this is a consequence of the fact that, in deriving  \eqref{Wave equation k j careful}, we applied the Littlewood--Paley projection before the differentiation of  \eqref{Codazzi}.

Let us define $\mathcal F^{(j)}$ as the term inside the $\partial_0$ derivative in the last line of \eqref{Wave equation k j careful}, i.e.~$\mathcal F^{(j)}$ is of the form
\[
\mathcal F^{(j)} = \partial \Big[P_j\Big(\sum_{j'>j} g_{j'} \cdot k_{j'} \Big) \Big] + g_{\le j} \cdot P_j \big((g \cdot \partial g + \omega) \cdot k\big) + \partial\big(g_j \cdot k_{\le j}\big)+ \partial g_{\le j} \cdot k_{j} .
\]

 For each $j\in \mathbb N$, we will decompose $k_j$ as
\begin{equation}\label{Decomposition k j}
k_j = \partial_0 k^{(-1, j)}+k^{(0,j)},
\end{equation}
by the condition that $k^{(-1, j)}$ solves 
\begin{equation}\label{First system k -1}
\begin{cases}
(g_{\le j})^{\alpha\beta} \partial_\alpha \partial_\beta k^{(-1,j)} & = \mathcal F^{(j)} \\
& = \partial \Big[P_j\Big(\sum_{j'>j} g_{j'} \cdot k_{j'} \Big) \Big] + g_{\le j} \cdot P_j \big((g \cdot \partial g + \omega) \cdot k\big) \\
& \hphantom{=}+ \partial\big(g_j \cdot k_{\le j}\big)+ \partial g_{\le j} \cdot k_{j},\\[10pt]
\Big(k^{(-1,j)}|_{x^0=0}, & \hspace{-3em} \partial_0 k^{(-1,j)}|_{x^0=0} \Big)  = \big(0,0\big).
\end{cases}
\end{equation}
This implies, in view of the expression
\begin{equation}\label{Expansion wave operator}
\square_{g_{\le j}} = (g_{\le j})^{\alpha \beta} \partial_\alpha \partial_\beta - (g_{\le j})^{\alpha \beta}\Gamma[g_{\le j}]_{\alpha\beta}^\gamma \partial_\gamma
\end{equation}
for the wave operator, that $k^{(0,j)}$ solves
\begin{equation}\label{First system k 0}
\begin{cases}
(g_{\le j})^{\alpha\beta} \partial_\alpha \partial_\beta k^{(0,j)} = 
 & g_{\le j} \cdot \partial g_{\le j} \cdot \partial k_{j}   
+ \partial\bar\partial \big(g_j \cdot k_{\le j}\big)+ \bar\partial \big( \partial g_{\le j} \cdot k_{j}\big) 
+ \partial(g_{\le j}) \cdot P_j \big((g \cdot \partial g + \omega) \cdot k\big)
 \\
& \hphantom{=\sum\sum}
+ \bar\partial\partial \Big[P_j\Big(\sum_{j'>j} g_{j'} \cdot k_{j'} \Big) \Big] + \bar\partial \Big(g_{\le j} \cdot P_j \big((g \cdot \partial g + \omega) \cdot k\big)\Big)
\\
& \hphantom{=\sum\sum}
+ \partial g_{\le j} \cdot \partial^2 k^{(-1,j)}+  \partial g_{\le j} \cdot \partial k_j,
\\[10pt]
\Big(k^{(0,j)}|_{x^0=0},  & \hspace{-4em} \partial_0 k^{(0,j)}|_{x^0=0}\Big)  =\Big(k_j|_{x^0=0},k_j|_{x^0=0}- ((g_{\le j})^{00})^{-1} \cdot \mathcal F^{(j)}|_{x^0=0}\Big).
\end{cases}
\end{equation}
 Note that $k^{(-1, j)}$ and $k^{(0,j)}$ are not necessarily supported on frequencies $\|\xi\|\sim 2^j$, unlike $k_j$; however, we expect $k^{(-1, j)}$ and $k^{(0,j)}$ to satisfy similar estimates as $\JapD^{-1} k_j$ and $k_j$, respectively. Moreover, using the expression \eqref{Expansion wave operator} for the wave operator, we can reexpress \eqref{First system k -1}--\eqref{First system k 0} as
\begin{equation}\label{Wave system k -1}
\begin{cases}
\square_{g_{\le j}} k^{(-1,j)} & =  \partial \Big[P_j\Big(\sum_{j'>j} g_{j'} \cdot k_{j'} \Big) \Big] + g_{\le j} \cdot P_j \big((g \cdot \partial g + \omega) \cdot k\big) + \partial\big(g_j \cdot k_{\le j}\big) \\
& \hphantom{=\sum\sum}
+ \partial g_{\le j} \cdot k_{j} + \partial g_{\le j} \cdot \partial k^{(-1,j)},\\[10pt]
\Big(k^{(-1,j)}|_{x^0=0}, & \hspace{-0.5em} \partial_0 k^{(-1,j)}|_{x^0=0} \Big)  = \big(0,0\big)
\end{cases}
\end{equation}
and 
\begin{equation}\label{Wave system k 0}
\begin{cases}
\square_{g_{\le j}} k^{(0,j)} & = 
  g_{\le j} \cdot \partial g_{\le j} \cdot \partial k_{j}   
+ \partial\bar\partial \big(g_j \cdot k_{\le j}\big)+ \bar\partial \big( \partial g_{\le j} \cdot k_{j}\big) 
+ \partial(g_{\le j}) \cdot P_j \big((g \cdot \partial g + \omega) \cdot k\big)
 \\
& \hphantom{=\sum\sum}
+ \bar\partial\partial \Big[P_j\Big(\sum_{j'>j} g_{j'} \cdot k_{j'} \Big) \Big] + \bar\partial \Big(g_{\le j} \cdot P_j \big((g \cdot \partial g + \omega) \cdot k\big)\Big)
\\
& \hphantom{=\sum\sum}
+ \partial g_{\le j} \cdot \partial^2 k^{(-1,j)}+  \partial g_{\le j} \cdot \partial k^{(0,j)},
\\[10pt]
\Big(k^{(0,j)}|_{x^0=0}, & \hspace{-0.5em} \partial_0 k^{(0,j)}|_{x^0=0}\Big)  =\Big(k_j|_{x^0=0},k_j|_{x^0=0}- ((g_{\le j})^{00})^{-1} \cdot \mathcal F^{(j)}|_{x^0=0}\Big).
\end{cases}
\end{equation}

\medskip
\noindent \textbf{Remark.} Note that the equations \eqref{Wave system k -1}--\eqref{Wave system k 0} \emph{do not} contain second order time derivatives of the metric coefficients and first order time derivatives of the normal connection form (i.e.~terms of the form $\partial_0^2 g$ and $\partial_0 \omega$).
\medskip

Let us now collect the gauge equations satisfied by $(g, \omega)$. We will follow a similar process of frequency localization before the second differentiation for the elliptic equations implied by the balanced gauge condition for $h$ and $\bar\omega$.
\begin{itemize}
\item The metric components $N$, $\beta$ and $\bar g$ satisfy \eqref{Equation lapse}, \eqref{Equation shift} and \eqref{Equation g bar} respectively, namely
\begin{align}\label{Equation lapse once again}
\partial_0 (N-1) + |D| (N-1) = &  |D|  \Delta_{\bar g} ^{-1} \Bigg[ \f1N \bar{g}^{ab}\big( \partial_0  \tilde{\mathcal{F}}^{\natural}_{ab} - R_{a0b0}\big)\Bigg]\\
& + |D|\Delta_{\bar g}^{-1} \Bigg[ g \cdot R_{*} + g \cdot \partial g \cdot \tilde{\mathcal{F}}^{\natural} + (g-m_0) \cdot D\partial g  + g \cdot \partial g \cdot \partial g \Bigg],
\nonumber 
\end{align}
\begin{equation}
\partial_0 \b^k + |D|\b^k = |D|\Delta_{\bar g}^{-1} \Bigg[g \cdot R_*+ g \cdot Dh +  (g-m_0) \cdot D\partial g  +  g \cdot \partial g \cdot \partial g  \Bigg]
\end{equation}
and
\begin{equation}
\bar g^{ab}\partial_a \partial_b (\bar{g}) =  g \cdot R_{**} + g \cdot D^2\beta  +g\cdot \partial g \cdot \partial g + D \big( g \cdot \partial g\big),   
\end{equation}
where $R_{**}$ denotes any the purely spatial components of $R$. Using the expression \eqref{Calculation F natural difference 2 derivatives} for $ \partial_0  \tilde{\mathcal{F}}^{\natural}_{ab} - R^\natural_{a0b0}$, the expression \eqref{Schematic extraction two derivatives wedge product} for the low-high-high components $R^\natural_*$ (see also\eqref{Schematic extraction two spatial derivatives wedge product} for the case of $R^\natural_{**}$), as well as the expression \eqref{Decomposition R difference} for the components of $R-R^{\natural}$, the above relations yield:
\begin{align}\label{Equation lapse once more}
\partial_0 (N-1) + |D| (N-1) = &  
\partial_0 \big( \mathfrak F_{N}^{(0)} +|D|\Delta_{\bar g}^{-1} [(g-m_0)\cdot D g]\big) + \mathfrak F_{N}^{(1)}+|D|\Delta_{\bar g}^{-1}\Big[  (g-m_0) \cdot D^2 g + g \cdot \partial g \cdot \partial g \Big],
\end{align}

\begin{align}\label{Equation shift once more}
\partial_0 \b^k + |D|\b^k =&  \partial_0 \big(\mathfrak F_{\beta}^{(0)} +|D|\Delta_{\bar g}^{-1} [(g-m_0)\cdot D g]\big)+ \mathfrak F_{\beta}^{(1)} + |D|\Delta_{\bar g}^{-1} \Big[g \cdot D h +  (g-m_0) \cdot D^2 g  +  g \cdot \partial g \cdot \partial g \Big]
\end{align}
and
\begin{equation}\label{Equation g bar once more}
\bar g^{ab}\partial_a \partial_b (\bar{g}) =  g \cdot D^2 \beta+ (g-m_0)\cdot D^2 g + \mathfrak F_{\bar g},
\end{equation}
where
\begin{equation}\label{F N 0}
\mathfrak F_{N}^{(0)} =  |D|  \Delta_{\bar g} ^{-1} \sum_{j\in \mathbb N} \sum_{\substack{j_2 > j_1-2,\\|j_2-j_3|\le 2}} g \cdot  D  P_j \Big( P_{j_1} (m(e)) \cdot  P_{j_2} (\JapD^{-1}k) \cdot P_{j_3} (\JapD^{-1} k) \Big)
\end{equation}
(i.e.~the total time derivative term $\partial_0 \mathfrak F^{(0)}$  in the equation for $N$ arises from the term $\partial_0 D  P_j \Big( \big( P_{j_1} (m(e)) \cdot  P_{j_2} (\JapD^{-1}k) \cdot P_{j_3} (\JapD^{-1} k)\big) \Big) $ in  \eqref{Schematic extraction two derivatives wedge product}), 
\begin{align}\label{F N 1}
\mathfrak F_{N}^{(1)} =
& |D|  \Delta_{\bar g} ^{-1} \sum_{j\in \mathbb N} \Bigg[  \sum_{\substack{j_2 > j_1-2,\\|j_2-j_3|\le 2}} g \cdot  D^2  P_j \Big( P_{j_1} (m(e)) \cdot  P_{j_2} (\JapD^{-1}k) \cdot P_{j_3} (\JapD^{-1} k)\Big)  
 \\
& \hphantom{|D|  \Delta_{\bar g} ^{-1} \sum_{j\in \mathbb N} \Bigg[}
+ \sum_{\substack{j_2 > j_1-2,\\|j_2-j_3|\le 2}} \partial g \cdot  D  P_j \Big( P_{j_1} (m(e)) \cdot  P_{j_2} (\JapD^{-1}k) \cdot P_{j_3} (\JapD^{-1} k) \Big) \nonumber \\
& \hphantom{|D|  \Delta_{\bar g} ^{-1} \sum_{j\in \mathbb N} \Bigg[}
+ \sum_{\substack{j_2 > j_1-2,\\|j_2-j_3|\le 2}} D\partial g \cdot   P_j \Big(  P_{j_1} (m(e)) \cdot  P_{j_2} (\JapD^{-1}k) \cdot P_{j_3} (\JapD^{-1} k)\Big) \nonumber \\
& \hphantom{|D|  \Delta_{\bar g} ^{-1} \sum_{j\in \mathbb N} \Bigg[}
 +   \sum_{\substack{j_2 > j_1-2,\\|j_2-j_3|\le 2}} g\cdot P_j \big( P_{j_1} (\partial e)  \cdot P_{j_2} k \cdot P_{j_3} (\JapD^{-1} k) \big)  
\nonumber  \\
& \hphantom{|D|  \Delta_{\bar g} ^{-1} \sum_{j\in \mathbb N} \Bigg[}
+   \sum_{\substack{j_2 > j_1-2,\\|j_2-j_3|\le 2}} g \cdot P_j \big( P_{j_1} (m(e)) \cdot P_{j_2} ( g\cdot\partial g \cdot k + \omega \cdot k) \cdot P_{j_3} (\JapD^{-1} k) \big)  
\nonumber\\
& \hphantom{|D|  \Delta_{\bar g} ^{-1} \sum_{j\in \mathbb N} \Bigg[}
+   \sum_{\substack{j_2 > j_1-2,\\|j_2-j_3|\le 2}} g\cdot  P_j \big( P_{j_1} (\partial e) \cdot P_{j_2} (\JapD^{-1}( g\cdot\partial g \cdot k + \omega \cdot k)) \cdot P_{j_3} (\JapD^{-1} k) \big) 
 \nonumber\\
& \hphantom{|D|  \Delta_{\bar g} ^{-1} \sum_{j\in \mathbb N} \Bigg[} 
+   \sum_{\substack{j_2 > j_1-2,\\|j_2-j_3|\le 2}} g \cdot P_j \big( P_{j_1} (D\partial e) \cdot P_{j_2} (\JapD^{-1}k) \cdot P_{j_3} (\JapD^{-1} k) \big)
 \nonumber\\
& \hphantom{|D|  \Delta_{\bar g} ^{-1} \sum_{j\in \mathbb N} \Bigg[}
+ g \cdot P_j \Big(P_{\le j}(m(e)) \cdot P_{\le j}k \cdot P_j k\Big) + g\cdot P_j\Big( P_j(m(e)) \cdot P_{\le j}k\cdot P_{\le j} k\Big) 
\nonumber \\
& \hphantom{|D|  \Delta_{\bar g} ^{-1} \sum_{j\in \mathbb N} \Bigg[}
  + \sum_{j'>j} g\cdot P_j \Big( P_{j'}(m(e)) \cdot P_{\le j'}k \cdot P_{j'}k \Big) \Bigg]
 \nonumber\\
& + |D|\Delta_{\bar g}^{-1} \Bigg[ g \cdot \partial g \cdot \sum_j \sum_{\substack{j_2 > j_1-2,\\|j_2-j_3|\le 2}}P_j \Big(P_{j_1}(m(e))\cdot P_{j_2}k \cdot \JapD^{-1} P_{j_3} k\Big)\Bigg],
\nonumber 
\end{align}
the terms $\mathfrak F^{(0)}_\beta$ and $\mathfrak F^{(1)}_\beta$ have the same schematic expression as $\mathfrak F^{(0)}_N$ and $\mathfrak F^{(1)}_N$, respectively, and
\begin{align}\label{F bar g}
\mathfrak F_{\bar g} = &
\sum_{j\in \mathbb N} \Bigg[  \sum_{\substack{j_2 > j_1-2,\\|j_2-j_3|\le 2}} g \cdot D^2  P_j \Big( \big( P_{j_1} (m(e)) \cdot  P_{j_2} (\JapD^{-1}k) \cdot P_{j_3} (\JapD^{-1} k)\big) \Big)  
\\
& \hphantom{|D|  \Delta_{\bar g} ^{-1} \sum_{j\in \mathbb N} \Bigg[}
 +   \sum_{\substack{j_2 > j_1-2,\\|j_2-j_3|\le 2}} g\cdot P_j \big( P_{j_1} (\partial e)  \cdot P_{j_2} k \cdot P_{j_3} (\JapD^{-1} k) \big)  
\nonumber  \\
& \hphantom{|D|  \Delta_{\bar g} ^{-1} \sum_{j\in \mathbb N} \Bigg[}
+   \sum_{\substack{j_2 > j_1-2,\\|j_2-j_3|\le 2}} g \cdot P_j \big( P_{j_1} (m(e)) \cdot P_{j_2} ( g\cdot\partial g \cdot k + \omega \cdot k) \cdot P_{j_3} (\JapD^{-1} k) \big)  
\nonumber\\
& \hphantom{|D|  \Delta_{\bar g} ^{-1} \sum_{j\in \mathbb N} \Bigg[}
+   \sum_{\substack{j_2 > j_1-2,\\|j_2-j_3|\le 2}} g\cdot  P_j \big( P_{j_1} (\partial e) \cdot P_{j_2} (\JapD^{-1}( g\cdot\partial g \cdot k + \omega \cdot k)) \cdot P_{j_3} (\JapD^{-1} k) \big) 
 \nonumber\\
& \hphantom{|D|  \Delta_{\bar g} ^{-1} \sum_{j\in \mathbb N} \Bigg[} 
+   \sum_{\substack{j_2 > j_1-2,\\|j_2-j_3|\le 2}} g \cdot P_j \big( P_{j_1} (D\partial e) \cdot P_{j_2} (\JapD^{-1}k) \cdot P_{j_3} (\JapD^{-1} k) \big)
 \nonumber\\
& \hphantom{|D|  \Delta_{\bar g} ^{-1} \sum_{j\in \mathbb N} \Bigg[}
+ g \cdot P_j \Big(P_{\le j}(m(e)) \cdot P_{\le j}k \cdot P_j k\Big) + g\cdot P_j\Big( P_j(m(e)) \cdot P_{\le j}k\cdot P_{\le j} k\Big) 
\nonumber \\
& \hphantom{|D|  \Delta_{\bar g} ^{-1} \sum_{j\in \mathbb N} \Bigg[}
  + \sum_{j'>j} g\cdot P_j \Big( P_{j'}(m(e)) \cdot P_{\le j'}k \cdot P_{j'}k \Big) \Bigg]
 \nonumber 
\\
& +g\cdot \partial g \cdot \partial g + D \big( (g-m_0) \cdot \partial g\big). \nonumber 
\end{align}

\item Applying a Littlewood--Paley projection to the Codazzi equations \eqref{Codazzi 3+1}, we obtain for any $j\in \mathbb N$ and $a,b,c\in \{1,2,3\}$:
\begin{equation}\label{Littlewood Paley Codazzi 3+1}
\partial_a P_j h_{bc} - \bar{\nabla}_{b} P_j h_{ac} = P_j\big( g\cdot \partial g \cdot h \big) + P_j (g\cdot R_*).
\end{equation} 
Using  \eqref{Littlewood Paley Codazzi 3+1} repeatedly, we therefore infer:
\begin{align*}
(P_{\le j} \bar g)^{ab}\partial_a \partial_b P_j h_{cd} = & (P_{\le j} \bar g)^{ab}\partial_c \partial_d P_j h_{ab}+ P_{\le j} \bar g \cdot \bar{\partial} \big(P_j( g\cdot \partial g \cdot h )\big) + P_{\le j} \bar g \cdot \bar{\partial}P_j (g\cdot R_*)
\\
= & \bar\partial^2 \big( (P_{\le j} \bar g)^{ab}\cdot P_j h_{ab} \big) + \bar\partial P_{\le j} \bar g \cdot \bar\partial P_j h+ (\bar\partial^2P_{\le j} g )\cdot P_j h \\
& + P_{\le j} \bar g \cdot \bar{\partial} \big(P_j( g\cdot \partial g \cdot h )\big) + P_{\le j} \bar g \cdot \bar{\partial}P_j (g\cdot R_*).
\end{align*}
We now recall the gauge condition \eqref{Mean curvature condition} for the lapse $N$ which, after applying a Littlewood--Paley projection, yields the schematic relation:
\[
(P_{\le j} \bar g)^{ab} \cdot P_j h_{ab} + P_j \bar g \cdot P_{\le j} h+ \sum_{j'>j} P_j \big(P_{j'} \bar g\cdot P_{j'}h\big) = P_j \big(\bar g \cdot DN + \bar\partial \bar g \cdot N\big) + P_j \big( g\cdot \tilde{\mathcal{F}}^{\natural}\big) .
\]
Combining the above two relations and using the shorthand notation $f_j \doteq P_j f$,  we therefore infer:
\begin{align}
\Delta_{\bar g_{\le j}} h_j = &
D^2 \Big\{ - \bar g_j \cdot h_{\le j} - \sum_{j'>j} P_j \big( \bar g_{j'} \cdot h_{j'}\big)+ P_j \big(\bar g \cdot DN + \bar\partial \bar g \cdot N\big) + P_j \big( g\cdot \tilde{\mathcal{F}}^{\natural}\big)\Big\} \\
& + (D  \bar g_{\le j}) \cdot D h_j+ (D^2 g_{\le j} )\cdot  h_j + \bar g_{\le j} \cdot D P_j( g\cdot \partial g \cdot h ) +  \bar g_{\le j} \cdot D P_j (g\cdot R_*),
\nonumber 
\end{align}
where $\Delta_{\bar g_{\le j}} h_j$ is applied componentwise to $P_j h$. Using, as before, the expression \eqref{Schematic extraction two derivatives wedge product} for the low-high-high components $R^\natural_*$, \eqref{Decomposition R difference} for the components of $R-R^{\natural}$ and  the definition \eqref{F natural} for $\tilde{\mathcal F}^{\natural}$, we thus infer: 
\begin{equation}\label{Littlewood Paley Elliptic h}
\Delta_{\bar g_{\le j}} h_j =\mathfrak F^{(j)}_h,
\end{equation}
where
\begin{align}\label{Relation F j h}
\mathfrak F^{(j)}_h = &  D^2 P_j \Bigg[ g \cdot \sum_{\substack{j_2 > j_1-2,\\|j_2-j_3|\le 2}}\Big(P_{j_1}(m(e))\cdot P_{j_2}k \cdot \JapD^{-1} P_{j_3} k\Big)\Bigg]
\\
& + \bar g_{\le j} \cdot D P_j  \Bigg[  \sum_{\substack{j_2 > j_1-2,\\|j_2-j_3|\le 2}} g \cdot \partial D \Big( \big( P_{j_1} (m(e)) \cdot  P_{j_2} (\JapD^{-1}k) \cdot P_{j_3} (\JapD^{-1} k)\big) \Big)  
\nonumber \\
& \hphantom{|D|  \Delta_{\bar g} ^{-1} \sum_{j\in \mathbb N} \Bigg[}
 +   \sum_{\substack{j_2 > j_1-2,\\|j_2-j_3|\le 2}} g\cdot  P_{j_1} (\partial e)  \cdot P_{j_2} k \cdot P_{j_3} (\JapD^{-1} k)  
\nonumber  \\
& \hphantom{|D|  \Delta_{\bar g} ^{-1} \sum_{j\in \mathbb N} \Bigg[}
+   \sum_{\substack{j_2 > j_1-2,\\|j_2-j_3|\le 2}} g \cdot  P_{j_1} (m(e)) \cdot P_{j_2} ( g\cdot\partial g \cdot k + \omega \cdot k) \cdot P_{j_3} (\JapD^{-1} k) 
\nonumber\\
& \hphantom{|D|  \Delta_{\bar g} ^{-1} \sum_{j\in \mathbb N} \Bigg[}
+   \sum_{\substack{j_2 > j_1-2,\\|j_2-j_3|\le 2}} g\cdot   P_{j_1} (\partial e) \cdot P_{j_2} (\JapD^{-1}( g\cdot\partial g \cdot k + \omega \cdot k)) \cdot P_{j_3} (\JapD^{-1} k)  
 \nonumber\\
& \hphantom{|D|  \Delta_{\bar g} ^{-1} \sum_{j\in \mathbb N} \Bigg[} 
+   \sum_{\substack{j_2 > j_1-2,\\|j_2-j_3|\le 2}} g \cdot  P_{j_1} (D\partial e) \cdot P_{j_2} (\JapD^{-1}k) \cdot P_{j_3} (\JapD^{-1} k) 
 \nonumber\\
& \hphantom{|D|  \Delta_{\bar g} ^{-1} \sum_{j\in \mathbb N} \Bigg[}
+ \sum_{\bar j} \Big\{g \cdot P_{\bar j} \Big(P_{\le \bar j}(m(e)) \cdot P_{\le \bar j}k \cdot P_{\bar j} k\Big) + g\cdot  P_{\bar j} \Big(P_{\bar j}(m(e)) \cdot P_{\le \bar j}k\cdot P_{\le \bar j} k\Big)
\nonumber \\
& \hphantom{|D|  \Delta_{\bar g} ^{-1} \sum_{j\in \mathbb N} \Bigg[+ \sum_{\bar j}\Big\{}
  + \sum_{j'> \bar j} g\cdot P_{\bar j} \Big( P_{j'}(m(e)) \cdot P_{\le j'}k \cdot P_{j'}k \Big) \Big\}\Bigg]
 \nonumber \\
& +D^2 \Big\{ - \bar g_j \cdot h_{\le j} - \sum_{j'>j} P_j \big( \bar g_{j'} \cdot h_{j'}\big)+ P_j \big(\bar g \cdot DN + \bar\partial \bar g \cdot N\big) \Big\} \nonumber \\
& + (D  \bar g_{\le j}) \cdot D h_j+ (D^2 g_{\le j} )\cdot  h_j + \bar g_{\le j} \cdot D P_j( g\cdot \partial g \cdot h ).
\nonumber 
\end{align}
Recall also that $\partial_0 \bar g$ is related to $h$ via \eqref{Variation metric}, which takes the schematic form:
\[
\partial_0 \bar g = g \cdot h + Dg.
\]

\item  Applying a Littlewood--Paley projection to the expression \eqref{Normal curvature coordinates} for $R^\perp$, we obtain for any $j\in \mathbb N$ and $a\in \{1,2,3\}$:
\[
P_j (R^\perp)_{a0\hphantom{\bA}\bB}^{\hphantom{\a\b}\bA} = \partial_a P_j \omega_{0\bB}^{\bA}-\partial_0 P_j\omega_{a\bB}^{\bA} + P_j(\omega \cdot \omega).
\]
Differentiating the above relation and contracting with $P_{\le j}\bar g^{-1}$, we therefore infer:
\begin{equation}\label{Auxiliary for frequency projected omega 0 eq}
(P_{\le j}\bar g^{ab}) \partial_b  P_j (R^\perp)_{a0\hphantom{\bA}\bB}^{\hphantom{\a\b}\bA} 
= (P_{\le j}\bar g^{ab}) \partial_a \partial_b  P_j \omega_{0\bB}^{\bA} -  (P_{\le j} \bar g^{ab}) \partial_0 \partial_a \omega^{\bA}_{b \bB} + P_{\le j} g \cdot D P_j (\omega \cdot \omega). 
\end{equation}
On the other hand, applying a Littlewood--Paley projection on the gauge condition \eqref{Divergence condition frame}, we obtain schematically:
\begin{align}\label{Frequency projected gauge condition omega}
(P_{\le j} g^{ab}) \cdot \partial_a P_j \omega_{b \bB}^{\bA} & + P_j g \cdot \bar\partial P_{\le j}\bar\omega + \sum_{j'>j} P_j \big( P_{j'} g\cdot \bar\partial P_{j'} \bar \omega\big)\\
= & -P_j\big((\bar g^{ab}) \cdot \partial_a \partial_b |D|^{-1} \omega_{0\bB}^{\bA}\big) + P_j (\bar g \cdot D \bar g \cdot \omega)+P_j \big( m(e)_{\bB \bar C} \cdot \tilde{\mathcal F}^{\bA \bar C}_\perp \big)
\nonumber\\
= & 
 -P_{\le j}(\bar g^{ab})\cdot  \partial_a \partial_b |D|^{-1} P_j\omega_{0\bB}^{\bA} +P_{j}\bar g\cdot  D P_{\le j}\omega_0 + \sum_{j'>j} P_j \big( P_{j'}g\cdot  DP_{j'}\omega_0 \big)
\nonumber\\
& + P_j (\bar g \cdot D \bar g \cdot \omega)+P_j \big( m(e)_{\bB \bar C} \cdot \tilde{\mathcal F}^{\bA \bar C}_\perp \big).\nonumber
\end{align}
Differentiating the above relation with respect to $\partial_0$ and using it to reexpress the term $(P_{\le j}\bar g)^{ab} \partial_0 \partial_a  P_j \omega_{b\bB}^{\bA}$ in the right hand side of \eqref{Auxiliary for frequency projected omega 0 eq}, we therefore obtain:
\begin{align*}
(P_{\le j}\bar g^{ab}) \partial_a  P_j (R^\perp)_{b0\hphantom{\bA}\bB}^{\hphantom{\a\b}\bA} 
= & (P_{\le j}\bar g^{ab}) \partial_a \partial_b  \Big(P_j \omega_{0\bB}^{\bA} + \partial_0 |D|^{-1} P_j\omega_{0\bB}^{\bA}\Big) -\partial_0 P_j \big(  m(e)_{\bB \bar C} \cdot \tilde{\mathcal F}^{\bA \bar C}_\perp  \big)
\\
&  +\partial_0 \Big(P_{j}g\cdot  D P_{\le j}\omega + \sum_{j'>j} P_j \big( P_{j'}g \cdot  DP_{j'}\omega \big)+P_j (\bar g \cdot D \bar g \cdot \omega)\Big)\\
&+  (\partial P_{\le j} g) \cdot D P_j \omega + P_{\le j} g \cdot D P_j (\omega \cdot \omega),
\end{align*}
which can be rearranged as follows (using again the schematic notation $f_j = P_j f$ when no confusion arises; recall, however, that $\omega_0$ denotes the temporal component of $\omega$):
\begin{align}
\partial_0 (P_j\omega_{0}) +& |D| (P_j \omega_0) \\
=  & |D| \mathcal L_{P_{\le j}(\bar g^{-1})^{-1}}^{-1} \Bigg[  \Big(\partial_a [(P_{\le j}\bar g^{ab})  P_j (R^\perp)_{b0\hphantom{\bA}\bB}^{\hphantom{\a\b}\bA}]+ \partial_0 P_j \big(  m(e)_{\bB \bar C} \cdot \tilde{\mathcal F}^{\bA \bar C}_\perp  \big)\Big) \nonumber \\
&\hphantom{ |D| \mathcal L_{P_{\le j}(\bar g^{-1})^{-1}}^{-1} \Bigg[}
 + (Dg_{\le j}) \cdot  R^\perp_j +  (\partial P_{\le j} g) \cdot D P_j \omega + P_{\le j} g \cdot D P_j (\omega \cdot \omega)\Bigg] \nonumber \\
& + |D| \Bigg[ \partial g_{\le j} \cdot  \mathcal L_{P_{\le j}(\bar g^{-1})^{-1}}^{-1} |D|^2 \Big( P_{j}g\cdot  D P_{\le j}\omega + \sum_{j'>j} P_j \big( P_{j'}g \cdot  DP_{j'}\omega \big)+P_j (\bar g \cdot D \bar g \cdot \omega)\Big) \Bigg]
\nonumber \\
&  +\partial_0 \Bigg[  |D| \mathcal L_{P_{\le j}(\bar g^{-1})^{-1}}^{-1}\Bigg(P_{j}g\cdot  D P_{\le j}\omega + \sum_{j'>j} P_j \big( P_{j'}g \cdot  DP_{j'}\omega \big)+P_j (\bar g \cdot D \bar g \cdot \omega)\Bigg) \Bigg],   \nonumber
\end{align}
where the $2^{nd}$ order elliptic operator $\mathcal L_{P_{\le j}(\bar g^{-1})^{-1}}$ is defined by
\[
\mathcal L_{P_{\le j}(\bar g^{-1})^{-1}} \doteq  (P_{\le j}\bar g^{ab}) \partial_a \partial_b.
\]
Using the expression \eqref{Computation F perp} for $ \partial_a  P_j (\bar g^{ab}) R^\perp)_{b0\hphantom{\bA}\bB}^{\hphantom{\a\b}\bA}+ \partial_0 P_j \big(  m(e)_{\bB \bar C} \cdot \tilde{\mathcal F}^{\bA \bar C}_\perp  \big)$ together with the fact that 
\[
(P_{\le j}\bar g^{ab}) \cdot  P_j (R^\perp)_{b0} = P_j \big( \bar g^{ab} (R^\perp)_{b0}\big) + g_j \cdot R^\perp_{\le j} + \sum_{j'>j} P_j\big(g_{j'} \cdot R^\perp_{j'} \big)
\]
 and the fact that the low-high-high interactions in the expression $R^\perp = g \cdot k \wedge k$ for $R^\perp$ can be expressed using \eqref{Schematic extraction one derivative wedge product time indices}, we infer from the above:
\begin{equation}\label{Gauge condition omega 0 j}
\partial_0 (P_j\omega_{0}) +|D|  (P_j \omega_0)  = \partial_0 \dot{\mathfrak F}^{(0)}_{\omega_0} + \dot{\mathfrak F}_{\omega_0}^{(1)},
\end{equation}
where 
\begin{align}\label{F 0 omega 0}
\mathfrak F^{(0)}_{\omega_0} = 
& |D| \mathcal L_{P_{\le j}(\bar g^{-1})^{-1}}^{-1}\Bigg(
 \sum_{\substack{j_1,j_2,j_3,j_4:\\|j_3-j_4|\le 2}}  D P_j \Big\{
 P_{j_1} g \cdot P_{j_2} g \cdot P_{j_3} k \cdot P_{j_4}(\JapD^{-1} k) \Big\} \\
& \hphantom{ +\partial_0 \Bigg[  |D| \mathcal L_{P_{\le j}(\bar g^{-1})^{-1}}^{-1}\Bigg(}
+P_{j}g\cdot  D P_{\le j}\omega + \sum_{j'>j} P_j \big( P_{j'}g \cdot  DP_{j'}\omega \big)+P_j ( g \cdot D g \cdot \omega)\Bigg)  \nonumber 
\end{align}
and
\begin{align}\label{F 1 omega 0}
\mathfrak F^{(1)}_{\omega_0}
= &  |D| \mathcal L_{P_{\le j}(\bar g^{-1})^{-1}}^{-1} \Bigg[  
 \sum_{\substack{j_1,j_2,j_3,j_4:\\|j_3-j_4|\le 2}}  D^2 P_j \Big\{
 \Big(P_{j_1} g \cdot P_{j_2} g \cdot P_{j_3} k \cdot P_{j_4}(\JapD^{-1} k) \Big) 
 \\
& \hphantom{ |D|  \Delta^{-1}_{\big(P_{\le j}(\bar g^{-1})\big)^{-1}} \Bigg[} \hphantom{ \sum_{\substack{j_1,j_2,j_3,j_4:\\|j_3-j_4|\le 2}}  P_j \Bigg[  D}
 +D \Big(P_{j_1} ( \partial g + \partial e) \cdot P_{j_2} g \cdot P_{j_3} k \cdot P_{j_4}(\JapD^{-1} k) \Big) \Big\} \nonumber \\
&\hphantom{ |D|  \Delta^{-1}_{\big(P_{\le j}(\bar g^{-1})\big)^{-1}} \Bigg[}
D \big( g_{\le j} \cdot g_{\le j} \cdot k_{\le j} \cdot k_{\le j}\big)   + \sum_{j'>j} DP_j \Big( g_{j'}\cdot g_{\le j'} \cdot k_{\le j'} \cdot k_{\le j'} \Big) 
\nonumber \\
& \hphantom{ |D|  \Delta^{-1}_{\big(P_{\le j}(\bar g^{-1})\big)^{-1}} \Bigg[}
  +  \sum_{j''>j'>j} DP_j \Big( g_{j'}\cdot g_{j''} \cdot k_{j''} \cdot k_{\le j''} \Big) +  (\partial P_{\le j} g) \cdot D P_j \omega + P_{\le j} g \cdot D P_j (\omega \cdot \omega)\Bigg] 
\nonumber \\
& + |D| \Bigg[ \partial g_{\le j} \cdot  \mathcal L_{P_{\le j}(\bar g^{-1})^{-1}}^{-1} |D|^2 \Big( P_{j}g\cdot  D P_{\le j}\omega + \sum_{j'>j} P_j \big( P_{j'}g \cdot  DP_{j'}\omega \big)+P_j ( g \cdot D  g \cdot \omega)\Big) \Bigg].
\nonumber 
 \end{align}

\item Similarly, applying a Littlewood--Paley projection to the expression \eqref{Normal curvature coordinates} for $R^\perp$ and then differentiating and contracting with $P_{\le j} g^{-1}$, we obtain for any $j\in \mathbb N$ and any $c\in \{1,2,3\}$ (analogously to \eqref{Auxiliary for frequency projected omega 0 eq}):
\begin{equation}\label{Auxiliary for frequency projected omega bar eq}
(P_{\le j}\bar g^{ab}) \partial_a  P_j (R^\perp)_{ac\hphantom{\bA}\bB}^{\hphantom{\a\b}\bA} 
= (P_{\le j}\bar g^{ab}) \partial_a \partial_b  P_j \omega_{c\bB}^{\bA} -  (P_{\le j} \bar g^{ab}) \partial_c \partial_a \omega^{\bA}_{b \bB} + P_{\le j} g \cdot D P_j (\omega \cdot \omega). 
\end{equation}
Applying a $\partial_c$ derivative to the frequency projected version \eqref{Frequency projected gauge condition omega} of the gauge condition for $\omega$ and using the resulting expression to substitute the term   $(P_{\le j} \bar g^{ab}) \partial_c \partial_a \omega^{\bA}_{b \bB}$ in \eqref{Auxiliary for frequency projected omega bar eq}, we infer that:
\begin{align*}
(P_{\le j}\bar g^{ab}) \partial_a  P_j (R^\perp)_{ac\hphantom{\bA}\bB}^{\hphantom{\a\b}\bA} 
= & (P_{\le j}\bar g^{ab}) \partial_a \partial_b  \Big(P_j \omega_{c\bB}^{\bA} + \partial_c |D|^{-1} P_j\omega_{0\bB}^{\bA}\Big)
+ \bar \partial \Big(P_j g \cdot \bar\partial P_{\le j}\omega\Big) + \sum_{j'>j} \bar\partial P_j \big( P_{j'} g\cdot \bar\partial P_{j'} \omega\big)
\\
&  +\bar\partial \Big(P_{j}g\cdot  D P_{\le j}\omega\Big) + \bar\partial \Big(\sum_{j'>j} P_j \big( P_{j'}g DP_{j'}\omega \big)\Big)\\
& + \bar\partial P_j (\bar g \cdot D \bar g \cdot \omega)+\partial P_j \big( m(e) \cdot \tilde{\mathcal F}_\perp \big)\\
&+  (\partial P_{\le j} g) \cdot D P_j \omega + P_{\le j} g \cdot D P_j (\omega \cdot \omega),
\end{align*}
which can be rearranged as follows (using, as before,  the schematic notation $f_j = P_j f$ and with $\bar\omega$ the spatial  components of $\omega$):
\begin{align}\label{Almost gauge condition omega bar j}
 \Delta_{\big(P_{\le j} (\bar g^{-1})\big)^{-1}} (P_j\bar\omega) = &  P_{\le j} g \cdot D^2(P_j \omega_0) +  g_{\le j} \cdot D (R_*^\perp)_j + D \Big(g_j \cdot D \omega_{\le j}\Big) 
\\
&  +D \Big(g_j \cdot  D \omega_{\le j}\Big) + D \Big(\sum_{j'>j} P_j \big( g_{j'} \cdot D\omega_{j'} \big)\Big)
 + D P_j ( g \cdot D  g \cdot \omega)+D P_j \big( m(e) \cdot \tilde{\mathcal F}_\perp \big)
\nonumber\\
& +  g_{\le j} \cdot D \omega_j + Dg_{\le j} \cdot \omega_j + g_{\le j} \cdot D P_j (\omega \cdot \omega).   \nonumber
\end{align}
Using the expression \eqref{Schematic extraction one derivative wedge product} for the low--high--high components of $R^\perp_*$ and the definition \eqref{F perp} of $\tilde{\mathcal F}_\perp$,  we thus obtain:
\begin{equation}\label{Gauge condition omega bar j}
 \Delta_{\big(P_{\le j} (\bar g^{-1})\big)^{-1}} (P_j\bar\omega) = P_{\le j} g \cdot D^2 (P_j \omega_0) + \mathfrak F^{(j)}_{\bar \omega},
\end{equation}
where
\begin{align}\label{F bar omega}
\mathfrak F^{(j)}_{\bar\omega} = &    \sum_{\substack{j_2>j_1-2,\\ |j_3-j_2|\le 2}} \Big\{ 
 g_{\le j} \cdot D^2 P_j \Big( P_{j_1} g \cdot P_{j_2}(\JapD^{-1} k) \cdot P_{j_3} k \Big) \\
& \hphantom{\sum_{\substack{j_2>j_1-2,\\ |j_3-j_2|\le 2}} \Big\{}
+ g_{\le j} \cdot D P_j \Big(P_{j_1} g \cdot P_{j_2}\big(\JapD^{-1}(g\cdot \partial g \cdot k+\omega \cdot k)\big)\cdot P_{j_3}k\Big) \Big\}
\nonumber \\
  & +\sum_{\substack{j_2>j_1-2,\\ |j_3-j_2|\le 2}} D P_j \Big\{ 
 m(e)\cdot D \Big( g \cdot P_{j_1} g \cdot P_{j_2}(\JapD^{-1} k) \cdot P_{j_3} k \Big) \Big\} \nonumber \\
&  + g_{\le j} \cdot D P_j \Big(g_{\le j} \cdot k_{\le j} \cdot k_j\Big) +  g_{\le j} \cdot D\Big( g_j \cdot k_{\le j} \cdot k_{\le j}\Big)+ \sum_{j'>j} g_{\le j} \cdot DP_j \Big(  g_{j'}\cdot k_{\le j'} \cdot k_{j'}\Big) 
\nonumber \\
&  + D \Big(g_j \cdot D \omega_{\le j}\Big)  +D \Big(g_j \cdot  D \omega_{\le j}\Big) + D \Big(\sum_{j'>j} P_j \big( g_{j'} \cdot D\omega_{j'} \big)\Big)
\nonumber\\
&+ D P_j ( g \cdot D  g \cdot \omega) +  g_{\le j} \cdot D \omega_j + Dg_{\le j} \cdot \omega_j + g_{\le j} \cdot D P_j (\omega \cdot \omega).   \nonumber
\end{align}

\item Recall also  that $\partial^2 Y$ and $\partial e$ can be expressed in terms of the tensors $k$, $g$ and $\omega$ via the relations 
\begin{align}
\partial_\alpha \partial_\beta Y^A & = \Gamma_{\alpha\beta}^\gamma \partial_\gamma Y^A + k^{\bA}_{\alpha\beta} e_{\bA}^A, \nonumber\\
\partial_\alpha e_{\bA}^A &= \omega_{\alpha \bA}^{\bB} e_{\bB}^A - g^{\beta\gamma} m(e_{\bA}, e_{\bB}) k^{\bB}_{\alpha \beta} \partial_\gamma Y^A. \label{Relation e omega k once again}
\end{align}
\end{itemize}

\bigskip
\noindent \textbf{Remark.} Note that in the gauge-related equations \eqref{Equation lapse once again}--\eqref{Relation e omega k once again} above, no terms of the form $\partial_0^2 g$ appear, while terms of the form $\partial_0 \omega$ only appear in the last term of the right hand side of \eqref{Gauge condition omega 0 j}, which is a total time derivative (and hence can be treated like a spatial derivative when it comes to estimating $\omega$, in view of the parabolic estimates of Lemma \ref{lem:Parabolic estimates model}). We will make use of this structure when deriving the estimates for $(\dot k, \dot g, \dot\omega)$.

\medskip

\subsection{The difference equations}
In order to derive the system of equations for $(\dot k, \dot g, \dot\omega)$, we will subtract from equations \eqref{Wave system k -1}--\eqref{Relation e omega k once again}for $(k^{(2)}, g^{(2)}, \omega^{(2)})$ the corresponding equations for $(k^{(1)}, g^{(1)}, \omega^{(1)})$. Using the convention that, in the schematic expressions below, $(k, g, \omega)$ denotes either of $(k^{(i)}, g^{(i)}, \omega^{(i)})$, $i=1,2$, we obtain the following system of equations which is linear in $(\dot k, \dot g, \dot\omega)$:

\begin{itemize}
\item \textbf{Wave equations for $\dot k$:} Denoting with $\dot k^{(-1,j)}$  the differences of the tensors $k^{(-1,j)}$ defined by \eqref{First system k -1} (and similarly for $\dot k^{(0,j)}$), so that
\[
\dot k_j = \partial_0 \dot k^{(-1,j)} + \dot k^{(0,j)}
\]
we obtain from \eqref{Wave system k -1}--\eqref{Wave system k 0}:
\begin{equation}\label{Wave system k -1 differences}
\begin{cases}
\square_{g_{\le j}} \dot k^{(-1,j)} & =  \dot g_{\le j} \cdot \partial^2 k^{(-1,j)} + g_{\le j}\cdot \partial \dot g_{\le j} \partial k^{(-1,j)} + \dot g_{\le j}\cdot \partial g_{\le j} \partial k^{(-1,j)}\\
& \hphantom{=\sum\sum}
+\partial \Big[P_j\Big(\sum_{j'>j} g_{j'} \cdot \dot k_{j'} \Big) \Big] + g_{\le j} \cdot P_j \big((g \cdot \partial g + \omega) \cdot \dot k\big) + \partial\big(g_j \cdot \dot k_{\le j}\big) 
\nonumber \\
& \hphantom{=\sum\sum}
+ \partial g_{\le j} \cdot \dot k_{j} + \partial g_{\le j} \cdot \partial \dot k^{(-1,j)}\\
& \hphantom{=\sum\sum} 
+\partial \Big[P_j\Big(\sum_{j'>j} \dot g_{j'} \cdot k_{j'} \Big) \Big] + g_{\le j} \cdot P_j \big((g \cdot \partial \dot g + \dot g \cdot \partial  g  + \dot\omega) \cdot k\big) \\
& \hphantom{=\sum\sum}
 +\dot g_{\le j} \cdot P_j \big((g \cdot \partial  g +  g \cdot \partial  g  + \omega) \cdot k\big)
+ \partial\big(\dot g_j \cdot k_{\le j}\big) 
\nonumber \\
& \hphantom{=\sum\sum}
+ \partial \dot g_{\le j} \cdot k_{j} \\[10pt]
\Big(\dot k^{(-1,j)}|_{x^0=0}, & \hspace{-0.5em} \partial_0 \dot k^{(-1,j)}|_{x^0=0} \Big)  = \big(0,0\big)
\end{cases}
\end{equation}
and 
\begin{equation}\label{Wave system k 0 differences}
\begin{cases}
\square_{g_{\le j}} \dot k^{(0,j)} &  =  \dot g_{\le j} \cdot \partial^2 k^{(0,j)} + g_{\le j}\cdot \partial \dot g_{\le j} \partial k^{(0,j)} + \dot g_{\le j}\cdot \partial g_{\le j} \partial k^{(0,j)}\\
& \hphantom{=\sum\sum}
  g_{\le j} \cdot \partial g_{\le j} \cdot \partial \dot k_{j}   
+ \partial\bar\partial \big(g_j \cdot \dot k_{\le j}\big)+ \bar\partial \big( \partial g_{\le j} \cdot \dot k_{j}\big) 
+ \partial(g_{\le j}) \cdot P_j \big((g \cdot \partial g + \omega) \cdot \dot k\big)
 \\
& \hphantom{=\sum\sum}
+ \bar\partial\partial \Big[P_j\Big(\sum_{j'>j} g_{j'} \cdot \dot k_{j'} \Big) \Big] + \bar\partial \Big(g_{\le j} \cdot P_j \big((g \cdot \partial g + \omega) \cdot \dot k\big)\Big)
\\
& \hphantom{=\sum\sum}
+ \partial g_{\le j} \cdot \partial^2 \dot k^{(-1,j)}+  \partial g_{\le j} \cdot \partial \dot k^{(0,j)},
\\
& \hphantom{=\sum\sum}
+  g_{\le j} \cdot \partial \dot g_{\le j} \cdot \partial k_{j}    + \dot g_{\le j} \cdot \partial g_{\le j} \cdot \partial k_{j}
+ \partial\bar\partial \big(\dot g_j \cdot k_{\le j}\big)+ \bar\partial \big( \partial \dot g_{\le j} \cdot k_{j}\big) \\
& \hphantom{=\sum\sum}
+ \partial(g_{\le j}) \cdot P_j \big((\dot g \cdot \partial g + g \cdot \partial \dot g + \dot\omega) \cdot k\big)+ \partial(\dot g_{\le j}) \cdot P_j \big((g \cdot \partial g + \omega) \cdot k\big)
 \\
& \hphantom{=\sum\sum}
+ \bar\partial\partial \Big[P_j\Big(\sum_{j'>j} \dot g_{j'} \cdot k_{j'} \Big) \Big] + \bar\partial \Big(g_{\le j} \cdot P_j \big((\dot g \cdot \partial g + g \cdot \partial \dot g+ \dot \omega) \cdot k\big)\Big)
\\
& \hphantom{=\sum\sum}
+ \bar\partial \Big(\dot g_{\le j} \cdot P_j \big((g \cdot \partial g + \omega) \cdot k\big)\Big)
\\
& \hphantom{=\sum\sum}
+ \partial \dot g_{\le j} \cdot \partial^2 k^{(-1,j)}+  \partial \dot g_{\le j} \cdot \partial k^{(0,j)},
\\[10pt]
\Big(\dot k^{(0,j)}|_{x^0=0}, & \hspace{-0.5em} \partial_0 \dot k^{(0,j)}|_{x^0=0}\Big)  =\big(0, 0\big).
\end{cases}
\end{equation}

\item \textbf{The elliptic-parabolic system for $(\dot k, \dot g, \dot \omega)$:} We can readily obtain from \eqref{Equation lapse once more}--\eqref{Relation e omega k once again} the following set of equations for $(\dot k, \dot g, \dot \omega)$:
\begin{itemize}
\item The lapse difference $\dot N$ satisfies:
\begin{align}\label{Equation lapse once more differences}
\partial_0 \dot N + |D| \dot N = & \partial_0 \big( \dot{\mathfrak F}_{N}^{(0)} +|D|\Delta_{\bar g}^{-1} [(g-m_0)\cdot D\dot g]\big) + \dot{\mathfrak F}_{N}^{(1)} \\
& + |D| \Delta_{\bar g}^{-1} \Bigg[ \dot g \cdot D\partial g + (g-m_0)\cdot D^2 \dot g + \dot g \cdot \partial g \cdot \partial g + g \cdot \partial g \cdot \partial\dot g  \Bigg],  \nonumber
\end{align}
where $\dot{\mathfrak F}_{N}^{(0)}$ is the difference of the expressions $\mathfrak F^{(0)}_N$ for $(Y^{(1)}, e^{(1)})$ and $(Y^{(2)}, e^{(2)})$; one obtains the schematic expression for $\dot{\mathfrak F}_{N}^{(0)}$ by formally ``differentiating" the corresponding expression for  $\mathfrak F^{(0)}_N$ (given by \eqref{F N 0}), using the rule that, if $G=a\cdot b$, then $\dot G = \dot a \cdot b + a \cdot \dot b$. In particular, we have: 
\begin{align*}
\dot{\mathfrak F}_{N}^{(0)} =  & |D|  \Delta_{\bar g} ^{-1} \sum_{j\in \mathbb N} \sum_{\substack{j_2 > j_1-2,\\|j_2-j_3|\le 2}} \Bigg[ g \cdot  D  P_j \Big(  P_{j_1} (m(e)) \cdot  P_{j_2} (\JapD^{-1}k) \cdot P_{j_3} (\JapD^{-1} \dot k) + P_{j_1}\dot e \cdot P_{j_2}(\JapD^{-1} k) \cdot P_{j_3}(\JapD^{-1} k)\Big)\\
& \hphantom{ |D|  \Delta_{\bar g} ^{-1} \sum_{j\in \mathbb N} \sum_{\substack{j_2 > j_1-2,\\|j_2-j_3|\le 2}}}
+  \dot g \cdot  D  P_j \Big( \big( P_{j_1} (m(e)) \cdot  P_{j_2} (\JapD^{-1}k) \cdot P_{j_3} (\JapD^{-1} k)\big) \Big)\Bigg]\\
+&|D| \Delta_{\bar g}^{-1} \Bigg[ \dot g \cdot D\mathfrak F_N^{(0)} + Dg \cdot \mathfrak F_N^{(0)} \Bigg].
\end{align*}
Similarly, the expression for  $\dot{\mathfrak F}_{N}^{(1)}$ is obtained by ``differentiating''  the expression \eqref{F N 1} for  $\mathfrak F_{N}^{(1)}$.

\item The difference of the foliation second fundamental forms $\dot h$ satisfies
\begin{equation}\label{Littlewood Paley Elliptic h differences}
\Delta_{\bar g_{\le j}} \dot h_j = \dot g_{\le j} \cdot D^2 h_j + D\dot g_{\le j}\cdot Dh_j +\dot{\mathfrak F}^{(j)}_{h},
\end{equation}
where $\dot{\mathfrak F}^{(j)}_h$ is obtained by formally differentiating (as before) the expression \eqref{Relation F j h} for $\mathfrak F^{(j)}_h$.

\item The shift difference $\dot\beta$ satisfies:
\begin{align}\label{Equation shift once more differences}
\partial_0 \dot\b^k + |D|\dot\b^k =& \partial_0 \big(\dot{\mathfrak F}_{\beta}^{(0)} +|D|\Delta_{\bar g}^{-1} [(g-m_0)\cdot D\dot g]\big)+ \dot{\mathfrak F}_{\beta}^{(1)}\\
& + |D|\Delta_{\bar g}^{-1} \Bigg[g \cdot D\dot h + \dot g \cdot D\partial g +  (g-m_0) \cdot D^2 \dot g  +  \dot g \cdot \partial g \cdot \partial g+ g \cdot \partial g \cdot \partial \dot g \Bigg],
\nonumber
\end{align}
where  $ \dot{\mathfrak F}_{\beta}^{(0)}$ and $\dot{\mathfrak F}_{\beta}^{(1)}$ are obtained from $ \mathfrak F_{\beta}^{(0)}$ and $\mathfrak F_{\beta}^{(1)}$ by formal differentiation as before; recall $ \mathfrak F_{\beta}^{(0)}$ and $\mathfrak F_{\beta}^{(1)}$ are schematically of the same form as $ \mathfrak F_{N}^{(0)}$ and $\mathfrak F_{N}^{(1)}$, respectively.

\item The difference of the spatial metrics $\dot{\bar g}$ satisfies
\begin{equation}\label{Equation g bar once more differences}
\bar g^{ab}\partial_a \partial_b (\dot{\bar{g}}) = g \cdot D^2\dot\beta + \dot g \cdot D^2 g + \dot{\mathfrak F}_{\bar g},
\end{equation}
where $ \dot{\mathfrak F}_{\bar g}$ is obtained by formally differentiating the expression \eqref{F bar g} for $\mathfrak F_{\bar g}$.

In view of \eqref{Variation metric}, $\partial_0 \dot{\bar g}$ is also related to $\dot h$ by:
\begin{equation}\label{Relation time derivative g bar differences}
\partial_0 \dot{\bar g} = g\cdot \dot h+ \dot g \cdot \partial g + g\cdot D\dot g.
\end{equation}

\item The difference of the temporal components of the normal connection coefficients $\dot\omega_0$ satisfies:
\begin{equation}\label{Gauge condition omega 0 j differences}
\partial_0 (P_j\dot\omega_{0}) +|D| (P_j \dot\omega_0)  = \partial_0 \dot{\mathfrak F}^{(0)}_{\omega_0} + \dot{\mathfrak F}_{\omega_0}^{(1)},
\end{equation}
where $\dot{\mathfrak F}^{(0)}_{\omega_0}$ and $\dot{\mathfrak F}_{\omega_0}^{(1)}$ are obtained by formally differentiating, respectively, the expressions  \eqref{F 0 omega 0} and \eqref{F 1 omega 0} for $\mathfrak F^{(0)}_{\omega_0}$ and $\mathfrak F^{(1)}_{\omega_0}$, respectively.

\item The difference of the spatial components of the normal connection coefficients $\dot{\bar\omega}$ satisfies:
\begin{equation}\label{Gauge condition omega bar j differences}
 \Delta_{\big(P_{\le j} (\bar g^{-1})\big)^{-1}} (P_j\dot{\bar\omega}) = g_{\le j} \cdot D^2 (P_j \dot\omega_0) + \dot g_{\le j} \cdot D^2 (P_j \omega_0) + \dot{\mathfrak F}^{(j)}_{\bar\omega},
\end{equation}
where $\dot{\mathfrak F}^{(j)}_{\bar\omega}$ is obtained by formally differentiating the expression \eqref{F bar omega} for $\mathfrak F^{(j)}_{\bar\omega}$.

\item In view of the relations \eqref{Relation e omega k once again} for  $\partial^2 Y$ and $\partial e$, the corresponding differences satisfy:
\begin{align} \label{Relation Y k once again differences}
\partial_\alpha \partial_\beta \dot Y = & (\dot g \cdot \partial g + g \cdot \partial \dot g )\cdot \partial Y + g\cdot \partial g \cdot \partial \dot Y  + \dot k \cdot e +k \cdot \dot e, \\
\partial_\alpha \dot e = & \dot \omega \cdot e + \omega \cdot \dot e +\dot g\cdot m(e)\cdot k\cdot \partial Y+ g\cdot \dot e\cdot k\cdot \partial Y \label{Relation e omega k once again differences}\\
& +g\cdot m(e)\cdot \dot k\cdot \partial Y+ g\cdot m(e)\cdot k\cdot \partial \dot Y.\nonumber 
\end{align}
\end{itemize}
\end{itemize}

\subsection{Spacetime norms for the difference tensors}
We will now specify the spacetime norms with respect to which we will estimate the size of $(\dot k, \dot g, \dot \omega)$. To this end, let us first recall that, as a consequence of the assumptions of Proposition \ref{prop:Uniqueness parabolic}, the immersions $Y^{(1)}, Y^{(2)}$ and the frames $e^{(1)}, e^{(2)}$ satisfy the following estimate (dropping, as before, the superscripts $(1), (2)$):
\begin{align}\label{Norm from existence}
\| (Y,e)\|_{\mathfrak E} \doteq \, 
& \sum_{l=0}^2 \| \partial^l (Y - Y_0) \|_{L^{\infty}H^{s-l}} + \| \partial^2 Y \|_{L^4 W^{\f1{12},4}}+\|\partial^2 Y \|_{L^{\f72} L^{\f{14}3}}\\
& + \sum_{l=0}^2\Big( \|\partial^l k\|_{L^\infty H^{s-2-l}} + \|\partial^l k\|_{L^2 W^{-l-\f12+s_0,\infty}} \Big)
\nonumber \\
& + \| \partial g \|_{L^1 W^{\f14\delta_0,\infty}} +\sum_{l=1}^2\Big( \| \partial^l g \|_{L^2 H^{2-l+\f16+s_1}} +\|\partial^l g\|_{L^{\f74} W^{2-l+s_1, \f73}} \|  \partial^l g \|_{L^{\infty} H^{s-1-l}}\Big) 
\nonumber \\
& + \| \omega \|_{L^1 W^{\f14\delta_0,\infty}} +  \sum_{l=0}^1\Big(\| \partial^l \omega \|_{L^2 H^{1-l+\f16+s_1}} +\|\partial \omega \|_{L^{\f74} W^{1-l+s_1, \f73}}+ \|  \partial^l \omega \|_{L^{\infty} H^{s-2-l}}\Big) \nonumber\\
&+\sum_{l=0}^1\Big( \| \partial^{1+l} e \|_{L^2 H^{1-l+\f16+s_1}} +\|\partial^{1+l} e\|_{L^{\f74} W^{1-l+s_1, \f73}}+ \|  \partial^{1+l} e \|_{L^{\infty} H^{s-2-l}}\Big)
\nonumber \\
 \lesssim & \epsilon. \nonumber 
\end{align}
Repeating the proof of the energy and Strichartz estimates for $k_j$ (i.e.~Lemmas \ref{lem:Energy estimates} and \ref{lem:Strichartz estimates}) in the case of $k^{(-1,j)}$ and $k^{(0,j)}$ (which satisfy \eqref{Wave system k -1} and \eqref{Wave system k 0}, respectively), we also infer that $k^{(-1,j)}$ and $k^{(0,j)}$ satisfy the same estimates as $\JapD^{-1}k_j$ and $k_j$, respectively, namely:
\begin{align*}
\Bigg[\sum_{j\in \mathbb N} \sum_{l=0}^2\Big( \|2^{(s-2-l)j} & \partial^l k^{(0,j)}\|^2_{L^\infty L^2} +\|2^{(s-1-l)j}\partial^l k^{(-1,j)}\|^2_{L^\infty L^2} \\
&+ \|2^{(-l-\f12+s_0+\f18\delta_0)j}\partial^l k^{(0,j)}\|^2_{L^2 L^\infty}+ \|2^{(-l+\f12+s_0+\f18\delta_0)j}\partial^l k^{(-1,j)}\|^2_{L^2 L^\infty} \Big)\Bigg]^{\f12} \lesssim \epsilon.
\end{align*}

We will estimate the size of $(\dot Y, \dot e)$ with respect to a norm $\|\cdot \|_{\mathfrak U}$ which lies at one order of differentiability below $\|\cdot \|_{\mathfrak E}$:
\begin{align}\label{Norm for uniqueness}
\|(\dot Y, \dot e)\|_{\mathfrak U} \doteq \, 
& \Bigg[\sum_{j\in \mathbb N} \sum_{l=0}^2\Big( \|2^{(s-3-l)j} \partial^l \dot k^{(0,j)}\|^2_{L^\infty L^2} +\|2^{(s-2-l)j}\partial^l \dot k^{(-1,j)}\|^2_{L^\infty L^2}  \\
& \hphantom{+ \Bigg[\sum_{j\in \mathbb N} \sum_{l=0}^2\Big(}
+ \|2^{(-l-\f32+s_0+\f18\delta_0)j}\partial^l \dot k^{(0,j)}\|^2_{L^2 L^\infty}+ \|2^{(-l-\f12+s_0+\f18\delta_0)j}\partial^l \dot k^{(-1,j)}\|^2_{L^2 L^\infty} \Big)\Bigg]^{\f12}
\nonumber \\
& + \| \dot g \|_{L^1 W^{\f14\delta_0,\infty}} +\| \partial \dot g \|_{L^2 H^{\f16+s_1}} +\|\partial \dot g\|_{L^{\f74} W^{s_1, \f73}}+ \|  \partial \dot g \|_{L^{\infty} H^{s-3}} \nonumber \\
&\hphantom{+ \| \dot g \|_{L^1 W^{\f14\delta_0,\infty}}}
+ \| \dot \omega \|_{L^2 H^{\f16+s_1}} +\|\dot\omega\|_{L^{\f74} W^{s_1, \f73}}+ \|  \dot \omega \|_{L^{\infty} H^{s-3}} \nonumber \\
&+\| \partial \dot e \|_{L^2 H^{\f16+s_1}} + \|  \partial \dot e \|_{L^{\infty} H^{s-3}}+\|\partial \dot e\|_{L^{\f74} W^{s_1, \f73}}\nonumber \\
& + \sum_{l=0}^2 \| \partial^l \dot Y \|_{L^{\infty}H^{s-l-1}} + \| \partial^2 \dot Y \|_{L^4 W^{-1+\f1{12},4}}+\|\partial^2 \dot Y \|_{L^{\f72} W^{-1,\f{14}3}}.
\nonumber
\end{align}
Note that, by interpolation, we also have
\begin{align*}
\Bigg[\sum_{j\in \mathbb N} \sum_{l=0}^2\Big( & \|2^{(-l+\f1{12}-1+s_0+\f18\delta_0)j} \partial^l \dot k^{(0,j)}\|^2_{L^4 L^4} +\|2^{(-l+\f1{12}+s_0+\f18\delta_0)j}\partial^l \dot k^{(-1,j)}\|^2_{L^4 L^4}  \\
& \hphantom{+ \Bigg[\sum_{j\in \mathbb N} \sum_{l=0}^2\Big(}
+ \|2^{(-l+s_0+\f18\delta_0)j}\partial^l \dot k^{(0,j)}\|^2_{L^{\f72} L^{\f{14}3}}+ \|2^{(-l+s_0+\f18\delta_0)j}\partial^l \dot k^{(-1,j)}\|^2_{L^{\f72} L^{\f{14}3}} \Big)\Bigg]^{\f12}
\lesssim \|(\dot Y, \dot e)\|_{\mathfrak U}.
\end{align*}

\begin{remark*}
The norm $\|\cdot \|_{\mathfrak U}$ provides control at one order of differentiability below that of $\|\cdot\|_{\mathfrak E}$. In view of the quasilinear character of the wave equation \eqref{Wave equation k j careful} for $k$, controlling the differences $(\dot Y, \dot e)$ at one order of differentiability lower compared to $(Y,e)$ is in a sense optimal: From \eqref{Wave equation k j careful}, we infer that $\dot k_j$ satisfies an equation of the form
\[
\square_{g_{\le j}} \dot k_j = \dot g_{\le j} \partial^2 k_j+\ldots
\]
(note that, in the case of a semilinear wave equation, the first term in the right hand side above would be absent). Thus, when performing an energy estimate for $\dot k_j$, the presence of the term $\dot g_{\le j} \partial^2 k_j$ on the right hand side implies that $\|D^a \dot k_j\|_{L^\infty L^2}$ can be no better than $\|D^{a+1} k_j\|_{L^1 L^2}$, i.e.~the energy norm in which $\dot k$ is estimated has to ``lose'' at least one derivative compared to the analogous norm for $k$. The gauge equations for $(g, \omega)$ then imply that an analogous loss has to be propagated to the estimates for $(\dot g, \dot \omega)$.

We should also note that the expression \eqref{Norm for uniqueness} for $\|\cdot \|_{\mathfrak U}$ \textbf{does not} include an analogue of each term in the expression \eqref{Norm from existence} for $\|\cdot \|_{\mathfrak E}$ (even with a loss of one derivative). Most notably:
\begin{enumerate}
\item $\|(\dot Y, \dot e)\|_{\mathfrak U}$ does not control $\| \JapD^{-1} \partial \dot g\|_{L^1 L^\infty}$ or $\|\JapD^{-1}\dot \omega\|_{L^1 L^\infty}$ (while $\|( Y,  e)\|_{\mathfrak E}$ controls $\| \partial g\|_{L^1 L^\infty}$ and  $\| \omega\|_{L^1 L^\infty}$). It does control, however, $\|\dot g\|_{L^1 L^\infty}$.
\item $\|(\dot Y, \dot e)\|_{\mathfrak U}$ does not control second order time derivatives of $\dot g$ or first order time derivatives of $\dot\omega$. The structure of the equations that allows us to get away without estimating $\partial_0^2 \dot g$ and $\partial_0 \dot \omega$  is explained in the remark below \eqref{Relation e omega k once again}.
\end{enumerate}
\end{remark*}

\subsection{Estimates for the difference equations}
We will now proceed to estimate the size of the difference  $(\dot Y, \dot e)$ with respect to the norm \eqref{Norm for uniqueness}. The process will be very similar to the one we followed in establishing the bootstrap estimates of Proposition \ref{prop:Bootstrap}; for this reason, we will mostly highlight the points of difference between the two processes and briefly sketch the rest of the arguments. 
The reader should keep in mind that there is a close correspondence (in terms of both the statement and the method of proof) between the estimates that we establish for $(\JapD^{-1}\dot k, \JapD^{-1}\dot g, \JapD^{-1}\dot\omega)$ for our uniqueness result and the estimates for $(k,g,\omega)$ appearing earlier in the bootstrap argument employed in the existence result.

\bigskip
\noindent \textbf{Energy and Strichartz estimates for $\dot k^{(-1,j)}$ and $\dot k^{(0,j)}$.} The tensor fields  $\dot k^{(-1,j)}$ and $\dot k^{(0,j)}$ satisfy the system of wave equations \eqref{Wave system k -1 differences}--\eqref{Wave system k 0 differences}. In view of the fact that they both tensor fields have vanishing initial data at $x^0=0$, using $\partial_0$ as a multiplier for \eqref{Wave system k -1 differences}--\eqref{Wave system k 0 differences} (and the fact that $\|\partial g_{\le j}\|_{L^1 L^\infty} \le \|(Y,e)\|_{\mathfrak E}\lesssim 1$) we obtain the energy estimates:
\[
\sum_j \| 2^{(s-3)j} \partial \dot k^{(-1,j)}\|^2_{L^\infty L^2} \lesssim \sum_j \Big|  \int_{\{0\le x^0\le T\}} 2^{2(s-3)j}\square_{g_{\le j}}\dot k^{(-1,j)} \cdot \partial_0 \dot k^{(-1,j)} \, dx \Big|
\]
and
\[
\sum_j \| 2^{(s-4)j} \partial \dot k^{(0,j)}\|^2_{L^\infty L^2} \lesssim \sum_j  \Big|  \int_{\{0\le x^0\le T\}} 2^{2(s-4)j}\square_{g_{\le j}}\dot k^{(0,j)} \cdot \partial_0 \dot k^{(0,j)} \, dx^0 \dots dx^3 \Big|.
\]
We will treat each term in the right hand side of the equations above in a similar way as we did for the corresponding terms in \eqref{Error energy estimates localized}, highlighting the main differences. In order to keep track with the analogies in dealing with \eqref{Error energy estimates localized}, the reader should keep in mind that $\dot k^{(-1,j)}$ is controlled  in spaces of the same regularity as $k_j$, while $\dot k^{(0,j)}$ is controlled  in spaces of the same regularity as $\JapD^{-1}k_j$; moreover, $\partial k^{(-1,j)}$ and $k^{(0,j)}$ are controlled in the same regularity  spaces as $k_j$.
\begin{itemize}
\item In the case of $\square_{g_{\le j}}\dot k^{(-1,j)}$ (see \eqref{Wave system k -1 differences} for its expression), the term $ \dot g_{\le j} \cdot \partial^2 k^{(-1,j)}$ has no analogue in the corresponding expression for $\square_{g_{\le j}} k_j$ (see \eqref{Wave equation projection}). We estimate this term as follows:
\begin{align}\label{First new term in dot k equation}
\sum_j \Big|  \int_{\{0\le x^0\le T\}} 2^{2(s-3)j}\dot g_{\le j} \cdot & \partial^2 k^{(-1,j)} \cdot \partial_0 \dot k^{(-1,j)} \, dx \Big| \\
\lesssim & \sum_j  \| \dot g_{\le j}\|_{L^1 L^\infty} \cdot \| 2^{(s-3)j} \partial^2 k^{(-1,j)}\|_{L^\infty L^2} \cdot \|2^{(s-3)j} \partial \dot k^{(-1,j)}\|_{L^\infty L^2} \nonumber \\
\lesssim &  \|(Y,e)\|_{\mathfrak E} \cdot \|(\dot Y, \dot e)\|_{\mathfrak U}^2  \nonumber  \\
\lesssim & \epsilon \cdot \|(\dot Y, \dot e)\|_{\mathfrak U}^2. \nonumber 
\end{align}
The rest of the terms appearing in the expression \eqref{Wave system k -1 differences} for $\square_{g_{\le j}}\dot k^{(-1,j)}$ are similar to terms appearing in \eqref{Error energy estimates localized}; one crucial difference is that no terms of the form $\partial^2 g, \partial^2 \dot g, \partial\omega$ or $\partial \dot\omega$ appear in the right hand side of \eqref{Wave system k -1 differences}.\footnote{Note that $\|(\dot Y, \dot e)\|_{\mathfrak U}$ doesn't control terms of the form $\partial_0^2 \dot g$ or $\partial_0\dot\omega$. Thus, if  \eqref{Wave system k -1 differences} contained such terms, we wouldn't be able to close the necessary estimates with respect to the norm $\|(\dot Y, \dot e)\|_{\mathfrak U}$.} All of them can be estimated similarly as the corresponding terms in \eqref{Error energy estimates localized}, with the exception of the terms 
\begin{equation}\label{Extra terms energy estimates}
g_{\le j}\cdot \partial \dot g_{\le j} \partial k^{(-1,j)}, \quad (g \cdot \partial \dot g)_{\le j}  \cdot k_j, \quad \dot\omega_{\le j} \cdot k_j.
\end{equation}
The analogous terms in \eqref{Error energy estimates localized}, namely 
\[
g_{\le j} \cdot \partial g_{\le j} \cdot  \partial k_j,  \quad (g \cdot \partial \dot g)_{\le j}\cdot \partial k_j, \quad \dot \omega_{\le j} \cdot\partial  k_j,
\]
were estimated using the bounds for $\|\partial g\|_{L^1 L^\infty}$ and $\|\omega\|_{L^1 L^\infty}$ provided by the bootstrap assumption for $\|(Y, e)\|_{\mathfrak E}$. However, $\|\partial \JapD^{-1} \dot g\|_{L^1 L^\infty}$ and $\|\JapD^{-1} \dot \omega\|_{L^1 L^\infty}$ are \textbf{not} controlled by $\|(\dot Y, \dot e)\|_{\mathfrak U}$; instead, we estimate the terms in \eqref{Extra terms energy estimates} as follows:\footnote{Note that the estimate \eqref{Estimates for dot k not requiring L1Linfty} for the terms involving $\partial \dot g_{\le j}$ is similar to the estimate \eqref{Low high L74} for terms involving $\partial^2 g_{\le j}$. This should again be viewed as an instance of the correspondence between our estimates for $(\JapD^{-1} \dot k, \JapD^{-1}\dot g, \JapD^{-1}\dot\omega)$ in the proof of the uniqueness statement and the our estimates for $(k, g, \omega)$ in the proof of the existence statement.}
\begin{align}\label{Estimates for dot k not requiring L1Linfty}
\sum_j & \Big|  \int_{\{0\le x^0\le T\}} 2^{2(s-3)j} g_{\le j}\cdot \partial \dot g_{\le j} \cdot  \partial k^{(-1,j)} \cdot \partial_0 \dot k^{(-1,j)} \, dx \Big|\\
& \lesssim \|g\|_{L^\infty L^\infty} \|\partial \dot g\|_{L^{\f{7+12\delta_0}{4+12\delta_0}} L^{\f73+4\delta_0}} \sum_j \|2^{(s-3)j} \partial  k^{(-1,j)}\|_{L^{\f73+4\delta_0} L^{\f{14+24\delta_0}{1+12\delta_0}}} \|2^{(s-3)j}\partial \dot k^{(-1,j)}\|_{L^\infty L^2}  \nonumber
\\
& \lesssim \|g\|_{L^\infty L^\infty} \|\partial \dot g\|_{L^{\f{7}{4}} W^{s_1,\f73}} \sum_j \|2^{(s-3-\f16-2\delta_0)j} \partial  k^{(-1,j)}\|^{\f6{7+6\delta_0}}_{L^2L^\infty} \|2^{(s-2)j} \partial  k^{(-1,j)}\|^{\f{1+6\delta_0}{7+6\delta_0}}_{L^\infty L^2} \|2^{(s-3)j}\partial \dot k^{(-1,j)}\|_{L^\infty L^2}  \nonumber
\\
& \lesssim \epsilon \cdot \|(\dot Y, \dot e)\|_{\mathfrak U}^2     \nonumber
\end{align} 
and similarly (since $k_j$ satisfies the same estimates as $\partial k^{(-1,j)}$ and $\dot\omega$ the same estimates as $\partial \dot g$):
\[
\sum_j \Big|  \int_{\{0\le x^0\le T\}} 2^{2(s-3)j}  (g \cdot \partial \dot g)_{\le j}\cdot \partial k_j \cdot \partial_0 \dot k^{(-1,j)} \, dx \Big|, \quad 
\sum_j \Big|  \int_{\{0\le x^0\le T\}} 2^{2(s-3)j}  \dot\omega_{\le j}\cdot \partial k_j \cdot \partial_0 \dot k^{(-1,j)} \, dx \Big| \lesssim \epsilon \cdot \|(\dot Y, \dot e)\|_{\mathfrak U}^2.
\]
Thus, combining the above bounds, we infer:
\begin{equation}\label{Energy estimate dot k -1}
\sum_j \| 2^{(s-3)j} \partial \dot k^{(-1,j)}\|^2_{L^\infty L^2} \lesssim \epsilon \cdot \|(\dot Y, \dot e)\|_{\mathfrak U}^2.
\end{equation}

\item In the case of $\square_{g_{\le j}}\dot k^{(0,j)}$ (see \eqref{Wave system k 0 differences} for its expression), the term $ \dot g_{\le j} \cdot \partial^2 k^{(0,j)}$ can be estimated exactly like \eqref{First new term in dot k equation}, yielding:
\begin{align*}
\sum_j \Big|  \int_{\{0\le x^0\le T\}} 2^{2(s-4)j}\dot g_{\le j} \cdot & \partial^2 k^{(0,j)} \cdot \partial_0 \dot k^{(0,j)} \, dx \Big|  \\
& \lesssim \sum_j  \| \dot g_{\le j}\|_{L^1 L^\infty} \cdot \| 2^{(s-4)j} \partial^2 k^{(0,j)}\|_{L^\infty L^2} \cdot \|2^{(s-4)j} \partial \dot k^{(0,j)}\|_{L^\infty L^2}\\
& \lesssim  \epsilon \cdot \|(\dot Y, \dot e)\|_{\mathfrak U}^2.
\end{align*}
The terms 
\begin{gather}\label{Terms for alternative treatment}
g_{\le j}\cdot \partial \dot g_{\le j} \partial k^{(0,j)}, \quad g_{\le j}\cdot \partial \dot g_{\le j} \partial^2 k^{(-1,j)}, \quad g_{\le j} \cdot \partial \dot g_{\le j} \cdot \partial k_{j}, \quad \bar\partial \big( \partial \dot g_{\le j} \cdot k_{j}\big),\\
 \partial(g_{\le j}) \cdot \big((g \cdot \partial \dot g + \dot\omega)_{\le j} \cdot k_j, \quad \partial(\dot g_{\le j}) \cdot P_j \big((g \cdot \partial g + \omega) \cdot k\big) \quad \text{ and }\quad   \bar\partial \Big(g_{\le j} \cdot ( g \cdot \partial \dot g+ \dot \omega)_{\le j} \cdot k_j\Big)
\end{gather}
 can be estimated in a similar way as \eqref{Estimates for dot k not requiring L1Linfty}. The rest of the terms can be treated in the same way as the analogous terms in  \eqref{Error energy estimates localized}; to this end, it is crucial to note the following two facts about the structure of the right hand side of \eqref{Wave system k 0 differences}:
\begin{itemize}
\item No terms of the form $\partial^2 \dot g$ or $\partial \dot\omega$ appear in the right hand side of \eqref{Wave system k 0 differences}. 
\item In the case of the terms containing high-high interactions, there is always a total spatial derivative or a factor of the form $\partial g_{\le j}$ appearing outside the high-high product, compared to the analogous term in \eqref{Error energy estimates localized} (for which the same derivative is applied to one of the high frequency factors): For instance, compare the term
\[
 \bar\partial \Big( P_j \sum_{j'>j}\big(g \cdot \partial \dot g)_{j'} \cdot k_{j'}\big)\Big)
\]
in \eqref{Wave system k 0 differences} with the corresponding term 
\[
 P_j \sum_{j'>j}\big(g \cdot \partial g)_{j'} \cdot  \partial k_{j'}\big)\Big)
\]
in \eqref{Error energy estimates localized}. Sine we aim to bound $\JapD^{-1}\dot k^{(0,j)}$ in the same energy space as we did for $k_j$ in Lemma \ref{lem:Energy estimates}, this structure allows us to follow the same process for the high-high interactions (using, in particular, the same functional inequalities provided by Lemma \ref{lem:Functional inequalities}) as we did in the proof of Lemma \ref{lem:Energy estimates}. For instance, in the case of the above term, we can bound:
\begin{align*}
 \Big|  \int_{\{0\le x^0\le T\}} 2^{2(s-4)j}\bar\partial \Big( P_j & \sum_{j'>j}\big(g \cdot \partial \dot g)_{j'} \cdot k_{j'}\big)\Big) \cdot \partial_0 \dot k^{(0,j)} \, dx \Big|\\
& \lesssim 
\sum_{j'>j} \int_{\{0\le x^0\le T\}} \big|2^{(s-3)j} (2^{-j}\bar\partial)P_j\big(g \cdot \partial \dot g)_{j'} \cdot k_{j'}\big)\big|\cdot \big|2^{(s-4)j} \partial \dot k^{(0,j)}\big| \, dx.
\end{align*}
The above sum can then be estimated in exactly the same way as we did in \eqref{Second high high}, noting that $2^{-j}\bar\partial P_j f$ satisfies the same $L^p(\mathbb T^3)$ estimates as $P_j f$ when $1<p<+\infty$, the term  $\partial \dot g_{j'}$ is controlled in the same space $L^2 L^2$, $L^{\f74} L^{\f73}$ and $L^\infty L^2$-based Sobolev spaces as $\partial^2 g_{j'}$, while  $2^{(s-4)j} \partial \dot k^{(0,j)}$ is controlled in the same spaces as $2^{(s-3)j} \partial k_j$. The same arguments can be applied to the rest of the high-high interactions (we will omit the details).
\end{itemize}
Overall, we obtain:
\begin{equation}\label{Energy estimate dot k 0}
\sum_j \| 2^{(s-4)j} \partial \dot k^{(0,j)}\|^2_{L^\infty L^2} \lesssim \epsilon \cdot \|(\dot Y, \dot e)\|_{\mathfrak U}^2.
\end{equation}
\end{itemize}

Combining \eqref{Energy estimate dot k -1} and \eqref{Energy estimate dot k 0} and then using the wave equations \eqref{Wave system k -1 differences} and \eqref{Wave system k 0 differences}  to express $\partial_0^2 k^{(-1,j)}$ and $\partial_0^2 k^{(0,j)}$ in terms of lower order terms (in the same way as in the proof of Lemma \ref{lem:Energy estimates}), we infer:
\begin{equation}\label{Energy estimate dot k final}
\sum_{j\in \mathbb N} \sum_{l=0}^2 \| 2^{(s-2-l)j} \partial^l \dot k^{(-1,j)}\|^2_{L^\infty L^2} + \| 2^{(s-3-l)j} \partial^l \dot k^{(0,j)}\|^2_{L^\infty L^2} \lesssim \epsilon \cdot \|(\dot Y, \dot e)\|_{\mathfrak U}^2.
\end{equation}

The Strichartz estimate
\begin{equation}\label{Strichartz estimate dot k final}
\sum_{j\in \mathbb N} \sum_{l=0}^2  \Big(\|2^{(-l-\f32+s_0+\f18\delta_0)j}\partial^l \dot k^{(0,j)}\|^2_{L^2 L^\infty}+ \|2^{(-l-\f12+s_0+\f18\delta_0)j}\partial^l \dot k^{(-1,j)}\|^2_{L^2 L^\infty} \Big) \lesssim \epsilon \cdot \|(\dot Y, \dot e)\|_{\mathfrak U}^2
\end{equation}
can be established by following exactly the same steps as in the proof of Lemma \ref{lem:Strichartz estimates} in the case of equations \eqref{Wave system k -1 differences} and \eqref{Wave system k 0 differences},  employing, in particular, the analogy between $(\JapD^{-1} \dot k^{0,j)}, \dot k^{-1,j})$ and $k_j$ that was used in the proof of \eqref{Energy estimate dot k final} and handling the low-high terms of the form \eqref{Extra terms energy estimates}  and \eqref{Terms for alternative treatment} similarly as before (see also the two remarks below \eqref{Terms for alternative treatment}). We will omit the details.

\bigskip
\noindent \textbf{Estimates for $\dot g$ and $\dot \omega$.} We will first establish the bound 
\begin{equation}\label{Bound L1Linfty dot g}
\|\dot g\|_{L^1 W^{\f14\delta_0, \infty}} \lesssim \epsilon \|(\dot Y, \dot e)\|_{\mathfrak U}.
\end{equation}
To this end, it suffices to show that, for any $p\gg 1$, we have
\begin{equation}\label{Bound L1Linfty Besov dot g}
\|\dot g\|_{L^1 B^{\f12\delta_0}_{p,1}} \lesssim_p \epsilon \|(\dot Y, \dot e)\|_{\mathfrak U} + \epsilon \|\dot g\|_{L^1 B^{\f12\delta_0}_{p,1}}.
\end{equation}

Let us consider first the parabolic equation \eqref{Equation lapse once more differences} for $\dot N$: Using the parabolic estimates provided by Lemma \ref{lem:Parabolic estimates model}, we have:
\begin{align*}
\|\dot N\|_{L^1 B^{\f12\delta_0}_{p,1}} \lesssim_p & \|\dot{\mathfrak F}_{N}^{(0)}+|D|\Delta_{\bar g}^{-1} [(g-m_0)\cdot D\dot g]\|_{L^1 B^{\f12\delta_0}_{p,1}} + \|\dot{\mathfrak F}_{N}^{(1)}\|_{L^1 B^{-1+\f12\delta_0}_{p,1}} \\
& + \Bigg\||D| \Delta_{\bar g}^{-1} \Bigg[ \dot g \cdot D\partial g + (g-m_0)\cdot D^2 \dot g + \dot g \cdot \partial g \cdot \partial g + g \cdot \partial g \cdot \partial\dot g  \Bigg] \Bigg\|_{L^1 B^{-1+\f12\delta_0}_{p,1}}
\\
\lesssim_p & \|\dot{\mathfrak F}_{N}^{(0)}\|_{L^1 B^{\f12\delta_0}_{p,1}}+\|(g-m_0)\cdot D\dot g\|_{L^1 B^{-1+\f12\delta_0}_{p,1}} + \|\dot{\mathfrak F}_{N}^{(1)}\|_{L^1 B^{-1+\f12\delta_0}_{p,1}} \\
& + \Bigg\| \dot g \cdot D\partial g + (g-m_0)\cdot D^2 \dot g + \dot g \cdot \partial g \cdot \partial g + g \cdot \partial g \cdot \partial\dot g  \Bigg\|_{L^1 B^{-2+\f12\delta_0}_{p,1}}.
\end{align*}
We will estimate each term in the right hand side above as follows:
\begin{itemize}
\item For the first term, note that  $\dot{\mathfrak F}_{N}^{(0)}$ is expressed as:
\begin{align*}
\dot{\mathfrak F}_{N}^{(0)} =  \sum_{j\in \mathbb N} \sum_{\substack{j_2 > j_1-2,\\|j_2-j_3|\le 2}} \Bigg[&  |D|  \Delta_{\bar g} ^{-1} \Big(g \cdot  D  P_j \Big( \big( P_{j_1} (m(e)) \cdot  P_{j_2} (\JapD^{-1}\dot k) \cdot P_{j_3} (\JapD^{-1} k)\big) \Big)\Big)\\
&+  |D|  \Delta_{\bar g} ^{-1}\Big( g \cdot  D  P_j \Big( \big( P_{j_1} \dot e \cdot  P_{j_2} (\JapD^{-1} k) \cdot P_{j_3} (\JapD^{-1} k)\big) \Big)  \Big)\\
& 
 |D|  \Delta_{\bar g} ^{-1}\Big( \dot g \cdot  D  P_j \Big( \big( P_{j_1} (m(e)) \cdot  P_{j_2} (\JapD^{-1} k) \cdot P_{j_3} (\JapD^{-1} k)\big) \Big)\Big) \Bigg]\\
& +  |D|\Delta_{\bar g} ^{-1} \Big( \dot g \cdot D \mathfrak F_{N}^{(0)} + D\dot g \cdot  \mathfrak F_{N}^{(0)} \Big).
\end{align*}
We will estimate the above terms as follows:
\begin{align*}
\Bigg\|  \sum_{j\in \mathbb N} & \sum_{\substack{j_2 > j_1-2,\\|j_2-j_3|\le 2}}   |D|  \Delta_{\bar g} ^{-1} \Big(g \cdot  D  P_j \Big( \big( P_{j_1} (m(e)) \cdot  P_{j_2} (\JapD^{-1}\dot k) \cdot P_{j_3} (\JapD^{-1} k)\big) \Big)\Big) \Bigg\|_{L^1 B^{\f12\delta_0}_{p,1}} 
\\
\lesssim_p & 
 \sum_{j\in \mathbb N}   \sum_{\substack{j_2 > j_1-2,\\|j_2-j_3|\le 2}} \|g\|_{L^\infty H^{s-1}} \Big\|P_j \Big( \big( P_{j_1} (m(e)) \cdot  P_{j_2} (\JapD^{-\f32} \dot k) \cdot P_{j_3} (\JapD^{-\f12} k)\big) \Big)\Big\|_{L^1 W^{\delta_0,p}}\\
\lesssim_p & \|g\|_{L^\infty H^{s-1}} \|m(e)\|_{L^\infty H^{s-1}} \| \dot k\|_{L^2 W^{-\f32+2\delta_0,\infty}} \|k \|_{L^2 W^{-\f12+2\delta_0, \infty}} \\
\lesssim & \epsilon \|(\dot Y, \dot e)\|_{\mathfrak U}.
\end{align*}
The rest of the terms in the expression for $\dot{\mathfrak F}_{N}^{(0)}$ can be estimated in a similar way, thus yielding:
\[
\| \dot{\mathfrak F}_{N}^{(0)}\|_{L^1 B^{\f12\delta_0}_{p,1}} \lesssim_p \epsilon \|(\dot Y, \dot e)\|_{\mathfrak U}.
\]
\item For the second term, we can estimate
\[
 \|(g-m_0)\cdot D\dot g\|_{L^1 B^{-1+\f12\delta_0}_{p,1}} \lesssim_p \|g-m_0\|_{L^\infty H^{s-1}} \|\dot g\|_{L^1 B^{\f12\delta_0}_{p,1}}.
\]
\item For the term involving $\dot{\mathfrak F}_{N}^{(1)}$ (note that $\dot{\mathfrak F}_{N}^{(1)}$ is obtained by considering the difference for the two immersions of the expression \eqref{F N 1} for $\mathfrak F_{N}^{(1)}$), we argue similarly as for the term  involving $\dot{\mathfrak F}_{N}^{(0)}$. Note that the we need to estimate  $\dot{\mathfrak F}_{N}^{(1)}$ in a space of one order of differentiability below that in which we estimated  $\dot{\mathfrak F}_{N}^{(0)}$; thus, it is crucial to take advantage of the fact that the high-high terms in the expression for $\dot{\mathfrak F}_{N}^{(1)}$ which are analogous to terms in  $\dot{\mathfrak F}_{N}^{(0)}$ have an extra derivative or a factor of the form $\partial g$ or $\partial e$ outside the corresponding high-high  product (it is at this point that we make use of the expression \eqref{Calculation F natural difference 2 derivatives}  from Lemma \ref{lem:Cancellations}). For instance, for the first two of the terms in the expression for $\dot{\mathfrak F}_{N}^{(1)}$, we can bound:
\begin{align*}
\Bigg\|  \sum_{j\in \mathbb N} & \sum_{\substack{j_2 > j_1-2,\\|j_2-j_3|\le 2}}   |D|  \Delta_{\bar g} ^{-1} \Big(g \cdot  D^2  P_j \Big( \big( P_{j_1} (m(e)) \cdot  P_{j_2} (\JapD^{-1}\dot k) \cdot P_{j_3} (\JapD^{-1} k)\big) \Big)\Big) \Bigg\|_{L^1 B^{-1+\f12\delta_0}_{p,1}} 
\\
\lesssim_p & 
 \sum_{j\in \mathbb N} \sum_{\substack{j_2 > j_1-2,\\|j_2-j_3|\le 2}}  \Bigg\|   \Big(g \cdot  D^2  P_j \Big( \big( P_{j_1} (m(e)) \cdot  P_{j_2} (\JapD^{-1}\dot k) \cdot P_{j_3} (\JapD^{-1} k)\big) \Big)\Big) \Bigg\|_{L^1 W^{-2+\delta_0,p}} 
\\
\lesssim_p & 
\sum_{j\in \mathbb N} \sum_{\substack{j_2 > j_1-2,\\|j_2-j_3|\le 2}} \Bigg[ \|g\|_{L^\infty H^{s-1}} \Big\| D^2  P_j \Big( \big( P_{j_1} (m(e)) \cdot  P_{j_2} (\JapD^{-1}\dot k) \cdot P_{j_3} (\JapD^{-1} k)\big) \Big)\Big) \Big\|_{L^1 W^{-2+\delta_0,\infty}}\\
&\hphantom{\sum_{j\in \mathbb N} \sum_{\substack{j_2 > j_1-2,\\|j_2-j_3|\le 2}} }
 +   \|D^2 g\|_{L^2 H^{\delta_0}} \Big\|  P_j \Big( \big( P_{j_1} (m(e)) \cdot  P_{j_2} (\JapD^{-1}\dot k) \cdot P_{j_3} (\JapD^{-1} k)\big) \Big)\Big) \Big\|_{L^2 W^{\delta_0,6}}  \Bigg]
\\
\lesssim &
 \|g\|_{L^\infty H^{s-1}} \|m(e)\|_{L^\infty H^{s-1}} \|\dot k\|_{L^2 W^{-\f32+\delta_0, \infty}}\| k\|_{L^2 W^{-\f12+\delta_0, \infty}}
+ \|D^2 g\|_{L^2 H^{\delta_0}} \Big\| m(e)\|_{L^\infty H^{s-1}} \|\dot k\|_{L^2 W^{-\f32+\delta_0,\infty}}\|k\|_{L^\infty W^{-\f12+\delta_0, 6}}  \Bigg]
\\
\lesssim & \epsilon \|(\dot Y, \dot e)\|_{\mathfrak U}
\end{align*}
and 
\begin{align*}
\Bigg\|  \sum_{j\in \mathbb N} & \sum_{\substack{j_2 > j_1-2,\\|j_2-j_3|\le 2}}   |D|  \Delta_{\bar g} ^{-1} \Big( \partial g \cdot  D P_j \Big( \big( P_{j_1} (m(e)) \cdot  P_{j_2} (\JapD^{-1}\dot k) \cdot P_{j_3} (\JapD^{-1} k)\big) \Big)\Big) \Bigg\|_{L^1 B^{-1+\f12\delta_0}_{p,1}} 
\\
\lesssim_p & 
 \sum_{j\in \mathbb N} \sum_{\substack{j_2 > j_1-2,\\|j_2-j_3|\le 2}}  \Bigg\|  \Big(\partial g \cdot  D  P_j \Big( \big( P_{j_1} (m(e)) \cdot  P_{j_2} (\JapD^{-1}\dot k) \cdot P_{j_3} (\JapD^{-1} k)\big) \Big)\Big) \Bigg\|_{L^1 W^{-1+\delta_0,3}} 
\\
\lesssim & 
\sum_{j\in \mathbb N} \sum_{\substack{j_2 > j_1-2,\\|j_2-j_3|\le 2}}  \Bigg[\|\partial  g\|_{L^2 W^{\delta_0, 6}} \Big\| D  P_j \Big( \big( P_{j_1} (m(e)) \cdot  P_{j_2} (\JapD^{-1}\dot k) \cdot P_{j_3} (\JapD^{-1} k)\big) \Big)\Big) \Big\|_{L^2 W^{-1+\delta_0,6}}   \\
& \hphantom{\sum_{j\in \mathbb N} \sum_{\substack{j_2 > j_1-2,\\|j_2-j_3|\le 2}}}
+ \|D\partial  g\|_{L^2 H^{\delta_0}} \Big\|   P_j \Big( \big( P_{j_1} (m(e)) \cdot  P_{j_2} (\JapD^{-1}\dot k) \cdot P_{j_3} (\JapD^{-1} k)\big) \Big)\Big) \Big\|_{L^2 W^{\delta_0,6}} 
\Bigg]
\\
\lesssim &
   \|\partial^2  g\|_{L^2 H^{\delta_0}} \Big\| m(e)\|_{L^\infty H^{s-1}} \|\dot k\|_{L^2 W^{-\f32+\delta_0,\infty}}\|k\|_{L^\infty W^{-\f12+\delta_0, 6}}  \Bigg]
\\
\lesssim & \epsilon \|(\dot Y, \dot e)\|_{\mathfrak U}.
\end{align*}
The rest of the terms in the expression for  $\dot{\mathfrak F}_{N}^{(1)}$ can be estimated in a similar way, thus yielding:
\[
\| \dot{\mathfrak F}_{N}^{(1)}\|_{L^1 B^{-1+\f12\delta_0}_{p,1}} \lesssim_p \epsilon \|(\dot Y, \dot e)\|_{\mathfrak U}.
\]

\item For the last term, we can estimate:
\begin{align*}
\Bigg\| \dot g \cdot D\partial g & + (g-m_0)\cdot D^2 \dot g + \dot g \cdot \partial g \cdot \partial g + g \cdot \partial g \cdot \partial\dot g  \Bigg\|_{L^1 B^{-2+\f12\delta_0}_{p,1}}
\\
& \lesssim_p \|(g-m_0)\cdot D^2 \dot g\|_{L^1 B^{-2+\f12\delta_0}_{p,1}}+\Bigg\| \dot g \cdot D\partial  g   + \dot g \cdot \partial g \cdot \partial g + g \cdot \partial g \cdot \partial\dot g  \Bigg\|_{L^1 W^{\delta_0,\f32}}\\
& \lesssim  \|g-m_0\|_{L^\infty H^{s-1}} \| \dot g\|_{L^1 B^{\f12\delta_0}_{p,1}}+ \|\dot g\|_{L^2 W^{\delta_0,6}} \|D \partial g\|_{L^2 H^{\delta_0}} \\
&\hphantom{\|g-m_0\|_{L^\infty H^{s-1}}}
+\|\dot g\|_{L^2 W^{\delta_0,6}} \| \partial g\|_{L^2 W^{\delta_0,6}}  \| \partial g\|_{L^\infty W^{\delta_0,3}}
+ \| g\|_{L^\infty H^{s-1}} \| \partial g\|_{L^2 W^{\delta_0,6}}  \| \partial g\dot \|_{L^2 H^{\delta_0}} \\
& \lesssim_p \epsilon\| \dot g\|_{L^1 B^{\f12\delta_0}_{p,1}} + \epsilon \|(\dot Y,\dot e)\|_{\mathfrak U}.
\end{align*}
\end{itemize}
Combining the above bounds, we obtain that, for any $p>1$,
\begin{equation}\label{Bound L1Linfty Besov dot N}
\|\dot N\|_{L^1 B^{\f12\delta_0}_{p,1}} \lesssim_p \epsilon \|(\dot Y, \dot e)\|_{\mathfrak U} + \epsilon \|\dot g\|_{L^1 B^{\f12\delta_0}_{p,1}}.
\end{equation}

The tensor field $\dot h$ satisfies the elliptic equation \eqref{Littlewood Paley Elliptic h differences}. Therefore, using Lemma \ref{lem:Elliptic estimates model} (in particular, the bound \eqref{Very low regularity elliptic estimate} for $p=3$),  we can estimate for any $j\in \mathbb N$ (recall that $h_j=P_j h$):
\begin{align*}
\|\dot h_j\|_{L^1 W^{\delta_0,3}} \lesssim&  \Big\|\dot g_{\le j} \cdot D^2 h_j + D\dot g_{\le j}\cdot Dh_j +\dot{\mathfrak F}^{(j)}_{h}\Big\|_{L^1 W^{-2+\delta_0,3}}+ \|g-m_0\|_{L^2 H^{2+\f16}} \Big\|D\dot h_j\Big\|_{L^2 H^{-1+\delta_0}}\\
\lesssim & \| \dot{\mathfrak F}^{(j)}_{h}\|_{L^1 W^{-2+\delta_0,3}} + \|\dot g_{\le j} \cdot D^2 h_j + D\dot g_{\le j}\cdot Dh_j\|_{L^1 W^{-1+\delta_0, \f32}}+  2^{-\delta_0 j}\|g-m_0\|_{L^2 H^{2+\f16}} \Big\|\partial \dot g\Big\|_{L^2 H^{2\delta_0}}\\
\lesssim & \| \dot{\mathfrak F}^{(j)}_{h}\|_{L^1 W^{-2+\delta_0,3}} + 2^{-\delta_0 j}\|\dot g\|_{L^2 W^{\delta_0,6}} \|D^2 \partial g\|_{L^2 H^{-1+2\delta_0}} + 2^{-\delta_0 j}\|D\dot g\|_{L^2 H^{\delta_0}} \|D \partial g\|_{L^2 W^{-1+2\delta_0, 6}}\\
&\hphantom{ \| \dot{\mathfrak F}^{(j)}_{h}\|_{L^1 W^{-2+\delta_0,3}}}
+  \|g-m_0\|_{L^2 H^{2+\f16}} \Big\|\partial \dot g\Big\|_{L^2 H^{\delta_0}}\\
\lesssim & \| \dot{\mathfrak F}^{(j)}_{h}\|_{L^1 W^{-2+\delta_0,3}}+ 2^{-\delta_0 j}\epsilon \|(\dot Y, \dot e)\|_{\mathfrak U}.
\end{align*}
In view of the expression \eqref{Relation F j h} for $\mathfrak F^{(j)}_{h}$ (recalling also that $\dot{\mathfrak F}^{(j)}_{h}$ is defined as the difference of the expressions $\mathfrak F^{(j)}_{h}$ for the two embeddings and noting that all the high-high interactions in \eqref{Relation F j h} have a spatial derivative outside the corresponding product), we can readily estimate
\[
 \| \dot{\mathfrak F}^{(j)}_{h}\|_{L^1 W^{-2+\delta_0,3}} \lesssim 2^{-\delta_0 j} \epsilon \epsilon \|(\dot Y, \dot e)\|_{\mathfrak U}.
\]
 Thus, returning to the previous estimate, we obtain:
\begin{equation}\label{Bound L1L3 dot h}
\|\dot h_j\|_{L^1 W^{\delta_0,3}} \lesssim 2^{-\delta_0 j}\epsilon \|(\dot Y, \dot e)\|_{\mathfrak U}.
\end{equation}

The shift difference $\dot \beta$ satisfies the parabolic equation \eqref{Equation shift once more differences}, for which all the terms on the right hand side are similar to the ones appearing in the equation for $\dot N$, with the exception of the additional term $|D|\Delta_{\bar g}^{-1} \Big(g \cdot D\dot h\Big)$. Thus, arguing exactly as in the case of $\dot N$, we obtain for any $p>1$:
\begin{align*}
\|\dot \beta\|_{L^1 B^{\f12\delta_0}_{p,1}} \lesssim_p & \epsilon \|(\dot Y, \dot e)\|_{\mathfrak U} + \epsilon \|\dot g\|_{L^1 B^{\f12\delta_0}_{p,1}} +\| |D|\Delta_{\bar g}^{-1} \Big(g \cdot D\dot h\Big)\|_{L^1 B^{-1+\f12\delta_0}_{p,1}}
\\
\lesssim_p & \epsilon \|(\dot Y, \dot e)\|_{\mathfrak U} + \epsilon \|\dot g\|_{L^1 B^{\f12\delta_0}_{p,1}} + \|\dot h\|_{L^1 W^{\delta_0,3}}
\nonumber
\end{align*}
which, in view of \eqref{Bound L1L3 dot h}, implies that
\begin{equation}\label{Bound L1Linfty Besov dot beta}
\|\dot \beta\|_{L^1 B^{\f12\delta_0}_{p,1}} \lesssim_p  \epsilon \|(\dot Y, \dot e)\|_{\mathfrak U} + \epsilon \|\dot g\|_{L^1 B^{\f12\delta_0}_{p,1}}.
\end{equation}

The tensor field $\dot{\bar g}$ satisfies the elliptic equation \eqref{Equation g bar once more differences}. Therefore, using the elliptic estimates of Lemma \ref{lem:Elliptic estimates model}, we obtain for any $p\in (1, +\infty)$:
\begin{align}\label{Bound dot g bar almost}
\| \dot{\bar g}\| _{L^1 B^{\f12\delta_0}_{p,1}} \lesssim_p & \Big\| g \cdot D^2\dot\beta + \dot g \cdot D^2 g + \dot{\mathfrak F}_{\bar g}  \Big\|_{L^1 B^{-2+\f12\delta_0}_{p,1}} \\
& + \|g-m_0\|_{L^2 H^{2+\f16}} \| \dot{\bar g}\|_{L^2 W^{6,\delta_0}} 
  \nonumber\\
\lesssim_p & 
\|\dot{\mathfrak F}_{\bar g} \|_{L^1 B^{-2+\f12\delta_0}_{p,1}} + \|g\|_{L^\infty H^{s-1}} \| D^2 \dot \beta\|_{L^1 B^{-2+\f12\delta_0}_{p,1}} + \| D g\|_{L^2 W^{\delta_0,6}} \|D\dot \beta\|_{L^2 H^{\delta_0}} + \|\dot g\|_{L^2 W^{\delta_0, 6}} \|D^2 g\|_{L^2 H^\delta_0}
  \nonumber\\
& + \|g-m_0\|_{L^2 H^{2+\f16}} \| \dot g\|_{L^2 W^{6,\delta_0}}   \nonumber
\\
\lesssim_p & \|\dot{\mathfrak F}_{\bar g} \|_{L^1 B^{-2+\f12\delta_0}_{p,1}} + \epsilon \|(\dot Y, \dot e)\|_{\mathfrak U} + \epsilon \|\dot g\|_{L^1 B^{\f12\delta_0}_{p,1}}.  \nonumber
\end{align}
In view of the fact that $ \dot{\mathfrak F}_{\bar g}$ is obtained by considering the difference of the expressions \eqref{F bar g} for the two embeddings, we can readily estimate:
\begin{equation}\label{Bound for dot F g bar}
\|\dot{\mathfrak F}_{\bar g} \|_{L^1 B^{-2+\f12\delta_0}_{p,1}} \lesssim_p \epsilon \|(\dot Y, \dot e)\|_{\mathfrak U}.
\end{equation}
We will illustrate how to obtain the above bound in the case of three of the terms in the expression for $ \dot{\mathfrak F}_{\bar g}$ (the rest following in a similar way):\footnote{As can be seen from the way we handle the bound \eqref{Estimate two D derivatives}, it is necessary for us at this point to use \eqref{Schematic extraction two spatial derivatives wedge product}, namely the fact that one can pull out two spatial derivatives in the high-high interactions in $R_{**} = m(e) \cdot k \wedge k$, which is the part of the Riemann curvature tensor appearing as a source term in the elliptic equation for $\bar g$.} 
\begin{align}\label{Estimate two D derivatives}
\Bigg\| \sum_{j\in \mathbb N} &   \sum_{\substack{j_2 > j_1-2,\\|j_2-j_3|\le 2}} g \cdot D^2  P_j \Big( \big( P_{j_1} (m(e)) \cdot  P_{j_2} (\JapD^{-1}k) \cdot P_{j_3} (\JapD^{-1} \dot k)\big) \Big)  \Bigg\|_{L^1 B^{-2+\f12\delta_0}_{p,1}} \\
\lesssim_p & 
 \sum_{j\in \mathbb N}  \sum_{\substack{j_2 > j_1-2,\\|j_2-j_3|\le 2}} \Bigg[ \|g\|_{L^\infty H^{s-1}} \Big\|D^2  P_j \Big( \big( P_{j_1} (m(e)) \cdot  P_{j_2} (\JapD^{-1}k) \cdot P_{j_3} (\JapD^{-1} \dot k)\big) \Big)  \Big\|_{L^1 B^{-2+\f12\delta_0}_{p,1}}
\nonumber\\
&\hphantom{\sum_{j\in \mathbb N}  \sum_{\substack{j_2 > j_1-2,\\|j_2-j_3|\le 2}} \Bigg[}
+ \|D^2 g\|_{L^2 H^{\delta_0}} \Big\|P_j \Big( \big( P_{j_1} (m(e)) \cdot  P_{j_2} (\JapD^{-1}k) \cdot P_{j_3} (\JapD^{-1} \dot k)\big) \Big)  \Big\|_{L^2 W^{\delta_0,6}}
\Bigg]
\nonumber\\
\lesssim_p & 
 \|g\|_{L^\infty H^{s-1}} \|m(e)\|_{L^\infty H^{s-1}} \|k\|_{L^2 W^{-\f12+\delta_0, \infty}} \|\dot k\|_{L^2 W^{-\f32+\delta_0, \infty}} 
\nonumber\\
& + \|D^2 g\|_{L^2 H^{\delta_0}} \|m(e)\|_{L^\infty H^{s-1}} \|k\|_{L^2 W^{-\f12+\delta_0, \infty}} \|\dot k\|_{L^\infty W^{-\f32+\delta_0, 6}}
\nonumber\\
\lesssim & \epsilon \|(\dot Y, \dot e)\|_{\mathfrak U},\nonumber
\end{align}
\begin{align*}
\Bigg\| \sum_{j\in \mathbb N} &   \sum_{\substack{j_2 > j_1-2,\\|j_2-j_3|\le 2}} g\cdot P_j \big( P_{j_1} (\partial e)  \cdot P_{j_2} k \cdot P_{j_3} (\JapD^{-1} \dot k) \big)  \Bigg\|_{L^1 B^{-2+\f12\delta_0}_{p,1}} \\
\lesssim_p & 
 \sum_{j\in \mathbb N}  \sum_{\substack{j_2 > j_1-2,\\|j_2-j_3|\le 2}} \Bigg[ \|g\|_{L^\infty H^{s-1}} \Big\|P_j \big( P_{j_1} (\partial e)  \cdot P_{j_2} ( \JapD^{-\f12} k) \cdot P_{j_3} (\JapD^{-\f12} \dot k) \big)   \Big\|_{L^1 W^{\delta_0,\f32}}\\
\lesssim_p & 
 \|g\|_{L^\infty H^{s-1}} \|\partial e\|_{L^2 W^{\delta_0,6}} \|k\|_{L^2 W^{-\f12+\delta_0, \infty}} \|\dot k\|_{L^\infty H^{-\f12+\delta_0}} 
\\
\lesssim & \epsilon \|(\dot Y, \dot e)\|_{\mathfrak U}
\end{align*}
and
\begin{align*}
\|g\cdot \partial g \cdot \partial \dot g\|_{L^1 B^{-2+\f12\delta_0}_{p,1}}
\lesssim_p & \|g\cdot \partial g \cdot \partial \dot g\|_{L^1 W^{\delta_0,\f32}} \\
\lesssim & \|g\|_{L^\infty H^{s-1}} \|\partial g\|_{L^2 W^{\delta_0, 6}} \|\partial \dot g\|_{L^2 H^{\delta_0}} \\
\lesssim & \epsilon \|(\dot Y, \dot e)\|_{\mathfrak U}.
\end{align*}
Using \eqref{Bound for dot F g bar} in \eqref{Bound dot g bar almost}, we infer:
\begin{equation}\label{Bound dot g bar Besov}
\| \dot{\bar g}\| _{L^1 B^{\f12\delta_0}_{p,1}} \lesssim_p  \epsilon \|(\dot Y, \dot e)\|_{\mathfrak U} + \epsilon \|\dot g\|_{L^1 B^{\f12\delta_0}_{p,1}}. 
\end{equation}
Combining \eqref{Bound L1Linfty Besov dot N}, \eqref{Bound L1Linfty Besov dot beta} and \eqref{Bound dot g bar Besov}, we finally obtain \eqref{Bound L1Linfty Besov dot g} and, therefore, \eqref{Bound L1Linfty dot g}.

Let us introduce the norm
\[
\|f\|_{\mathfrak D} \doteq \|f\|_{L^2 H^{\f16+s_1}} + \|f\|_{L^{\f74} B^{s_1+\f18\delta_0}_{\f73,\f74}} + \|f\|_{L^\infty H^{s-3}}.
\]
The bound
\begin{equation}\label{Bound elliptic parabolic dot g dot omega almost}
\|\partial \dot g \|_{\mathfrak D} + \|\dot \omega \|_{\mathfrak D} \lesssim \epsilon \| (\dot Y, \dot e)\|_{\mathfrak U} + \epsilon\cdot \Big(\|\partial \dot g \|_{\mathfrak D} + \|\dot \omega \|_{\mathfrak D} \Big)
\end{equation}
can be obtained using the parabolic estimates from Lemma \ref{lem:Parabolic estimates model} (note that we have to use the bounds \eqref{Estimate parabolic LinftyL2 lower order} in the case of the $\partial_0 \dot{\mathfrak F}^{(0)}_{\omega_0}$ term in the right hand side of equation \eqref{Gauge condition omega 0 j differences} for $\omega_0$) and the elliptic estimates from Lemma \ref{lem:Elliptic estimates model}, in a similar way as for the derivation of the corresponding estimates for $\partial g$ and $\omega$ in the proof of Proposition \ref{prop:Bootstrap} (see Section \ref{subsec:Closing bootstrap metric}).  We will omit the details; let us remark once more that, in obtaining similar estimates for $\JapD^{-1} \partial \dot g$ and $\JapD^{-1} \dot\omega$ as we did for $\partial g$ and $\omega$ in the proof of Proposition \ref{prop:Bootstrap},  the fact that the spatial derivatives in the high-high interactions in the equations for $h$, $\bar g$ and $\bar\omega$ appear outside the corresponding products plays a key role. 

From \eqref{Bound elliptic parabolic dot g dot omega almost}, we readily obtain:
\begin{align}\label{Bound elliptic parabolic dot g dot omega}
\|\partial \dot g\|_{L^2 H^{\f16+s_1}} & + \|\partial \dot g\|_{L^{\f74} W^{s_1,\f73}} + \|\partial \dot g\|_{L^\infty H^{s-3}} \\
&\| \dot \omega\|_{L^2 H^{\f16+s_1}} + \|\dot \omega\|_{L^{\f74} W^{s_1,\f73}} + \|\dot \omega\|_{L^\infty H^{s-3}} \lesssim \epsilon \| (\dot Y, \dot e)\|_{\mathfrak U}. \nonumber 
\end{align}

\bigskip
\noindent \textbf{Estimates for $\dot Y$ and $\dot e$.} Using the relations \eqref{Relation Y k once again differences}--\eqref{Relation e omega k once again differences} for $\dot Y$ and $\dot e$, we can readily infer that
\begin{equation}\label{Bound dot e final}
\| \partial \dot e \|_{L^2 H^{\f16+s_1}} + \|  \partial \dot e \|_{L^{\infty} H^{s-3}}+\|\partial \dot e\|_{L^{\f74} W^{s_1, \f73}} \lesssim \epsilon  \| (\dot Y, \dot e)\|_{\mathfrak U}
\end{equation}
and
\begin{equation}\label{Bound dot Y final}
 \sum_{l=0}^2 \| \partial^l \dot Y \|_{L^{\infty}H^{s-l-1}} + \| \partial^2 \dot Y \|_{L^4 W^{-1+\f1{12},4}}+\|\partial^2 \dot Y \|_{L^{\f72} W^{-1,\f{14}3}} \lesssim \epsilon  \| (\dot Y, \dot e)\|_{\mathfrak U}.
\end{equation}

\subsection{Completing the proof of Proposition \ref{prop:Uniqueness parabolic}}
Combining the bounds  \eqref{Energy estimate dot k final},  \eqref{Strichartz estimate dot k final}, \eqref{Bound L1Linfty dot g},  \eqref{Bound elliptic parabolic dot g dot omega}, \eqref{Bound dot e final} and \eqref{Bound dot Y final}, we obtain:
\[
\|(\dot Y, \dot e)\|_{\mathfrak U} \lesssim \epsilon^{\f12} \|(\dot Y, \dot e)\|_{\mathfrak U}.
\]
Therefore, provided $\epsilon$ was chosen small enough in terms of $s$, we infer that
\[
\|(\dot Y, \dot e)\|_{\mathfrak U}=0,
\]
i.e.~that $(Y^{(1)},e^{(1)}) = (Y^{(2)}, e^{(2)})$. This completes the proof of Proposition \ref{prop:Uniqueness parabolic}. 
\qed

\appendix

\section{Functional inequalities on $\mathbb T^3$}
In this section, we will collect a number of functional inequalities that are used repeatedly throughout the proofs of Theorems \ref{thm:Existence} and \ref{thm:Uniqueness}.

\begin{lemma}\label{lem:Functional inequalities}
Let $f_1, f_2 : \mathbb T^3 \rightarrow \mathbb R$ be smooth functions and let $s>\f52$, $\delta \in [0, \f14 (s-\f52)]$ and $p\in [2,+\infty)$. Then, we can estimate
\begin{equation}\label{H s-2 bound}
\| f_1 \cdot f_2 \|_{W^{s-2,p}} \le C \| f_1 \|_{H^{s-1-\delta}} \| f_2\|_{W^{s-2,p}},
\end{equation}

\smallskip
\begin{equation}\label{H s-3 bound}
\| f_1 \cdot f_2 \|_{W^{s-3,p}} \le C \| f_1 \|_{H^{s-2+z-\delta}} \| f_2\|_{W^{s-2-z,p}} \quad \text{for any } z\in[0,1]
\end{equation}
and 
\begin{equation}\label{H s-4 bound}
\| f_1 \cdot f_2 \|_{W^{s-4,p}} \le C \| f_1 \|_{H^{s-2+z-\delta}} \| f_2\|_{W^{s-3-z,p}} \quad \text{for any } z\in[0,1],
\end{equation}
where the constant $C>0$ depends only on $s$ and $p$.
\end{lemma}

\begin{proof}
Let us fix a number $\delta_1 \in (0, s-\f52-\delta)$. In order to establish \eqref{H s-2 bound}--\eqref{H s-4 bound}, we will utilize the high-low decomposition of the Littlewood--Paley projections of the product $f_1 \cdot f_2$:
\begin{align*}
\| f_1 \cdot f_2 \|^2_{W^{s-2,p}} 
& \sim \sum_j 2^{2(s-2)j} \|P_j(  f_1 \cdot f_2) \|^2_{L^p}  \\
& \lesssim \sum_j 2^{2(s-2)j} \|P_{\le j}f_1 \cdot P_j f_2 \|^2_{L^p} + \sum_j 2^{2(s-2)j} \|P_{j}f_1 \cdot P_{\le j} f_2 \|^2_{L^p}\\
 & \hphantom{\lesssim} + \sum_j \Big( \sum_{j'>j} 2^{(s-2)j} \|P_{j} \big( P_{j'}f_1 \cdot P_{j'} f_2\big) \|_{L^p}  \Big)^2 \\
 & \lesssim \sum_j \Big( \| f_1 \|^2_{L^\infty}  \cdot 2^{2(s-2)j} \|P_j f_2 \|^2_{L^p}\Big) + \sum_j \Big( 2^{2(s-\f52+\f3p-\delta)j} \|P_{j}f_1\|^2_{L^p} \cdot 2^{2(\f12-\f{3}{p}+\delta)j}\| P_{\le j} f_2 \|^2_{L^\infty} \Big)\\
 & \hphantom{\lesssim} + \sum_j \Big( \sum_{j'>j} 2^{(s-2)(j-j')-\delta_1 j'} 2^{\delta_1 j'}\|P_{j'}f_1\|_{L^\infty} \cdot 2^{(s-2)j'}\|P_{j'} f_2\|_{L^p}  \Big)^2 \\
 & \lesssim \| f_1 \|^2_{L^\infty} \|f_2\|^2_{W^{s-2,p}} + \|f_1\|^2_{W^{s-\f52+\f3p-\delta,p}} \|f_2\|^2_{W^{\f12-\f3p+\delta,\infty}} +  \| f_1 \|^2_{W^{\delta_1,\infty}} \|f_2\|^2_{W^{s-2,p}}.
\end{align*}
The bound \eqref{H s-2 bound} then follows by using the Sobolev embeddings $H^{s-1-\delta} \hookrightarrow W^{\delta_1,\infty}$, $H^{s-1-\delta}\hookrightarrow W^{s-\f52+\f3p-\delta,p}$ and  $W^{s-2,p} \hookrightarrow W^{\f12-\f3p+\delta,\infty}$ (recalling that $s>\f52+\delta + \delta_1$ and $p\ge 2$).

We can also estimate for any $z\in [0,1]$:
\begin{align*}
\| f_1 \cdot f_2 \|^2_{W^{s-3,p}} 
& \sim \sum_j 2^{2(s-3)j} \|P_j(  f_1 \cdot f_2) \|^2_{L^p}  \\
& \lesssim \sum_j 2^{2(s-3)j} \|P_{\le j}f_1 \cdot P_j f_2 \|^2_{L^p} + \sum_j 2^{2(s-3)j} \|P_{j}f_1 \cdot P_{\le j} f_2 \|^2_{L^p}\\
 & \hphantom{\lesssim} + \sum_j \Big( \sum_{j'>j} 2^{(s-3)j} \|P_{j} \big( P_{j'}f_1 \cdot P_{j'} f_2\big) \|_{L^p}  \Big)^2 \\
 & \lesssim \sum_j \Big( 2^{-2(1-z)j}\|P_{\le j} f_1 \|^2_{L^\infty}  \cdot 2^{2(s-2-z)j} \|P_j f_2 \|^2_{L^p}\Big)\\
&\hphantom{\lesssim}
 + \sum_j \Big( 2^{2(s-\f72+\f3p+z-\delta)j} \|P_{j}f_1\|^2_{L^p} \cdot 2^{2(\f12-\f3p-z+\delta)j}\| P_{\le j} f_2 \|^2_{L^\infty} \Big)\\
 & \hphantom{\lesssim} + \sum_j \Big( \sum_{j'>j} 2^{(s-\f52)j} \|P_{j} \big( P_{j'}f_1 \cdot P_{j'} f_2\big) \|_{L^{\f{6p}{6+p}}}  \Big)^2 \\
 & \lesssim \| f_1 \|^2_{W^{-1+z,\infty}} \|f_2\|^2_{W^{s-2-z,p}} + \|f_1\|^2_{W^{s-\f72+\f3p+z-\delta,p}} \|f_2\|^2_{W^{\f12-\f3p-z+\delta,\infty}} \\
 & \hphantom{\lesssim} + \sum_j \Big( \sum_{j'>j} 2^{(s-\f52)(j-j')-\delta_1 j'} 2^{(-\f12+z+\delta_1)j'}\| P_{j'}f_1\|_{L^6} \cdot 2^{(s-2-z)j'}\|P_{j'} f_2\|_{L^p}  \Big)^2 \\
 & \lesssim \| f_1 \|^2_{W^{-1+z,\infty}} \|f_2\|^2_{W^{s-2-z,p}} + \|f_1\|^2_{W^{s-\f72+\f3p+z-\delta,p}} \|f_2\|^2_{W^{\f12-\f3p-z+\delta,\infty}} + \| f_1\|^2_{W^{-\f12+z+\delta_1,6}} \| f_2\|^2_{W^{s-2-z,p}}, 
\end{align*}
from which the bound \eqref{H s-3 bound} follows by using the Sobolev embeddings $H^{s-2+z-\delta} \hookrightarrow W^{-1+z,\infty}$, $H^{s-2+z-\delta}\hookrightarrow W^{s-\f72+\f3p+z-\delta,p}$,  $W^{s-2-z,p} \hookrightarrow W^{\f12-\f3p-z+\delta,\infty}$ and $H^{s-2+z-\delta} \hookrightarrow W^{-\f12+z+\delta_1,6}$ (recall that $p\ge 2$).

Similarly:
\begin{align*}
\| f_1 \cdot f_2 \|^2_{W^{s-4,p}} 
& \sim \sum_j 2^{2(s-4)j} \|P_j(  f_1 \cdot f_2) \|^2_{L^p}  \\
& \lesssim \sum_j 2^{2(s-4)j} \|P_{\le j}f_1 \cdot P_j f_2 \|^2_{L^p} + \sum_j 2^{2(s-4)j} \|P_{j}f_1 \cdot P_{\le j} f_2 \|^2_{L^p}\\
 & \hphantom{\lesssim} + \sum_j \Big( \sum_{j'>j} 2^{(s-4)j} \|P_{j} \big( P_{j'}f_1 \cdot P_{j'} f_2\big) \|_{L^p}  \Big)^2 \\
 & \lesssim \sum_j \Big( 2^{-2(1-z)j}\|P_{\le j} f_1 \|^2_{L^\infty}  \cdot 2^{2(s-3-z)j} \|P_j f_2 \|^2_{L^p}\Big) \\
&\hphantom{\lesssim}
+ \sum_j \Big( 2^{2(s-\f72+\f3p+z-\delta)j} \|P_{j}f_1\|^2_{L^p} \cdot 2^{2(-\f12-\f3p-z+\delta)j}\| P_{\le j} f_2 \|^2_{L^\infty} \Big)\\
 & \hphantom{\lesssim} + \sum_j \Big( \sum_{j'>j} 2^{(s-\f52)j} \|P_{j} \big( P_{j'}f_1 \cdot P_{j'} f_2\big) \|_{L^{\f{2p}{p+2}}}  \Big)^2 \\
 & \lesssim \| f_1 \|^2_{W^{-1+z,\infty|}} \|f_2\|^2_{W^{s-3-z,p}} + \|f_1\|^2_{W^{s-\f72+\f3p+z-\delta,p}} \|f_2\|^2_{W^{-\f12-\f3p-z+\delta,\infty}} \\
 & \hphantom{\lesssim} + \sum_j \Big( \sum_{j'>j} 2^{(s-\f52)(j-j')-\delta_1 j'} 2^{(\f12+z+\delta_1)j'}\| P_{j'}f_1\|_{L^2} \cdot 2^{(s-3-z)j'}  \|P_{j'} f_2\|_{L^p}  \Big)^2 \\
 & \lesssim \| f_1 \|^2_{W^{-1+z,\infty|}} \|f_2\|^2_{W^{s-4-z,p}} + \|f_1\|^2_{W^{s-\f72+\f3p+z-\delta,p}} \|f_2\|^2_{W^{-\f12-\f3p-z+\delta,\infty}} + \| f_1\|^2_{H^{\f12+z+\delta_1}} \| f_2\|^2_{W^{s-2-z,p}}, 
\end{align*}
from which the bound \eqref{H s-4 bound} follows by using the Sobolev embeddings $H^{s-2+z-\delta,p} \hookrightarrow W^{-1+z,\infty}$,  $H^{s-2+z-\delta}\hookrightarrow W^{s-\f72+\f3p+z-\delta,p}$ and $W^{s-3-z,p} \hookrightarrow W^{-\f12-\f3p-z+\delta,\infty}$  (recall, again, that $p\ge 2$).

\end{proof}

\section{Auxiliary estimates for parabolic and elliptic equations}
In this section, we will gather a few fundamental results regarding solutions of parabolic and elliptic equations similar to the ones appearing in Lemma \ref{lem:Parabolic elliptic system}.

\begin{lemma}\label{lem:Parabolic estimates model}
Let $f:[0,T)\times \mathbb T^3 \rightarrow \mathbb R$ be a smooth solution of the initial value problem
\begin{equation}\label{Parabolic IVP}
\begin{cases}
\partial_0 f + |D|f = \mathfrak F,\\
f(0,\cdot) = \mathfrak D(\cdot).
\end{cases}
\end{equation}
Then, the following estimates hold:
\begin{align}
\| \partial f \|_{L^\infty L^2} & \lesssim   \| \mathfrak D  \|_{H^1}  +\|\mathfrak F \|_{L^\infty L^2}  \label{Estimate parabolic LinftyL2}\\
\| \partial f \|_{L^2 L^2 } & \lesssim  \| \mathfrak D \|_{H^{\f12}} + \| \mathfrak F \|_{L^2 L^2}    \label{Estimate parabolic L2L2}\\
 \| \partial f\|_{L^q B^0_{p,q}} & \lesssim_{p,q} \| \mathfrak D \|_{B^{1-\f1q}_{p,q}}  + \| \mathfrak F \|_{L^q B^0_{p,q}}   \label{Estimate parabolic L1Lp}
\end{align}
for any $q\in [1,+\infty)$ and $p\in (1,+\infty)$ (see \eqref{Besov norm} for the definition of the Besov space $B^s_{p,r}$). Moreover, if $\mathfrak F$ is of the form
\[
\mathfrak F =  \partial_0 \mathfrak F_0,
\]
we can estimate
\begin{align}
\| f \|_{L^\infty L^2} & \lesssim   \| \mathfrak D - \mathfrak F_0|_{x^0=0} \|_{L^2}  +\|\mathfrak F_0 \|_{L^\infty L^2}  \label{Estimate parabolic LinftyL2 lower order}\\
\| f \|_{L^2 L^2 } & \lesssim  \| \mathfrak D - \mathfrak F_0|_{x^0=0} \|_{H^{-\f12}} + \| \mathfrak F_0 \|_{L^2 L^2}    \label{Estimate parabolic L2L2 lower order}\\
 \| f\|_{L^q B^0_{p,q}} & \lesssim_{p,q} \| \mathfrak D - \mathfrak F_0|_{x^0=0} \|_{B^{-\f1q}_{p,q}}  + \| \mathfrak F_0 \|_{L^q B^0_{p,q}}  \label{Estimate parabolic L1Lp lower order}
\end{align}
 for any $q\in [1,+\infty)$ and $p\in (1,+\infty)$.
\end{lemma}

\begin{proof}
Performing a Fourier transform in the spatial variables,  \eqref{Parabolic IVP} becomes
\begin{equation}\label{Parabolic Fourier}
\begin{cases}
\partial_0 \hf +|\xi| \hf = \hat{\mathfrak F},\\
\hf(0, \xi) = \hat{\mathfrak D}(\xi). 
\end{cases}
\end{equation}
The solution to the initial value problem \eqref{Parabolic Fourier} can be written down explicitly:
\begin{equation}\label{Explicit solution}
\hf(x^0, \xi) =  B(x^0, \xi) \hat{\mathfrak D}(\xi) +  \int_{0}^{x^0} G (s;x^0,\xi) \hat{\mathfrak F}(s,\xi) \, ds,
\end{equation}
where
\[
B (x^0,\xi) \doteq  e^{-|\xi|x^0}
\]
and
\[
G(s;x^0,\xi) \doteq 
\begin{cases}
e^{-|\xi| (x^0-s)}, \quad s \le x^0, \\
0, \quad s>x^0.
\end{cases}
\]
In view of the bounds
\begin{equation}\label{Norm boundary operators}
\sup_{x^0\in[0,T]} \| B (x^0, \xi) \|_{L^\infty_\xi} \lesssim 1 
\end{equation}
and
\[
\sup_{x^0\in[0,T]}  \int_{0}^{T}\big(  \| |\xi|  G(s;x^0,\xi) \|_{L^\infty_\xi} \big) \, ds   \lesssim 1,
\]
we deduce from the expression \eqref{Explicit solution} that
\[
\| |\xi|  \hat{f} \|_{L^\infty_{x^0} L^2_\xi} \lesssim \||\xi| \hat{\mathfrak D}\|_{L^2_\xi} + \| \hat{\mathfrak F}\|_{L^\infty_{x^0} L^2_\xi}.
\]
Combining the above inequality with the trivial estimate
\[
\| \partial_0 \hf \|_{L^\infty_{x^0} L^2_\xi} \le \| |\xi| \hf \|_{L^\infty_{x^0} L^2_\xi} + \| \mathfrak F \|_{L^\infty_{x^0} L^2_\xi}
\]
(obtained directly from the expression for $\partial_0 \hf$ from \eqref{Parabolic Fourier}),
we therefore deduce that
\begin{equation}
\| |\xi|  \hat{f} \|_{L^\infty_{x^0} L^2_\xi} + \| \partial_0 \hf \|_{L^\infty_{x^0} L^2_\xi} \lesssim \||\xi| \hat{\mathfrak D}\|_{L^2_\xi} + \| \hat{\mathfrak F}\|_{L^\infty_{x^0} L^2_\xi}
\end{equation}
from which \eqref{Estimate parabolic LinftyL2} immediately follows.

For any $k\in \mathbb{N}$ and for any $t,s\in[0,T]$, let us define the Fourier multiplier operators $\mathcal{B}^k(t)$ and $\mathcal{G}^k(s;t)$ on $\mathbb{T}^3$ so that, for any $h:\mathbb{T}^3 \rightarrow \mathbb{R}$:
\[
\widehat{\mathcal{B}^k (t) h}(\xi) = B (t,\xi) \tilde{\chi}(2^{-k}|\xi|)\hat{h}(\xi)
\]
and
\[
\widehat{\mathcal{G}^k (s;t) h}(\xi) = G (s;t,\xi) \tilde{\chi}(2^{-k}|\xi|)\hat{h}(\xi),
\]
where $\tilde{\chi}:\in C_0^\infty[0,+\infty)$ is a bump function supported on $[\f14, 4]$ with $\tilde{\chi} \equiv 1$ on $[\f12, 2]$. 
Note that the expression \eqref{Explicit solution} implies that
\begin{equation}\label{Solution localized}
P_k f(x^0, \bar{x}) = \mathcal{B}^k (x^0) P_k \mathfrak D (\bar{x}) + \int_{0}^{T} \mathcal{G}^k (s;x^0) P_k \mathfrak F(s,\bar{x})\,ds,
\end{equation}
where $\bar{x}=(x^1, x^2, x^3)$.

Notice that $\mathcal{B}^k(x^0)$ and $\mathcal{G}^k(s;x^0)$ are Mikhlin operators. In particular, since
\begin{equation} \label{Mikhlin boundary}
\sum_{j=0}^3  \big\| |\xi|^{1+j} \partial_\xi^{j} \big( B (x^0,\xi) \tilde{\chi}(2^{-k}|\xi|)\big) \big\|_{L^\infty_\xi} \lesssim 2^{k} \big( 1+ (2^k x^0)^3\big) e^{-2^k x^0} 
\end{equation}
and
\begin{align*}
\sum_{j=0}^3  \big\| |\xi|^{1+j} \partial_\xi^{j} &  \big( G (s;x^0,\xi) \tilde{\chi}(2^{-k}|\xi|)\big) \big\|_{L^\infty_\xi} \\
& \lesssim 2^{k} \big( 1+(2^k|s-x^0|)^3\big)e^{-2^k|s-x^0|}  
\end{align*}
uniformly in $k$, we can estimate for any $s<x^0 \in[0,T]$ and any $p\in (1,+\infty)$:
\begin{equation}
 \|  \partial_{\bar{x}} \big( \mathcal{B}^k(x^0) P_k \mathfrak D (\bar{x}) \big) \|_{L^p_{\bar{x}}} 
\lesssim_{p}  2^{k} \big( 1+ (2^k x^0)^3\big) e^{-2^k x^0}  \| P_k \mathfrak D (\bar{x}) \|_{L^p_{\bar{x}}}
\end{equation}
and
\begin{equation}
\|  \partial_{\bar{x}} \big( \mathcal{G}^k(s;x^0) P_k \mathfrak F (s,\bar{x}) \big) \|_{L^p_{\bar{x}}} 
\lesssim_{p} 2^{k} \big( 1+(2^k|s-x^0|)^3\big)e^{-2^k|s-x^0|}  \| P_k  \mathfrak F (s,\bar{x})  \|_{L^p_{\bar{x}}}.
\end{equation}
Therefore, using the expression \eqref{Solution localized} for $P_k f$, we can bound for any $q\in [1,+\infty)$:
\begin{align}\label{L1Lp}
\|  \partial_{\bar{x}} \big( P_k f  &  \big) \|^q_{L^q L^p} 
 = \int_{0}^{T} \| \partial_{\bar x} \big( P_k f\big)(x^0,\bar{x}) \|^q_{L^p_{\bar{x}}} \, dx^0 \\
& \lesssim   \int_{0}^{T} \| \partial_{\bar{x}} \big( \mathcal{B}^k (x^0) P_k \mathfrak D(\bar{x}) \big) \|^q_{L^p_{\bar{x}}} \, dx^0  + \int_{0}^{T} \Bigg(\int_{0}^{T} \| \partial_{\bar{x}} \big( \mathcal{G}^k(s;x^0) P_k \mathfrak F(s,\bar{x}) \big) \|_{L^p_{\bar{x}}}  \,ds\Bigg)^q dx^0    \nonumber \\
& \lesssim_p   2^{qk} \int_{0}^{T}  \big( 1+ (2^k x^0)^3\big)^q e^{-2^k q x^0}  \,dx^0  \cdot \| P_k \mathfrak D \|^q_{L^p}  \nonumber \\ 
& \hphantom{andan} +  \int_{0}^{T} \Big(\int_{0}^{T}2^k   \big( 1+(2^k|s-x^0|)^3\big)e^{-2^k|s-x^0|}  \| P_k  \mathfrak F (s,\bar{x})  \|_{L^p_{\bar{x}}} \, ds \Big)^{\f1q}  \,dx^0  \nonumber \\
& \lesssim_q  2^{(q-1)k}\|P_k \mathfrak D \|^q_{L^p} + \Big( \int_\mathbb R 2^k \big(1+(2^k |y|)^3\big) e^{-2^k|y|}\, dy\Big)^q \cdot \Big( \int_0^T \| P_k  \mathfrak F (s,\bar{x})  \|^q_{L^p_{\bar{x}}} \, ds \Big)
  \nonumber \\
& \lesssim 2^{(q-1)k}\|P_k \mathfrak D \|^q_{L^p} + \|P_k \mathfrak F\|^q_{L^q L^p}, \nonumber 
\end{align}
from which \eqref{Estimate parabolic L1Lp} follows considering the sum of the above expression over $k\in \mathbb{N}$ (and then raising to the power $\f1q$), using  also the relation
\begin{equation}\label{Identity time derivative}
\partial_0 f = -|D| f + \mathfrak F
\end{equation}
for $\partial_0 f$.

Similarly, we can estimate :
\begin{align}\label{L2L2}
\|\partial_{\bar{x}} \big( P_k f  & \big) \|^2_{L^2 L^2}  
 \lesssim   \int_{0}^{T} \|  \partial_{\bar{x}} \big( \mathcal{B}^k(x^0) P_k \mathfrak D (\bar{x}) \big) \|^2_{L^2_{\bar{x}}} \, dx^0 \\
&\hphantom{andandan}+ \int_{0}^T \Big( \int_0^T \|  \partial_{\bar{x}} \big( \mathcal{G}^k (s;x^0) P_k \mathfrak F(s,\bar{x}) \big) \|_{L^2_{\bar{x}}}  \,ds \Big)^2 \,dx^0    \nonumber \\
& \lesssim    2^{2k} \Big(\int_{0}^{T} \big( \big( 1+ (2^k x^0)^3\big) e^{-2^k x^0} \big)^2 \,dx^0  \Big) \| P_k \mathfrak D \|^2_{L^2}  \nonumber \\ 
& \hphantom{andan} +  \int_{0}^{T} 2^{2k} \Big(\int_{0}^{T}  \big( 1+(2^k|s-x^0|)^3\big)e^{-2^k|s-x^0|} \| P_k  \mathfrak F(s,\bar{x})  \|_{L^2_{\bar{x}}} \,ds \Big)^2  \,dx^0 \nonumber \\
& \lesssim  2^{k} \|P_k \mathfrak D \|^2_{L^2} \nonumber \\
& \hphantom{andan} + 2^{2k} \int_{0}^{T} \Big(\int_{0}^{T} \big( 1+(2^k|s-x^0|)^3\big)e^{-2^k|s-x^0|} \, ds \Big)  \nonumber \\
& \hphantom{andandanda} \times  \Big( \int_{0}^{T} \big( 1+(2^k|s-x^0|)^3\big)e^{-2^k|s-x^0|} \| P_k  \mathfrak F(s,\bar{x})  \|^2_{L^2_{\bar{x}}} \,ds \Big)  \,dx^0 \nonumber \\
& \lesssim  2^{ k} \|P_k \mathfrak D \|^2_{L^2} \nonumber \\ 
& \hphantom{andan} + 2^k \int_{0}^{T} \Big( \int_{0}^T \big( 1+(2^k|s-x^0|)^3\big)e^{-2^k|s-x^0|} \,dx^0 \Big) \| P_k  \mathfrak F(s,\bar{x})  \|^2_{L^2_{\bar{x}}} \,ds \nonumber \\
& \lesssim 2^{ k} \|P_k \mathfrak D \|^2_{L^2} + \| P_k \mathfrak F \|^2_{L^2 L^2} \nonumber 
\end{align}
from which \eqref{Estimate parabolic L2L2} follows after summing over $k$ and using \eqref{Identity time derivative} for $\partial_0 f$.

Finally, in the case when $\mathfrak F = \partial_0 \mathfrak F_0$, the function $y = \JapD^{-1}\big(f-\mathfrak F_0\big)$ satisfies the initial value problem
\begin{equation}\label{Parabolic IVP y}
\begin{cases}
\partial_0 y + |D| y = \JapD^{-1}|D| \mathfrak F_0,\\
y(0,\cdot) = \JapD^{-1} \big(\mathfrak D(\cdot) - \mathfrak F_0(0, \cdot)\big).
\end{cases}
\end{equation}
Therefore, the estimates \eqref{Estimate parabolic LinftyL2}--\eqref{Estimate parabolic L1Lp} applied to \eqref{Parabolic IVP y} immediately yield: 
\begin{align}
\| \JapD y \|_{L^\infty L^2} & \lesssim   \| \JapD^{-1} \big(\mathfrak D(\cdot) - \mathfrak F_0(0, \cdot)\big)  \|_{H^1}  +\| \JapD^{-1}|D| \mathfrak F_0 \|_{L^\infty L^2}  \label{Estimate parabolic LinftyL2 y}\\
\| \JapD y \|_{L^2 L^2 } & \lesssim  \|\JapD^{-1} \big(\mathfrak D(\cdot) - \mathfrak F_0(0, \cdot)\big) \|_{H^{\f12}} + \|  \JapD^{-1}|D| \mathfrak F_0 \|_{L^2 L^2}    \label{Estimate parabolic L2L2 y}\\
 \| \JapD y\|_{L^q B^0_{p,q}} & \lesssim_{p,q} \| \JapD^{-1} \big(\mathfrak D(\cdot) - \mathfrak F_0(0, \cdot)\big) \|_{B^{1-\f1q}_{p,q}}  + \| \JapD^{-1}|D| \mathfrak F_0 \|_{L^q B^0_{p,q}}   \label{Estimate parabolic L1Lp y}
\end{align}
for any $p\in (1,+\infty)$. The estimates \eqref{Estimate parabolic LinftyL2 y}--\eqref{Estimate parabolic L1Lp y} follow directly from  \eqref{Estimate parabolic LinftyL2 y}--\eqref{Estimate parabolic L1Lp y} and the trivial inequality
\[
\|f\|_X \le \|f-\mathfrak F_0\|_X+\| \mathfrak F_0\|_X = \|\JapD y\|_X +\mathfrak F_0\|_X 
\]
for $X$ being any of the spaces $L^\infty L^2, L^2L^2, L^q B^0_{p,q}$.

\end{proof}

\begin{lemma}\label{lem:Elliptic estimates model}
Let $\bar g$ be a metric on $\mathbb T^3$ and $f: \mathbb T^3 \rightarrow \mathbb R$ be a smooth solution of the elliptic equation
\begin{equation}\label{Elliptic model}
\mathfrak L f = \mathfrak F,
\end{equation}
where
\[
\mathfrak L = \Delta_{\bar g} \quad \text{or} \quad \mathfrak L = \bar g^{ij} \partial_i \partial_j.
\]
For any $p\in (1,+\infty)$, any $r\in [1,+\infty]$, any $\bar s>\f32$ and any $\sigma \in \big(-\min\{\bar s, 3\},\, \bar s  -\max\{0, \f32-\f3p\} \big)$, there exist constants $c_0 = c_0(p,\sigma,\bar s) \in (0,1)$ and $C_1= C_1(p,\sigma,\bar s) >1$ such that, if the metric $\bar g$ satisfies in the product coordinates on $\mathbb T^3$
\begin{equation}\label{Bound model Riemannian metric}
\| \bar g _{ij} - \delta_{ij}\|_{H^{\bar s}} \le c_0,
\end{equation}
then the following estimate holds for $f$:
\begin{align}
\| \bar \partial  f \|_{B^{\sigma + 1}_{p,r}}  \le C_1    \|\mathfrak F \|_{B^{\sigma}_{p,r}}. \label{Estimate elliptic}
\end{align}
Moreover, if $\bar s<2$, then for $\sigma \in(-2, -\bar s)$ we can estimate
\begin{equation}\label{Very low regularity elliptic estimate}
\|\bar\partial f\|_{B^{\sigma+1}_{p,r}} \le C_1    \Big( \|\mathfrak F \|_{B^{\sigma}_{p,r}} + \mathfrak B_{\bar g } \cdot \|\bar\partial f\|_{W^{-1,\frac{6p}{(2|\sigma|-3)p+6}}}\Big),
\end{equation}
where
\[
\mathfrak B_{\bar g } \doteq \| \bar g _{ij} - \delta_{ij} \|_{H^{\bar s+\f12}}.
\]
\end{lemma}

\begin{remark*}
In the proof of Theorem \ref{thm:Existence}, we use \ref{lem:Elliptic estimates model} above for $\bar s = s-1$ and for 
\[
(p, \sigma) = \big\{ (2, s-3), (2, s-4), (2, \f16+s_1),  (2, -\f56+s_1), (\f73,s_1), (\f73,-1+s_1),(\infty, -1+\delta_0), (\infty, -2+\delta_0)\big\}.
\]
 Note that, in this context, the bound \ref{Very low regularity elliptic estimate} is relevant only for the last pair of exponents (and, in the formalism of Theorem \ref{thm:Existence}, is used in proving the $L^1L^\infty$ estimates for the second fundamental form $h$ and the spatial components $\bar\omega$ of the connection 1-form).
\end{remark*}

\begin{proof}
Let us show \eqref{Estimate elliptic} in the case when $\mathfrak L = \bar g^{ij} \partial_i \partial_j$ (the case when $\mathfrak L = \Delta_{\bar g}$ follows in exactly the same way). Setting
\[
b^{ij} = \delta^{ij}-\bar g^{ij}
\]
and reexpressing $\mathfrak Lf = \mathfrak F$ as
\[
\Delta_{\mathbb T^3} f = \mathfrak F +b^{ij} \partial_i \partial_j f,
\]
 we can readily infer  after applying a Littlewood--Paley projection $P_j$ on both sides of the equation for $j\ge 1$:
\begin{equation*}
2^{(\sigma+2)j} \|P_j f\|_{L^p} \lesssim_{\sigma} 2^{\sigma j}\|P_j \mathfrak F\|_{L^p} + 2^{\sigma j} \|P_j (b \cdot \bar\partial^2 f) \|_{L^p}. 
\end{equation*}
Considering the $\ell^r$ sum (with respect to $j$) of the above expression, we obtain
\begin{equation}\label{Besov elliptic estimate}
\|\bar\partial f\|_{B^{\sigma+1}_{p,r}} \lesssim_{\sigma} \|\mathfrak F\|_{B^\sigma_{p,r}} + \| b\cdot \bar\partial^2 f \|_{B^\sigma_{p,r}}.
\end{equation}

For any pair of smooth functions $h_1, h_2:\mathbb T^3 \rightarrow \mathbb R$, using a high-low decomposition we can establish the following bilinear estimates for any  $p\in (1,+\infty)$ and $\delta \in (0, \f12\min\{\bar s-\f32,\f3p\})$:
\begin{itemize}
\item For $\sigma\in [0, \bar s  -\max\{0, \f32-\f3p\})$, if we set $\sigma_* = \min\{ \sigma, \f3p\}-\delta$, we have:
\begin{align*}
2^{\sigma j}\|P_j(h_1 \cdot h_2) \|_{L^p} & \lesssim 2^{\sigma j} \|P_{\le j} h_1 \cdot P_j h_2\|_{L^p} + 2^{\sigma j} \|P_{j} h_1 \cdot P_{\le j} h_2\|_{L^p} + 2^{\sigma j} \|\sum_{j'>j} P_j \big( P_{j'} h_1 \cdot P_{j'} h_2\big)\|_{L^p}\\ 
& \lesssim_{\sigma, p}  \|h_1\|_{L^\infty} \cdot 2^{\sigma j} \|P_j h_2\|_{L^p} + 2^{\sigma j}\|P_j h_1\|_{L^{\f3{\sigma_*}}} \cdot \| h_2\|_{L^{\f{3p}{3-\sigma_* p}}} \\
& \hphantom{\lesssim_{\sigma, p}  \|h_1\|_{L^\infty} \cdot 2^{\sigma j}}
+  \Big( \sum_{j'>j} \|P_{j'} h_1\|_{L^\infty}\Big) \cdot \sup_{j'} \big( 2^{\sigma j'} \|P_{j'} h_2\|_{L^p} \big).
\end{align*}
Therefore, for any $r\in [1,+\infty]$, using the Sobolev embeddings $H^{\bar s} \hookrightarrow B^{\delta}_{\infty,1}\cap B^{\sigma}_{\frac{3}{\sigma_*},1}$ and $B^{\sigma}_{p,r} \hookrightarrow B^{\sigma_*}_{p,1} \hookrightarrow L^{\f{3p}{3-\sigma_* p}}$, we obtain  after considering the $\ell^r$ sum of the above expression:
\[
\|h_1 \cdot h_2\|_{B^{\sigma}_{p,r}} \lesssim_{\sigma, \bar s, p} \|h_1\|_{H^{\bar s}} \|h_2\|_{B^{\sigma}_{p,r}}.
\]

\smallskip
\item For $\sigma \in (-\min\{\bar s,3\},0]$, we can bound using the Sobolev embedding theorem for the high-high interactions:
\begin{align}\label{Negative Sobolev exponent}
2^{\sigma j}\|P_j & (h_1 \cdot h_2) \|_{L^p} \\
&  \lesssim 2^{\sigma j} \|P_{\le j} h_1 \cdot P_j h_2\|_{L^p} + 2^{\sigma j} \|P_{j} h_1 \cdot P_{\le j} h_2\|_{L^p} + 2^{\sigma j} \|\sum_{j'>j} P_j \big( P_{j'} h_1 \cdot P_{j'} h_2\big)\|_{L^p} \nonumber \\ 
& \lesssim_{\sigma, p, \delta} 
\|h_1\|_{L^{\infty}} \cdot 2^{\sigma j}\|P_j h_2\|_{L^p}+2^{\f12\delta j}\|P_j h_1\|_{L^\infty} \cdot 2^{-(\f12\delta+|\sigma|)j}\big(\sum_{j'\le j}\|P_{j'} h_2\|_{L^p}\big)  \nonumber\\
&\hphantom{\lesssim_{\sigma, p, \delta} \|h_1\|_{L^{\infty}} }
 + \|\sum_{j'>j} P_j \big( P_{j'} h_1 \cdot P_{j'} h_2\big)\|_{L^{\frac{3p}{3+|\sigma|p}}}   \nonumber \\
& \lesssim_{\sigma, p, \delta} 
\|h_1\|_{L^{\infty}} \cdot 2^{\sigma j}\|P_j h_2\|_{L^p}+2^{\f12\delta j}\|P_j h_1\|_{L^\infty} \cdot 2^{-(\f12\delta+|\sigma|)j}\big(\sum_{j'\le j}\|P_{j'} h_2\|_{L^p}\big)  \nonumber\\
&\hphantom{\lesssim_{\sigma, p, \delta} \|h_1\|_{L^{\infty}} }
 + \Big( \sum_{j'>j} 2^{|\sigma|j'} \|P_{j'}h_1\|_{L^{\f3{|\sigma|}}}\Big) \cdot \sup_{j'}\big( 2^{-|\sigma|j'} \|P_{j'} h_2\|_{L^p} \big).    \nonumber 
\end{align}
For any $r\in [1,+\infty]$, we obtain after considering the $\ell^r$ sum of the above expression and using the Sobolev embeddings $H^{\bar s} \hookrightarrow B^{\f12 \delta}_{\infty,1}  \cap B^{|\sigma|}_{\frac{3}{|\sigma|},1}$ and $B^{\sigma}_{p,r}\hookrightarrow B^{\sigma-\f12\delta}_{p,1}$:
\[
\|h_1 \cdot h_2\|_{B^{\sigma}_{p,r}}\lesssim_{\sigma, \bar s, p} \|h_1\|_{H^{\bar s}} \|h_2\|_{B^{\sigma}_{p,r}}.
\]
\end{itemize}

 Combining the above bounds we have
\[
\|h_1 \cdot h_2\|_{B^{\sigma}_{p,r}}\lesssim_{\sigma, \bar s, p} \|h_1\|_{H^{\bar s}} \|h_2\|_{B^{\sigma}_{p,r}} \quad \text{for all } \sigma\in \big(-\min\{\bar s, 3\}, \, \min\{\bar s, 3\}\big), p\in(1,+\infty), r\in [1,+\infty].
\]
Applying the above estimate for $h_1=b$, $h_2 = \bar\partial^2 f$, we thus infer:
\[
\|b \cdot \bar\partial^2 f \|_{B^{\sigma}_{p,r}} \lesssim_{\sigma, \bar s, p} \|b\|_{H^{\bar s}} \|\bar\partial f\|_{B^{\sigma+1}_{p,r}}.
\]
Using this bound to estimate the last term in the right hand side of \eqref{Besov elliptic estimate}, we obtain \eqref{Estimate elliptic} provided the constant $c_0$ in 
\eqref{lem:Elliptic estimates model} (which controls $\|b\|_{H^{\bar s}}$) has been chosen small enough in terms of $\sigma, \bar s, p$.

In the case when $\bar s\in(\f32,2)$, we can similarly estimate for $\sigma \in (-2, -\bar s)$, $\delta \in (0, \f12(\bar s-\f32))$ and $p\in (1,+\infty)$:
\begin{align*}
2^{\sigma j}\|P_j (b \cdot \bar\partial^2 f) \|_{L^p} & \lesssim 2^{\sigma j} \|P_{\le j}b \cdot P_j (\bar\partial^2 f) \|_{L^p} + 2^{\sigma j}\|P_j b \cdot P_{\le j}(\bar\partial^2 f)\|_{L^p} +2^{\sigma j}\|\sum_{j'>j} P_j (P_{j'} b \cdot P_{j'}(\bar\partial^2 f)) \|_{L^p}\\ 
& \lesssim_{\sigma,p} 
\|b\|_{L^\infty} \cdot 2^{\sigma j} \| P_j (\bar\partial^2 f) \|_{L^p}  +  \|P_j b\|_{L^\infty} 2^{-|\sigma| j} \| P_{\le j} (\bar\partial^2 f) \|_{L^p} \\[5pt]
&\hphantom{\lesssim_{\sigma,p, \delta} \|b\|_{L^\infty}}
 +  \|\sum_{j'>j} P_j (P_{j'} b \cdot P_{j'}(\bar\partial^2 f)) \|_{L^{\frac{3p}{3+|\sigma|p}}} \\
& \lesssim_{\sigma,p} 
\|b\|_{L^\infty} \cdot 2^{\sigma j} \| P_j (\bar\partial^2 f) \|_{L^p}  +  \|P_j b\|_{L^\infty} 2^{-|\sigma| j} \| P_{\le j} (\bar\partial^2 f) \|_{L^p} \\[5pt]
&\hphantom{\lesssim_{\sigma,p, \delta} \|b\|_{L^\infty}}
 +  \sup_{j'>j} \big(2^{(\bar s+\f12)j'} \|P_{j'} b\|_{L^2} \big) \cdot \sum_{j'>j} \big(2^{-(\bar s+\f12)j'}\|P_{j'}(\bar\partial^2 f)\|_{L^{\f{6p}{(2|\sigma|-3)p+6}}} \big).
\end{align*}
Using the Sobolev embedding $W^{-1,\frac{6p}{(2|\sigma|-3)p+6}} \hookrightarrow B^{-\bar s+\f12}_ {\frac{6p}{(2|\sigma|-3)p+6},1}$ (since $\bar s>\f32$), we can thus bound for any $r\in [1,+\infty]$:
\[
\|b \cdot \bar\partial^2 f \|_{B^{\sigma}_{p,r}}  \lesssim_{\sigma, \bar s, p} \|b\|_{H^{\bar s}} \|\bar\partial f\|_{B^{\sigma+1}_{p,r}}  + \|b\|_{H^{\bar s+\f12}} \|\bar\partial f\|_{W^{-1,\frac{6p}{(2|\sigma|-3)p+6}}}.
\]
Using this bound to estimate the last term in the right hand side of \eqref{Besov elliptic estimate}, we obtain \eqref{Very low regularity elliptic estimate}.

\end{proof}

\end{document}